\documentclass[reqno,11pt]{amsart}

\usepackage{
    amssymb, 
    amsbsy,
    amsmath,
    amsfonts, 
    stix,
    mathtools,
    xcolor,
    mathdots, 
    multirow,
    accents,
    tensor,
    xstring,
    calc,
    ifthen,
    enumitem,
    multicol,
    appendix,
    nicematrix,
    tikz-cd,
    adjustbox,
    etoolbox,
    setspace,
    changepage
    }   

\usepackage{palatino} 
\usepackage[normalem]{ulem}

\usepackage[all]{xy}
\input{diagxy}

\usepackage[OT2,OT1]{fontenc}

\usepackage[bookmarks=true,%
    colorlinks=true,%
    linkcolor=blue,%
    citecolor=blue,%
    filecolor=blue,%
    menucolor=blue,%
    urlcolor=blue,%
    breaklinks=true]{hyperref}

\letcs\replicate{prg_replicate:nn}

\newcommand*\longsum[1][1]{%
  \mathop{\textnormal{%
    \clipbox{0pt 0pt {.367\width} 0pt}{\scalebox{1.2}{$\displaystyle\sum$}}%
    \replicate{#1*4}{\clipbox{{.67\width} 0pt {.32\width} 0pt}{\scalebox{1.2}{$\displaystyle\sum$}}}%
   \clipbox{{.638\width} 0pt 0pt 0pt}{\scalebox{1.2}{$\displaystyle\sum$}}}}%
}

\newcommand\cyr
{
\renewcommand\rmdefault{wncyr}
\renewcommand\sfdefault{wncyss}
\renewcommand\encodingdefault{OT2}
\normalfont
\selectfont
}
\DeclareTextFontCommand{\textcyr}{\cyr}

\numberwithin{equation}{section}

\newtheorem{thm}{Theorem}[section]
\newtheorem{lem}[thm]{Lemma}
\newtheorem{lemma}[thm]{Lemma}
\newtheorem{prop}[thm]{Proposition}
\newtheorem{cor}[thm]{Corollary}
\newtheorem{conj}[thm]{Conjecture}

\newtheorem{remark}[thm]{Remark}
\newtheorem{defn}[thm]{Definition}

\newenvironment{rem}
  {\pushQED{\qed}\remark}
  {\popQED\endremark}

\newtheorem{Thm}{Theorem}

\newcommand{\bracket}[1]{\mathchoice{\left\langle{#1}\right\rangle}{\langle{#1}\rangle}{\langle{#1}\rangle}{\langle{#1}\rangle}}
\newcommand{\floor}[1]{\lfloor{#1}\rfloor}

\newcommand{\widebarr}[1]{\accentset{\rule{\widthof{$#1\!$}}{0.7pt}}{#1}}

\newcommand{\id}{\mathrm{id}}

\newcommand{\GrLi}[1]{Gr(#1)}
\newcommand{\Spec}{\mathrm{Spec}}
\newcommand{\AChar}[1]{\mathrm{Spec}_{\mathfrak m}(#1)}

\newcommand{\LT}[1]{\ifmmode \mathsf{#1} \else \textsf{\small #1}\fi}

\newcommand{\PLexp}{\#}
 
\newcommand{\kay}{\mathscr k}
\newcommand{\elln}{\ell\mkern -3mu} 
\newcommand{\ellfun}{\mathfrak l}

\newcommand{\gcdle}{\mathfrak d}
\newcommand{\gcdleopp}{\mathfrak n}
\newcommand{\delzeta}[1]{{\scalebox{.9}{$\nabla$}\mkern -3mu}_{#1}\mkern -1mu}

\newcommand{\ellsc}[1]{\raise 3pt \hbox{\scalebox{0.7}{$\ell$}}\mkern -2mu {#1}}
\newcommand{\ellscscr}[1]{\raise 1.6pt \hbox{\scalebox{0.5}{$\ell$}}\mkern -2mu {#1}}

\newcommand{\slashone}{\scalebox{.66}{$\boldsymbol \succ$}}
\newcommand{\slashoned}{\scalebox{.75}{$\boldsymbol\sim$}}
\newcommand{\slashtwo}{\scalebox{.7}{$\boldsymbol =$}}
\newcommand{\slashthree}{\rotatebox{90}{\scalebox{1.15}{$\bullet$}}}

\newcommand{\Lsignp}{\raise 0.43mm \hbox{$\slashone$}\mkern -8mu \zeta}
\newcommand{\Lsign}{\raise 0.43mm \hbox{$\slashoned$}\mkern -8mu \zeta}
\newcommand{\LLsign}{\raise 0.48mm \hbox{$\slashtwo$}\mkern -7.7mu \zeta}

\newcommand{\Qsign}{\raise 0.29mm \hbox{$\slashthree$}\mkern -8.3mu \zeta}

\def\invexcl{{\normalfont\rotatebox[origin=c]{180}!}}

\newcommand{\qfacrelA}[2]{[#1\mkern -3mu :\mkern -3mu#2]\mkern -1.6mu \invexcl\mkern -1.6mu\invexcl} 

\newcommand{\qfacrelB}[2]{\left\{#1\mkern -3mu :\mkern -3mu #2\right \}}

\newcommand{\qfactcollJ}[2]{\llparenthesis #1, j\rrparenthesis_{#2}}
\newcommand{\qfactcollI}[2]{\llparenthesis #1, i\rrparenthesis_{#2}}

\newcommand{\elli}{{\ell_i}}
\newcommand{\ellj}{{\ell_j}}

\newcommand{\qbin}[3]{\begin{bmatrix} \vspace{-0.7mm} \raise 1mm \hbox{$#1$}\vspace{-0.7mm}\\ #2\vspace{-0.8mm}\end{bmatrix}_{#3}}
\newcommand{\qbinsmall}[2]{\begin{bmatrix} \vspace{-2.7mm} \raise 1.1mm \hbox{$\scriptstyle{#1}$} \\ \scriptstyle{#2}\vspace{-1.1mm}\end{bmatrix}}
\newcommand{\Lfact}{\mathcal{D}}
\newcommand{\twococ}{\mathbb{c}}

\newcommand{\lclbd}{\chi}
\newcommand{\bbz}{\mathbb{Z}}
\newcommand{\bbq}{\mathbb{Q}}
\newcommand{\bbc}{\mathbb{C}}
\newcommand{\bbn}{\mathbb{N}}
\newcommand{\nnN}{\mathbb{N}_0}

\newcommand{\VanLocus}{\mathcal V}
\newcommand{\VanIdeal}{\mathcal I}

\newcommand{\lpip}{\boldsymbol{\left(\right.}}
\newcommand{\rpip}{\boldsymbol{\left.\right)}}

\newcommand{\ZarClos}[1]{\mathrm{Zcl}(#1)}
\newcommand{\radic}[1]{\raise 0.8pt \hbox{\scalebox{.85}{$\surd$}} \mkern -1mu {#1}}

\newcommand{\bbk}{\mathbb{k}}
\newcommand{\Qh}{\bbq\lBrack \hbar\rBrack}
\newcommand{\Qq}{\bbq(q)}
\newcommand{\Qz}{\bbq(\zeta)}

        \newcommand{\vaccent}{{\mkern -4mu \raisebox{.11em}{\scalebox{.47}{$\blacktriangledown$}}}}
         \newcommand{\Zgen}{\Lambda}
        \newcommand{\Zgenv}{\Zgen^{\vaccent}}

        \newcommand{\Zqq}{\Zgen_q}
        \newcommand{\Zqqv}{\Zgen_q^\vaccent} 
        \newcommand{\Zqqn}[1]{\Zgen_{q,#1}}
        \newcommand{\Zqqvn}[1]{\Zgen_{q,#1}^\vaccent}

        \newcommand{\Zz}{\Zgen_\zeta}
        \newcommand{\Zzv}{\Zgen_\zeta^\vaccent} 
        \newcommand{\Zzn}[1]{\Zgen_{\zeta,#1}}
        \newcommand{\Zzvn}[1]{\Zgen_{\zeta,#1}^\vaccent}

\newcommand{\utypechar}{{\raisebox{1pt}{$\scriptscriptstyle\Box$}}}

\newcommand{\Tinv}{\mathit{\Gamma}}

\newcommand{\lbacc}[1]
{  
  \phantom{\scalebox{.65}{$#1$}}
  \accentset{\bullet\mkern -1.1 mu\bullet}
   {
   \scalebox{.82}{\vphantom{$#1$}}
    }
  \negphantom{\scalebox{.65}{$#1$}}
    {#1}
 }
\newcommand{\ELu}{\lbacc{E}}
\newcommand{\FLu}{\lbacc{F}}

\newcommand{\Esing}{\breve{E}}
\newcommand{\Fsing}{\breve{F}}

\newcommand{\wgrad}{\mathbb w}
\newcommand{\wgradp}{\wgrad^+}
\newcommand{\wgradn}{\wgrad^-}

\newcommand{\wgradZBtwo}{\mathbb u}

\newcommand{\ztgrad}{\mathbb o}
\newcommand{\ztgradp}{\ztgrad^+}
\newcommand{\ztgradn}{\ztgrad^-}

\newcommand{\ZYXgrad}{\mathbb z}
\newcommand{\pgrad}{\mathbb p}

\newcommand{\fg}{\mathfrak{g}}

\newcommand{\Uhg}{U_{\hbar}(\fg)}
\newcommand{\Uqg}{U_{q}(\fg)}
\newcommand{\Uzg}{U_{\zeta}(\fg)}
\newcommand{\Uq}{U_q}
\newcommand{\Uqre}[1]{U_{q,#1}}
\newcommand{\UqQ}{\Uqre{\mathbb{Q}}}
\newcommand{\UqreC}[1]{\hat U_{q,#1}}
\newcommand{\UqQC}{\UqreC{\mathbb{Q}}}
\newcommand{\Uqrest}[2]{{U}_{\ \mkern -9mu q}\mkern -1.5mu \raise 0.08em\hbox{$\scriptstyle (#2)$}^{#1}}
\newcommand{\Uqrerest}[3]{{U}_{\ \mkern -9mu q, #3}\mkern -1.5mu \raise 0.08em\hbox{$\scriptstyle (#2)$}^{#1}}
\newcommand{\Uz}{U_\zeta}

\newcommand{\Uzn}[1]{U_{\zeta,#1}}
\newcommand{\Uzk}[1]{U_{\zeta,#1}}
\newcommand{\UzQ}{\Uzk{\mathbb Q}}
\newcommand{\Uzrest}[2]{{U}_{\ \mkern -9mu \zeta}\mkern -1.5mu \raise 0.08em\hbox{$\scriptstyle (#2)$}^{#1}} 

\newcommand{\LUq}{{\raise -0.15em \hbox{$\scriptstyle\mathscr L$}}{U}}

\newcommand{\DCPU}{{\raise -0.15em \hbox{$\scriptstyle\mathscr S$}}{U}}

\newcommand{\Zsubalgchar}{\mathcal Z}
\newcommand{\Zmaxidchar}{\mathcal K}
\newcommand{\Zaugideal}[1]{\Zmaxidchar\raise 0.7pt \hbox{\scalebox{.82}{$(#1)$}}}
\newcommand{\Zaugidealhat}[1]{\widehat{\Zmaxidchar}\raise 0.7pt \hbox{\scalebox{.82}{$(#1)$}}}

\newcommand{\ZaugidealhatIn}[2]{
        \Zaugidealhat{#1}_{\scalebox{.62}{$#2$}}
        }

\newcommand{\pTmIdeal}[1]{{\mathcal N\mkern -5mu}_\circ(#1)}
\newcommand{\FullpTmIdeal}[1]{{\mathcal N}\mkern -1mu(#1)}
\newcommand{\FullpTmIdealSym}[1]{{\mathcal N\!}_*\mkern -1mu(#1)}

\newcommand{\ZsubalgcharT}{\mkern 1.5mu\overset{\sim}{\rule{0pt}{3.7pt}}\mkern -12mu {\Zsubalgchar}}

\newcommand{\ZSisom}{\mathscr m}

\newcommand{\ZmaxId}[1]{\Zmaxidchar_{#1}}
\newcommand{\ZmaxIdHat}[1]{\widehat{\Zmaxidchar}_{#1}}

\newcommand{\siAN}{
            \scalebox{.97}{$\bullet$}
            \mkern -1.6 mu
            \raise .06em \hbox{\scalebox{.7}{$\circ$}}
            }
\newcommand{\siNA}{
            \raise .06em \hbox{\scalebox{.7}{$\circ$}}
            \mkern -1.6 mu
            \scalebox{.97}{$\bullet$}
            }
\newcommand{\siAA}{
            \scalebox{.95}{$\bullet$}
            \mkern -1.4 mu
            \scalebox{.95}{$\bullet$}
            }

\newcommand{\FullpTmIdealMaxT}[1]{{\mathcal N}^{\,#1}_{\siAA}}

\newcommand{\geqzero}{{\mathchoice
   {\scalebox{.86}{$\mkern 1mu \geq\mkern -4.3mu 0$}}
   {\scalebox{.86}{$\mkern 1mu \geq\mkern -4.3mu 0$}}
   {\scalebox{.5}{$\mkern 1mu \geq\mkern -4.3mu 0$}}
   {\scalebox{.5}{$\mkern 1mu \geq\mkern -4.3mu 0$}}
   }}
\newcommand{\leqzero}{{\mathchoice
   {\scalebox{.86}{$\mkern 1mu \leq\mkern -4.3mu 0$}}
   {\scalebox{.86}{$\mkern 1mu \leq\mkern -4.3mu 0$}}
   {\scalebox{.5}{$\mkern 1mu \leq\mkern -4.3mu 0$}}
   {\scalebox{.5}{$\mkern 1mu \leq\mkern -4.3mu 0$}}
   }}

\newcommand{\opp}{{\mathchoice
   {\scalebox{.86}{$\mkern 1mu \mathsf{opp}$}}
   {\scalebox{.86}{$\mkern 1mu \mathsf{opp}$}}
   {\scalebox{.52}{$\mkern 1mu \mathsf{opp}$}}
   {\scalebox{.52}{$\mkern 1mu \mathsf{opp}$}}
   }}

\newcommand{\Deltaopp}{{\Delta\mkern -3mu}^\opp}

\newcommand{\cop}{{\mathchoice
   {\scalebox{.86}{$\mkern 1mu \mathsf{cop}$}}
   {\scalebox{.86}{$\mkern 1mu \mathsf{cop}$}}
   {\scalebox{.52}{$\mkern 1mu \mathsf{cop}$}}
   {\scalebox{.52}{$\mkern 1mu \mathsf{cop}$}}
   }}

\newcommand{\redsup}{{\mathchoice
   {\scalebox{.93}{$\mkern 1mu \mathsf{rest}$}}
   {\scalebox{.93}{$\mkern 1mu \mathsf{rest}$}}
   {\scalebox{.62}{$\mkern 1mu \mathsf{rest}$}}
   {\scalebox{.62}{$\mkern 1mu \mathsf{rest}$}}
   }}

\newcommand{\mytypeqg}{restricted\ }

\newcommand{\Bor}{U^{\geq0}}
\newcommand{\Borq}{\Bor_q}

\newcommand{\Cartaninv}{\Omega}
\newcommand{\CartaninvTwo}{\Omega^\#}
\newcommand{\Cartanaut}{\Pi}
\newcommand{\Kinvaut}{\mho}
\newcommand{\Kconaut}{\amalg}
\newcommand{\Qinvaut}{\Upsilon}
\newcommand{\DynkInv}{\mbox{\hspace*{-4.5mm} \cyr I}}
\newcommand{\qitgen}{\varpi}
\newcommand{\htsign}{\boldsymbol{\iota}}
\newcommand{\sdrchar}[2]{
 \mathchoice{\vartheta^{#2}\mkern -9mu\raisebox{-1.7pt}{$\scriptstyle #1$}\mkern 2.7mu}
              {\vartheta^{#2}\mkern -9mu\raisebox{-1.7pt}{$\scriptstyle #1$}\mkern 2.7mu}
              {\vartheta^{#2}\mkern -9mu\raisebox{-1.7pt}{$\scriptscriptstyle #1$}\mkern 2.7mu}
              {\vartheta^{#2}\mkern -3mu{\scriptscriptstyle #1}}
              }
\newcommand{\reldrchar}{\boldsymbol{\kappa}}

\newcommand{\KinvautTwoCoeff}[3]{\scalebox{.76}{$\boldsymbol{\mho}$}^{#1}_{#2}{}}
\newcommand{\KinvautExpSet}[2]{\Kinvaut_{\{#1\}}(#2)}
 
\newcommand\Kscalechar{\mbox{\hspace{-4mm}\cyr Ch}} 

\newcommand{\Kscale}[2]{\Kscalechar^{#2}_{#1}}

\newcommand{\Kscaleword}{Che\ }

\newcommand{\barDelta}{\raise 0.72em\hbox{\rule{1.8mm}{0.27mm}}\mkern -10.6mu \Delta}

\newcommand{\emptyword}{\scalebox{.9}{$\emptyset$}}

\newcommand{\grf}[1]{
   \mathchoice{\raisebox{-1pt}{\scalebox{1.18}{$\mathsf{#1}$}}}
              {\raisebox{-1pt}{\scalebox{1.18}{$\mathsf{#1}$}}}
              {\raisebox{-0.15pt}{\scalebox{0.8}{$\mathsf{#1}$}}}
              {\raisebox{-0.15pt}{\scalebox{0.8}{$\mathsf{#1}$}}}}

\newcommand{\sroots}{\grf{\Delta}}
\newcommand{\roots}{\grf{\Phi}}
\newcommand{\rootsL}{\roots_L}
\newcommand{\rootsS}{\roots_S}

\newcommand{\proots}{{\roots^+}}
\newcommand{\nroots}{{\roots^-}}
\newcommand{\prootsL}{{\roots^+_L}}
\newcommand{\nrootsL}{{\roots^-_L}}
\newcommand{\prootsS}{{\roots^+_S}}
\newcommand{\nrootsS}{{\roots^-_S}}

\newcommand{\coproots}{{\breve\roots}^+}

\newcommand{\symCart}{C}
\newcommand{\symCartinv}{B}

\newcommand{\prootstwo}[1]{{\roots^+}\mkern -2mu {\scriptstyle (#1)}}
\newcommand{\rootstwo}[1]{{\roots}\mkern -.3mu {\scriptstyle (#1)}}

\newcommand{\longchar}[1]{{#1\mkern -1mu _\bullet}}
\newcommand{\longweyl}{\longchar{s}}

\newcommand{\flweyl}[1]{#1\strut^{\llcorner}}

\newcommand{\wcomplrel}{\dualmap}

\newcommand{\winvchar}{\dagger}
\newcommand{\dyninv}{\eta}

\newcommand{\longtwo}[2]{
   \mathchoice
      {{#2}^\circ\mkern -7mu \raise -0.21em \hbox{$\scriptstyle #1$}}
      {{#2}^\circ\mkern -7mu \raise -0.21em \hbox{$\scriptstyle #1$}}
      {{#2}^\circ\mkern -7mu \raise -0.16em \hbox{$\scriptscriptstyle #1$}}
      {{#2}^\circ\mkern -7mu \raise -0.11em \hbox{\small $\scriptscriptstyle #1$}}
      }
\newcommand{\longtwoalt}[2]{
   \mathchoice
      {{#2}\mkern -1mu \raise -0.17em \hbox{\scalebox{0.6}{$|#1|$}}}
     {{#2}\mkern -0.8mu \raise -0.17em \hbox{\scalebox{0.6}{$|#1|$}}}
     {{#2}\mkern -0.77mu \raise -0.15em \hbox{\scalebox{0.46}{$|#1|$}}}
      {{#2}^\circ\mkern -7mu \raise -0.11em \hbox{\small $\scriptscriptstyle #1$}}
      }
\newcommand{\longtwoweyl}[1]{\longtwoalt{#1}{s}}
\newcommand{\longtwoword}[1]{\longtwoalt{#1}{z}}

\newcommand{\symbrack}[2]{\mathchoice{\left( #1 | #2 \right)}{( #1 | #2)}{( #1 | #2)}{( #1 | #2)}}
\newcommand{\Weyl}{\mathscr{W}}
\newcommand{\WeylS}{\Weyl_S}
\newcommand{\WeylL}{\Weyl_L}
\newcommand{\SimRefl}{\mathscr{S}}
\newcommand{\wordset}{\mathscr{W}^*}
\newcommand{\wordsetmax}{\mathscr{W}^*_{\mkern -2mu \mathfrak m}}
 
\newcommand{\Weylrest}[1]{\Weyl\mkern -4mu \raise -0.04em \hbox{$\scriptstyle(#1)$}}
\newcommand{\wordsetrest}[1]{\wordset\mkern -4mu \raise -0.06em \hbox{$\scriptstyle(#1)$}}
\newcommand{\wordsetmaxrest}[1]{\wordsetmax\mkern -2.5mu \raise -0.06em \hbox{$\scriptstyle(#1)$}}

\newcommand{\weightvecchar}{\rho}
\newcommand{\weightvec}[1]{\weightvecchar_{#1}}
\newcommand{\weightlattice}{\grf{\Sigma}}

\newcommand{\len}[1]{\length(#1)}
\newcommand{\length}{l}
\newcommand{\height}{\mathrm{ht}}

\newcommand{\edgenum}{\mathsf e}
\newcommand{\maxd}{\mathsf e}
\newcommand{\dpone}{{\mathsf e^*}}

\newcommand{\Weylpres}{\mathscr r}
\newcommand{\Matsec}{\mathscr s}
\newcommand{\wordroot}{\gamma}

\newcommand{\relrootsplit}[1]{\mathscr b_{(#1)}}

\newcommand{\descrootschar}{\mkern 0.4mu\pmb{\mathscr N}\mkern -1mu}

\newcommand{\descroots}[1]{\descrootschar(#1)}
\newcommand{\descrootsS}[1]{\descrootschar(#1)_S}
\newcommand{\descrootsL}[1]{\descrootschar(#1)_L}
\newcommand{\sumdescrootssym}{\theta}
\newcommand{\sumdescroots}[1]{\sumdescrootssym(#1)}
\newcommand{\sumdescrootsrel}[2]{\sumdescrootssym_{#2}(#1)}
\newcommand{\sumdesccoroots}[1]{\breve\sumdescrootssym(#1)}

\newcommand{\wordindexset}[1]{J({#1})}
\newcommand{\wordindexlattice}[1]{\grf{\Xi}(#1)}

\newcommand{\idsymm}{\mathsfit i\!d}

\newcommand{\symmmapspace}[1]{{\mathbb S}(#1)}

\newcommand{\leqRB}{\leq_R}
\newcommand{\leqLB}{\leq_L}
\newcommand{\geqRB}{\geq_R}

\newcommand{\prLleq}{\precsim}
\newcommand{\prLgeq}{\succsim}

\newcommand{\leqwt}[1]{\preceq_{#1}}
\newcommand{\letwt}[1]{\prec_{#1}}

\newcommand{\garside}{\boldsymbol{\partial}}

\newcommand{\dBsg}[1]{{\mathrm W}_{#1}}

\newcommand{\qRchar}{\mathcal{P}}
\newcommand{\qRm}{\qRchar}

\newcommand{\Ri}{\qRm_i}
\newcommand{\Rw}{\qRm_w}
\newcommand{\Ra}[1]{\qRm_{#1}}

\newcommand{\Rprod}[2]{\qRchar_{\mkern -4mu \lflex{#1}}^{#2}}
\newcommand{\RprodInv}[2]{\widetilde{\qRchar}_{\mkern -4mu \lflex{#1}}^{#2}}
\newcommand{\Rprodmax}{\qRchar_{\mkern -4mu \bullet}}

\newcommand{\Rexp}{\mathcal{D}}
\newcommand{\taniauto}{\mbox{\hspace{-4mm} \cyr L}}

\newcommand{\Rcoeff}[2]{\mathbf{c}_{#1}\mkern -.7 mu \scalebox{.83}{$(#2)$}\mkern 1.6 mu}
\newcommand{\Rcoeffdenom}[2]{\boldsymbol{\sigma}_{#1}\mkern -.7 mu \scalebox{.83}{$(#2)$}\mkern 1.6 mu}

\newcommand{\SingCoeff}[1]{\boldsymbol{\delta}^{#1}}
\newcommand{\SingCoeffvar}[2]{\boldsymbol{\delta}^{#1}\mkern -1.7 mu \scalebox{.83}{$(#2)$}\mkern 1.6 mu}

\newcommand{\Multcoeff}[4]{\mathbf{m}_{#1}^{#2,#3}\mkern -.7 mu \scalebox{.83}{$(#4)$}}
\newcommand{\CoMultcoeff}[4]{\mathbf{d}^{#1}_{#2,#3}\mkern -.7 mu \scalebox{.83}{$(#4)$}}

\newcommand{\ExpGrp}{\scalebox{-.93}[1]{\hspace{-4.75mm} \cyr Z}}

\newcommand{\trunacc}{{\scalebox{.35}{$\blacklozenge$}}}
\newcommand{\Rt}{\qRchar^\trunacc}
\newcommand{\Rtinv}{
    \raise .73em\hbox{\rule{2.3mm}{.28mm}}
    \mkern -11.8mu
     {\qRchar}^\trunacc
     }

\newcommand{\RtprodInv}[1]{\widetilde{\qRchar}^\trunacc_{\mkern -4mu \lflex{#1}}}

\newcommand{\redacc}{{\scalebox{.75}{$\star$}}}

\newcommand{\Rtred}{\qRchar^\redacc}
\newcommand{\RtredInv}{\widetilde{\qRchar}^\redacc}

\newcommand{\Rtsub}[1]{\Rt_{\!#1}}
\newcommand{\Rtinvsub}[1]{\Rtinv_{\!#1}}

\newcommand{\Rti}{\Rtsub{i}}

\newcommand{\Rtii}{\Rtinvsub{i}}

\newcommand{\Rtprod}[1]{\Rtsub{\lflex{#1}}} 
\newcommand{\Rtprodi}[1]{\Rtinvsub{{#1}\rflex}}

\newcommand{\Rtprodfac}[1]{\Rtred_{\lflex{#1}}}

\newcommand{\Rtani}{\mathcal{F}}
\newcommand{\RtaniImp}[2]{D_{#1,#2}}

\newcommand{\commphchar}{\grf{\kappa}}
\newcommand{\commph}[2]
   {\raise 0.12em \hbox{$\commphchar$}
       \mkern -1.5mu \hbox{\scalebox{0.85}{$(#1,#2)$}}}
\newcommand{\commphelchar}{\grf{\hat\kappa}} 
\newcommand{\commphel}[3]
   {\raise 0.12em \hbox{$\commphelchar$}
       \mkern -1.5mu \hbox{\scalebox{0.85}{$(#2,#3)^{#1}$}}}

\newcommand{\commphinchar}{\grf{\pi}}
\newcommand{\commphin}[2]
   {\raise 0.12em \hbox{$\commphinchar$}
       \mkern -1.5mu \hbox{\scalebox{0.85}{$(#1,#2)$}}}
\newcommand{\commphinelchar}{\grf{\hat\pi}} 
\newcommand{\commphinel}[2]
   {\raise 0.12em \hbox{$\commphinelchar$}
       \mkern -1.5mu \hbox{\scalebox{0.85}{$(#1,#2)$}}}
       
\newcommand{\commphshufexpchar}{\grf{\varsigma}} 
\newcommand{\commphshufexp}[3]
   {\raise 0.12em \hbox{$\commphshufexpchar$}
       \mkern -1.5mu \hbox{\scalebox{0.85}{$(#2,#3)_{#1}$}}}
\newcommand{\commphallexpchar}{\grf{\hat{\varsigma\,}\!}} 
\newcommand{\commphallexp}[4]
   {\raise 0.12em \hbox{$\commphallexpchar$}
       \mkern -1.5mu \hbox{\scalebox{0.85}{$(#3,#4)_{#1,#2}$}}}
       
\newcommand{\commphXEpwchar}{\grf{\overrightarrow\kappa}} 
\newcommand{\commphXEpw}[3]
   {\raise 0.12em \hbox{$\commphXEpwchar$}
       \mkern -1.5mu \hbox{\scalebox{0.85}{$(#2,#3)_{#1}$}}}
  \newcommand{\commphEXpwchar}{\grf{\overleftarrow\kappa}} 
\newcommand{\commphEXpw}[3]
   {\raise 0.12em \hbox{$\commphEXpwchar$}
       \mkern -1.5mu \hbox{\scalebox{0.85}{$(#2,#3)_{#1}$}}}

\newcommand{\commphbihomchar}
    {\raise 0.07em \hbox{$\grf{\varepsilon}$}}
\newcommand{\commphbihom}[2]
   {\raise 0.01em \hbox{$\commphbihomchar$}
       \mkern 0mu \hbox{\scalebox{0.92}{$(#1,#2)$}}}
       
\newcommand{\commphbigrchar}
    {\raise 0.07em \hbox{$\grf{\upsilon}$}}
\newcommand{\commphbigr}[2]
   {\raise 0.01em \hbox{$\commphbigrchar$}
       \mkern 0.15mu \hbox{\scalebox{0.92}{$(#1,#2)$}}}
       
\newcommand{\Epw}{X}
\newcommand{\Fpw}{Y}
\newcommand{\Kpw}{L}

\newcommand{\Glg}{P}
\newcommand{\Glgpw}{Q}

\newcommand{\Wops}{\mathsf M}
\newcommand{\Wopsalg}{\mathbb M}

\newcommand{\Commsigngrp}{\mathscr G}
\newcommand{\Commsignact}{\raise -2pt \hbox{\scalebox{0.6}{$\smalltriangleright$}}}

\newcommand{\skewcoeffchar}{\mathsf c}
\newcommand{\skewcoeffgen}[3]{#1\scalebox{0.75}{$(#2,#3)$}}
\newcommand{\skewcoeff}[2]{\skewcoeffgen{\skewcoeffchar}{#1}{#2}}

\newcommand{\UtriM}[1]{\mathrm{T}_{#1}}
\newcommand{\SUtriM}[1]{\mathrm{U}_{#1}}
\newcommand{\DiagM}[1]{\mathrm{D}_{#1}}

\newcommand{\QMatAlg}[3]{\mathcal O(\UtriM{#3},#1,#2)}
\newcommand{\QMatAlgS}[2]{\mathcal O(\UtriM{#2},#1)}

\newcommand{\SQMatAlg}[3]{\mathcal{O}\!_\circ(\UtriM{#3},#1,#2)}
\newcommand{\SQMatAlgS}[2]{\mathcal{O}\!_\circ(\UtriM{#2},#1)}

\newcommand{\derideal}[3]{{\widehat {#1}}}

\makeatletter
\newlength{\negph@wd}
\DeclareRobustCommand{\negphantom}[1]{%
  \ifmmode
    \mathpalette\negph@math{#1}%
  \else
    \negph@do{#1}%
  \fi
}
\newcommand{\negph@math}[2]{\negph@do{$\m@th#1#2$}}
\newcommand{\negph@do}[1]{%
  \settowidth{\negph@wd}{#1}%
  \hspace*{-\negph@wd}%
}
\makeatother

\makeatletter
\newlength{\lenone@wd}
\newlength{\lentwo@wd}
\DeclareRobustCommand{\maxphantom}[2]{%
  \ifmmode
    \maxph@math{#1}{#2}%
  \else
    \maxph@do{#1}{#2}%
  \fi
}
\newcommand{\maxph@math}[2]{\maxph@do{$#1$}{$#2$}}
\newcommand{\maxph@do}[2]{%
  \settowidth{\lenone@wd}{#1}%
  \settowidth{\lentwo@wd}{#2}%
  \ifthenelse{\lenone@wd>\lentwo@wd}
        {\hspace*{\lenone@wd}}%
        {\hspace*{\lentwo@wd}}%
}
\makeatother

\newcommand{\lmin}{}
\newcommand{\rflex}{\raisebox{0.8pt}{\scalebox{.33}{$\mkern 2mu  \blacktriangleright\mkern 1mu $}}}
\newcommand{\rmin}{}
\newcommand{\lflex}{\raisebox{0.8pt}{\scalebox{.33}{$\mkern 1.3mu  \blacktriangleleft\mkern 3.8mu $}}}
 
\newcommand{\Genbase}[3]{ {#3}\mkern 1.5mu\raisebox{5.5pt}{$\scriptstyle #2$}\mkern -2.8mu {\scriptstyle \negphantom{#2\!}} \raisebox{-2.7pt}{$\scriptstyle \lmin #1 \rflex$}
{\scriptstyle\negphantom{{\lmin #1 \rflex}}}
 {\scriptstyle \maxphantom{{\lmin #1 \rflex}}{#2\!\!\!}}
}

\newcommand{\Genbaseopp}[3]{ {#3}\mkern 1.3mu\raisebox{5.5pt}{$\scriptstyle #2$}\mkern -2.8mu {\scriptstyle \negphantom{#2\!}} \raisebox{-2.7pt}{$\scriptstyle  \lflex #1 \rmin$}
{\scriptstyle\negphantom{{ \lflex #1 \rmin}}}
 {\scriptstyle \maxphantom{{ \lflex #1 \rmin}}{#2\!\!\!}}
}

\newcommand{\Ebase}[2]{\Genbase{#1}{#2}{E}}
\newcommand{\Fbase}[2]{\Genbase{#1}{#2}{F}} 
\newcommand{\Ebaseopp}[2]{\Genbaseopp{#1}{#2}{E}}
\newcommand{\Fbaseopp}[2]{\Genbaseopp{#1}{#2}{F}}  

\newcommand{\Edivbase}[2]{\Genbase{#1}{(#2)}{E}}
 
\newcommand{\Edivbaseopp}[2]{\Genbaseopp{#1}{(#2)}{E}}
\newcommand{\Fdivbaseopp}[2]{\Genbaseopp{#1}{(#2)}{F}}

\newcommand{\Epwbase}[2]{\Genbase{#1}{#2}{\Epw}}
\newcommand{\Fpwbase}[2]{\Genbase{#1}{#2}{\Fpw}} 
\newcommand{\Epwbaseopp}[2]{\Genbaseopp{#1}{#2}{\Epw}}
\newcommand{\Fpwbaseopp}[2]{\Genbaseopp{#1}{#2}{\Fpw}}

\newcommand{\circL}{\mkern 5mu\raise 1.4ex\hbox{$\scriptstyle{\circ}$}\mkern -11mu\mathcal L}
\newcommand{\barL}{\hat{\mathcal L}}

\newcommand{\TCompModopp}[1]{\Genbaseopp{#1\!}{}{\barL}}
\newcommand{\TCompModopptwo}[1]{\Genbaseopp{#1\!}{}{\barL^{\,2}\!\!\!}}
\newcommand{\TCompModGenopp}[1]{\Genbaseopp{#1\!}{}{\circL}}
\newcommand{\TCompModGenopptwo}[1]{\Genbaseopp{#1}{}{\circL^{\,2}\!\!\!}}

\newcommand{\EbaseLU}[2]{\Genbase{#1}{#2}{\ELu}}
\newcommand{\EbaseoppLU}[2]{\Genbaseopp{#1}{#2}{\ELu}}

\newcommand{\Ebasesing}[2]{\Genbase{#1}{#2}{\Esing}}
\newcommand{\Ebaseoppsing}[2]{\Genbaseopp{#1}{#2}{\Esing}}
\newcommand{\Fbasesing}[2]{\Genbase{#1}{#2}{\Fsing}}

\newcommand{\basischar}{\scalebox{.96}{${\mathfrak B}\mkern -1mu$}}
\newcommand{\spanchar}{\scalebox{.96}{${\mathfrak S}\mkern -1mu$}}

\newcommand{\basisgen}[4]{
    \IfEqCase{#3}{
        {0}{\Genbase{#1}{#2}{#4}}
        {1}{\Genbaseopp{#1}{#2}{#4}}
        }
    } 

\newcommand{\basis}[2]{\basisgen{#1}{#2\hspace{-1ex}}{0}{\basischar}}
\newcommand{\basisopp}[2]{\basisgen{#1}{#2\hspace{-1ex}}{1}{\basischar}}
\newcommand{\bspan}[2]{\basisgen{#1}{#2\hspace{-1ex}}{0}{\spanchar}}
\newcommand{\bspanopp}[2]{\basisgen{#1}{#2\hspace{-1ex}}{1}{\spanchar}}

\newcommand{\basisp}[1]{\basis{#1}{+}}
\newcommand{\basisn}[1]{\basis{#1}{-}}
\newcommand{\basispopp}[1]{\basisopp{#1}{+}}
\newcommand{\basisnopp}[1]{\basisopp{#1}{-}}

\newcommand{\bspanp}[1]{\bspan{#1}{+}}

\newcommand{\bcdot}{\mkern -1.8mu\centerdot\mkern -1.8mu}

\newcommand{\expsetZ}{\scalebox{.45}{$\star$}}
\newcommand{\expsetchar}{\mathcal E}
\newcommand{\expset}[1]{\expsetchar^{\expsetZ}\mkern -1.2mu (#1)}
\newcommand{\expsetmax}{\expsetchar^{\expsetZ}_\bullet\mkern -1.2mu}
\newcommand{\expsetup}[1]{\expsetchar(#1)} 
\newcommand{\expsetupmax}{\expsetchar_\bullet} 

\newcommand{\NexpSet}[1]{{V\mkern -5mu}_{#1}}
\newcommand{\NexpExpSet}[1]{V^{[#1]}}
\newcommand{\NexpExpSetij}[1]{V^{[#1]}}
\newcommand{\NexpExpCompSet}[1]{ \bar V^{]#1[}}
\newcommand{\NexpExpCompSetij}[1]{\bar V^{]#1[}}

\newcommand{\restr}{\iota}

\newcommand{\Lexpset}[1]{(\ellfun, #1)}
\newcommand{\Lmaxexpset}[1]{[0,\ellfun)_{#1}}
\newcommand{\Lminexpset}[1]{[\ellfun,\infty)^{#1}}

\newcommand{\transp}{{\mkern -1mu \scalebox{.63}{$\mathsf T$}}}

\newcommand{\GL}{\mathrm{GL}}
\newcommand{\SO}{\mathrm{SO}}

\newcommand{\SOFBchar}{\boldsymbol{\phi}}
\newcommand{\SOFB}[6]{
\setlength{\extrarowheight}{-2mm}
\SOFBchar\mkern -2mu
\left(
\begin{NiceArray}{c@{\;}c@{\;}c@{\;}}
#1&#3&#5\\
#2&#4&#6
\end{NiceArray}\!\!
\right)
\setlength{\extrarowheight}{0mm}
}

\newcommand{\Borelsof}{\mathrm{SO}_5^\geqzero}
\newcommand{\Borelsofopp}{\mathrm{SO}_5^\leqzero}
\newcommand{\Unipsof}{\mathrm{SO}_5^+}
\newcommand{\Cartsof}{\mathrm{SO}_5^0}

\newcommand{\cofun}[1]{\hat #1}
\newcommand{\cofuntr}[1]{\bar #1}
\newcommand{\cofuntrpar}[2]{\cofuntr{#1}_{#2}}

\newcommand{\cogen}[1]{\tilde #1}

\newcommand{\Funsofb}{\mathbb C[\Borelsof]}

\newcommand{\SOFAugIdchar}{\mathcal{I}^{\scalebox{.55}{$\mathrm{SO}$}}_5}

\newcommand{\AbstMatSof}[1]{\mathcal M^{}_{#1}}

\makeatletter

\renewcommand\thepart{\@Roman\c@part}%
\renewcommand\part{%
   \if@noskipsec \leavevmode \fi
   \par
   \addvspace{6.7ex}%
   \@afterindentfalse
   \secdef\@part\@spart}
\def\@part[#1]#2{%
    \ifnum \c@secnumdepth >\m@ne
      \refstepcounter{part}%
      \addcontentsline{toc}{part}{Part~\thepart.\ #1}%
    \else
      \addcontentsline{toc}{part}{#1}%
    \fi
    {\parindent \z@ \raggedright
     \interlinepenalty \@M
     \normalfont
     \ifnum \c@secnumdepth >\m@ne
       \centering\large\scshape \partname~\thepart.%
       \hspace{1ex}%
     \fi%
     \large\scshape #2%
     \markboth{}{}\par}%
    \nobreak
    \vskip 4.7ex
    \@afterheading}
  \def\@spart#1{
  \refstepcounter{part}%
  \addcontentsline{toc}{part}{#1}%
    {\parindent \z@ \raggedright
     \interlinepenalty \@M
     \normalfont
     \centering\large\scshape #1\par}%
     \nobreak
     \vskip 4.7ex
     \@afterheading}
\renewcommand*\l@part[2]{%
  \ifnum \c@tocdepth >-2\relax
    \addpenalty\@secpenalty
    \addvspace{10.75em \@plus\p@}%
    \begingroup
      \parindent \z@ \rightskip \@pnumwidth
      \parfillskip -\@pnumwidth
      {\leavevmode
       \normalsize \bfseries #1\hfil \hb@xt@\@pnumwidth{\hss #2}}\par
       \nobreak
       \if@compatibility
         \global\@nobreaktrue
         \reverypar{\global\@nobreakfalse\reverypar{}}%
      \fi
    \endgroup
  \fi}

\def\l@subsection{\@tocline{2}{0pt}{2pc}{6pc}{}}
\makeatother

\address{Department of Mathematics, University of California, Riverside
\newline
\indent Department of Mathematics, Michigan State University}
\email{mrhmath@proton.me} 
\address{Department of Mathematics, The Ohio State University}
\email{kerler.2@osu.edu}

\begin{document}
\title[Hopf Ideals, Integrality, and Automorphisms  of Quantum Groups at Roots of 1]
{Hopf Ideals, Integrality, and Automorphisms \\ of Quantum Groups at Roots of 1}
\bigskip
\author{Matthew Harper  \ \ and \ \  Thomas Kerler}
\allowdisplaybreaks

\begin{abstract}
We consider skew-commutative subalgebras in Drinfeld-Jimbo quantum groups at a root of unity $\zeta$ generated by primitive power elements. We classify the centrality and commutativity of these skew-polynomial algebras depending on the Lie type and the order of $\zeta$ modulo 8. We describe Hopf ideals in the quantum group induced from ideals in these subalgebras, including the non-commutative cases.

Among these, we construct and analyze a family of Hopf ideals that depend on the choice of an element in the Weyl group. We show that they arise naturally both in the construction of (partial) $R$-matrices and as vanishing ideals of Bruhat subgroups. Specialization to the maximal element yields a rigorous construction of restricted quantum groups as pre-triangular Hopf algebras, independent of any choices.

Our treatment also includes even orders of $\zeta$, non-simply laced Lie types, and minimal ground rings. Consequently, we extend some results of De Concini-Kac-Procesi, whose work focuses on odd orders of $\zeta$, which forces the subalgebra to be strictly central, and complex ground fields. This includes the identification of the subalgebras for Lie types $\mathsf{A}_n$ and $\mathsf{B}_2$ with the coordinate rings of associated algebraic groups in the commutative cases, even if $\zeta$ has even order. Our descriptions are computationally explicit and do not utilize Poisson structures. 

As technical preparations, we discuss PBW bases over minimal rings, dependencies on choices of convex orderings, as well as various new constructions of, and relations among, automorphisms on quantum groups. The latter include formulae for the Garside element in the Lustzig-Artin group action and the family of Che-transformations.

\end{abstract}

\maketitle

\begin{adjustwidth}{50pt}{50pt}
\begin{spacing}{.75}
{\footnotesize
\tableofcontents
}
\end{spacing}
\end{adjustwidth}
\section{Introduction}
One of the original motivations for this monograph was to provide accessible constructions and computational tools for quantum groups at roots of unity. Although initially aimed at practitioners engaged in newer developments in quantum topology, the topics and theorems presented here focus exclusively on the quantum algebra side, introducing several new results and perspectives. Keeping the intended audience in mind, algebraic background material will be selectively introduced.

Broadly speaking, a quantum group $U$ is a Hopf algebra associated to a simple Lie algebra $\mathfrak g$ over a ground ring $\Zgen$. In the original definitions by Drinfeld and Jimbo \cite{Dri87,Ji85} $\Zgen$ is naturally given as $\mathbb Q(q)$, where $q$ is viewed as a formal ``deformation parameter'' relating $U$ to the standard universal enveloping algebra of $\mathfrak g$. In many variants of the original definition $\Zgen$ can be chosen as a subring of  $\mathbb Q(q)$ or an extension of scalars of such a subring. Analogous to the universal envelope of $\mathfrak g$, any relevant version of $U$ has a decomposition $U=U^+\,U^0\,U^-$ into subalgebras with generating sets typically of the form $\{E_i\}$, $\{K_i\}$, and $\{F_i\}$, respectively, where the indices enumerate a set of simple roots of $\mathfrak g$.

Quantum groups play a central role in low-dimensional and quantum topology thanks to their associated $R$-matrices, which give rise to solutions of the Yang Baxter equation. The latter yield representations of braid groups and invariants of knots, links, and 3-manifolds. Early 
constructions of 3-manifold invariants and TQFTs, inspired by quantum field theories in physics, 
start from a Hopf algebra $H$ \cite{Hennings,KerlerTowards}, its representation theory \cite{RT}, 
or, more generally, certain types of tensor categories \cite{TuraevBook,KL}. An important requirement 
for these constructions is that $H$ is effectively finite-dimensional, to assure the existence of integrals, well-defined and finite Kirby colors, or certain coends (which are all essentially equivalent). 

The so-called \emph{small quantum groups} provide the quintessential examples of finite dimensional Hopf algebras in our context. Lusztig \cite{lu90b,lu} obtains these by first defining a divided power algebra $\LUq$ version of $U$ over $\Zgen=\mathbb Z[q,q^{-1}]$, and then extending this ring to $\mathbb Z[\zeta]$, where $q$ is specialized to a primitive $\kay$-th root of unity $\zeta$. The small quantum group appears then as the subalgebra generated by the original $\{E_i,F_i,K_i\}$ in $\LUq$. The construction is, thus, independent of further choices and may also be formulated as the kernel of a Frobenius homomorphism. An important caveat is that $\Spec(U^0)$ is, of course, finite and discrete for these algebras, a feature that is not easily circumvented in this construction due to extra generators required in $\LUq^0$.

More recently, invariants constructed from quantum groups with \emph{continuous} $\Spec(U^0)$ have gained considerable interest as they establish connections to classical topological invariants. Typically, parameters associated with $\Spec(U^0)$ are identified with formal variables that appear in algebraic topology constructions. For generic $q$, one obtains $R$-matrix representations of braid groups for which the continuous weights of the $K_i$ correspond to deck transformation parameters of homological representations \cite{JK11, Ito,Anghel,Martel}.

At roots of unity $U$ admits finite-dimensional representations and a rigidity structure, while retaining continuity of $\Spec(U^0)$. Invariants of knots and links are obtained by extending to modified quantum traces on the braid representations derived from $R$-matrices.
In this setting of non-semisimple $U$ and modified traces,
the continuous parameters supplied by $\Spec(U^0)$ often imply that the quantum invariants   
subsume 
classical invariants
as special cases. 
Among the  first such examples is the Alexander polynomial, for which highest weights yield the 
indeterminate variable \cite{Mu92,Mu93,oh}. For the more involved Akutsu-Deguchi-Ohtsuki (ADO) invariant \cite{ADO}, the continuous weights have similar homological interpretations \cite{ItoAlexander,AnghelADO}.

All the examples mentioned thus far assume the simplest Lie type $\mathfrak g=\mathfrak{sl}_2$\,. A higher rank example, leading to non-commutative properties of the knot invariant, 
has been studied in \cite{Har20}. In 
\cite{CGP,BCGP16}  quantum groups at roots of unity (in rank one) are used to define infinitely graded categories from which 3-manifold invariants and TQFTs are constructed. These, in turn, extend both the Reshetikhin-Turaev (RT) invariants and the classical Turaev-Reidemeister torsion. The invariant depends on a 1-cohomology class $\omega\in H^1(M,\mathbb{C}/\mathbb{Z})$ whose coefficients stem again from the continuous spectrum of $K$. Specifically, the torsion map assigns the meridian homology generators of a surgery link to $\zeta^\lambda$, where $\zeta$ is a fourth root of unity and $\lambda$ a complex weight. Constructions for more general categories, such as those in \cite{Renzi}, are, thus, likely candidates to admit similar connections to more involved classical geometric invariants. 

The \emph{restricted quantum groups} underlying these invariants may be characterized by the conditions that $U^-$ and $U^+$ are finite-dimensional, that $U^0=\Zgen[\{K_i^{\pm 1}\}]$, and that there is at least a quasi-$R$-matrix. Their construction typically starts from the specialization of the full (unrestricted) quantum to a root of unity $\zeta$, sometimes also called the De Concini-Kac-Procesi (DCKP) algebra over $\mathbb Q(\zeta)$. The allowed restrictions to subrings $\Zgen$ in $\mathbb Q(\zeta)$ heavily depend on renormalizations of generators and the condition that the algebra is stable under the Artin group action defined by Lusztig.

The restricted quantum group, for some given $\mathfrak g$ and $\zeta$, is then commonly written as the quotient of the DCKP algebra by an ideal generated by a collection of elements $\Epw_\alpha=E_\alpha^{\ell_\alpha}$ and $\Fpw_\alpha=F_\alpha^{\ell_\alpha}$, which we refer to as \emph{primitive power generators}. Here $\alpha$ ranges over the set of positive roots $\proots$ and $E_\alpha$ is defined via the Artin group action, depending on the choice of a convex ordering of $\proots$. The exponent $\ell_\alpha$ is the order of $\zeta^{(\alpha|\alpha)}$, thus depending on the length of $\alpha$. 

Although it is clear that this definition yields continuous $\Spec(U^0)$ as desired, it does not immediately imply that the constructed quotient is indeed a Hopf algebra with finite-dimensional $U^\pm$. To illustrate the involved subtleties, assume that $\mathfrak g$ is of type $\LT{B}$, $\LT{C}$, or $\LT{F}$ and the order $\kay$ of $\zeta$ is a multiple of 4. In this case, there exists \emph{no} (uniform) $\ell$ for which $\{E_\alpha^\ell, F_\alpha^\ell\}$ generates a Hopf ideal. 
Hence, the quotient fails to be a Hopf algebra for any $\ell$. 

A basic computation shows that the $\Epw_i=\Epw_{\alpha_i}$ are \emph{skew-primitive} for simple roots $\alpha_i$\,, meaning $\Delta(\Epw_i)=\Epw_i\otimes L_i+ 1\otimes \Epw_i$ with $L_i=K_i^\elli$ group-like. One easily infers from this that the set  $\{\Epw_i,\Fpw_i: i=1,\ldots,n\}$ indeed generates a Hopf ideal. However, the quotient algebra by this ideal yields \emph{infinite} dimensional $U^\pm$ in ranks greater than 1. Finally, for non-simple roots $\alpha\in\proots\setminus\sroots$, the $\Epw_\alpha$ are far from skew-primitive. 
Consequently, the ideal generated by an arbitrary subset of these elements is generally also \emph{not} a Hopf ideal.

Computationally explicit and self-contained proofs that the $\Epw_\alpha$ and $\Fpw_\alpha$ generate Hopf ideals for all Lie types and roots of unity tend to be difficult to locate in the existing literature. For the $\LT{A}_n$ Lie type one such treatment can be found in \cite{BW04}, relying on recursions that are considerably more complicated for other Lie types. For $\zeta$ of odd order $\kay$, respective statements can be inferred indirectly from the differential methods developed in \cite{dck92}. 
Our initial aim at the onset of this project has been the clarification of these points, which 
has led us to the exploration of numerous other directly related questions.

One of them is, which subsets of the $\Epw_\alpha$ generate Hopf ideals. The answer we give here is in terms of the inversion sets $\descroots s =\proots\cap s(\nroots)$ (and unions thereof), where $s\in \Weyl$ is an element of the Weyl group. The associated ideals $\Zaugidealhat{s}$ also occur naturally in the iterative construction of quasi-$R$-matrices, in the sense that intertwining relations for partial $R$-matrices hold on the (partial) quotients by the $\Zaugidealhat{s}$. Like the generators themselves, the $\Zaugidealhat{s}$ ideal \textit{a priori} depends on the choice of a total convex ordering or, equivalently, a reduced word presentation of $s$. We show independence of the latter so that, in particular, the fully restricted quantum groups are also independent of these choices and admit unique quasi-$R$-matrices.

Another source of classical parameter spaces inherent to quantum groups at roots of unity and, thus, potential starting points for quantum invariants extending classical invariants, are Lie groups associated to central subalgebras. Assign to a connected, reductive algebraic group $G$ a solvable group $G^\PLexp<G\times G$ of the same dimension by choosing Borel subgroups $B^+<G$ and $B^-<G$ for which $H=B^+\cap B^-$ is a maximal torus in $G$ and denote 
\begin{equation}\label{eq:DefPoissonLie}
G^\PLexp
\,=\,\ker\!\left(m\circ(\pi^-\times\pi^+):B^-\times B^+\rightarrow H\right)
\,=\,\left\{(g,h)\in B^-\times B^+\,:\,\pi^-(g)\cdot \pi^+(h)=1\right\}\;,
\end{equation}
where $\pi^{\pm}:B^{\pm}\rightarrow H$ are the usual abelianization maps. This group naturally appears in various contexts, such as the study of flag varieties, compactifications of algebraic groups, and Poisson Lie groups.  

The algebra $\Zsubalgchar$ generated by the $\Epw_\alpha$\,, $\Fpw_\alpha$\,, and $L_i$ is in the center of the DCKP algebra for {\em odd} orders $\kay$ and is complementary to the generic Harish-Chandra center. In \cite{dck92} De Concini, Kac, and Procesi establish an isomorphism of Poisson groups between $\Spec(\Zsubalgchar\otimes \mathbb{C})$ and $G^\PLexp$, or, equivalently, a Poisson Hopf isomorphism between $\Zsubalgchar\otimes \mathbb{C}$ and the coordinate ring of $G^\PLexp$. The construction relies on differential Poisson structures that reduce the proofs to the comparison of skew-primitive elements for simple roots and, therefore, side-steps explicit assignments of the $\Epw_\alpha$ generators for non-simple roots and comparisons of those coproducts.

For even $\kay$, the algebra $\Zsubalgchar$ is generally only a \emph{skew-commutative} (commutative up to sign or unit factor) Hopf algebra. In \cite{tan16} Tanisaki defines the Frobenius center as the intersection of $\Zsubalgchar$ with the center of $U$, which is generally no longer a Hopf subalgebra. He shows it is  isomorphic to the coordinate ring of the quotient $G^\PLexp$/$\mathscr G$ as an algebraic variety for a finite group $\mathscr G$, utilizing similar techniques as in \cite{dck92}.

In this paper, we provide Hopf algebra isomorphisms from $\Zsubalgchar\otimes \mathbb{C}$ to $\mathbb{C}[G^\PLexp]$ with \emph{explicit} assignments of \emph{all} $\Epw_\alpha$ generators for Lie type $\LT{A}_n$ with $\kay\not\equiv 2 \mod 4$ and type $\LT{B}_2$ for any $\kay$. Our computations directly compare the coalgebra structures 
of $\Zsubalgchar$
with those of $\mathbb{C}[\mathrm{GL}(n,\mathbb{C})^\PLexp]$ and $\mathbb{C}[\mathrm{SO}(5,\mathbb{C})^\PLexp]$, and do not rely on Poisson structures. The concrete forms of these identifications make them more suitable for applications in quantum topology.

The congruence restriction above for $\LT{A}_n$ is imposed solely to assure that $\Zsubalgchar$ is commutative. For $\kay\equiv 2 \mod 4$, the full $\Zsubalgchar$ algebra is isomorphic to 
a non-commutative Artin-Schelter-Tate central extension \cite{AST91} of classical matrix algebras. For this case, we suggest modifications of the Hopf algebra structures on $\mathbb{C}[G^\PLexp]$ via $\mathscr G$ weights. We provide a concrete description of $\mathscr G$ and the isomorphism of $\mathbb C[G^\PLexp/\mathscr G]$ with the Frobenius center, as proven abstractly in  \cite{tan16}.  

Our calculations for the $\LT{B}_2$ algebra illustrate further that, for $\kay$ a multiple of four, the group $G^\PLexp$ needs to be replaced by $\breve G^\PLexp$. Here, the Langlands dual group $\breve G$ is the algebraic group associated to the coroot system $\breve \roots$, reflecting the fact that long and short roots are exchanged for the primitive power generators $\Epw_\alpha$\,.
In both examples, we also show the ideals $\Zaugideal{s}$ in $\Zsubalgchar$ encountered above are 
isomorphically mapped to the vanishing ideals on the Bruhat subgroups. The latter are defined here  as
\begin{equation}\label{def:BorelDef}
    B(s)=B\cap \dot sB\dot s^{-1}\;,
\end{equation}
where $\dot s$ is a lift from the geometric Weyl group to $G$.

An additional important feature of many quantum invariants is that they are valued in integral polynomial rings, since one may expect integer coefficients to relate to topological information. For example, RT-type invariants of homology spheres are valued in $\mathbb{Z}[\zeta]$ with coefficients giving rise to so-called finite type invariants of homology 3-spheres \cite{Oh95,Oh96,oh}, which are, in turn, 
often related to geometric invariants. 
In a more direct approach, the Hennings invariant of a 3-manifold is valued in the ground ring of the Hopf algebra chosen for the construction, assuming adequate properties of (co)integrals \cite{Hennings,GNPHennings22,GeerGcoalg23}.

Among the different versions of quantum groups, we consider here   the original Jimbo normalization with generators $\{E_i, F_i, K_i\}$ as well as elements obtained from these via the Artin group action. We write $\UqQ$ for the rational form of the quantum group with these generators. It is easy to see that the minimal ground ring for the Hopf subalgebra $U^\geqzero$ generated by $\{E_i,K_i\}$ must contain $\mathbb{Z}[q,q^{-1},[\edgenum]^{-1}]$. Here $[n]$ indicates the usual quantum number and $\edgenum=1$ for Lie types $\LT{ADE}$\,, $\edgenum=2$ for types $\LT{BCF}$\,, and $\edgenum=3$ for type $\LT{G}$\,.

We prove that a PBW theorem holds for this minimal ring with respect to 
appropriate monomial bases.
For the full quantum group with the described standard generators $\{E_i, F_i, K_i\}$ the  ring over which the algebra is defined also includes $(q-q^{-1})^{-1}$. Minimal ground rings for quantum algebras with differently normalized generators are discussed as well.

Throughout this project, (anti)automorphisms of quantum groups are an important tool for generating and relating monomial basis elements. In addition to the Artin group action, symmetries are generated by numerous (anti)involutions as well as an infinite family of so-called \emph{\Kscaleword transformations} $\Kscale{u}{\mathsfit h}$. The latter act on generators via multiplication with Cartan generators and scalars and appear in numerous calculations, particularly, relations and expressions involving the antipode.  

We carefully organize the catalog of relations between all types of automorphisms, allowing us to derive explicit expressions for other automorphisms, most notably, the action of the Garside element with respect to the Artin group action. The study of these automorphisms serves to provide insights into symmetries of quantum invariants. 

\subsection{Organization and Summary of Main Results
}
Section~\ref{sec:Weyl} introduces basic background on spherical Artin and Coxeter groups as well as their actions on roots systems. In Sections~\ref{subsec:descroots}, \ref{subsec:ComplWords}, and \ref{subsec:LettSpecElem} we present numerous properties of inversion sets $\descroots s$ and discuss relations between reduced word expressions, weak Bruhat orders, as well as convex orderings of roots. We introduce root sums for inversion sets and a relation $v \wcomplrel w$ for reduced words, which is later used to describe involutions of basis elements. The notion of symmetric root lattice maps, discussed in Section~\ref{subsec:rootsymmaps}, prepares the definition of \Kscaleword transformations. We further compile formulae for rank 2 root systems in Section~\ref{subsec:rank2coxsys}, as they will be frequently used in later proofs.

Section~\ref{sec:quantumgen} covers the basic definitions of the Drinfeld-Jimbo quantum groups, which will be the focus of the paper as the most commonly used version.  The section includes discussions of their various minimal ground rings, extensions of scalars, basic Hopf algebra structures, and pointedness. Specifically, the subalgebras $U_q^\geqzero$ generated by $E_i$ and $K_i$ are defined over $\Zqqn{\dpone}$\,, which is the extension of $\mathbb Z[q,q^{-1}]$ by $([\edgenum]!)^{-1}\,$, where $[n]!$ is the usual quantum factorial and $\edgenum\in\{1,2,3\}$ depends on the Lie type as above. The full algebra $\Uq$ is defined over $\Zqqvn{\dpone}$\,, which is obtained by further extending  $\Zqqn{\dpone}$ by $(q-q^{-1})^{-1}$. 

In addition to the standard integer root lattice grading of quantum groups, we introduce in Section~\ref{subsec:gradings} an independent $\mathbb{F}_2=\mathbb{Z}/2\mathbb{Z}$ root lattice grading, which is useful for the derivation and structure analysis of the skew-commutative subalgebras at roots of unity. We assign to each $\roots$-symmetric lattice map $\mathsfit h$ and homomorphism $u:\mathbb Z^{\sroots}\rightarrow \Zgen^{\star}$ a \Kscaleword automorphism $\Kscale{u}{\mathsfit h}$ in Section~\ref{subsec:Kscale}. We also introduce in that subsection (anti)involutions, such as $\Cartaninv$, $\Kinvaut$, $\Qinvaut$, $\Cartanaut$, and $\Kconaut$, and list numerous relations between them, including a formula for the antipode.

Lusztig’s action of the Artin group $\mathscr A$ on $U$ and its relations with the other automorphisms is explored in Section~\ref{subsec:LATM}. Our convention differs slightly as we use instead the automorphisms $\Tinv_b=T_{\iota(b)}\,$, where $b\in\mathscr{A}$ and $\iota$ the inversion automorphism on $\mathscr{A}$.

In Section~\ref{sec:PBW} we construct families of bases for $U$, starting in Section~\ref{subsec:genwords} with the definition of elements $E_w$ and $F_w$ assigned to any reduced word $w$ of the underlying Weyl group. There we show that $E_w=E_k$ if the standard degree of $E_w$ is a simple root $\alpha_k$\,, extending the proof in \cite{Lu90a} to all Lie types. As an application, we derive in Section~\ref{subsec:GarsAut+DictGens} the formula
\[
\Tinv_{\garside}=\DynkInv\circ\Cartanaut\circ\Kinvaut\circ S
\]
for the Lusztig-Artin action $\Tinv_{\garside}$ of the Garside element $\garside\in\mathscr{A}$ in terms of the automorphisms from Section~\ref{sec:gradingauts} as well as the involution $\DynkInv$ given by symmetries of the Dynkin diagram. In the remainder of the subsection, we provide a dictionary between generators in Lusztig’s and our convention in the rank 2 case via the $\Kinvaut$-involution. Aside from facilitating the translation of results and formulae, it serves to show that later definitions of ideals and quotients are independent of which of the two $\mathscr{A}$-actions is chosen.  

Section~\ref{subsec:Emonomials} develops a formalism for PBW-style monomial elements $\Ebase{w}{\psi}$ and $\Ebaseopp{w}{\psi}\,$, depending on reduced words $w$ as well as an exponent functions $\psi: \descroots w\rightarrow \nnN$\,. We derive explicit formulae for the actions of automorphisms from Section~\ref{sec:gradingauts} and introduce notations for spanning sets and spans. Various commutation relations between these monomials are gathered in Section~\ref{subsec:Ecommrels} for later use. 

In Sections~\ref{subsec:orderings} and \ref{subsec:PBWspec} we adapt Lusztig’s proof of a PBW basis for the divided power algebras from \cite{Lu90a,lu90b} to produce a PBW theorem for $U_q^+$ and $U_q$  over respective minimal extensions of $\mathbb Z[q,q^{-1}]$ for one particular convex ordering. In Section~\ref{subsec:spanorder} we state criteria for subsets of exponent functions from which one can deduce  that their respective spans are independent of the choice of a reduced word expression for a given Weyl element. The results are combined in Section~\ref{subsec:PBWmain} to yield Proposition~\ref{prop:PBW2} and Theorem~\ref{thm:mainPBW}. They are restated in less formal terms as follows.

\begin{Thm}
    Assume $\mathfrak g$ to be of any simple Lie type and let $z$ be any reduced word of maximal length. 
Then $\{\Ebase z \psi\}$ and $\{\Ebase z \psi \cdot K^\mu\}$ are $\Zqqn{\dpone}$-bases of $U_q^+$ and $U_q^\geqzero\,$, respectively. 

Moreover, the set $\{\Ebase z \psi \cdot \Fbase{z}{\psi'}\cdot K^\mu\}$ is a $\Zqqvn{\dpone}$-basis of the Hopf subalgebra $U_q\subset \UqQ$\,.
\end{Thm}  
\noindent The exponents are integer functions $\psi,\psi'\in\nnN^{{\proots}}$ and $\mu\in\mathbb Z^{{\sroots}}$ on the positive and simple roots, respectively. Analogous statements hold for bases obtained from the monomials $\Ebaseopp z \psi$ defined with respect to the reverse direction of multiplication. 
We discuss a few less commonly used normalizations of generators that yield smaller ground rings. 

In Section~\ref{sec:rootsofunity} we turn to the root of unity case, beginning in Section~\ref{subsec:QnumsRingsRo1} with the introduction of various rings related to $\bbz[\zeta]=\Zz$\,, notation for various orders, and formulae for multinomial coefficients. In Section~\ref{subsec:CommPrimGen} we derive skew-commutation relations involving the primitive power generators $\Epw_w=E_w^{\ell_w}$, $\Fpw_w=F_w^{\ell_w}$, and $\Kpw_i=K_i^{\ell_i}$ in terms of special functions $(\proots)^2\rightarrow \{\pm 1\}$. The discussion is extended in Section~\ref{subsec:Z=SignPolyn} to monomials, where we also introduce unital algebras such as $\Zsubalgchar^+_w$\,,  $\Zsubalgchar^\geqzero_w$\,, and $\Zsubalgchar_w$\,, generated over $\Zzn{\dpone}$ or $\Zzvn{\dpone}$ by $\Epw_u$\,, $\Fpw_u$\,, and $\Kpw^\mu$ with $u\leqRB w$. We conclude the section with a full classification of Lie types and $\kay$ congruences for which these algebras are central as subalgebras of $\Uz$ and for which they are commutative. For $\utypechar\in\{+,-,0,\geqzero, \leqzero\}$ and a given reduced word $w$, they give rise to the ideal $\Zaugideal{w}^{\utypechar}$ in a maximal $\Zsubalgchar_{\bullet}^\utypechar$ generated by the augmentation ideal $A_w^+$ of $\Zsubalgchar^+_w$\,.

For each rank 2 type $\LT{A}_2$\,, $\LT{B}_2$\,, $\LT{G}_2$ there are two possible convex orderings, each of which is represented by a maximal reduced word. In Section~\ref{subsec:Mho-Invar-Z}  we compute explicit formulae that express the $\Epw_w$ from one ordering in terms of those of the other ordering, with the exception of words of length 3 and 4 in type $\LT{G}_2$\,. Drawing on the results in previous section, the implied closure under the $\Kinvaut$-action entails that the algebras above do not depend on the word $w$ representing an element $s$ of the Weyl group, but only on $s$ itself. Summarizing Theorem \ref{thm:Zwordindep}, Corollary \ref{cor:Zorderindep}, Theorem \ref{thm:ZTinvInvar}, and Corollary \ref{cor:KidealWordIndep} this observation results in the following.

\begin{Thm}
    Assume a Lie type different from $\LT{G}_2$ and $\zeta$ of any order $\kay\not\in\{1,2,\edgenum,2\edgenum\}$. Suppose $w$ is a reduced word and $s$ the Weyl group element it represents. For $\utypechar\in\{+,-,0,\geqzero, \leqzero\}$ we have the following. 
    \begin{enumerate}[label=\roman*), leftmargin=12mm,] 
        \item 
        The algebra $\Zsubalgchar^\utypechar_w$ depends only on $s$.
        \vspace*{1.3mm}\item 
        The maximal skew-commutative algebra $\Zsubalgchar_\bullet^{\utypechar}$ is defined independently of the choice of ordering.
         \vspace*{1.3mm}\item
         $\Zsubalgchar_\bullet$ is invariant under the Artin group action.
         \vspace*{1.3mm}\item
         The augmentation ideal $\Zaugideal{w}^{\utypechar}$ in  $\Zsubalgchar_\bullet^\utypechar$ is only dependent on $s$.
    \end{enumerate}
\end{Thm}

In Section~\ref{subsec:Z_Ideals}  we consider the action of an elementary abelian 2-group $\Commsigngrp$ on $\Zsubalgchar_\bullet$ and induce for any $\Commsigngrp$-stable ideal $J$ in $\Zsubalgchar_\bullet$ a two-sided ideal $\widehat J$ in $\Uz$\,. We introduce notations for special and extended cases of the resulting $\Zaugidealhat{s}$ and discuss various examples. 
 
Sections \ref{sec:Ideals_An=GLn} and \ref{sec:Ideals_B2=SO5} are devoted to the identifications of the $\Zsubalgchar_\bullet$ algebras with coordinate rings of algebaric groups  in the basic $\LT{A}_n$ type as well as the doubly laced $\LT{B}_2$ case. Section~\ref{subsec:AlgGrps-basics} collects background on algebraic groups and their coordinate rings, with emphasis on the correspondence between subgroups and Hopf ideals, and a discussion of relevant elementary examples.

Generators, commutation relations, and coproducts are computed for the quantum algebras for $\mathfrak g=\mathfrak{gl}_n$ and $\mathfrak g=\mathfrak{sl}_n$ in Section~\ref{subsec:CoalgAn}. Specializing to a root of unity of order $\kay$, we identify in Section~\ref{subsec:Zn=CAn} the respective $\Zsubalgchar_\bullet^\geqzero$ algebra with the AST algebra over $\Zz$ \cite{AST91}. For $\kay\not\equiv 2 \mod 4$ the extension is trivial, resulting in a Hopf algebra isomorphism to the coordinate ring of the group of upper triangular matrices over $\mathbb k=\Qz$ or $\mathbb k=\mathbb C$. In Section~\ref{subsec:AnBruhat} the full $\Zsubalgchar_\bullet$ algebra for $\mathfrak{gl}_n$ over $\mathbb C$ is identified with the coordinate ring of $\mathrm{GL}(n,\mathbb{C})^\PLexp$ and the augmentation ideals $\Zaugideal{s}$ with vanishing ideals of Bruhat subgroups $B(s)$ defined in \textit{loc. cit}. The section concludes with a discussion of potential interpretations of the skew-commutative case. 

We start our treatment of the $\LT{B}_2$ case with a complete description of the Hopf algebra structure of $\Zsubalgchar_\bullet^\geqzero$ in Section~\ref{subsec:Coalg-ZB2}, deferring the detailed computations of the coproducts to Section~\ref{subsec:ProofCoprodB2} and Appendix \ref{app:powerformula}. The required replacement of the root system by the coroot system, as mentioned above, is achieved by the choice of an isomorphism between the $\LT{B}_2$ and $\LT{C}_2$ algebras.  

In Section~\ref{subsec:BorelCordSO5} we use a specific complex parametrization of $\mathrm{SO}(5,\bbc)^\geqzero$ to derive formulae for the coalgebra structure of its coordinate ring. Using its Lie-Kolchin form, the vanishing ideals of the Bruhat subgroups are readily identified. Comparing the coalgebra structures obtained in Section~\ref{subsec:ProofCoprodB2} and \ref{subsec:BorelCordSO5}, we obtain Hopf algebra isomorphisms between the central subalgebra and the coordinate ring in Section~\ref{subsec:IsomHA-B2}. The involved Hopf algebras fit into parametrized families of abstract weighted Hopf algebras, which are classified in the same sections. 

The theorem below summarizes the main results from Lemma \ref{lem:fullZ=Qn}, Theorem \ref{thm:AB=Ks}, and Theorem \ref{thm:Z=CSO5}. The field $\mathbb k$ can be chosen as either $\Qz$ or $\bbc$. The notation $G^\PLexp$, $B$, and $B(s)$ for an algebraic group $G$ is as in \eqref{eq:DefPoissonLie} and \eqref{def:BorelDef}. The isomorphisms are given as explicit assignments of polynomial generators. 

\begin{Thm}
    Suppose $\Zsubalgchar_\bullet$ is the full commutative subalgebra for $\mathfrak g = \mathfrak{gl}_n$ with $\kay \not\equiv 2 \mod 4$ or $\mathfrak g$ of type $\LT{B}_2$ for any $\kay$. Let $G=\GL(n,\mathbb k)$ or $G=\SO(5,\mathbb k)$ respectively. Then 
    \begin{enumerate}[label=\roman*), leftmargin=12mm,] 
        \item 
        There is a Hopf algebra isomorphism from  $\Zsubalgchar_\bullet\otimes\mathbb k$ to $\mathbb k[G^\PLexp]$, which restricts to an isomorphism between $\Zsubalgchar_\bullet^\geqzero\otimes\mathbb k$ and $\mathbb k[B]$. 
         \vspace*{1.3mm}\item
         The isomorphism maps the augmentation ideal  $\Zaugideal{s} \subseteq \Zsubalgchar_\bullet^\geqzero $ exactly to the vanishing ideal in $\mathbb k [B]$ of the Bruhat subgroup $B(s)<B$. 
    \end{enumerate}
\end{Thm}
Section~\ref{sec:QuasiTriang} is mainly concerned with the construction of $R$-matrices and quasi-triangular structures. We review relevant algebra completions as well as Tanisaki’s axioms for pre-triangular bialgebras in Section~\ref{subsec:GenTensStr}. The algebraic properties of elementary and partial quasi-$R$-matrices for generic $q$ are discussed in Sections \ref{subsec:QuasiR_genq} and \ref{subsec:PreTriang}. We consider their relations with the Artin group action, introduce the Taniski automorphism $\taniauto$ for the pre-triangular structure, and define coefficients related to the Tanisaki-Lusztig pairings. Specialized at a root of unity, the latter is used in Section~\ref{subsec:LuszTaniPair} to define the full quotient into the restricted quantum group, where maximal ideals are identified with null spaces.

The technical core result of Section~\ref{subsec:TruncQmats} is a set of intertwining relations via truncated partial $R$-matrices between the coproduct $\Delta(x)$ and a twisted coproduct $\Delta^s(x)$, modulo ideals derived from the ideals $\Zaugidealhat{w}$ introduced in Section~\ref{sec:rootsofunity}. Here $\Delta^s(x)$ is the coproduct conjugated by $\Tinv_s$ and the underlying quantum group $\Uzn{\ell}$ is now over the ring $\Zzvn{\ell}\,$, which also includes $([\ell-1]!)^{-1}$. The rather involved computations for generators are deferred to Appendix \ref{sec:proofcoprodbase}. An important consequence of these relations is that the two-sided ideals $\Zaugidealhat{w}$ depend only on the Weyl elements represented by $w$ and are Hopf for all cases, including $\LT{G}_2$\,. The following combines the statements in Theorem \ref{thm:MainIdeals} and Proposition \ref{prop:PartialIntertw}.

\begin{Thm}\label{Thm:RIdeal}
    Let $\zeta$ be of any order $\kay\not\in\{1,2,\edgenum, 2\edgenum\}$ and $\Uzn{\ell}$ of any simple Lie type. Suppose $w$ is a reduced word representing a Weyl element $s$. Let $\Zaugidealhat{w,w}$ be the two-sided ideal generated by all $\Epw_u$ and $\Fpw_u$ with $u\leqRB w$. Then
    \begin{enumerate}[label=\roman*), leftmargin=12mm,] 
        \item 
        $\Zaugidealhat{w,w}$ is a Hopf ideal depending only on $s$ so that the quotient $\Uzn{\ell}^{[s]}$ by this ideal is a well-defined Hopf algebra over $\Zzvn{\ell}$\,.
         \vspace*{1.3mm}\item
         PBW bases of $\Uzn{\ell}^{[s]}$ are given by the set of exponents with $\psi(\alpha)<\ell_\alpha\,$ for all $\alpha\in\descroots s$. 
         \vspace*{1.3mm}\item
         For any reduced word $u$ representing some $t\leqRB s$ there is a partial quasi-$R$-matrix $\Rtprodfac u$ such that 
\[
\Delta^{[t]}(x)\cdot \Rtprodfac u\,=\,\Rtprodfac u\cdot \Delta(x)\;,
\]
where $x\in \Uzn{\ell}^{[s]}$ and $\Delta^{[t]}$ is the coproduct on $\Uzn{\ell}^{[s]}$ obtained by conjugation with $\Tinv_t$\,.
    \end{enumerate}
\end{Thm}
 
In Section~\ref{subsec:RestQRMats} we consider the specialization to the maximal length element 
$\longweyl$ in the Weyl group, defining the restricted quantum group as $\Uzn{\ell}^\redsup= \Uzn{\ell}^{[ \longweyl]}$. Using our formula for the Garside element, we identify $\Delta^{[\longweyl]}$ with the Tanisaki-Lusztig conjugate $\barDelta$ and show the maximal $\Rtprodfac z$ is independent of the chosen word. A slightly less formal statement of Theorem \ref{thm:final} is as follows, which now encompasses also the $\LT{G}_2$ case.

\begin{Thm}
    For any Lie type and root of unity as in Theorem \ref{Thm:RIdeal}, the triple $(\Uzn{\ell}^\redsup,\Rtred,\taniauto)$ is a pre-triangular Hopf algebra over $\Zzvn{\ell}$\,. Its construction is independent of any choices, it has the properties of a restricted quantum group, and it admits the actions of the Artin group and all involutions.
\end{Thm}

\subsection*{Acknowledgments}
MH was partially supported through the NSF-RTG grant
\#DMS-2135960.

\section{Root Systems and Weyl Groups}\label{sec:Weyl}

This section combines a general review of basic notions of
root systems, Coxeter groups, and Artin groups with a discussion of 
several more specialized lemmas and concepts. The latter serve as technical tools in later sections, but 
are typically not treated in standard textbooks on Coxeter groups. Examples of such topics 
are various properties of inversion sets, complementary words, as well as $\roots$-symmetric lattice maps. 

\subsection{Root Systems and Weyl Reflections}\label{subsec:rootweyl} Let $\roots$ be an irreducible root system of rank $n$ in some Euclidean space $E$. For a choice of simple roots $\sroots=\{\alpha_1,\dots,\alpha_n\}$ denote by $\proots$ and $\nroots$ the sets of positive and negative roots, with $\roots=\proots\cup\nroots\,$. Although it is not dependent on the choice of $\sroots\,$, we will frequently denote the root lattice of $\roots$ in $E$ by $\mathbb Z^{\sroots}\,$.

The Weyl group $\Weyl$ is, by definition, generated by the reflections $s_\alpha$ of $E$ along roots $\alpha\in\roots$, which are 
expressed as $s_\alpha(x)=x-\bracket{\alpha, x}\alpha$. The form $\bracket{\,\cdot\,,\,\cdot\,}$ is linear in the second entry 
and gives rise to the Cartan matrix $A$ associated to $\roots$ via $A_{ij}=\bracket{{\alpha_i,\alpha_j}}\,$. The convention for $A$ chosen here conforms with that in \cite{dck90,lu90b} but is transposed to that of textbooks such as \cite{Bou02}. Clearly,
$\Weyl$ also acts on $\mathbb Z^{\sroots}\,$ as an invariant subspace of $E$.

For a given Cartan matrix, we may find integers $d_i\in \{1,2,3\}$ that symmetrize the Cartan matrix in the sense that $\symCart_{ij}=d_iA_{ij}$ is symmetric. Accordingly, we set $d_\alpha=d_i$ if $\alpha$ has the same length as $\alpha_i$ or, equivalently, if $\alpha$ is in the $\Weyl$-orbit of $\alpha_i\,$. We will use the following convention and notation,
\begin{equation}\label{eq:def_maxd}
    \min_i\{d_i\}=1\qquad \mbox{and}\qquad \maxd=\max_i\{d_i\}\,\in\{1,2,3\}\,.
\end{equation}
Thus, in the simply laced cases \LT{ADE} we have $\maxd=1$, meaning all $d_i=1$. We have $\maxd=2$ for 
doubly laced Lie types $\LT{BCF}$ and $\maxd=3$ for type $\LT{G}\,$.
The renomalization leads to the definition of an inner form
\begin{equation}\label{eq:def:symbrack}
    \symbrack{\,\cdot\,}{\,\cdot\,}\,:\mathbb Z^{\sroots}\times \mathbb Z^{\sroots}\rightarrow\mathbb Z\,
\end{equation}
by setting $\symbrack{\alpha_i}{\alpha_j}=\symCart_{ij}=d_iA_{ij}\,$ for the basis roots vectors. It is readily verified that the Weyl group $\Weyl$ acts by isometries with respect to this form, meaning $\symbrack{s(\alpha)}{s(\beta)}=\symbrack{\alpha}{\beta}\,$ for any $s\in\Weyl$ and $\alpha,\beta\in\mathbb Z^{\sroots}\supset\roots\,$. So $\symbrack{\,\cdot\,}{\,\cdot\,}$ is a scalar multiple of the standard Euclidean inner product on $E$ for an irreducible root system $\roots\,$. Specifically, it coincides with the canonical inner product for the realizations in the appendix of \cite{Bou02} when the Lie type is
one of $\LT{ACDEG}$\,, and it is twice the canonical inner product given there if the Lie type is $\LT{B}$ or $\LT{F}\,$.
The integer from \eqref{eq:def_maxd} may thus also be understood as the lower bound on inner products of positive roots,
\begin{equation}\label{eq:Ddef}
    \maxd= \max_{i,j}\{-\symCart_{ij}\} = \max_{i,j}\{-\symbrack{\alpha_i}{\alpha_j}\}= \max\{-\symbrack{\alpha}{\beta}:\alpha,\beta\in\proots\}\,. 
\end{equation}
In these conventions we also find $\symbrack{\alpha}{\alpha}=2d_\alpha$ by choosing some $s\in\Weyl$ and $\alpha_i\in\sroots$
with $\alpha=s(\alpha_i)$ and, hence, $d_i=d_\alpha$\,. As usual, for a fixed root system, we will call $\alpha$ a {\em short} root if $d_\alpha=1$ and, in the non-simply laced case, we say $\alpha$ is a {\em long} root if $d_\alpha=\maxd>1\,$. The initial integral pairing is recovered from $\symbrack{\,\cdot\,}{\,\cdot\,}$ via
\begin{equation}\label{eq:bracketrel}
   \bracket{\,\cdot\,,\,\cdot\,}:\,\roots\times\roots\rightarrow\mathbb Z \qquad \mbox{with}\quad  \bracket{\alpha,\beta} \,=\,\frac 1 {d_\alpha}\symbrack{\alpha}{\beta}\quad \in\mathbb Z\,.
\end{equation}
For later notational convenience  we also introduce 
\begin{equation}\label{eq:Donedef}
    \dpone=\maxd+1 \in\{2,3,4\}\,.
\end{equation}
That is, $\dpone=2$ for Lie types \LT{ADE}, $\dpone=3$ for the \LT{BCF} types, and $\dpone=4$ for $\LT{G}_2$\,.

\subsection{Coxeter Systems, Reduced Words, and Bruhat Orders}\label{subsec:coxeter}
Recall that $(\Weyl,\SimRefl)$ forms a Coxeter system, where the set of generators $\SimRefl=\{s_1,\ldots, s_n\}$ consists of  the simple reflections $s_i=s_{\alpha_i}\,$. The relations for a presentation of $\Weyl$ in generators $\SimRefl$ are given by the involution $s_i^2=1$ and the Artin-Tits relations. They are given for each pair $(i,j)$ with $i\neq j$ by
\begin{equation}\label{eq-ATrels}
     {s_is_j\ldots}  =  {s_js_i\ldots} 
\end{equation}
with $m_{ij}$ reflections terms on each side with alternating indices. Here $m_{ij}=2, 3, 4, \mbox{or } 6$ depending on whether the edge number $\max(|A_{ij}|,|A_{ji}|)$ is 0, 1, 2, or 3, respectively. 

As usual, the length $\length (s)$ of an element $s\in\Weyl$ is the smallest number $k$ such that $s=s_{i_1}\ldots s_{i_k}$ can be expressed as a product of $k$ simple reflections. 
A {\em reduced word} is a formal word expression $w_{i_1}\ldots w_{i_k}$ in the letters
$w_1, \ldots, w_n\,$ such that the respective product $s=s_{i_1}\ldots s_{i_k}$
in $\Weyl$ is of minimal length $k=\len a\,$. We use the separate notation $\wordset$ for the set of reduced words, since several objects in the article will explicitly depend on words rather than just elements in $\Weyl$.

An obvious surjection $\Weylpres: \wordset\twoheadrightarrow\Weyl$ is obtained by replacing each letter $w_i$ by the respective $s_i$ and mapping the empty word $\emptyword$ to the identity element $1\in\Weyl$. We say $s\in \Weyl$ is represented by a word $w\in \wordset$ if $w$ maps to $s$ so that, by definition, $l(s)$ is the number of letters in $w$. Denote by $\mathscr A$ the {\em Artin group} associated to the same Coxeter data as $\Weyl$. That is, $\mathscr A$ shares the same generator set $\mathscr S$ and the relations in (\ref{eq-ATrels}) but lacks the relations $s_i^2=1$.   Also denote by $\mathscr A^+$ the {\em Artin monoid}, defined as the subset of elements in $\mathscr A$ that are a product of generators with positive powers.

Matsumoto's Theorem states that there is a well-defined map $\Matsec: \Weyl \rightarrow \mathscr A^+\subset \mathscr A$ such that $\Matsec\circ \Weylpres: \wordset \rightarrow \mathscr A^+$ is given as the assignment of letters $w_i$ to the respective Artin generators (see Part IV, \S 1.5 in \cite{Bou02} or \cite{Mat64}). The more practical formulation is as follows.

\begin{cor}[Matsumoto]\label{cor:matsumoto} Suppose two words in $\wordset$ represent the same element in $\Weyl$. Then they are related by a sequence of Artin relations given in \eqref{eq-ATrels}. 
\end{cor}

Given a reduced word $w=w_{i_1}\ldots w_{i_k}\in\wordset$, we write $w[a,b]=w_{i_a}\ldots w_{i_b}$ for the subword between indices $a$ and $b$ with $1\leq a\leq b\leq k$. Further, denote $w^*=w_{i_k}\ldots w_{i_1}$ the word with letters multiplied in the opposite direction so that $\Weylpres(w^*)=\Weylpres(w)^{-1}\,$. 

The \emph{weak Bruhat orders} on $\Weyl$ and $\wordset$ are defined in terms of smaller elements occurring as subwords. For our purposes, the most relevant is the weak right Bruhat order $\leqRB$\,. For two words $w,u\in \wordset$ we say $u\leqRB w$ if there exists $v\in\wordset$ such that $w=u\cdot v$ or, equivalently, if $u=w[1,j]$ for some $j$. Similarly, for two elements $s,t\in \Weyl$ we have $s\leqRB t$ if $t$ and $s$ are represented by words $w$ and $u$, respectively, with $u\leqRB w$. 

To define the left weak Bruhat order, replace $u=w[j,k]$ in the above definition. If $\longweyl\in\Weyl$ is the unique longest element, then $\longweyl$ is also the unique maximal element for any of the Bruhat orders, including $\leqRB$ (e.g., Proposition 3.1.2 in \cite{BB05}). 

Correspondingly, we write $\wordsetmax=\Weylpres^{-1}(\longweyl)$ for the set of reduced words of maximal length. The set will be used to bijectively label all convex orders on $\proots\,$. See, for example, \cite{St84} for formulae for the size and rapid growth of $|\wordsetmax|\,$.

Note that any reduced word $u\in\wordset$ can be written as the initial subword of some word in $\wordsetmax$. Explicitly, let $s=\Weylpres(u)$, set $s'=s^{-1}\longweyl$\,, and choose any $u'\in\wordset$ representing $s'$. Since $s\leqRB \longweyl$ we have $\length(u')=\length(s')=\length(\longweyl)-\length(s)=\length(\longweyl)-\length(u)$ by the same proposition in \cite{BB05}. Now, $w=u\cdot u'$ represents $\longweyl$ and, since $\length(u)+\length(u')=\length(\longweyl)$, it is a reduced expression so that $w\in\wordsetmax$. 

Similarly, the element $\garside=\Matsec(\longweyl)\in\mathscr A^+$ plays an important role as the fundamental element that renders $(\mathscr A^+,\garside)$ a Garside monoid (see, for example, Proposition~1.29 in \cite{Deh15}). In particular, Lemma~5.2 in \cite{BS72} asserts that for all 
$b\in\mathscr A^+\,$ the commutation relation 
\begin{equation}\label{eq:gardsidecomm}
    \garside\cdot b=\sigma(b)\cdot \garside 
\end{equation}
holds, where $\sigma$ is an involutive automorphism of $\mathscr A^+$. Since $\sigma$ needs to map generators to themselves, we may rewrite \eqref{eq:gardsidecomm} as the relation 
$\garside\cdot w_i\cdot \garside^{-1}=w_{\sigma(i)}\,$ in $\mathscr A\,$, with $\sigma$ redefined as an involution on $\{1,\ldots,n\}\,$. 

For example, for type $\LT{A}_n$ the Artin-Tits monoid is the set of positive braids $B_n^+$ in the braid group in $n$ strands, and $\garside$ is the full positive half-twist. In this situation, relation \eqref{eq:gardsidecomm} is readily visualized by moving a single positive neighbor crossing through this twist resulting in a crossing in reflected position, implying that $\sigma(w_i)=w_{n+1-i}\,$. For other types, the involution coincides with the respective ones in the Coxeter or Weyl group situation as in Section~\ref{subsec:ComplWords}.

Note that relation \eqref{eq:gardsidecomm} can also be derived directly from the presentations of $\longweyl$ for other Lie types given in Section~\ref{subsec:orderings} below, essentially following the methods for relative Garside elements developed in \cite{ABI15}.

\subsection{Inversion Root Sets and Convex Orderings}\label{subsec:descroots} In numerous calculations with words, it will be convenient to formally denote the index of the last letter as well as the word with the last letter deleted. That is, for a non-empty ($k\geq 1$) word
\begin{equation}\label{def:flat+tau_word}
    w=w_{i_1}\ldots w_{i_k}\in\wordset \qquad\mbox{set}\qquad w^\flat=w_{i_1}\ldots w_{i_{k-1}}\quad\mbox{and}\quad \tau(w)=i_k\,. 
\end{equation}
These components naturally enter the statement of the following lemma, which asserts an assignment of a positive root to any reduced word.

\begin{lemma} \label{lem:red=pos}
Suppose $w=w_{i_1}\ldots w_{i_k}\in\wordset$ is a non-empty reduced word. Then we have  
$$
\wordroot(w)=w^\flat(\alpha_{\tau(w)})=w_{i_1}\ldots w_{i_{k-1}}(\alpha_{i_k})\quad\in\proots\,, 
$$
where $w$ acts on roots by its respective element in $\Weyl$. 
\end{lemma}
  
To see this, note that $s(\beta)\in\proots$ if and only of $\length(s\cdot s_\beta)>\length(s)$ by Proposition 4.4.6 in \cite{BB05}, where $s_\beta$ denotes the reflection along $\beta\in\proots$. The lemma follows by specializing to $s=\Weylpres(w^\flat)$ and $\beta=\alpha_{i_k}\,$. 
Next, define for any element $s\in\Weyl$ the {\em inversion set} of $s$ as 
\begin{equation}\label{eq:defdescroots}
    \descroots{s}=\,\{\alpha\in\proots : s^{-1}(\alpha)\in\nroots\}\,=\,\proots\cap s(\nroots)\,.
\end{equation}
For a type $\LT{A}_{n-1}$ root system realized as  $\proots=\{\alpha_{i,j}=\epsilon_i-\epsilon_j:\,1\leq i<j\leq n\}\subset \mathbb R^n$ with $\Weyl=S_{n}$\,, the symmetric group in $n$ letters, $\descroots{s}$ is indeed given by the literal inversion set of the permutation $s^{-1}$.
Proposition 3.6 in \cite{Hil82} or Corollary 2 in Chapter VI \S 1.6 of \cite{Bou02} show that for any reduced word $w\in\wordset$ as in (\ref{def:flat+tau_word}) with $s=\Weylpres(w)\in\Weyl$ this set can be enumerated as 
\begin{equation}\label{eq:defdescrootsEnum}
    \descroots{s}=\{\beta_m=\wordroot(w[1,m]):1\leq m\leq k\}\qquad
    \mbox{where}\quad\beta_m=s_{i_1}\ldots s_{i_{m-1}}(\alpha_{i_m})
\end{equation}
and $k=\len{w}=\len{s}$\,. Thus, any word representing $s$ imposes a total ordering on $\descroots{s}$
given by  
$\beta_p\leqwt{w}\beta_j$ if $p\leq j\,$, writing also $\beta_p\letwt{w}\beta_j$ if $p< j\,$. 
Clearly, $\descroots{1}=\emptyset$\,, $\descroots{s_i}=\{\alpha_i\}$\,, and 
$\descroots{w}=\{\wordroot(u): u\leqRB w\}$ if the notation is extended to words.  Additional immediate 
consequences of the above are 
\begin{equation}\label{eq:descrootsbasics}
   \len s \,=\,    |\descroots s |\;,\quad  s^{-1}\left({\descroots s}\right)=-\descroots{s^{-1}}\;,
    \qquad \mbox{and} \qquad \descroots \longweyl =\proots\;.
\end{equation}

We next explain how these sets determine the defect in length additivity for composites of Weyl group elements. In the first observation, we denote for an arbitrary  subset $A\subseteq\roots$ 
two subsets of $\proots\,$ as 
$$
A^+=A\cap\proots\qquad \mbox{and}\qquad A^-=-(A\cap\nroots)=(-A)\cap\proots\,,
$$
so that $A=A^+\sqcup -A^-\,$. Also write $B\setminus A$ for the complement of $A$ in $B$.
\begin{lemma}\label{lem:desradditivdef}
    Suppose $s=a\cdot b$ for $a,b\in \Weyl$. Then $\,\descroots a \cap a(\descroots b )=\emptyset\,$ and 
    \begin{equation}\label{eq:Weylgenunion}
        \descroots s \;=\;a(\descroots b )^+\,\sqcup\, \descroots a \setminus a(\descroots b )^-\,.
    \end{equation} 
    Moreover,
    \begin{equation}\label{eq:Weyladddefect}
    \len s \,=\, \len a +\len b \,-\,2\varsigma(a,b)\qquad \mbox{where} \quad \varsigma(a,b)=|a(\descroots b )^-|=|\descroots b \cap \descroots {a^{-1}}|\,.
    \end{equation} 
\end{lemma}
\begin{proof} The disjointness statement is clear since $\descroots b\subseteq\proots$ and $a^{-1}(\descroots{a})\subseteq \nroots\,$. For \eqref{eq:Weylgenunion} we may split $\descroots s$
    according to the partition $\roots=a(\proots)\sqcup a(\nroots\,)$. Then 
    $a(\proots)\cap \descroots s=a(\proots\cap a^{-1}(\descroots s))=a(\proots\cap a^{-1}(\proots)\cap b(\nroots))=a(a^{-1}(\proots)\cap \descroots b)=\proots\cap a( \descroots b)=a( \descroots b)^+\,$.  

    Now, from $\proots=b(\proots)\cap\proots\sqcup b(\nroots)\cap\proots$ we see that $b(\proots)\cap\proots=\proots\setminus\descroots b\,$ and, hence, $b(\nroots)\cap\nroots=\nroots\setminus(-\descroots b)\,$. Applying $a$ to this we get 
    $s(\nroots)\cap a(\nroots)=a(\nroots)\setminus(-a(\descroots b))\,$, and intersecting this with 
    $\proots$ we find $\descroots s \cap a(\nroots)=\descroots a\setminus((-a(\descroots b))\cap\proots)=\descroots a\setminus a(\descroots b)^-\,$. 

    Note that $-a(\descroots b)^-=a(\descroots b)\cap\nroots=a(\descroots b\cap a^{-1}(\nroots))=
    a(\descroots b\cap \descroots{a^{-1}}$ as $\descroots b\subseteq \proots$. This readily implies the equality of the two expressions for $\varsigma(a,b)\,$. Clearly, $|a(\descroots b)^+|=|a(\descroots b)|-|a(\descroots b)^-|=|\descroots b|-\varsigma(a,b)=\len b -\varsigma(a,b)\,$. Now, $a^{-1}(a(\descroots b)^-)=(-\descroots b)\cap a^{-1}(\proots)=-(\descroots b\cap a^{-1}(\nroots)=-(\descroots b\cap \descroots {a^{-1}})\,$. Thus, using also \eqref{eq:descrootsbasics},
    $|\descroots a \setminus a(\descroots b )^-|=|a^{-1}(\descroots a )\setminus a^{-1}(a(\descroots b )^-)|=
    |(-\descroots {a^{-1}} )\setminus -(\descroots b\cap \descroots {a^{-1}})|=|\descroots {a^{-1}}|-|\descroots b\cap \descroots {a^{-1}}|=\len {a^{-1}}-\varsigma(a,b)\,$. Using $\len {a^{-1}}=\len a$ and adding these expressions yields the length defect equation in \eqref{eq:Weyladddefect}. 
\end{proof}

The following criteria for exact additivity and monotonicity are now easily derived from Lemma~\ref{lem:desradditivdef}. For example, if $\descroots{a}\subseteq\descroots{s}$  we have $\varsigma(a^{-1},s)=|\descroots{s}\cap\descroots{a}|=|\descroots{a}|=\len a\,$. 
So, for $b=a^{-1}s$ \eqref{eq:Weyladddefect} implies $\len{b}=\len{a^{-1}}+\len{s}-2\len{a}=\len{s}-\len{a}\,$ or $\len{s}=\len{a}+\len{b}\,$. The latter and $s=a\cdot b$ now imply $a\leqRB s\,$ as claimed in \eqref{eq:descrmonot}.

\begin{cor}\label{cor:sumdescroots} For two elements $a, b\in\Weyl$ the following conditions are equivalent.
\begin{multicols}{2} 
\begin{enumerate}[label=\roman*), leftmargin=11mm,]
    \item $\len{a\cdot b}=\len{a} + \len{b}$\vspace{1.5mm}
    \item $\descroots {a\cdot b}=a(\descroots{b})\sqcup \descroots a$\vspace{1.5mm}
    \item $a(\descroots b)\subseteq \proots$\vspace{1.5mm}
    \item $\descroots b \cap \descroots {a^{-1}}=\emptyset$\vspace{1.5mm}
    \end{enumerate} 
\end{multicols}\vspace{-2mm}
\noindent Moreover, we have for elements $a,b,s\in\Weyl$ that 
\begin{equation}\label{eq:descrmonot}
    a\leqRB s\;\Leftrightarrow \;\descroots a \subseteq \descroots s\qquad\mbox{and}\qquad
    \qquad b\leqLB s \;\Leftrightarrow \;\descroots {b^{-1}} \subseteq \descroots {s^{-1}}\,. 
\end{equation}
In particular, $\descroots{s}=\descroots{a}$ implies $s=a\,$. 
\end{cor}

The fact that the longest element of $\Weyl$ yields all of $\proots$ implies that any reduced word of maximal length $w\in\wordsetmax$ imposes a total order $\letwt{w}$ on the set of positive roots via $\,\proots=\descroots{w}\,$. Specifically, if we denote 
$\beta_j=\wordroot(w[1,j])$ we have $\beta_1\letwt{w}\beta_2\letwt{w}\ldots \letwt{w}\beta_L$ where $L=\len{w}\,$. We
write $\leqwt{w}$ if equality is included. 
 
The result of Papi in \cite{Pa94} states that the assignment $w\mapsto\leqwt{w}$ establishes a bijection between elements in $\wordsetmax$ and the (total) {\em convex} orders on $\proots$. Here convexity means that if for two roots $\alpha,\beta\in\proots$ we have $\alpha\letwt{w}\beta$ and  
$\alpha+\beta\in\proots$, then $\alpha\letwt{w}\alpha+\beta\letwt{w}\beta\,$.
 
Let $s=\Weylpres(w)$ and observe that, since the $\beta_j$ as above are distinct and $L=|\descroots{s}|$, we have that the map $[1,L]\rightarrow\descroots{s}:\,j\mapsto\beta_j$ is a bijection. 
 Thus, to each $\alpha\in \descroots{s}$ we can assign a unique $j_\alpha$ such that 
$\alpha=\beta_{j_\alpha}\,$. It is convenient to also introduce the inverse map from the root set to words,
\begin{equation}\label{eq:root2word}
    \descroots{s}\rightarrow\{u\in\wordset :u\leqRB w\}\,:\;\alpha\mapsto w[\alpha]\,,
\end{equation}
where we denote $w[\alpha]=w[1,j_\alpha]\,$ so that $\alpha=\Weylpres(w[\alpha])\,$. In particular, if $w\in\wordsetmax$ the correspondence \eqref{eq:root2word} assigns to any positive root a reduced word $\leqRB w\,$.

For non-simply laced root systems (types $\LT{BCFG}$) with $\maxd >1$ it is useful to distinguish the short and long roots portions of the various root sets. 
That is, we have $\roots=\rootsL\sqcup \rootsS$ where $\rootsL=\{\alpha\in \roots:d_\alpha=\maxd\}$ and 
$\rootsS=\{\alpha\in \roots:d_\alpha=1\}$. 
Correspondingly, denote 
\begin{equation}
\begin{aligned}
    \prootsL&=\proots\cap \roots_L\,, \qquad&
    \nrootsL&=\nroots\cap \roots_L\,, \qquad&
    \descrootsL{s}&=\descroots{s}\cap\prootsL\,,
    \\
    \prootsS&=\proots\cap \roots_S\,, \qquad&
    \nrootsS&=\nroots\cap \roots_S\,, \qquad&
    \descrootsS{s}&=\descroots{s}\cap\prootsS\,.
\end{aligned}
\end{equation}
Since $s(\roots_L) = \roots_L$ and $s(\roots_S) = \roots_S$ for any $s\in\Weyl$, we also have  
$\descroots{s}_L=s(\nrootsL)\cap \prootsL$ and $\descroots{s}_S=s(\nrootsS)\cap \prootsS\,$. Thus, for example, the relations between root sets in Lemma~\ref{lem:desradditivdef} will hold for the respective long root and short root portions separately. Furthermore, introduce  the following subsets of the Weyl group.
\begin{equation}\label{eq:def:WeylSL}
    \WeylS=\left\{{s\in\Weyl\,:\,\descrootsS{s}=\emptyset}\right\}
    \qquad\mbox{and}\qquad 
    \WeylL=\left\{{s\in\Weyl\,:\,\descrootsL{s}=\emptyset}\right\}
\end{equation}

\begin{lem}\label{lm:WeylSL}
Suppose $\roots$ is a non-simply laced root system with $\maxd>1\,$. Then $\WeylS$ and $\WeylL$ are the subgroups of $\Weyl$ generated by $\{s_i: d_i=\maxd\}$ and $\{s_i: d_i=1\}$ respectively.  

\noindent Thus, for type $\LT{B}_n$\,, we have that $\WeylS\cong S_{n-1}$ is the canonical $\LT{A}_{n-1}$-subgroup and $\WeylS\cong\mathbb F_2$ the $\LT{A}_1$-subgroup. The assignments are reversed for $\LT{C}_n$\,.
For $\LT{F}_4$ they are the respective $\LT{A}_2$-subgroups and for $\LT{G}_2$ the respective $\LT{A}_1$-subgroups.
\end{lem}
\begin{proof} $\WeylS$ is readily identified with the stabilizer $\{s:s(\prootsS)=\prootsS\}$ and thus a subgroup of $\Weyl$. 
If $d_i=\maxd$ we also have $\descroots{s_i}=\{\alpha_i\}$ so that  $\descrootsS{s_i}=\emptyset$ and, hence, $s_i\in \WeylS$\,,
implying that $\WeylS$ contains the subgroup generated by $\{s_i: d_i=\maxd\}$. For the converse inclusion, assume 
$s=\Weylpres(w)\in\WeylS$ is given by a reduced expression $w=w_{i_1}\ldots w_{i_k}$\,. Then $\beta_m=\wordroot(w[1,m])=t(\alpha_{i_m})$ where $t=\Weylpres(w[1,m-1])$ so that $\beta_m$ has the same length as $\alpha_{i_m}$\,.
Hence, $\descrootsS{s}=\emptyset$ if and only if $d_{i_1}=\ldots =d_{i_k}=\maxd$, implying that $s$ is in the subgroup generated by $s_i$ with $d_i=\maxd\,$. The arguments for $\WeylL$ are analogous. 
\end{proof}

We will need the following observation about root decompositions later only in the $\LT{A}$ case. The elementary argument, though, extends to all classical types. The exceptional types are not discussed here. 

\begin{lem} \label{lem:sumposroots}
Let $\roots$ be a root system of classical type $\LT{ABCD}\,$. Assume that $\proots$ is equipped with any convex ordering $\letwt{w}\,$. Suppose $\beta=\beta_1+\ldots+\beta_k$
with $\beta,\beta_i\in\proots$ for $i=1,\ldots,k\,$ with $k>1\,$. Then\vspace*{-1mm}
\begin{enumerate}[label=\roman*), leftmargin=2cm,] 
    \item There exists $j$ for which $\beta-\beta_j\in\proots$.\label{item:sumposroots:subsum}\vspace*{2mm}
    \item There exists $r$ for which $\beta_r\letwt{w}\beta\,$.  \label{item:sumposroots:min}
\end{enumerate} 
\end{lem}

\begin{proof}
    Consider the same numbering of simple roots as in \cite[Plate I-IV]{Bou02} so that $\alpha_1$ is opposite the special roots. The formulae in 
    \cite{Bou02} imply that every positive root is of the form $\beta_{a,b}=\sum_{s=a}^b\alpha_{\mathscr c(s)}$, where
    $\mathscr c:\{1,\ldots,N\}\rightarrow\{1,\ldots,n\}\,$ is a 
    numbering sequence depending on Lie type, $n=|\sroots|\,$, and $1\leq a\leq b\leq N$. 

    For $\beta=\sum_ic_i\alpha_i$ set $\,\height(\beta)=\sum_ic_i\,$ and 
    $\,r(\beta)=\min\{i:c_i\neq 0\}\,$. The formulae $\height(\beta_{a,b})=b-a+1\,$
    and $r(\beta_{a,b})=\min\{\mathscr c(a),\mathscr c(b)\}\,$ are readily checked for all cases below.
  We use this notation to prove the statement {\em \ref{item:sumposroots:subsum}} for each Lie type separately.
    
    For type $\LT{A}_n$ set $N=n$ and $\mathscr c(i)=i$. Then any positive root is of the form $\beta=\beta_{a,b}$\,. Now, for $\beta$ as in the lemma, at least one $\beta_j$ must have non-zero coefficient for $\alpha_a$ and no $\beta_i$ can have $\alpha_s$ with $s<a$ as a summand. Thus, we must have $r(\beta_j)=a$
    for at least one  $j$, which implies $\beta_j=\beta_{a,c}$ for some $c\leq b$ and, with $k>1$, also $c<b$. Then
    $\beta-\beta_j=\beta_{a,b}-\beta_{a,c}=\beta_{c+1,b}\in\proots\,$. 

    Next consider type $\LT{C}_n\,$, setting $N=2n-1$ and $(\mathscr c(1),\ldots,\mathscr c(N))=(1,\ldots, n-1,n,n-1,\ldots,1)\,$. 
    It follows by inspection that all positive roots are of the form $\beta_{a,b}\,$ and all $\beta_{a,b}\in\proots\,$. In fact, $(a,b)\mapsto\beta_{a,b}$ defines a bijection between $\proots$ and the set of pairs $(a,b)$ for which $a\leq\min(n,b)$ and $\mathscr c(a)\leq \mathscr c(b)\,$. Let $\beta=\beta_{a,b}$ with these constraints so that $a=r(\beta)\,$. As before, 
     we can pick $\beta_j$ such  that $a=r(\beta_j)$ and, hence, $\beta_j=\beta_{a,c}$ for some $c<b\,$. Then, 
     $\beta-\beta_j=\beta_{c+1,b}$ which is again in $\proots$ by previous remarks. 
    
    For $\LT{B}_n$ we set $N=2n$ and $(\mathscr c(1),\ldots,\mathscr c(N))=(1,\ldots, n-1,n,n,n-1,\ldots,1)\,$. As opposed to previous Lie types, not all $\beta_{a,b}$ are roots. Specifically, $\beta_{a,b}\not\in\proots$ if and only if
    $a\neq b$ but $\mathscr c(a)=\mathscr c(b)\,$. This case has to be excluded to establish 
    a one-to-one correspondence,
    as in the $\LT{C}_n$ case. We may thus assume $\beta=\beta_{a,b}$ with $a=r(\beta)\leq n$ and $\mathscr c(a)<\mathscr c(b)\,$
    since $k>1$ implies $a\neq b\,$. Then, as before, we have $\beta_j=\beta_{a,c}$ for some  $c<b$ so that 
    $\gamma=\beta-\beta_j=\beta_{c+1,b}\,$. If either $c+1=b$ or $\mathscr c(c+1)\neq \mathscr c(b)$ we are done since
    $\beta_{c+1,b}\in\proots\,$. In case $c+1<b$ and $\mathscr c(c+1)=\mathscr c(b)$ we have $\gamma=2\sum_{c<i\leq n}\alpha_i=\sum_{t\neq j}\beta_t\,$ so that there has to be an $s\neq j$ with $\beta_s=\beta_{c+1,d}$ with $d<b$. 
    With $\mathscr c(N+1-i)=\mathscr c(i)$, $b=N-c$, and $d'=N-d$ we have $\beta_s=\beta_{d'+1,b}$\,. Now, $\beta-\beta_s=\beta_{a,d'}$ which is in $\proots$ since $\mathscr{c}(a)<\mathscr{c}(d')\,$.

    For type $\LT{D}_n$ we may choose $N=2n-2$ and $(\mathscr c(1),\ldots,\mathscr c(N))=(1,\ldots, n-2, n-1,n,n-2,\ldots,1)\,$. A second sequence $\mathscr c'$ may be defined via $\mathscr c'(i)=\mathscr c(N+1-i)=\eta(\mathscr c(i))$, where 
    $\eta=(n-1,n)$ is the non-trivial involution of the Dynkin diagram. All positive roots are of the form $\beta_{a,b}$ or
    $\beta'_{a,b}$ defined via the respective sequences and $\beta_{\mathscr c(a),\mathscr c(b)}=\beta'_{a,b}$ 
    if $\{a,b\}\cap\{n-1,n\}=\emptyset\,$. As for $\LT{B}_n$\,, $\beta_{a,b}\not\in\proots$ if and only if
    $a\neq b$ and $\mathscr c(a)=\mathscr c(b)\,$. The argument now proceeds as in the $\LT{D}_n$ case, except that a subsequence
    may be chosen from either $\mathscr c$ or $\mathscr c'\,$.

    Item {\em \ref{item:sumposroots:min}} follows now by induction in $k\,$. That is, if $\beta=\beta_1+\ldots+\beta_k$ 
    and $j$ such that $\beta'=\beta-\beta_j=\sum_{i\neq j}\beta_i\,\in\proots\,$, convexity implies that
    either $\,\beta_j\letwt{w}\beta \letwt{w}\beta'\,$ or $\,\beta'\letwt{w}\beta \letwt{w}\beta_j\,$.  
    In the former case we are done. In the latter, the induction hypothesis implies $\beta_i\letwt{w}\beta'$
    for some $i$ ($i\neq j$) so that, by transitivity,  also $\beta_i\letwt{w}\beta\,$.
\end{proof}

\subsection{Longest Elements, Involutions, and Complementary Words}\label{subsec:ComplWords}
The longest element $\longweyl\in\Weyl$ is an involution by uniqueness. Its respective involutive action on the root system maps $\proots$ isometrically to $\nroots$ (e.g., Corollary 3 in \cite[VI.1.6]{Bou02}) and is thus of the form
$\longweyl (\alpha_i)=-\alpha_{\dyninv(i)}$ for an involution $\dyninv$ of the Dynkin diagram associated to $\roots\,$. 

As specified in Item XI of each plate in the appendix of \cite{Bou02}, $\dyninv$ is the non-trivial involution for $\LT{A}_n$\,, $\LT{D}_{2k+1}$\,, and $\LT{E}_6$\,, and it is identity in all other cases.  Mapping the $\mathscr A$-version of \eqref{eq:gardsidecomm} to $\Weyl$ then implies that $\sigma=\dyninv\,$. We use the same notation $\dyninv$ for the conjugation action of $\longweyl$ on $\Weyl$\,. Since $s_i$ is given by reflection along $\alpha_i$ we, thus, find  
\begin{equation}\label{eq:def_eta}
\dyninv(s_i)=s_{\dyninv(i)}\qquad\mbox{where}\quad
    \dyninv(t)={\longweyl} t {\longweyl}^{-1}={\longweyl} t {\longweyl}\,.
\end{equation}
We use the same notation for the obvious extension of $\dyninv$ to the set of 
reduced words $\wordset\,$. We will frequently use the involution
\begin{equation}\label{eq:def:dyninv}
    w^\winvchar=\dyninv(w)^*=\dyninv(w^*)
\end{equation}
on $\wordset$, explicitly given as $(w_{i_1}\ldots w_{i_N})^\winvchar=w_{\dyninv(i_N)}\ldots w_{\dyninv(i_1)}\,$ for reduced expressions. Clearly, $\len{w^\winvchar}=\len{w}$ and $s^\winvchar=\wordroot(w^\winvchar)=\dyninv(s)^{-1}\,$.  A related involutive map on $\Weyl$ is given by right multiplication with the longest element by 
\begin{equation}\label{eq:def:flweyl}
    t\,\mapsto\,\flweyl{t}=t\cdot \longweyl\,. 
\end{equation}

\begin{lem}\label{lem:longinvol}
Suppose $t\in\Weyl$ and $\flweyl{t}$ is as in \eqref{eq:def:flweyl}. Then\vspace{-2mm}

\begin{multicols}{2}
\begin{enumerate}[label=\roman*), leftmargin=11mm,]
    \item $t=\flweyl{\flweyl{t}}\,,$\vspace*{2mm}
      \label{item:longinvol:invol}
    \item $\longweyl \cdot t=\dyninv(\flweyl{t})=\flweyl{\dyninv(t)}\,,$\vspace*{2mm}
      \label{item:longinvol:oppact}
    \item $\len{\flweyl{t}}=\len{\longweyl}-\len{t}\,,$\vspace*{2mm}
      \label{item:longinvol:length}
    \item $t\leqRB u\;\Leftrightarrow \;\flweyl{u}\leqRB \flweyl{t}\,,$\vspace*{2mm}
      \label{item:longinvol:order}
    \item $\descroots{\flweyl{t}}=\proots\setminus \descroots{t} = \proots\cap t(\proots)\,,$\vspace*{2mm}
      \label{item:longinvol:descsets}
    \item $\flweyl{t}\left({\descroots{t^\winvchar}}\right)=\descroots{t}\,.$\vspace*{2mm}
      \label{item:longinvol:daggerdesc}
\end{enumerate}
\ 
    \end{multicols}
\vspace{-8mm}
    
    \end{lem}
\begin{proof}
    Items {\em \ref{item:longinvol:invol}}, {\em \ref{item:longinvol:oppact}},  and {\em \ref{item:longinvol:length}} are immediate from the definitions. See also Proposition~2.3.2 in \cite{BB05}.
    For Item {\em \ref{item:longinvol:order}}
    assume $t\leqRB t$, meaning $u=t\cdot a$ with $\len{u}=\len{t}+\len{a}$. Using \eqref{eq:def_eta} this yields the decomposition $\flweyl{t}=\flweyl{u}\cdot \dyninv(a)^{-1}\,$.
    Further, using {\em \ref{item:longinvol:length}}, $\len{\flweyl{t}}-\len{\longweyl}-\len{t}=\len{\flweyl{u}}+\len{a} =\len{\flweyl{u}}+\len{\dyninv(a)^{-1}}\,$. So, $\flweyl{u}\leqRB \flweyl{t}\,$. For Item {\em \ref{item:longinvol:descsets}} note that 
    $\descroots{\flweyl{t}}=\proots\cap \flweyl{t}(\nroots)=\proots\cap t(\longweyl(\nroots))= 
    \proots\cap t(\proots)=\proots\setminus\descroots{t}\,$. Finally,
    Item {\em \ref{item:longinvol:order}} follows from
    $\flweyl{t}\left({\descroots{t^\winvchar}}\right)
      =t\longweyl(\proots\cap t^\winvchar(\nroots))
      =t(\longweyl(\proots)\cap \longweyl t^\winvchar(\nroots))
       =t(\nroots\cap  t^{-1}\longweyl(\nroots))
       =t(\nroots\cap  t^{-1}(\proots))
       =t(\nroots)\cap  \proots=\descroots{t}\,
      $.
\end{proof}
 
For the description of involutive automorphisms on quantum groups, the related notion of {\em complementary words} will be relevant. Given a non-empty reduced word $w\in\wordset$, we say that $v\in\wordset$ is complementary to $w$, writing $v\wcomplrel w$, if the following two conditions hold.
\begin{equation}\label{eq:def_complwords}
 v\wcomplrel w\qquad \Leftrightarrow \qquad  \tau(v)=\dyninv(\tau(w)) \qquad \mbox{and} \qquad \flweyl{\Weylpres(v)}=\Weylpres(w^\flat)\,.
\end{equation}
If we write $s=\Weylpres(w^\flat)$, $t=\Weylpres(v^\flat)$, $j=\tau(w)$, and $k=\tau(v)$ the condition in \eqref{eq:def_complwords} can be rephrased more symmetrically as $k=\dyninv(j)$ and 
\begin{equation}\label{eq:defalt_complwords}
    t^{-1}\cdot s=q_j\qquad\mbox{where}\quad q_j=\longweyl\cdot s_j=s_{\dyninv(j)}\cdot \longweyl\,. 
\end{equation}
Note, with $N=\len{\longweyl}$, these imply $\len{w}+\len{v}=N+1$ or $\len{t}+\len{s}=N-1$.  
 
\begin{lem} \label{lm:complwordprops}
Suppose $v\wcomplrel w$ for non-empty, reduced words $v,w\in\wordset\,$. Then\vspace{-2mm}

\begin{multicols}{2} 
\begin{enumerate}[label=\roman*), leftmargin=11mm,]
    \item $w\wcomplrel v\,,$\vspace*{2mm} \label{item:complwordprops:symmrel}
    \item $\wordroot(v)=\wordroot(w)\,,$\vspace*{2mm}\label{item:complwordprops:eqroot}
    \item $\descroots{w}\cap\descroots{v}=\{\wordroot(w)\}\,,$\vspace*{2mm}\label{item:complwordprops:descrcap}
    \item $\proots=\descroots{w}\cup\descroots{v}\,.$
    \label{item:complwordprops:descrcup}
    \end{enumerate} 
\end{multicols}\vspace{-7mm}

\noindent
    Moreover, if $z\in\wordsetmax$ is a reduced word of maximal length $N=\len{z}$ and $y=z^\winvchar$, then 
    \begin{equation}\label{eq:complmaxword}
        z[1,r]\wcomplrel z[r,N]^\winvchar=y[1,N+1-r]\,.
    \end{equation}
\end{lem}
\begin{proof}
Symmetry in {\em \ref{item:complwordprops:symmrel}} is immediate from \eqref{eq:defalt_complwords} together with
$q_j^{-1}=q_{\dyninv(j)}$ and $\dyninv^2=\id\,$. Note further that $q_j(\alpha_j)=\longweyl(s_j(\alpha_j))=\longweyl(-\alpha_j)=\alpha_{\dyninv(j)}\,$. So, with notation as in \eqref{eq:defalt_complwords}, $\wordroot(w)=s(\alpha_j)=t(q_j(\alpha_j))
=t(\alpha_{\dyninv(j)})=\wordroot(v)\,$. For {\em \ref{item:complwordprops:descrcap}} and {\em \ref{item:complwordprops:descrcup}}  note that $\descroots{w}=\descroots{w^\flat}\sqcup \{\wordroot(w)\}$ by 
Corollary~\ref{cor:sumdescroots}, and $\descroots{w^\flat}=\proots\setminus\descroots{v}$ by \eqref{eq:def_complwords} and Lemma~\ref{lem:longinvol}
{\em \ref{item:longinvol:descsets}}. 

For the last statement, let $z=w_{i_1}\ldots w_{i_N}$ so that $y=w_{\dyninv(i_N)}\ldots w_{\dyninv(1)}\,$. Denote further 
$w=z[1,r]=w_{i_1}\ldots w_{i_r}\,$ and $v=z[r,N]^\winvchar=\dyninv(w_{i_r}\ldots w_{i_N})^*=w_{\dyninv(i_N)}\ldots w_{\dyninv(i_r)}=y[1,N+1-r\,]$. Clearly, $\dyninv(\tau(w))=\dyninv(i_r)=\tau(v)$, and $z=z[1,r-1]z[r,N]=w^\flat \dyninv(v^{*})\,$.
The latter implies $\longweyl=\Weylpres(w^\flat)\Weylpres(\dyninv(v)^*)=\Weylpres(w^\flat)\Weylpres(\dyninv(v))^{-1}$ and hence $\Weylpres(w^\flat)=\longweyl\cdot\dyninv(\Weylpres(v))
=\Weylpres(v)\cdot\longweyl=\flweyl{\Weylpres(v)}\,$ as desired. 
\end{proof}

Combining Item~{\em \ref{item:complwordprops:eqroot}} and \eqref{eq:complmaxword} the following order reversal relation is an immediate consequence for the total ordering $\letwt{z}$ on $\proots$ depending on a maximal length reduced word $z\in\wordsetmax$ as introduced in Section~\ref{subsec:descroots}.

\begin{cor}\label{cor:revwordorder} Suppose $z\in \wordsetmax$ and $y=z^\winvchar\,$. Then  $\,\alpha\letwt{z}\beta\,$ if and only if 
    $\,\beta\letwt{y}\alpha\,$. 
\end{cor}

Relation \eqref{eq:complmaxword} provides explicit constructions of complementary words that imply a partitioning of $\proots$ into two totally ordered sets, which have only their maximal elements in common. That is, for a given $w\in\wordset$ find
$z\in\wordsetmax$ with $z\geqRB w$ and hence $w=z[1,r]$. Set then $v=y[N+1-r]$ with $y=z^\winvchar\,$.  
Setting $\beta_i=\wordroot(z[1,i])$ and $\mu_i=\wordroot(y[1,i])$, the sets $\descroots{w}$ and
$\descroots{v}$ are then given by the elements in respective maximal chains $\beta_1\letwt{z}\beta_2\letwt{z}\ldots \letwt{z}\beta_r$ and $\mu_1\letwt{y}\mu_2\letwt{y}\ldots \letwt{y}\mu_{N+1-r}$, with $\beta_r=\mu_{N+1-r}$ but 
$\beta_i\neq\mu_j$ otherwise. 

For example, consider the reduced word $w=w_2w_3w_2w_1$ in a type $\LT{A}_3$ Coxeter system. A maximal length reduced word
$z\in\wordsetmax$ with $w\leqRB z$ is then given by $z=w_2w_3w_2w_1w_2w_3\,$ so that $w=z[1,4]\,$. With $\dyninv(1)=3$ we find
$y=\dyninv(z)^*=w_1w_2w_3w_2w_1w_2$, which implies that $v=y[1,3]=w_1w_2w_3$ is a complementary word $w\wcomplrel v\,$. Then 
the ordered elements of $\descroots{w}$ are $\alpha_2\letwt{z} \alpha_2+\alpha_3\letwt{z} \alpha_3\letwt{z} \alpha_1+\alpha_2+\alpha_3$ and those of $\descroots{v}$ are $\alpha_1\letwt{y} \alpha_1+\alpha_2\letwt{y} \alpha_1+\alpha_2+\alpha_3\,$. Complementary words can also be constructed directly from the conditions in \eqref{eq:defalt_complwords} by choosing a presentation for $t=\Weylpres(w^\flat)q_j^{-1}$ and adjoining $w_{\tau(w)}\,$.

In the rank 2 case, complementary words are unique. A complete discussion of this situation is provided in Section~\ref{subsec:rank2coxsys} below.

\subsection{Root Lattices and Special Elements}\label{subsec:LettSpecElem}

Let $\wordindexset s\subseteq\{1,\ldots,n\}$ be the set of indices $\wordindexset s=\{i_1,\ldots,i_M\}$ occurring in a reduced expression $s=s_{i_1}\ldots s_{i_M}\,$. It is clear by either Matsumoto's theorem or properties of parabolic subgroups in Section~2.4 of \cite{BB05} (or the lemma below) that $\wordindexset s$  does not depend on the choice of this presentation. 

Denote also by $\,\wordindexlattice{s}=\langle\descroots s\rangle_{\mathbb Z} \subset\mathbb Z^{\sroots}$ the sublattice of the root lattice generated by roots in the inversion set $\descroots s$. It admits the following simple description in terms of the set $\wordindexset s\,$. 

\begin{lem}\label{lem:descrootlattice}
The subset $\{\alpha_i\in\sroots\,:\,i\in \wordindexset s\}$ is a basis for 
$\,\wordindexlattice{s}\,$.     
\end{lem}
\begin{proof}
    Proceeding by induction in $\len s$, we note that for $s=s_j$ we have 
    $\wordindexlattice{s}=\langle\descroots {s_j}\rangle_{\mathbb Z}
    =\langle\{\alpha_j\}\rangle_{\mathbb Z}=\mathbb Z\alpha_j\,$ as desired. 
    Suppose now $\,t=s_k\cdot s\,$ with $\len t = \len s +1$ and assume the claim holds for $s\,$.
    By Corollary~\ref{cor:sumdescroots} we then have 
       $\descroots {t}=s_k\left({\descroots {s}}\right)\sqcup \{s_k\}$ and, hence,
      $\,\wordindexlattice{t}=s_k \left({\wordindexlattice {s}}\right)+\mathbb Z\alpha_k\,$. 
      By induction hypothesis $\wordindexlattice {s}$ is spanned by all $\alpha_i$ with $i\in\wordindexset{s}$ so that the first summand is spanned by all 
      $s_k(\alpha_i)=\alpha_i-A_{ki}\alpha_k$ with $i\in\wordindexset{s}$. Hence,
      $\wordindexlattice {t}$ is spanned as a lattice by all $\alpha_i$ with 
      $i\in\wordindexset{s}\cup\{k\}=\wordindexset{t}$ as desired. 
\end{proof}

For later use in computations, we introduce the sum of roots $\sumdescroots{s}$ and sum of coroots
$\sumdesccoroots{s}$ as elements of lattices in $E$, namely,
\begin{equation}\label{eq:deftheta}
    \sumdescroots{s}=\sum_{\alpha\,\in\,\descroots{s}}\!\alpha
     \;\;\in\mathbb Z^{\sroots}
     \qquad\mbox{and}\qquad \sumdesccoroots s =\sum_{\alpha\,\in\,\descroots{s}}\tfrac 1 {d_\alpha}\alpha
     =\sum_{\alpha\,\in\,\descroots{s}}\breve\alpha
     \;\;\in\tfrac 1 {\maxd}\mathbb Z^{\sroots}\;.
\end{equation}

Here $\breve\alpha=d^{-1}_\alpha\alpha$ are the usual coroots. 
Observe also that, by \eqref{eq:bracketrel}, the pairings $\symbrack{\sumdesccoroots s}{\beta}$ are integers for any $\beta\in\roots$ and any $s\in\Weyl$. Corollary~\ref{cor:sumdescroots} and $\Weyl$-invariance of $d_\alpha$ readily imply that for elements $t_1,t_2\in\Weyl$ and $s=t_1\cdot t_2$ with $\length(s)=\length(t_1)+\length(t_2)$   we have 
\begin{equation}
\label{eqn:theta}
    \sumdescroots{s}\,=\,\sumdescroots{t_1}\,+\,t_1\!\left({\sumdescroots{t_2}}\right)\qquad\mbox{and}\qquad
    \sumdesccoroots{s}\,=\,\sumdesccoroots{t_1}\,+\,t_1\!\left({\sumdesccoroots{t_2}}\right)\,.
\end{equation}

The following mirror relations for both roots sets and root sums are immediate from \eqref{eq:defdescroots} and \eqref{eq:deftheta}. 
\begin{align}\label{eq:descrootsmirror}
    s^{-1}(\sumdescroots s )=-\sumdescroots{s^{-1}}
  \qquad\mbox{and}\qquad
    s^{-1}(\sumdesccoroots s )=-\sumdesccoroots{s^{-1}} 
\end{align}
 
Finally, the vectors $\sumdescroots{s}$ and  $\sumdesccoroots{s}$ can be expressed more directly using the usual half sums of positive roots or coroots via the  formulae, 
\begin{equation}\label{eq:thetaviarho}
\begin{aligned}
    &&\sumdescroots{s}&=\weightvecchar-s(\weightvecchar)  
       \qquad\mbox{and}&\quad 
    \sumdesccoroots{s}&=\breve\weightvecchar-s(\breve\weightvecchar)\,, \\
   \mbox{where}\qquad\rule{0mm}{5mm} &&
\weightvecchar&=\tfrac 12 \sum_{\alpha\in\proots}\alpha=\tfrac{1}{2}\sumdescroots{\longweyl} 
    \qquad\mbox{and}&\quad 
\breve\weightvecchar&=\tfrac 12 \sum_{\alpha\in\proots}\breve\alpha=\tfrac{1}{2}\sumdesccoroots{\longweyl}\,.
\end{aligned}
\end{equation}
The formula for $\sumdescroots{s}$ is proved in Proposition~21.15 in \cite{Bu04} and the proof for the one for $\sumdesccoroots{s}$ is analogous using also invariance $d_{s(\alpha)}=d_\alpha\,$. Applying a simple reflection to \eqref{eq:thetaviarho}, a standard argument shows that $\symbrack{\weightvecchar}{\alpha_i}=\frac 12 \symbrack{\alpha_i}{\alpha_i}=d_i$ and $\symbrack{\breve\weightvecchar}{\alpha_i}=1$ for all $\alpha_i\in\sroots\,$. 

The latter allows us to express the height function $\height$, given as $\height(\beta)=\sum_ib_i$ for a root vector $\beta=\sum_ib_b\alpha_i$\,, as $\height(\beta)=\symbrack{\breve\weightvecchar}{\beta}\,$. Thus, together with 
\eqref{eq:thetaviarho} we can write
\begin{equation}\label{eq:sumdesccoheight}
    \symbrack{\sumdesccoroots{s}}{\beta}\,=\,\height(\beta)-\height(s^{-1}(\beta))\quad\in\mathbb Z\,.
\end{equation}
Finally, we note that $\weightvecchar$ is an element of the weight lattice $\weightlattice$ and that 
$\breve\weightvecchar\in\tfrac 1 {\maxd}\weightlattice\,$.

\subsection{\texorpdfstring{$\roots$}{Φ}-Symmetric Lattice Maps}\label{subsec:rootsymmaps} In preparation for the  construction
of a large family of automorphisms on quantum groups in Section~\ref{subsec:Kscale}, we generalize some  classical notions of root systems to non-canonical integral inner products on $E$. 

We say that a lattice homomorphism $\mathsfit h: {\mathbb Z}^{\sroots}\rightarrow {\mathbb Z}^{\sroots}$ 
is {\em $\roots$-symmetric} if $\symbrack{\mathsfit h(\mu)}{\nu}=\symbrack{\mu}{\mathsfit h(\nu)} $ for all $\mu,\nu\in\mathbb Z^{\sroots}\,$, that is, $\mathsfit h$ is symmetric with respect to an underlying Euclidean space $E\,$. Denote also $\symmmapspace{\roots}$ the space of such $\roots$-symmetric maps, which, thus, forms a sublattice of $\,\mathrm{End}({\mathbb{Z}}^{\sroots})$ of 
rank $\binom{n+1}{2}\,$. Note that  the quotient $\mathrm{End}({\mathbb Z}^{\sroots})/\symmmapspace{\roots}$ will, generally, contain torsion. 

Clearly, orthogonality of the Weyl group action implies that for any $s\in\Weyl$ and $\roots$-symmetric map $\mathsfit h$ also $s{\mathsfit h}s^{-1}$ is $\roots$-symmetric. Note that for any $\alpha\in\roots$\,, the Weyl reflection $s_\alpha\in \symmmapspace{\roots}$ since $s_\alpha=s_\alpha^{-1}$\,. Hence, any integer combination of such reflections $\sum_\alpha c_\alpha s_\alpha$ is also $\roots$-symmetric. 
Generalizing the element $\weightvecchar$ from \eqref{eq:thetaviarho}, we define for
a given $\roots$-symmetric lattice map $\mathsfit h$ the vector $\weightvec {\mathsfit h}$ in $E$ via the condition
\begin{equation}\label{eq:defrhosymm}
 \symbrack{\weightvec {\mathsfit h}}{\alpha_i}\,=\,\tfrac 1 2 \symbrack{\alpha_i}{\mathsfit h(\alpha_i)} \,.
\end{equation}
While the concept of $\roots$-symmetric maps does not depend on the choice of a simple root system, the definition of $\weightvec {\mathsfit h}$ strongly depends on a choice of $\sroots\,$. 
We, further, generalize formula \eqref{eq:thetaviarho} for the sum $\sumdescroots{s}$ above by setting for any $s\in\Weyl$
\begin{equation}\label{eq:defthetasymm}
 \sumdescrootsrel{s}{\mathsfit h} \,=\,\weightvec {s{\mathsfit h}s^{-1}}\,-\,s(\weightvec {\mathsfit h})\,.
\end{equation}

Several basic properties of these objects are summarized in the following lemma, where $\weightlattice$ is the weight lattice associated to $\roots$ and $\idsymm:\mathbb Z^{\sroots}\rightarrow\mathbb Z^{\sroots}$ is the identity map, which is obviously $\roots$-symmetric. 

\begin{lem}\label{lem:symmbasicprops}
For any $\roots$-symmetric lattice maps $\mathsfit h, \mathsfit k\in\symmmapspace{\roots}$, any root vector $\beta\in\mathbb Z^{\sroots}$, any Weyl group elements $s, t\in\Weyl$, and any $m\in\mathbb Z$ the following hold:\vspace*{-2mm}
\begin{multicols}{2}
\begin{enumerate}[label=\roman*), leftmargin=2cm,]
    \item $\weightvec {\mathsfit h+\mathsfit k}=\weightvec {\mathsfit h}+\weightvec {\mathsfit k}\,,$\vspace*{2mm}
      \label{item:symmbasicprops:rhoadd}
    \item $\weightvec {\mathsfit h}\in\frac 12 \weightlattice\,,$\vspace*{2mm}
      \label{item:symmbasicprops:rhoweight}
    \item $\weightvec {m\cdot \idsymm}=m\cdot\weightvecchar\,,$\vspace*{2mm}
      \label{item:symmbasicprops:rhogen}
    \item $\symbrack{\beta}{\mathsfit h(\beta)}-2\symbrack{\weightvec{\mathsfit h}}{\beta}\;
          \in\;2\mathbb Z$\vspace*{2mm}
      \label{item:symmbasicprops:rhoint}
    \item $\sumdescrootsrel s {\mathsfit h+ \mathsfit k}=\sumdescrootsrel s {\mathsfit h}+\sumdescrootsrel s {\mathsfit k}\,,$\vspace*{2mm}
      \label{item:symmbasicprops:thetaadd}
    \item $\symbrack{\sumdescrootsrel s {\mathsfit h}}{\beta}\;\in\;\mathbb Z\,,$\vspace*{2mm}
      \label{item:symmbasicprops:thetaint}
    \item $\sumdescrootsrel s {m\cdot \idsymm}=m\cdot \sumdescroots{s}\,,$\vspace*{2mm}
      \label{item:symmbasicprops:thetagen}
    \item $\sumdescrootsrel {ts} {\mathsfit h}=\sumdescrootsrel {t} {s{\mathsfit h}s^{-1}}+t\left({\sumdescrootsrel {s} {\mathsfit h}}\right)\,.$
      \label{item:symmbasicprops:thetarec}
\end{enumerate}
\ 
    \end{multicols}
\end{lem}

\begin{proof} Items {\em \ref{item:symmbasicprops:rhoadd}}, {\em \ref{item:symmbasicprops:rhogen}},  
    {\em \ref{item:symmbasicprops:thetaadd}}, {\em \ref{item:symmbasicprops:thetagen}}, and {\em \ref{item:symmbasicprops:thetarec}} are immediate from the definitions in \eqref{eq:defrhosymm} and \eqref{eq:defthetasymm} as well as the formulae  for $\sumdescroots s$ and $\weightvecchar$ in \eqref{eq:thetaviarho}.
    For the statement in {\em \ref{item:symmbasicprops:rhoweight}} note that $\symbrack{\weightvec {\mathsfit h}}{\breve\alpha_i}=d_i^{-1}\symbrack{\weightvec {\mathsfit h}}{\alpha_i}=\frac 12 d_i^{-1}\symbrack{\alpha_i}{ {\mathsfit h}(\alpha_i)}=\frac 12 \sum_jd_i^{-1}\symbrack{\alpha_i}{\alpha_j}{\mathsfit h}_{ji}=\frac 12 \sum_jA_{ij}{\mathsfit h}_{ji}\in\frac 12\mathbb Z$, where $\breve\alpha_i=d^{-1}_i\alpha_i$ is the usual coroot, $A_{ij}$ the Cartan data, and ${\mathsfit h}_{ji}$ are the matrix coefficients of $\mathsfit h$ in the $\sroots$-basis. 
    
    For $\beta=\sum_ib_i\alpha_i$ and using symmetry of $n_{ij}=\symbrack{\alpha_i}{\mathsfit h(\alpha_j)}$ the expression in {\em \ref{item:symmbasicprops:rhoint}} can be worked out as $2\sum_{i>j}b_ib_jn_{ij}\,+\,\sum_i(b_i^2-b_i)n_{ii}\,$, which is clearly in $2\mathbb Z\,$. For Item  
    {\em \ref{item:symmbasicprops:thetaint}}, note that both terms $\symbrack{s^{-1}(\beta)}{\mathsfit h(s^{-1}(\beta))}-2\symbrack{\weightvec{\mathsfit h}}{s^{-1}(\beta)}=\symbrack{\beta}{(s\mathsfit h s^{-1})(\beta)}-2\symbrack{s(\weightvec{\mathsfit h})}{\beta}$ and $\symbrack{\beta}{(s\mathsfit h s^{-1})(\beta)}-2\symbrack{\weightvec{s{\mathsfit h}s^{-1}}}{\beta}$ are in $2\mathbb Z$ by {\em \ref{item:symmbasicprops:rhoint}}. Thus, also their difference $2\symbrack{\sumdescrootsrel{s}{\mathsfit h}}{\beta}$ has to be in $2\mathbb Z\,$. 
\end{proof}

For an illustration of these notions, we briefly discuss the rank 2 root systems. We assume $A_{12}=\edgenum\in \{1,2,3\}$ depending on whether $\roots$ is of type $\LT{A}_2$\,, $\LT{B}_2$\,, or $\LT{G}_2$\,, as well as $A_{21}=-1\,$. A $\mathbb Z$-basis for $\symmmapspace{\roots}$ for any type is given by $\{\idsymm, s_1, s_2\}$ where $s_i=s_{\alpha_i}$ are the usual reflections along the simple roots. 
Note that $s_{12}=s_{\alpha_1+\alpha_2}=s_2s_1s_2=(\edgenum-2)s_2-s_1$\,. So, $\{\idsymm, s_1, s_{12}\}$ would be an alternate basis for $\edgenum\in\{1,3\}$ but not for $\edgenum=2\,$. 
It follows from elementary calculations that $\,\mathrm{End}({\mathbb Z}^{\sroots})/\symmmapspace{\roots}$ is $\mathbb Z\oplus \mathbb F_2\,$ for $\edgenum=2$ but only $\mathbb Z$ for $\edgenum\in\{1,3\}\,$. 

Since $\weightvec{\mathsfit h}$ is additive in ${\mathsfit h}$ is suffices to compute the vector for the basis elements $\{\idsymm, s_1, s_2\}$, for which we already know that $\weightvec{\idsymm}=\weightvecchar\,$. The other two elements are computed as
\begin{align*}
    \weightvec{s_1}&=-\tfrac{1}{2}( 1+(\edgenum-1)(2\edgenum-3))\alpha_1\,-\,\tfrac{1}{2}(\edgenum-1)^2\alpha_2\,,\\
    \weightvec{s_2}&= - (\edgenum-1)^2\alpha_1-\tfrac{1}{2}( 1+(\edgenum-1)^2)\alpha_2\,.
\end{align*}  
Thus, even in the $\edgenum=1$ ($\LT{A}_2$) case we have $\weightvec{s_i}=-\tfrac{1}{2}\alpha_i\not\in\weightlattice\,$, as opposed to $\weightvecchar\,$. Similarly, for the $\sumdescrootsrel{s}{\mathsfit h}$ is suffices to consider $\mathsfit h\in\{s_1,s_2\}\,$. The values $\sumdescrootsrel{s_j}{s_i}$ are explicitly computed below: 
\begin{align*}
\sumdescrootsrel{s_1}{s_1}&=s_1(\alpha_1)=-\alpha_1\,, & \sumdescrootsrel{s_2}{s_1}&=s_1(\alpha_2)=\edgenum\alpha_1 +\alpha_2\,,\\
\rule{0mm}{3mm}\sumdescrootsrel{s_2}{s_2}&=s_2(\alpha_2)=-\alpha_2 \,,
& 
\sumdescrootsrel{s_2}{s_1}&=s_2(\alpha_1)+
\mbox{${\binom{\edgenum}{2}}$}
\alpha_1=\tfrac{1}{2}(\edgenum^2-\edgenum+2)\alpha_1 + \edgenum\alpha_2\,.
\end{align*}
From these one can obtain $\sumdescrootsrel{s}{s_i}$ using the 
recursion in Item {\em \ref{item:symmbasicprops:thetarec}} of Lemma~\ref{lem:symmbasicprops} together with the explicit formulae $s_is_js_i=(\edgenum-2)s_i-s_j$ for $i\neq j$ for the adjoint action of $\Weyl$ on $\langle s_1,s_2\rangle\subset \symmmapspace{\roots}\,$. In particular, we see that in this example 
$\sumdescrootsrel{s}{\mathsfit h}\in\roots\,$ for all $s\in\Weyl$ and $\mathsfit h\in\symmmapspace{\roots}\,$. 
Consider the case $\edgenum=1$ ($\LT{A}_2$)\,, for which $\sumdescrootsrel{s_i}{s_j}=s_i(\alpha_j)=s_j(\alpha_i)$\,, suppose $\mathsfit h = n_1s_1+n_2s_2\,$, and set $\beta=n_1\alpha_1+n_2\alpha_2\,$. The suggested computation then yields that 
\begin{align*}
    \sumdescrootsrel{s}{\mathsfit h}=s(\beta) \quad\mbox{if $\,\len s\,$ is odd} 
       \qquad\mbox{and }\qquad 
    \sumdescrootsrel{s}{\mathsfit h}=0 \quad\mbox{if $\,\len s\,$ is even.} 
\end{align*} 

We defer further discussion of $\roots$-symmetric maps to future work. Open questions include whether one can always find a subset $\mathscr T\subset \proots$ such that $\{s_\alpha:\alpha\in\mathscr T\}$ is a $\mathbb Z$-basis for $\symmmapspace{\roots}$ and how to use this to efficiently compute the $\weightvec{\mathsfit h}$ and 
$\sumdescrootsrel{s}{\mathsfit h}$ in generality. Other questions are whether $\sumdescrootsrel{s}{\mathsfit h}\in\roots$ also for higher ranks and what the possible torsion factors of $\,\mathrm{End}({\mathbb Z}^{\sroots})/\symmmapspace{\roots}$ are.

\subsection{Rank 2 Subsystems}\label{subsec:rank2coxsys}
     
 As a simple but useful example, consider the Coxeter subsystem for the generator set $\{s_i,s_j\}\,$, given two distinct but fixed indices $i,j\in\{1,\ldots,n\}\,$. The respective Weyl subgroup $\Weylrest{i,j}$ is a dihedral group of order $2m_{ij}\in\{4,6,8,12\}\,$. The longest element $\longtwoweyl{ij}$ is given by the expressions in \eqref{eq-ATrels} and, hence, of length $\len {\longtwoweyl{ij}}=m_{ij}$\,.  
 The element of maximal length has two word presentations given by 
\begin{equation}\label{eq:maxwordrk2}
    \wordsetmaxrest{i,j}=\{\longtwoword{ij},\longtwoword{ji}\}=\Weylpres^{-1}(\longtwoweyl{ij})\,,
\end{equation}
where the words with $m_{ij}$ letters are $\longtwoword{ij}=w_iw_j\ldots$ and $\longtwoword{ji}=w_jw_i\ldots\,$. 
Thus, Matsumoto's Theorem says that two reduced $w,w'\in\wordset$ represent the same 
$s=\Weylpres(w)=\Weylpres(w')$ if $w'$ is obtained from $w$ by successively swapping sub words $\longtwoword{ij}\leftrightharpoons\longtwoword{ji}\,$.

For any element $s\in\Weylrest{i,j}$ with $\len s<m_{ij}$ there is exactly one reduced word 
in $\Weylpres^{-1}(s)\subset\wordsetrest{i,j}$ representing $s$. To simplify notation here and in later applications, let $m=m_{ij}$ as well as 
\begin{equation}\label{eq:def:abwords}
    a_r=\longtwoword{ij}[1,r]=
\begin{cases}
    (w_iw_j)^{\frac r2} & r \mbox{ even}   \\
     (w_iw_j)^{\frac {r-1}2}w_i & r \mbox{ odd}
\end{cases}
\qquad
b_r=\longtwoword{ji}[1,r]=
  \begin{cases}
    (w_jw_i)^{\frac r2} & r \mbox{ even}   \\
    (w_jw_i)^{\frac {r-1}2}w_j & r \mbox{ odd}
\end{cases}\;
\end{equation}
for $0\leq r\leq m$, so, $a_m=\longtwoword{ij}$ and $b_m=\longtwoword{ji}\,$ and $\len {a_r}=\len{b_r}=r\,$. The $a_r$ and $b_r$ provide a distinct enumeration of elements in $\wordsetrest{i,j}$ with the exception of the empty word $a_0=b_0=\emptyword\,$. This organization 
of words provides a succinct way of expressing the complementary word relations introduced in  
Section~\ref{subsec:ComplWords}.

\begin{lem} \label{lm:rk2complw}
With notation as above, (non-empty) complementary words are uniquely given for all finite Coxeter types by 
    $$
    a_r\wcomplrel b_{m+1-r}\;. 
    $$
\end{lem} 

\begin{proof} Note that for $m=3$ we have $\longtwoword{ij}^*=\longtwoword{ij}$ and 
$\dyninv(\longtwoword{ij})=\longtwoword{ji}$\,, whereas for $m\in\{4,6\}$ we have 
$\longtwoword{ij}^*=\longtwoword{ji}$ and $\dyninv(\longtwoword{ji})=\longtwoword{ji}\,$. Thus,  $\longtwoword{ij}^\winvchar=\dyninv(\longtwoword{ij}^*)=\longtwoword{ji}$ for all Lie types, which implies the statement by  
\eqref{eq:complmaxword}. Uniqueness is also clear since $b_{m+1-r}$ is the only word matching the length constraint and the constraint on the final letter index. 
\end{proof}

In order to describe the total orders on rank 2 root systems implied by the two words $\longtwoword{ij}$ and 
$\longtwoword{ji}$ we choose $\alpha_i$ to be the shorter and $\alpha_j$ to be the longer root. Thus, in the conventions from 
Section~\ref{subsec:rootweyl}, the Cartan data is given as $-A_{ji}=d_i=1$ and
$-A_{ij}=d_j=\edgenum\in\{1,2,3\}\,$ depending on Lie type. Following the definitions in 
Section~\ref{subsec:descroots}, the maximal word $\longtwoword{ji}$ imposes a total order 
$\letwt{\longtwoword{ji}}$ with maximal chains given by the sequence of roots 
$\beta_r=\wordroot(b_r)=\wordroot(\longtwoword{ji}[1,r])$. Their explicit expressions in terms of simple roots $\alpha_i$ and $\alpha_j$ are as follows for Lie types $\LT{A}_2$\,, 
$\LT{B}_2$\,, and $\LT{G}_2$\,, respectively.
\begin{equation}\label{eq:rank2rootorder}
   \begin{aligned}
\edgenum&=1 \qquad &&
\alpha_j\letwt{\longtwoword{ji}}\alpha_i+\alpha_j\letwt{\longtwoword{ji}}\alpha_i\\
\rule{0mm}{4.5mm}\edgenum&=2 \qquad   &&
\alpha_j\letwt{\longtwoword{ji}}\alpha_i+\alpha_j\letwt{\longtwoword{ji}}2\alpha_i+\alpha_j\letwt{\longtwoword{ji}}\alpha_i\\
\rule{0mm}{4.5mm}
\edgenum&=3 \qquad   &&
\alpha_j\letwt{\longtwoword{ji}}\alpha_i+\alpha_j
        \letwt{\longtwoword{ji}}3\alpha_i+2\alpha_j
        \letwt{\longtwoword{ji}}2\alpha_i+\alpha_j
        \letwt{\longtwoword{ji}}3\alpha_i+\alpha_j
        \letwt{\longtwoword{ji}}\alpha_i
\end{aligned} 
\end{equation} 
By Corollary~\ref{cor:revwordorder} and the fact that $\longtwoword{ij}=\longtwoword{ji}^\winvchar$
the total $\letwt{\longtwoword{ij}}$ order imposed by the other maximal reduced word $\longtwoword{ij}$ is exactly the opposite one.

Often  questions for general Coxeter systems can be reduced to ones of rank 2 situations. This is illustrated in the following method of extracting from some $s\in\Weyl$ a maximal element $r\leqLB s$ that belongs to such a subsystem. The proof here extends the argument used in the proof of Proposition~1.8 of \cite{Lu90a} to also the non-simply laced Lie types. 

\begin{lem}\label{lem:Weyl_ijextr}
Suppose $s\in\Weyl$ and $1\leq i\leq n$ are such that $\len{s\cdot s_i}=\len{s}+1\geq 2\,$. Then there exists $j\neq i$,
$t\in\Weyl$ and $y\in\wordsetrest{i,j}$\,, such that for $r=\Weylpres(y)$ we have 
$$
s=t\cdot r\;,\quad \len s=\len t + \len r\;,\quad 1\leq \len r < m \;, \quad t(\alpha_i), t(\alpha_j)\in\proots\;, \quad \mbox{and} \quad \tau(y)=j\,.
$$
Moreover, if $s(\alpha_i)=\alpha_k\in\sroots$, then $r\cdot s_i=\longtwoweyl{ij}\,$ so that $r(\alpha_i)\in\{\alpha_i,\alpha_j\}\,$ and, hence, either $t(\alpha_i)=\alpha_k$ or $t(\alpha_j)=\alpha_k$\,. 
\end{lem}

\begin{proof}
   If $u\in\wordset$ with $\Weylpres(u)=s$ we have by the first assumption that $u\cdot w_i$ is reduced as well. Thus, by Lemma~\ref{lem:red=pos}, we have 
   $\wordroot(u\cdot w_i)=s(\alpha_i)\in\proots\,$. Now let $j=\tau(u)$ so that $s=s'\cdot s_j$ with $s'=\Weylpres(u^\flat)$\,. Therefore,
   $\wordroot(u)=s'(\alpha_j)=-s'(s_j(\alpha_j))=-s(\alpha_j)\in\proots$ again by Lemma~\ref{lem:red=pos}, which implies
   $s(\alpha_j)\in\nroots$ and, hence, $j\neq i$\,. 

   Next, let $t\in\Weyl$ be an element of minimal length in the coset $\,s\cdot \Weylrest{ij}$\,. Since 
   $s\cdot s_j=\Weylpres(u^\flat)$ we have $\len {s\cdot s_j}<\len s$ so that $\len t < \len s\,$ and $\len r\geq 1\,$. 
   Minimality of $t$ further implies that $\len {t\cdot s_i}>\len t$ and $\len {t\cdot s_j}>\len t\,$.
   Thus, if $x\in\wordset$ is a reduced word representing $t$ also $x\cdot w_i$ and $x\cdot w_j$ are reduced, implying positivity of both $\wordroot(x\cdot w_i)=t(\alpha_i)\,,\, \wordroot(x\cdot w_j)=t(\alpha_j) \in \proots\,$. 

    Since $\descroots r\subset \nnN\alpha_i+\nnN\alpha_j$ this also implies that $t(\descroots r)\subseteq \proots$. By Corollary~\ref{cor:sumdescroots} we thus obtain that $\len s=\len t + \len r\,$. 

    Observe further that $r(\alpha_i)=n_i\alpha_i+n_j\alpha_j$ with either both $n_i, n_j>0$ if $r(\alpha_i)\in\proots$ or both $n_i, n_j<0$ if $r(\alpha_i)\in\nroots$. The latter case is not possible since it would imply $s(\alpha_i)=n_it(\alpha_i)+n_jt(\alpha_j)\in\nroots\,$ by the previous observation. Now, $r(\alpha_i)\in\proots$ implies that a reduced word representing $r$ ends in $w_j$ and that $\len r<m\,$. 

    Finally, the condition $s(\alpha_i)=\alpha_k=n_it(\alpha_i)+n_jt(\alpha_j)$ implies that either $t(\alpha_i)=\alpha_k$\,, with $(n_i,n_j)=(1,0)$ and hence $r(\alpha_i)=\alpha_i$\,, or 
    $t(\alpha_j)=\alpha_k$\,, with $(n_i,n_j)=(0,1)$ and hence $r(\alpha_i)=\alpha_j\,$. Given that $\len r\geq 1$ this now entails that $\len r=m-1$ so that $y\cdot w_i\in \wordsetmaxrest{i,j}\,$. 
\end{proof}

\medskip  

\section{Quantum Groups at Generic Parameter \texorpdfstring{$q$}{q}}\label{sec:quantumgen}
 
We recall the standard definitions of quantum groups as defined, for example, in \cite{Ji85, CP95}. Various ground rings are introduced over which subalgebras and variants of the quantum groups can be defined. Throughout this section we fix a simple Lie algebra $\mathfrak g$ with root system $\roots$ and Cartan data $A_{ij}$ as introduced in the previous section.

\subsection{Quantization Rings, Quantum Numbers, Integral Subrings} \label{subsec:q-nums}
In their original definitions, such as in \cite{Dri87}, quantum groups are obtained as deformations of the universal enveloping algebra of a Lie algebra $\mathfrak g$ along a formal parameter $\hbar\,$. The natural ground ring in this context is the ring of formal power series $\Qh\,$, which contains $q=e^\hbar=\sum_{n=0}^\infty \frac{1}{n!}\hbar^n\,$. The parameter identification yields a ring embedding $\Qq^1\hookrightarrow \Qh$\,, where $\Qq^1$ denotes the subring of rational functions without poles at $q=1$. It extends to an embedding $\Qq\hookrightarrow \Qh_\hbar$\,, where the localization $\Qh_\hbar$ is given by formal Laurent series with finite numbers of negative powers. We continue our discussion with various frequently used quantities and formulae in $\mathbb Z[q,q^{-1}]$.

For a positive root $\alpha\in\proots$ in the $\Weyl$-orbit of $\alpha_i\in\sroots$\,, let $d_\alpha=d_{i}$ and $q_\alpha=q_i=q^{d_i}$\,. Given $a,b\in\mathbb N$\,, quantum numbers, factorials, and binomials are denoted with respect to this data as 
\begin{align}\label{eqn:quantumnumbers}
    [a]_\alpha=\dfrac{q_\alpha^a-q_\alpha^{-a}}{q_\alpha-q_\alpha^{-1}}\,,
    &&
    [a]!_\alpha=\prod_{j=1}^a \dfrac{q^{j}_\alpha-q^{-j}_\alpha}{q_\alpha-q^{-1}_\alpha}\,,
    &&
    \mbox{and}
    &&
    \qbin{a+b}{a}{\alpha}
    =\dfrac{[a+b]!_\alpha}{[a]!_\alpha[b]!_\alpha}\,,
\end{align}
all of which are valued in $\mathbb Z[q_\alpha,q^{-1}_\alpha]\subseteq \mathbb Z[q,q^{-1}]\,$. The subscript $\alpha$ is suppressed if $q_\alpha=q$, meaning $d_i=d_\alpha=1$. We will frequently also use the abbreviation $[n]_i=[n]_{\alpha_i}$ for simple roots and analogous notation for factorials and binomial coefficients. The binomial recursion below is easily verified and frequently used.  
\begin{equation}\label{eq:qbin-recurs}
    \qbin{n+1}{t}{\alpha}=q_\alpha^{\pm t}\qbin{n}{t}{\alpha}+q_\alpha^{\mp(n+1-t)}\qbin{n}{t-1}{\alpha}
\end{equation}
The $q$-binomial coefficients arise naturally in the
respective binomial formula for $q$-commuting generators. Specifically,
assume the algebra $\mathcal H_q$ over $\mathbb Z[q,q^{-1}]$ is generated  by  $U$ and $V$ subject to the relation $VU=q^2UV$.
Then the following is easily
proved by induction,
\begin{equation}\label{eq:q-binom-form}
    (U+V)^a\,=\,\sum_{k=0}^aq^{-k(a-k)}\qbin{a}{k}{}U^{k}V^{a-k}\,=\,\sum_{k=0}^aq^{k(a-k)}\qbin{a}{k}{}V^{k}U^{a-k}\,.
\end{equation} 
Adjoining an inverse $V^{-1}$ to $\mathcal H_q$ and multiplying \eqref{eq:q-binom-form} by $V^{-a}$ from the right one readily obtains, after the substitution $z=V^{-1}U$, the following variant of Gauss's formula,
\begin{equation}\label{eq:q-combin-form}
    \sum_{k=0}^aq^{-k(a+1)}\qbin{a}{k}{}z^k\,=\,\prod_{t=1}^a(q^{-2t}z+1)\,=\,q^{-\binom{a+1}{2}}\prod_{t=1}^a(q^t+q^{-t}z)\,.
\end{equation} 

For convenience in later computations we express ratios of factorials for different root lengths $d_j\in\{1,2,3\}$ as 
\begin{equation}\label{eq:fracfactorials}
    \frac{[d_j k]!}{[k]!_j}\,=\,[d_j]^k\cdot \qfacrelA k{d_j}
    \qquad \mbox{and}\qquad
    \frac{[k]!_j}{[k]!}\,=\,[d_j]^{-k}\cdot \qfacrelB k{d_j}\,, 
\end{equation}
where we use the following notation for elements in $\mathbb Z[q,q^{-1}]$ ,
\begin{equation}\label{eq:factorialratios}
\qfacrelA k\edgenum \,=\,\prod_{s=0}^{k-1}\prod_{t=1}^{\edgenum-1}[s\edgenum + t]
    \qquad \mbox{and}\qquad
    \qfacrelB k\edgenum\,=\,\prod_{s=1}^{k}\frac{q^{\edgenum s}-q^{-\edgenum s}}{q^s-q^{-s}}\,.
\end{equation}
Note that $\qfacrelA k 1=1=\qfacrelB k 1$ for all $k\in\mathbb N\,$. On occasion, we will also use notation for quantum multinomial coefficients in $\mathbb Z[q,q^{-1}]$ given by
\begin{equation}\label{eq:def-qmultinom}
    \qbin{a_1+a_2+\ldots+a_r}{a_1\,,\,a_2\,,\,\ldots\,,\,a_r}{\alpha}
=
 \frac{[a_1+a_2+\ldots+a_r]!_{\alpha}}{[a_1]!_\alpha[a_2]!_\alpha\ldots[a_r]!_\alpha}
=\qbin{a_1+\ldots+a_r}{a_r}{\alpha}\qbin{a_1+\ldots+a_{r-1}}{a_{r-1}}{\alpha}\ldots
\qbin{a_1+a_2}{a_2}{\alpha}\,.
\end{equation} 

Concluding, we introduce the following notation for subrings of $\Qq\,$, over which various versions of quantum groups can be defined. Clearly, all of these are integral domains and have $\Qq$ as a common field of fractions. However, as opposed to $\Qq\,$, they have well-defined specializations to roots of unity of large enough orders.  
\begin{equation}
    \begin{aligned}\label{eq:defgenrings}
    \Zqq&=\mathbb Z[q,q^{-1}] & \quad & \Zqqn{n}=\mathbb Z\textstyle{\left[q,q^{-1},\frac 1 {[n-1]!}\right]}\\
    \rule{0mm}{7mm}\Zqqv&=\mathbb Z\textstyle{\left[{q,q^{-1},\frac 1{q-q^{-1}}}\right]} && \Zqqvn{n}=\mathbb Z\textstyle{\left[q,q^{-1},\frac 1 {q-q^{-1}},\frac 1 {[n-1]!}\right]}
\end{aligned}
\end{equation} 

\subsection{Quantum Group Generators and Relations}  \label{subsec:qgroups-genrel}
For Cartan data as above, the quantum universal enveloping algebra $\Uhg$ is defined in \cite{Dri87,CP95} as the algebra over $\Qh$ topologically generated by elements $E_i,F_i,H_i$ for $1\leq i\leq n$ subject to relations 
\begin{align}
    &[H_i,H_j]=0\,,
    &&
    [E_i,F_j]=\delta_{ij}\dfrac{e^{d_i\hbar H_i}-e^{-d_i\hbar H_i}}{q_i-q_i^{-1}}\,, \label{eq:Lierels1}
    \\
    &[H_i,E_j]=A_{ij}E_j\,, 
    &&
    [H_i,F_j]=-A_{ij}F_j\,,\label{eq:Lierels2}
    \end{align}
    \vspace{-.5em}
    \begin{align}
    \sum_{r+s=1-A_{ij}}(-1)^s\qbin{1-A_{ij}}{s}{i}E_i^rE_jE_i^s=0 \quad (\text{for } i\neq j)\,,\label{eq:SerreE}\\
    \sum_{r+s=1-A_{ij}}(-1)^s\qbin{1-A_{ij}}{s}{i}F_i^rF_jF_i^s=0 \quad (\text{for } i\neq j)\,. \label{eq:SerreF}
\end{align}
Although \eqref{eq:Lierels1} and the other relations are well-defined over $\Qh$ we  extend the ground ring of $\Uhg$ to $\Qh_\hbar\,$ for convenience. If we denote as usual $K_i=e^{d_i\hbar H_i}$ the above imply the following relations:
\begin{align}
    &K_iK_j=K_jK_i\,,
    &&
    [E_i,F_j]=\delta_{ij}\dfrac{K_i-K_i^{-1}}{q_i-q_i^{-1}}\,,\label{eq:EFcomm}
    \\
    &K_iE_jK_i^{-1}=q_i^{A_{ij}}E_j=q^{\symbrack{\alpha_i}{\alpha_j}}E_j\,, 
    && 
    K_iF_jK_i^{-1}=q_i^{-A_{ij}}F_j=q^{-\symbrack{\alpha_i}{\alpha_j}}F_j\label{eq:Kcomm}\,. 
\end{align}

Observing $(q_i-q_i^{-1})=(q-q^{-1})[d_i]\,$, we note that all relations (\ref{eq:SerreE}-\ref{eq:Kcomm}) are indeed over $\Zqqvn{\dpone}$ where $\dpone\in\{2,3,4\}$ is as in \eqref{eq:Ddef} and we view $\Zqqvn{\dpone}\subset \Qq\subset \Qh_\hbar\,$. We thus define $\Uqg$ as the $\Zqqvn{\dpone}$-algebra generated by the elements $\{E_i,F_i,K_i:1\leq i\leq n\}$ and subject to relations (\ref{eq:SerreE}-\ref{eq:Kcomm}). If $\mathfrak g$ is a general simple Lie algebra or its type is clear from context, we will usually write $\Uq$ instead. 

Additionally, we denote by $\Uq^+$ the subalgebra generated by only the $E_i$\,, $\Uq^-$ the subalgebra generated by only the $F_i$\,, which may be considered as algebras over the subring $\Zqqn{\dpone}$ of $\Zqqvn{\dpone}$\,. The algebra $\Uq^0$ generated by only the $K_i$ may be considered over $\Zqq$. Further, we set $\Uq^{\geqzero}=\Uq^+\Uq^0$ and $\Uq^{\leqzero}=\Uq^-\Uq^0$ the respective subalgebras obtained from two sets of generators, which we also take to be defined over $\Zqqn{\dpone}\,$ by default. 

Suppose $\,f\!:\Zqqn{\dpone}\rightarrow\Zgen$ or $\,f^{\,\,\vaccent}\!:\Zqqvn{\dpone}\rightarrow\Zgenv$ is a ring homomorphism where $\Zgen$ or $\Zgenv$ is commutative with unit. We may define quantum algebras over ground rings $\Zgen$ or $\Zgenv$ by respective extensions of scalars of the original algebras, viewed as $\Zqqn{\dpone}$ or $\Zqqvn{\dpone}$ modules. Specifically, for any $\,\utypechar\in\{+,-,0,\geqzero, \leqzero\}\,$ we denote 
\begin{equation}\label{eq:ringchange}
    \Uqre{\Zgenv}=f^{\,\,\vaccent}_{!}\Uq=\Uq\otimes_{\Zqqvn{\dpone}}\Zgenv
    \qquad\mbox{and}\qquad
    \Uqre{\Zgen}^{\utypechar}=f_{!}\Uq^{\utypechar}=\Uq^{\utypechar}\otimes_{\Zqqn{\dpone}}\Zgen \,.
\end{equation} 

The further abbreviated notation $\Uqre{\mathbb Q}$ and $\Uqre{\mathbb Q}^\utypechar$ is used for algebras over $\mathbb Q(q)$ derived from the inclusions of $\Zqqvn{\dpone}$ and $\Zqqn{\dpone}$ in $\mathbb Q(q)$\,. Similarly, for $m\geq \dpone$ we write
$\Uqre{m}$ and $\Uqre{m}^\utypechar$ for the  inclusions of $\Zqqvn{\dpone}$ and $\Zqqn{\dpone}$ in $\Zqqvn{m}$ and $\Zqqn{m}\,$, respectively. That is, the ground rings of $\Uqre{m}$ and $\Uqre{m}^\utypechar$
contain also elements $\frac 1 {[m-1]!}$.  

It is easy to see that the free abelian group algebra $\Zqq[\{K_i^{\pm 1}\}]$ embeds in the polynomial algebra $\Qh_\hbar[\{H_i\}]$ via $K_i\mapsto e^{d_i\hbar H_i}\,$. Using the PBW results below and in \cite{Ro89,CP95} it further follows that the respective algebra homomorphism $\Uq\rightarrow\Uhg$ is injective.

An important variant of quantum groups at roots of unity are Lusztig's divided power algebras $\LUq$, see \cite{lu,Lu90a,lu90b}. They are defined as the $\Zqq$-subalgebras of $\Uqre{\mathbb Q}$ generated by elements $E_i^{(k)}=\frac 1 {[k]!_i}E_i^k$ and $F_i^{(k)}=\frac 1 {[k]!_i}F_i^k$ as well as  certain additional rational expressions in the $K_i$ generators.

We denote the respective subalgebras by  $\LUq^{+}$, $\LUq^{-}$, and $\LUq^{0}\,$. Although they are not discussed explicitly in this paper, we will refer to these algebras in some of the proofs below as we adopt arguments from \cite{Lu90a,lu90b,lu}.

\subsection{Hopf Algebra Structures} Coalgebra structures on $\Uhg$ and $\Uq$ are given by requiring $K_i$ to be group-like, $H_i$ to be primitive, and $E_i$ and $F_i$ to be {\em skew-primitives} via the  co-relations
\begin{align}\label{eq:corels_gen}
		\Delta(E_i)&=E_i\otimes K_i+1\otimes E_i & \mbox{and} && \Delta(F_i)&=F_i\otimes 1+ K_i^{-1}\otimes F_i\,.
	\end{align}
In this sense we have that the inclusions of the quantum algebras above, such as $\Borq\hookrightarrow\Uq\hookrightarrow\Uhg\,$, are inclusions of Hopf algebras. A counit is given by
$\varepsilon(E_i)=\varepsilon(F_i)=\varepsilon(H_i)=0$ and $\varepsilon(K_i)=1$. For later computations we also record here the coproducts for powers of generators, which easily follow from \eqref{eq:Kcomm} and \eqref{eq:corels_gen}.
\begin{align}\label{eq:corels_pwrs_gen}
		\Delta(E_i^m)&=\sum_{s=0}^mq_i^{-s(m-s)}\qbinsmall{m}{s}_iE_i^s\otimes K_i^sE_i^{m-s}  & && \Delta(F_i^m)&=\sum_{s=0}^mq_i^{-s(m-s)}\qbinsmall{m}{s}_iK_i^{-s}F_i^{m-s}\otimes F_i^{s}
	\end{align}

The antipode acts on the Cartan generators as $S(K_i)=K_i^{-1}$ and $S(H_i)=-H_i$ and is given on the other generators as  
\begin{align}\label{eq:antipode_gen}
    S(E_i)&=-E_iK_i^{-1}=-q_i^2K_i^{-1}E_i & \mbox{and} && S(F_i)&=-K_iF_i=-q_i^{-2}F_iK_i\,. 
\end{align}
Again, it will be useful to also state the antipode of powers of generators.
\begin{align}\label{eq:antipode_pow}
    S(E_i^m)&=(-1)^mq_i^{-m(m-1)}E_i^mK_i^{-m} &   S(F_i^m)&=(-1)^mq_i^{m(m-1)}K_i^{m}F_i^m\,. 
\end{align}
Clearly, $\Uq^{\geqzero}$ and $\Uq^{\leqzero}$ form Hopf subalgebras of $\Uq$\,, and $\Uq^{0}=\Zqqn{\dpone}[\{K_i^{\pm}\}]$ is isomorphic as a Hopf algebra to the group algebra of $\mathbb Z^n\,$. 

Recall \cite{Sw69, Mon93b, Ra12} that a coalgebra $C$ over a field $\mathbb k$ is called {\em pointed} if all its (co)simple subcoalgebras are 1-dimensional or, equivalently, if its coradical $C_0$ coincides with the subcoalgebra $\mathbb k[\GrLi{C}]$ of group-like elements. Pointed coalgebras and Hopf algebras have been extensively studied, though most results assume the ground ring to be a field.

Suppose $\mathbb k$ is a field and $\vartheta:\Zqqvn{\dpone}\rightarrow \mathbb k\,$ a ring homomorphism. As noted in \eqref{eq:ringchange}, this yields quantum algebras $\Uqre{\mathbb k}$ and  $\Uqre{\mathbb k}^{\utypechar}$ over the field $\mathbb k$ for any $\utypechar\in\{0, \geqzero, \leqzero\}\,$. Besides the case $\mathbb k=\mathbb Q(q)$ given by the obvious inclusions, the following result will also be of interest in the case of a cyclotomic field $\mathbb k=\mathbb Q(\zeta)$\,.

\begin{prop}[{\cite[Theorem 2.2]{Mon93a}}]\label{prop:pointed}
The Hopf algebras $\Uqre{\mathbb k}\,$, $\Uqre{\mathbb k}^{\geqzero}\,$, and $\Uqre{\mathbb k}^{\leqzero}\,$ as defined above are pointed. Their common coradical is  $\Uqre{\mathbb k}^0\,$.
\end{prop}

\noindent The proof uses coalgebra filtrations $\{A_n\}$ of the quantum algebras, where $A_0=\Uqre{\mathbb k}^0$\,. Set $A_1=A_0\oplus \sum_i^\oplus E_iA_0 \oplus F_iA_0\,$ for $\Uqre{\mathbb k}\,$
as well as $A_1=A_0\oplus \sum_i^\oplus E_iA_0$ for $\Uqre{\mathbb k}^{\geqzero}\,$ and analogous for $\Uqre{\mathbb k}^{\leqzero}\,$. The subspaces are then given as $A_n=(A_1)^n\,$, which clearly exhaust the entire algebra. Lemma~5.3.4 in \cite{Mon93a}, for example, now implies that $C_0$ is contained in and thus equal to 
$\Uqre{\mathbb k}^0\,$. Pointedness now follows immediately from $C_0=\Uqre{\mathbb k}^0\,$, which is then also equal to the total subcoalgebra spanned by group-like elements.

\section{Gradings and Automorphisms on Quantum Groups}\label{sec:gradingauts}

Several automorphisms and anti-automorphsisms will play an important role in constructions throughout this paper. They include, besides the antipode,  Lusztig's Artin group actions, involutions, and \Kscaleword transformations. Additionally, we will refer to two types of gradings in several of our later the proofs. This section will collect basic properties of these automorphisms and gradings as well as establish various useful relations between them.

\subsection{Gradings}\label{subsec:gradings}  Recall that a grading of an algebra $A$ by a monoid $G$ entails a decomposition $A=\oplus_{g\in G}A_g$ with $A_g\cdot A_h\subseteq A_{gh}\,$. We call elements $b\in A_g\setminus\{0\}$ for some $g\in G$ {\em homogeneous} with respect to such a given grading. If $A$ is defined in terms of generators and relations, a grading on the free algebra of generators descends to $A$ if the relations are homogeneous elements. 
Let  $G$ be a monoid and  $S$ any set.
In the following discussion, we write $G^S$ to mean interchangeably additive combinations $\sum_{s\in S}c_ss$ as well as maps $c:S\rightarrow G\,:\,s\mapsto c_s\,$.

A grading in $G=\nnN^{\sroots}$ with $\nnN=\mathbb N\cup\{0\}$ on the algebra $\Uq^+$ is given by defining 
$(\Uq^+)_\mu$ as the subspace spanned by $E_{i_1}\ldots E_{i_m}$ such that $\mu=\sum_{s=1}^m\alpha_{i_s}\in \nnN^{\sroots}\,$. It is clear that the relations in \eqref{eq:SerreE}-\eqref{eq:Kcomm} are homogeneous with respect to this grading, which is therefore well-defined on $\Uq^+$. For a homogeneous element $b\in (\Uq^+)_\mu$ we sometimes also write $\wgradp(b)=\mu$.

An analogous grading in $\nnN^{\sroots}$ is given on $\Uq^-$ by the function 
$\wgradn$ from the set of homogeneous elements in $\Uq^-$ to $\nnN^{\sroots}$\,. 
These extend to a $\mathbb Z^{\sroots}$ grading on $\Uq$ with $(\Uq)_\mu$ spanned by 
elements $K_{j_1}\ldots K_{j_k}E^{\epsilon_1}_{i_1}\ldots E^{\epsilon_m}_{i_m}$ with $\mu=\sum_{s=1}^m\epsilon_s\alpha_{i_s}\,$ where $E_i^+=E_i$ and $E_i^-=F_i\,$, observing that both sides of the relation for $[E_i,F_j]$ are in $(\Uq)_{\alpha_i-\alpha_j}\,$. For a homogeneous element $a\in (\Uq)_\mu$ with respect to this weight grading we also write $\wgrad(a)=\mu\,$ 
Thus, if $b\in \Uq^+$ and $c\in \Uq^-$ are homogeneous elements we have $\wgrad(bc)=\wgrad(cb)=\wgradp(b)-\wgradn(c)$, which readily implies, for example 
$(\Uq^+)_{\mu_1}\cdot (\Uq^-)_{\mu_2}\subseteq (\Uq)_{\mu_1-\mu_2}$\,.  

For an element 
$\nu=\sum_ic_i\alpha_i\,\in\,\mathbb Z^{\sroots}\,$ in the root lattice, we denote a basis element of $\Uq^0$ by
\begin{equation}\label{eq:Kexpdef}
    K^\nu=K_1^{c_1}\ldots K_n^{c_n}\;. 
\end{equation}
Suppose $b\in \Uq$ is homogeneous with $\mu=\wgrad(b)\in \mathbb Z^{\sroots}$\,. It is easy to derive the following commutation relation,
where the inner form $\symbrack{\,\cdot\,}{\,\cdot\,}$ is naturally extended to a biadditive form on $\mathbb Z^{\sroots}\supset \roots\,$,
\begin{equation}\label{eq:Kwgrad} 
     K^\nu b  = q^{\symbrack{\nu}{\wgrad(b)}}bK^\nu \,.
\end{equation} 

Additional gradings in $\mathbb F_2^{\sroots}\,$, where $\mathbb F_2=\mathbb Z/2\mathbb Z=\{0,1\}\,$, are artifacts of the $q\leftrightharpoons q^{-1}$-symmetric definition of quantum numbers and Cartan elements for quantum groups. Two such gradings are defined on generators as follows:
\begin{align}
\ztgradp(E_i)&=\alpha_i\,, & \ztgradp(K_i)&=\alpha_i\,, & \ztgradp(F_i)=0\,,\label{eq:Topgrad}\\
\ztgradn(E_i)&=0\,, & \ztgradn(K_i)&=\alpha_i\,, & \ztgradn(F_i)=\alpha_i\,.\label{eq:Tomgrad}
\end{align}

As before, it is easy to check that relations \eqref{eq:SerreE}-\eqref{eq:Kcomm} are homogeneous with respect to both $\ztgradp$ and $\ztgradn$\,, which are thus well-defined on $\Uq$\,. For example, $\ztgradp(E_iF_i)=\ztgradp(F_iE_i)=\ztgradp(K_i)
=\ztgradp(K_i^{-1})=\alpha_i$ over $\mathbb F_2$ in the case of relation \eqref{eq:EFcomm}. We also note the following interactions of the antipode with these gradings, when evaluated on homogeneous elements:
\begin{align}
     \wgrad\circ S& =\wgrad\,, & \ztgradp\circ S& =\ztgradn\,, & \ztgradn\circ S& =\ztgradp \,.
\end{align}
    
Although the gradings are related by $\,\ztgradp-\ztgradn=\wgrad\mod 2$, it is useful to denote them separately.
Combining $\wgrad$ with either $\ztgradp$ or $\ztgradn$ yields a grading in $(\mathbb Z\times \mathbb F_2)^{\sroots}$ that provides a refinement of the usual weight decomposition, splitting summands further according to respective multi-parities of Cartan elements in the sense that $\ztgradp(K^\nu)=\nu \mod 2\,$.

For a discussion of the behavior of the coproduct with respect to these gradings denote by $U$ any of the Hopf algebras $\Uq$\,, $\Uq^{\geqzero}$ and $\Uq^{\leqzero}$. For $\mu\in\mathbb Z^{\sroots}$ and $\eta\in\mathbb F_2^{\sroots}$ write also  $U_\mu$ and $U^\eta$ for the graded components for $\wgrad$ and $\ztgradp$ respectively. The algebra $U\otimes U$ thus assumes natural bi-gradings with 
$(U\otimes U)_{(\mu_1,\mu_2)}=U_{\mu_1}\otimes U_{\mu_2}$ and 
$(U\otimes U)^{(\eta_1,\eta_2)}=U^{\eta_1}\otimes U^{\eta_2}$. A simple verification on generators now yields 
\begin{equation}\label{eq:gradscoprod}
    \Delta(U_\mu)\subseteq \sum_{\nu} U_{\mu-\nu}\otimes U_{\nu}
    \qquad \mbox{and}\qquad \Delta(U^\eta)\subseteq U\otimes U^\eta\,.
\end{equation}

Finally, let $b\in (\Uq^+)_\mu$ and  $c\in (\Uq^-)_{-\nu}$ be $\wgrad$-homogeneous elements with $\mu,\nu\in\nnN^{\sroots}$\,. It follows from  an easy inductive verification on generators $E_{i_1}\ldots E_{i_k}$ and $F_{j_1}\ldots F_{j_m}\,$ that 
the $(\mu,0)$ and $(0,\mu)$ graded components of $\Delta(b)$ as well as the $(-\nu,0)$ and $(0,-\nu)$ graded components of 
$\Delta(c)$ have the form
\begin{equation}\label{eq:xgradcoprod}
    \Delta(b)=b\otimes K^\mu +\,\ldots\, +1\otimes b \qquad\mbox{and}\qquad \Delta(c)=c\otimes 1   +\,\ldots \,+ K^{-\nu}\otimes c\,. 
\end{equation}

\medskip

\subsection{Che Transformations and Anti-Involutions}\label{subsec:Kscale}
A basic family of automorphisms can be obtained by multiplying the step generators by Cartan elements and units subject to certain consistency conditions. The main data characterizing such a 
{\em \Kscaleword }($\Kscalechar$) {\em transformation} is a $\roots$-symmetric map $\mathsfit h\in\symmmapspace{\roots}$ as introduced in Section~\ref{subsec:rootsymmaps}.  

Additionally, assume a homomorphism $u:\mathbb Z^{\sroots}\rightarrow \Zgen^{\star}$ from the root lattice to the multiplicative subgroup of units $\Zgen^{\star}=\{\pm q^k:k\in\mathbb Z\}$ of the ground ring. We define a map $\Kscale{u}{\mathsfit h}$ on generators as follows:  
\begin{align}\label{eq:defKscale}
    \Kscale{u}{\mathsfit h}(E_i)&=u(\alpha_i)K^{\mathsfit h(\alpha_i)}E_i\,,  &  
      \Kscale{u}{\mathsfit h}(F_i)&=u(\alpha_i)^{-1}F_iK^{-\mathsfit h(\alpha_i)}\,, &
      \Kscale{u}{\mathsfit h}(K_i)&=K_i\,.
\end{align}

The next proposition establishes the existence and basic properties of these automorphisms. As before, $\Uq^\utypechar$ may be $\Uq$\,, $\Uq^{\geqzero}$\,, or $\Uq^{\leqzero}$ defined over the respective minimal ground rings.
The $q$-exponent in the formula in {\em \ref{item:Kscaleprops:genform}} is integral by Lemma~\ref{lem:symmbasicprops}, so that the $q$-term is well-defined by
{\em \ref{item:symmbasicprops:rhoint}} of the same lemma. 
\begin{prop}\label{prop:Kscaleprops}
Let $\mathsfit h, \mathsfit k \in \symmmapspace{\roots}$ and
$u, v \in \mathrm{Hom}(\mathbb Z^{\sroots},\Zgen^{\star})$ as above.  Then the following hold.
\begin{enumerate}[label=\roman*), leftmargin=8mm,]
    \item 
    The map $\Kscale{u}{\mathsfit h}$  defined in \eqref{eq:defKscale} extends uniquely to an automorphism of $\Uq^\utypechar$\,.
    \vspace{2mm}\label{item:Kscaleprops:defin}
    \item $\Kscale{u}{\mathsfit h}\circ\Kscale{v}{\mathsfit k}
    =\Kscale{u\cdot v}{\mathsfit h+\mathsfit k}\,$, and 
    $\Kscale{1}{0}$ \  is identity, where $1$ is the constant map to $1\in\Zgen^{\star}\,$.
    \vspace{2mm}\label{item:Kscaleprops:homom}
\item If $b$ is a homogeneous element with grading $\beta=\wgrad(b)$
we have 
$$
\Kscale{u}{\mathsfit h}(b)=u(\beta) q^{-\frac 12(\symbrack{\beta}{\mathsfit h(\beta)}-2\symbrack{\weightvec{\mathsfit h}}{\beta})}K^{\mathsfit h(\beta)}b\;.
$$\label{item:Kscaleprops:genform}
\item With $b$ as above, the action on gradings is \ $\wgrad\circ \Kscale{u}{\mathsfit h}=\wgrad$ \ and \ 
$\ztgradp(\Kscale{u}{\mathsfit h}(b))=\ztgradp(b)+\mathsfit h(\beta)\mod 2\,$.\vspace{2mm}
\label{item:Kscaleprops:grad}
\end{enumerate}  
\end{prop}
\begin{proof} Consider the algebra $L_q^\utypechar$ with the same generators as $\Uq^\utypechar$ but subject only to the relations \eqref{eq:Kcomm} in addition to those for $\Uq^0\cong\Zgen[\{K_i^{\pm 1}\}]\,$. $L_q^\utypechar$ admits the same gradings as $\Uq^\utypechar$\,, and commutation of the $K_i$ in $L_q^0\subset L_q^\utypechar$ implies that $\Kscale{u}{\mathsfit h}$ is well-defined on $L_q^\utypechar$\,. We next prove that 
{\em \ref{item:Kscaleprops:genform}} holds on $L_q^\utypechar$ by induction in the number of $E_i$ and $F_i$ generators in an expression for $b\,$. Since $\Kscale{u}{\mathsfit h}$ acts trivially on $\Uq^0$ we may omit any additional $K_i$ factors in $b\,$. 

For $b=E_i$ and $\beta=\alpha_i$ the $q$-exponent is zero because of 
\eqref{eq:defrhosymm} so that the expression is the same as in \eqref{eq:defKscale}. For $b=F_i$ and $\beta=-\alpha_i$ use
$u(-\alpha_i)=u(\alpha_i)^{-1}\,$, $\mathsfit h(-\alpha_i)=- \mathsfit h(\alpha_i)\,$, and 
$\symbrack{-\alpha_i}{\mathsfit h(-\alpha_i)}-2\symbrack{\weightvec{\mathsfit h}}{-\alpha_i}
=2\symbrack{\alpha_i}{\mathsfit h(\alpha_i)}$ again by \eqref{eq:defrhosymm}.

For the induction step, we compute the formula for $E_ib$ with $\wgrad(E_ib)=\alpha_i+\beta$ assuming the asserted expressions for $E_i$ and $b$ and using the fact that  $u$ and $\mathsfit h$ are homomorphisms,
\begin{align*}
\Kscale{u}{\mathsfit h}(E_ib)&=\Kscale{u}{\mathsfit h}(E_i)\Kscale{u}{\mathsfit h}(b)\\
    &=u(\alpha_i)u(\beta)q^{-\frac 12(\symbrack{\alpha_i}{\mathsfit h(\alpha_i)}-2\symbrack{\weightvec{\mathsfit h}}{\alpha_i})}q^{-\frac 12(\symbrack{\beta}{\mathsfit h(\beta)}-2\symbrack{\weightvec{\mathsfit h}}{\beta})}
    K^{\mathsfit h(\alpha_i)}E_iK^{\mathsfit h(\beta)}b\\
    &=u(\alpha_i+\beta)q^{-\frac 12(\symbrack{\alpha_i}{\mathsfit h(\alpha_i)}+\symbrack{\beta}{\mathsfit h(\beta)}-2\symbrack{\weightvec{\mathsfit h}}{\alpha_i+\beta})}q^{-\symbrack{\alpha_i}{\mathsfit h(\beta)}}
    K^{\mathsfit h(\alpha_i)}K^{\mathsfit h(\beta)}E_ib\\
    &=u(\alpha_i+\beta)q^{-\frac 12(\symbrack{\alpha_i}{\mathsfit h(\alpha_i)}+2\symbrack{\alpha_i}{\mathsfit h(\beta)}+\symbrack{\beta}{\mathsfit h(\beta)}-2\symbrack{\weightvec{\mathsfit h}}{\alpha_i+\beta})} 
    K^{\mathsfit h(\alpha_i+\beta)}E_ib\\
   &=u(\alpha_i+\beta)q^{-\frac 12(\symbrack{\alpha_i+\beta}{\mathsfit h(\alpha_i+\beta)}-2\symbrack{\weightvec{\mathsfit h}}{\alpha_i+\beta})} 
    K^{\mathsfit h(\alpha_i+\beta)}E_ib\,.\\
\end{align*}
In the last step, we invoke the identity 
$\symbrack{\alpha_i}{\mathsfit h(\beta)}=\symbrack{\beta}{\mathsfit h(\alpha_i)}$\,, which is 
equivalent to our assumption that $\mathsfit h$ is $\roots$-symmetric. The respective computation for $F_ib$ is analogous, and the extension to $\Zgen$-linear combinations obvious. 

Since $\Kscale{u}{\mathsfit h}(b)=g_\beta b$ for a $\wgrad$-homogeneous elements $b$ with $g_\beta$ invertible,  
$\Kscale{u}{\mathsfit h}$ maps the two-sided ideal generated by $b$ to itself. Thus, $\Kscale{u}{\mathsfit h}$
is well-defined on any quotient of $L_q^\utypechar$ obtained by dividing by $\wgrad$-homogeneous relations.
Since all relations in \eqref{eq:SerreE}-\eqref{eq:Kcomm} are indeed $\wgrad$-homogeneous, we conclude the statement in Item {\em \ref{item:Kscaleprops:defin}}. The statement in Item {\em \ref{item:Kscaleprops:homom}} is easily checked on the generators, and Item {\em \ref{item:Kscaleprops:grad}} is immediate from Item {\em \ref{item:Kscaleprops:genform}}. 
\end{proof}

Given some $\mathsfit h\in\symmmapspace{\roots}$ and $\beta\in\mathbb Z^{\sroots}$, we know 
from Lemma~\ref{lem:symmbasicprops} {\em \ref{item:symmbasicprops:rhoweight}} that $2\symbrack{\weightvec{\mathsfit h}}{\beta}\in\mathbb Z\,$. We can thus define a character  $\qitgen_{\mathsfit h}\in\mathrm{Hom}(\mathbb Z^{\sroots},\Zgen^{\star})$ by
\begin{equation}\label{eq:defhqi}
    \qitgen_{\mathsfit h}:\,\mathbb Z^{\sroots}\rightarrow \Zgen^{\star}\;:\quad \beta\,\mapsto\,\qitgen_{\mathsfit h}(\beta)=q^{2\symbrack{\weightvec{\mathsfit h}}{\beta}}\,.
\end{equation}
Note that $\qitgen_{\mathsfit h}\cdot\qitgen_{\mathsfit k}=\qitgen_{\mathsfit h+k}$\,. We will also use the abbreviation $\qitgen=\qitgen_{\idsymm}\,$ so that $\qitgen(\alpha_i)=q_i^2\,$.

In \cite{lu90b} Lusztig introduces two anti-involutions. The Cartan anti-involution $\Cartaninv:\Uq\rightarrow \Uq^{\opp,\cop}$ is defined as the $\mathbb{Z}$-algebra homomorphism with the following actions on algebra generators and the indeterminate of the ground ring.
\begin{align}\label{eq:defCinv}
    \Cartaninv(E_i)=F_i\,, && 
    \Cartaninv(F_i)=E_i\,, && 
    \Cartaninv(K_i)=K_i^{-1}\,, && 
    \Cartaninv(q)=q^{-1}\,.
\end{align}
Secondly, let $\Kinvaut:\Uq\rightarrow \Uq^{\opp}$ be the involutive $\Zqqvn{\dpone}$-algebra homomorphism defined on generators as follows.
\begin{align}\label{eq:defKinv}
    \Kinvaut(E_i)=E_i\,, && 
    \Kinvaut(F_i)=F_i\,, && 
    \Kinvaut(K_i)=K_i^{-1}\,.
\end{align}
Clearly, $\Kinvaut$ preserves both the $\wgrad$ and $\ztgrad^{\pm}$ gradings, while \ $\wgrad\circ \Cartaninv=-\wgrad$ \ and \ $\ztgradp\circ \Cartaninv=\ztgradn\,$. In addition to Lusztig's involutions, it will be useful to also introduce the conjugation automorphism $\Qinvaut:\Uq\rightarrow \Uq^{\opp}$ defined on generators by
\begin{align}\label{eq:defQinv}
    \Qinvaut(E_i)=E_i\,, && 
    \Qinvaut(F_i)=F_i\,, && 
    \Qinvaut(K_i)=K_i\,, && 
    \Qinvaut(q)=q^{-1}\,.
\end{align}
All three of the above involutions may also be viewed  as involutive anti-homomorphisms. It follows easily by inspection on generators that $\Cartaninv$\,, $\Kinvaut$\,, and $\Qinvaut$ commute with each other and, thus, implement an action of $(\mathbb F_2)^3$ on $\Uq\,$. 

For later use, we introduce notation for two composites of these. Firstly, $\Cartanaut=\Qinvaut\circ\Cartaninv=\Cartaninv\circ\Qinvaut: \Uq\rightarrow \Uq$ is a $\Zqqvn{\dpone}$-algebra homomorphism given on generators by 
\begin{align}\label{eq:defCaut}
    \Cartanaut(E_i)=F_i\,, && 
    \Cartanaut(F_i)=E_i\,, && 
    \Cartanaut(K_i)=K_i^{-1}\,.
\end{align}
Secondly, let $\Kconaut=\Qinvaut\circ\Kinvaut=\Kinvaut\circ\Qinvaut: \Uq\rightarrow \Uq$\,, which is a $\mathbb Z$-algebra homomorphism with non-trivial action on the ground ring given by 
\begin{align}\label{eq:defKaut}
    \Kconaut(E_i)=E_i\,, && 
    \Kconaut(F_i)=F_i\,, && 
    \Kconaut(K_i)=K_i^{-1}\,,&&
    \Kconaut(q)=q^{-1\,}\,.
\end{align}
Clearly, the automorphisms $\Kinvaut$\,, $\Qinvaut$\,, and $\Kconaut$ map each of the subalgebras $\Uq^{\utypechar}$ with $\utypechar\in\{0, +, -, \geqzero, \leqzero\}\,$ to itself, while the $\Cartaninv$ and $\Cartanaut$ map these algebras to their respective opposites.  

Additional commutation relations  with the antipode and the \Kscaleword transformations from \eqref{eq:defKscale} are immediate from checks on generators as well. More specifically, for a $\roots$-symmetric map $\mathsfit h \in\mathrm{End}(\mathbb Z^{\sroots})$ and a character $u\in\mathrm{Hom}(\mathbb Z^{\sroots},\Zgen^\star)$ we obtain the following commutations relations of these involutions with the $\Kscale{u}{\mathsfit h}$ transformations, where $\qitgen_{\mathsfit h}$ is as in \eqref{eq:defhqi}.
\begin{equation}\label{eq:InvKScalComm}
\begin{aligned}
    \Cartaninv\,\circ\Kscale{u}{\mathsfit h}&\,=\,\Kscale{u}{\mathsfit h}\,\circ\,\Cartaninv
    \hspace*{21mm}&\hspace*{21mm}
    \Kinvaut\,\circ\Kscale{u}{\mathsfit h}&\,=\,\Kscale{u\cdot\qitgen_{\mathsfit h}}{-\mathsfit h}\,\circ\,\Kinvaut
    \\ \rule{0mm}{5mm}
    \Cartanaut\,\circ\Kscale{u}{\mathsfit h}&\,=\,\Kscale{u^{-1}\cdot\qitgen^{-1}_{\mathsfit h}}{\mathsfit h}\,\circ\,\Cartanaut
    &
    \Qinvaut\,\circ\Kscale{u}{\mathsfit h}&\,=\,\Kscale{u^{-1}\cdot\qitgen_{\mathsfit h}^{-1}}{\mathsfit h}\,\circ\,\Qinvaut
\end{aligned}
\end{equation} 

Furthermore, the antipode can be expressed as the composite 
\begin{equation}\label{eq:SXiMhoRel}
    S\,=\,\Kscale{\htsign\cdot\qitgen}{-\idsymm}\,\circ\,\Kinvaut\,=\,\Kinvaut\,\circ\Kscale{\htsign}{\idsymm}\,,
    \qquad\mbox{where}\quad 
    \htsign(\beta)=(-1)^{\height(\beta)}=(-1)^{\symbrack{\breve\rho}{\beta}}\,,
\end{equation} 
with $\height$ and $\breve\weightvecchar$ as discussed at the end of Section~\ref{subsec:descroots}. Commutation relations of the antipode with the other automorphisms can now be derived directly from \eqref{eq:SXiMhoRel}
and the commutation relations in \eqref{eq:InvKScalComm}. For example, $S$ commutes with $\Cartaninv$\,, $S\circ \Kinvaut=\Kscale{\qitgen}{-2\idsymm}\circ \Kinvaut\circ S\,$, and $S^2=\Kscale{\qitgen}{0}\,$. 

Finally, note that $\Kinvaut$ and $\Qinvaut$ preserve the $\wgrad$ and $\ztgrad^{\pm}$ gradings. For $\Cartaninv$ we find
\begin{align}\label{eq:CartinvGrad}
     \wgrad\circ \Cartaninv& =-\wgrad\,, & \ztgradp\circ \Cartaninv& =\ztgradn\,, & \mbox{and}&&\ztgradn\circ \Cartaninv& =\ztgradp\,, 
\end{align}
implying the same equations for $\Cartanaut$.

\subsection{Lusztig's Action of Artin-Tits Monoid}\label{subsec:LATM}
In \cite{Lu90a,lu90b,lu} Lusztig introduces for any Lie type $\mathfrak g$ algebra automorphisms $T_i$ on $\Uqre{\mathbb Q}$ as well as on $\LUq$ for $1\leq i\leq n$ where $n$ is the rank of $\mathfrak g$\,. Importantly, they fulfill the Artin-Tits relations (\ref{eq-ATrels}) for the respective Cartan data. 

In order to conform with conventions in the $R$-matrix constructions in \cite{KR90} we will prefer to work with the inverses $\Tinv_i=T_i^{-1}$ of these automorphisms. Their explicit actions on the generators of $\Uq$ are as follows: 
\begin{align}
    \Tinv_i(E_i)&=- K_i^{-1}F_i\,,&
    \Tinv_i(F_i)&=-E_iK_i\,, & 
    \Tinv_i(K_j)&=K_jK_i^{-A_{ij}}\,,\label{eq:defLeasy}
    \\
     \Tinv^{-1}_i(E_i)&=-F_iK_i\,, &
    \Tinv^{-1}_i(F_i)&=-K_i^{-1}E_i\,, & \Tinv^{-1}_i(K_j)&=K_jK_i^{-A_{ij}}\,,\label{eq:defTeasy}
\end{align} and for $i\neq j$
\begin{align}
    \Tinv_i(E_j)=\sum_{r+s=-A_{ij}}\frac {(-1)^rq_i^{-s}} {[s]!_i[r]!_i}E_i^{s}E_jE_i^{r}\,, && 
    \Tinv_i(F_j)=\sum_{r+s=-A_{ij}}\frac {(-1)^rq_i^{s}} {[r]!_i[s]!_i}F_i^{r}F_jF_i^{s}\,,
    \label{eq:defLserre}
    \\
     \Tinv^{-1}_i(E_j)=\sum_{r+s=-A_{ij}}\frac {(-1)^rq_i^{-s}} {[r]!_i[s]!_i}E_i^rE_jE_i^{s}\,, && \Tinv^{-1}_i(F_j)=\sum_{r+s=-A_{ij}}\frac {(-1)^rq_i^{s}} {[s]!_i[r]!_i}F_i^{s}F_jF_i^{r}\,.\label{eq:defTserre}
\end{align}

\begin{lemma}
The automorphisms $\Tinv_i$ restrict to automorphisms on the algebra $\Uq$ over $\Zqqvn{\dpone}$  and fulfill the Artin-Tits relations \eqref{eq-ATrels} for all Lie types. 
\end{lemma}

The proof of the first assertion is immediate, observing that the denominators in the above expressions divide $[-A_{ij}]!_i$ and hence 
$[-d_iA_{ij}]!$ as well as $[\edgenum]!=[\dpone-1]!$\,. Thus, the $\Tinv_i$ automorphisms are well-defined over the ground ring $\Zqqvn{\dpone}$\,. Moreover, the Artin-Tits relations have been verified for the $T_i$ automorphisms in \cite{Lu90a,lu90b,lu} and thus also hold for the $\Tinv_i=T_i^{-1}$.

An immediate consequence is that we can associate an automorphism to any Weyl group element. For each reduced word $w=w_{i_1}\dots w_{i_k}\in\wordset$\,, let $\Tinv_w=\Tinv_{i_1}\circ\ldots\circ \Tinv_{i_k}$\,. By Lemma~\ref{cor:matsumoto}, the composed automorphism only depends on the Weyl group element $s={\Weylpres(w)}$ represented by $w$\,, allowing us to write $\Tinv_w=\Tinv_s$\,. Now, for $s,t\in\Weyl$ the construction implies 
\begin{equation}\label{eq:Thomrel}
    \Tinv_{st}=\Tinv_s\circ \Tinv_t \qquad \mbox{if}\quad \len {st}=\len s +\len t\,.
\end{equation} 
That is, due to the length additivity condition $s\mapsto\Tinv_s$ is {\em not} a representation of $\Weyl\,$. Instead, the above notation should be understood as
the composite of the Matsumoto map 
$\mathscr s: \Weyl\rightarrow \mathscr A^+\,$ from Section~\ref{subsec:coxeter} with the natural monoid homomorphism $\Tinv:\,\mathscr A^+\rightarrow \mathrm{Aut}(\Uq)\,$ defined on the positive Artin-Tits monoid, which extends to an action $\Tinv:\,\mathscr A\rightarrow \mathrm{Aut}(\Uq)\,$ of the Artin group on $\Uq$\,. 
Thus, depending on context, we allow subscripts of $\Tinv$ to be in $\Weyl$\,, $\wordset$\,, or 
$\mathscr A$. These notations are related, for example, by  $\Tinv_w=\Tinv_{\Weylpres(w)}$ if $w\in \wordset$ or $\Tinv_s=\Tinv_{\mathscr s(s)}$ if $s\in \Weyl\,$.

Observe that $\Cartaninv$ commutes with all $\Tinv_i$ and $T_i\,$. Also, for the $\Kinvaut$ involution one 
readily infers $\Kinvaut\circ \Tinv_i\circ \Kinvaut=\Tinv_i^{-1}=T_i$\,, or, more generally,
\begin{equation}\label{eq:LTconjinv}
    \Kinvaut \,\circ \Tinv_s \circ \,\Kinvaut\,=\,(\Tinv_{s^{-1}})^{-1}\,=\,T_s
    \qquad \forall s\in \Weyl\,.
\end{equation}
Here $T_s=T_{i_1}\circ\ldots\circ T_{i_k}$ denotes the analogous automorphisms used by Lusztig  to define generators associated to general positive roots. The relation in \eqref{eq:LTconjinv} will thus serve to relate our conventions of choosing generators to those in \cite{Lu90a,lu90b}. 
The second equality may also be written as $\Tinv_b=T_{\iota(b)}$ for an element $b\in\mathscr A$,
where $\iota\in\mathrm{Aut}(\mathscr A)$ maps each generator of $\mathscr A$ to its inverse. 
The action of the braid automorphisms on the $\Uq^0$ basis elements is naturally given by
\begin{align}
    \Tinv_s(K^\nu)  & =K^{s(\nu)}\,.  \label{eq:TKact} 
\end{align}

Finally, for a homogeneous element $b\in\Uq$ one readily works out the following equivariance relations for the gradings, where the action of $s\in \Weyl$ is naturally extended to the root lattice $\mathbb Z^{\sroots}\supset \roots\,$ as well as $\mathbb F_2^{\sroots}\,$.
\begin{align}
 \wgrad(\Tinv_s(b)) & =s(\wgrad(b))  \label{eq:Twgrad}\\
    \ztgrad^{\pm}(\Tinv_s(b)) & =s(\ztgrad^{\pm}(b))   \label{eq:Togradequiv}
\end{align}
Clearly, \eqref{eq:TKact}, \eqref{eq:Twgrad} and \eqref{eq:Togradequiv} hold also if the $\Tinv_s$ are replaced by the original $T_s$ automorphisms.

\subsection{More \texorpdfstring{$\mathscr A^+$}{A+} Commutation Relations}
In this section, we  provide commutation relations between the $\Tinv_s$ automorphisms and the \Kscaleword transformations $\Kscale{u}{\mathsfit h}$\,, from which we also obtain respective commutation relations with the antipode. We denote the natural (left) actions of $\Weyl$ on elements $\mathsfit h\in \symmmapspace{\roots}$ and $u\in \mathrm{Hom}(\mathbb Z^{\sroots},\Zgen^\star)$ by $\mathsfit h^s=shs^{-1}$ and $s_*u=u\circ s^{-1}\,$, respectively. Additionally, define the character $\sdrchar{s}{\mathsfit h}\in \mathrm{Hom}(\mathbb Z^{\sroots},\Zgen^\star)$ by 
\begin{equation}
    \sdrchar{s}{\mathsfit h}(\beta)=q^{-\symbrack{\sumdescrootsrel{s}{\mathsfit h}}{\beta}}\,,
\end{equation}
which is well-defined by Lemma~\ref{lem:symmbasicprops}
{\em \ref{item:symmbasicprops:thetaint}}. The following basic relations involving this character are easily derived from statement {\em\ref{item:symmbasicprops:thetarec}} as well as the combination of \eqref{eq:thetaviarho} and \eqref{eq:defhqi}:
\begin{equation}\label{eq:thetacharprops}
    \sdrchar{t\cdot s}{\mathsfit h}=\sdrchar{t}{\mathsfit h^s}\cdot t_*(\sdrchar{s}{\mathsfit h})\, 
    \qquad\mbox{and}\qquad
    (\sdrchar{s}{\mathsfit h})^2=(s_*\qitgen_{\mathsfit h})\cdot (\qitgen_{\mathsfit h^s})^{-1}\,.
\end{equation}

\begin{prop}\label{prop:ArtinKscaleComm} Suppose $\mathsfit h\in\symmmapspace{\roots}$, $u\in\mathrm{Hom}(\mathbb Z^{\sroots},\Zgen^\star)$, and $s\in\Weyl$. Using notations and conventions as above we have
    \begin{equation}\label{eq:ArtinKscaleComm}
        \Tinv_s\,\circ \,\Kscale{u}{\mathsfit h}\,\,=\,\, \Kscale{(s_*u)\cdot \sdrchar{s}{\mathsfit h}}{\mathsfit h^s}\,\circ\, \Tinv_s\;.
    \end{equation} 
\end{prop}
\begin{proof} Since all maps in the relation are automorphisms it suffices to check \eqref{eq:ArtinKscaleComm} for the  generators of $\Uq\,$. Moreover, since the $\Kscale{u}{\mathsfit h}$ act trivially on the $K_i$'s and since $\Cartaninv$ commutes with both $\Kscale{u}{\mathsfit h}$ and $\Tinv_s$ it suffices to check \eqref{eq:ArtinKscaleComm} for only the $E_i$ generators. Note further that for a homogeneous element $b$ with $\wgrad(b)=\beta$ we have 
   $\Kscale{s_*u}{0}(\Tinv_s(b))=(s_*u)(s(\beta))\Tinv_s(b)=u(\beta)\Tinv_s(b)=\Tinv_s(u(\beta)b)=\Tinv_s(\Kscale{u}{0}(b))$\,, using $\wgrad(\Tinv_s(b))=s(\beta)$ by \eqref{eq:Twgrad}, so that \eqref{eq:ArtinKscaleComm} holds for $\mathsfit h=0$\,. Since $\Kscale{u}{\mathsfit h}=\Kscale{u}{\mathsfit 0}\circ \Kscale{1}{\mathsfit h}$ we may, therefore, assume $u=1$ in the following verification. 
    \begin{align*}
        \Kscale{\sdrchar{s}{\mathsfit h}}{\mathsfit h^s}(\Tinv_s(E_i))&
           =\sdrchar{s}{\mathsfit h}(s(\alpha_i))
             q^{-\frac 12 \left({\symbrack{s(\alpha_i)}{\mathsfit h^s(s(\alpha_i))}-2\symbrack{\weightvec{\mathsfit h^s}}{s(\alpha_i)}}\right)}K^{\mathsfit h^s(s(\alpha_i))}\Tinv_s(E_i)
    \\
    &
           =q^{-\symbrack{\sumdescrootsrel{s}{\mathsfit h}}{s(\alpha_i)}}
             q^{-\frac 12 \left({\symbrack{s(\alpha_i)}{s(\mathsfit h(\alpha_i))}-2\symbrack{\weightvec{s\mathsfit h s^{-1}}}{s(\alpha_i)}}\right)}K^{s(\mathsfit h(\alpha_i))}\Tinv_s(E_i)
    \\
    &
            =q^{-\left({\symbrack{\weightvec{s\mathsfit h s^{-1}}}{s(\alpha_i)}-\symbrack{s(\weightvec{\mathsfit h})}{s(\alpha_i)}}\right)}
             q^{-\frac 12 \left({\symbrack{\alpha_i}{\mathsfit h(\alpha_i)}-2\symbrack{\weightvec{s\mathsfit h s^{-1}}}{s(\alpha_i)}}\right)}\Tinv_s(K^{\mathsfit h(\alpha_i)})\Tinv_s(E_i)
     \\
    &
            = 
             q^{-\frac 12 \left({\symbrack{\alpha_i}{\mathsfit h(\alpha_i)}-2\symbrack{s(\weightvec{\mathsfit h})}{s(\alpha_i)}}\right)}\Tinv_s(K^{\mathsfit h(\alpha_i)}E_i)
      \\
      &
            = 
             q^{-\frac 12 \left({\symbrack{\alpha_i}{\mathsfit h(\alpha_i)}-2\symbrack{\weightvec{\mathsfit h}}{\alpha_i}}\right)}\Tinv_s(\Kscale{1}{{\mathsfit h}}(E_i))
             \qquad = \Tinv_s\circ\Kscale{1}{{\mathsfit h}}(E_i)
    \end{align*}
    In the first line, we use Proposition~\ref{prop:Kscaleprops}{\em \ref{item:Kscaleprops:genform}} for $\beta=\wgrad(\Tinv_s(E_i))=s(\alpha_i)\,$. In the second line, definitions are inserted, including 
    $\mathsfit h^s(s(\alpha_i))=s\mathsfit h s^{-1}s(\alpha_i)=s\mathsfit h(\alpha_i)\,$. The third line 
    invokes isometry of $s$ and \eqref{eq:TKact}. In the fourth line, exponent terms are combined and canceled, and we use that $\Tinv_s$ is an automorphism. In the last lines $s$-isometry, the definition for  $\Kscale{1}{{\mathsfit h}}$\,, and \eqref{eq:defrhosymm} are employed. 
\end{proof}

The action of the Artin-Tits monoid $\mathscr A^+$ on the set of \Kscaleword transformations given by conjugation thus factors into an action of $\Weyl$ on $\symmmapspace{\roots}\times \mathrm{Hom}(\mathbb Z^{\sroots})$\,. The latter is defined as
$s.(\mathsfit h,u)=(\mathsfit h^s, (s_*u)\cdot \sdrchar{s}{\mathsfit h})$ using \eqref{eq:thetacharprops}. An immediate consequence is that any \Kscaleword transformation commutes with the action of the {\em pure} Artin-Tits group on $\Uq$\,, defined as the kernel of the map $\mathscr A\rightarrow\Weyl\,$.

As an application, we also derive an explicit commutation relation between the $\Tinv_s$ maps and the antipode. To this end, let $\reldrchar_s\in\mathrm{Hom}(\mathbb Z^{\sroots},\Zgen^\star)\,$ be the character defined by
\begin{equation}
    \reldrchar_s(\beta)=(-1)^{\symbrack{\sumdesccoroots{s}}{\beta}}q^{-\symbrack{\sumdescroots{s}}{\beta}}\,, 
\end{equation}
with $\sumdescroots{s}$ and $\sumdesccoroots{s}$ as in \eqref{eq:deftheta} or \eqref{eq:thetaviarho}. Relations
\eqref{eqn:theta} and \eqref{eq:descrootsmirror} translate for $\len {t\cdot s}=\len t + \len s$ into 
\begin{equation}\label{eq:reldrcharrec}
    \reldrchar_{t\cdot s}=\reldrchar_t\cdot t_*(\reldrchar_s)\qquad\mbox{and} \qquad s_*(\reldrchar_{s^{-1}})=(\reldrchar_s)^{-1}\,.
\end{equation}

\begin{prop}
For any $s\in\Weyl$ we have 
\begin{equation}
    S\circ (\Tinv_s)^{-1}=\Tinv_{(s^{-1})}\circ S\circ \Kscale{\reldrchar_s}{0}
    \qquad\mbox{and}\qquad 
   (\Tinv_s)^{-1}\circ  S\circ \Kscale{\reldrchar_s}{0}=S\circ \Tinv_{(s^{-1})}\;.
  \end{equation}
\end{prop}

\begin{proof}
    We first establish that $\reldrchar_s=(s_*\htsign)\cdot\htsign^{-1}\cdot \sdrchar{s}{\idsymm}$ with $\htsign$ 
    as in \eqref{eq:SXiMhoRel}. Using \eqref{eq:thetaviarho} and the fact that $s$ is an isometry we compute 
    \begin{align*}
        \left({(s_*\htsign)\cdot\htsign^{-1}\cdot \sdrchar{s}{\idsymm}}\right)(\beta)
          &=
        \htsign(s^{-1}(\beta))\htsign(\beta)^{-1}\sdrchar{s}{\idsymm}(\beta)
          &&=
        (-1)^{\symbrack{\breve\weightvecchar}{s^{-1}(\beta)}}(-1)^{-\symbrack{\breve\weightvecchar}{\beta}}q^{-\symbrack{\sumdescrootsrel{s}{\idsymm}}{\beta}}
        \\
        \ 
        &=(-1)^{\symbrack{s(\breve\weightvecchar)}{\beta}-\symbrack{\breve\weightvecchar}{\beta}} q^{-\symbrack{\sumdescroots{s}}{\beta}}
       &&=(-1)^{\symbrack{\sumdesccoroots{s}}{\beta}} q^{-\symbrack{\sumdescroots{s}}{\beta}}
        \qquad =\reldrchar_s(\beta)\,.
    \end{align*}
    The following verification of the first identity invokes the expression in \eqref{eq:SXiMhoRel} for the antipode, the commutation rule from Proposition~\ref{prop:ArtinKscaleComm}, the conjugation with $\Kinvaut$ from \eqref{eq:LTconjinv}, homomorphy from Proposition~\ref{prop:Kscaleprops}{\em\ref{item:Kscaleprops:homom}}, and the above relation for $\reldrchar_s\,$. 
    \begin{align*}
        S\circ (\Tinv_s)^{-1} &= 
           \Kinvaut\circ \Kscale{\htsign}{\idsymm}\circ (\Tinv_s)^{-1} 
            &&=
           \Kinvaut\circ (\Tinv_s)^{-1} \circ \Kscale{(s_*\htsign)\cdot\sdrchar{s}{\idsymm}}{\idsymm}
            &&=
            \Tinv_{(s^{-1})} \circ \Kinvaut\circ \Kscale{(s_*\htsign)\cdot\sdrchar{s}{\idsymm}}{\idsymm}
            \\
            &=\Tinv_{(s^{-1})} \circ S\circ (\Kscale{\htsign}{\idsymm})^{-1}\circ \Kscale{(s_*\htsign)\cdot\sdrchar{s}{\idsymm}}{\idsymm}
           &&=\Tinv_{(s^{-1})} \circ S\circ \Kscale{\htsign^{-1}\cdot(s_*\htsign)\cdot\sdrchar{s}{\idsymm}}{0}
           &&=\Tinv_{(s^{-1})} \circ S\circ \Kscale{\reldrchar_s}{0}
    \end{align*}
    The second identity can be derived from the first by replacing $s$ with $s^{-1}$ and \eqref{eq:reldrcharrec}. Specifically, 
    $S\circ (\Tinv_{(s^{-1})})^{-1}=\Tinv_{s}\circ S\circ \Kscale{\reldrchar_{(s^{-1})}}{0}$ implies 
    $$(\Tinv_{s})^{-1}\circ S= S\circ \Kscale{\reldrchar_{(s^{-1})}}{0}\circ \Tinv_{(s^{-1})}
    = S\circ \Kscale{(s^{-1})_*(\reldrchar_s^{-1})}{0}\circ \Tinv_{(s^{-1})}= S\circ \Tinv_{(s^{-1})}\circ \Kscale{(\reldrchar_s^{-1})}{0}= S\circ \Tinv_{(s^{-1})}\circ (\Kscale{\reldrchar_s}{0})^{-1}\,,$$
    and, hence, the second identity.
\end{proof}

The character $\reldrchar_s$ also appears in  commutation relations of the Artin-Tits groups action with the conjugation anti-automorphism $\Qinvaut\,$.

\begin{prop}\label{lem:ArtinConjComm} For any element $s\in\Weyl$ we have
\begin{equation}\label{eq:ArtinConjComm}
    \Tinv_s\circ \Qinvaut\,=\,\Qinvaut\circ \Tinv_s\circ\Kscale{\reldrchar_{(s^{-1})}}{0}\,=\,\Kscale{\reldrchar_{s}}{0}\circ\Qinvaut\circ \Tinv_s\,.
\end{equation}
\end{prop}

\begin{proof}The second equality follows from $\Tinv_s\circ\Kscale{\reldrchar_{(s^{-1})}}{0}=\Kscale{s_*\reldrchar_{(s^{-1})}}{0}\circ\Tinv_s=\Kscale{\reldrchar_{s}^{-1}}{0}\circ\Tinv_s$ by \eqref{eq:ArtinKscaleComm} and \eqref{eq:reldrcharrec} as well as \eqref{eq:InvKScalComm}.
 For $s=s_i$ equation \eqref{eq:ArtinConjComm} applied to a generator $E_j$ reads 
 \begin{align*}
     \Tinv_i(E_j) &= \Qinvaut(\Tinv_i(\reldrchar_{s_i}(\alpha_j)E_j)) 
  = \reldrchar_{s_i}(\alpha_j)^{-1}\Qinvaut(\Tinv_i(E_j))\\
  &=(-1)^{\symbrack{\breve\alpha_i}{\alpha_j}}q^{\symbrack{\alpha_i}{\alpha_j}}\Qinvaut(\Tinv_i(E_j))=(-1)^{A_{ij}}q_i^{A_{ij}}\Qinvaut(\Tinv_i(E_j))\,.
 \end{align*} 
  This is readily verified for both the $i=j$ and $i\neq j$ cases from \eqref{eq:defLeasy}, \eqref{eq:defLserre}, and \eqref{eq:defQinv}. Since $\Cartaninv$ commutes with both $\Tinv_i$ and $\Qinvaut$ and, furthermore, inverts $q$ the relation for $s=s_i$ also holds on $F_i$\,. As it is also trivially true on $K_i$ we proved \eqref{eq:ArtinConjComm} for $s=s_i$\,.
  Suppose now \eqref{eq:ArtinConjComm} holds for two elements $s,t\in\Weyl$ with $\len s\,,\len t\geq 1$ and $\len {ts}=\len t + \len s$\,. Then  
  \begin{align*}
    \Tinv_{ts}\circ\Qinvaut&=\Tinv_{t}\circ\Tinv_{s}\circ\Qinvaut
    = \Tinv_{t}\circ\Kscale{\reldrchar_{s}}{0}\circ \Qinvaut\circ\Tinv_{s}
    = \Kscale{t_*(\reldrchar_{s})}{0}\circ \Tinv_{t}\circ\Qinvaut\circ\Tinv_{s} \\
  &= \Kscale{t_*(\reldrchar_{s})}{0}\circ\Kscale{\reldrchar_{t}}{0}\circ\Qinvaut\circ \Tinv_{t}\circ\Tinv_{s} 
  = \Kscale{t_*(\reldrchar_{s})\cdot\reldrchar_{t}}{0} \circ\Qinvaut\circ \Tinv_{ts} 
  = \Kscale{\reldrchar_{ts}}{0} \circ\Qinvaut\circ \Tinv_{ts}\,,  
  \end{align*}
   where we used \eqref{eq:ArtinKscaleComm} and \eqref{eq:reldrcharrec}. The assertion now follows by induction.
\end{proof}

 \section{Generators, Bases, and Integrality} \label{sec:PBW}

 Ponicar\'e-Birkhoff-Witt (PBW) bases have been established for quantum deformations of Lie algberas in many settings. These typically consider quantum algebras  defined over a field such as $\mathbb Q(q)$ or $\mathbb C$ and assume specific fixed choices with respect to which the quantum analogs of the $E_\alpha$ operators are defined. See, for example, \cite{Ro89,CP95,Jan96,Ri96,lu}.

 Lusztig describes the algebras $\LUq^{\utypechar}$ of divided powers as free modules over $\Zqq$ in \cite{Lu90a,lu90b} for specific choices of $E_\alpha\,$. In \cite{lu} explains independence of some of these choices for respective bases over $\mathbb Q(q)\,$. Specifically, 
 certain spanning sets $\basisp{w}$ of monomial expressions, to be defined below, will depend on the choice of a maximal reduced word $w\in\wordsetmax\,$. 
 
Corollary~40.2.2 in \cite{lu} implies that $\basisp{w}$ is a basis of $\Uqre{\mathbb Q}^+$ over the field $\mathbb Q(q)\,$ for all Lie types and all choices of $w$\,, which also implies that $\basisp{w}$ is a linearly independent set of generators for the $\Zqqn{\dpone}$-module $\Uq^+\,$ and that
its span $\bspanp{w}$ is a free $\Zqqn{\dpone}$-module. The  $\mathbb Q(q)$-basis does, however, not exclude the possibility that 
$\mathcal T_w=\Uq^+/\bspanp{w}$ is a non-zero torsion module and, thus, does not imply that $\basisp{w}$ is also a $\Zqqn{\dpone}$-basis.

 For versions of $\Uq^+$ for classical Lie types, bases over $\Zqq$ are obtained from iterated $q$-brackets in \cite{Tak90} for specific lexicographic orderings. In the formalism and constructions of generators used there, however, it is difficult to describe quasi-triangular Hopf algebra structures or braid actions for reordering. Still, the example provided in Proposition~5.6 in the same paper illustrates the delicate dependence of integrality of bases on the chosen orders. 

 Specifically, Takeuchi shows for an algebra $\bar{\mathscr B}_3^\circ\,$, which can be roughly understood as a quotient of $\Uq^+$ for Lie type $\LT{B}_3\,$, that the analog of the torsion module $\mathcal T_w$ is indeed zero for a convex ordering but contains non-zero torsion summands such 
as $\mathbb Z[q,q^{-1}]/\langle q^3-q^2+1\rangle$ for a non-convexly ordered monomial basis. A careful treatment of orderings will thus have to be a crucial component in the derivation of integral PBW bases.

 The plan of this section is to adapt arguments in \cite{Lu90a,lu90b} and related references to infer PBW bases that generate the algebras $\Uq^{\pm}$ and $\Uq$ as free $\Zqqn{\dpone}$ and $\Zqqvn{\dpone}$ modules respectively. Generators $E_w$ depending on reduced words $w\in\wordset$ are constructed in Section~\ref{subsec:genwords} and some of their basic properties discussed. From these monomial expressions that, in addition, depend on exponent functions are introduced, their relations with the two anti-involutions $\Cartaninv$ and $\Kinvaut$ are explained in Section~\ref{subsec:Emonomials}. A formalism for the spanning sets associated to the monomials is laid out in Section~\ref{subsec:spanorder}, where a general criterion for order independence is derived. 

The relations between the box pre-orderings introduced by Lusztig, convex orderings, and words representing the longest elements in $\Weyl$ are explained in Section~\ref{subsec:orderings}. Lusztig's proof of PBW bases is adapted to the $\Uq$-algebras in Section~\ref{subsec:PBWspec} and the general theorem for integral PBW bases for general orders is stated in Section~\ref{subsec:PBWmain}.
 
 \subsection{Generators Associated to Words}\label{subsec:genwords} In a series of articles, Lusztig defines of a generator $\ELu_\alpha$ for each $\alpha\in\proots$, starting with a  
presentation $\alpha=s_{i_1}\ldots s_{i_{k-1}}(\alpha_{i_k})$ obtained from a maximal reduced word. The simple root $\alpha_j$ is replaced by $E_j\,$ and each $s_i$ by an automorphism $T_i$ described in Section~\ref{subsec:LATM}. 
We will adopt and adjust the definition here by replacing the $T_i$ by their inverses 
$\Tinv_i\,$ and explicitly retain the dependence on words in the notation. That is, for a non-empty
reduced word $w=w_{i_1}\ldots w_{i_k}$ and using the notation in \eqref{def:flat+tau_word} we set 
\begin{equation}\label{eq:EFwordDef}
    E_w=\Tinv_{w^\flat}(E_{\tau(w)})=\Tinv_{i_1}\ldots\Tinv_{i_{k-1}}(E_{i_k})
    \qquad \mbox{and} \qquad 
    F_w=\Tinv_{w^\flat}(F_{\tau(w)})\,. 
\end{equation}
As remarked in Section~\ref{subsec:LATM}, we may write $E_w=\Tinv_s(E_j)$ where $s=\Weylpres(w^\flat)$ and $j=\tau(w)\,$,
implying that $E_w$ only depends on the pair $(s,j)\,$. Further redundancies in the labeling are implied by Proposition~\ref{prop:wordgencont} below, from which one can easily obtain examples of distinct pairs $(s,j)\neq(s',j')$ with $\Tinv_s(E_j)=\Tinv_{s'}(E_{j'})\,$.  For computations with specific words $w=w_{i_1}\ldots w_{i_k}$ it is often convenient to also use the index notation
\begin{equation}\label{eq:Esubseqind}
    E_w=E_{(i_1\ldots i_k)}\,. 
\end{equation}

Given a {\em fixed} reduced word $z\in\wordsetmax$ of {\em maximal} length or, equivalently, a convex order $\leqwt{w}$ on $\proots$, we will sometimes also use the abbreviated notation 
\begin{equation}\label{eq:Genbyroot}
E_\alpha=E_{z[\alpha]} \qquad \mbox{and} \qquad F_\alpha=F_{z[\alpha]}
\end{equation}
for an $\alpha\in\proots$ and with $z[\alpha]$  as defined in \eqref{eq:root2word}. On occasion, we may also use the analogous notation for the respectively defined Cartan element $K_w=\Tinv_{w^\flat}(K_{\tau(w)})=K^\nu\,$, where $K^\nu$ is as defined above \eqref{eq:Kwgrad} with $\nu=\wordroot(w)\,$. 
For most applications in this article the word labeling is, however, the preferred one
as we often consider orderings depending on choices of words. 
We list next several immediate properties of these generators.

Suppose $w=u\cdot v$ for non-empty reduced words $u,v\in\wordset$  with $\len w=\len u + \len v$ and let $s=\Weylpres(u)$\,. Then \eqref{eq:EFwordDef} and \eqref{eq:Thomrel} imply
\begin{equation}\label{eq:EFdefRecurs}
    E_w=\Tinv_s(E_v)  \qquad \mbox{and} \qquad F_w=\Tinv_s(F_v)\,.
\end{equation}
Furthermore, commutation of the $\Tinv_i$ with $\Cartaninv$ yields the identity
\begin{equation}\label{eq:CinvEFwords}
    \Cartaninv(E_w)=F_w\qquad\mbox{and }\qquad\Cartaninv(F_w)=E_w\;.
\end{equation}
The generators from \eqref{eq:EFwordDef} as well as the Cartan element $K^\nu$  are homogeneous elements,
for which gradings  follow immediately from \eqref{eq:Twgrad} and 
\eqref{eq:Togradequiv} with notation from Lemma~\ref{lem:red=pos}.
\begin{align}\label{eq:gradgens}
    \wgrad(E_w)&=\wordroot(w) & \wgrad(F_w)&=-\wordroot(w) & \wgrad(K_\nu)&=0\\
    \ztgradp(E_w)&=\wordroot(w)\mod 2 & \ztgradp(F_w)&= 0    & \ztgradp(K_\nu)&=\nu \mod 2
\end{align}

Extending the definition in \eqref{eq:EFwordDef} to non-reduced words leads to elements outside the respective Borel subalgebra, such as $E_{w_i\cdot w_i}=\Tinv_i(E_i)\not\in\Uq^+\,$. The proposition below states that this does not happen for reduced words. 
The second assertion on the uniqueness of generators graded by simple roots will be used in later computation of braid automorphisms.

The statement was first proved by Lusztig for simply laced Lie types 
$\LT{ADE}$ in \cite{Lu90a}. The extension of the result to general Lie types can be found, for example, in \cite[Proposition~8.20]{Jan96}. Keeping our presentation self-contained, we offer below a more explicit version of the proof, based the arguments in \cite{Lu90a}, formulae in \cite{lu}, as well as Lemma~\ref{lem:Weyl_ijextr}. It applies to the regular power quantum algebras $\Uq^{\pm}$ over $\Zqqn{\dpone}\,$.

\begin{prop}\label{prop:wordgencont}
Suppose $w\in\wordset$ is any non-empty reduced word. Then $E_w\in \Uq^+$ and $F_w\in \Uq^-$.

Moreover, if $\wordroot(w)=\alpha_k\in\sroots\,$, then $E_w=E_k$ and $F_w=F_k\,$.
\end{prop}

\begin{proof} We first provide the proof of the statement for a rank 2 subalgebra $\Uqrest{}{i,j}$ generated by $\{E_i, E_j, F_i, F_j, K_i, K_j\}$ for some fixed pair of indices $i\neq j\,$. The notations  $\Uqrest{\pm}{i,j}$ and minimal ground rings are chosen analogous to those for $\Uq\,$. 
Omitting the trivial $A_1\times A_1$ case, we thus consider the Lie types $\LT{A}_2$\,, $\LT{B}_2=\LT{C}_2$\,, and $\LT{G}_2\,$. As in Section~39.2 of \cite{lu} assume that $\alpha_i$ is the shorter root. So, $A_{ji}=-1$ and $A_{ij}=-\edgenum$ with $\edgenum\in\{1,2,3\}\,$, as well as $d_i=1$ and $d_j=\edgenum\,$. 

The respective sets of reduced words $\wordsetrest{i,j}$ are given by alternating sequences of $w_i$ and $w_j$ of length at most $m_{ij}\in\{3,4,6\}$ have been described in Section~\ref{subsec:rank2coxsys} and subject to \eqref{eq-ATrels}. 

Note that the automorphisms $\Tinv_i$ from \eqref{eq:defLeasy} through \eqref{eq:defTserre} coincide with the $T'_{i,-1}$ 
operators defined in Section~37.1.3 of \cite{lu}. There, Lusztig computes the $E_z$ explicitly for every non-empty reduced word $z\in\wordsetrest{i,j}$ in Section~39.2.2 for $\edgenum=3$\,, in Section~39.2.3 for $\edgenum=2$\,, and in Section~39.2.4 for $\edgenum=1$\,. The (a) sequences in each section contain all words $z$ with $\tau(z)=j$\,, that is, ending in the long root generator, and the (b) sequences correspond to words $z$ starting in $\tau(z)=i$\,. 

For $\edgenum=1$ or $2$, each $E_z$ is expressed as an element $x_{1,m;-1}$ or $x'_{1,m;-1}$ in the subalgebra generated by $E_i$ and $E_j$ over $\Zqqn{m+1}\,$. In the $\edgenum=3$ case, additional elements $\hat x$ and $\hat x'$ are  included, which are easily seen to be elements of the subalgebra over $\Zqqn{4}\,$. Hence, each $E_z$ is in $\Uqrest{+}{ij}\,$ defined over $\Zqqn{\dpone}$ as desired. 

For a word $w\in \wordsetrest{i,j}$ of maximal length $\length(w)=m_{ij}$\,, the formulae in Section~39.2 in \cite{lu} imply the identity
$E_w=E_{k}$ where $k\in\{i,j\}$ such that $\wordroot(w)=\alpha_k\,$. 
So, for example, we have $\Tinv_i\Tinv_j(E_i)=E_j$ in the $\edgenum=1$ case and 
 $\Tinv_i\Tinv_j\Tinv_i(E_j)=E_j$ in the $\edgenum=2$ case. The same relations hold for $i$ and $j$ exchanged, as well as analogous relations for $\edgenum=3\,$ according to the computations in \cite{lu}. 
The respective relations for the $F_i$ generators are obtained by application of the $\Cartaninv$ involution, completing the proof in the rank 2 case.  

The proof of the statement for higher ranks extends the induction argument in the length $\len w$ in the proof of Proposition~1.8 of \cite{Lu90a} to all Lie types. 
Let $s=\Weylpres(w^\flat)$ and $i=\tau(w)$ so that $s(\alpha_i)=\wordroot(w)\in\proots$ by  Lemma~\ref{lem:red=pos} and $E_w=\Tinv_s(E_i)\,$. Also $\len {s\cdot s_i}=\len w=\len {w^\flat}+1=\len s +1\,$. 

Following Lemma~\ref{lem:Weyl_ijextr} we next write $s=t\cdot r$ with $r\in\Weylrest{i,j}\,$. Since $\len s = \len t +\len r$ we may infer from \eqref{eq:Thomrel} that $\Tinv_s(E_i)=\Tinv_t\circ\Tinv_r(E_i)=\Tinv_t(E_z)$ where 
$E_z=\Tinv_r(E_i)$\,. Here $z=y\cdot w_i$ with $y\in\wordsetrest{i,j}$ and $r=\Weylpres(y)$ as in Lemma~\ref{lem:Weyl_ijextr}. Now $\len s = \len t +\len r$ and $\len {s\cdot s_i} = \len s +1$ also imply 
$\len {r\cdot s_i} = \len r +1$ so that $z$ is indeed a reduced word. The statement above for the rank 2 case now implies that $E_z\in\Uqrest{+}{i,j}$ and, thus, can be written as a $\Zqqn{\dpone}$ combination of monomials in $E_i$ and $E_j\,$.

Let now $x\in\wordset$ be a reduced word presenting $t=\Weylpres(x)$\,. With $t(\alpha_i), t(\alpha_j)\in\proots$ implied by Lemma~\ref{lem:Weyl_ijextr}, we also know that $u_i=x\cdot w_i$ and $u_j=x\cdot w_j$ are reduced words. Since  $\len r\geq 1$ we have $\len t<\len s$ and, hence, $\len {u_i}, \len{u_j} <\len w\,$. Thus, by induction hypothesis, we have that
$E_{u_i}=\Tinv_t(E_i),\,E_{u_j}=\Tinv_t(E_j)\in\Uq^+\,$. Since $E_z$  is a combination of monomials in $E_i$ and $E_j$ we conclude that $E_w=\Tinv_t(E_z)$ is the same combination of monomials in $E_{u_i}$ and $E_{u_j}$\,. In particular, $E_w\in\Uq^+$ as desired. 

Suppose, further, that $\wordroot(w)=s(\alpha_i)=\alpha_k\in\sroots\,$. By Lemma~\ref{lem:Weyl_ijextr} we then have $r\cdot s_i=\longtwoweyl{ij}$\,, meaning that $z=y\cdot w_i\in \wordsetmaxrest{i,j}\,$ and $r(\alpha_i)\in\{\alpha_i,\alpha_j\}\,$. The rank 2 results above assure that $E_z=E_m$ where $m\in\{i,j\}$ and $r(\alpha_i)=\alpha_m$\,. Also, $t(\alpha_m)=t(r(\alpha_i))=s(\alpha_i)=\alpha_k\,$. The induction hypothesis, therefore, implies that 
$\Tinv_t(E_m)=E_k$  so that also $\Tinv_s(E_i)=\Tinv_t(\Tinv_r(E_i))=\Tinv_t(E_m)=E_k\,$ as claimed. 

The respective statements for the $F_i$ follow by applications of the $\Cartaninv\,$ involution. 
\end{proof}

An important special case of Proposition~\ref{prop:wordgencont} are words $z\in\wordsetmax$ of maximal length. 
If $j=\tau(z)\,$,   recall from \eqref{eq:defalt_complwords} that $q_j=\Weylpres(z^\flat)=\longweyl\cdot s_j=s_{\dyninv(j)}\cdot \longweyl\,$, where $\dyninv$ is the involution of the Dynkin diagrams as defined at the start of Section~\ref{subsec:ComplWords}.
As noted in the proof of Lemma~\ref{lm:complwordprops},
this yields $\wordroot(z)=q_j(\alpha_j)=\alpha_{\dyninv(j)}\,$. Thus, 
\begin{equation}\label{eq:Elong=simple}
    E_z=E_{\dyninv(\tau(z))} \qquad\mbox{for all}\quad z\in\wordsetmax\,.
\end{equation}

The next formulae summarize notations for two relevant renormalizations  of a generator $E_w$ for some word $w\in\wordset$ with $\alpha=\wordroot(w)\,$. The quantities $d_\alpha$\,, $q_\alpha$\,, and $[n]_\alpha$ are as defined in \eqref{eqn:quantumnumbers} and $k\in\mathbb N\,$.
\begin{align}
    E^{(k)}_w&=\frac 1 {\,[k]!_\alpha\!\!}E_w^k\label{eq:def:EgenDivPow}\\
    \Esing_w&=(q_\alpha^{-1}-q_\alpha)E_w=(q^{-1}-q)[d_\alpha]E_w\label{eq:def:EgenSing}
\end{align}

The first extends Lusztig's {\em divided power generators}, already mention in Section~\ref{subsec:qgroups-genrel} for simple roots, to all positive roots and \emph{a priori}
requires the ground ring to contain the reciprocals of all quantum numbers, such as for $\Zqqn{\infty}$ or $\mathbb Q(q)$. 

The {\em singularized generators} $\Esing_w$ coincide,
for appropriate ordering and  up to signs, with the generators $\overline{E}_\beta$ introduced in \cite[Sec 12.1]{dcp93a}. In calculations, this renomalization often removes the $(q-q^{-1})$ denominators, such as in 
\eqref{eq:EFcomm}, and avoids the use of the extended ground rings such as  $\Zqqv$ or $\Zqqvn{n}$\,.

\subsection{Garside Automorphisms and Dictionary between Generators}\label{subsec:GarsAut+DictGens}

In this section, we aim to provide precise correspondences between Lusztig's and our definition of generators, which entails various useful identities. For a reduced word
$w=w_{i_1}\ldots w_{i_k}\,$, the $T_i$-analog of \eqref{eq:EFwordDef} is defined as 
\begin{equation}\label{eq:defEgenLu}
    \ELu_w=T_{w^\flat}(E_{\tau(w)})=T_{i_1}\ldots T_{i_{k-1}}(E_{i_k})\,,
\end{equation}
with the analogous formula for $\FLu_w\,$. As in  \eqref{eq:Genbyroot}, for a positive root $\alpha\in\proots$, we set  $\ELu_\alpha=\ELu_{z[\alpha]}$\,, matching the definition of generators in \cite{Lu90a,lu90b} for an adequately chosen reduced word $z\in\wordsetmax$ of maximal length. They can be related to our convention using the $\Kinvaut$ involution as well as the notion of complementary words defined in Section~\ref{subsec:ComplWords}.

\begin{prop}\label{prop:MohGenConv}
For non-empty reduced words $v, w\in\wordset$, assume that $w\wcomplrel v$. Then 
    $$
    \ELu_w=\Kinvaut(E_w)=E_v\qquad\mbox{and}\qquad \FLu_w=\Kinvaut(F_w)=F_v\,. 
    $$
\end{prop}

\begin{proof} The first equality is immediate from \eqref{eq:LTconjinv} expressed as $\Kinvaut\circ\Tinv_{i_1}\circ\ldots \circ\Tinv_{i_{k-1}} = T_{i_1}\circ\ldots \circ T_{i_{k-1}}\circ\Kinvaut\,$. 
For the second equality set $s=\Weylpres(v^\flat)$\,, $t=\Weylpres(w^\flat)$\,, and 
$j=\tau(v)\,$.
As in \eqref{eq:defalt_complwords} we have $t^{-1}\cdot s=q_j$ with $\len{q_j}=\len{t^{-1}}+\len{s}$ so $\Tinv_{q_j}=\Tinv_{t^{-1}}\circ \Tinv_s$ by 
\eqref{eq:Thomrel}. 
With $\Tinv_{q_j}(E_j)=E_{\dyninv(j)}$ implied by \eqref{eq:Elong=simple} as well as 
\eqref{eq:LTconjinv}, we thus compute
$\Kinvaut(E_w)=\Kinvaut(\Tinv_{t}(E_{\dyninv(j)}))=\Tinv_{t^{-1}}^{-1}(\Kinvaut(E_{\dyninv(j)}))=\Tinv_{t^{-1}}^{-1}(E_{\dyninv(j)})=\Tinv_{t^{-1}}^{-1}(\Tinv_{q_j}(E_{j}))=\Tinv_s(E_j)=E_v\,$. The assertion for $F_w$ follows again by application of $\Cartaninv\,$, which commutes with $\Kinvaut\,$. 
\end{proof}

As a second application of Proposition~\ref{prop:wordgencont} we are able to explicitly compute the Lusztig automorphism $\Tinv_{\garside}$ 
associated to the Garside element $\garside=\mathscr s(\longweyl)\in\mathscr A^+$
as in \eqref{eq:gardsidecomm}. 
In our conventions as explained in Section~\ref{subsec:LATM}, 
we may also write  $\Tinv_{\longweyl}=\Tinv_{\garside}$\,, where $\longweyl\in\Weyl$ is the longest element. 

The relation also involves  the automorphism $\DynkInv:\Uq\rightarrow\Uq\,$ that implements the involution
$\dyninv$ on Dynkin diagrams defined at the beginning of Section~\ref{subsec:ComplWords}.
It is given on generators  by 
\begin{align}\label{eq:defDynkInv}
    \DynkInv(E_i)=E_{\dyninv(i)}\,,
    &&
    \DynkInv(F_i)=F_{\dyninv(i)}\,,
    &&
    \DynkInv(K_i)=K_{\dyninv(i)}\,.
\end{align} 
The maps $\htsign, \qitgen\in\mathrm{Hom}(\mathbb Z^{\sroots},\Zgen^\star)$ are as in  
\eqref{eq:defhqi} and \eqref{eq:SXiMhoRel} and the remaining automorphisms are as defined in Section~\ref{sec:gradingauts}. 

\begin{cor} \label{cor:GarsideAutom}
   The Gardside automorphism and its square are given by the following composites of elementary automorphisms:
    \begin{equation}\label{eq:GarsideAutom}
    \Tinv_{\garside}=\DynkInv\circ\Cartanaut\circ\Kscale{\htsign}{\idsymm}
    =\DynkInv\circ\Cartanaut\circ\Kinvaut\circ S\qquad\mbox{and}\qquad \Tinv_{\garside}^2=\Kscale{\qitgen^{-1}}{2\idsymm}\;.
    \end{equation}
\end{cor}

\begin{proof} As explained above, we express $\Tinv_{\longweyl}=\Tinv_{\garside}$ in terms of the longest Weyl element. With the same notation as in \eqref{eq:LTconjinv} or the proof of Proposition~\ref{prop:MohGenConv} we have $\longweyl=s_{\dyninv(i)}\cdot q_i$ with $\len{\longweyl}=\len{q_i}+1$ and $\Tinv_{q_i}(E_i)=E_{\dyninv(i)}\,$ for given $i\in\{1,\ldots,n\}\,$. Hence,  
\begin{equation}\label{eq:GarsideAutomEi}
\begin{aligned}
    \Tinv_{\longweyl}(E_i)&=\Tinv_{\dyninv(i)}(\Tinv_{q_i}(E_i))=\Tinv_{\dyninv(i)}(E_{\dyninv(i)})=-K_{\dyninv(i)}^{-1}F_{\dyninv(i)}=\\
    &=\DynkInv(-K_i^{-1}F_i))=\DynkInv\circ\Cartanaut(-K_iE_i))=\DynkInv\circ\Cartanaut\circ\Kscale{\htsign}{\idsymm}(E_i)\,,
\end{aligned}
\end{equation}
which proves the first identity \eqref{eq:GarsideAutom} on $E_i\,$. The verification for $F_i$ follows, as usual, by application of $\Cartaninv$ to  \eqref{eq:GarsideAutomEi}, and the automorphisms on both ends of \eqref{eq:GarsideAutomEi} clearly map $K^\nu$ to $K^{-\dyninv(\nu)}$, where $\dyninv$ is the obvious extension to $\mathbb Z^{\sroots}\,$. The formula for the square of the Garside automorphism is readily derived from the first identity and the commutation relations in \eqref{eq:InvKScalComm}.
\end{proof}

Aside from $\Kinvaut$ and $\Cartaninv$ introduced in Section~\ref{subsec:Kscale} we also record here the action of 
the conjugation automorphism $\Qinvaut$ on the generators. The formula contains an additional factor in $\Zgen^\star$, which only depends on the root associated with the word label. 
\begin{lem}\label{lem:ConjWordGen}
Let $w\in\wordset$ be a non-empty, reduced word and $\beta=\wordroot(w)$ the associated root. Then
\begin{equation}\label{eq:ConjWordGen}
    \Qinvaut(E_w)= (-1)^{\height(\beta)-1}
       q^{\symbrack{\weightvecchar}{\beta}-d_\beta}E_w  
       \qquad\mbox{and}\qquad 
  \Qinvaut(F_w)= (-1)^{\height(\beta)-1}
       q^{-\symbrack{\weightvecchar}{\beta}+d_\beta}F_w\,.    
\end{equation}    
\end{lem}

\begin{proof} Let $s=\wordroot(w^\flat)$ and $i=\tau(w)$ so that $E_w=\Tinv_s(E_i)\,$ and $\beta=s(\alpha_i)\,$. 
  Relation \eqref{eq:ArtinConjComm} implies 
  $\Qinvaut(E_w)=\Qinvaut(\Tinv_s(E_i))=(\Kscale{\reldrchar_{s}}{0})^{-1}(\Tinv_s(\Qinvaut(E_i)))=(\reldrchar_{s}(\beta))^{-1}\Tinv_s(E_i)=(-1)^{\symbrack{\sumdesccoroots{s}}{\beta}}
       q^{\symbrack{\sumdescroots{s}}{\beta}}E_w\,$. Now, by \eqref{eq:thetaviarho}, 
       $\symbrack{\sumdescroots{s}}{\beta}=\symbrack{\weightvecchar-s(\weightvecchar)}{\beta}=
       \symbrack{\weightvecchar}{\beta}-\symbrack{s(\weightvecchar)}{s(\alpha_i)}=
       \symbrack{\weightvecchar}{\beta}-\symbrack{\weightvecchar}{\alpha_i}=\symbrack{\weightvecchar}{\beta}-d_i=\symbrack{\weightvecchar}{\beta}-d_\beta\,$ since $d_\beta$ is $\Weyl$-invariant. Also, 
       \eqref{eq:sumdesccoheight} implies $\symbrack{\sumdesccoroots{s}}{\beta}=\height(\beta)-\height(\alpha_i)=
       \height(\beta)-1\,$. The relation for $F_w$ follows by application of  $\Cartaninv$ and 
       \eqref{eq:CinvEFwords}. 
\end{proof}

In the remainder of this section, 
we discuss the generators and actions more explicitly in the rank 2 case. 
As in the proof of Proposition~\ref{prop:wordgencont},   let $\Uqrest{+}{i,j}$ be the subalgebra generated by $E_i$ and $E_j$
for two indices $i$ and $j$. We assume $A_{ji}=-1\,$, implying $d_i=1$ and
$-A_{ij}=d_j=\edgenum\in\{1,2,3\}\,$, corresponding to Lie type $\LT{A}_2$\,, $\LT{B}_2$\,, or $\LT{G}_2\,$, respectively, for $\Uqrest{+}{i,j}$\,. 
As before, let $m=m_{ij}\in\{3,4,6\}$ be the length of the longest element in $\Weylrest{i,j}\,$. 

For $r\in\{1,\ldots, m\}$ let $a_r=w_iw_j\ldots$ and $b_r=w_jw_i\ldots$ be the two possible reduced words of length $r$ in $\wordsetrest{i,j}\,$ as in \eqref{eq:def:abwords}. Combining Lemma~\ref{lm:rk2complw} and Proposition~\ref{prop:MohGenConv} now yields the identities
\begin{equation}\label{eq:Eabcompl}
  \Kinvaut(E_{a_r})=E_{b_{m-r+1}}=\ELu_{a_r}  \qquad\mbox{and}\qquad 
   \Kinvaut(F_{a_r})=F_{b_{m-r+1}}=\FLu_{a_r}\,. 
\end{equation}

Since later proofs will refer to and invoke several commutation relations in  \cite{lu90b}, we provide here a dictionary from the notation used there for generators in a rank 2 system to our conventions. Specifically, in Section~{5.1} of \cite{lu90b} Lusztig uses subscripts of the form $1\ldots 1 2\ldots 2$ for the $E$-generators. Replacing $i\leftrightharpoons 1$ to index the short root and $j\leftrightharpoons 2$ for the (possibly) long root, the definition there can be summarized in the formula
\begin{equation}
    \ELu_{\underbracket[0.5pt]{\scriptstyle i\ldots i}_k\underbracket[0.5pt]{\scriptstyle j\ldots j}_p}=\ELu_{\longtwoword{ji}[\alpha]}\qquad\mbox{  with} \quad\alpha=k\alpha_i+p\alpha_j\,\in\prootstwo{ij}\;.
\end{equation}
That is, the frequency of $i$'s and $j$'s determined the coefficients of the root $\alpha$ in the root sublattice. This is identified with the unique word $b_r\leqRB \longtwoword{ji}$ as in  \eqref{eq:def:abwords} for which $\alpha=\wordroot(b_r)$\,, with $r$ determined by the position of the root in the ordering in \eqref{eq:rank2rootorder}. The respective $E$-generator in our convention is then obtained via \eqref{eq:Eabcompl}, which, in turn, can be expressed in the notation from \eqref{eq:Esubseqind}.

For example, in the $\LT{G}_2$ ($\edgenum=3$, $m=6$) case, $\ELu_{iiijj}$ indicates in \cite{lu90b} the generator associated to root $3\alpha_i+2\alpha_j$\,. In the ordering for $\longtwoword{ji}$ in
\eqref{eq:rank2rootorder} it appears in position $r=3$ so that $\ELu_{iiijj}=\ELu_{b_3}=\ELu_{w_jw_iw_j}=\ELu_{(jij)}=T_jT_i(E_j)\,$. From \eqref{eq:Eabcompl} we next obtain
that $\ELu_{b_3}=E_{a_{6+1-3}}=E_{a_4}=E_{w_iw_jw_iw_j}=E_{(ijij)}=\Tinv_i\Tinv_j\Tinv_i(E_j)\,$. 
Similarly, the generators associated to simple roots are readily identified as  
\begin{equation}\label{eq:Edictsimple}
\ELu_j=\ELu_{w_j}=E_{\longtwoword{ij}}=E_{(\underbracket[0.5pt]{\scriptstyle ij\ldots}_m)}=E_j
\qquad\mbox{and}
\qquad 
\ELu_i=\ELu_{\longtwoword{ji}}=\ELu_{(\underbracket[0.5pt]{\scriptstyle ji\ldots}_m)}=E_{(i)}=E_i\,.
\end{equation}
The associated simple roots appear either as maximal or minimal elements in the two total orders on $\prootstwo{ij}$\,. 

The full dictionary translating between the notation in \cite{lu90b} and our conventions is as follows. Here, generators are ordered by the roots and $\letwt{\longtwoword{ji}}$ as in \eqref{eq:rank2rootorder}, omitting the first and last terms given in \eqref{eq:Edictsimple} for even $m$\,. For $m=3$ the identity 
is $\ELu_i=E_i=E_{(iji)}\,$.

\begingroup
\allowdisplaybreaks
\begin{align}
\label{eq:LusztigDictionary} 
\begin{split}
\edgenum=1 \quad &\left\{{\quad 
\begin{aligned}
\rule[-2mm]{0mm}{5mm}  
\ELu_{ij} = \ELu_{(ji)}=E_{(ij)}\,\;   
\end{aligned}
}\right.
\\
\rule{0mm}{11mm}
\edgenum=2 \quad &\left\{{\quad 
\begin{aligned}
\rule[-2mm]{0mm}{5mm}  
    \ELu_{ij}&= \ELu_{(ji)}=E_{(iji)}\,,
    &
    \ELu_{iij}&= \ELu_{(jij)}=E_{(ij)}
\end{aligned}
}\right.
\\
\rule{0mm}{15mm}
\edgenum=3 \quad &\left\{{\quad 
\begin{aligned}
    \ELu_{ij}&= \ELu_{(ji)}=E_{(ijiji)}\,,
    &
    \ELu_{iiijj}&=\ELu_{(jij)}=E_{(ijij)}\,,
    \\
    \rule{0mm}{6mm}\ELu_{iij}&= \ELu_{(jiji)}=E_{(iji)}\,,
    &
    \ELu_{iiij}&= \ELu_{(jijij)}=E_{(ij)}
\end{aligned}
}\right. 
\end{split}
\end{align}
\endgroup

\bigskip

\subsection{Monomial Expressions, Exponent Sets, and Their Spans}\label{subsec:Emonomials}
The PBW type bases of interest for us are given by monomials whose exponents are labeled by positive roots. To describe these more formally, we introduce {\em exponents sets} 
\begin{equation}\label{eq:ExpsetDef}
    \expsetup s =\mathrm{Maps}\left({\descroots s,\nnN}\right)=\nnN^{\descroots s}
\qquad\mbox{and}\qquad 
\expset s =\mathrm{Maps}\left({\descroots s,\mathbb Z_{}}\right)=\mathbb Z^{\descroots s}\,,
\end{equation}
where the inversion sets $\descroots s$ are as in (\ref{eq:defdescroots}).
Additionally, we use the more compact notation for the maximal sets
\begin{equation}\label{eq:ExpsetDefMax}
    \expsetupmax=\expsetup {\longweyl} =\mathrm{Maps}\left({\proots,\nnN}\right) =\nnN^{{\proots}}
    \qquad\mbox{and}\qquad 
    \expsetmax=\expset {\longweyl} =\mathrm{Maps}\left({\proots,\mathbb Z}\right)\,. 
\end{equation}

Suppose $s=t_1\cdot t_2$ with $\length(s)=\length(t_1)+\length(t_2)$. Recall from Corollary~\ref{cor:sumdescroots}   that $\descroots{s}$ is then the disjoint union of $\descroots{t_1}$ and $t_1\!\left({\descroots{t_2}}\right)$, which, correspondingly, implies the following bijections of exponent sets:
\begin{align}\label{eq:ExpsetCorr}
    \hspace*{12mm} \expsetup {t_1} \times \expsetup {t_2}&\rightarrow \expsetup{s}\,&&:&\quad (\psi_1,\psi_2)\,&\mapsto \,\psi=\psi_1\sqcup (t_1^{-1})^*(\psi_2)\,,\nonumber\hspace*{12mm}\\
     \expsetup{s} &\rightarrow\expsetup {t_1} \times \expsetup {t_2} \,&&:&\quad \psi \,&\mapsto \,(\restr_1^*\psi,\restr_2^*t_1^*(\psi)))\,.        
\end{align}
Here, $s^*(\psi)=\psi\circ s$ denotes the usual pull-back, the disjoint union of maps indicates the respective definition on the two sets, and $\restr_i^*$ is the restriction to $\descroots{s_i}$\,.  
 
Suppose $w\in \wordset$\,, $s=\Weylpres(w)\in\Weyl\,$ with $L=\length(w)=|\descroots s|\,$, and $\psi\in\expsetup{s}\,$. 
Following \eqref{eq:root2word} and the explanations in Section~\ref{subsec:descroots}, we write $\beta_j=\wordroot(w[1,j])$ for the ordered root sequence implied by $w$ as well as  $w^j=w[1,j]=w[\beta_j]$ for the words $\leqRB w$\,. Given these conventions, we define the following monomials, expressed in two equivalent notations and for each of two opposing directions of multiplication.
\begin{equation}\label{eq:PBWEgenDef}
    \begin{aligned}
    \Ebase{w}{\psi}&=E_{w^1}^{\psi(\wordroot(w^1))}\ldots E_{w^L}^{\psi(\wordroot(w^L))}
    =E_{w[\beta_1]}^{\psi(\beta_1)}\ldots E_{w[\beta_L]}^{\psi(\beta_L)}\vspace*{2mm}
    \\
 \Ebaseopp{w}{\psi}&=E_{w^L}^{\psi(\wordroot(w^L))}\ldots E_{w^1}^{\psi(\wordroot(w^1))} 
    =E_{w[\beta_L]}^{\psi(\beta_L)}\ldots E_{w[\beta_1]}^{\psi(\beta_1)} \,
\end{aligned}
\end{equation}
Generators $\Fbase{w}{\psi}$ and $\Fbaseopp{w}{\psi}$ are defined by the same formulae \eqref{eq:PBWEgenDef} with $E$ replaced by $F$. Using that $\Cartaninv$ is an  anti-homomorphism we  find from \eqref{eq:CinvEFwords} the direction reversing relations 
\begin{equation}\label{eq:EFCartOrder}
    \Cartaninv(\Ebase{w}{\psi})=\Fbaseopp{w}{\psi} \qquad \mbox{and} \qquad \Cartaninv(\Fbase{w}{\psi})=\Ebaseopp{w}{\psi}\;. 
\end{equation}

Similarly, from Lemma~\ref{lem:ConjWordGen} we find that the conjugation automorphism reverses 
the direction for the same type of monomial up to a unit factor. Specifically, if we denote $\vec\psi=\sum_\alpha\psi(\alpha)\alpha\in\mathbb Z^{\sroots}$\,, $d(\psi)=\sum_\alpha\psi(\alpha)d_\alpha\in\nnN$\,, and $n(\psi)=\sum_\alpha\psi(\alpha)\in\nnN$ the following formulae hold for the $\Qinvaut$ anti-involution. 
    \begin{equation}\label{eq:EFConjOrder}
    \begin{aligned}
    \Qinvaut(\Ebase{w}{\psi})&= (-1)^{\height(\vec\psi)-n(\psi)}
       q^{\symbrack{\weightvecchar}{\vec\psi}-d(\psi)}\Ebaseopp{w}{\psi}  \vspace*{4mm}
       \\ 
 \Qinvaut(\Fbase{w}{\psi})&= (-1)^{\height(\vec\psi)-n(\psi)}
       q^{-\symbrack{\weightvecchar}{\vec\psi}+d(\psi)}\Fbaseopp{w}{\psi} 
    \end{aligned}    
\end{equation}

Suppose $w=u_1\cdot u_2\in\wordset$ with $\length(w)=\length(u_1)+\length(u_2)$ and $s=\Weylpres(w)$ and $t_i=\Weylpres(u_i)$ for $i=1, 2\,$. Moreover, assume that $\psi\in\expsetup{s}$ is assigned to 
$(\psi_1,\psi_2)$ in the correspondence in \eqref{eq:ExpsetCorr}. Then the relation from \eqref{eq:EFdefRecurs} applied to each term of the monomial and the fact that the $\Tinv_i$ are homomorphisms implies the following recursions for monomials:
\begin{equation}\label{eq:PBWgenRecur}
    \Ebase{w}{\psi}\,=\,\Ebase{u_1}{\psi_1}\,\cdot\,\Tinv_{t_1}\!\left({\Ebase{u_2}{\psi_2}}\right)
    \qquad\mbox{and}\qquad 
    \Ebaseopp{w}{\psi}\,=\,\Tinv_{t_1}\!\left({\Ebaseopp{u_2}{\psi_2}}\right)\,\cdot\,\Ebaseopp{u_1}{\psi_1}\,.
\end{equation}

Analogous monomial expressions $\EbaseLU{w}{\psi}$ and $\EbaseoppLU{w}{\psi}$ are defined 
by replacing the $E_{w^j}^{\psi(\beta_j)}$ factors in \eqref{eq:PBWEgenDef}
by Lusztig's $\ELu^{\psi(\beta_j)}_{w^j}$ generators, with $\ELu_w$ as defined in \eqref{eq:defEgenLu}.  
The next proposition extends the correspondence in Proposition~\ref{prop:MohGenConv} to these expressions. 

\begin{prop} \label{prop:mhomonomial}
For a non-empty $w\in\wordset$\,, let $t=\Weylpres(w)\in\Weyl$ and $r=\flweyl{t}$ as in 
\eqref{eq:def:flweyl}. Moreover, let $\psi\in\expsetup t \,$ and
$\phi=r_*\psi=\psi\circ r\in \expsetup {t^\winvchar} \,$. Then 
    \begin{equation}\label{eq:mhomonomial}
        \Kinvaut(\Ebase{w}{\psi})\,=\,\EbaseoppLU{w}{\psi}\,=\,\Tinv_r\left({\Ebase{w^\winvchar}{\phi}}\right)\,. 
    \end{equation} 
The analogous relation holds with directions flipped in all terms as well as with $E$ generators replaced by $F$ generators. 
\end{prop}
\begin{proof}
    As noted in Section~\ref{subsec:coxeter}, we can find $u\in\wordset$ such that $z=w\cdot u$ is a reduced word of maximal length. So, with $N=\len z $ and $L=\len w $ we have $\len u =N-L\,$. Moreover, $g=\Weylpres(u)=t^{-1}\cdot \longweyl=\longweyl \cdot \dyninv(t)^{-1}=\longweyl \cdot t^\winvchar\,$ and, hence, $\Weylpres(u^\winvchar) = g^\winvchar=t \cdot \longweyl=\flweyl{t}=r\,$. Now, Lemma~\ref{lem:longinvol}{\em \ref{item:longinvol:daggerdesc}} implies $r(\descroots{t^{\winvchar}})=\flweyl{t}(\descroots{t^{\winvchar}})=\descroots{t}$ so that, indeed,
    $r_*$ maps $\expsetup t$ to $\expsetup {t^{\winvchar}}\,$.

    Assume that $j\leq L\,$, and let $y=z^\winvchar=u^\winvchar\cdot w^\winvchar$. Then we have $z[1,j]=w[1,j]$ and 
    $y[1,N+1-j]=u^\winvchar\cdot (w^\winvchar[1,L+1-j])$\,. Thus,  
    Lemma~\ref{lm:complwordprops} implies
     $w[1,j]\wcomplrel u^\winvchar\cdot (w^\winvchar[1,L+1-j])\,$, which entails together with 
    Proposition~\ref{prop:MohGenConv} and \eqref{eq:EFdefRecurs} that $\Kinvaut(E_{w[1,j]})=\ELu_{w[1,j]}
    =E_{u^\winvchar\cdot (w^\winvchar[1,L+1-j])}=\Tinv_{u^\winvchar}(E_{w^\winvchar[1,L+1-j]})=\Tinv_r(E_{w^\winvchar[1,L+1-j]})\,$. 
    From Lemma~\ref{lm:complwordprops} we also find that $\beta_j=\wordroot(w[1,j])=
    \wordroot(u^\winvchar\cdot (w^\winvchar[1,L+1-j]))=r(\delta_{L+1-j})$ where $\delta_i=\wordroot(w^\winvchar[1,i])\,$. Thus,
    \begin{align}
            \Kinvaut(\Ebase{w}{\psi})&=\Kinvaut\left({E_{w[1,1]}^{\psi(\beta_1)}\ldots E_{w[1,L]}^{\psi(\beta_L)}}\right)
    &&=\Kinvaut\left({E_{w[1,L]}^{\psi(\beta_L)}}\right)\ldots \Kinvaut\left({E_{w[1,1]}^{\psi(\beta_1)}}\right)\label{eq:mhomonmcomp}\\
    & =\Tinv_r\left({E_{w^\winvchar[1,1]}}\right)^{\psi(\beta_L)}\ldots \Tinv_r\left({E_{w^\winvchar[1,L]}}\right)^{\psi(\beta_1)}
    &&=\Tinv_r\left({E_{w^\winvchar[1,1]}^{\psi(r(\delta_1))}}\right)\ldots \Tinv_r\left({E_{w^\winvchar[1,L]}^{\psi(r(\delta_L))}}\right)\nonumber\\
    &=\Tinv_r\left({E_{w^\winvchar[1,1]}^{\phi(\delta_1)} \ldots   E_{w^\winvchar[1,L]}^{\phi(\delta_L)}}\right)
    &&=\Tinv_r(\Ebase{w^\winvchar}{\phi})\,,\nonumber
        \end{align} 
which shows equality of the first and third term in \eqref{eq:mhomonomial}. Furthermore, the last expression in the first line of \eqref{eq:mhomonmcomp} is readily identified with 
$\ELu^{\psi(\beta_L)}_{w[1,L]}\ldots \ELu^{\psi(\beta_1)}_{w[1,1]}=\EbaseoppLU{w}{\psi}\,$.
Identities for reversed directions follow analogously. The respective statements for $F$ generators are obtained by application of $\Cartaninv\,$, using 
\eqref{eq:EFCartOrder} as well as commutation with both $\Kinvaut$ and $\Tinv_t\,$. 
\end{proof}

We will use this result mainly in the case of reduced words of maximal lengths, for which $t=\longweyl$ and $r=\id\,$. 
That is, let $z\in\wordsetmax$ and $\psi\in \expsetupmax\,$ as in (\ref{eq:ExpsetDefMax}). 
Then \eqref{eq:mhomonomial} specializes to
\begin{equation}\label{eq:mhomonomax}
\Kinvaut(\Ebase{z}{\psi})\,=\,\EbaseoppLU{z}{\psi}\,=\,{\Ebase{z^\winvchar}{\psi}}\,. 
\end{equation} 
The respective identity for the opposite  directions is  readily derived from commutation of $\Kinvaut$ with
$\Qinvaut$. The variants with $E$'s replaced by $F$'s follows from commutation with $\Cartaninv\,$.

Further relations that directly relate (maximal) monomial expressions in the $T$ and $\Tinv$ conventions are provided
by the composite automorphism $\Kconaut=\Qinvaut\circ\Kinvaut$. For example, for $z$ and $\psi$ as above we find
 \begin{equation}\label{eq:Kconautmonom}
    \begin{aligned}
    \Kconaut(\Ebase{z}{\psi})&= (-1)^{\height(\vec\psi)-n(\psi)}
       q^{\symbrack{\weightvecchar}{\vec\psi}-d(\psi)}\EbaseLU{z}{\psi}  
       &=(-1)^{\height(\vec\psi)-n(\psi)}
       q^{\symbrack{\weightvecchar}{\vec\psi}-d(\psi)}\Ebaseopp{z^\winvchar}{\psi}   
    \end{aligned}    
\end{equation}   
using the notation from \eqref{eq:EFConjOrder}. That is, up to unit factors, $\Kconaut$ maps monomials in our convention to those in Lusztig's convention for the same choices of words and directions. Similarly,
formulae for the antipode can be derived from those in \eqref{eq:SXiMhoRel}, \eqref{eq:mhomonomax}, and Proposition~\ref{prop:Kscaleprops}.
\begin{equation}\label{eq:AntipBasis}
\begin{aligned}
    S(\Ebase{z}{\psi})
&=(-1)^{\height(\vec\psi)}q^{-\frac 12(\symbrack{\vec\psi}{\vec\psi}-2\symbrack{\weightvecchar}{\vec\psi})}\Ebase{z^\winvchar}{\psi}K^{-\vec\psi}
&&=(-1)^{\height(\vec\psi)}q^{\frac 12(\symbrack{\vec\psi}{\vec\psi}+2\symbrack{\weightvecchar}{\vec\psi})}K^{-\vec\psi}\Ebase{z^\winvchar}{\psi}\\
\rule{0mm}{6mm}
    S(\Fbase{z}{\psi})
&=(-1)^{\height(\vec\psi)}q^{-\frac 12(\symbrack{\vec\psi}{\vec\psi}+2\symbrack{\weightvecchar}{\vec\psi})}\Fbase{z^\winvchar}{\psi}K^{\vec\psi}
&&=(-1)^{\height(\vec\psi)}q^{\frac 12(\symbrack{\vec\psi}{\vec\psi}-2\symbrack{\weightvecchar}{\vec\psi})}K^{\vec\psi}\Fbase{z^\winvchar}{\psi}
\end{aligned}
\end{equation}

The same equations hold for the basis elements with respect to the opposite directions, given the same is
true for those in (\ref{eq:mhomonomax}).
For later reference, we also record the rank 2 case discussed in Section~\ref{subsec:rank2coxsys}, for which $\longtwoword{ij}$ and $\longtwoword{ji}=\longtwoword{ij}^\winvchar$ are maximal words:
\begin{equation}\label{eq:KinvEtworefl}
 \Kinvaut(\Ebase{\longtwoword{ij}}{\psi})\,=\,\EbaseoppLU{\longtwoword{ij}}{\psi}\,=\,{\Ebase{\longtwoword{ji}}{\psi}}   \,.
\end{equation}
 
 Finally, we introduce notation for basis elements in $\UqQ$ given by divided power generators. 
 Recall from \eqref{eq:def:EgenDivPow} that for $k\in\mathbb N$ and $i=\tau(w)$ we have $E^{(k)}_w=\frac 1 {[k]!_i}E^k_w\,$. 
 Given $w\in\wordset$\,, $L=\len w$\,, $s=\Weylpres(w)\in\Weyl$\,, and $\psi\in\expsetup s$\,, set, as before, $w^j=w[1,j]$\,, $\beta_j=\wordroot(w^j)$\,, and $k_j=\psi(\beta_j)\,$. Set
\begin{equation}\label{eq:PBWEgenDivDef}
 \begin{aligned}
 &   \Edivbase{w}{\psi}=E_{w^1}^{(k_1)}\ldots E_{w^L}^{(k_L)}= \tfrac 1 {[\psi]!} \Ebase{w}{\psi}
      \qquad\mbox{and}\qquad 
     \Edivbaseopp{w}{\psi}=E_{w^L}^{(k_L)}\ldots E_{w^1}^{(k_1)}= \tfrac 1 {[\psi]!} \Ebaseopp{w}{\psi}\\
 \mbox{where }
 &\rule{0pt}{26pt}\hspace*{33mm}[\psi]!=\prod_{j=1}^L[k_j]!_{d_j}=\prod_{\alpha\in\proots}[\psi(\alpha)]!_{\alpha} \quad\mbox{with}\quad d_j=\tfrac 12 \symbrack{\beta_j}{\beta_j} \,.
\end{aligned}
\end{equation}
Here we wrote $[\psi]!$ as a product over $\proots$ rather than $\descroots{s}\,$, tacitly extending the exponent function 
as $\psi(\gamma)=0$ for $\gamma\not\in \descroots{s}\,$. With the same convention we introduce analogous notation for the singularized generators from \eqref{eq:def:EgenSing} as
\begin{equation}\label{eq:PBWEgenSingDef}
 \begin{aligned}
&\Ebasesing{w}{\psi}=\Esing_{w^1}^{k_1}\ldots \Esing_{w^L}^{k_L}=\SingCoeff{\psi}\Ebase{w}{\psi}
      \qquad\mbox{and}\qquad 
  \Ebaseoppsing{w}{\psi}=\Esing_{w^L}^{k_L}\ldots \Esing_{w^1}^{k_1}=\SingCoeff{\psi}\Ebaseopp{w}{\psi}
\\
\mbox{where }
 &\rule{0pt}{26pt}\hspace*{33mm}
\SingCoeff{\psi}=\SingCoeffvar{\psi}{q}=\prod_{\alpha\in\proots}(q^{-1}_\alpha-q_\alpha)^{\psi(\alpha)}
=(q^{-1}-q)^{n(\psi)}\prod_{\alpha\in\proots}[d_\alpha]^{\psi(\alpha)}\,.
\end{aligned}
\end{equation}
As usual, respective elements in $\Uq^-$ are defined by replacing all $E$'s by $F$'s.
See also again \cite[Sec 12.1]{dcp93a} for a variant of the $\Ebasesing{w}{\psi}\,$.

We conclude this section with notation for subsets of generators and their spans determined by subsets of exponent
functions.
Suppose $w\in\wordset$ with $s=\Weylpres(w)\,$. We introduce the monomial  spanning sets corresponding to a subset 
$\mathsf S\subseteq \expsetup{s}$ of exponents as follows:
\begin{equation}\label{eq:PBW_p_basis}
    \basis{w}{\mathsf S,+}  =\left\{{\Ebase{w}{\psi}\,:\,\psi\in \mathsf S}\right\}
    \qquad \mbox{and} \qquad 
    \basisopp{w}{\mathsf S,+}=\left\{{\Ebaseopp{w}{\psi}\,:\,\psi\in\mathsf S}\right\}\,.
\end{equation}
The sets $\basis{w}{\mathsf S,-}$ and $\basisopp{w}{\mathsf S,-}$ are defined analogously with $E$'s replaced by $F$'s. The relation \eqref{eq:EFCartOrder}
then translates to the relations of sets 
\begin{equation}\label{eq:CartBasisRefl}
    \basis{w}{\mathsf S,-}=\Cartaninv\left(\basisopp{w}{\mathsf S,+}\right) \qquad \mbox{and}\qquad \,\basisopp{w}{\mathsf S,-}=\Cartaninv\left(\basis{w}{\mathsf S,+}\right)\,.
\end{equation}
In the case that $\mathsf S=\expsetup{s}$\,, that is, if all exponent functions are included, we use the shorthand $\basis{w}{\pm}=\basis{w}{\expsetup{s},\pm}\,$ and $\basisopp{w}{\pm}=\basisopp{w}{\expsetup{s},\pm}$\,. 

Recursion relations derived from (\ref{eq:PBWgenRecur}) exist only in the special case when the exponent set $\mathsf S$
splits into a respective Cartesian product. More precisely, suppose $w=u_1\cdot u_2$ and $s=t_1\cdot t_2$ with
$t_i=\Weylpres(u_i)$ and $\len{s}=\len{t_1}+\len{t_2}$\,. Assume that there are sets $\mathsf S_i\subseteq\expsetup{t_i}$
such that the bijection from (\ref{eq:ExpsetCorr}) restricts to a bijection 
$\mathsf S\rightarrow \mathsf S_1\times \mathsf S_2$\,. In this case, (\ref{eq:PBWgenRecur}) indeed implies
\begin{equation}\label{eq:BasisRecurs}
\begin{split}
    \basis{w}{\mathsf S,+} &= 
    \basis{u_1}{\mathsf S_1,+}\,\bcdot \Tinv_{t_1}\!\bigl({\basis{u_2}{\mathsf S_2,+}}\,\bigr)
    \qquad\mbox{and}\qquad  
    \basisopp{w}{\mathsf S,+} = 
    \Tinv_{t_1}\bigl({\basisopp{u_2}{\mathsf S_2,+}}\bigr)\bcdot \basisopp{u_1}{\mathsf S_1,+}.
\end{split}
\end{equation} 
The above splitting property $\mathsf S\cong \mathsf S_1\times \mathsf S_2$ may be rephrased as the condition
that $\psi\in\mathsf S$ if and only if $\restr_1^*(\psi)\in\restr_1^*(\mathsf S)$ and 
$\restr_2^*(t_1^*(\psi))\in\restr_2^*(t_1^*(\mathsf S))$, independently. Although very restrictive in general, the property 
applies to several useful situations below. 

Assume now that $\Zgen$ is an integral domain with a ring homomorphism $f\!:\Zqqn{\dpone}\rightarrow\Zgen$\,, and set $\Uqre{\Zgen}^{\pm}=\Uq^{\pm}\otimes_f\!\Zgen$ to be the respective quantum algebra over $\Zgen$\,. We denote the respective $\Zgen$-spans of the above sets in 
$\Uqre{\Zgen}^{\pm}$ by 
\begin{equation}\label{eq:DefSspans}
    \bspan{w}{\mathsf S,\pm}\,=\,\bigl\langle \basis{w}{\mathsf S,\pm}\bigr\rangle_\Zgen
    \qquad\mbox{and}\qquad
    \bspanopp{w}{\mathsf S,\pm}\,=\,\bigl\langle \basisopp{w}{\mathsf S,\pm}\bigr\rangle_\Zgen\,.
\end{equation}

As before, we write $\bspan{w}{\pm}$ or $\bspanopp{w}{\pm}$ if $\mathsf S=\expsetup{s}$. Additional identities for the action
of $\Kinvaut$ on either basis sets or spans, such as $\Kinvaut\bigl(\basis{w}{\mathsf S,+}\bigr)
=\Tinv_r\bigl(\basis{w^\winvchar}{r^*(\mathsf S),+}\,\bigr)$\,, follow from Proposition~\ref{prop:mhomonomial}.

\subsection{Some Commutation Relations}
\label{subsec:Ecommrels} We collect next various commutation relations among the $E_w$ generators that are useful in establishing integral PBW bases and also support later computations of ideal properties. The most basic commutation relation for 
a pair of generators $(A,B)$ is that they {\em skew-commute}, meaning that $BA=zAB$ where $z$ is a unit in the ground ring. An important skew-commutation relation among the $E_w$ generators is given by the following.

\begin{lemma}\label{lm:scommgens} Suppose $w\in\wordset$ with reduced decomposition $w=u\cdot v$ for non-empty $u,v\in\wordset$. Denote $j=\tau(u)$, $\alpha=\wordroot(u)$, and $\beta=\wordroot(w)$. Assume further that $w_j$ does not occur as a letter in $v$. Then
\begin{equation}\label{eq:scommgens}
    E_w E_u=q^{\symbrack{\alpha}{\beta}}E_uE_w\,.
\end{equation}
\end{lemma}
\begin{proof} Since $w_j$ is not a letter in $v$, $E_v$ is an expression in $\{E_i\}_{i\neq j}$ and, hence, commutes with $F_j$\,. Also, by \eqref{eq:gradgens} and \eqref{eq:Kwgrad} we find
$E_vK_j=q^{-\symbrack \nu {\alpha_j}}K_jE_v$ where $\nu=\wordroot(v)\,$\,. Observe further that by \eqref{eq:PBWgenRecur} $E_w=\Tinv_u(E_v)$ and 
   $E_u=\Tinv_{u^\flat}(E_j)=\Tinv_u(\Tinv_j^{-1}(E_j))=\Tinv_u(-F_jK_j)\,$. Thus
   $E_wE_u=\Tinv_u(E_v(-F_jK_j))=q^{-\symbrack \nu {\alpha_j}}\Tinv_u((-F_jK_j)E_v)=q^{-\symbrack  {\alpha_j} \nu}E_uE_w\,$. Now, let $s=\Weylpres(u)$ and $t=\Weylpres(u^\flat)$ so that
   $s=t\cdot s_j\,$. Then $s(\alpha_j)=t(s_j(\alpha_j))=-t(\alpha_j)=-\wordroot(u)=-\alpha\,$. Moreover, we have $\beta=\wordroot(u\cdot v)=s(\wordroot(v))=s(\nu)\,$. $\Weyl$-invariance of the form now implies $\symbrack  {\alpha_j} \nu= - \symbrack \alpha {\beta}\,$, from which we infer the assertion.
\end{proof}

The most basic case is given by $v=w_k$\,, when $\len w =\len u +1$ already implies the requirement $k\neq j\,$ of the lemma. It encompasses, for example, 2.3.(d3,d4) in \cite{Lu90a} and several skew-commutation relations in Section~5.2 of \cite{Lu90a}. The same references also contain strict commutation relations $E_\alpha E_\beta=E_\beta E_\alpha$ in special situations when $\alpha+\beta\not\in\proots\,$ and, hence, $\symbrack{\alpha}{\beta}=0$. We will see, however, in the proof of Proposition~\ref{prop:PBW1} below that for non-simply laced Lie type \LT{B} we may have weakly orthogonal pairs of roots $\alpha,\beta\in\proots$, meaning $\alpha+\beta\in\proots$ but $\symbrack{\alpha}{\beta}=0$, for which the associated generators do not commute.

We next reformulate the commutation relations for divided and ordinary powers of simple root generators derived in (c), (i), and (a1) from Section~5.2 in \cite{lu90b} into our conventions and a more compact formalism. 
As in Section~\ref{subsec:rank2coxsys}, assume $-A_{ji}=d_i=1$ and
$-A_{ij}=d_j=\edgenum\in\{1,2,3\}\,$ depending on Lie type. The translation utilizes the dictionary of generators from \eqref{eq:LusztigDictionary} as well as additional notation as follows.

 To begin with, define for a root lattice vector $\vec n=M\alpha_i+N\alpha_j\in\mathbb Z^{\sroots}$ the set
\begin{equation}
    S_\edgenum(\vec n)=\left\{\psi \in \nnN^{\proots}: \vec\psi=M\alpha_i+N\alpha_j\right\}\,,
\end{equation}
where $\vec\psi=\sum_\alpha\psi(\alpha)\alpha\in\mathbb Z^{\sroots}$ as in  Section~\ref{subsec:Emonomials}.
Note that the ordering of generators from right to left is the one determined by $\longtwoword{ij}$
as discussed in Section~\ref{subsec:rank2coxsys}, opposite to the one displayed in \eqref{eq:rank2rootorder}
for $\longtwoword{ji}$\,. As before, set $\beta_s=\wordroot(\longtwoword{ij}[1,s])$ so that $\beta_1=\alpha_i$\,,
$\beta_2=s_i(\alpha_j)=\edgenum\alpha_i+\alpha_j$ and so forth. 

For a given $\psi\in \nnN^{\proots}$
define $k_s=\psi(\beta_s)$. So, explicitly, for $\edgenum=1$ we have $k_1=\psi(\alpha_i)$\,, $k_2=\psi(\alpha_i+\alpha_j)$\,, and $k_3=\psi(\alpha_j)$\,, and for $\edgenum=2$ we have $k_1=\psi(\alpha_i)$\,, $k_2=\psi(2\alpha_i+\alpha_j)$\,, $k_3=\psi(\alpha_i+\alpha_j)$ and $k_4=\psi(\alpha_j)$\,. Moreover, if $\edgenum=3$\,, we have $k_1=\psi(\alpha_i)$\,, $k_2=\psi(3\alpha_i+\alpha_j)$\,, $k_3=\psi(2\alpha_i+\alpha_j)$\,, $k_4=\psi(3\alpha_i+2\alpha_j)$\,, $k_5=\psi(\alpha_i+\alpha_j)$\,, and $k_6=\psi(\alpha_j)\,$. The explicit forms of the associated root lattice vectors are thus
\begin{align}\label{eq:deffvecexp}
    \vec\psi&=
    \begin{cases}
    \;\; (k_1+k_2)\alpha_i+ (k_2+k_3)\alpha_j& \;\edgenum=1\vspace*{2mm}
    \\
    \;\; (k_1+2k_2+k_3)\alpha_i+ (k_2+k_3+k_4)\alpha_j& \;\edgenum=2\vspace*{2mm}
    \\
    \;\;  (k_1+3k_2+2k_3+3k_4+k_5)\alpha_i+ (k_2+k_3+2k_4+k_5+k_6)\alpha_j& \;\edgenum=3\,.
    \end{cases}\,
\end{align}
With these conventions we also define a quadratic function $f_\edgenum:\nnN^{\proots}\rightarrow\nnN\,$ as 
\begin{align}\label{eq:deffexp}
    f_\edgenum(\psi)&=
    \begin{cases}
    \;\; k_1k_3\,+\,k_2 & \edgenum=1\vspace*{2mm}
    \\
    \;\; 2k_4(k_1+k_2)+k_3k_1\,+\,2k_2+2k_3 & \edgenum=2\vspace*{2mm}
    \\
    \begin{array}{l}
          3k_6(k_1+2k_2+k_3+k_4)\,+\,k_5(2k_1+3k_2+k_3)\,+\,3k_4(k_1+k_2)\,+\,k_3k_1\\
    \,\;\;+\,3k_2+4k_3+6k_4+3k_5  
    \end{array}& \edgenum=3\;. 
    \end{cases}\,
\end{align}

Suppose $\hat\chi_i,\hat\chi_j\in \nnN^{\proots}$ are the indicator functions for the simple roots $\alpha_i$ and $\alpha_j$ respectively. That is, $(k_1,\ldots,k_m)$ is equal to $(1,0,\ldots,0)$ for the former and 
$(0,\ldots,0,1)$ for the latter. The variation of $f_\edgenum$ with respect to the simple root exponents can then be expressed as 
\begin{equation}\label{eq:fshiftform}
\begin{aligned} 
f_\edgenum(\psi+s\hat\chi_j)\,&= \, f_\edgenum(\psi)\,+\,s\edgenum\left({\psi(\alpha_j)+\mathscr u_j(\vec\psi)}\right)&\\
\mbox{and}\rule{0pt}{8pt}\quad \qquad
f_\edgenum(\psi+s\hat\chi_i)\,&= \, f_\edgenum(\psi)\,+\,s\left({\psi(\alpha_i)+\mathscr u_i(\vec\psi)}\right)\,. &
\end{aligned}   
\end{equation} 
Here $\mathscr u_i, \mathscr u_j\in (\mathbb Z^{\sroots})^*$ are defined for $\vec \psi=M\alpha_i+N\alpha_j$ as
\begin{align}
    \mathscr u_j(\vec\psi)&=M-N=
\begin{cases}
    k_1-k_3 &  \edgenum=1\vspace*{2mm}\\
    k_1+k_2-k_4 & \edgenum=2\vspace*{2mm}\\
    k_1+2k_2+k_3+k_4-k_6&  \edgenum=3
\end{cases}
\label{eq:ulf-j}
\\[2mm]
    \mathscr u_i(\vec\psi)&=\edgenum N-M=
\begin{cases}
    -k_1+k_3 &  \edgenum=1\vspace*{2mm}\\
    -k_1+k_3+2k_4&  \edgenum=2\vspace*{2mm}\\
    k_1+k_3+3k_4+2k_5+3k_6& \edgenum=3\,.
\end{cases}
\label{eq:ulf-i}
\end{align}
Using the basis elements from \eqref{eq:PBWEgenDivDef}, the mentioned equations in \cite{lu90b} can thus be written in the following form, which applies to all Lie types.
\begin{equation}\label{eq:LusztigCommsSumm-div} 
    E_i^{(M)}E_j^{(N)}=\sum_{\psi\in S_\edgenum(\vec n)}q^{f_\edgenum(\psi)} \Edivbaseopp{\longtwoword{ij}}{\psi}
\end{equation}
For convenience, we provide the explicit forms of the basis elements 
\begin{align}\label{eq:Eexp}
    \Edivbaseopp{\longtwoword{ij}}{\psi}&=
    \begin{cases}
    \;\; E_j^{(k_3)}E_{(ij)}^{(k_2)}E_i^{(k_1)} & \edgenum=1\vspace*{3mm}
    \\
    \;\; E_j^{(k_4)}E_{(iji)}^{(k_3)}E_{(ij)}^{(k_2)}E_i^{(k_1)} & \edgenum=2\vspace*{3mm}
    \\
    \;\; E_j^{(k_6)}E_{(ijiji)}^{(k_5)}E_{(ijij)}^{(k_4)}E_{(iji)}^{(k_3)}E_{(ij)}^{(k_2)}E_i^{(k_1)}& \edgenum=3\;. 
    \end{cases}\,
\end{align}

Multiplying both sides of \eqref{eq:LusztigCommsSumm-div} by respective quantum factorials, we obtain the  commutation relation of ordinary powers,
 \begin{equation}\label{eq:LusztigCommsSumm-ord} 
    E_i^{M}E_j^{N}=\sum_{\psi\in S_\edgenum(\vec n)}q^{f_\edgenum(\psi)} \lBrack\psi\rBrack_{\edgenum}\Ebaseopp{\longtwoword{ij}}{\psi}\,,
\end{equation}
where all coefficients are in $\mathbb Z[q,q^{-1}]\,$. Explicitly, these are expressed in terms of quantum factorials and quantum multinomial coefficients as follows.
\begin{align}\label{eq:defpsimulitnom}
    \lBrack\psi\rBrack_{\edgenum}&=
    \begin{cases}
    \;\; [k_2]!\qbin{k_1+k_2}{k_2}{}\qbin{k_2+k_3}{k_2}{}\ & \edgenum=1\vspace*{4.5mm}
    \\
    \;\; [2k_2]![k_3]!_j\qbin{k_1+2k_2+k_3}{k_1\,,\,2k_2\,,\,k_3}{}
\qbin{k_2+k_3+k_4}{k_2\,,\,k_3\,,\,k_4}{j} & \edgenum=2\vspace*{4.5mm}
    \\
    \begin{array}{l}
          [3k_2]![k_3]![k_3]!_j[3k_4]![k_4]!_j[k_5]!_j
\qbin{2k_3}{k_3}{}
\qbin{2k_4}{k_4}{j}\!\!\cdot\\
    \,\;\;\cdot\!\qbin{k_1+3k_2+2k_3+3k_4+k_5}{k_1\,,\,3k_2\,,\,2k_3\,,\,3k_4\,,\,k_5}{} 
\qbin{k_2+k_3+2k_4+k_5+k_6}{k_2\,,\,k_3\,,\,2k_4\,,\,k_5\,,\,k_6}{j}  
    \end{array}& \edgenum=3
    \end{cases}\,
\end{align}

As usual, the $j$ subscripts at the quantum numbers and quantum multinomial coefficients indicate that $q_j=q^\edgenum$ is used instead of $q_i=q$. The resummation formulae in the next lemma are an important
tool for relating bases in later sections.
\begin{lem} \label{lem:EEEresum}
Suppose $\,N,M\in\nnN\,$ and $\,\vec n=M\alpha_i+N\alpha_j\,$. Then, for an indeterminate $z$,  
    \begin{align}
    \sum_{k=0}^Nz^kq_j^{k(\mathscr u_j(\vec n)-1)}E_j^{(k)}E_i^{(M)}E_j^{(N-k)}
       &\,=\,\sum_{\phi\in S_\edgenum(\vec n)}\,q^{f_\edgenum(\phi)}
       \prod_{t=1}^{\phi(\alpha_j)}\bigl(q_j^{-2 t}z+1\bigr)
       \Ebaseopp{\longtwoword{ij}}{(\phi)}\label{eq:EEEresum-j}\\
     \sum_{k=0}^Mz^kq_i^{k(\mathscr u_i(\vec n)-1)}E_i^{(M-k)}E_j^{(N)}E_i^{(k)}
     &\,=\,\sum_{\phi\in S_\edgenum(\vec n)}\,q^{f_\edgenum(\phi)}
     \prod_{t=1}^{\phi(\alpha_i)}\bigl(q_i^{-2 t}z+1\bigr)
     \Ebaseopp{\longtwoword{ij}}{(\phi)}\,.\label{eq:EEEresum-i}
\end{align}
\end{lem}
\begin{proof} We show here only \eqref{eq:EEEresum-j} 
since the proof for \eqref{eq:EEEresum-i} is analogous. Set $w=zq_j^{\mathscr u_j(\vec n)-1}=zq^{\edgenum(M-N-1)}$. The  left side of \eqref{eq:EEEresum-j} can be evaluated as 
\begin{align*}
\sum_{k=0}^Nw^kE_j^{(k)}E_i^{(M)}E_j^{(N-k)}
=\sum_{\substack{0\leq k\leq N\\ \psi\in S_\edgenum(\vec n-k\alpha_j) }}w^kq^{f_\edgenum(\psi)}E_j^{(k)}\Ebaseopp{\longtwoword{ij}}{(\psi)}
=\sum_{\substack{0\leq k\leq N\\ \psi\in S_\edgenum(\vec n-k\alpha_j) }}w^kq^{f_\edgenum(\psi)}\qbinsmall{\psi(\alpha_j)+k}{k}_j\Ebaseopp{\longtwoword{ij}}{(\psi+k\hat\chi_j)}\\ 
=\sum_{\substack{0\leq k\leq \phi(\alpha_j)\\ \phi\in S_\edgenum(\vec n) }}w^kq^{f_\edgenum(\phi-k\hat\chi_j)}\qbinsmall{\phi(\alpha_j)}{k}_j\Ebaseopp{\longtwoword{ij}}{(\phi)}
=\sum_{\phi\in S_\edgenum(\vec n)}q^{f_\edgenum(\phi)}\left(
\sum_{k=0}^{\phi(\alpha_j)}w^kq^{-k\edgenum(\phi(\alpha_j)+M-N)}\qbinsmall{\phi(\alpha_j)}{k}_j\right)\Ebaseopp{\longtwoword{ij}}{(\phi)}\,.
\end{align*}
Here we use \eqref{eq:LusztigCommsSumm-div} in the first equality with $\vec n-k\alpha_j=M\alpha_i+(N-k)\alpha_j\,$. The second step is immediate from the forms \eqref{eq:deffexp} and combining quantum factorials. In the third equality, we substitute $\phi=\psi+k\hat\chi_j$\,, summing instead over $\phi\in S_\edgenum(\vec n)$ with $\phi(\alpha_j)\geq k$\,. The last step invokes 
\eqref{eq:fshiftform} and rearranges summations. After resubstituting $z$ in the last expression, we  apply Gauss's formula from \eqref{eq:q-binom-form} with $a=\phi(\alpha_j)$ to arrive at the desired expression.
\end{proof}

Specializing \eqref{eq:EEEresum-j} to $z=-q_j^2$ and \eqref{eq:EEEresum-i} to $z=-q_i^2$ yields the commutation relations
in Section~5.5 of \cite{lu90b}. In this case, the product terms will be zero unless $\phi(\alpha_j)=k_m=0$ or $\phi(\alpha_i)=k_1=0$\,, respectively. Assuming further $\mathscr u_j(\vec n)=M-N=0$ for \eqref{eq:EEEresum-j}, it is immediate from \eqref{eq:ulf-j} that $k_1=\ldots=k_{m-2}=0\,$. Thus, the only exponent function $\phi'$ that occurs in the summation
is the one that is $N$ on $\wordroot(a_{m-1})=\alpha_i+\alpha_j$ and zero on all other roots, where $a_{j}$ is as in 
\eqref{eq:def:abwords}. Observing also that
$f_\edgenum(\phi')=\edgenum N\,$, the right side of \eqref{eq:EEEresum-j} thus reduces to $q^{\edgenum N}E_{a_{m-1}}^{(N)}\,$. 

Analogously, for \eqref{eq:EEEresum-i} the constraint $\mathscr u_i(\vec n)=\edgenum N-M=0$ implies $k_3=\ldots=k_m=0$ so that
the only remaining exponent function is the one which is $N$ on $\wordroot(a_2)=\edgenum \alpha_i+\alpha_j$ and zero on others. The right side of \eqref{eq:EEEresum-j} thus becomes $q^{\edgenum N}E_{a_{2}}^{(N)}\,$. The resulting identities are summarized in the left column of the corollary below.  
\begin{cor}\label{cor:EdivpwSimpExp} The following identities hold:
    \begin{equation}\label{eq:EdivpwSimpExp-a}
    E_{a_{m-1}}^{(N)}=\sum_{k=0}^N (-1)^{N-k}q_j^{-k} E_j^{(N-k)}E_i^{(N)}E_j^{(k)}\,,
\qquad \quad 
E_{(ji)}^{(N)}=\sum_{k=0}^N (-1)^{k}q_j^{-(N-k)} E_j^{(N-k)}E_i^{(N)}E_j^{(k)}\,,
\end{equation}
\begin{equation}\label{eq:EdivpwSimpExp-b}
    E_{(ij)}^{(N)}=\sum_{k=0}^{\edgenum N} (-1)^{\edgenum N-k}q^{-k}E_i^{(k)}E_j^{(N)}E_i^{(\edgenum N-k)}\,,
    \qquad\quad 
    E_{b_{m-1}}^{(N)}=\sum_{k=0}^{\edgenum N} (-1)^{k}q^{-(\edgenum N-k)}E_i^{(k)}E_j^{(N)}E_i^{(\edgenum N-k)}\,.
\end{equation}
\end{cor}
The identities in the right column are obtained by applying $\mho$ and using \eqref{eq:Eabcompl} as well as the relations
$a_{m-1}\wcomplrel b_2=w_jw_i\,$ and $b_{m-1}\wcomplrel a_2=w_iw_j\,$ as in Lemma~\ref{lm:rk2complw}. 

As already indicated, \eqref{eq:EdivpwSimpExp-a} corresponds to the first relation in Section~5.5 of \cite{lu90b}. 
For the change in conventions, note that $\ELu_{ij}=\ELu_{(ji)}=E_{a_{m-1}}$ by \eqref{eq:LusztigDictionary}. 
Similarly, the right identity of \eqref{eq:EdivpwSimpExp-b}  
is the translation of the second expression in \cite[Section 5.5]{lu90b}.
For this, the generators $\ELu_{ij}$\,, $\ELu_{iij}$\,, and $\ELu_{iiij}$ listed in \cite{lu90b}
for $\edgenum=1$, $2$, and $3$, respectively, coincide in all cases with $\ELu_{b_{m-1}}=E_{(ij)}$ via 
\eqref{eq:LusztigDictionary}.

We thus have expansions into products of generators for simple roots of divided power generators $E_w^{(N)}$  for all words, with exception of the words $a_3\wcomplrel b_4$ and $a_4\wcomplrel b_3$ in the $\LT{G}_2$-case. Yet, unlike
\eqref{eq:LusztigCommsSumm-ord}, equations \eqref{eq:EdivpwSimpExp-a} and \eqref{eq:EdivpwSimpExp-b} cannot be written as expansions of ordinary powers over $\mathbb Z[q,q^{-1}]\,$.

\subsection{Numberings, Boxes, and Refining Convex Orderings} \label{subsec:orderings}
The integral PBW bases derived in \cite{Lu90a,lu90b} for the divided power algebra $\LUq^+$ require orderings on $\proots$ that are compatible with a specific pre-order $\prLleq\,$ on $\proots\,$, defined in Section~4.3 of \cite{lu90b}. As we adapt the proofs there to the situation of regular powers, we review this pre-order next and establish various relations to convex orderings and maximal words. 

Lusztig's construction of $\prLleq\,$ starts by imposing a {\em good} numbering on the set $\sroots\,$ of simple roots. For a given indexing $\,\sroots=\{\alpha_1,\ldots,\alpha_n\}\,$, write $\proots[1,j]$ for the respective subroot system generated by  $\{\alpha_1,\ldots,\alpha_j\}$\,. Assuming that each $\proots[1,j]$ is connected, set $L_j=\proots[1,j]-\{\rho_j\}$\,, where $\rho_j$ is the highest root in $\proots[1,j]\,$. The numbering is called good if for any root $\beta\in L_j$ the coefficient of $\alpha_j$ is at most 1. 

The numberings in appendix of \cite{Bou02} are good with the exception of types \LT{B} and \LT{F}. For the latter two, the reverse numberings given by $i\leftrightharpoons l+1-i$ are good. For type A any numbering is good, for \LT{C} also the reverse of \cite{Bou02} is good, and for type \LT{D} any numbering coinciding with the one in \cite{Bou02} in the highest index as well as the reverse of the one in \cite{Bou02} are good. 

For a given good numbering, any positive root $\beta$ can be uniquely written as $\beta=b_j\alpha_j+b_{j-1}\alpha_{j-1}+\ldots+b_1\alpha_1$ with $b_j>0\,$. Lusztig introduces maps $g:\proots\rightarrow \{1,\ldots,n\}$ and $c:\proots\rightarrow \{1,2\}\,$ given by $g(\beta)=j$ and $c(\beta)=c_j\,$. It is clear from the good numbering construction that $c(\beta)=2$ implies that $\beta=\rho_j$ in $\proots[1,j]\,$. 
Lusztig further defines $h':\proots\rightarrow\mathbb Q$ by $h'(\beta)=c(\beta)^{-1}\height(\beta)\,$, where $\height(\beta)=\sum_ib_i$ is the height of the root with respect to $\sroots\,$. 

The pre-order is now defined by writing $\,\alpha\prLleq \beta\,$ if $\,g(\alpha)\geq g(\beta)\,$ and $\,h'(\alpha)\leq h'(\beta)\,$. The equivalence classes $R_{m,k}=\{\beta\in\proots: g(\beta)=m,\,h'(\beta)=k\}$ are called {\em boxes} in \cite{lu90b}. We record the following two observations.

\begin{lemma}\label{lm:nosumbox}
    Suppose $\alpha,\beta\in\proots$ ($\alpha\neq\beta$) belong to the same box. Then none of $\pm\alpha\pm\beta$ are roots. 
\end{lemma}

\begin{proof}
    Given $m=g(\alpha)=g(\beta)$ we have $\alpha,\beta\in M=\proots[1,m]\,$, where $M$ itself is a root system with highest root $\rho$, so that it suffices to show $\pm\alpha\pm\beta\not\in M\,$. Assume first $c(\alpha)=c(\beta)=1$\,. Then $\height(\alpha-\beta)=\height(\alpha)-\height(\beta)=h'(\alpha)-h'(\beta)=0=\height(\beta-\alpha)$\,, which is not possible. For $\gamma=\alpha+\beta$ we have 
    $c(\gamma)=2$ and $\height(\gamma)=2k\,$. If $\gamma\in M$ the former implies that $\gamma=\rho$ and that $M$ is of type $\LT{E}_8$\,, $\LT{F}_4$\,, or $\LT{G}_2$\,. In all three cases $\height(\rho)$ is odd (29, 11, or 5), leading to a contradiction. Now, $c(\alpha)=2$ implies $\alpha=\rho$ and hence $k=h'(\alpha)\in\{14.5, 5.5, 2.5\}$ so that there is no other root in the same box.
\end{proof}

An immediate corollary is that any two roots in the same box are (strongly) orthogonal. In \cite{Lu90a,lu90b} Lusztig goes on to consider total orders refining $\prLleq\,$. With good numberings as above, we notice the following.

\begin{lemma}
    Any total order on $\proots$ refining $\prLleq\,$ is convex and thus of the form $\leqwt{w}$ for some $w\in\wordsetmax\,$.
\end{lemma}
\begin{proof} Let $\alpha,\beta,\gamma=\alpha+\beta\in\proots$ and suppose $\alpha\in R_{m,k}$ and $\beta\in R_{m',k'}$\,.
Set $t=c(\alpha)c(\gamma)^{-1}$ and $s=c(\beta)c(\gamma)^{-1}$ so that $h'(\gamma)=t\cdot k+s\cdot k'\,$,
using that $\height(\gamma)=\height(\alpha)+\height(\beta)$\,. Write also $\mu\not\sim \nu$ to indicate two roots are not in the same box. 

Suppose now $<$ is a total order refining $\prLleq$ and assume $\alpha < \beta$\,. Since $\gamma\in\proots$, Lemma~\ref{lm:nosumbox} implies $\alpha\not\sim \beta$\,. Hence, either we have $m>m'$ or we have $m=m'$ and $k<k'\,$. The former case implies $g(\gamma)=m$\,, which means $\gamma\prLleq\beta$ and $\beta\not\sim\gamma$ and, hence, $\gamma<\beta\,$. Moreover, with $c(\gamma)=c(\alpha)\,$, we find $t=1$ so that 
$h'(\gamma)=k+s\cdot k'>k\,$. Hence, $\gamma\prLgeq\alpha$ and $\gamma\not\sim \alpha$ so that $\alpha<\gamma\,$. 

In the second case we have $g(\alpha)=g(\beta)=g(\gamma)$ and $c(\gamma)=c(\alpha)+c(\beta)$ so that $s=(1-t)$\,. Thus $h'(\gamma)=t\cdot k+(1-t)\cdot k'$ and $k<h'(\gamma)<k'\,$ since $0<t<1$\,. We, therefore, obtain $\alpha<\gamma<\beta$ in this case as well and, hence, convexity of the total ordering. As shown in \cite{Pa94} this implies that $<$ needs to be of the form $\leqwt{w}$ as desired.
\end{proof}

Note, for Lie types \LT{A}\,, \LT{B}\,, and \LT{G} each box contains no more than one element. Thus, there is only one $w\in\wordsetmax$ for which the total order $\leqwt{w}$ refines $\prLleq\,$. Specifically, the chosen good numbering for $\LT{A}_n$ is realized by the maximal word $w_{\LT{A}_n}=a_1a_2\ldots a_n$\,, where $a_k=w_nw_{n-1}\ldots w_k\,$. Also, for $\LT{B}_n$ we have $w_{\LT{B}_n}=b_n\ldots b_1$\,, where $b_k=w_kw_{k-1}\ldots w_2w_1w_2\ldots w_{k-1}w_k$\,, and for $\LT{G}_2$ that $w_{\LT{G}_2}=(w_2w_1)^3\,$. 

For $\LT{C}_n$ with the Bourbaki numbering one longest word choice refining $\prLleq\,$
is given as a product $w=u\cdot v$. Here $u=\alpha_0\beta_0\alpha_1\beta_1\ldots \beta_1\alpha_1\beta_0\alpha_0$ is a palindromic expression in $\alpha_j=w_nw_{n-2}\ldots w_{n-2j}$ and $\beta_j=w_{n-1}w_{n-3}\ldots w_{n-2j-1}$ centered around the longest $\alpha_k$ or $\beta_k$\,, and $v$ is a  longest word for $A_{n-1}$\,. The reduced word 
$$
w_4w_3w_2w_3w_1w_4w_2w_3w_4w_2w_3w_1w_2w_3w_4w_3w_2w_3w_1w_2w_3w_1w_2w_1
$$ is compatible with the pre-order for $\LT{F}_4$ using the reverse numbering of that in \cite{Bou02}. The words for the simply laced \LT{D} and \LT{E} Lie types are similarly worked out with additional choices within boxes. 

\subsection{PBW Bases for Box Orderings}\label{subsec:PBWspec} The sole focus of this section is the adaption of the proofs in \cite{Lu90a,lu90b} to yields a PBW theorem for regular power 
generators instead of divided power generators. The main 
change from the original proof is that the involved
commutation and recursion relations for generators are valid only over
$\Zqqn{\dpone}$ rather than $\Zqq$ as in \cite{Lu90a,lu90b}. 
The second difference is merely concerning conventions, 
namely the alternate choice of generators $E_w$ defined
via the automorphisms $\Tinv_i$ in Section~\ref{subsec:genwords} rather than the $\ELu_w$ used in \cite{Lu90a,lu90b}. Thanks to the relation in (\ref{eq:mhomonomax}) 
it is not difficult to move between 
these two settings using the $\Kinvaut$ involution, 
while also replacing the root ordering for $z$ by 
the one for $z^\winvchar$. 
The form (\ref{eq:EFConjOrder}) for
the action of the $\Qinvaut$ involutions 
implies, further,  
that bases with generators multiplied in opposing directions have the same spans. 

Although the inductive argument for a PBW theorem itself is unchanged from \cite{Lu90a,lu90b}, we outline  the main steps
again in the proof below. We also adapt the formulae to fit our conventions, even as the mentioned
automorphisms and symmetries would afford us the flexibility to retain those in \cite{Lu90a,lu90b}. We freely use the notation for monomial sets and spans introduced at the end of Section~\ref{subsec:Emonomials}, with a focus on the  minimal
ground ring $\Zqqn{\dpone}\,$ and the maximal exponent set $\mathsf S=\expsetup{\longweyl}\,$.

\begin{prop}\label{prop:PBW1}
Suppose $w\in\wordsetmax$ is such that $\leqwt{w}$  refines $\prLleq\,$. Then $\Uq^+$ is free a $\Zqqn{\dpone}$-module with basis given by either the generators in $\basisp{w}$ or those in $\basispopp{w}$ as  defined in \eqref{eq:PBW_p_basis}.
The analogous statement holds for $\Uq^-$ and the sets $\basisn{w}$ and $\basisnopp{w}\,$. 
\end{prop}

\begin{proof} Recall that, by Proposition~\ref{prop:wordgencont}, the  $\Zqqn{\dpone}$-module $\bspanp{w}$, generated by elements in $\basisp{w}$, is contained in $\Uq^+$\,. 
Since the elements in $\basisp{w}$ form a basis of $\Uqre{\mathbb Q}$ over the field $\mathbb Q(q)$\,, see Corollary 40.2.2 in \cite{lu}, we also know that they form a basis of $\bspanp{w}$ as a $\Zqqn{\dpone}$-module and that  $\Uq^+/\bspanp{w}$ is at most a torsion module. 
Since $1\in \bspanp{w}$ is a cyclic vector for  the left or right regular action of $\Uq^+$ on itself,
either $\Uq\cdot \bspanp{w}\subseteq \bspanp{w}$ or $\bspanp{w}\cdot \Uq \subseteq \bspanp{w}$
would imply $\Uq= \bspanp{w}\,$.
Thus, it remains to show that  $\bspanp{w}\cdot E_i\subseteq \bspanp{w}\,$ for all $i\,$ or $E_i\cdot \bspanp{w}\subseteq \bspanp{w}\,$ for all $i\,$. 

The inductive proof in the simply laced case can be obtained by modifying the arguments in Section~2 of \cite{Lu90a}. Specifically, instead of the spaces in 2.10 we consider analogous spaces defined with regular rather than divided powers. For example, $X_{m,k}$ is replaced by $X_{m,k}^*\,$, defined as the $\Zqqn{\dpone}$-algebra of $\Uq^+$ generated by $E_{w[\alpha]}$ (as in \eqref{eq:root2word}) with $\alpha\in R_{m,k}$ as well as $E_j$ with $j<m\,$. Similarly, $Y^*_{m,k}$ is the free $\Zqqn{\dpone}$-submodule generated by $\basisp{u}$ where $u\leqRB w$ is the maximal word with $\descroots u \cap R^+_{m,k}=\emptyset\,$.  
Statement (b) of Lemma~2.11 in 
\cite{Lu90a}
is then changed to $Y^*_{m,k}\cdot X_{m,k}^* \subseteq 
Y^*_{m',k'}\cdot X_{m',k'}^*\,$ if $R^+_{m',k'}$ is the next smaller box in the $\prLleq$ 
order after $R^+_{m,k}$ with integer $k'\,$. 

In the adjustment of the proof, we first redefine the generators $\xi_t$ as ordinary powers $E^{N_t}_{\alpha_t}$. The notions of distinguished and non-distinguished generators as well as 
bad or good pairs are adapted to ordinary powers in the same manner.
The commutation argument invokes equations (d2)-(d5) in $\S2.3$ of \cite{Lu90a} that are imposed on the algebra $V^+$ defined there. The isomorphism $V^+\cong \LUq^+$ in Section~4.7 implies that these relations also hold in $\Uqre{\mathbb Q}^+$ for the divided powers of the $\ELu_\alpha$ generators. Analogous equations can be readily derived from these for ordinary powers of the $E_\alpha$ generators, as defined in \eqref{eq:EFwordDef} and \eqref{eq:Genbyroot}. For example, for $\alpha$ and $\alpha'$ fulfilling (e1) or (e2), equation (d3) in \cite{Lu90a} specializes for $N=M=1$ to $\ELu_{\alpha'}\ELu_{\alpha}=q\ELu_{\alpha+\alpha'}+q\ELu_{\alpha}\ELu_{\alpha'}$\,. Applying $\Kinvaut$ with $\Kinvaut(\ELu_\alpha)=E_\alpha$ we obtain $E_{\alpha'}E_{\alpha}=(-1)E_{\alpha+\alpha'}+q^{-1}E_{\alpha}E_{\alpha'}$\,. Analogs for (d2), (d4), and (d5) are also obtained by application of $\Kinvaut$ and multiplication of factorials. The desired analog of (d3) is then obtained from these relations by the same induction, 
\begin{equation}\nonumber
E_{\alpha'}^NE_{\alpha}^M=\sum_{j=0}^{\min(N,M)}(-1)^jq^{-(N-j)(M-j)}[j]!\qbinsmall{N}{j}\qbinsmall{M}{j} E_{\alpha}^{M-j}E_{\alpha+\alpha'}^jE_{\alpha'}^{N-j}\,.
\end{equation}
Clearly, all coefficients are in $\Zqq=\Zqqn{2}\,$. The same process is applied to (b) in Proposition 2.7 to express $E_{\alpha_i}^NE_{\beta_0}^M$ as
$\Zqq$-combinations of respective ordinary power generators with reversed cyclic order of indices, using also that $\psi_0=1$ and $\qbinsmall{N}{s}\qbinsmall{M}{s}\in\Zqq\,$. The other equations in Proposition 2.7 and 2.8 are modified, as before, by application of $\Kinvaut$ and multiplication of factorials.  
 
 All other parts of the proof of Lemma~2.11 in \cite{Lu90a} can be adapted verbatim. As in Proposition~2.12(b) the inclusion can be iterated to infer $Y^*_{n,1}\cdot X_{n,1}^* \subseteq 
Y^*_{1,1}\cdot X_{1,1}^*\,$, which is exactly the inclusion $\Uq^+\subseteq \mathbf{B}^+\,$.

As noted in \cite{lu90b} for the case of divided powers, the proof for the non-simply laced cases follows the same type of inductive argument. Commutation relations of all divided power generators in the rank 2 case $\LT{B}_2=\LT{C}_2$ are explicitly worked out in Sections~5.2 and 5.3 of \cite{lu90b}. As with the simply laced cases, it is straightforward to derive relations for $\edgenum=2$ and $\edgenum=3$ in Section~5.2 for the opposite orders of generators and, hence, after application of $\Kinvaut$, respective relations for our convention of generators. From these, analogs of the equations in Section~5.3 can be readily derived for our generators by rescaling arguments. The coefficient will thus differ only by units in $\Zqqn{\dpone}\,$. 

Multiplication of the analogs of equations (d) through (i) in the same section of \cite{lu90b} by respective quantum factorials readily yields commutation relations for ordinary powers with coefficients in $\Zqqn{3}\,$. The  variant of (g) for our generators is, for example, is as follows, 
\begin{equation}\label{eq:LuR2_g}
    E_{\gamma}^kE_{\delta}^{k'}=\sum_{s=0}^{\min(k,k') }\frac 1 {[2]^s}q^{2s(k+k')-3s^2}(q-q^{-1})^s[s]_2!\qbinsmall{k}{s}_2\qbinsmall{k'}{s}_2 E_{\delta}^{k'-s}E_{\beta}^{2s}E_{\gamma}^{k-s}\,.
\end{equation}
Here $\delta=\alpha_2$\,, $\gamma=\alpha_2+2\alpha_1$\,, and $\beta=\alpha_1+\alpha_2$ with $\alpha_2$ the long root. The coefficients are all in $\Zqqn{3}$ but not in $\Zqq$ due to the additional $\frac 1 {[2]}$ factors. The respective variant for (h) is 
\begin{equation}\label{eq:LuR2_h}
E_{1}^kE_{\tau}^{k'}=\sum_{s=0}^{\min(k,k') }(-1)^s {[2]^s}q^{s(k+k')-\frac 12 s(7s+3)}[s]!\qbinsmall{k}{s}\qbinsmall{k'}{s} E_{\tau}^{k'-s}E_{\tau+\alpha_1}^{s}E_{1}^{k-s}\,,
\end{equation}
where $\tau=\alpha_2+\alpha_1$\,. Using respective multinomial coefficients and $\nu=\alpha_2$ we also reexpress equation (i) in \cite{lu90b} as
\begin{equation}\label{eq:LuR2_i}
E_{1}^kE_{\nu}^{k'}=\sum_{\substack{r,s,t,u\geq 0\\r+s+t=k'\\ s+2t+u=k}}(-1)^s  q^{-2r(u+t)+us}[s]!_2[2t]!\qbinsmall{k}{s, 2t, u}\qbinsmall{k'}{r,s,t}_2 
E_{\nu}^{r}E_{\nu+\alpha_1}^sE_{\nu+2\alpha_1}^{t}E_{1}^{u}\,.
\end{equation}

It is immediate that the set of relations allows reordering of any monomial over $\Zqqn{3}$ into the order on the positive roots of $\LT{B}_2$ given by $w_2w_1w_2w_1$\,. Explicit relations for $\LT{G}_2$ are given in Section~5.4 of \cite{lu90b} as well. Although more complicated, all of them can be multiplied by respective factorials and transformed into relations for our generators in the same manner, except that powers of $\frac 1 {[3]}$ occur in addition to those of  $\frac 1 {[2]}\,$.  

The higher rank \LT{BCF} cases follow as in the simply laced case from commutation relations obtained by application of Artin automorphisms to the rank 2 relations as well as those for the respective \LT{A} type subalgebras. In the above ordering for the \LT{B} case Lemma~\ref{lm:scommgens} readily implies that generators from consecutive boxes $R^+_{m,k}$ and $R^+_{m,k+1}$ skew-commute. Consider next required
commutation relations between elements $E_\tau$ with $\tau\in R^+_{m,k}$ and generators $E_j$ for $j<m\,$. From 
Proposition~\ref{prop:wordgencont}
it follows that $\Tinv_{d_m}(E_j)=E_j$ for $d_m=b_n\ldots b_{m+1}\in\wordset\,$, reducing the arguments to the case $m=n\,$ and, hence, $E_\tau=E_v$ with $v\leqRB b_m\,$.

In all cases in which $\tau+\alpha_j\not\in \proots$ it is then readily checked that $E_\tau$ and $E_j$ and, hence, all their powers commute up to signs. 
For $\tau=\alpha_m+\ldots+\alpha_{j+1}$ when $j>1$ we invoke the known commutation relations for $\LT{A}_{n-1}\subset\LT{B}_n$\,. Applying $\Tinv_{w_m\ldots w_3}$ to 
\eqref{eq:LuR2_h} we see that the relation holds more generally for $\nu=\alpha_m+\ldots+\alpha_{2}$\,, yielding the needed relation between $E_\nu$ and $E_1\,$. Similarly, applying $\Tinv_{w_m\ldots w_3}$ to \eqref{eq:LuR2_h} extends the relation to the case
$\tau=\alpha_m+\ldots+\alpha_{1}$\,. Note that in this case $\symbrack{\tau}{\alpha_1}=0$, even as $\tau+\alpha_1\in\proots\,$.

Finally, for $j>1$ and 
$\tau=\alpha_m+\ldots+\alpha_j+2\alpha_{j-1}+\ldots+2\alpha_{1}$ we find $E_jE_\tau=q^{-2}E_\tau E_j-E_{\tau+\alpha_j}$\,, $E_{\tau+\alpha_j}E_j=q^{-2}E_jE_{\tau+\alpha_j}$\,, and 
$E_{\tau+\alpha_j}E_\tau=q^{2}E_\tau E_{\tau+\alpha_j}\,$, which implies $E_\tau$ and $E_j$ fulfill the $\LT{A}_2$ Serre relations with $q$ replaced by $q^2$\,. This, again, implies the usual $\LT{A}$ type commutation relations, only with $q$ replaced by $q^2$\,.  

The treatment for the \LT{C} and \LT{F} cases uses the same arguments and analogous computations. Indeed, choosing the numbering opposite to that of \cite{Bou02} is also good for \LT{C}\,, leading to essentially the same relations as in the \LT{B} case. The $\LT{F}_4$ can be treated explicitly as an extension of $\LT{B}_3$ or $\LT{C}_3\,$. This completes the proof for the basis $\basisp{w}$ for one direction of multiplication.

It follows now immediately from \eqref{eq:EFConjOrder} that $\basispopp{w}$ provides a basis of $\Uq^+$ as well, since $\Qinvaut$ maps elements from $\basisp{w}$ to those of $\basispopp{w}$ up to unit factors. Similarly, \eqref{eq:EFCartOrder} implies the respective statements for $\Uq^-\,$ with basis elements mapped to each other by $\Cartaninv\,$. 
\end{proof} 

As noted in the introductory remarks, we could have also followed the Lusztig's conventions with 
the respectively renormalized commutation relations in the $\ELu_w$\,. This would have yielded 
$\EbaseLU{w}{\psi}$ as a $\Zqqn{\dpone}$-basis for $\Uq^+$ for $w$ compatible with the box ordering chosen in \cite{Lu90a,lu90b}. 
Since $\Uq^+$ is invariant under $\Kinvaut$\,, we have from (\ref{eq:mhomonomax}) that  
$\Ebaseopp{w^\winvchar}{\psi}$ is also a basis. So, after application $\Qinvaut$\,, we find that
$\Ebase{w^\winvchar}{\psi}$ is a basis as well. For the general PBW theorem,  
it suffices to establish a PBW statement for only one ordering, so that 
the replacement of $w$ by $w^\winvchar$ does not change the proof of Propositions~\ref{prop:PBW2}
or Theorem~\ref{thm:mainPBW} below.

\begin{cor}\label{cor:PBW1F} For $w\in\wordsetmax$ as in Proposition~\ref{prop:PBW1} the sets $\basis{w}{\pm}$ and $\basisopp{w}{\pm}$ are bases of $\Uq^-$ as a free $\Zqqn{\dpone}$-module. 
\end{cor}

We conclude with another useful consequence of Proposition~\ref{prop:PBW1}, which allows us to identify subalgebras generated by a subset of generators as quantum algebras themselves. 

\begin{cor}\label{cor:Uqtwoembed}
    For any pair of distinct indices $i,j\in\{1,\ldots,n\}$, let $\Uqrest{+}{i,j}$ be the rank 2 quantum algebra with respectively restricted Cartan data. Then the canonical algebra homomorphism  $\Uqrest{+}{i,j}\rightarrow \Uq^+$ is well-defined and injective, and its image is the subalgebra of $\Uq^+$ generated by $E_i$ and $E_j\,$.
\end{cor}

\begin{proof} Surjectivity on the subalgebra generated by $E_i$ and $E_j$ is obvious. Similarly, it is clear the map is well-defined, since the relations of $\Uqrest{+}{i,j}$ are also fulfilled in $\Uq^+\,$. For $\LT{G}_2$ (or any rank 2 case)  the statement is also obvious. To show injectivity in the $\LT{A}_1\times \LT{A}_1$\,, $\LT{A}_2$\,, and $\LT{B}_2=\LT{C}_2$ cases, we note that either numbering for each of these is a good numbering and can thus be chosen to inherit the order of the numbering imposed by $\Uq^+\,$. 

Since also the height and lead coefficients of positive roots are the same, the pre-order 
$\prLleq\,$ of $\proots$ for $\Uq^+$ restricts to the pre-order of positive root system $\prootstwo{i,j}$ of $\Uqrest{+}{i,j}\,$. This implies, by Proposition~\ref{prop:PBW1}, that we have a basis for $\Uqrest{+}{i,j}$ that is mapped to basis elements of $\Uq^+\,$, implying injectivity. Indeed, 
$\Uqrest{+}{i,j}$ is a direct summand of $\Uq^+$ as a $\Zqqn{\dpone}$-module. 
\end{proof}

\medskip

We conclude our discussion of the existence of bases with a brief outline of other approaches. 
As mentioned previously, Lusztig proves in \cite{Lu90a,lu90b,lu} that the monomials 
$\Edivbase{w}{\psi}$ from (\ref{eq:PBWEgenDivDef}) define a PBW basis of his divided
power algebra $\LUq^+\,$ over $\Zqq=\mathbb Z[q,q^{-1}]\,$. It can be correspondingly extended 
to a basis of all of 
$\LUq$ over $\Zqq$ at the expense of introducing additional generators in $\LUq^0\,$ given by rational expressions in the $K_i$ and $q$.

In \cite[Sec 12]{dcp93a}, De Concini and Procesi define the algebra $\DCPU$ to be the smallest subalgebra of 
$\Uq\otimes \mathbb k(q)$ over 
$\mathbb k[q,q^{-1}]$ that contains the 
singularized generators $\Esing_i$ and $\Fsing_i$  from (\ref{eq:def:EgenSing}) and 
is stable under the Artin group action generated by the $\Tinv_i\,$. They offer a proof 
that the monomials $\Ebasesing{w}{\psi}K^\mu\Fbasesing{w}{\chi}$ from (\ref{eq:PBWEgenSingDef}) form a 
$\Zqq$-basis of $\DCPU$. Their argument is based on recursions formulae over $\mathbb C$ in \cite{LS91} and proves that recursion coefficients are indeed in $\mathbb k[q,q^{-1}]$ for an appropriate $\mathbb k\,$. 

The fact that the $\Ebaseoppsing{w}{\psi}$ are a basis of $\DCPU^+$ may also be inferred from
the Lusztig-Tanisaki pairing discussed in Section~\ref{subsec:LuszTaniPair} and the PBW theorem for
$\LUq^+$ in
\cite{Lu90a,lu90b,lu}. Comparing the 
coefficients in (\ref{eq:PBWEgenDivDef}), (\ref{eq:PBWEgenSingDef}), and (\ref{eq:RcoeffGen})
the pairing in (\ref{eq:LTpairs}) can be rewritten as 
\begin{equation}\label{eq:RenomLTpairs}
    \langle\,\cdot\,,\,\cdot\,\rangle_z:\LUq^-\times \DCPU^+\rightarrow \Zqq\qquad
    \mbox{with}\quad 
\langle \Fdivbaseopp{w}{\psi},\Ebaseoppsing{z}{\phi}\rangle\,=\,\delta_{\psi,\phi}\cdot (-1)^{n(\psi)} \cdot q^{-\sum_\alpha d_\alpha\binom{\psi(\alpha)}{2}}
\end{equation}

Since the pairings for $\psi=\phi$ are units in $\Zqq$ the span of the $\Ebaseoppsing{z}{\phi}$
is the exact $\Zqq$-dual to $\LUq^-\,$. The product and the  $q$-braided coproduct $\barDelta$
defined by Lusztig \cite{lu} on $\LUq^-$ are then dual to respective structures on this span. 
This now identifies the latter with $\DCPU^+$ so that the basis result for $\DCPU^\geqzero$ is essentially the dual equivalent of Lusztig's basis result on $\LUq^\leqzero\,$. 

However, as algebras over $\Zqq$ or $\Zqqn{\dpone}\,$, $\DCPU^+$ and $\Uq^+$ exhibit rather different behaviors. This becomes especially apparent in the $q\mapsto 1$ specialization, realized by the obvious
ring maps $\Zqq\rightarrow\mathbb Z$ or $\Zqqn{\dpone}\rightarrow\mathbb Z[{\edgenum}^{-1}]$\,. 
For example, for Lie type  $\LT{A}_n$  one easily sees from Proposition~\ref{prop:AComms}
that the specialization $\DCPU^+\otimes\mathbb Z$ is the {\em commutative} polynomial algebra 
in the $E_{i,j}\,$, while the specialization $\Uq^+\otimes\mathbb Z$ is the universal envelope $U(\mathfrak n)$\,, where $\mathfrak n$
is the Lie algebra of the strictly upper triangular matrices. 

Another variant is to consider renormalized generators 
$\mathring E_w=[d_w]E_w=\Tinv_{w^\flat}(\mathring E_i)\,$ with $i=\tau(w)$ and 
$\mathring E_i=[d_i]E_i\,$. The respective renormalization of the formulae (\ref{eq:LuR2_g}, \ref{eq:LuR2_h}, \ref{eq:LuR2_i}) then yields expansion coefficients entirely in $\Zqq\,$. We, thus,
expect that there is a Hopf   algebra $\mathring U_q^\geqzero$ over $\Zqq\,$ 
with a PBW basis over this ring, even for the non-simply laced Lie types.

\subsection{Spanning Sets and Their Order Dependence}\label{subsec:spanorder} For any subset of exponent 
functions $\psi\in \expsetup s$ for some $s\in\Weyl$ 
one can consider the submodule spanned by the respective  monomials 
$\Ebase{w}{\psi}$, where $s=\Weylpres(w)\,$. Numerous subalgebras and ideals we will consider later can be defined as such spans. 
In this section we extract sufficient conditions on a set of exponent functions which imply
that the respective free submodule is independent of the choice of a convex root ordering  or, equivalently, 
dependent only on the Weyl element $s=\Weylpres(w)\,$.

Suppose $s\in\Weyl\,$. We say $(a,b)$ with $a,b\in\Weyl$ is a {\em relator pair} for $s$ if  
\begin{equation}\label{eq:relatorRootSplit}
    s=a\cdot\longtwoweyl{ij}\cdot b\qquad\mbox{with}\quad \len{s}=\len{a}+\len{\longtwoweyl{ij}}+\len{b}\,
\end{equation}
for the longest element $\longtwoweyl{ij}$ is some rank 2 subsystem $\rootstwo{ij}\,$ of $\roots\,$. Observe now that
Corollary~\ref{cor:sumdescroots} provides us with a bijection of inversion sets given as 
\begin{equation}\label{eq:relatorSetSplit}
    \relrootsplit{a,b}=\id\sqcup a\sqcup a\cdot\longtwoweyl{ij}\,:\,\descroots{a}\sqcup \descroots{\longtwoweyl{ij}}\sqcup \descroots{b}\,\longrightarrow\,\descroots{s}\,. 
\end{equation}
Correspondingly, as in (\ref{eq:ExpsetCorr}), we obtain the bijection of exponent sets
\begin{equation}\label{eq:relatorExpSplit}
 \relrootsplit{a,b}^*\,:\,\expsetup{s}\,\rightarrow\,\expsetup{a}\times \expsetup{\longtwoweyl{ij}}\times \expsetup{b}\,.
\end{equation}
Given a subset $\mathsf S\subseteq\expsetup{s}\,$ we abbreviate its image in the Cartesian product as $\mathsf S_{(a,b)}=\relrootsplit{a,b}^*(\mathsf S)\,$.

The following discussion assumes a general ground ring $\Zgen$ with an extension of scalars function $f:\Zqqn{\dpone}\rightarrow\Zgen\,$. Note, that
Proposition~\ref{prop:PBW1} implies, that both $\basis{\longtwoword{ij}}{+}$ and 
 $\basis{\longtwoword{ji}}{+}$  are bases of $\Uqrerest{+}{i,j}{\Zgen}=\Uqrest{+}{i,j}\otimes_f\Zgen\,$. 
 Recall from Corollary~\ref{cor:Uqtwoembed}, that the latter may be viewed as the subalgebra generated by $E_i$ and $E_j$\,. 
 We, thus, can define coefficients $\KinvautTwoCoeff{\varphi}{\varphi'}{ij}\in\Zgen$ for any $\varphi,\varphi'\in\expsetup{\longtwoweyl{ij}}$ by 
\begin{equation}\label{eq:MhoTwoCoeff}
    \Kinvaut(\Ebase{\longtwoword{ij}}{\varphi})=\Ebase{\longtwoword{ji}}{\varphi}=\sum_{\varphi'}\KinvautTwoCoeff{\varphi}{\varphi'}{ij}\cdot\Ebase{\longtwoword{ij}}{\varphi'}\,. 
\end{equation}
From these we define the following subsets of  $\expsetup{\longtwoweyl{ij}}$ for any $\varphi\in \expsetup{\longtwoweyl{ij}}$
and $\mathsf Q\subseteq \expsetup{\longtwoweyl{ij}}\,$:
\begin{equation}\label{eq:MhoTwoSets}
    \KinvautExpSet{ij}{\varphi}=
    \bigl\{\varphi'\in \expsetup{\longtwoweyl{ij}}\,:\,\KinvautTwoCoeff{\varphi}{\varphi'}{ij}\neq 0\bigr\}
    \qquad\mbox{and}\qquad
    \KinvautExpSet{ij}{\mathsf Q}=\bigcup_{\varphi\in\mathsf Q} \KinvautExpSet{ij}{\varphi}\,. 
\end{equation}
Note that, since $\Kinvaut$ is an involution, 
we have $\KinvautExpSet{ij}{\mathsf Q}=\KinvautExpSet{ji}{\mathsf Q}$\,. Commutation of $\Kinvaut$ with
$\Qinvaut$ and (\ref{eq:EFConjOrder}) imply that the sets are the same if the oppositely ordered basis elements are used instead. Commutation with $\Cartaninv$ shows the same for the respective bases of $\Uq^-\,$. The sets may, however, depend on the choice of 
the ground ring as $f$ may map some non-zero coefficients in $\Zqqn{\dpone}$ to 
zero in $\Zgen\,$. 

For given $s\in\Weyl$ we say an exponent set $\mathsf S\subseteq\expsetup{s}$ is {\em relator stable} if for 
any relator pair $(a,b)$ for $s$ we have the implication 
\begin{equation}\label{eq:RelStabCond}
    (\chi,\varphi,\xi)\in\mathsf S_{(a,b)}\quad\mbox{and}\quad \varphi'\in \Kinvaut_{\{ij\}}(\varphi)
    \qquad\Longrightarrow\qquad (\chi,\varphi',\xi)\in\mathsf S_{(a,b)}\,. 
\end{equation} 
Defining $U^{\mathsf S,a,b}_{\chi,\xi}\subseteq \Uqrerest{+}{i,j}{\Zgen}$ as the   $\Zgen$-span 
of the set $\{\Ebase{\longtwoword{ij}}{\varphi}:\,(\chi,\varphi,\xi)\in \mathsf S_{(a,b)}\}$\,, it is clear from (\ref{eq:MhoTwoCoeff}) that this
condition may be equivalently restated as the requirement that $\Kinvaut(U^{\mathsf S,a,b}_{\chi,\xi})\subseteq U^{\mathsf S,a,b}_{\chi,\xi}$ for all relator pairs $(a,b)$ and $(\chi,\xi)\in\expsetup{a}\times \expsetup{b}$\,.

Note that we only need to consider cases when $m_{ij}=\len{\longtwoweyl{ij}}\in \{3,4,6\}$ since the condition  is void for $m_{ij}=2$. 
Indeed, in the latter case $\rootstwo{ij}=\{\alpha_i,\alpha_j\}$ and $E_i$ and $E_j$ commute so that 
$\Ebase{\longtwoword{ij}}{\varphi}=\Ebase{\longtwoword{ji}}{\varphi}$. Thus 
$\KinvautExpSet{ij}{\varphi}=\{\varphi\}$\,, rendering (\ref{eq:RelStabCond}) trivially true.

Equipped with this terminology, we state the following criterion for word independence of spanning sets, which may be viewed as a reformulation of Matsumoto's equivalence as a property of exponent sets. 
 
\begin{prop}\label{prop:KinvRelCrit}
Let $w,\tilde w\in\wordset$ be two reduced words representing the same element 
$s=\Weylpres(w)=\Weylpres(\tilde w)$ in $\Weyl$. Assume further that $\mathsf S\subseteq\expsetup{s}$ is a relator stable exponent set for given ground ring $\Zgen\,$. Then we have for the $\Zgen$-spans $\bspan{w}{\mathsf S,+}=\bspan{\tilde w}{\mathsf S,+}\,$.  Further, if the elements of $\basis{w}{\mathsf S,+}$ are linearly independent, then so are those of $\basis{\tilde w}{\mathsf S,+}$\,.

\noindent The analogous assertions hold for the  $\basis{w}{\mathsf S,-}$ and $\basisopp{w}{\mathsf S,\pm}$ and their spans. 
\end{prop}

\begin{proof}
By Matsumoto's Theorem (see Corollary~\ref{cor:matsumoto}) the exists a sequence of reduced words, starting with $w$ and ending with $\tilde w$\,, each of which is obtained from the previous one by replacing a substring of the form  
$\longtwoword{ij}$ by $\longtwoword{ji}\,$, vice versa. By transitivity of the assertions, it suffices to consider the case in which $\tilde w$ is obtained from $w$ by one such substitution. That is, $w=u\cdot \longtwoword{ij}\cdot v$ and $\tilde w=u\cdot \longtwoword{ji}\cdot v$ as reduced decompositions. Thus, with $a=\Weylpres(u)$ and $b=\Weylpres(v)$\,, we find $s=a\cdot \longtwoweyl{ij}\cdot b$ and $\len s=\len a+m_{ij}+\len b\,$. In particular, $(a,b)$ is a relator pair for $s$\,. 

Suppose now $\psi\in\mathsf S$ so that $\Ebase{w}{\psi}$ is a spanning element for
$\bspan{w}{\mathsf S,+}$ and let $(\chi,\varphi,\xi)=\relrootsplit{a,b}^*(\psi)$\,. 
Applying (\ref{eq:PBWgenRecur}) twice, we find 
$\Ebase{w}{\psi}=\Ebase{u}{\chi}\cdot\Tinv_a(\Ebase{\longtwoword{ij}}{\varphi})\cdot \Tinv_{a\cdot\longtwoweyl{ij}}(\Ebase{v}{\xi})\,$. Expanding 
$\Ebase{\longtwoword{ij}}{\varphi}$ as in (\ref{eq:MhoTwoCoeff}) into a sum of 
$\Ebase{\longtwoword{ji}}{\varphi'}\,$, each term can then be rewritten as some 
$\Ebase{\tilde w}{\psi}$ using again (\ref{eq:PBWgenRecur}). More precisely, writing 
$\psi(\varphi')=(\relrootsplit{a,b}^*)^{-1}(\chi,\varphi',\xi)$ for fixed $\chi$ and $\xi$ we find
\begin{equation}\label{eq:EpsiMTtransf}
\Ebase{w}{\psi}\,=\,\Ebase{w}{\psi(\varphi)}\,=\,\sum_{\varphi'}\KinvautTwoCoeff{\varphi}{\varphi'}{ji}\cdot\Ebase{\tilde w}{\psi(\varphi')}\,.
\end{equation}
Condition (\ref{eq:RelStabCond}) finally ensures that all basis elements on the right side are 
in $\basis{\tilde w}{\mathsf S,+}$ whenever $\KinvautTwoCoeff{\varphi}{\varphi'}{ji}\neq 0$\,.
Thus $\Ebase{w}{\psi}\in \bspan{\tilde w}{\mathsf S,+}$ and, therefore, 
$\bspan{w}{\mathsf S,+}\subseteq\bspan{\tilde w}{\mathsf S,+}\,$.
The opposite inclusion is implied by $\mho$-symmetry. 

Now, for fixed $\chi$ and $\xi$\,, (\ref{eq:EpsiMTtransf}) yields a transformation
of the spanning elements $B^{(\chi,\xi)}_w=\{\Ebase{w}{\psi(\varphi)}:\,\psi(\varphi)\in\mathsf S\}$
to those of $B^{(\chi,\xi)}_{\tilde w}$ via the restricted involution expressed by $\KinvautTwoCoeff{\varphi}{\varphi'}{ji}$\,. Thus,  $B^{(\chi,\xi)}_w$ is a basis if and only if  
$B^{(\chi,\xi)}_{\tilde w}$ is a basis, proving the second statement.  
Similarly, applications of $\Qinvaut$
and $\Cartaninv$ yield the assertions for the other types of spans. 
\end{proof}

As an example, we may consider the following basic type of exponent sets.
Given two sets $\NexpSet{1}, \NexpSet{\edgenum}\subseteq\nnN$ of non-negative integers and $s\in\Weyl$,
define 
\begin{equation}\label{eq:VexpSetEx}
    \NexpExpSet{s}\,=\,\bigl\{\psi\in\expsetup{s}:\,\psi(\alpha)\in\NexpSet{d_\alpha}\,,\,\forall \alpha\in\descroots{s}\bigr\}
\qquad\mbox{and}\quad
\NexpExpSetij{ij}=\NexpExpSet{\longtwoweyl{ij}}\,. 
\end{equation}
Suppose that $s=r\cdot t$ with $\len{s}=\len{r}+\len{t}$\,. Observe that the bijection 
$\expsetup{s}\rightarrow\expsetup{r}\times\expsetup{t}$ from (\ref{eq:ExpsetCorr})
restricts to a bijection $\NexpExpSet{s}\rightarrow\NexpExpSet{r}\times \NexpExpSet{t}$\,.
Conversely, it is not hard to show that any family of exponent sets $\{\NexpExpSet{s}:s\in\Weyl\}$ with this {\em splitting property} has to be
of the form in (\ref{eq:VexpSetEx}). For convenience, we also 
introduce  notation for the complements
\begin{equation}\label{eq:VexpSetExComp}
    \NexpExpCompSet{s}=\expsetup{s}\setminus\NexpExpSet{s}=\bigl\{\psi\in\expsetup{s}:\,\exists\alpha\in\descroots{s}
    \;\mbox{ with }\;
    \psi(\alpha)\not\in\NexpSet{d_\alpha}\bigr\}\,, 
\end{equation}
and, analogously, $\NexpExpCompSetij{ij}=\NexpExpCompSet{\longtwoweyl{ij}}\,$.

\begin{cor}\label{cor:BaseOrdSplit}\ \vspace*{-2mm}

\noindent Let $w,\tilde w\in\wordset$ with $s=\Weylpres(w)=\Weylpres(\tilde w)\,$ and $\{\NexpExpSet{s}\}$ be a family of exponent sets as
in (\ref{eq:VexpSetEx}). \vspace*{-1mm}
\begin{enumerate}[label=\roman*), leftmargin=11mm,]
    \item\label{item:BaseOrdSplit:Vs} 
    Suppose $\Kinvaut$ maps the $\Zgen$-span $\bspan{\longtwoword{ij}}{\NexpExpSetij{ij},+}$
to itself for all $1\leq i,j\leq n\,$. 
Then we have for the $\Zgen$-spans $\bspan{w}{\NexpExpSet{s},+}=\bspan{\tilde w}{\NexpExpSet{s},+}$\,.
Moreover, $\basis{w}{\NexpExpSet{s},+}\,$ is linearly independent if and only if $\basisopp{w}{\NexpExpSet{s},+}$\, is.\vspace*{3mm}
    \item\label{item:BaseOrdSplit:VsComp}  
    Suppose $\Kinvaut$ maps the $\Zgen$-span $\bspan{\longtwoword{ij}}{\NexpExpCompSetij{ij},+}$
to itself for all $1\leq i,j\leq n\,$. 
Then we have for the $\Zgen$-spans $\bspan{w}{\NexpExpCompSet{s},+}=\bspan{\tilde w}{\NexpExpCompSet{s},+}$\,.
Moreover,  $\basis{w}{\NexpExpCompSet{s},+}$\, is linearly independent if and only if $\basisopp{w}{\NexpExpCompSet{s},+}$\, is. 
\end{enumerate} 
The analogous statements hold for opposite directions and respective bases of $\Uqre{\Zgen}^-\,$.   
\end{cor} 

\begin{proof} Suppose $(a,b)$ is a relator pair for $s$\,. Then the aforementioned splitting property implies that 
$\NexpExpSet{s}_{(a,b)}=\NexpExpSet{a}\times\NexpExpSetij{ij}\times \NexpExpSet{b}\,$. So, in the case of {\em \ref{item:BaseOrdSplit:Vs} }, (\ref{eq:RelStabCond})
reduces to the localized condition that $\varphi\in\NexpExpSetij{ij}$ and $\KinvautTwoCoeff{\varphi}{\varphi'}{ji}\neq 0$
implies $\varphi'\in\NexpExpSetij{ij}$, which is equivalent to
the condition that $\Kinvaut$ preserves $\bspan{\longtwoword{ij}}{\NexpExpSetij{ij},+}$\,.
Alternatively, this can be inferred from the remark following in (\ref{eq:RelStabCond}), observing that $U^{\mathsf S,a,b}_{\chi,\xi}=\bspan{\longtwoword{ij}}{\NexpExpSetij{ij},+}$ whenever 
$(\chi,\xi)\in\NexpExpSet{a}\times\NexpExpSet{b}$ and zero otherwise. 

For the complementary statement in {\em\ref{item:BaseOrdSplit:VsComp}}, note that 
$$
\NexpExpCompSet{s}_{(a,b)}\,=\,
   \NexpExpCompSet{a}\times\expsetup{\longtwoweyl{ij}}\times \expsetup{b}
   \,\cup\,
   \expsetup{a}\times\NexpExpCompSet{\longtwoweyl{ij}}\times \expsetup{b}
   \,\cup\,
   \expsetup{a} \times\expsetup{\longtwoweyl{ij}}\times \NexpExpCompSet{b}\,.
$$
If $(\chi,\varphi,\xi)$ is either in the first or third set, the condition (\ref{eq:RelStabCond}) is 
trivially true since $\Kinvaut$ preserves 
$U^{\mathsf S,a,b}_{\chi,\xi}= \Uqrerest{+}{i,j}{\Zgen}$\,. If $(\chi,\varphi,\xi)$ is in the second set
the condition (\ref{eq:RelStabCond}) is, again, equivalent to $\Kinvaut$-stability of
$U^{\mathsf S,a,b}_{\chi,\xi}=\bspan{\longtwoword{ij}}{\NexpExpCompSetij{ij},+}$\,.
\end{proof}
 
Observe, finally, that Corollary~\ref{cor:BaseOrdSplit} can be slightly strengthened by noting that only pairs 
$(i,j)$ need to be considered  for which there is a reduced decomposition $s=a\cdot\longtwoweyl{ij}\cdot b$\,.

\subsection{General Ordering Bases and Module Restrictions}\label{subsec:PBWmain}
The aim of this section is to extend the special bases from Proposition~\ref{prop:PBW1} to bases for any convex ordering and the entire $\Uq$ algebra. We will, further, consider extensions of scalars to ground rings $\Zgen$ or $\Zgenv$ as in 
\eqref{eq:ringchange}
for given homomorphisms $f:\Zqqn{\dpone}\rightarrow\Zgen$ and   $f^{\,\,\vaccent}:\Zqqvn{\dpone}\rightarrow\Zgenv\,$. Since the associated  
functor $-\otimes_{\Zqqn{\dpone}}\Zgen$   maps free modules to free modules, any $\Zqqn{\dpone}$-basis of $\Uq^{\utypechar}$ is also a $\Zgen$-basis for $\Uqre{\Zgen}^{\utypechar}\,$, with
$\utypechar\in\{+,-,0,\geqzero,\leqzero\}\,$. Similar considerations apply to $\Uqre{\Zgenv}\,$.

We begin by adapting Rosso's arguments for the tensor decomposition of the full quantum algebra over a field to our situation for rings $\Zgen$ or $\Zgenv\,$. 
Lemma~2 in \cite{Ro88} proves that $\Uqre{\Zgen}^0$ is isomorphic to $\Zgen[\{K^{\pm}_i\}]\cong \Zgen[\mathbb Z^n]$ and admits a  $\Zgen$-basis is given by 
\begin{equation}\label{eq:defKbasis}
    \basischar^{\,0}\,=\,\left\{{K^\nu:\,\nu\in\mathbb Z^{\sroots}}\right\}
\end{equation} 
for $\Zgen$ a field. The argument there does not require the latter assumption  so that the 
statement holds for any integral domain $\Zgen$ as a ground ring.
We next consider  multiplication maps, such as the following.
\begin{equation}\label{eq:Umults}
    \begin{aligned}
    \Uqre{\Zgen}^0\otimes_\Zgen \Uqre{\Zgen}^{+}\; &\rightarrow \;\Uqre{\Zgen}^{\geqzero}\\
   \rule{0mm}{7mm} \Uqre{\Zgenv}^{+}\otimes_{\Zgenv} \Uqre{\Zgenv}^0\otimes_{\Zgenv} \Uqre{\Zgenv}^{-} \; &\rightarrow \;\Uqre{\Zgenv}
\end{aligned}
\end{equation}
Obvious additional maps are given by permuting the tensor factors as well as swapping the algebras  
$\Uq^+\leftrightharpoons\Uq^-$ and $\Uq^{\geqzero}\leftrightharpoons\Uq^{\leqzero}$. 

To show surjectivity of these maps, it suffices to show that their respective images are preserved under the left regular action of the respective target algebra, which is immediate from the commutation relations \eqref{eq:EFcomm} and \eqref{eq:EFcomm}. For example, applying $E_k$ to a general element in the image of the first type of map yields with  $E_k(K^\nu E_{i_1}\ldots E_{i_N})=q^{-\symbrack{\nu}{\alpha_k}}(K^\nu E_kE_{i_1}\ldots E_{i_N})$ again an element in the image. Similarly, for the second type of maps, 
$F_k(E_{i_1}\ldots E_{i_N}K^\nu F_{j_1}\ldots F_{j_M})$ can be expressed as a $\Zqqv$-combination of  summands of the form
$E_{i_1}\ldots E_{i_N}K^\nu F_kF_{j_1}\ldots F_{j_M}$ and 
$E_{i_1}\ldots \widehat{E_{i_s}} \ldots E_{i_N}K^\mu F_{j_1}\ldots F_{j_M}$\,, where $\mu\in \mathbb Z^{\sroots}$ and the hat denotes omission whenever $i_s=k\,$. 

The proofs of Proposition~1 and 2 in \cite{Ro88} for injectivity assume that the quantum algebra is defined over a field. As before, a straightforward inspection of the arguments there immediately shows that the only assumptions needed are  the existence of a free basis over the ground ring as well as a grading-preserving coalgebra structure. The former is assured by  Proposition~\ref{prop:PBW1} and Corollary~\ref{cor:PBW1F} above. The latter is clearly also defined  over the respective integral domains, and its grading properties are listed in \eqref{eq:gradscoprod} and \eqref{eq:xgradcoprod}. These observations are summarized next.

 \begin{lemma}[\cite{Ro88}]\label{lm:multiso}
    The multiplication maps from \eqref{eq:Umults} and their obvious variants are isomorphisms of free $\Zgen$ or $\Zgenv$ modules.  
 \end{lemma}

 We next combine results from Sections~\ref{subsec:spanorder} and \ref{subsec:PBWspec} to obtain PBW bases for the $\Uqre{\Zgen}^{\pm}$ algebras for general orderings.

\begin{prop}\label{prop:PBW2}
Suppose  $w\in\wordsetmax$ is any reduced word of maximal length. Then both $\basis{w}{+}$ and $\basisopp{w}{+}$
are $\Zgen$-bases for $\Uqre{\Zgen}^+$\,. Similarly,  $\basis{w}{-}$ and $\basisopp{w}{-}$
are $\Zgen$-bases for $\Uqre{\Zgen}^-$\,. 
\end{prop}

\begin{proof} Proposition~\ref{prop:PBW1} ensures that there is at least one reduced word 
$w^*\in \Weylpres^{-1}(\longweyl)$ for which $\basis{w^*}{+}$ and $\basisopp{w^*}{+}$ are $\Zqqn{\dpone}$-bases of $\Uq^+$ and hence also $\Zgen$-bases of $\Uqre{\Zgen}^+\,$. Suppose now $w\in\wordsetmax$ is another presentation of $\longweyl\,$.

The full exponent set $\mathsf S=\expsetup \longweyl =\nnN^{\proots}\,$ is clearly of the form (\ref{eq:VexpSetEx})
with $\NexpSet{d}=\nnN\,$. Moreover, as already noted in the proof of Corollary~\ref{cor:BaseOrdSplit}, we have for 
the full $\Zgen$-span 
$\bspan{\longtwoword{ij}}{+}\!\!=\bspan{\longtwoword{ji}}{+}\!\!=\Uqrerest{+}{i,j}{\Zgen}\,$.
Thus, the second condition of  Corollary~\ref{cor:BaseOrdSplit} is also fulfilled, which
implies that $\bspan{w^*}{+}\!\!=\bspan{w}{+}$. Thus $\basis{w}{+}$ also spans $\Uqre{\Zgen}^+\,$ and is linearly independent.
The second case for $\Uq^-$ follows readily from Corollary~\ref{cor:PBW1F} in the same manner.  
\end{proof}

Clearly, Lemma~\ref{lm:multiso} implies that, given bases for the tensor factors on the left sides of \eqref{eq:Umults}, their products as spanning sets yield bases for the quantum algebras on the right side. This, combined with Proposition~\ref{prop:PBW2} and the observation for $\Uq^0$ in \eqref{eq:defKbasis}, now yields the following general construction of PBW bases.

\begin{thm}\label{thm:mainPBW}
Let $w,u\in\wordsetmax$ be any reduced words of maximal length and $\Zgen$\,, $\Zgenv$ rings as in \eqref{eq:ringchange}.
Then any of the spanning sets $\,\basis{w}{+}\!\bcdot \basischar^{\,0}\,$, $\,\basischar^{\,0}\!\bcdot\basis{w}{+}\,$, 
$\,\basisopp{w}{+}\!\bcdot \basischar^{\,0}\,$, or $\,\basischar^{\,0}\bcdot\basisopp{w}{+}$ is a $\Zgen$-basis for $\Uqre{\Zgen}^{\geqzero}$\,. 

\noindent The analogous statement holds for $\Uqre{\Zgen}^{\leqzero}\,$. 
Moreover, the spanning set $\basis{w}{+}\!\bcdot \basischar^{\,0}\bcdot \basis{u}{-}$ and any reordering of the factor sets as well as any direction reversals (24 versions in total) is a $\Zgenv$-basis for $\Uqre{\Zgenv}$\,. 
\end{thm}

One elementary but useful application of this theorem arises  when subspaces constructed over fields need to be pulled back to submodules over the other rings. More concretely, suppose $\basischar$ is any $\Zgen$-basis as above for one of the algebras $\Uqre{\Zgen}^{\utypechar}\,$. Assume that $\Zgen$ is an integral domain and $\mathbb F$ its field of fractions so that we may view
$\Uqre{\Zgen}^{\utypechar}$ as a subalgebra of $\Uqre{\mathbb F}^{\utypechar}\,$. 

Suppose now $\mathcal I\subset\mathcal B$ is a subset of basis elements and denote $I=\langle \mathcal I\rangle$ its $\Zgen$-span in $\Uqre{\Zgen}^{\utypechar}\,$. Similarly, let $I_{\mathbb F}\cong I\otimes\mathbb F$ be the respective $\mathbb F$-span in  
$\Uqre{\mathbb F}^{\utypechar}\,$. Since $I$ is a direct summand, we then immediately have  
\begin{equation}\label{eq:restrImodule}
    I\,=\,I_{\mathbb F}\cap \Uqre{\Zgen}^{\utypechar}\,. 
\end{equation}

\section{Skew-Commutative Algebras at Roots of Unity}\label{sec:rootsofunity}

In this section, we turn to the specialization $q\mapsto\zeta$ of a quantum group to a primitive
$\kay$-th root of unity $\zeta$. The main focus of our discussion are the primitive power generators
$\Epw_\alpha=E_\alpha^{\ell_\alpha}\,$, where $\ell_\alpha$ is the order of 
$\zeta^{\symbrack{\alpha}{\alpha}}$\,. For odd $\kay$\,, the algebra $\Zsubalgchar^+_\bullet$ generated 
by these elements lies in the center of $\Uz$ 
and has been studied, for example, in \cite{dck90,dcp93a}.

Our treatment here includes, in addition,  {\em all} even $\kay>2\dpone$. 
In many of these cases,  the generators $\Epw_\alpha$
 commute only up to signs with other elements in $\Uz$\,. Depending on Lie type and
the congruence of $\kay$ modulo 8, the algebra $\Zsubalgchar^+_\bullet$ may be central, 
commutative but not central, or neither. Besides a classification of these cases, we also prove 
word independence of respective partial subalgebras $\Zsubalgchar^+_w$ (with one exception) 
and invariance under automorphisms. 
Various types of implied and induced ideals are discussed as well. 

\subsection{Quantum Numbers and Rings at Roots of Unity} \label{subsec:QnumsRingsRo1} 
Assume a root system with integers $d_i\in\{1,2,3\}$ as in Section~\ref{subsec:rootweyl} such that 
$1=\min_i\{d_i\}$ and $\maxd=\max_i\{d_i\}\,$. Fix an integer $\kay\in \mathbb N$ with
\begin{equation}\label{eq:kaycond}
    \kay\not\in\{1,2,\maxd, 2\maxd\}\,.
\end{equation} 
The ring $\mathbb Z[\zeta]$ of cyclotomic integers may be understood as the quotient of either $\mathbb Z[q]$ or $\mathbb Z[q,q^{-1}]$ by the principal ideal $J_\kay=(\Phi_\kay(q))$ of the $\kay$-th cyclotomic polynomial. Equivalently, we may view $\mathbb Z[\zeta]$ as the ring of integers of the cyclotomic field $\mathbb Q(\zeta)$, considered as a subfield of $\mathbb C$ for a given a choice of a primitive $\kay$-th root of unity $\zeta\in\mathbb C\,$. That is, the kernel of the map $\mathbb Z[q]\rightarrow\mathbb C: q\mapsto \zeta$ is precisely $J_\kay\,$ (see, for example, \cite{Wa82} for details). 
We use the same notation as in \eqref{eqn:quantumnumbers} for the images of the quantum numbers in $\mathbb Z[\zeta]$. 

Denote by $\elln=\kay/\gcd(\kay,2)\,$ the order of $\zeta^2\,$. For integers $d_i$ as above write $\zeta_i=\zeta^{d_i}\,$ and denote by $\kay_i=\kay/\gcd(\kay,d_i)\,$ the order of $\zeta_i\,$. Similarly, we set
$\ell_i=\elln/\gcd(\elln,d_i)=\kay/\gcd(\kay,2d_i)=\kay_i/\gcd(\kay_i,2)$ to be the order of 
$\zeta_i^2=\zeta^{2d_i}=\zeta^{\symbrack{\alpha_i}{\alpha_i}}\,$.
Note also that for short roots we have $\ell_i=\elln\,$ since, by definition, $d_i=1\,$.

Condition \eqref{eq:kaycond} is indeed equivalent to requiring $\zeta^2_i\neq 1$ for all 
$i\in\{1,\ldots, n\}$. This ensures that the factor $(\zeta_i-\zeta_i^{-1})^{-1}$ occurring in 
\eqref{eq:EFcomm} exists in $\mathbb Q(\zeta)$ so that the specialization $q\mapsto\zeta$ for quantum groups  is well-defined. Similarly, \eqref{eq:kaycond} implies that $m=\ell_i$ is precisely the smallest positive integer for which $[m]_i=0\,$.

The following additional powers of $\zeta$ will frequently appear in subsequent formulae. 
\begin{equation}\label{eq:defspeczetas}
    \Lsign_i=\zeta_i^{\elln_i}=\zeta^{d_i\elln_i}=\zeta^{\mathrm{lcm}(d_i,\ell)}
    \qquad \qquad  \LLsign_i=\Lsign_i^{\elln_i}=\zeta_i^{{\elln_i}^2}
    \qquad\qquad \Qsign_i=\zeta_i^{\binom{\elln_i}{2}}
\end{equation} 

We note that $\Lsign_i, \LLsign_i\in\{+1,-1\}\,$. Specifically, $\Lsign_i=-1$ if and only if $\kay_i$ is even, and 
$\LLsign_i=-1$ if and only if $\kay_i\equiv 2 \mod 4\,$. Moreover, $\Qsign_i\,$ is a primitive fourth root of unity ($\Qsign_i^2=-1$) if $\kay_i\equiv 0\mod 4$, we have $\Qsign_i=-1\,$ if $\kay_i\equiv 6\mod 8$, and $\Qsign_i=1\,$ in all other cases. These relations also imply the identity $\Qsign_i^2=(-1)^{\ell_i-1}\,$ in all cases.

As for generic $q$\,, if $d_i=1\,$, we will often drop the $i$ subscript  from the 
roots of unity appearing in \eqref{eq:defspeczetas} as well as the quantum numbers. 
In any of these, we may replace any subscript $i$ by a subscript $\alpha$ if $\alpha\in\proots$ is in the $\Weyl$-orbit of $\alpha_i\,$. So, for example, $\elln_\alpha=\elln_i$\,, $\kay_\alpha=\kay_i$\,, $\zeta_\alpha=\zeta_i$\,, and $\Lsign_\alpha=\Lsign_i$\,. Analogously, for a reduced word
$w\in\wordset$ we will write $\elln_w=\elln_i$\,, $\zeta_w=\zeta_i$\,, $\Lsign_w=\Lsign_i$\,, and so forth if $i=\tau(w)$ or, equivalently, $\wordroot(w)$ is in the  $\Weyl$-orbit of $\alpha_i\,$. 
 
The commutation relations in \eqref{eq:EFcomm} for the standard definition of quantum groups also contain denominators $(q_i-q_i^{-1})^{-1}\,$, which \textit{a priori} map only to $\mathbb Q(\zeta)$ but not $\mathbb Z[\zeta]$\,. A slightly sharper statement asserts that 
the image $(\zeta_i-\zeta_i^{-1})^{-1}$ is an element of the smaller ring $\mathbb Z[\zeta,\tfrac 1 {\elln_i}]\subset \mathbb Q(\zeta)\,$. This follows readily from the following formula, in which $r=\elln/\gcd(m,\ell)$ is the order of $\zeta^{2m}\,$:
\begin{equation}\label{eq:zetadenomexp}
    (\zeta^m-\zeta^{-m})^{-1}\,=\,-\frac 1 {r}\sum_{j=0}^{r-2}(r-j-1)\zeta^{m(2j+1)}\qquad\in\;\mathbb Z[\zeta,\tfrac 1 {r}] \,\subset \, \mathbb Q(\zeta)\,.
\end{equation}
For non-prime orders $\kay$, smaller denominators may be chosen as, for example, $(\zeta-\zeta^{-1})^{-1}\in \mathbb Z[\zeta,\tfrac 1 {3}]$ if $\elln=\kay=9\,$. The use of singularized generators as in
(\ref{eq:def:EgenSing}) would avoid these denominators. 

Aside from vanishing quantum numbers, the specialization also leads to additional units. For example, if $r$ and $\elln_i$ are coprime, then $[r]_i$ is a 
unit in $\mathbb Z[\zeta]\,$. More generally, for any $m\in\mathbb N$ and $d=\gcd(m,\elln_i)$ we have the factorization $[m]_i=[d]_i\cdot u$\,, where $u$ is a (cyclotomic) unit in $\mathbb Z[\zeta]$\,. 

From \eqref{eq:zetadenomexp} we also find that $ ([m]_i)^{-1}\in\mathbb Z[\zeta,\tfrac 1 s]$, where $s=\elln_i/d\in\mathbb Z\,$ is the order of $\zeta_i^{2m}=\zeta^{2d_im}\,$. 
The next identity expresses the last non-vanishing factorial in terms of the order $\elln_i$ and the element $(\zeta_i-\zeta_i^{-1})\,$.
\begin{equation}\label{eq:Dfactell}
    \Lfact_i\,=\,[\elln_i-1]_{i}!\,=\,\elln_i\cdot  \Qsign_i^{-1}\cdot (\zeta_i^{-1}-\zeta_i)^{-\elln_i+1}\;
\end{equation}

For an integer $n>1$, we introduce the following notation for specializations of the rings in \eqref{eq:defgenrings}, which will be used frequently in later sections. 
They are well-defined as subrings of $\mathbb Q(\zeta)\,$
if and only if $n\leq\elln\,$, which we will thus assume. 
\begin{equation}
    \begin{aligned}\label{eq:defzetarings}
    \Zz&=\mathbb Z[\zeta] & \qquad & \Zzn{n}=\mathbb Z\textstyle{\left[\zeta,\frac 1 {[n-1]!}\right]}\\
    \rule{0mm}{7mm}\Zzv&=\mathbb Z\textstyle{\left[{\zeta,\frac 1{\zeta-\zeta^{-1}}}\right]} && \Zzvn{n}=\mathbb Z\textstyle{\left[\zeta,\frac 1 {\zeta-\zeta^{-1}},\frac 1 {[n-1]!}\right]}
\end{aligned}
\end{equation}

The minimal rings $\Zqqn{\dpone}$ and $\Zqqvn{\dpone}$ over which $\Uq^{\pm}$ and $\Uq$ are defined specialize to $\Zzn{\dpone}$ and $\Zzvn{\dpone}$\,, respectively, if and only if $\elln\geq \dpone\,$, where $\dpone=\maxd+1\in\{2,3,4\}$ is as in \eqref{eq:Ddef} depending on Lie type. As a condition on $\kay$\,, the bound on $\ell$ is equivalent to requiring \eqref{eq:kaycond} and, additionally, $\kay\neq 4$ whenever $\maxd=3$\,. 
The latter exclusion refers to the $\LT{G}_2$ quantum group, which is well-defined for 
$\kay= 4$ but for which the braid automorphisms from Section~\ref{subsec:LATM} are not defined.

Previous remarks about cyclotomic units imply that $\Zz=\Zzn{n}$ and 
$\Zzv=\Zzvn{n}$  if $j$ and $\elln$ are coprime ($\gcd(\elln,j)=1$) 
for all $1\leq j\leq n$. This is, of course, the case if $\elln$ itself is a prime or if $n=1\,$. 
Consequently, we have 
$\Zzn{\dpone}=\Zz$ and $\Zzvn{\dpone}=\Zzv$ for types $\LT{ADE}$ and all $\elln>1$\,, for 
types $\LT{BCF}$ if $\elln$ is odd, and for type $\LT{G}_2$ if $\gcd(\elln,6)= 1\,$. 

As before, the relation $(\zeta_i-\zeta_i^{-1})=[d_i](\zeta-\zeta^{-1})$ implies that 
$(\zeta_i-\zeta_i^{-1})^{-1}\in \Zzvn{\dpone}\,$. The definition of an $R$-matrix in our setting will require all non-zero factorials to appear as denominators, thus requiring $n=\elln\,$.
It is clear that all rings listed in \eqref{eq:defzetarings} are subrings of  
$\Zzvn{\elln}\,$. In fact, equations \eqref{eq:zetadenomexp} and \eqref{eq:Dfactell} imply 
the equality of rings 
$$
\Zzvn{\elln}=\mathbb Z[\zeta,\frac 1 {\elln}]\;.
$$
 
The evaluation of quantum binomial coefficients at a root of unity also leads to vanishing patterns and factorizations of classical binomial coefficients that will be important later. In order to describe the former, we introduce functions $\twococ_i:\nnN\times\nnN\rightarrow \{0,1\}$ depending on $\elln_i$ as follows. For two integers $u, v\in \nnN\,$, write $u=a\cdot \elln_i+r$ and $v=b\cdot \elln_i+s$ with remainders $r,s\in\{0,\ldots,\elln_i-1\}\,$. Define then
\begin{align}
\twococ_i(u,v)=
   \left\lfloor{\frac{u+v}{\elln_i}}\right\rfloor-
\left\lfloor{\frac{u}{\elln_i}}\right\rfloor-
\left\lfloor{\frac{v}{\elln_i}}\right\rfloor
\,=\,
\twococ_i(r,s)
\,=\,
\begin{cases}
1 & \text{for } r+s\geq \elln_i\\
0 & \text{for } r+s< \elln_i\,.
\end{cases}\;
\end{align}
Here $\floor{x}$ is the floor function, which yields the greatest integer less than or equal to $x$. The expression clearly only depends on the classes $\overline u,\overline v\in \mathbb Z/\elln_i\mathbb Z$ and, as a map on these, is the standard 2-cocycle in $\mathrm Z^2(\mathbb Z/\elln_i\mathbb Z)$ generating $\mathrm H^2(\mathbb Z/\elln_i\mathbb Z)\,$. 

\begin{lem}\label{lem:qbin-zetaform}
With integers $u,v,r,s\in\nnN$ as above, the image of the quantum binomial coefficient in $\mathbb Z[\zeta]$ is given by the formula
\begin{equation}\label{eq:qbinform}
\qbin{u+v}{u}{i}\,=\,
\Lsign_i^{(as+br)}\cdot \LLsign_i^{ab}\cdot \qbin{r+s}{r}{i}\cdot {\binom{a+b}{a}}\;,
\end{equation}\vspace*{-1mm}

so that
\hspace*{28mm}
$\displaystyle
\qbin{u+v}{u}{i}\,=\,0\qquad\Longleftrightarrow\qquad \twococ_i(u,v)\neq 0\,. 
$
\end{lem}

Both the quantum binomial coefficient in $r$ and $s$ at $\zeta_i$ and the classical binomial coefficient in $a$ and $b$ are elements in $\mathbb Z[\zeta]\,$ and the powers of $\Lsign$ and $\LLsign$  are signs $\pm 1$.
The vanishing criterion follows from inspection of the quantum binomial term in $r$ and $s$.

Importantly, this implies that if $u+v$ is a multiple of $\elln_i$\,, then the expression is non-zero only if also both $u$ and $v$ are multiples of $\elln_i$ as well. In this case, $r=s=0$ and the quantum binomial is given by the respective classical binomial up to a sign. A main ingredient in the derivation of \eqref{eq:qbinform} is that the quantum factorial for $u$ in $\mathbb Z[q,q^{-1}]$ is divisible by $[\elln_i]^a$ and can thus be evaluated at $\zeta_i$\,. More precisely, we have for $\elln_i>1$ that
\begin{align}\label{eq:factelllim}
\left.\frac {[u]_{q_i}!}{[\elln_i]_{q_i}^a}\right\vert_{q_i\rightarrow \zeta_i}
=\Lsign_i^{ar}\cdot \LLsign_i^{\binom{a}{2}}\cdot [r]_{\zeta_i}!\cdot a!\cdot \Lfact_i^a\,,
\end{align} 
where the vertical line indicates the map to $\mathbb Z[\zeta]\,$. We will later use the following basic consequence of Lemma~\ref{lem:qbin-zetaform} for the mulitnomial coefficients from \eqref{eq:def-qmultinom}. 
\begin{cor}\label{cor:zeros-mulitnom}
Suppose $a_1,\ldots,a_r\in\nnN$ such that $a_1+\ldots+a_r=\ell_i\,$. Then
    $$
    \qbin{a_1+\ldots+a_r}{a_1\,,\,\ldots\,,\,a_r}{i}=
    \begin{cases}
       \; 1 & \mbox{ if } \, \exists s: \, a_s=\ell_i \, \mbox{ and } \, a_{t}=0 \, \mbox{ for } t\neq s
       \vspace*{2mm}\\
        \; 0 & \mbox{ otherwise }
    \end{cases}
    $$
\end{cor}

We will distinguish two types of specializations of quantum groups to roots of unity, one given by the minimal choice $n=\dpone$ and the other by the maximal choice $n=\elln\,$. The notation is as follows, where $\utypechar\in\{\geqzero,\leqzero, +, -\}\,$ and the indicated extensions of scalars
are with respect to the evaluation maps $q\mapsto \zeta$ between the rings in  
\eqref{eq:defgenrings} and those in \eqref{eq:defzetarings}.

\begin{equation}
    \begin{aligned}\label{eq:defzetaQG}
         \Uz^{\utypechar}&=\Uq^{\utypechar}\otimes\Zzn{\dpone} & 
         \qquad && \Uz=\Uq\otimes\Zzvn{\dpone} &\qquad &
         \UzQ^{\utypechar}&=\Uq^{\utypechar}\otimes\mathbb{Q} (\zeta) 
         \\
        \rule{0mm}{7mm} 
        \Uzn{\ell}^{\utypechar}&=\Uq^{\utypechar}\otimes\Zzn{\elln} && &\Uzn{\ell}=\Uq\otimes\Zzvn{\elln}&& \\ 
    \end{aligned}
\end{equation}

Many properties established for  quantum groups at generic $q$ in previous sections extend without 
changes to these quantum groups at roots of unity. 

\pagebreak[2]

\begin{samepage}
\begin{cor} The following statements hold. \label{cor:UzProps}\  \vspace*{-1.6mm}

\begin{enumerate}[label=\roman*), leftmargin=8mm,] 
\item\label{item:UzProps:free} 
The algebras defined in \eqref{eq:defzetaQG} are free modules over their respective ground rings, with bases given as in Theorem~\ref{thm:mainPBW}.\vspace*{1.5mm}

\item\label{item:UzProps:incl}  
The natural inclusions of algebras in \eqref{eq:defgenrings} induce respective embeddings for the algebras in \eqref{eq:defzetaQG}.\vspace*{1.5mm}

\item\label{item:UzProps:autom} 
All gradings and automorphisms defined  in Section~\ref{sec:gradingauts} factor into well-defined gradings and automorphisms for the algebras in \eqref{eq:defzetaQG}.
\end{enumerate}
\end{cor}
\end{samepage}

The first assertion is obvious, since all modules are free. Similarly, the existence of compatible free bases for the natural inclusions assures that the induced maps are also injective, mapping basis elements to each other. The last assertion is clear as well, since all automorphisms are, in particular,  module isomorphisms or module anti-isomorphisms with respect to the conjugations on the ground rings.  

In later discussions, we will  consider root lattice vectors 
$\ellsc{\alpha}=\elln_\alpha\cdot \alpha$ for a given root $\alpha\in\roots$ in a connected root system $\roots$ of finite type and fixed $\ell>1\,$. 
Let $\edgenum=\max\{d_i\}_i\in\{1,2,3\}\,$ as in Section~\ref{subsec:rootweyl}. If $\edgenum$ and $\ell$ are coprime, then all $\elln_\alpha=\ell$ so that $\ellsc{\alpha}=\ell\cdot \alpha\,$. Conversely, if $\edgenum>1$ and divides $\ell$, we have $\ell_\alpha=d_\alpha^{-1}\ell$ so that 
$\ellsc{\alpha}=\ell\cdot \breve\alpha\,$ in our convention for coroots. 

The collection $\roots_\ell=\{\ellsc{\alpha}=\elln_\alpha\cdot \alpha:\alpha\in\roots\}$ thus has the following properties.

\begin{lem}\label{lm:ell-root-lattice}  We have $\,\roots_\ell=\ell\cdot\roots\,$ 
if $\,\gcd(\edgenum,\ell)=1\,$
and $\,\roots_\ell=\ell\cdot\breve\roots\,$ if $\,\gcd(\edgenum,\ell)>1\,$, where $\breve\roots$ denotes the coroot system. Thus $\roots_\ell$ is itself a root system that is canonically isomorphic to either $\roots$ or 
$\breve\roots$.

\noindent The Weyl group $\Weyl$ of $\roots$ preserves $\,\roots_\ell\,$ and its action coincides with that of the intrinsic Weyl group of  $\,\roots_\ell\,$ in the sense that 
$s_{\ellscscr{\alpha}}(\ellsc{\beta})=\ellsc{(s_{\alpha}(\beta))}=s_{\alpha}(\ellsc{\beta})\,$.
Moreover, the set $\sroots_\ell=\{\elln_i\alpha_i\}_i$ is a basis of  $\roots_\ell\,$.
\end{lem}

The assertions are immediate from the preceding discussion, the fact that uniform scaling of a root system does not change its isomorphism class, and well-known properties of coroot systems.

\subsection{Skew-Commutation Relations of Primitive Power Generators} \label{subsec:CommPrimGen}

In Lemma~3.1 of \cite{dck90} De Concini and Kac observe that $\elln$-th powers of $E_\alpha$ and $F_\alpha$ generators commute with other generators of $\Uz$ up to unit factors. In this section, we will sharpen this result to include the $\elln_\alpha$-th powers, where
$\elln_\alpha=\elln/\gcd(\elln,d_\alpha)\,$. This extension is required for our later discussion since for $\elln_\alpha<\elln$\,, certain subsets of $E_\alpha^{\elln_\alpha}$ elements will  generate  Hopf ideals, while the 
$E_\alpha^{\elln}$ elements generally do not give rise to Hopf ideals.

 We provide here an explicit and conceptual proof of the commutation relations based on the $\ztgrad^{\pm}$-gradings introduced in Section~\ref{subsec:gradings} and their $\Weyl$-covariance properties. Moreover, statements will be in terms of word-dependent generators, which allows us to study independence of the relations of root orderings.

The function $\commphchar$\,, defined in \eqref{eq:defsigncomm}, keeps track of the signs in commutation relations. Here  $\Lsign_\alpha=\zeta_\alpha^{\elln_\alpha}=\zeta^{d_\alpha\elln_\alpha}\in\{+1,-1\}$ is as in \eqref{eq:defspeczetas} and $\breve\alpha=d^{-1}_\alpha\alpha$ denotes the  coroot, implying the last equality below.
\begin{equation}\label{eq:defsigncomm}
    \commphchar\,:\,\roots\times {\mathbb F_2}^{\Delta}\,\rightarrow\,\{+1,-1\}\qquad 
(\alpha,\bar\mu)\,\mapsto\,
\commph{\alpha}{\bar\mu}=\zeta^{\elln_\alpha\symbrack{\alpha}{\bar\mu}}=\Lsign_\alpha^{\symbrack{\breve\alpha}{\bar\mu}}\,. 
\end{equation}
Note, the respective map $\roots\times {\mathbb Z}^{\Delta}\,\rightarrow\,\{+1,-1\}$\,, given by the same formula for 
$(\alpha,\mu)\in\roots\times {\mathbb Z}^{\Delta}\,$, is a homomorphism in the $\mu$ argument  onto the group $\{+1,-1\}\cong\mathbb Z/2\mathbb Z\,$. The map thus factors into  
$\roots\times {\mathbb F_2}^{\Delta}$ as implied. It is also immediate from $\ell_\alpha=\ell_{s(\alpha)}$ that $\commphchar$ is $\Weyl$-bivariant in the sense that $\commph{s(\alpha)}{s(\bar\mu)}=\commph{\alpha}{\bar\mu}$ for all $s\in\Weyl\,$.

Note that $\commph{\alpha}{\bar\mu}=1\,$ for all
$\alpha$ and $\bar\mu\,$ if  $\kay$ is odd since in this case $\Lsign_\alpha=1\,$. The next lemma discusses several cases for even $\kay\,$.

\begin{lem}\label{lem:kappaprops}
Let $\kay=2\ell\,$. 
Suppose $\ell$ is odd {\em or} $\roots$ is of oddly laced Lie type $\LT{ADEG}$\,. Then 
$\,\commph{\alpha}{\bar\mu}=\commphbihom{\bar\alpha}{\bar\mu}\,$ for all $\alpha\in\roots$ and $\bar\mu\in\mathbb F^{\sroots}$, where 
$\commphbihomchar$ is the bihomomorphism
\begin{equation}\label{eq:signbihom}
\commphbihomchar\,:\;\mathbb F_2^{\sroots}\times \mathbb F_2^{\sroots}\,\longrightarrow\,\{+1,-1\}\;:\quad
(\bar\nu,\bar\mu)\,\mapsto\,\commphbihom{\bar\nu}{\bar\mu}=(-1)^{\symbrack{\bar\nu}{\bar\mu}}\,,
\end{equation}
and $\bar\alpha$ is the image of $\alpha$ under the canonical map 
$\roots\rightarrow\mathbb Z^{\sroots}\rightarrow \mathbb F_2^{\sroots}\,$.

\noindent If $\ell$ is even ($\kay\equiv 0\mod 4$) {\em and} $\roots$ is of Lie type $\,\LT{B}_n\,$, $\LT{C}_n\,$, or $\LT{F}_4\,$ with $n\geq 2\,$, then
\begin{equation}\label{eq:kappaexcept}
\commph{\alpha}{\bar\mu}\,=\,(-1)^{\symbrack{\breve \alpha}{\bar\mu}}\,,
\end{equation}
which is {\em not} additive in the first argument and, thus, does {\em not} extend to a bihomomorphic lattice map.  
\end{lem}

The assertion for oddly laced types follows from the observations that in these cases $\Lsign_\alpha=-1$ and $\symbrack{\breve  \alpha}{\bar\mu}\equiv \symbrack{\alpha}{\bar\mu} \mod 2\,$. Moreover, for types $\LT{B}, \LT{C}$ and $\LT{F}$ and odd $\ell$ one verifies $\Lsign_\alpha=(-1)^{d_\alpha}\,$. Note that for $\LT{B}$ the form $\symbrack{\,\cdot\,}{\,\cdot\,}$ is always even so that 
$\commphchar=\commphbihomchar=1$, given odd $\ell\,$.

For the doubly laced types $\LT{BCF}$ but with even $\ell$, equation 
\eqref{eq:kappaexcept} follows from the fact that $\Lsign_\alpha=-1$ for all $\alpha\in\proots\,$. 

Failure of additivity for even $\ell$ is readily verified for type $\LT{B}_2$\,, which occurs as a subsystem in the listed Lie types. If $\alpha_1$ is the short and $\alpha_2$ the long root (so, $d_1=1$ and $d_2=2$) one finds that $\commph{\alpha_1}{\bar\alpha_1}=1$\,, $\commph{\alpha_2}{\bar\alpha_1}=-1$\,, and  
$\commph{\alpha_1+\alpha_2}{\bar\alpha_1}=1$ as $\alpha_1+\alpha_2$ is also a short root. So, clearly,
$\commph{\alpha_1}{\bar\alpha_1}\commph{\alpha_2}{\bar\alpha_1}\neq \commph{\alpha_1+\alpha_2}{\bar\alpha_1}\,$, obstructing an extension as a lattice map to $\mathbb Z^{\sroots}\supset\proots\,$.  

Since $\commphchar$ will be more frequently used in combination with the $\ztgrad^{\pm}$-gradings from \eqref{eq:Topgrad} and \eqref{eq:Tomgrad} it is convenient to introduce the following functions $\commphel+{\,\cdot\,}{\cdot\,}$ and $\commphel-{\,\cdot\,}{\,\cdot\,}$\,. Here $w\in\wordset$ is a (non-empty) reduced word and $b$ is a 
$\ztgrad^{\pm}$-homogeneous element in either $\Uq$ or $\Uz$\,:
\begin{equation}\label{eq:commphelDef}
    \commphel{\pm}{w}{b}\,=\,\commph{\wordroot(w)}{\ztgrad^{\pm}(b)}\,.
\end{equation}
For $w$ and $b$ as above, suppose $w_i\cdot w$ is also a reduced word. Using $\wordroot(w_i\cdot w)=s_i(\wordroot(w))$\,,
\eqref{eq:Togradequiv}, and $\Weyl$-bivariance  of $\commph{\,\cdot\,}{\,\cdot\,}$ we infer that for any $i$\,,
\begin{equation}\label{eq:commphelEquivar}
    \commphel{\pm}{w_i\cdot w}{\Tinv_i(b)}= \commphel{\pm}{w}{b}\,.
\end{equation}

For a given non-empty, reduced word $w\in\wordset$\,, $\elln_w=\elln_{\tau(w)}=\elln_{\wordroot(w)}$\,, and root $\alpha\in\roots$ we now introduce the notation for the respective {\em primitive power generators}:
\begin{equation}\label{eq:DefXY}
    \Epw_w=E_w^{\elln_w}\;,\qquad \Fpw_w=F_w^{\elln_w}\,,\quad \mbox{and}\quad \Kpw_\alpha=K^{(\elln_\alpha\cdot\alpha)}\,,
\end{equation}
where $K^\mu$ is as in \eqref{eq:Kexpdef} and $\elln_\alpha\cdot\alpha$ indicates the integer multiple of the root vector. That is, $\Kpw_i=\Kpw_{\alpha_i}=K_i^{\elln_i}\,$. 
For $\mu\in\mathbb Z^{\Delta}$\,, $\mu_i=\mu(\alpha_i)$\,, and  $\eta=\sum_i\mu_i\ell_i\alpha_i$ set
\begin{equation}\label{eq:DefLmu}
    \Kpw^{\mu}=\Kpw_1^{\mu_1}\cdot\ldots\cdot \Kpw_n^{\mu_n}=K^\eta \,. 
\end{equation}

The $\wgrad$-grading in $\mathbb Z^{\sroots}$ from Section~\ref{subsec:gradings} on a generator for $\alpha=\wordroot(w)$
is then given by 
\begin{equation}\label{eq:wgradX_rvscor}
    \wgrad(\Epw_w)=\ell_\alpha\cdot\alpha=\frac{\ell}{\gcd(\ell,d_\alpha)}\cdot\alpha
=\begin{cases}
    \ell\cdot\alpha & \mbox{ if } \gcd(\ell,\maxd)=1\\
    \ell\cdot\breve\alpha & \mbox{ if } \gcd(\ell,\maxd)=\maxd\,.
\end{cases}
\end{equation}
That is, if $\maxd$ and $\ell$ are coprime the generators are graded by the lattice of $\ell\cdot\roots\cong\roots\,$. However, if  $\maxd$ divides $\ell$ the grading is given by the $\ell$-scaled (but equivalent) version of the lattice of the coroot system  $\breve \roots\,$. So, for example, if $\ell$ is even and $\roots$ is of type $\LT{B}_n$\,, then the $\Epw_w$
generators admit a  natural grading in the root lattice for $\LT{C}_n\,$.

Recall that in Lusztig's construction of the divided power algebras the elements $\Epw_w$ and $\Fpw_w$ are mapped to zero in the specialization 
from a $\Zqq$-algebra to a $\Zz$-algebra. 
Since the $\Tinv_j$ are algebra automorphisms and the $\elln_\alpha$ are $\Weyl$-invariant, the definition in \eqref{eq:EFwordDef} readily translates into 
\begin{equation}\label{eq:TinvXY}
\Epw_w=\Tinv_{w^\flat}(\Epw_{\tau(w)})
=\Tinv_{u}(\Epw_v)
\qquad\mbox{and}\qquad 
\Fpw_w=\Tinv_{w^\flat}(\Fpw_{\tau(w)})
=\Tinv_{u}(\Fpw_v)\,,
\end{equation}
assuming $w=u\cdot v$ is reduced and $\len{v}\geq 1\,$. Similarly, \eqref{eq:TKact}, \eqref{eq:DefXY}, and the same $\Weyl$-invariance imply
\begin{equation}\label{eq:TinvL}
\Tinv_s(\Kpw_\alpha)=\Kpw_{s(\alpha)}\qquad\mbox{for all}\;\alpha\in\roots\;\mbox{ and }\; s\in\Weyl\,. 
\end{equation}

Various other relations and properties for the $E_w$ and $F_w$ generators from Section~\ref{subsec:genwords} are easily extended to properties of the $\Epw_w$ and $\Fpw_w$ elements in the same manner. Remarkably, this includes the following identity that holds for all  $\kay\,$, due to the identity  $\Qsign^2=(-1)^{\ell_i-1}$ mentioned in Section~\ref{subsec:QnumsRingsRo1}:
\begin{equation}\label{eq:TinvXYsame}
    \Tinv_i(\Epw_i)=-\Kpw_i^{-1}\Fpw_i=-\Fpw_i\Kpw_i^{-1} \qquad\mbox{and}\qquad \Tinv_i(\Fpw_i)=-\Kpw_i \Epw_i=-\Epw_i\Kpw_i\,.
\end{equation}

The following commutation relations for a rank 2 subsystem in $\Uq^+$ are obtained from equations (c) and (i) in Section~5.3 and equations (a2) and (a6) in Section~5.4 of \cite{lu90b}. As in Section~\ref{subsec:rank2coxsys} above, assume for a pair of indices $(i,j)$ that $-A_{ji}=d_i=1$ and $-A_{ij}=d_j=\edgenum\in\{1,2,3\}\,$. The index notation for generators is as defined in  \eqref{eq:Esubseqind}
or, specialized to the rank 2 setting, in \eqref{eq:LusztigDictionary}, where also translations to Lusztig's conventions are explained. 

\begingroup
\allowdisplaybreaks
\begin{align}\label{eq:EpwE-genrels} 
\edgenum=1 \quad &\left\{{\quad 
\begin{aligned}
\rule[-2mm]{0mm}{5mm}E_i^kE_j&=q^{-k}E_jE_i^k+q^{-1}[k]_iE_{(ij)}E_i^{k-1}
\end{aligned}
}\right.
\\
\rule{0mm}{13mm}
\edgenum=2 \quad &\left\{{\quad 
\begin{aligned}
    E_j^kE_i&=q^{-2k}E_iE_j^k+q^{-2}[k]_jE_{(ji)}E_j^{k-1}
    \\
    \rule{0mm}{6.5mm}E_i^kE_j&=q^{2k}E_jE_i^k-q^{(k-1)}[k]_iE_i^{k-1}E_{(ji)}-q^{2(k-1)}[k]_i[k-1]_iE_i^{k-2}E_{(jij)}
\end{aligned}
}\right.
\\
\rule{0mm}{18mm}
\edgenum=3 \quad &\left\{{\quad 
\begin{aligned}
    E_j^kE_i&=q^{-3k}E_iE_j^k + q^{-3}[k]_jE_{(ji)}E_j^{k-1}
    \\
    \rule{0mm}{6.5mm}
    E_i^kE_j&=q^{3k}E_jE_i^k-q^{(k-1)}[k]_iE_i^{k-1}E_{(ji)}-q^{2(k-1)}[k]_i[k-1]_iE_i^{k-2}E_{(jiji)}\\
    &\quad \; -q^{3(k-1)}[k]_i[k-1]_i[k-2]_iE_i^{k-3}E_{(jijij)}
\end{aligned}
}\right. 
\end{align}
\endgroup

Since all equations in \eqref{eq:EpwE-genrels} are well-defined equations in $\Uq^+$ over $\Zqqn{\dpone}$
they descend to the same equations in $\Uz^+$  or $\Uz$ via \eqref{eq:defzetaQG}, upon specializing $q\mapsto \zeta\,$ with the previous assumptions in the order of $\zeta$\,. If we set $k=\elln_i$ in the equations with an $E_i^k$-term and $k=\elln_j$ in equations with an $E_j^k$-term, only the first term on the right side of each equation remains, since 
$[\elln_i]_i=[\elln_j]_j=0$ by the restriction in \eqref{eq:kaycond}. 
The resulting relations can be summarized in the following lemma, where $\Epw_s=\Epw_{\alpha_s}=E_s^{\elln_s}$\,.

\begin{lem}\label{lem:XEcommrel}  
For all $r,s\in\{1,\ldots,n\}$ we have  $  \Epw_sE_r=\zeta^{\elln_sd_sA_{sr}}E_r\Epw_s\,$.
\end{lem}

\begin{proof}
    If $r\neq s$ and $A_{rs}\neq 0$ the $\zeta$ power is easily verified to be identical to the respective lead coefficient in the formulae in \eqref{eq:EpwE-genrels} for $(s,r)=(i,j)$ and $(s,r)=(j,i)\,$. If $A_{rs}= 0$\,,
    the relation is clear since $E_s$ and $E_r$ strictly commute to begin with. If $r=s$ we find $\zeta^{\elln_sd_sA_{ss}}=\zeta^{\elln_sd_s2}=1$ by definition and as desired. 
\end{proof}

Note, further, that the relations in \eqref{eq:EFcomm} and \eqref{eq:Kcomm} entail commutation relations  $E_s^{k}K_r=q^{-k\symbrack{\alpha_s}{\alpha_r}}K_rE_s$ and  $[E_s^k,F_r]=\delta_{sr}(q_s-q_s^{-1})^{-1}[k]_sE_s^{k-1}(q_s^{k-1}K_s-q_s^{-(k-1)}K_s^{-1})$ in $\Uq$\,,
which specialize for $k=\elln_i$ and $q\mapsto \zeta$ to 
\begin{equation}\label{eq:XFXKcommrel}
    \Epw_sF_r=F_r\Epw_s \qquad \mbox{and}\qquad \Epw_sK_r=\zeta^{\elln_s\symbrack{\alpha_s}{\alpha_r}} K_r\Epw_s\;. 
\end{equation}
 
\begin{prop} \label{prop:PwGenCommRels}
Let $w\in\wordset$ be a non-empty, reduced word and $b$ a $\ztgrad$-homogeneous element. Then 
$$
\Epw_wb\,=\,\commphel{+}{w}{b}\cdot b\Epw_w\qquad\mbox{and}\qquad \Fpw_wb\,=\,\commphel{-}{w}{b}\cdot b\Fpw_w\,.
$$
Each relation holds in $\Uz$ as well as each $\Uz^\utypechar$ with $\utypechar\in\{+,-,\geqzero,\leqzero\}\,$, provided the elements in the relation belong to the respective algebra. 
\end{prop}

\begin{proof}
We begin by considering the relation for $\Epw_s=\Epw_{w_s}$ in the full algebra $\Uz$\,. Since 
$\commphel{+}{w_s}{E_r}=\commph{\wordroot(w_s)}{\ztgradp(E_r)}=\commph{\alpha_s}{\alpha_r}=\zeta^{\elln_s\symbrack{\alpha_s}{\alpha_r}}=\zeta^{\elln_sd_sA_{sr}}\,$, Lemma~\ref{lem:XEcommrel}
 implies the desired relation for $b=E_r\,$. The calculation for $\commphel{+}{w_s}{K_r}$ is exactly the same so that, by the second identity in \eqref{eq:XFXKcommrel}, the assertion also holds for $b=K_r\,$. Further, the first identity in \eqref{eq:XFXKcommrel} together with $\commphel{+}{w_s}{F_r}=\commph{\wordroot(w_s)}{\ztgradp(F_r)}=\commph{\alpha_s}{0}=1\,$ shows the commutation formula for $b=F_r\,$.

 Next, observe that $\commphel{+}{w}{bc}=\commphel{+}{w}{b}\commphel{+}{w}{c}$ for $\ztgrad$-homogeneous elements $b$ and $c$ because $\commphchar$ is a homomorphism in the second argument and $\ztgrad$ is a grading. Thus, we can establish $\Epw_sb=\commphel{+}{w}{b}\cdot b\Epw_s$ for any $b=b_1\ldots b_N$  with $b_j\in\{E_r,F_r,K_r\}$ by iteration and, by additivity, further extend the relation to all 
 homogeneous $b\in\Uz\,$.

 The proof for general $w\in\wordset$ follows now by  induction in $\len{w}\,$.
 Suppose  $w$ and $w_i\cdot w$ are reduced words so that $\Tinv_i(X_w)=X_{w_i\cdot w}\,$ by \eqref{eq:TinvXY}. Assume further the commutation relation holds for $w$ and any homogeneous element in $\Uz\,$. Applying the automorphism $\Tinv_i$ to 
 $\Epw_wc\,=\,\commphel{+}{w}{c}\cdot c\Epw_w$ then yields 
 $\Epw_{w_i\cdot w}b\,=\,\commphel{+}{w}{c}\cdot b\Epw_{w_i\cdot w}\,$, where $b=\Tinv_i(c)\,$. From 
 \eqref{eq:commphelEquivar} we find that $\commphel{+}{w}{c}=\commphel{+}{w_i\cdot w}{b}\,$. Since $\Tinv_i$ is a bijection among homogeneous elements, this shows that the relation also holds for $w_i\cdot w$ and, hence, for all reduced words.

 Note that \eqref{eq:CinvEFwords} implies $\Cartaninv(\Epw_w)=\Fpw_w$\,, and we have
 $\commphel{+}{w}{\Cartaninv(b)}=\commphel{-}{w}{b}$ by \eqref{eq:CartinvGrad}. 
 From this, the second relation with $\Fpw_w$ is inferred by applying $\Cartaninv$ to the first, using also that 
 the $\commphel{\pm}{w}{b}\in\{+1,-1\}\,$. Finally, 
 the relations in $\Uz$ imply the same relations in the embedded algebras $\Uz^\utypechar\,$ via Corollary~\ref{cor:UzProps}.\ref{item:UzProps:incl}.
 \end{proof}

Define also $\commphel{\circ}{\alpha}{b}=\commph{\alpha}{\wgrad(b)}=\commph{\alpha}{\ztgradp(b)}\commph{\alpha}{\ztgradn(b)}\,$ for a root $\alpha\in\roots$ and a  
fully homogeneous element $b$\,, where the second equality results from
$\wgrad=\ztgradp-\ztgradn$\,, noted in Section~\ref{subsec:gradings}. Equation
\eqref{eq:Kwgrad} now immediately implies the following 
commutation relations for generators defined 
in \eqref{eq:DefXY} and \eqref{eq:DefLmu}: 
\begin{equation}\label{eq:Lbcommrel}
    \Kpw_\alpha b=\commphel{\circ}{\alpha}{b}\cdot b \Kpw_\alpha 
    \qquad \mbox{ and } \qquad 
    \Kpw^\mu b=\commphel{\circ}{\mu}{b}\cdot b \Kpw^\mu\,.
\end{equation}

Commutation relations among the primitive power generators themselves are readily derived from Proposition~\ref{prop:PwGenCommRels}:
\begin{equation}\label{eq:XYcommrels}
    \Epw_v\Epw_w=\commphinel v w \Epw_w\Epw_v\,,\qquad
    \Fpw_v\Fpw_w=\commphinel v w \Fpw_w\Fpw_v\,,
    \quad\mbox{and}\quad
    \Epw_v\Fpw_w= \Fpw_w\Epw_v
\end{equation}
where 
\begin{equation}
 \commphinel v w = \commphin{\wordroot(v)}{\wordroot(w)}
\qquad\mbox{with}\quad 
 \commphin \alpha \beta =\zeta^{\elln_\alpha\elln_\beta\symbrack{\alpha}{\beta}}\in\{+1,-1\}\;.
 \end{equation}

For the generators $\Kpw_\alpha$ as in \eqref{eq:DefXY} we also derive from 
$K^\eta E_w=q^{\symbrack{\eta}{\wordroot(w)}}E_wK^\eta$ that
\begin{equation}\label{eq:LcommXYrels}
 \Kpw_\alpha \Epw=  \commphin {\alpha}{\wordroot(w)}\Epw\Kpw_\alpha
 \qquad\mbox{and}\qquad
 \Kpw_\alpha \Fpw=  \commphin {\alpha}{\wordroot(w)}\Fpw\Kpw_\alpha\,.
\end{equation}

We conclude this section with criteria for strict commutativity of a primitive power generator $X_v$ with $\alpha=\wordroot(v)\,$ for non-empty $v\in\wordset\,$. As before, we assume that the root system is indecomposable of rank $\geq 2\,$ and denote by $\kay$ the order of $\zeta\,$. 

\begin{lem} \label{lm:Xcentral}
The power generator $X_v$ is central in $\Uz^+$ if and only if (at least) one of the following conditions holds.
\vspace{-2mm}

\begin{enumerate}[label=\roman*), leftmargin=13mm,] 
    \item $\kay$ is odd.\vspace{1.5mm}
    \item $\kay\equiv 2 \mod 4$\,, $d_\alpha=2$ (that is, the Lie type is $\LT{BCF}$ and $\alpha$ is a long root).  \vspace{1.5mm}
    \item The Lie type is $\LT{B}$ and $\alpha$ is a short root (that is, $d_\alpha=1$). 
\end{enumerate}
\vspace{-1.5mm}

\noindent
In these cases $\Epw_v$ also commutes with all elements in $\Uz$\,. The analogous statement holds for $\Fpw_v$ and $\Kpw_\alpha$\,.  
\end{lem}

Note that the first condition includes all Lie types and was already observed in \cite{dck90}. See also the remarks following 
Proposition~\ref{prop:Zcentral} below.
The third condition applies to all $\kay\,$ and it is the only allowed situation if $\kay\equiv 0 \mod 4\,$. The proof consists of checking for which $\alpha$\,, Lie types, and orders $\kay$ we have that $\commph{\alpha}{\bar\mu}=1$ for all $\bar\mu$\,. As a weaker condition, one may consider commutation of $\Epw_v$ with only the other primitive power generators. 

\begin{lem}\label{lm:Xcommut} $\Epw_v$ commutes with all $\Epw_w$ if and only if (at least) one of the following conditions holds.
\vspace{-2mm}

\begin{enumerate}[label=\roman*), leftmargin=13mm,] 
    \item $\kay$ is odd or $\kay \equiv 0 \mod 8\,$.\vspace{1.4mm}
    \item $\kay\equiv 2 \mod 4$ and the Lie type is $\LT{B}$\,.\vspace{1.4mm}
    \item $\kay\equiv 2 \mod 4$, the Lie type is $\LT{CF}$\,, 
    and $\alpha$ is a long root ($d_\alpha=2$).\vspace{1.4mm}
    \item $\kay\equiv 4 \mod 8$ and the Lie type is
    $\LT{ACDEG}$\,.
    \vspace{1.4mm}
    \item $\kay\equiv 4 \mod 8$, the Lie type is $\LT{F}$\,, and $\alpha$ is a short root.\vspace{1.4mm}
    \item $\kay$ is arbitrary, the Lie type is $\LT{B}$\,, and $\alpha$ is a short root.
\end{enumerate}
\vspace{-1.5mm}

\noindent
    In these cases $\Epw_v$ also commutes with all $\Fpw_v$ and $\Kpw_\mu$\,. 
    The analogous statement holds for $Y_v$ and $\Kpw_\alpha$\,. 
\end{lem}

\subsection{Skew-Central Subalgebras and Monomial Bases}\label{subsec:Z=SignPolyn} 
The $\Epw_w$ and $\Fpw_w$ elements introduced in Section~\ref{subsec:CommPrimGen} generate almost polynomial subalgebras whose indeterminate variables commute up to signs. This section establishes the basic properties of these algebras and classifies when they are central, commutative but not central, or neither. We begin with definitions of the monomial elements in analogy to those introduced in Section~\ref{subsec:Emonomials}.

Let $\ellfun\in\nnN^{\roots}$ be the length function $\ellfun(\alpha)=\elln_\alpha\,$ and, for $\psi\in \nnN^{A}$ with $A\subseteq\roots$\,, denote by $\ellfun\cdot\psi$ the point-wise product 
$(\ellfun\cdot\psi)(\alpha)=\elln_\alpha\psi(\alpha)\in\nnN\,$. For a non-empty, reduced word $w\in\wordset\,$,
$\psi\in\expsetup s\,$ with $s=\Weylpres(w)\,$,
and with conventions as for \eqref{eq:ExpsetDef}, set
\begin{equation}\label{eq:defXmonom}
    \Epwbase{w}{\psi}\,=\,\Ebase{w}{\ellfun\cdot\psi}
\,=\,\Epw_{w[\beta_1]}^{\psi(\beta_1)}\ldots \Epw_{w[\beta_L]}^{\psi(\beta_L)}
\,=\,E_{w[\beta_1]}^{\elln_{\beta_1}\psi(\beta_1)}\ldots E_{w[\beta_L]}^{\elln_{\beta_L}\psi(\beta_L)}\,.
\end{equation}
Analogously, we denote \ $\Epwbaseopp{w}{\psi}=\Ebaseopp{w}{\ellfun\cdot\psi}$\,, $\Fpwbase{w}{\psi}=\Fbase{w}{\ellfun\cdot\psi}$\,, and $\Fpwbaseopp{w}{\psi}=\Fbaseopp{w}{\ellfun\cdot\psi}\,$. Suppose next that $w,u\in\wordset$ are reduced words with associated Weyl group elements $s=\Weylpres(w)$ and $t=\Weylpres(u)$\,. Further, let $\psi,\phi\in \expsetup s$ and $\chi\in \expsetup t$ be respective exponent maps. Using the root sets and orderings explained in   
Section~\ref{subsec:descroots} we then define bihomomorphisms valued in $\{+1,-1\}$ as  
\begin{equation}
    \commphallexp{w}{u}{\psi}{\chi}   =\prod_{\substack{\alpha\in \descroots s \\ \beta\in \descroots t }}
\commphin{\alpha}{\beta}^{\psi(\alpha)\chi(\beta)}
\qquad\mbox{and}\qquad
\commphshufexp{w}{\psi}{\phi}   =\prod_{\substack{\alpha, \beta\in \descroots s \\ \alpha \,\letwt{w}\,\beta}}
\commphin{\alpha}{\beta}^{\psi(\beta)\phi(\alpha)}\,,
\end{equation}
for which we observe that $\;\commphshufexp{w}{\psi}{\phi}\,\commphshufexp{w}{\phi}{\psi}= \commphallexp{w}{w}{\psi}{\phi}\,$.

The commutation relations from \eqref{eq:XYcommrels} readily generalize to  
\begin{equation}
    \Epw_w^\psi\,\Epw_u^\chi\,=\,\commphallexp{w}{u}{\psi}{\chi}\,\Epw_u^\chi\,\Epw_w^\psi
    \;,\quad
    \Fpw_w^\psi\,\Fpw_u^\chi\,=\,\commphallexp{w}{u}{\psi}{\chi}\,\Fpw_u^\chi\,\Fpw_w^\psi
    \;,\quad\mbox{and}\quad 
    \Fpw_w^\psi\,\Epw_u^\chi\,=\, \Epw_u^\chi\,\Fpw_w^\psi\,.
\end{equation}
Here, $\Epw^\psi_w$ and $\Fpw^\psi_w$ denotes a monomial with any fixed direction of multiplication. That is, we may replace 
 $\Epw^\psi_w$ in each equation by either  $\Epwbase{w}{\psi}$ or $\Epwbaseopp{w}{\psi}$ 
 and similarly for $\Fpw^\psi_w\,$. 
Commutation relations for a fixed reduced word are given by 
\begin{equation}\label{eq:Xmultrels}
    \Epwbase{w}{\psi}\Epwbase{w}{\phi}\,=\,\commphshufexp{w}{\psi}{\phi} \,\Epwbase{w}{\psi+\phi}
    \,,\;
    \Epwbaseopp{w}{\psi}\Epwbaseopp{w}{\phi}\,=\,\commphshufexp{w}{\phi}{\psi} \,\Epwbaseopp{w}{\psi+\phi}
    \,,\;\;\mbox{and}\;\;
    \Epwbase{w}{\psi}\,=\commphshufexp{w}{\psi}{\psi}\, \Epwbaseopp{w}{\psi}\,.
\end{equation}
The same equations hold if all $\Epw$'s are replaced by $\Fpw$'s. We also record formulae for 
the involutions from Section~\ref{subsec:Kscale} given by 
\begin{equation}\label{eq:CartMhoXYacts}
    \Cartaninv(\Epwbase{w}{\psi})=\Fpwbaseopp{w}{\psi}=\commphshufexp{w}{\psi}{\psi}\Fpwbase{w}{\psi}
\qquad\mbox{and}\qquad 
    \Kinvaut(\Epwbase{w}{\psi})\,=\,\Tinv_r\left({\Epwbase{w^\winvchar}{\phi}}\right)\,.
\end{equation}
The first is immediate from \eqref{eq:EFCartOrder} and the second follows from 
Proposition~\ref{prop:mhomonomial} with elements $r$ and $\phi$ as stated there.
In either formula we use that $\Weyl$ preserves root lengths.

For a non-empty, reduced word $w\in\wordset$ we denote by $\Zsubalgchar^+_w\subseteq \Uz^+$ and $\Zsubalgchar^-_w\subseteq \Uz^-$ the unital subalgebras generated by the sets $\{\Epw_v:\,v\leqRB w\}\cup\{1\}\,$ and 
$\{\Fpw_v:\,v\leqRB w\}\cup\{1\}\,$, respectively. Further, let 
$\Zsubalgchar^0_s$ be the unital subalgebra generated by 1 and all $\Kpw_\alpha^{\pm 1}$ with $\alpha\in\descroots{s}$\,.  
For a given $s\in\Weyl$ denote also the $\ell$-scaled exponent set
\begin{equation}\label{eq:Lexpsetdef}
    \Lexpset s\,=\,\ellfun\cdot \expsetup{s}\,=\,\left\{\phi\,:\,\phi(\alpha)\in\elln_\alpha\nnN\right\}\,.
\end{equation}
In the notation from \eqref{eq:PBW_p_basis} we thus have $\basis{w}{{\Lexpset s},+} =\{\Epwbase{w}{\psi}:\,\psi\in\expsetup{s}\}\,$ and analogous identities for $\basisopp{w}{{\Lexpset s},+}$,   
$\basis{w}{{\Lexpset s},-}$, and  $\basisopp{w}{{\Lexpset s},-}$. The ground ring $\Zgen$ in the following result is normally assumed to be $\Zzn{\dpone}$ for $\Zsubalgchar^{\pm}_w$ or $\Zz$ for $\Zsubalgchar^{0}_s$ but may be extended to
any $\Zzn{n}$ with $n\leq\ell$ or $\mathbb Q(\zeta)$\,. We are also reusing the notation for spans $\basis{w}{{\mathsf S},\pm}$ from
\eqref{eq:DefSspans} for these ground rings $\Zgen\,$.

\begin{prop}\label{prop:ZPBWbases} Suppose $w\in\wordset$ and $s=\Weylpres(w)$\,. 
Then both $\basis{w}{{\Lexpset s},+}$ and $\basisopp{w}{{\Lexpset s},+}$ are $\Zgen$-bases for  $\Zsubalgchar^+_w\,$.
In particular, $\Zsubalgchar^+_w=\bspan{w}{{\Lexpset s},+}\!\!=\bspanopp{w}{{\Lexpset s},+}$.

\noindent The subalgebra $\Zsubalgchar^+_w\,$ is isomorphic to  the abstract $\Zgen$-algebra with $1$, freely generated by elements 
$\Epwbase{w}{\psi}$ and subject to the first set of relations in \eqref{eq:Xmultrels}.
The analogous statements for   $\Zsubalgchar^-_w\,$ apply.

\noindent The subalgebra $\Zsubalgchar^0_s\,$ is generated by $\Kpw_i^{\pm 1}$ with $i\in\wordindexset{s}\,$ as in Section~\ref{subsec:descroots}. It is thus isomorphic to the group algebra $\Zgen\left[\left\{L_i^{\pm}:\,i\in\wordindexset{s}\right\}\right]\cong\Zgen\left[{\wordindexlattice{s}}\right]\,$.  
\end{prop} 

\begin{proof}
    The relations in \eqref{eq:Xmultrels} imply that for any $v\in\wordset$ with $v\leqRB w$ we have 
    $\Epw_v\cdot\bspan{w}{{\Lexpset s},+}\subseteq\bspan{w}{{\Lexpset s},+}\!$, writing $\Epw_v=\Epwbase{w}{\chi_\alpha}$
    with $\chi_\alpha$ the indicator function for $\alpha=\wordroot(v)\,$. Thus 
    $\Zsubalgchar^+_w=\Zsubalgchar^+_w\cdot 1\subseteq\bspan{w}{{\Lexpset s},+}$. From 
    the definition in \eqref{eq:defXmonom}, it is further clear that $1$ is a cyclic element of $\bspan{w}{{\Lexpset s},+}\!$
    as a $\Zsubalgchar^+_w$-module and hence $\Zsubalgchar^+_w=\bspan{w}{{\Lexpset s},+}\!$ as desired. Linear independence established in Section~\ref{sec:PBW} then completes the proof of the first claim. Independence under reversal of the direction of multiplication is
    immediate from \eqref{eq:Xmultrels}.

    It is clear that the set of generators $\Epwbase{w}{\psi}$ of an abstractly defined algebra with relations as in 
    \eqref{eq:Xmultrels} is also a basis. Thus, the canonical map from this algebra onto $\Zsubalgchar^+_w$ is an 
    isomorphism of $\Zgen$-modules and, hence, also an algebra isomorphism.

    By its definition, $\Zsubalgchar^0_s\,$ is the $\Zgen$-span of all $K^\nu$ with $\nu$ in the root 
    sublattice generated by all $\ellsc{\alpha}=\elln_\alpha\cdot \alpha\,$ with $\alpha\in\descroots{s}\,$.
    Following Lemma~\ref{lm:ell-root-lattice}, we consider the root system $\roots_\ell$ with the same Weyl group 
    action. Since, clearly, $s(\alpha)\in\nroots$ iff $s(\ellsc{\alpha})\in\nroots_\ell$ the 
    inversion set $\descroots{s}_\ell$
    with respect to $\roots_\ell$ is given precisely by $\{\ellsc \alpha:\,\alpha\in \descroots{s}\}$\,.
    By Lemma~\ref{lem:descrootlattice} the root sublattice generated by this set is spanned by the simple roots 
    of $\roots_\ell$ with indices in $\wordindexset{s}\,$. By Lemma~\ref{lm:ell-root-lattice} these are
    $\ellsc{\alpha_i}=\elln_i\alpha_i$\,, yielding the $\Kpw_i=K^{\ellsc{\alpha_i}}$ as free generators.  
\end{proof}

 The next observations are consequences of \eqref{eq:CartMhoXYacts}, particularly, in the case when 
 the word in the $\Kinvaut$ formula is maximal so that $r=1\,$. The stated equivariance under the
 antipode follows from \eqref{eq:SXiMhoRel}.

\begin{cor}\label{cor:InvolZalg}
Suppose $w\in\wordset$ is a reduced word and $z\in\wordsetmax$ one of maximal length. 
Then $\Cartaninv(\Zsubalgchar_s^{0})=\Kinvaut(\Zsubalgchar_s^{0})=S(\Zsubalgchar_s^{0})=\Zsubalgchar_s^{0}$ as well as  
\begin{equation}\label{eq:InvolZalg}
    \Cartaninv(\Zsubalgchar_w^{\pm})=\Zsubalgchar_w^{\mp}\,,\qquad 
    \Kinvaut(\Zsubalgchar_z^{\pm})=\Zsubalgchar_{z^\winvchar}^{\pm}\,,\qquad 
    \mbox{and}
    \qquad
    S(\Zsubalgchar_z^{\geqzero})=\Zsubalgchar_{z^\winvchar}^{\geqzero}\;. 
\end{equation}
\end{cor}

Proposition~\ref{prop:PwGenCommRels} further implies the following relations between $\Epwbase{w}{\psi}$ and
$\Ebase{w}{\rho}$ elements. Here $w$ is a reduced word and $\psi,\rho\in\expsetup{s}$ with $s=\Weylpres(w)\,$,
and $\commphchar$ from \eqref{eq:defsigncomm} is lifted to $\roots\times\mathbb Z^{\sroots}\,$:
\begin{equation}\label{eq:XexEex=Eex}
\begin{aligned}
     \Epwbase{w}{\psi}\cdot \Ebase{w}{\rho}\,&=\,\commphXEpw{w}{\psi}{\rho}\cdot \Ebase{w}{\ellfun\cdot \psi +\rho}
     &&\mbox{and}\qquad&
     \Ebase{w}{\rho}\cdot \Epwbase{w}{\psi}\,&=\,\commphEXpw{w}{\psi}{\rho}\cdot \Ebase{w}{\ellfun\cdot \psi +\rho}
     \\ 
     \mbox{where}\rule{0mm}{7mm}\qquad
     \commphXEpw{w}{\psi}{\rho}
            &=\!\prod_{\substack{\alpha, \beta\in \descroots s \\ \alpha \,\letwt{w}\,\beta}}
            \!\commph{\alpha}{\beta}^{\psi(\beta)\rho(\alpha)}
    &&\mbox{and}\quad&
    \commphEXpw{w}{\psi}{\rho}&=\!\prod_{\substack{\alpha, \beta\in \descroots s \\ \alpha \,\letwt{w}\,\beta}}
    \!\commph{\alpha}{\beta}^{\psi(\alpha)\rho(\beta)}\,.
\end{aligned}
\end{equation}

 Using division with remainder pointwise for each $\alpha$\,, any $\phi\in \expsetup s$ can be uniquely written as $\phi=\ellfun\cdot \psi +\rho$ with $\rho\in \Lmaxexpset s$\,, where 
 \begin{equation}\label{eq:Lmaxexpset}
     \Lmaxexpset s =\left\{{\psi\in\expsetup s\,:\,\psi(\alpha)<\elln_\alpha\;,\;\forall\alpha\in\descroots{s}}\right\}
     \quad\mbox{and}\qquad
     \Lmaxexpset \bullet =\Lmaxexpset {\longweyl}
     \,.
 \end{equation}
 It is now immediate that the multiplication map below is an isomorphism of $\Zgen$-modules, mapping pairs of PBW basis elements to another PBW basis element up to a sign.
\begin{equation}\label{eq:ZSrem=S}
  \ZSisom^+_w:\,  \Zsubalgchar^+_w\otimes_\Zgen\bspan{w}{\Lmaxexpset s, +}\!\longrightarrow \,\bspan{w}{ +}
\end{equation}

Viewing $\bspan{w}{ +}\!$ as a $\Zsubalgchar^+_w$-module with its left regular action given by the signed shifts in \eqref{eq:XexEex=Eex}, this implies that $\bspan{w}{ +}$ splits into directed summands, each of which has  $\Ebase{w}{\rho}$
for some $\rho\in\Lmaxexpset s$ as cyclic element. 

Note also that $\bspan{w}{\Lmaxexpset s, +}\!$ generally depends explicitly on the choice of the
reduced word $w$ rather than just the respective Weyl element $s=\Weylpres(w)\,$. For example, for Lie type $\LT{A}_2$ and $\kay=4\,$ both $w=w_1w_2w_1$
and $\widebarr w=w_2w_1w_2$ represent the maximal element in $\Weyl$. Given that $\ell=2$\,, the 
components of  $\bspan{w}{\Lmaxexpset s, +}$ and $\bspan{\widebarr w}{\Lmaxexpset s, +}$ with maximal $\wgrad$-grading of 
$2\alpha_1+2\alpha_2$ are $\Zgen \Ebase{w}{\mathbf 1}$ and $\Zgen \Ebase{\widebarr w}{\mathbf 1}\,$ respectively, where
$\mathbf 1(\alpha)=1$ for all $\alpha\in\proots\,$. With  
$\,\Ebase{w}{\mathbf 1}=E_1E_{(12)}E_2\,$ and 
$\,\Ebase{\widebarr w}{\mathbf 1}=E_2E_{(21)}E_1=\Ebase{w}{\mathbf 1}+E_{(12)}^2-2\zeta^{-1}E_1^2E_2^2\,$, these are clearly different spaces. We will, however, prove below that $\Zsubalgchar^+_w$ only depends on $s=\Weylpres(w)$ for almost all Lie types.

Recall from Theorem~\ref{thm:mainPBW} and Corollary~\ref{cor:UzProps} that $\bspan{w}{ +}=\Uz^+$ for a reduced word
$w\in\wordsetmax$ of maximal length. Variations of \eqref{eq:XexEex=Eex} and \eqref{eq:ZSrem=S} are readily derived 
with reversed directions or replacing $E$'s by $F$'s and $\Epw$'s by $\Fpw$'s. 

Similar to the full quantum groups, we introduce combined subalgebras such as 
\begin{equation}\label{eq:defZ-othtyp}
    \Zsubalgchar^{\geqzero}_w=\Zsubalgchar^{+}_w\cdot\Zsubalgchar^{0}_s
\;,\quad 
\Zsubalgchar^{\leqzero}_w=\Zsubalgchar^{-}_w\cdot\Zsubalgchar^{0}_s
\;,\quad \mbox{and}\quad
\Zsubalgchar^{}_w=\Zsubalgchar^{+}_w\cdot\Zsubalgchar^{-}_w\cdot\Zsubalgchar^{0}_s\,,
\end{equation} 
where $s=\Weylpres(w)$\,. Analogous statements about PBW bases and identifications with signed polynomial algebras hold for each of these subalgebras. Since all elements in $\Zsubalgchar^{+}_w$ commute with all elements in $\Zsubalgchar^{-}_w$ we do not require the denominators $(\zeta-\zeta^{-1})^{-1}$ so that each $\Zsubalgchar^{}_w$ is well-defined over $\Zzn{\dpone}\,$.
We next record the precise criteria for these subalgebras to be strictly central or strictly commutative. 

\begin{prop} \label{prop:Zcentral}
Suppose $w\in\wordset$ is non-empty, $s=\Weylpres(w)\,$, and $\utypechar\in\{\phantom{+},+,\geqzero\}$\,.
The algebra $\Zsubalgchar^+_w$ is central in $\Uq^\utypechar$ if and only if (at least) one of the following conditions holds:
\vspace{-1mm}

\begin{enumerate}[label=\roman*), leftmargin=13mm,] 
    \item $\kay$ is odd (any Lie type).\vspace{1.4mm}\label{item:Zcentral:1}
    \item $\kay\not\equiv 0 \mod 4$ and the Lie type is $\LT{B}$\,.\vspace{1.4mm}\label{item:Zcentral:2}
    \item $\kay\equiv 2 \mod 4$\,, the Lie type is $\LT{CF}$\,,
    and $s\in\WeylS$\,.\vspace{1.4mm}\label{item:Zcentral:3}
    \item Lie type is $\LT{B}$ and $s\in\WeylL$ (any $\kay$). \label{item:Zcentral:4}
\end{enumerate}
Thus, if $w\in\wordsetmax$\,, then $\Zsubalgchar^+_w$ is central if and only if (at least) one of {\em \ref{item:Zcentral:1}} and {\em \ref{item:Zcentral:2}} hold. Analogous statements hold for $\Zsubalgchar^\utypechar_w\,$. 
\end{prop}

The subgroups $\WeylL$ and $\WeylS$ are as defined in \eqref{eq:def:WeylSL}. The one referred to in items {\em \ref{item:Zcentral:3}} and {\em \ref{item:Zcentral:4}} are of order 2 for the $\LT{B,C}$ types and of order 6 for $\LT{F}_4$ by Lemma~\ref{lm:WeylSL}. The assertions in 
Proposition~\ref{prop:Zcentral} are readily derived from Lemma~\ref{lm:Xcentral}. We note that the last statement for maximal words is commensurable with Corollary~3.1 of \cite{dck90}. The elements  considered there in part (b) correspond to our $X_w$ if $\alpha=\wordroot(w)$ is a short root but to our $X_w^2$ if $\alpha$ is a long root in case $\kay$ is a multiple of 4. Requiring centrality of $X_w$ rather than $X_w^2$ then excludes this case. 

In order to formulate criteria for commutativity of the subalgebras, we say that an element $s\in\Weyl$ is
{\em mod-2 discrete} if 
 \begin{equation}\label{eq:def:mod2discrete}
     \symbrack{\breve \alpha}{\beta}\equiv 0\,\mod 2\qquad \forall\alpha,\beta \in\descroots{s}\;\mbox{ with }\,d_\alpha\geq d_\beta\,. 
 \end{equation}
 For example, for Lie type $\LT{A}_5$ and $s=s_1s_3s_6$ we have $\descroots{s}=\{\alpha_1,\alpha_3,\alpha_6\}$\,, which are all orthogonal so that $s$ is mod-2 discrete. For non-simply laced types we will also call an element $s\in\Weyl$
 {\em mod-2 long discrete} if condition \eqref{eq:def:mod2discrete} is required only for pairs $\alpha,\beta\in\descrootsL{s}$\,. 
 The notion of {\em mod-2 short discrete} is analogous.

\begin{prop} \label{prop:Zcommut} Suppose $w\in\wordset$ is non-empty and $s=\Weylpres(w)$\,. The subalgebra $\Zsubalgchar^+_w$ is commutative if and only if (at least) one of the following conditions holds:
\vspace{-1mm}

\begin{enumerate}[label=\roman*), leftmargin=13mm,] 
    \item $s$ is mod-2 discrete  (any Lie type and $\kay$).\vspace{1.4mm}\label{item:Zcommut:1}
    \item $\kay$ is odd or $\kay\equiv 0\mod 8$  (any Lie type).\vspace{1.4mm}\label{item:Zcommut:2}
     \item $\kay\equiv 0\mod 4$ and the Lie type is $\LT{ADEG}$\,.
     \vspace{1.4mm}\label{item:Zcommut:3}
     \item $\kay\equiv 2\mod 4$\,, the Lie type is $\LT{BCF}$\,,
     and $s$ is mod-2 short discrete.\vspace{1.4mm}\label{item:Zcommut:4}
     \item $\kay\equiv 4\mod 8$\,, the Lie type is $\LT{BCF}$\, 
     and $s$ is mod-2 long discrete.\label{item:Zcommut:5}
\end{enumerate}
The identical statement holds for $\Zsubalgchar^\geqzero_w\,$. 
\end{prop}

The extension from $\Zsubalgchar^+_w$ to $\Zsubalgchar^\geqzero_w$ holds since, by \eqref{eq:LcommXYrels}, the sign occurring in the commutation of $\Epw_\alpha$ with $\Epw_\beta$ (with respect to some order) is
the same as the one in the commutation of $\Epw_\alpha$ with $\Kpw_\beta\,$. Note also that for type
$\LT{A}_1$ case {\em \ref{item:Zcommut:1}} holds for any $s\,$.  

For the maximal case, observe that the longest element $\longweyl\in\Weyl$ is never mod-2 discrete for any Lie type of rank $\geq 2$\,. In the non-simply laced case we have that $\longweyl\in\Weyl$ is mod-2 short discrete if and only if the Lie type is $\LT{B}$\,. It is 
mod-2 long discrete if and only if the Lie type is $\LT{C}\,$. This yields the following characterization of maximal commutative subalgebras, which can also be derived from Lemma~\ref{lm:Xcommut}. 

\begin{cor}\label{cor:Zcommutmax}
    Let $z\in\wordsetmax$ be a reduced word of maximal length. Then $\Zsubalgchar^+_z$ is commutative if and only if (at least) one of the following applies:
\vspace{-1mm}

\begin{enumerate}[label=\roman*), leftmargin=13mm,] 
    \item $\kay\equiv 0, 1, 3, 5, \mbox{ or } 7\mod 8\,$ (any Lie type).\vspace{1.4mm}
    \item $\kay\equiv 2 \mbox{ or } 6\mod 8\,$ and the Lie type is $\LT{B}\,$.\vspace{1.4mm}
    \item $\kay\equiv 4\mod 8\,$ and the Lie type is $\LT{ACDEG}$\,.
    \vspace{1.4mm}
    \item The Lie type is $\LT{A}_1$ (any $\kay$). 
\end{enumerate}
The identical statement holds for $\Zsubalgchar^\geqzero_z\,$.
\end{cor}
\begin{rem}\label{rem:B2C2commcentr}
    Corollary \ref{cor:Zcommutmax} implies that $\Zsubalgchar^+_z$ in type $\LT{B}_2=\LT{C}_2$ is commutative for any $\kay$ and any $z\in\wordsetmax$\,. However, it is central only if $\kay$ is not a multiple of 4\,. 
\end{rem}
It is immediate from Proposition~\ref{prop:ZPBWbases} that, if $\Zsubalgchar^+_w$ is strictly commutative, it is isomorphic to a free polynomial algebra, assigning each $X_v$ with $v\leqRB w$ to an in indeterminate $x_\alpha$\,, where $\alpha=\wordroot(v)\in\descroots{s}\,$ and $s=\Weylpres(w)\,$. That is,
\begin{equation}\label{eq:Z=Polyn}
    \Zsubalgchar^+_w\,\cong\,\Zgen\left[\{x_\alpha:\alpha\in\descroots s\}\right]\,.
\end{equation}

Comparing the criteria in Propositions~\ref{prop:Zcentral} and Corollary~\ref{cor:Zcommutmax}, we note that there are various situations in which the subalgebras are commutative but not central. Given a maximal $w\in\wordsetmax$\,, $\Zsubalgchar^+_w$ has this property if $\kay\equiv 0\mod 8$ for any Lie type 
as well as if $\kay\equiv 4\mod 8$ and the Lie type is 
neither $\LT{B}$ nor $\LT{F}\,$. So, for example, in the simply laced cases $(\LT{ADE})$ and $\kay\equiv 0\mod 4$
we have that $\Zsubalgchar^+_w$ is commutative as in \eqref{eq:Z=Polyn} but for any $b\in \Uq$ with $\mu=\ztgradp(b)$ and $\alpha=\wordroot(v)$ we have
\begin{equation}
 X_vb=(-1)^{\symbrack{\alpha}{\mu}}bX_v\,.    
\end{equation}

\subsection{\texorpdfstring{$\Kinvaut$}{ʊ}-Invariance and Word Independence of \texorpdfstring{$\Zsubalgchar_w$\,}{Zw}}\label{subsec:Mho-Invar-Z} 
The definition of the algebras $\Zsubalgchar^\utypechar_w$  in 
Section~\ref{subsec:Z=SignPolyn} \textit{a priori} depends on a choice of a word $w\in\wordset\,$. In this section, we show for Lie types different from $\LT{G}_2$ that $\Zsubalgchar^\utypechar_w$ depends 
only on the respective Weyl element $s=\Weylpres(w)\,$, employing the techniques from Section~\ref{subsec:spanorder}. 

We begin with the derivation of formulae for the action of $\mho$ on generators in $\Uq$\,, starting with  the generic case and coefficients in $\mathbb Q(q)\,$.
We continue here the  conventions from Section~\ref{subsec:Ecommrels}, 
with $i$ the index of a short root $(d_i=1)$
and $j$ the index of a (possibly) long root $(d_j=\edgenum)$.
Observe that specializing the left side of \eqref{eq:EEEresum-j} to $z=-1$ and $\mathscr{u}_j(\vec n)=0$ yields the expression for $E^{(N)}_{(ji)}$ in \eqref{eq:EdivpwSimpExp-a} multiplied by $(-1)^N\,$. Similarly, for $z=-1$ and $\mathscr{u}_i(\vec n)=0$\,, 
the left side of \eqref{eq:EEEresum-i} yields the expression for $E^{(N)}_{b_{m-1}}$ in \eqref{eq:EdivpwSimpExp-a} multiplied by $(-1)^{\edgenum N}\,$. The products over the respective $(1-q_\alpha^{-2t})=q_\alpha^{-t}(q_\alpha-q_\alpha^{-1})[t]_\alpha$ terms are then expressed as factorials, resulting in the following formulae.  
\begin{align}
    E^{(N)}_{(ji)}\,&=\,(-1)^N\longsum[15]_{\substack{\phi\in S_{\edgenum}(\vec \sigma)\\ k_m=\phi(\alpha_j)}}q^{f_\edgenum(\phi)}(q_j-q_j^{-1})^{k_m}q_j^{-\binom{k_m+1}{2}}[k_m]!_j\Ebaseopp{\longtwoword{ij}}{(\phi)}
&\quad \vec \sigma &=N s_j(\alpha_i)=N\alpha_i+N\alpha_j
\label{eq:EdivRexpandJ} 
\\
\rule{0mm}{7mm}
E^{(N)}_{b_{m-1}}\,&=\,(-1)^{\edgenum N}\longsum[15]_{\substack{\phi\in S_{\edgenum}(\vec \tau)\\ k_1=\phi(\alpha_i)}}q^{f_\edgenum(\phi)}(q-q^{-1})^{k_1}q^{-\binom{k_1+1}{2}}[k_1]!\Ebaseopp{\longtwoword{ij}}{(\phi)}&\quad \vec \tau&=N s_i(\alpha_j)=\edgenum N\alpha_i+N\alpha_j\label{eq:EdivRexpandI}
\end{align}

Multiplying these expressions by $[N]!$ and $[N]!_j\,$, respectively, we obtain expressions for ordinary powers in terms of  elements $\qfactcollI{\phi}{\edgenum}$ and $\qfactcollJ{\phi}{\edgenum}$\,, obtained by combining $q$-factorials. 
\begin{align}
    E^{N}_{(ji)}\,&=\,(-1)^N\longsum[15]_{\substack{\phi\in S_{\edgenum}(\vec \sigma)\\ k_m=\phi(\alpha_j)}}q^{f_\edgenum(\phi)}(q-q^{-1})^{k_m}q_j^{-\binom{k_m+1}{2}}
    \qfactcollJ{\phi}{\edgenum} 
    \Ebaseopp{\longtwoword{ij}}{\phi}  
    \label{eq:EordRexpandJ}
    \\
\rule{0mm}{7mm}    
E^{N}_{b_{m-1}}\,&=\,(-1)^{\edgenum N}\longsum[15]_{\substack{\phi\in S_{\edgenum}(\vec \tau)\\ k_1=\phi(\alpha_i)}}q^{f_\edgenum(\phi)}(q-q^{-1})^{k_1}q^{-\binom{k_1+1}{2}}
\qfactcollI{\phi}{\edgenum} 
\Ebaseopp{\longtwoword{ij}}{\phi}
\label{eq:EordRexpandI}
\end{align} 
In the case of \eqref{eq:EordRexpandJ} we have $\mathscr u_j(\vec\sigma)=0$\,, which implies $k_3=k_1$ for $\edgenum=1$\,,
$k_4=k_1+k_2$ for $\edgenum=2$\,, and $k_6=k_1+2k_2+k_3+k_4$ for $\edgenum=3\,$. The ratios of factorials can readily be
worked out as follows, using the elements in \eqref{eq:factorialratios} and identities in \eqref{eq:fracfactorials}.
\begin{equation}\label{eq:factcomb-J}
\begin{aligned}
    &\qfactcollJ{\phi}{1}=\qbin{k_1+k_2}{k_1}{} \hspace*{13mm}
    \qfactcollJ{\phi}{2}=[2]^{k_2+k_4}\qfacrelA{k_2}{2}\cdot \qbin{k_1+2k_2+k_3}{k_1\,,\, 2k_2\,,\,k_3 }{}
     \\
    \rule{0pt}{23pt}
    &\qfactcollJ{\phi}{3}=[3]^{k_2+k_4+k_6}\cdot \qfacrelA{k_2}{3}\cdot \qfacrelA{k_4}{3}\cdot [k_3]!\cdot \qbin{2k_3}{k_3}{}\cdot 
    \qbin{k_1+3k_2+2k_3+3k_4+k_5}{k_1\,,\,3k_2\,,\,2k_3\,,\,3k_4\,,\,k_5}{}
\end{aligned}
\end{equation}

The multinomial coefficients are as defined in \eqref{eq:def-qmultinom}. 
Similarly, for \eqref{eq:EordRexpandJ}, the condition $\mathscr u_i(\vec\tau)=0$
implies $k_1=k_3$ for $\edgenum=1$,
$k_1=k_3+2k_4$ for $\edgenum=2$, and $k_1=k_3+3k_4+2k_5+3k_6$ for $\edgenum=3\,$. 
As before, with notation from \eqref{eq:factorialratios} and \eqref{eq:fracfactorials},
we find analogous formulae.
\begin{equation}\label{eq:factcomb-I}
\begin{aligned}
    &\qfactcollI{\phi}{1}=\qbin{k_2+k_3}{k_3}{}  \hspace*{13mm}
    \qfactcollI{\phi}{2}=[2]^{-k_3}\cdot \qfacrelB{k_3}{2}\cdot \qbin{k_2+k_3+k_4}{k_2\,,\, k_3\,,\,k_4 }{j}
     \\
    \rule{0pt}{23pt}
    &\qfactcollI{\phi}{3}=[3]^{-(k_3+k_5)}\cdot \qfacrelB{k_3}{3}\cdot \qfacrelB{k_5}{3}\cdot [k_4]!_j\cdot \qbin{2k_4}{k_4}{j}
    \cdot \qbin{k_2+k_3+2k_4+k_5+k_6}{k_2\,,\,k_3\,,\,2k_4\,,\,k_5\,,\,k_6}{j}
\end{aligned}
\end{equation}

Clearly, all coefficients in \eqref{eq:factcomb-J} and \eqref{eq:factcomb-J} lie in 
$\Zqqn{\dpone}$ as in \eqref{eq:defgenrings}. It is easy to check that in a corresponding
relation for rescaled generators $\mathring  E_\alpha=[d_\alpha]E_\alpha$\,,
already mentioned at the end of Section~\ref{subsec:PBWspec}, the 
 factors of the form $[\edgenum]^{-k}$  drop out so that  the coefficients would be entirely in $\Zqq\,$. 

In either case, we may specialize the identity to a root of unity $\zeta$, assuming that $\ell>\edgenum\,$ where
$\ell$ is the order of $\zeta^2\,$. This allows us, in particular, to derive 
expansions for $\Epw_{(ji)}=E^\ell_{(ji)}$ and $\Epw_{b_{m-1}}=E^{\ell_j}_{b_{m-1}}$
in terms of the $\Ebaseopp{\longtwoword{ij}}{(\phi)}$ basis elements. The number of 
non-zero terms will depend in either case on whether $\ell=\ell_j$ or $\ell>\ell_j\,$, that is,
whether $\gcd(\ell,\edgenum)=1$ or not. To keep subsequent formulae more compact, it is  convenient to introduce 
\begin{equation}\label{eq:gcdle}
    \gcdle=\gcd(\ell,\edgenum)\,\in\,\{1,\edgenum\}
\end{equation} 
for given $\ell$ and $\edgenum$. Using the Kronecker notation, we can then express the evaluation of
the factorial ratio from \eqref{eq:factorialratios} at $q=\zeta\,$ as 
\begin{equation}\label{eq:factratformZ-I}
    \qfacrelB{\ell_j}{\edgenum}\,=\,\delta_{\gcdle,1}\cdot \edgenum\cdot  \zeta^{-(\edgenum-1)\binom{\ell+1}{2}}\,,
\end{equation} 
implying that $\qfacrelB{\ell_j}{\edgenum}=0$ if $\elln_j<\ell\,$. The other factorial ratio 
from \eqref{eq:factorialratios} can be evaluated at $q=\zeta\,$ using \eqref{eq:Dfactell}.
\begin{lem}\label{lm:factratformZ-J}
Suppose $\ell=\edgenum\cdot\ell_j$\,. Then
 $$
\qfacrelA{\ell_j}{\edgenum}= \edgenum\cdot\zeta^{-\binom{\edgenum}{ 2}\ell_j^2}
\cdot (\zeta^{-1}-\zeta)^{-(\edgenum-1)\ell_j}\,.
$$   
\end{lem}

Another quantity that frequently appears in the computations below is the  expression 
\begin{equation}\label{eq:def-delzeta}
    \delzeta{j}\coloneqq(\zeta_j-\zeta_j^{-1})^{\ell_j}=[\edgenum]^{\ell_j}(\zeta-\zeta^{-1})^{\ell_j}\,.
\end{equation} 
This is the singularization factor of $\Epw_j$ in the sense of \eqref{eq:def:EgenSing}, as we introduce   the analogous notation  $\breve\Epw_j=\Esing^\ellj_j=(-1)^\ellj\delzeta{j}E_j^\ellj$.
Notice also that $\delzeta{j}^\gcdle=\delzeta{i}\cdot[\edgenum]^\ell$.

The next two propositions consider the specializations of \eqref{eq:EordRexpandJ} and \eqref{eq:EordRexpandI} to a root of unity $\zeta$ with $N=\ell$ and $N=\ell_j\,$, respectively. 
These yield explicit expressions for $\Epw_{(ji)}=E_{(ji)}^\ell$ and  $\Epw_{b_{m-1}}=E_{b_{m-1}}^{\ell_j}\,$. 
\begin{prop}\label{prop:XJ-form} We have $\Epw_{(ji)}\in\Zsubalgchar^+_{\longtwoword{ij}}$ as a ring over
$\mathbb Z[\zeta]$, given by 
\begin{equation}\label{eq:XJ-form}
\begin{aligned}
   (-1)^\ell\Epw_{(ji)}\,&= \,\zeta_j^{\ell}\Epw_{a_{m-1}}\,+\,
\zeta_j^{\binom{\ell}{2}}\delzeta{j}^{\gcdle}\Epw_j^\gcdle\Epw_i\,+\,
(1-\delta_{\gcdle,1})\cdot\edgenum\cdot(\zeta-\zeta^{-1})^{-\ell} \cdot T_1\vspace*{8mm}\\
\mbox{where} & \hspace*{24mm} 
     T_1= 
     \begin{cases}
    (-\zeta)^{\ell_j}\delzeta{j}^2 \Epw_j\Epw_{(ij)} & \edgenum=2\vspace*{2.5mm}\\
    (-\zeta_j)^{-\ell_j^2}\delzeta{j}^3\Epw_j^2\Epw_{(ij)}
    \,+\,
    \zeta_j^{-\binom{\ell_j+1}{ 2}}\delzeta{j}^2\Epw_j\Epw_{(ijij)}& \edgenum=3\,.
    \end{cases}
\end{aligned}
\end{equation}
\end{prop}

The computation starts with the observation that, by Corollary~\ref{cor:zeros-mulitnom}, we only need to consider cases in which one of the terms in the large multinomial coefficient for each $\qfactcollJ{\phi}{\edgenum}$ in
\eqref{eq:factcomb-J} is equal to $\ell$ and all the other $k_s$ with $1\leq s\leq m-1$ are zero. The value of $k_m$ is then determined via 
$\mathscr u_j(\phi)=0$ and \eqref{eq:ulf-j} from the non-zero $k_s\,$.

The first term in \eqref{eq:XJ-form} results from the case $k_{m-1}=\ell$ and all other $k_s=0$. The second term accounts for the case $k_1=k_m=\ell$ and all other $k_s=0$. Note also that $2k_3=\ell$ for $\edgenum=3$ will still lead to $\qfactcollJ{\phi}{\edgenum}=0$ due to the extra binomial term.

All other cases require that $\edgenum$ divides $\ell$, meaning $\gcdle=\edgenum>1\,$ and $\ell=\edgenum\ell_j$\,. For $\edgenum=2$ 
the only remaining situation is then $k_2=k_4=\ell_j$ and $k_1=k_3=0$. 
For $\edgenum=3$ we have two cases, the first of which is
$k_2=\ell_j$\,, $k_6=2\ell_j$ and $k_1=k_3=k_4=k_5=0$. The second is
given by $k_4=\ell_j$\,, $k_6=\ell_j$ and $k_1=k_2=k_3=k_5=0$.

The coefficients for each of these values of $k_s$ are then worked out using Lemma~\ref{lm:factratformZ-J}. Next, we provide the expressions for $\Epw_{b_{m-1}}$\,, where we use the additional notation
\begin{equation}\label{eq:gcdleopp}
\gcdleopp=\frac {\edgenum}{\gcdle}\qquad\mbox{so that}\qquad \gcdleopp\cdot\ell=\edgenum\cdot \ell_j\,.
\end{equation} 
Note that the third term in \eqref{eq:XI-form} is non-zero only if $\edgenum>1$ and $\gcd(\ell,\edgenum)=1$. As remarked in Section~\ref{subsec:QnumsRingsRo1} this condition 
implies that $[\edgenum]$ is a unit, meaning $[\edgenum]^{-1}\in\mathbb Z[\zeta]$. Analogous to 
\eqref{eq:def-delzeta} we also write
$\delzeta{i}=(\zeta-\zeta^{-1})^\ell$.

\begin{prop}\label{prop:XI-form} We have $\Epw_{b_{m-1}}\in\Zsubalgchar^+_{\longtwoword{ij}}$ as a ring over
$\mathbb Z[\zeta]$, given by 
\begin{equation}\label{eq:XI-form}
\begin{aligned}
   (-1)^{\edgenum\ell_j}\Epw_{b_{m-1}}\,&= \,\zeta_j^{\ell_j}\Epw_{(ij)}\,+\,
\zeta^{\binom{{\gcdleopp\cdot\ell}}{2}}\delzeta{i}^{\gcdleopp}\Epw_j\Epw_i^\gcdleopp\,+\,
\delta_{\gcdle,1}\cdot(1-\delta_{\edgenum,1})\cdot\edgenum\cdot [\edgenum]^{-\ell}\cdot T_2\vspace*{8mm}\\
\mbox{where} & \hspace*{24mm} 
     T_2= 
     \begin{cases}
    \zeta^{\ell}\delzeta{i} \Epw_{(iji)}\Epw_{i} & \edgenum=2\vspace*{3mm}\\
    \zeta^{-\binom{\ell}{ 2}}\delzeta{i}\Epw_{(iji)}\Epw_i\,
    +\,\zeta^{\ell(\ell+1)}\delzeta{i}^2\Epw_{(ijiji)}\Epw_{i}^2
    & \edgenum=3\,.
    \end{cases}
\end{aligned}
\end{equation}
\end{prop}

As for Proposition~\ref{prop:XJ-form}, the only contributing terms are those for which exactly one of 
the integers in the multinomial expressions of \eqref{eq:factcomb-I} is $\ell_j$ and all others are zero, once we set $N=\ell_j\,$ and $q=\zeta$ in \eqref{eq:EordRexpandI}. The first term in \eqref{eq:XI-form} accounts for the case when $k_2=\ell_j$\,, implying by $\mathscr u_i=0$ that all other $k_s=0$. The second term 
corresponds to the case $k_{m}=\ell_j$\,, which entails $k_1=\edgenum\ell_j=\gcdleopp\ell$ and all remaining $k_s=0\,$. 

The last terms for $\edgenum>1$ all involve coefficients $\qfacrelB{\ell_j}{\edgenum}$, which, by
\eqref{eq:factratformZ-I}, are non-zero only if $\gcdle=1\,$ and hence $\ell=\ell_j\,$. The extra term for $\edgenum=2$ corresponds to $k_3=k_1=\ell_j=\ell$ and all other indices zero. As before, for $\edgenum=3$ we only need to consider the cases
$k_3=\ell_j=\ell$ and $k_5=\ell_j=\ell$, which imply $k_1=\ell$ and $k_1=2\ell$, respectively.  

Comparing Propositions~\ref{prop:XJ-form} and \ref{prop:XI-form}, we note that additional terms and exponents occur in the non-simply laced cases in exactly the opposite situations with respect
to whether $\ell$ and $\edgenum$ are coprime or not. In addition, a term by term matching, up to coefficients, can also be achieved by switching an $\Epw_{a_j}$ by an $\Epw_{a_{m+1-j}}\,$. This 
can be thought of replacing the underlying root system by a corresponding coroot system as in \eqref{eq:wgradX_rvscor}, a phenomenon that  will resurface in Section~\ref{subsec:Coalg-ZB2}.

We discuss a normalization of the generators $\Epw_\alpha$ in Appendix \ref{sec:integralcoalg} for which the above formulae are integral. The main application of \eqref{eq:XJ-form} and \eqref{eq:XI-form} word independence of the skew-central subalgebras.

\begin{thm}\label{thm:Zwordindep}
Suppose $\roots$ is not of Lie type $\LT{G}_2$\,. Assume $w, u\in\wordset$ are reduced words for the same Weyl element $s=\Weylpres(w)=\Weylpres(u)$. Then $\Zsubalgchar^+_w=\Zsubalgchar^+_{u}$ as $\wgrad$-graded subalgebras in $\Uz^+$ over $\Zzn{\dpone}\,$. 
\end{thm}

\begin{proof} By Proposition~\ref{prop:ZPBWbases} $\Zsubalgchar^+_w$ is as a $\Zzn{\dpone}$-module given by 
$\bspanopp{w}{{\Lexpset s},+}$. We note that $\Lexpset s$, as defined in (\ref{eq:Lexpsetdef}), 
is of the form prescribed in 
(\ref{eq:VexpSetEx}) by choosing $\NexpSet{d}=\ell_d\cdot \nnN$ with $\ell_d=\ell/\gcd(d,\ell)\,$.

Following Corollary~\ref{cor:BaseOrdSplit} it thus suffices to show 
that $\Kinvaut$ maps $\Zsubalgchar^+_{\longtwoword{ij}}\!=\bspanopp{\longtwoword{ij}}{{\Lexpset {\longtwoweyl{ij}}},+}\!\!$
to itself for any pair 
$(i,j)$ that occurs for relator pairs.
Now, $\bspanopp{\longtwoword{ij}}{{\Lexpset s},+}$ is spanned by monomials in the skew-commuting generators $\Epw_{a_s}$. Note also that by \eqref{eq:KinvEtworefl} and \eqref{eq:Eabcompl} and with words as in \eqref{eq:def:abwords} we also have
\begin{equation}\label{eq:Xr2=mho}
    \Kinvaut(\Epw_{a_s})=\Epw_{b_{m+1-s}}\,. 
\end{equation}
 
Thus $\Kinvaut$-invariance of the space $\bspanopp{\longtwoword{ij}}{{\Lexpset s},+}$ is given by showing that it contains all
$\Epw_{b_k}$. This is obvious for $\Epw_{b_1}=X_j$ and $\Epw_{b_{m}}=\Epw_i\,$, implied, for example, by \eqref{eq:Edictsimple}.  
For $\Epw_{b_2}=\Epw_{(ji)}$ this is verified in Proposition~\ref{prop:XJ-form} and for 
$\Epw_{b_{m-1}}$ in Proposition~\ref{prop:XI-form}. Thus, if $m\in\{3,4\}$ all cases are covered. 

By construction any $\Zsubalgchar^+_w$ inherits the $\wgrad$-grading from $\Uz$ since it admits a $\wgrad$-graded basis. The graded components of $\Zsubalgchar^+_w$ and 
$\Zsubalgchar^+_u$ thus also have to coincide.  
 \end{proof}

Missing from the proof for the $\LT{G}_2$ Lie type are computations that show also $\Epw_{b_3}$ and $\Epw_{b_4}$ 
can be written as a polynomial expression in the $\{\Epw_{a_i}\}$ only, meaning they are in  
$\bspanopp{\longtwoword{ij}}{{\Lexpset {\longtwoweyl{ij}}},+}\!$. We leave this as a question, which appears to be open, at least for even $\kay\,$.
\begin{conj}\label{conj:Zwordindep}
    Theorem~\ref{thm:Zwordindep} extends also to the $\LT{G}_2$ case for all $\kay>6$. 
\end{conj}

For Lie types different from $\LT{G}_2$\,, we may thus unambiguously 
write $\Zsubalgchar^+_s$ for an element in $s\in\Weyl$. In the $\LT{G}_2$ case,
$\Zsubalgchar_{s}^{+}$ is, of course, also well-defined for $\len{s}\leq 5$,
since there is only one reduced expression for such $s\,$. For the longest element 
$\longweyl\,$, there are  exactly two reduced expressions. So,
the only missing equality between subalgebras  addressed by this  conjecture is  
the one between 
$\Zsubalgchar_{\longtwoword{ij}}^{+}$ and $\Zsubalgchar_{\longtwoword{ji}}^{+}\,$.

Assuming the other Lie types, we denote the maximal subalgebras associated to  
the unique longest element $\longweyl\in\Weyl$ as 
\begin{equation}\label{eq:Zmax}
  \Zsubalgchar^{\pm}_\bullet=\Zsubalgchar^{\pm}_{\longweyl}\,,
  \quad\Zsubalgchar^{\geqzero}_\bullet=\Zsubalgchar^{+}_{\bullet}\cdot\Zsubalgchar^{0}_{\bullet}\,,
  \quad\Zsubalgchar^{\leqzero}_\bullet=\Zsubalgchar^{-}_{\bullet}\cdot\Zsubalgchar^{0}_{\bullet}\,,
  \,\qquad\mbox{and}\qquad \Zsubalgchar_\bullet=\Zsubalgchar^{+}_\bullet\cdot\Zsubalgchar^{0}_\bullet\cdot\Zsubalgchar^{-}_\bullet\,,
\end{equation}
where $\Zsubalgchar^{0}_{\bullet}$ is the full subalgebra generated by the $\Kpw_i\,$. 
Also immediate from the definition is that 
\begin{equation}\label{eq:Zposet}
  \Zsubalgchar^+_t\subseteq\Zsubalgchar^+_s \qquad \forall s,t\in\Weyl\;\mbox{ with }\; t\leqRB s\,. 
\end{equation}

Analogous statements hold for the $\Zsubalgchar^\utypechar_s$ algebras with $\utypechar\in\{\phantom{+},-,\geqzero,\leqzero\}$. Note also that Corollary~\ref{eq:InvolZalg} and Theorem~\ref{thm:Zwordindep} imply for these Lie types  
\begin{equation}\label{eq:InvolZpm}
\Cartaninv(\Zsubalgchar_s^{\pm})=\Zsubalgchar_s^{\mp}\,,\qquad 
    \Kinvaut(\Zsubalgchar_\bullet^{\pm})=\Zsubalgchar_{\bullet}^{\pm}\,,\qquad 
    \mbox{and}
    \qquad
    S(\Zsubalgchar_\bullet^{\geqzero})=\Zsubalgchar_\bullet^{\geqzero}\;.     
\end{equation}
Combining this with $\Kinvaut$-correspondence with Lusztig's conventions from 
Proposition~\ref{prop:mhomonomial} we make the following observations. 
\begin{cor}\label{cor:Zorderindep}
The maximal algebras in \eqref{eq:Zmax} remain unchanged if the $\Tinv_i$ are replaced by 
$T_i$ in the definition of the generators $E_w$ and $\Epw_w\,$. 
\end{cor}

 We conclude this section with a proof that the full maximal algebra is invariant under the actions of Lusztig's $T_i$ automorphisms. It uses the explicit expression of the Garside element $\Tinv_{\longweyl}\,$ from Section~\ref{subsec:GarsAut+DictGens}. 

\begin{thm}\label{thm:ZTinvInvar}
Assuming Lie types other than $\LT{G}_2$ and any $\kay>2\edgenum$ (or $\kay=3$ for $\edgenum=2$), the full algebra  $\Zsubalgchar_\bullet$ over $\Zzvn{\dpone}$ is invariant under the
Artin group action generated by the $\Tinv_i\,$. 
\end{thm}

\begin{proof}
    Invariance of $\Zsubalgchar^0_\bullet$ under $\Tinv_i$ follows immediately from \eqref{eq:TinvL} and Proposition~\ref{prop:ZPBWbases}. It thus suffices to show $\Tinv_i(\Zsubalgchar^+_\bullet)\subset\Zsubalgchar_\bullet\,$ for each fixed $i\in\{1,\ldots ,n\}\,$,
    which implies $\Tinv_i(\Zsubalgchar^-_\bullet)\subset\Zsubalgchar_\bullet\,$ via 
    \eqref{eq:InvolZpm} and commutation of $\Cartaninv$ with $\Tinv_i\,$. 

    Let $u\in\wordset$ be a reduced word for $q_{\eta(i)}=s_i\cdot\longweyl=\longweyl\cdot s_{\eta(i)}$ as in \eqref{eq:defalt_complwords} so that $z=w_i\cdot u$ and $w=u\cdot w_{\eta(i)}$ are maximal reduced words
    $z,w\in\wordsetmax\,$. By Theorem~\ref{thm:Zwordindep} we can view $\Zsubalgchar^+_\bullet=\Zsubalgchar^+_w\,$, as
    generated by $\Epw_v$ with $v\leqRB w=u\cdot w_{\eta(i)}\,$. The latter generating set splits into the union
    $\{\Epw_v:v\leqRB u\}\sqcup\{\Epw_w\}$. If $v\leqRB u$ we have that $w_i\cdot v$ is reduced since 
    $w_i\cdot u$ is reduced. Hence $\Tinv_i(X_{v})=X_{w_i\cdot v}\in\Zsubalgchar^+_\bullet\subset \Zsubalgchar_\bullet\,$ by
    \eqref{eq:EFdefRecurs} and \eqref{eq:TinvXY}.

    Now, by definition, $\Epw_w=\Tinv_u(\Epw_{\eta(i)})$ so that $\Tinv_i(\Epw_w)=\Tinv_i\circ\Tinv_u(\Epw_{\eta(i)})=\Tinv_z(\Epw_{\eta(i)})=\Tinv_{\longweyl}(\Epw_{\eta(i)})$ since $z=w_i\cdot u$ is reduced and maximal. The formula \eqref{eq:GarsideAutomEi}
    in the proof of Corollary~\ref{cor:GarsideAutom} then implies that 
    $\Tinv_i(\Epw_w)=\Tinv_{\longweyl}(E_{\eta(i)}^{\ell_i})=(-K_i^{-1}F_i)^{\ell_i}=-\Kpw_i^{-1}\Fpw_i\in\Zsubalgchar^{\leqzero}_\bullet\subset \Zsubalgchar_\bullet\,$. Thus, $\Tinv_i$
    maps all generators of $\Zsubalgchar^+_w$ to $\Zsubalgchar_\bullet\,$ as desired.
\end{proof}

The action of the Artin groups on the maximal skew-commutative algebras is rather non-trivial. 
We leave it here as an open question whether $\mathscr A$ acts effectively or faithfully, as
suggested by the identifications with actions on Lie groups discussed in Sections~\ref{subsec:AnBruhat} and \ref{subsec:IsomHA-B2}. 

For illustration, we gather below explicit formulae for the Lie type $\LT{A}_2$ and odd $\kay\,$, with
$\delzeta{\;}=(\zeta-\zeta^{-1})^{\ell}$. The action of $\Tinv_1$ on the Cartan algebra, given by
$\Tinv_1(\Kpw_1)=\Kpw_1^{-1}$ and $\Tinv_1(\Kpw_2)=\Kpw_2\Kpw_1^{-1}$, factors through
the quotient $\mathscr A=B_3\rightarrow \Weyl=S_3\,$. The remaining expressions are
as follows.
\begin{equation}\label{eq:A2Tinv1Form}
\begin{aligned} 
     \rule{0mm}{4.7mm}
   \Tinv_1(\Epw_1)&=-\Kpw_1^{-1}\Fpw_1\
   \quad &
   \Tinv_1(\Epw_2)&=\Epw_{(12)} 
   \quad &
   \Tinv_1(\Epw_{(12)})&=-\Epw_2+\delzeta{\;}\Kpw_1^{-1}\Epw_{(12)}\Fpw_1\\
     \rule{0mm}{4.7mm}
   \Tinv_1(\Fpw_1)&=-\Kpw_1\Epw_1 
   &
   \Tinv_1(\Fpw_2)&=\Fpw_{(12)}
   &
   \Tinv_1(\Fpw_{(12)})&=-\Fpw_2-\delzeta{\;}\Kpw_1\Fpw_{(12)}\Epw_1
\end{aligned}   
\end{equation}

Those in the first two columns are immediate from (\ref{eq:TinvXYsame}) and (\ref{eq:TinvXY}).
For those in the third column, we invoke the identity $\Epw_{(12)}+\Epw_{(21)}+\delzeta{\;}\Epw_{1}\Epw_{2}=0$ with $\delzeta{\;}=(\zeta-\zeta^{-1})^{\ell}$,
which is obtained by specializing Propositions~\ref{prop:XJ-form}. We also use 
the identity $\Tinv_1(\Epw_{(21)})=\Epw_2$\,, which is implied by Propositions~\ref{prop:wordgencont}. 
The action on $\Fpw_{(12)}$ follows via application of $\Cartaninv\,$. 

It is now not hard to find many infinite families of braids in $B_3$ for which the action in (\ref{eq:A2Tinv1Form}) on generators produces polynomials whose (na\"ive) degrees monotonously increase
with the crossing numbers of the braids.  
 
 \subsection{\texorpdfstring{$\Zsubalgchar$}{Z}-Induced Ideals, \texorpdfstring{$\Commsigngrp$}{G}-Invariance, and \texorpdfstring{$\Commsigngrp$}{G}-Grading}\label{subsec:Z_Ideals} 
In this section, we describe the general process of obtaining ideals for the full quantum groups from ideals of the skew-central subalgebras. To keep the exposition more readable, we will often
suppress the choice of a maximal word $z\in\wordsetmax$
in the notation of the maximal algebra $\Zsubalgchar_\bullet^{\utypechar}=\Zsubalgchar_z^{\utypechar}$.
By Theorem~\ref{thm:Zwordindep} this is unambiguous for Lie types other than $\LT{G}_2$\,. In the case
$\LT{G}_2$ it is tacitly understood that there may be (at most) two or four choices, depending on $\utypechar\,$.

 As skew-polynomial algebras, the $\Zsubalgchar^{\utypechar}_\bullet$  have a  rich ideal structure. Ideals in $\Zsubalgchar^{\utypechar}_\bullet$ extend to ideals in $\Uz^{\utypechar}\,$ provided they are commensurable with the action of an elementary abelian 2-group stemming from  commutation relations as in Proposition~\ref{prop:PwGenCommRels}.

More formally, the latter group is defined as  $\Commsigngrp=\mathbb F_2^{\sroots}\times\mathbb F_2^{\sroots}$\,, acting 
 on the algebra $\Zsubalgchar_w$ for any $w\in\wordset\,$. 
An element $(\widebarr \mu, \widebarr \nu)\in \Commsigngrp$ acts on generators in 
$\Zsubalgchar_w$ associated to some $v\leqRB w$ by multiplication with $\pm 1$ via the
formulae 
 \begin{equation}\label{eq:Gactdefgen}
     (\widebarr \mu, \widebarr \nu)\Commsignact\Epw_v=\commph{\wordroot(v)}{\widebarr{\mu}} \Epw_v
     \;,\quad 
     (\widebarr \mu, \widebarr \nu)\Commsignact\Fpw_v=\commph{\wordroot(v)}{\widebarr{\nu}} \Fpw_v
     \;,\quad \mbox{and}\quad
     (\widebarr \mu, \widebarr \nu)\Commsignact\Kpw_i=\commph{\alpha_i}{\widebarr{\mu}+\widebarr{\nu}} \Kpw_i\;.
\end{equation}
Since the $\{X_v,Y_v,L_i:v\leqRB w, 1\leq i\leq n\}$ are independent (sign-commuting) generators of $\Zsubalgchar_w$\,, 
the formulae in \eqref{eq:Gactdefgen} uniquely extend to a $\Commsigngrp$-action on $\Zsubalgchar_w$ as algebra automorphisms. 
The action clearly preserves all $\Zsubalgchar^{\utypechar}_w$ subalgebras and specializes to a
$\mathbb F_2^{\sroots}$-action on either $\Zsubalgchar^{+}_w$ or $\Zsubalgchar^{-}_w\,$ as one of the components acts trivially.

The commutation relations from Proposition~\ref{prop:PwGenCommRels} and \eqref{eq:Lbcommrel} can now 
be rephrased more compactly using the $\Commsigngrp$-action. Specifically, for a fully homogeneous element $b\in\Uq$  and
$W\in \Zsubalgchar_w$ we find
\begin{equation}\label{eq:GactCommrel}
    W\cdot b=b\cdot\ztgrad^*(b)\Commsignact W\;,
    \quad\mbox{where}\quad \ \ztgrad^*(b)=(\ztgradp (b), \ztgradn (b))\in\Commsigngrp\;,
\end{equation}
 by straightforward verification on the generators. Here we use also $\wgrad=\ztgradp+\ztgradn\mod 2$ and the extension to general elements via standard action properties. We say that a subspace $A\subseteq\Zsubalgchar^{\utypechar}_w$ is 
 {\em $\Commsigngrp$-invariant} if $g\Commsignact x\in A$ for all $g\in\Commsigngrp$ and $x\in A\,$. 
 
\begin{lem}
    The $\Commsigngrp$-action on $\Zsubalgchar^{\utypechar}_\bullet$ does not depend on a choice of $w\in\wordsetmax\,$. 
\end{lem}

If $W\in\Zsubalgchar_\bullet=\Zsubalgchar_{\longweyl}$ and $w,z\in\wordsetmax$ represent
$\longweyl$ denote by $(\bar u,\bar\nu)\Commsignact\!^w$ and $(\bar u,\bar\nu)\Commsignact\!^z$ the respective actions defined via \eqref{eq:Gactdefgen}.
Equation \eqref{eq:GactCommrel} then implies
$b\cdot\ztgrad^*(b)\Commsignact\!^z W=b\cdot\ztgrad^*(b)\Commsignact\!^w W$ for homogeneous elements $b$. Choosing respective PBW basis elements one may then infer
that the actions need to be the same.  Another immediate consequence of \eqref{eq:GactCommrel} is as follows.

 \begin{lem}\label{lem:Ginvar2Sid}
     If $J$ is a $\Commsigngrp$-invariant left or right ideal in $\Zsubalgchar^{\utypechar}_\bullet$, then $J$ is also a two-sided ideal.
 \end{lem}

 For example, suppose $J$ is left sided, $W\in J$, and $b\in\Zsubalgchar^+_\bullet$ is homogeneous. Then $W\cdot b=b\cdot W'$ where
 $W'=\ztgrad^*(b)\Commsignact W\in J$ by $\Commsigngrp$-invariance. Thus $W\cdot b=b\cdot W'\in J$ since $J$ is a left ideal and, hence, also a right ideal.
 It is not difficult to find counterexamples to Lemma~\ref{lem:Ginvar2Sid} if the
 condition of $\Commsigngrp$-invariance is dropped. 
 In the next proposition, we
 consider {\em algebra} ideals only, disregarding any coalgebra structures. 

 \begin{prop}\label{prop:Zideal}Suppose $J$ is a $\Commsigngrp$-invariant
 (any-sided) ideal in $\Zsubalgchar^{\utypechar}_\bullet$.  
 Then $\derideal J \utypechar w = J\cdot \Uz^{\utypechar}\,$ is a two-sided ideal in $\Uz^{\utypechar}$.
 Moreover, the multiplication map from \eqref{eq:ZSrem=S} induces an isomorphism of $\Zsubalgchar^{\utypechar}_\bullet$-modules, 
 \begin{equation}
     \left(\Zsubalgchar^\utypechar_\bullet/J\right)\otimes_\Zgen\bspan{z}{\Lmaxexpset s, +}\!\longrightarrow \,\Uz^{\utypechar}/\derideal J \utypechar w\,:\;
     ([z],b)\,\mapsto\,[z\cdot b]\,,
 \end{equation}
 where $[.]$ denotes the class of an element in the respective quotient algebra and $z\in\wordsetmax$ is a reduced word of maximal length. 
 \end{prop}

 \begin{proof}
    $\derideal J \utypechar w$ is a right ideal by definition. To see it is also a left ideal, write  general elements
    $B\in\Uz^{\utypechar}$ and $C\in J$ as
    $B=\sum_sb_s$ and $C=\sum_tW_tc_t$ where $b_s, c_t\in \Uz^{\utypechar}$ are fully homogeneous and $W_t\in J\,$. 
    \eqref{eq:GactCommrel} then implies $B\cdot C=\sum_{s,t}\ztgrad^*(b_s)\Commsignact (W_t)b_sc_t\,$. By 
    $\Commsigngrp$-invariance of $J$ we have $\ztgrad^*(b_s)\Commsignact (W_t)\in J$ so that 
    $B\cdot C\in \derideal J \utypechar w$. Thus $\Uz^{\utypechar}\cdot J\subseteq \derideal J \utypechar w$,
    which implies 
    $\Uz^{\utypechar}\cdot \derideal J \utypechar w = \Uz^{\utypechar}\cdot    J \cdot \Uz^{\utypechar}\subseteq 
    \derideal J \utypechar w\cdot \Uz^{\utypechar}\subseteq\derideal J \utypechar w$ as desired. 

    For the second assertion, it suffices to show that $\,\ZSisom^+_z\left(J\otimes_\Zgen \bspan{z}{\Lmaxexpset s, +}\!\right)=\derideal J \utypechar w\,$. The inclusion of the image in $\derideal J \utypechar z$ is clear from 
    the ideal property. 
    Surjectivity of the restriction of $\ZSisom^+_z$ from $J\otimes_\Zgen \bspan{z}{\Lmaxexpset s, +}\!$ to $\derideal J \utypechar w$ is immediate from the commutative diagram below, where the top horizontal arrow is an isomorphism and the two vertical arrows are surjections.
    $$
     \bfig
    \square/>->>`>>`>>`>/<1330,500>[
      J\otimes_\Zgen \Zsubalgchar^\utypechar_\bullet \otimes_\Zgen \bspan{z}{\Lmaxexpset s, +}`
      J\otimes_\Zgen \Uz^{\utypechar} `
      \quad J\otimes_\Zgen \bspan{z}{\Lmaxexpset s, +}`
      \derideal J \utypechar w;
      \;\id_J\otimes \ZSisom^+_z\;`\ZSisom_J\otimes_\Zgen \id`\ZSisom_J`\ZSisom^+_z]
\efig$$
\vspace*{-8mm}
\end{proof}

An important special family of $\Commsigngrp$-invariant ideals are those generated by augmentation ideals of the
subalgebra $\Zsubalgchar^{+}_w$ for some reduced word $w\in\Weyl$ and $\utypechar=+$, $\utypechar=\geqzero$, or $\utypechar$ omitted for the full algebra. They are defined as 
\begin{equation}\label{eq:defAugZideal}
    \Zaugideal{w}^{\utypechar}=\left(A_w^{+}\right)=A_w^{+}\cdot \Zsubalgchar_\bullet^{\utypechar}
\qquad \mbox{ where } \qquad 
A_w^{+}=\mathrm{ker}\left(\varepsilon:\Zsubalgchar^{+}_w\rightarrow\Zgen\right)\,.   
\end{equation}
It is clear from the PBW bases that 
$\Zaugideal{w}^{+}$ is identical to the ideal generated by $\{X_v:\emptyword\neq v\leqRB w\}\,$. 
Assume some $z\in\wordsetmax$ with $w\leqRB z$ and let $s=\Weylpres(w)\,$. 
It is then not hard to see  that, as a
free $\Zgen$-module, $\Zaugideal{w}^{+}$ is spanned by all $\Epwbase{z}{\psi}$ in $\basis{z}{{\Lexpset s},+}$ 
for which the restriction of $\psi$ to the inversion set $\descroots{s}\subseteq\proots$ is non-zero.

For a more concise description of the ideals in the full quantum algebras, we
introduce abbreviated notation for exponent subsets complementary to the $\Lmaxexpset s$ from (\ref{eq:Lmaxexpset}), namely, 
\begin{equation}\label{eq:LminExpSet}
    \Lminexpset s=\left\{\psi\in\expsetup s:\,\exists\alpha\in\descroots{s} \;\mbox{ for which }\;\psi(\alpha)\geq \ell_\alpha  \right\}
    \qquad
    \mbox{and}
    \qquad
    \Lminexpset \bullet=\Lminexpset {\longweyl}\;.
\end{equation}
It is clear from the isomorphism in (\ref{eq:ZSrem=S}) and with notation as in (\ref{eq:DefSspans}) that
the associated ideals in $\Uz^+$ may also be written as the two-sided ideals 
generated by the spans corresponding to these exponent sets. That is, 
\begin{equation}\label{eq:Kw=SpangeqL}
    \Zaugidealhat{w}^{+}\,=\,\Uz^+\cdot\bspan{w}{\Lminexpset s,+}\!=\,\Uz^+\cdot\bspanopp{w}{\Lminexpset s,+}\!.
\end{equation}
Observe that for a maximal word $z\in\wordsetmax$  we have $\Zaugideal{z}^{+}=A^+_z\,$, as defined in
(\ref{eq:defAugZideal}), and $\Zaugideal{z}^{\geqzero}=A^+_z\cdot\Zsubalgchar^0\,$. 
Correspondingly, (\ref{eq:Kw=SpangeqL}) reduces to 
$\Zaugidealhat{z}^{+}\,=\,\bspan{z}{\Lminexpset{\bullet},+}\!$. 

Note that
since the automorphisms from Corollary~\ref{cor:InvolZalg} preserve the counit, the analogs of
the identities in (\ref{eq:InvolZalg}) also hold for the respective ideals. Applied to the 
induced ideals in the quantum algebras one finds
\begin{equation}\label{eq:KhatInvolIds}
    \Cartaninv\bigl(\Zaugidealhat{w}^{\pm}\bigr)=\Zaugidealhat{w}^{\mp}\,,
\qquad
\Kinvaut\bigl(\Zaugidealhat{z}^{\pm}\bigr)=\Zaugidealhat{z^\winvchar}^{\pm}\,,
\quad\mbox{and}\quad
S\bigl(\Zaugidealhat{z}^{\geqzero(\leqzero)}\bigr)=\Zaugidealhat{z^\winvchar}^{\geqzero(\leqzero)}\,.
\end{equation}
 
Similarly, Theorem~\ref{thm:Zwordindep}
and the fact that the definition in \eqref{eq:defAugZideal} does not refer to choices of bases
imply the next remark.

\begin{cor}\label{cor:KidealWordIndep}
Assume the Lie type is not $\LT{G}_2$\,, $\kay$ as before, and $s=\Weylpres(w)=\Weylpres(w')$ for two reduced
words $w,u\in\wordset$. Then $\Zaugideal{w}^{\utypechar}=\Zaugideal{u}^{\utypechar}\,$.
\end{cor}
We may thus write $\Zaugideal{s}^{\utypechar}$ for Lie types different from $\LT{G}_2\,$. 
We then adopt the abbreviated notation for a maximal reduced word $z\in\wordsetmax\,$
\begin{equation}\label{eq:MaxKidealNots}
    \ZmaxId{\bullet}^\utypechar=\Zaugideal{z}^\utypechar =\Zaugideal{\longweyl}^\utypechar 
       \qquad\mbox{and}\qquad
     \ZmaxIdHat{\bullet}^\utypechar=\Zaugidealhat{z}^\utypechar=\Zaugidealhat{\longweyl}^\utypechar\,,
\end{equation} 
where the first equalities also apply to the $\LT{G}_2$ case (for which $\ZmaxId{\bullet}^\utypechar$ may be 
ambiguous). Furthermore,
$\mathscr A$-invariance of $\Zsubalgchar_\bullet$\,, as asserted in Theorem~\ref{thm:ZTinvInvar}, the recursion \eqref{eq:TinvL}, and the basis
generators for $\Zaugideal{s}$ imply the following recursion for ideals. 

\begin{lem} Suppose $\Zaugideal{s}$ is the ideal of the {\em full}  algebra $\Zsubalgchar_\bullet$\,, $\len{s\cdot t}=\len{s}+\len{t}$ for $s,t\in\Weyl$, and Lie type $\neq\LT{G}_2\,$. Then\vspace*{-9mm}

    \begin{equation}
    \Zaugideal{s\cdot t}=\Zaugideal{s}+\Tinv_s(\Zaugideal{t})\,. 
\end{equation} 
\end{lem}

Analogous augmentation ideals $\Zsubalgchar_\bullet^\leqzero$ and $\Zsubalgchar_\bullet^-$ are defined in  using
$A_s^{-}=\mathrm{ker}\left(\varepsilon:\Zsubalgchar^{-}_s\rightarrow\Zgen\right)$ as generators. For a pair $s,t\in\Weyl$, combined ideals in the full algebra $\Zsubalgchar_\bullet$ are given
as
\begin{equation}\label{eq:defAugZidealComb}
    \Zaugideal{s,t}\,=\,\left(A_s^{+},A_t^{-}\right)\,=\,(A_s^{+}+A_t^{-})\Zsubalgchar_\bullet\;. 
\end{equation}

Observe also that $A_1^{\pm}=0$ and than \eqref{eq:InvolZpm} implies $\Cartaninv(A^+_s)=A^-_s\,$, which yields the identity
\begin{equation}\label{eq:Ks1Ids}
    \Zaugideal{s}=\Zaugideal{s,1}=\Cartaninv(\Zaugideal{1,s})\;.
\end{equation}
For later convenience, we introduce  compact notation for the ideals in $\Zsubalgchar_\bullet$ corresponding to the longest element $\longweyl\in\Weyl$ as 
\begin{equation}\label{eq:ZidealMaxNot}
\ZmaxId{\siAN}=\Zaugideal{\longweyl,1}\,,\qquad \ZmaxId{\siNA}=\Zaugideal{1,\longweyl}\,,
\quad\mbox{and}\quad
\ZmaxId{\siAA}=\Zaugideal{\longweyl,\longweyl}\,.
\end{equation}
Also for Lie types different from $\LT{G}_2$ (or keeping in mind the aforementioned possible 
ambiguity for $\LT{G}_2$) we use the suggested notation from Proposition~\ref{prop:Zideal} for the 
entailed ideals in $\Uz$ such as, for example,
\begin{equation}\label{eq:ZidealMaxNotHat}
    \Zaugidealhat{s}\,,\quad \Zaugidealhat{s,t}\,,\quad \ZmaxIdHat{\siAN}=\Zaugidealhat{\longweyl,1}\,,\quad \mbox{or}\quad 
    \ZmaxIdHat{\siAA}=\Zaugidealhat{\longweyl,\longweyl}\,.
\end{equation}

More generally, any subset $M\subset \basis{w}{{\Lexpset s},+}$ generates a $\Commsigngrp$-invariant ideal
$J=(M)$ of $\Zsubalgchar^+_w$ since all monomials in $\basis{w}{{\Lexpset s},+}$ are homogeneous in the 
sense that $g\Commsignact  \Epwbase{w}{\psi}=\pm \Epwbase{w}{\psi}\,$ for all $g\in \Commsigngrp\,$.
So, if $M=\left({\Epwbase{w}{\chi_1},\ldots, \Epwbase{w}{\chi_K}}\right)$ a PBW basis for $\Zsubalgchar^+_w/M$
is given by the set of all $\Epwbase{w}{\psi}$ for which $\psi\not\geq \chi_i$ for any $i=1,\ldots,K\,$. The 
partial order used here is defined as 
$\psi\geq\chi$ iff $\psi(\alpha)\geq\chi(\alpha)$ for all $\alpha\in\descroots{w}\,$. Consequently, a PBW basis
of $\Uz^+/\widehat{M}$ if given by the set of all $\Ebase{w}{\phi}$ for which $\phi\not\geq \ellfun\cdot\chi_i$ for any $i=1,\ldots,K\,$.

We briefly elaborate here on the caution expressed in the introduction to this monograph regarding quotients by only the 
$\Epw_i$ for simple roots. Let $M_{\sroots}$ be ideal in $\Zsubalgchar_z^+$ generated by the set 
$\{\Epw_i:i=1,\ldots,n\}=\{\Epw_{v}:\wordroot(v)\in\sroots\,,\;v\leqRB z\}$, see Proposition~\ref{prop:wordgencont}. Considering only the positive part, a PBW basis of $\Uz^+/\widehat M_{\sroots}$ is given by $\Epwbase{w}{\psi}$ with $\psi(\alpha_i)<\ell_i$ on simple roots 
but $\psi$ unbounded on $\proots\setminus\sroots\,$, so that $\Uz^+/\widehat M_{\sroots}$ is of {\em infinite} rank if $n>1\,$. 

Even more concretely, for Lie type $\LT{A}_2\,$ and $\kay=2\ell$, the algebra obtained by imposing only relations $E_1^\ell=0$ and 
$E_2^\ell=0$ (but not $E_{(12)}^\ell=0$) has PBW basis $\{E_1^{i_1}E_{(12)}^jE_2^{i_2}:\,0\leq i_1,i_2<\ell \mbox{ and } j\in\nnN\,\}\,$. The observations for the full quantum groups are analogous, leading to algebras without the desired properties of a restricted quantum group. 

We mention also a natural $\Commsigngrp$-grading $\ZYXgrad^*$ on any of the $\Zsubalgchar^\utypechar_w$ algebras, which provides a dual picture to $\Commsigngrp$-invariance and which is analogous to $\ztgrad^*$ on $\Uq\,$. Given $\alpha=\wordroot(v)$ 
we set $\,\ZYXgrad^*(\Epw_v)=(\bar\alpha,0)\,$ and $\,\ZYXgrad^*(\Fpw_v)=(0,\bar\alpha)\,$ where $\bar\alpha\in\mathbb F_2^{\sroots}$ is the image of the canonical lattice map $\proots\rightarrow \mathbb F_2^{\sroots}$ that assigns $\alpha_i\mapsto\alpha_i$\,. Similarly, we set 
$\ZYXgrad^*(\Kpw^\mu)=(\bar\mu,\bar\mu)\,$ for $\mu\in\mathbb Z^{\sroots}$ and $\Kpw$ as in \eqref{eq:DefLmu}. By Proposition~\ref{prop:ZPBWbases} the grading extends uniquely to all 
of $\Zsubalgchar^\utypechar_w$\,, yielding a decomposition
\begin{equation}\label{eq:ZGdecomp}
\Zsubalgchar^\utypechar_w\,=\,\bigoplus_{g\in\Commsigngrp}\left(\Zsubalgchar^\utypechar_w\right)_g\;.
\end{equation}
A pairing of $\Commsigngrp$ with itself is now given by the symmetric bihomomorphism
\begin{equation}
  \commphbigrchar:\Commsigngrp\times \Commsigngrp\longrightarrow \{+1,-1\}
    \qquad \mbox{with} \qquad 
\commphbigr{(\bar\mu,\bar\nu)}{(\bar\mu',\bar\nu')}=\commphbihom{\bar\mu}{\bar\mu'}\cdot \commphbihom{\bar\nu}{\bar\nu'}\,,
\end{equation}
where $\commphbihomchar$ is as defined in \eqref{eq:signbihom}. 

Suppose now 
that $\kay=2\ell\,$ and that $\roots$ is of oddly laced Lie type ($\LT{ADEG}$). Then
Lemma~\ref{lem:kappaprops} implies that \eqref{eq:Gactdefgen} can be expressed more compactly.
Namely, for any $g,h\in \Commsigngrp$ and $W\in \left(\Zsubalgchar^\utypechar_w\right)_g$ 
\begin{equation}
    h\Commsignact W= \commphbigr{g}{h}\cdot W\,. 
\end{equation} 
Assume, further, that the Lie type is $\LT{A}_{2k}$\,, $\LT{E}_6$\,, $\LT{E}_8$\,, or $\LT{G}_2$\,. For each of these, $\commphbihomchar$ and, hence, $\commphbigrchar$ are non-degenerate pairings of respective $\mathbb F_2$ lattices. So,
in these cases a submodule $J\subseteq \Zsubalgchar^\utypechar_w$ is $\Commsigngrp$-invariant if and only if it respects 
the decomposition in \eqref{eq:ZGdecomp}, meaning that for $J_g=J\cap (\Zsubalgchar^\utypechar_w)_g$ we have
\begin{equation}\label{eq:JGdecomp}
J\,=\,\bigoplus_{g\in\Commsigngrp}J_g\,.
\end{equation}

The decomposition in \eqref{eq:ZGdecomp} can be described in more concrete terms under further assumptions. 
Aside from the restriction to $\{\LT{A}_{2k}\,,\LT{E}_6\,,\LT{E}_8\,,\LT{G}_2\}\,$, suppose further $\ell$ is even ($\kay\equiv 0\mod 4$) and $w\in\wordsetmax$ is maximal so that, by Corollary~\ref{cor:Zcommutmax}, $\Zsubalgchar^\utypechar_w$ can be identified with a polynomial algebra as in \eqref{eq:Z=Polyn}. There is a natural epimorphism from the former with indeterminates labeled by all
positive roots to the polynomial algebra with inderminates given by only the simple roots.
$$
\textstyle 
\Zgen\left[\{x_\alpha:\,\alpha\in\proots\}\right]\,\rightarrow\,\Zgen\left[z_1,\ldots,z_n\right]\,:\;
x_\alpha\mapsto \prod_iz_i^{c_i}\quad \mbox{for } \; \alpha=\sum_ic_i\alpha_i\,. 
$$
Given $\mu\in\mathbb F_2^{\sroots}$ we say that a polynomial $f\in\Zgen\left[z_1,\ldots,z_n\right]$ has parity 
$\mu$ if 
$$
f(z_1,\ldots,z_{i-1},-z_i,z_{i+1},\ldots,z_n)=(-1)^{\mu_i}f(z_1,\ldots,z_n)\;\qquad \forall i=1,\ldots, n\,.
$$
Thus, a polynomial $p\in \Zgen\left[\{x_\alpha:\,\alpha\in\proots\}\right]$ represents an element 
in $(\Zsubalgchar_w^+)_{(\bar\mu,0)}$ with $\bar\mu=(\mu_1,\ldots,\mu_n)$ precisely if it is mapped to a polynomial in $\Zgen\left[z_1,\ldots,z_n\right]$ of
parity $\mu\,$, in which case we assign the parity $\bar\mu$ also to $p\,$ and well as the respective element in 
$\Zsubalgchar_w^+\,$.

Thus, for example for type $\LT{A}_2\,$, the element $X_{(12)}^3+X_1^5X_2$ is of parity $(1,1)$  in $(\Zsubalgchar_w^+)_{((1,1),0)}$, while the element $X_1+X_2$ is not of any fixed parity since 
$X_1\in(\Zsubalgchar_w^+)_{((1,0),0)}$ but $X_2\in(\Zsubalgchar_w^+)_{((0,1),0)}\,$. 
Thus, $(X_{(12)}^3+X_1^5X_2)$ is a $\Commsigngrp$-invariant ideal of $\Zsubalgchar_w^+$ but $(X_1+X_2)$ is not.
The smallest $\Commsigngrp$-invariant ideal the latter is contained in is $(X_1+X_2, 2X_1)\,$. 

Any finite collection of elements with fixed parity in $\Zsubalgchar_w^+$ thus generates a 
$\Commsigngrp$-invariant ideal. For a converse statement, suppose $\mathbb L$ is a Noetherian integral domain in which 2 is a unit, suppose $\mathscr j:\Zgen\rightarrow \mathbb L$ is a ring epimorphism, and let
$\Zsubalgchar_{w,\mathbb L}^+=\Zsubalgchar_w^+\otimes_{\mathscr j}\mathbb L\,$.
Then, in fact, {\em all} $\Commsigngrp$-invariant ideals of $\Zsubalgchar_{w,\mathbb L}^+$
are given as ideals generated by a finite collection of elements of fixed parity. 

\medskip

\section{Hopf Ideals for \texorpdfstring{$\Uz(\mathfrak{sl}_{n+1})$}{Uz(sl(n+1))} from Solvable Linear Algebraic Groups}
\label{sec:Ideals_An=GLn}

The $\LT{A}_n$ case is computationally more accessible than other Lie types, allowing 
explicit descriptions of the correspondence between the skew-central subalgebras discussed in the previous chapter and the (semi) classical coordinate rings of Borel subgroups. We start by providing 
basic background 
on linear algebraic groups and their coordinate rings in Section~\ref{subsec:AlgGrps-basics}.

In Section~\ref{subsec:CoalgAn} we construct specific generators of $\Uz(\mathfrak{g})$
for both $\mathfrak{g}=\mathfrak{sl}_{n+1}$ and $\mathfrak{g}=\mathfrak{gl}_{n+1}\,$. 
They are used in Section~\ref{subsec:Zn=CAn} to establish  
$\Zsubalgchar^{\geqzero}_\bullet$ as a Hopf subalgebra with explicit forms for 
the coproduct. The latter are used for identifications with centrally extended 
quantum matrix algebras,
whose complexifications for $\kay\not\equiv 2 \mod 4$, coincide with the coordinate rings 
of upper triangular matrices. 

Recall from Section~\ref{subsec:Z=SignPolyn} 
that for $\kay\equiv 0 \mod 4$ the algebra $\Zsubalgchar^{\geqzero}_\bullet$ is
commutative but {\em not} central in $\Uz(\mathfrak{g})$.
The remainder and Section~\ref{subsec:AnBruhat} discuss the construction of Hopf ideals from
subgroups and the relation of the $\Zaugideal{s}^{\geqzero}$ ideals from
\eqref{eq:defAugZideal} to respective Bruhat subgroups.

\subsection{Basic Background and Elementary Examples}\label{subsec:AlgGrps-basics}
We provide here some basic vocabulary and elementary examples as preparation and motivation for 
the correspondences between Hopf ideals and algebraic subgroups discussed in   
Sections~\ref{sec:Ideals_An=GLn}  and \ref{sec:Ideals_B2=SO5}. Most of the material reviewed here
can be found, for example, in the early chapters of \cite{Sp98} and \cite{Mi17}.

Suppose $G$ is any algebraic set
and $M$ some ring of functions on $G$. Further, assume $H$ is any subset of $G$ and $I$ some 
ideal in $M$. We denote   the {\em vanishing ideal} of $H$ in $M$ as well as the 
{\em vanishing locus} of $I$ in $G$ as 
\begin{equation}\label{eq:VanishingDefs}
     \VanIdeal(H)=\left\{f\in M:\,f(h)=0\;\forall h\in H\right\}
    \quad \mbox{ and }\quad
    \VanLocus(I)=\left\{g\in G: f(g)=0\;\forall f\in I\right\} \,,
   \end{equation}
respectively. We, trivially, have inclusions $H\subseteq \VanLocus(\VanIdeal(H))$ and 
$I\subseteq \VanIdeal(\VanLocus(I))$. Suppose $\mathbb k$ is an algebraically closed field,
$G=\mathbb k^m$, and $M=\mathbb k[x_1,\ldots,x_m]$ the ring of polynomials viewed as functions
on $G$. In this case, Hilbert's Nullstellensatz and the definition of the Zariski topology imply that 
\begin{equation}\label{eq:HilbertZariski}
    \VanIdeal(\VanLocus(I))=\radic{I}\quad \mbox{and} \quad\VanLocus(\VanIdeal(H))=\ZarClos{H}\,,
\end{equation}
where $\radic{I}$ denotes the radical of $I$ and $\ZarClos{H}$ the usual Zariski closure, establishing 
a one-to-one correspondence between Zariski closed subsets and radical ideals. This correspondence
generalizes from $\mathbb k^n$ to any algebraic set $G$. 

Recall  that a linear algebraic group is an affine algebraic set $G\subseteq \mathbb k^N$ with 
group operations that are morphisms of 
varieties. The coordinate ring on $G$, customarily denoted by
$\mathbb k[G]=\mathbb k[x_1,\ldots,x_N]/\VanIdeal(G)$, admits for each $g\in G$ a well-defined
algebraic evaluation map $\mathscr e_g:\mathbb k[G]\rightarrow \mathbb k$ with $\mathscr e_g([p])=p(g)\,$ where
$p$ is a polynomial in $\{x_i\}\,$.   

Importantly, these admit a natural Hopf algebra structure, see for example \cite[Ch 2]{Sp98}, \cite[Ch 3]{Mi17}. 
The coproduct $\Delta:\mathbb k[G]\rightarrow\mathbb k[G\times G]\cong \mathbb k[G]\otimes\mathbb k[G]$
is the dual of the product morphism on $G$ such that 
$(\mathscr e_a\otimes\mathscr e_b)\circ \Delta=\mathscr e_{ab}\,$ or, more informally,
$\Delta([p])(a,b)=p(ab)\,$. An antipode is obtained similarly from inversion on $G\,$.

As a central example, consider $\mathbb k[\mathrm{GL}(n,\mathbb k)]\cong \mathbb k[\{x_{ij}\},\tau]/\VanIdeal\,$, where 
$\VanIdeal=(\tau\cdot\det(x)-1)\,$. The presentation implies that any element $[p]\in\mathbb k[\mathrm{GL}(n,\mathbb k)]$
can be represented by a polynomial $p\in \mathbb k[\{x_{ij}\},\tau]$ in $N=n^2+1$ indeterminates. 
The coalgebra structure is readily worked out as 
\begin{equation}\label{eq:AlgGLnCoprod}
  \Delta([x_{ij}])=\sum_{k=1}^n[x_{ik}]\otimes [x_{kj}]\qquad \mbox{and}\qquad\Delta([\tau])= [\tau]\otimes[\tau] 
\end{equation}
and it is immediate that $\VanIdeal\,$ is a Hopf ideal. 

As usual, for a Hopf algebra $A$, denote the subgroup of {\em group-like} elements as $\GrLi{A}=\{a \in A: \Delta(a)=a\otimes a\}\,$. Applied to a coordinate ring $A=\mathbb k[G]$ this group is 
naturally identified with the algebraic character group $\AChar{G}=\mathrm{Hom}_{alg}(G,\mathbb k^\times)=\GrLi{\mathbb k[G]}$ 
consisting of group homomorphisms  $\sigma:G\rightarrow \mathbb k^\times\,$ that are also morphisms of algebraic sets. For algebraically closed $\mathbb k$ with $\mathrm{char}(\mathbb k)=0$ the character group 
of $\mathrm{GL}(n,\mathbb k)$ is infinite cyclic generated by the determinant function. That is,
\begin{equation}\label{eq:GLnChar}
    \AChar{\mathrm{GL}(n,\mathbb k)}=\mathbb Z_{(\mathrm{det})}\,.
\end{equation}
Closely related is the fact that $\mathrm{GL}(n,\mathbb k)$ is linearly reductive or, equivalently, that 
$\mathbb k[\mathrm{GL}(n,\mathbb k)]$ is cosemisimple for such $\mathbb k\,$.

Next, let $\delta:\mathbb k[\mathrm{GL}(n,\mathbb k)]\rightarrow\nnN$ be the function assigning the minimal degree of representatives $\delta(f)=\min\{\deg(p):f=[p]\}\,$, where $\deg(p)$ denotes the total degree of $p$ as a polynomial in the $n^2+1$ variables. 
Observe that the finite-dimensional linear subspaces $D_m=\{f\in \mathbb k[\mathrm{GL}(n,\mathbb k)]:\,\delta(f)\leq m\}$ form an ascending chain of subcoalgebras. Indeed, if $f=[p]$ and $p$ is of minimal degree $d=\delta(f)$, it is clear from the formulae in \eqref{eq:AlgGLnCoprod} that $\Delta(f)$ can be written as a summation of pure tensors of monomials, each of which with degree $\leq d\,$. 

It is a basic fact that for any linear algebraic group $H$ over $\mathbb k$ there exists a (Zariski) closed monomorphism 
$j_H:H\hookrightarrow\mathrm{GL}(n,\mathbb k)\,$ of algebraic groups for some $n\,$, e.g. \cite[2.3.7]{Sp98}.
Given such an embedding, we may define an ascending sequence of finite-dimensional linear spaces $D_m(H)=j^*_H(D_m)$, each of which is clearly a subcoalgebra of $\mathbb k[H]\,$. 

The next lemma establishes a one-to-one correspondence between Hopf ideals and (Zariski) closed algebraic subgroups in characteristic zero. It is easy to find examples in positive characteristic for which 
{\em \ref{item:ASG=HI:HI=radical}} does not hold.

\begin{lem}\label{lm:ASG=HI}
Let $G$ be a linear algebraic group over an algebraically closed $\mathbb k$ with $\mathrm{char}(\mathbb k)=0$.\vspace*{-1.5mm}

\begin{enumerate}[label=\roman*), leftmargin=12mm,]  
    \item If $J\subset \mathbb k[G]$ is a Hopf ideal, then $H=\VanLocus(J)$ is a  closed algebraic subgroup in $G$.
    \vspace*{1.3mm}\label{item:ASG=HI:VJ=grp}
    \item Further, $J$ is a radical ideal so that $J=\VanIdeal(H)=\VanIdeal(\VanLocus(J))\,$. 
    \vspace*{1.3mm}\label{item:ASG=HI:HI=radical}
    \item If $H\leq G$ is any subgroup, then $\VanIdeal(H)$ is a Hopf ideal in $\mathbb k[G]\,$. 
    \vspace*{1.3mm}\label{item:ASG=HI:VG=HI}
    \item In this case $\VanIdeal(H)=\VanIdeal(\ZarClos{H})\,$.
    \vspace*{1.3mm}\label{item:ASG=HI:ZarClos}
    \item If $I$ and $J$ are Hopf ideals in $\mathbb k[G]$, then $\VanLocus(I+J)=\VanLocus(I)\cap\VanLocus(J)\,$.
    \vspace*{1.3mm}\label{item:ASG=Group+cap}
    \item If $H$ and $K$ are closed algebraic subgroups of $G$, then $\VanIdeal(H\cap K)=\VanIdeal(H)+\VanIdeal(K)\,$.
    \vspace*{1.3mm}\label{item:ASG=Ideal+cap}
\end{enumerate}
\end{lem}

\begin{proof} For {\em \ref{item:ASG=HI:VJ=grp}} note that $a\in\VanLocus(J)$ iff $J\subseteq \mathrm{ker}(\mathscr e_a)$. 
So, for $a,b\in \VanLocus(J)$ and $J$ a bi-ideal we have 
$\mathscr e_{ab}(J)=\mathscr e_{a}\otimes\mathscr e_{b}(\Delta(J))\subseteq\mathscr e_{a}\otimes\mathscr e_{b}(J\otimes\mathbb k[G]+\mathbb k[G]\otimes J)=\{0\}$ and hence $ab\in \VanLocus(J)$. The argument for the antipode is analogous.
Given that $\mathrm{char}(\mathbb k)=0$ and using the fact that $\mathbb k[G]/J$ is a commutative Hopf algebra, Theorem~13.1.2 in \cite{Sw69} states that $\mathbb k[G]/J$ is a reduced $\mathbb k$-algebra. 
This, in turn, means that 
$J=\radic{J}$ is a radical ideal. The statements in {\em \ref{item:ASG=HI:HI=radical}} follow now from 
\eqref{eq:HilbertZariski} as
$J=\VanIdeal(\VanLocus(J))=\VanIdeal(H)\,$. 

For {\em\ref{item:ASG=HI:VG=HI}} use the chain of finite-dimensional coalgebras $D_0\subset D_1\subset \ldots$
with $\mathbb k[G]=\bigcup_jD_j\,$ described above. Set $\mathcal I_m=\VanIdeal(H)\cap D_m\,$, let
$D_m=\mathcal I_m\oplus R_m$ for some complementary space, and pick a basis $\{g_\nu\}$ for $R_m\,$ with
$\dim(R_m)=M$. Let $\vec g:G\rightarrow\mathbb k^M$ be given by $\vec g(a)=(g_1(a),\ldots,g_M(a))$. 
Thus, if $\vec v\,^\transp\vec g(a)=0$ for all $a\in H$ then $\vec v\,^\transp\vec g\in \VanIdeal(H)$ 
and hence $\vec v=0$ by linear independence. Given $f\in\mathcal I_m\,$, there are now unique coefficients $b_{\nu,\mu}\in\mathbb k$ such that $\Delta(f)=\sum_{\nu\mu}b_{\nu,\mu}g_\nu\otimes g_\mu+T$ with 
$T\in \mathcal I_m\otimes D_m+D_m\otimes \mathcal I_m\,$. Evaluating at $(a,b)$ for $a,b\in H$ so that $ab\in H$ 
we have
$0=f(ab)=\mathscr e_a\otimes \mathscr e_b \Delta(f)=\sum_{\nu\mu}b_{\nu,\mu}g_\nu(a)g_\mu(b)+\mathscr e_a\otimes \mathscr e_b(T)=\vec g(a)^\transp B\vec g(b)$, where $B=(b_{\nu,\mu})\,$ and $\mathcal{I}_m\subseteq\mathrm{ker}(\mathscr e_{a})\,$. 

As $\vec g(a)^\transp B\vec g(b)=0$ holds for all $a,b\in H$, the previous remarks imply
$B=0$ and, hence, $\Delta(f)=T\in \mathcal I_m\otimes D_m+D_m\otimes \mathcal I_m\subseteq \VanIdeal(H)\otimes\mathbb k[G]+k[G]\otimes \VanIdeal(H)\,$, proving that $\VanIdeal(H)$ is a bi-ideal. Invariance under the antipode is clear
from $S(f)(a)=f(a^{-1})=0\,$ if $a\in H$ and $f\in\VanIdeal(H)\,$. 
Finally, combining {\em \ref{item:ASG=HI:VG=HI}} with {\em \ref{item:ASG=HI:HI=radical}} and \eqref{eq:HilbertZariski} we find 
$\VanIdeal(H)=\VanIdeal(\VanLocus(\VanIdeal(H)))=\VanIdeal(\ZarClos{H})\,$.

The last two statements follow from the standard correspondence between algebraic sets and radical ideals. The only additional observation needed is that the sum of two Hopf ideals is again a Hopf ideal so that by 
{\em \ref{item:ASG=HI:HI=radical}}
$\VanIdeal(H\cap K)=\radic{(\VanIdeal(H)+\VanIdeal(K))}=\VanIdeal(H)+\VanIdeal(K)\,$. 
\end{proof}

The most elementary example is the additive group $\mathbb k_a$\,, for which $\mathbb k[\mathbb k_a]=\mathbb k[x]$ with $\Delta(x)=x\otimes 1+1\otimes x\,$ and $S(x)=-x\,$. 
The vanishing locus of any polynomial in $\mathbb k[x]$ is finite so that a non-trivial Hopf ideal needs to correspond to a finite subgroup of $\mathbb k_a$\,. In $\mathrm{char}(\mathbb k)=0$ the only such subgroup is $\{0\}$. Thus,
by Lemma~\ref{lm:ASG=HI}, the only proper Hopf ideal of $\mathbb k[x]$ is the augmentation ideal $\lpip x\rpip=\VanIdeal(\{0\})\,$. 
In characteristic $p>0$ we, additionally, have that the principal ideal $\lpip x^p\rpip$ generated by 
$x^p$ is a Hopf ideal. 

For the multiplicative group $\mathbb k_m=(\mathbb k^\times,\cdot)$, we have $\mathbb k[\mathbb k_m]=\mathbb k[\lambda,\lambda^{-1}]=\mathbb k[\lambda,\tau]/(\lambda\tau-1)\,$ with $\Delta(\lambda)=\lambda\otimes \lambda\,$. Recall any finite subgroup of 
$\mathbb k_m$ is cyclic and, thus, of the form $\mathsf Z_r=\langle \xi\rangle\,$ where $\xi$ a primitive $r$-th root of unity. We denote $\mathsf Z_1=\{1\}$ and use the convention $\mathsf Z_0=\mathbb k^\times\,$. 
The $\mathsf Z_r$ are algebraic subgroups with 
vanishing ideal $F_r=\VanIdeal({\mathsf Z}_r)=\lpip \lambda^r-1\rpip\,$ so that $\{F_r:r\in\mathbb N\}$ is the complete set of 
non-trivial Hopf ideals of $\mathbb k[\lambda,\lambda^{-1}]\,$. 

A somewhat more involved example is the semidirect product $P=\mathbb k^\times \ltimes\mathbb k$ with group
product given by $(\lambda,x)(\lambda',x')=(\lambda\lambda',x\lambda'+x')\,$. It may be viewed as the closed algebraic subgroup of the upper 
triangular matrices in $\mathrm{GL}_2$ with 1 in the upper left entry. 
Alternatively, $P$ may be viewed as $\mathrm{PGL_2^\geqzero}$, given by  
the image of the upper triangular matrices in the projective linear group $\mathrm{PGL}_2$\,.
Denote abelian algebraic subgroups as
$\bar{\mathsf Z}_r=\mathsf Z_r\times \{0\}$ and $U=\{1\}\times\mathbb k=[P,P]\cong \mathbb k_a\,$. 

\begin{lem}\label{lm:PAlgSubGr}
   Any closed algebraic subgroup of $P$ is of the form $g\bar{\mathsf Z}_rg^{-1}$ or $\bar{\mathsf Z}_rU$
    with $g\in U$ and $r\in\nnN$. 
\end{lem}
\begin{proof} Suppose $H$ is an algebraic subgroup in $P$. Let $\sigma: P\rightarrow \mathbb k^\times$ be 
the abelianization map and note that $S=\sigma(H)\leq \mathbb k^\times$ is an abelian (\textit{a priori} only constructible) subgroup.
Suppose $\rho\in S\setminus\{1\}$ is a non-trivial element, meaning $(\rho,y)\in H$ for some $y\in\mathbb k$.
It is then easy to check that $g(\rho,0)g^{-1}=(\rho,y)$ for $g=(1,(\rho-1)^{-1}y)\in U\,$.  

Note that $H\cap U=\mathrm{ker}(\sigma|_H)$ is an algebraic subgroup of $\mathbb k_a$ and, hence, either $\mathbb 1=(1,0)$ or $U$. If $H\cap U=\{\mathbb 1\}$ the map $\sigma: H\rightarrow S$ is a group isomorphism
so that $H$ is abelian. If $S\neq \{1\}$ there are, by the previous remark, $g\in U$ and $\rho\in\mathbb k^\times$ with $\rho\neq 1$ and $g(\rho,0)g^{-1}\in H$. Since $H$ is abelian it has to be in the centralizer
of $g\bar{\mathsf Z}_0g^{-1}=C(g(\rho,0)g^{-1})\cong\mathbb k_m\,$. Thus, $H=g\bar{\mathsf Z}_rg^{-1}$ for some $r\in\mathbb N\,$. If $S=\{1\}$ clearly $H=\{\mathbb 1\}\,$. 

Finally, assume $H\cap U=U$ and $S\neq \{1\}$. By the observation above we have that $(\rho,y)\in H$ 
implies $(\rho,0)\in H$ as now $g\in H\,$. Hence, $H=S\ltimes \mathbb k$ so that $S\cong H\cap \bar{\mathsf Z}_0$ is
an algebraic group. It follows that $H=K_r\ltimes \mathbb k=\bar{\mathsf Z}_rU\,$. If also $S=\{1\}$, clearly
$H=U=\bar{\mathsf Z}_1U\,$. 
\end{proof}

For a clearer formal distinction between coordinates of and functions on $G$ denote 
by $\cofun{t}, \cofun{x}:G\rightarrow\mathbb k$ the projections given by $\cofun{\lambda}(g)=\lambda$ 
and  $\cofun{x}(g)=x$ if $g=(\lambda,x)\,$. Similar to the previous examples, we have 
$\mathbb k[\mathrm{PGL_2^\geqzero}]\cong\mathbb k[P]=\mathbb k[\cofun{\lambda},\cofun{\lambda}^{-1},\cofun{x}]$ as a commutative algebra. The coproducts are easily 
worked out from the product rule on $P$ as
\begin{equation}\label{eq:xlambcop}
    \Delta(\cofun{\lambda}\,)=\cofun{\lambda}\otimes\cofun{\lambda}
\qquad\mbox{and}\qquad 
\Delta(\cofun{x})=\cofun{x}\otimes\cofun{\lambda}+1\otimes\cofun{x}\,,
\end{equation}
with $S(\cofun{\lambda})=\cofun{\lambda}^{-1}$ and $S(\cofun{x})=-\cofun{x}\cofun{\lambda}^{-1}\,$.

Assuming $\mathrm{char}(\mathbb k)=0$ and using Lemma~\ref{lm:ASG=HI}, we obtain a complete list of
Hopf ideals for $\mathbb k[P]$ as the vanishing ideals of the subgroups classified in 
Lemma~\ref{lm:PAlgSubGr}. For example, for $g=(1,\nu)$ with $\nu\in\mathbb k$ we find $J_\nu=\VanIdeal(g\bar{\mathsf Z}_0g^{-1})=\lpip\cofun{x}-\nu(\cofun{\lambda}-1)\rpip$ and for any $r\in\mathbb N$ we have $F_r=\VanIdeal(\bar{\mathsf Z}_rU)=\lpip \cofun{\lambda}^r-1\rpip$.  
For the remaining subgroups, the respective Hopf ideals are no longer principal, since for $r\in\mathbb N$ 
$\VanIdeal(g\bar{\mathsf Z}_rg^{-1})=\VanIdeal(g\bar{\mathsf Z}_0g^{-1}\cap \bar{\mathsf Z}_rU)=F_r+J_\nu=\lpip \cofun{\lambda}^r-1,\cofun{x}-\nu(\cofun{\lambda}-1)\rpip\,$. Note, finally, that the augmentation ideal is given by $\mathrm{ker}(\varepsilon)=F_1+J_0=J_\nu+J_\mu$ for $\nu\neq\mu$ and that $F_s+F_r=F_{\gcd(s,r)}\,$, providing a full description of the additive semigroup of Hopf ideals of $\mathbb k[P]\,$.

The discussion of $\mathbb k[P]$ now directly applies to quantum groups $\Uz$ or $\Uz^\geqzero$ at a root of unity $\zeta\,$. For a simple root $\alpha_i$\,, denote by $\Zsubalgchar_i^\geqzero$ the subalgebra 
generated by $\Kpw_i=K_i^{\ell_i}$ and $\Epw_i=E_i^{\ell_i}$. 

\smallskip

\begin{cor}\label{cor:ClassHI-A1}
The algebra $\Zsubalgchar_i^\geqzero$ is a commutative Hopf subalgebra of $\Uz^{(\geqzero)}$.

\noindent The semigroup of possible Hopf ideals of $\Zsubalgchar_i^\geqzero\otimes \mathbb C$ is given by 
the principal ideals
$\ddot F_r=\lpip L_i^r-1\rpip$ with $r\in\mathbb N$ and $\ddot J_\nu=\lpip X_i-\nu(L_i-1)\rpip$
with $\nu\in\mathbb C$ as well as the sums of ideals $\ddot F_r+\ddot J_\nu\,$. 
\end{cor}
\begin{proof}
    Note first $\Epw_i\Kpw_i=\zeta_i^{2\ell_i^2}\Kpw_i\Epw_i=\Kpw_i\Epw_i$ by definition of $\zeta_i$
    and $\ell_i$\,. So, $\Zsubalgchar_i^\geqzero$ is commutative, as also stated in Corollary~\ref{cor:Zcommutmax}.
    Clearly, $L_i$ is group-like since $K_i$ is. Moreover, we find from \eqref{eq:corels_pwrs_gen} and
    Corollary~\ref{cor:zeros-mulitnom}
    the coproduct formula 
    \begin{align}\label{eq:simplecoprod}
    \Delta(\Epw_i) =\sum_{s=0}^{\ell_i} \qbin{\ell_i}{s}{i}\zeta_i^{n(\ell_i-s)}E_i^s\otimes E_i^{\ell_i-s}K_i^s=E_i^\elli\otimes K_i^\elli+1\otimes E_i^\elli
    =
    \Epw_i\otimes \Kpw_i+1\otimes \Epw_i \,.
\end{align}
Thus, $\Zsubalgchar_i^\geqzero\otimes \mathbb C\cong \mathbb C[P]$ as Hopf algebras over $\mathbb C\,$. The claim now follows from the previously discussed classification of Hopf ideals of $\mathbb C[P]\,$. 
\end{proof}

Non-trivial Hopf ideals in $\Zsubalgchar_i^\geqzero$ for a general ground ring $\Zgen$ containing $\Zz$ can be found as preimages of the extension of scalars map $\Zsubalgchar_i^\geqzero\rightarrow\Zsubalgchar_i^\geqzero\otimes \mathbb C$, provided that $\nu\in\Zgen\,$. 
All of these are readily seen to be free and complemented $\Zgen$-submodules of  $\Zsubalgchar_i^\geqzero$.

The classification in Corollary~\ref{cor:ClassHI-A1} underscores the need to choose the minimal exponents 
$\ell_i=\ell/\gcd(d_i,\ell)$ in the definition of the $\Epw_i$ generators. 
Specifically, suppose $\alpha_i$ is a long root in the doubly laced Lie type case, meaning $d_i=2$, and $\ell$ is even. Then  $\ell=2\ell_i$ so that
$E_i^\ell=X_i^2\,$, generating an ideal $\lpip E_i^\ell\rpip=\lpip \Epw_i^2\rpip$ that is not even radical and hence not Hopf. The latter is also obvious from the coproduct 
$$
\Delta(\Epw_i^2)=\Epw_i^2\otimes \Kpw_i^2+2\Epw_i\otimes \Kpw_i\Epw_i+1\otimes \Epw_i^2\;\;\not\in\;\;
\lpip \Epw_i^2\rpip\otimes\Uz+\Uz\otimes\lpip \Epw_i^2\rpip\,.
$$
Since  restricted groups are occasionally defined in the literature with such generators 
we add the above observation as a more formal remark, which has an obvious analogue for $\LT{G}_2\,$.

\begin{rem}
    Assume $\Uz^{(\geqzero)}$ is of Lie type  $\LT{B}_n$\,, $\LT{C}_n$\,, or $\LT{F}_4$ with $n\geq 2$ as well as  $\kay\equiv 0\mod 4$.    
    Then the ideal generated by 
    $\{E_\alpha^\ell:\,\alpha\in \mathsf R\}$ is {\bf not} a Hopf ideal if $\mathsf R\subseteq\proots$ contains a long root. 
\end{rem}

Recall that any reductive algebraic subgroups over an algebraically closed field admits a Borel subgroup, unique up to conjugation. 
The fundamental example of such a Borel subgroup is the subgroup $\UtriM{n}<\mathrm{GL}(n,\mathbb k)$ of upper triangular matrices with non-zero diagonal entries. Denote also by $\DiagM{m}$ the subgroup of diagonal matrices and $\SUtriM{n}$
the subgroup of upper triangular matrices with 1's on the diagonal so that  $\UtriM{n}=\DiagM{m}\ltimes \SUtriM{n}\,$.
As algebraic sets, $\SUtriM{n}\cong\mathbb k^{\binom{n}{2}}$ and $\DiagM{m}\cong(\mathbb k^\times)^n$, where the latter is realized as either $\{(x_{1,1},\ldots,x_{n,n},\tau)\in\mathbb k^n\times\mathbb k: x_{1,1}\ldots x_{n,n}\tau=1\}$ or, equivalently, as  
$(D_1)^n$. The former is commensurable with the embedding in $\mathrm{GL}(n,\mathbb k)$, while the latter yields the
coordinate rings as Laurent polynomials. 

The coordinate ring is, thus, given as  
$\mathbb k[\UtriM{n}]=\mathbb k\bigl[\{\cofun{x_{i,j}}:i\leq j\},\{\cofun{x_{i,i}}^{-1}\}\bigr]$ 
so that the set of functions 
that are polynomials in off-diagonal entries and Laurent polynomials in the diagonal entries. As before,
we notationally distinguish coordinates and coordinate functions. The group of group-like elements in $\mathbb k[\UtriM{n}]$, corresponding to algebraic characters on
$\UtriM{n}$\,, are given 
\begin{equation}\label{eq:CharTn}
    \AChar{\UtriM{n}}\cong \mathbb Z^n\qquad \mbox{ with generators } \quad \{\cofun{x_{i,i}}:i=1,\ldots n\}\,.
\end{equation}
The counit is simply $\varepsilon(\cofun{x_{i,j}})=\delta_{ij}$\,. The coproduct and antipode for the remaining generators $i<j$ are  readily 
derived from \eqref{eq:AlgGLnCoprod} as 
\begin{equation}\label{eq:HopfStrTn}
    \Delta(\cofun{x_{i,j}})=\sum_{i\leq k\leq j}\cofun{x_{i,k}}\otimes \cofun{x_{k,j}}
    \quad\mbox{ and }\quad
    S(\cofun{x_{i,j}})= 
     \mkern -18mu \longsum[17]_{\substack{1\leq m\leq j-i\\i=k_0<k_1<\ldots<k_m=j}}
     \mkern -18mu 
     (-1)^{m}\cofun{x_{i,i}}^{-1}\cofun{x_{i,k_1}}\cofun{x_{k_1,k_1}}^{-1}
    \cofun{x_{k_1,k_2}}\ldots \cofun{x_{k_{m-1},j}}\cofun{x_{j,j}}^{-1}\;.
\end{equation} 
 
We finally mention the notion of a {\em cosolvable} Hopf algebra $H$ following \cite{MW98}. In the case
of a commutative Hopf algebra $H$, it starts from a finite ascending series of 
$\mathbb k1=H_0\subset H_1\subset \ldots \subset H_{n}=H$ Hopf subalgebras. 
From the augmentation ideals $H_s^\circ=H_s\cap\mathrm{ker}(\varepsilon)$ one obtains for each $s$
a Hopf ideal $H_{s-1}^\circ H_{s}$ in $H_s$ and, thus, a Hopf algebra 
$\bar H_s=H_s/H_{s-1}^\circ H_{s}$\,, called the $s$-th {\em factor} of the series. $H$ is said to be 
cosolvable if all of the factors are cocommutative for some such series. 

In the case $H=\mathbb k[\UtriM{n}]$, these properties can be easily verified if we define $H_s$ 
as the polynomial subalgebra generated by $\{\cofun{x_{i,j}}:\,j-i< s\}$ and 1.
It is immediate from \eqref{eq:HopfStrTn} that each $H_s$ defines a Hopf subalgebra. 
Moreover, it is clear that $\bar H_1=\mathbb k[\DiagM{m}]$, generated by group-like elements. 
For $s>1$, one readily verifies that  $\bar H_s\cong\mathbb k[\{\bar x_{i,j}:i-j=s-1\}]$ 
is a polynomial algebra, for which all generators are strictly primitive. Hence, each
$H_s$ is cocommutative and $\mathbb k[\UtriM{n}]$ is cosolvable. 

Recall, further, that by the Lie-Kolchin Theorem every closed, {\em connected}, and {\em solvable} linear 
algebraic group $B$ is conjugate to a subgroup of $\UtriM{n}$\,, see \cite[6.3.1.]{Sp98}. The images of the $H_s$
under restriction and conjugation map $\mathbb k[\UtriM{n}]\rightarrow\mathbb k[B]$ give rise to a respective series with the same properties as the one for $\mathbb k[\UtriM{n}]\,$. 

\begin{cor}\label{cor:kB=cosolv}
    Suppose $B$ is a closed, connected, solvable algebraic subgroup of $\mathrm{GL}(n,\mathbb k)$ over an algebraically closed field $\mathbb k$. Then $\mathbb k[B]$ is cosolvable in the sense of \cite{MW98}.
\end{cor}

In particular, 
the coordinate ring of $\mathbb k[\mathrm{GL}(n,\mathbb k)]$ is  fundamentally different from that of a (connected) Borel subgroup. The former is cosemisimple and thus not cosolvable and its character group
in \eqref{eq:GLnChar} also differs from that in  \eqref{eq:CharTn}.

\subsection{Standard Presentation of \texorpdfstring{$\Uq^\geqzero$}{Uq≥0} for \texorpdfstring{$\mathsf{A}_n$\,}{An}}\label{subsec:CoalgAn}
We  provide next detailed relations and coproduct formulae for generators $E_w$ of $\Uq^\geqzero(\mathfrak{sl}_{n+1})$ with respect to a particularly convenient choice of a maximal word.
It is given by 
the conjugate $z_n=w_{\LT{A}_n}^\winvchar$ of the one mentioned in Section~\ref{subsec:orderings} and inductively defined as 
\begin{equation}\label{eq:An-defzmcmk}
    z_{n}=z_{n-1}c_{n,1} \quad \mbox{ with } \quad c_{n,k}=w_n\ldots w_{n+1-k}\quad \mbox{ so that } \quad c_{n,k}^\flat=w_n\ldots w_{n+2-k}=c_{n,k-1}\,.
\end{equation}

Observe that $z_m$ is the word of maximal length in the canonically embedded $\LT{A}_m$ subsystem in $\LT{A}_n$ for $n>m$ with
$\len{z_m}=\binom{m+1}{2}$. 
Denote  by $\bar z_n=\Weylpres(z_n)$ and $\bar c_{n,k}=\Weylpres(c_{n,k})$  the respective elements in the Weyl group $\Weyl$. 
In the standard Euclidean realization of the $\LT{A}_n$ root system in $\mathbb R^{n+1}=\langle\epsilon_1,\ldots,\epsilon_{n+1}\rangle$, simple roots are given as 
$\alpha_i=\epsilon_i-\epsilon_{i+1}$ and $\Weyl$ acts as the group of permutations in $n+1$ letters 
via $s(\epsilon_i)=\epsilon_{s(i)}\,$, with
generators $s_i=(i,i+1)\,$. Specifically, $\bar c_{n,k}=(n+1, n, \ldots , n+1-k)$ is a cyclic permutation
and $\bar z_m(\epsilon_i)=\epsilon_{m+2-i}$ for $1\leq i\leq m+1$ and $\bar z_m(\epsilon_i)=\epsilon_i$ for $i>m+1\,$. 

Suppose $N=\binom{m}{2}+k$ for $1\leq k\leq m$. We have then  $z_n[1,N]=z_{m-1}c_{m,k}\,$ for the first $N$ letters in $z_n$\,. The 
$N$-th term in the ordering sequence is now computed as 
\begin{equation}\label{eq:An-betaN}
    \begin{aligned}
    \beta_N&=\wordroot(z_n[1,N])=\wordroot(z_{m-1}c_{m,k})=\bar z_{m-1}\bar c_{m,k-1}(\alpha_{m+1-k})\\
    &=\bar z_{m-1}(m+1\ldots m+2-k)(\epsilon_{m+1-k}-\epsilon_{m+2-k})\\
    &=\bar z_{m-1}(\epsilon_{m+1-k}-\epsilon_{m+1})=\epsilon_{k}-\epsilon_{m+1}=\alpha_{k}+\ldots+\alpha_{m}\,.
\end{aligned}
\end{equation}\label{eq:An-LexOrd}
Hence, the ordering $\letwt{z}$ defined in Section~\ref{subsec:descroots} is simply the lexicographic ordering. 
Namely, for $i<j$ and $r<s$ we have 
\begin{equation}\label{eq:AnCanOrd}
    \epsilon_{i}-\epsilon_{j}\letwt{z}\epsilon_{r}-\epsilon_{s} \qquad \mbox{ iff }\quad j<s \;\;\mbox{ or }\;(j=s \;\mbox{ and } \;\;i<r)\,.
\end{equation}

Consider next for $1\leq k\leq m$ the reduced words  
\begin{equation}\label{eq:An-defbmk}
    b_{m,k}=w_{k}\ldots w_{m}=c_{m,m+1-k}^*\, \quad \mbox{ so that } \quad b_{m,k}^\flat=w_{k}\ldots w_{m-1}=b_{m-1,k}\,. 
\end{equation}
As before, we denote the respective element in the Weyl group as $\bar b_{m,k}=(k,\ldots,m,m+1)\,$ and   compute
\begin{equation}\label{eq:An-rootbmk}
    \wordroot(b_{m,k})=\bar b_{m-1,k}(\alpha_m)=(k,\ldots,m)(\epsilon_m-\epsilon_{m+1})=\epsilon_{k}-\epsilon_{m+1}\,.
\end{equation}

Note next that $\bar b_{m,k}$ is in the $\LT{A}_m$ subsystem and thus $\bar b_{m,k}\leqRB \bar z_m\,$.  Thus, there exists 
some $\bar u_{m,k}$ in the Weyl group of the subsytem such that 
\begin{equation}
    \bar z_m =\bar b_{m,k}\cdot\bar u_{m,k} \quad\mbox{with }\;\;\len{\bar z_m}=\len{\bar b_{m,k}}+\len{\bar u_{m,k}}\,.
\end{equation} 
In fact, it is easy to verify that a representative reduced word is given by $u_{m,k}=z_{m-1}\cdot c_{m,k-1}\,$.
The calculations in \eqref{eq:An-betaN} and \eqref{eq:An-rootbmk} show that 
$$
\bar b_{m-1,k}(\alpha_m)=\bar z_{m-1}\bar c_{m,k-1}(\alpha_{m+1-k})=\bar b_{m-1,k}\bar u_{m-1,k}\bar c_{m,k-1}(\alpha_{m+1-k})\,.
$$
Together with length comparisons we find
\begin{equation}\label{eq:An-rootrefl}
    \begin{aligned}
        \alpha_m&=\bar u_{m-1,k}\bar c_{m,k-1}(\alpha_{m+1-k})\\
        \mbox{and}\qquad & \len{z_{m-1}c_{m,k-1}}=\len{b_{m-1,k}}+\len{u_{m-1,k}}+\len{c_{m,k-1}}\,.
    \end{aligned}
\end{equation}

For $1\leq i<j\leq n+1$, denote now the generators of $\Uq^+$ as follows.
\begin{equation}\label{eq:An-defEij}
    E_{i,j}=E_{z_n[\epsilon_i-\epsilon_j]}=E_{z_n[1,N]}\qquad\mbox{ with } \ N=\mbox{$\binom{j-1}{2}$}+i
\end{equation}

\begin{prop} 
For a given $j\in\{2,\ldots, n+1\}$, the elements in \eqref{eq:An-defEij} are determined by 
the identity $E_{j-1,j}=E_{j-1}$ and the recursive relation  
\begin{equation}\label{eq:higher_roots}
E_{i,j}=\Tinv_i(E_{i+1,j})=q^{-1}E_iE_{i+1,j}-E_{i+1,j}E_i \qquad \mbox{for} \quad i<j-1\;.
 \end{equation}   
\end{prop}

\begin{proof} Set $m=j-1$ and $k=i$\,. We have $\wordroot(z_n[1,N])=\epsilon_i-\epsilon_j$ for $N=\binom{j-1}{ 2}+i\,$ by \eqref{eq:An-betaN}. Thus, by definition, $E_{z_n[1,N]}=\Tinv_{z_{j-2}c_{j-1,i-1}}(E_{j-i})
=\Tinv_{b_{j-2,i}u_{j-2,i}c_{j-1,i-1}}(E_{j-i})=\Tinv_{b_{j-2,i}}(\Tinv_{u_{j-2,i}c_{j-1,i-1}}(E_{j-i}))$ since
the $\Tinv$ only depend on the Weyl element and $\len{b_{j-2,i}u_{j-2,i}c_{j-1,i-1}}=\len{b_{j-2,i}}+\len{u_{j-2,i}c_{j-1,i-1}}$ by \eqref{eq:An-rootrefl}. 

Note, further, that by \eqref{eq:An-rootrefl} we have that the weight of $E_{u_{j-2,i}c_{j-1,i}}=\Tinv_{u_{j-2,i}c_{j-1,i-1}}(E_{j-i})$ is $\alpha_{j-1}\,$. It, therefore, follows from Proposition~\ref{prop:wordgencont} that $E_{u_{j-2,i}c_{j-1,i}}=E_{j-1}\,$. Plugging this into the previous expression for $E_{z_n[1,N]}$ we find
\begin{equation}\label{eq:Eij-minword} 
    E_{i,j}=\Tinv_{b_{j-2,i}}(E_{j-1})=\Tinv_i\ldots\Tinv_{j-2}(E_{j-1})=\Tinv_i(E_{i+1,j})\,.
\end{equation}
In particular, $E_{j-1,j}=E_{j-1}\,$. The recursion in \eqref{eq:higher_roots} follows by induction, decreasing in $i$. Assuming the identity for $i+1$, one can express $E_{i+1,j}$ by $E_{i+1}$ and $E_{i+2,j}$\,. Since the latter is in the subalgebra 
$\langle E_{i+2},\ldots,E_{j-1}\rangle$ generated by
$\{E_{i+2},\ldots,E_{j-1}\}$, $\Tinv_i$ acts as identity on $E_{i+2,j}$ and $E_i$ commutes with $E_{i+2,j}$\,. The resulting expression 
in $E_i$\,, $E_{i+1}$\,, and $E_{i+2,j}$ can now be worked out and identified with the left side of \eqref{eq:higher_roots}.
\end{proof}

A full set of commutation  relations as well as nested subalgebras of $\LT{A}$-types are derived in the next proposition. Most of these can be found in  other references, with the possible exception of the relation in Item {\em\ref{item:AComms:3}} below.

 \begin{prop}\label{prop:AComms}
For $1\leq i<s<j<t\leq n+1$ the following commutation relations hold.  \vspace*{-2mm}
\begin{multicols}{2}
\begin{enumerate}[label=\roman*), leftmargin=2cm,]
\item\label{item:AComms:1}$E_{i,t}=q^{-1}E_{i,j}E_{j,t}-E_{j,t}E_{i,j}$\vspace*{2mm}
\item\label{item:AComms:2}$E_{i,t}E_{s,j}=E_{s,j}E_{i,t}$\vspace*{2mm}
\item\label{item:AComms:3}$E_{i,j}E_{s,t}-E_{s,t}E_{i,j}=-(q-q^{-1})E_{s,j}E_{i,t}$\vspace*{2mm}
\item\label{item:AComms:4}$E_{i,s}E_{j,t}=E_{j,t}E_{i,s}$\vspace*{2mm}
\item\label{item:AComms:5}$E_{i,t}E_{i,j}=qE_{i,j}E_{i,t}$\vspace*{2mm}
\item\label{item:AComms:6}$E_{j,t}E_{i,t}=qE_{i,t}E_{j,t}$\vspace*{2mm}
\end{enumerate}
    \end{multicols}\vspace*{-5mm}
    
\noindent More generally, suppose $\hat j=(j_1,\ldots, j_{m+1})$ is a sequence with 
$1\leq j_1<j_2<\ldots<j_{m+1}\leq n+1$. Then there exists an injective algebra homomorphism
$\phi_{\hat j}:\Uq^+(\mathfrak{sl}_{m+1})\rightarrow U_q^+(\mathfrak{sl}_{n+1})$ with
\begin{equation}\label{eq:injectEs}
    \phi_{\hat j}(E_{r,s})=E_{j_r,j_s}\,.
\end{equation} 
\end{prop}

\begin{proof} The relation in {\em \ref{item:AComms:1}} is obtained by applying 
$\Tinv_{b_{i-1,k}}=\Tinv_k\ldots\Tinv_{i-1}$ to \eqref{eq:higher_roots}, using that $\Tinv_{b_{i-1,k}}$ acts trivially on $E_{i+1,j}\in\langle E_{i+1},\ldots, E_{j-1}\rangle\,$, and substituting indices. Item {\em \ref{item:AComms:4}} is clear,
since $E_{i,s}\in\langle E_i,\ldots,E_{s-1}\rangle$ and $E_{j,t}\in\langle E_j,\ldots,E_{t-1}\rangle$ are in algebras for which the generators of one commute with the generators of the other. 

Next note that, by \eqref{eq:Eij-minword}, we can write $E_{i,j}=E_{w_{ij}}$\,, where $w_{ij}=w_i\ldots w_{j-1}\,$. 
Clearly, $w_{it}=w_{ij}w_{jt}$ with $\len{w_{it}}=\len{w_{ij}}+\len{w_{jt}}$. Further, 
the index of $\tau(w_{ij})=j-1$ does not occur $w_{jt}\,$. We also have 
$\alpha=\wordroot(w_{ij})=\epsilon_i-\epsilon_j$ and  
$\beta=\wordroot(w_{it})=\epsilon_i-\epsilon_t\,$, so that $\symbrack{\alpha}{\beta}=1$. The relation in 
{\em \ref{item:AComms:5}} follows now directly from Lemma~\ref{lm:scommgens}, with
$w=w_{it}$\,, $u=w_{ij}$\,, and $v=w_{jt}\,$. 

Alternatively, we find from \eqref{eq:An-betaN} and \eqref{eq:An-defEij} that $E_{it}=E_{z_{t-2}c_{t-1,i}}\,$. 
From \eqref{eq:An-defzmcmk} we have $c_{t-1,j}=c_{t-1,i}v\,$, where $v=w_{t-1-i}\ldots w_{t-j}\,$. So, if we 
set $w=z_{t-2}c_{t-1,j}$ and $u=z_{t-2}c_{t-1,i}$ we find $w=uv$ with $\len{w}=\len{u}+\len{v}\,$. 
As before $\tau(u)=t-j$ does not occur as an index in $v\,$. Moreover, $\alpha=\wordroot(u)=\epsilon_i-\epsilon_t$
and $\beta=\wordroot(w)=\epsilon_j-\epsilon_t\,$, so that  $\symbrack{\alpha}{\beta}=1$. 
Lemma~\ref{lm:scommgens}, therefore, implies the relation in Item {\em \ref{item:AComms:6}}.

Combining relation {\em \ref{item:AComms:1}} with   {\em \ref{item:AComms:5}} as well as 
combining relation {\em \ref{item:AComms:1}} with   {\em \ref{item:AComms:6}} we obtain respective Serre relations
$$
    E_{ij}^2E_{jt}-[2]E_{ij}E_{jt}E_{ij}+E_{jt}E_{ij}^2=0\quad\mbox{ and }\quad 
    E_{jt}^2E_{ij}-[2]E_{jt}E_{ij}E_{jt}+E_{ij}E_{jt}^2=0\,.
$$
Together with the relation in {\em \ref{item:AComms:5}} we  observe that, for a sequence $j_1<\ldots<j_{m+1}$ as above,
the elements $\tilde E_r=E_{j_r,j_{r+1}}$ fulfill the relations of the standard generators of $\Uq^+(\mathfrak{sl}_{m+1})$. 
Hence, we  have a well-defined algebra homomorphism $\phi_{\hat j}$ that maps the $E_r$ generator in $\Uq^+(\mathfrak{sl}_{m+1})$
to $\tilde E_r$ in $\Uq^+(\mathfrak{sl}_{n+1})$. Setting $(i,j,t)=(j_k,j_{k+1},j_m)$ 
and denoting $\tilde E_{k,m}=E_{j_k,j_m}$ relation {\em \ref{item:AComms:5}} can be written as
$\tilde E_{k,m}=q^{-1}\tilde E_{k}\tilde E_{k+1,m}-\tilde E_{k+1,m}\tilde E_{k}$\,. Since this is the same recursion 
\eqref{eq:higher_roots} for the generators in $\Uq^+(\mathfrak{sl}_m)$ and $\phi_{\hat j}$ is a homomorphism,
\eqref{eq:injectEs} follows inductively. This also means that the PBW basis elements given by ordered monomials
in the $E_{i,j}$ in $\Uq^+(\mathfrak{sl}_{m+1})$ are mapped to elements of standard PBW basis of $\Uq^+(\mathfrak{sl}_{n+1})$, which, in turn, implies that $\phi_{\hat j}$ is injective.

Given the $\LT{A}_m\hookrightarrow\LT{A}_n$ inclusions, it now suffices for the remaining relations to consider the
$\LT{A}_3$ case, using $\phi_{\hat j}$ with $\hat j=(j_1,j_2,j_3,j_4)=(i,s,j,t)\,$. The relation in 
{\em \ref{item:AComms:2}} is then implied by $E_{1,4}E_2=E_2E_{1,4}$\,. To see this, observe $E_{1,4}=\Tinv_1\Tinv_2(E_3)$
and $E_2=\Tinv_1\Tinv_2(E_1)$ as in Proposition~\ref{prop:wordgencont}. Thus, $[E_{1,4},E_2]=[\Tinv_1\Tinv_2(E_3),\Tinv_1\Tinv_2(E_1)]=\Tinv_1\Tinv_2([E_1,E_3])=0$
as desired. 

To derive the last relation {\em \ref{item:AComms:3}}, we set $\theta_2=\Tinv_2^{-1}(E_2)=-F_2K_2$ and compute from the relations in Section~\ref{subsec:qgroups-genrel} the following equations in the full quantum group $\Uq$\,. 
\begin{equation*}
    \theta_2E_i=q^{-1}E_i\theta_2 \quad\mbox{for } i\in\{1,3\}
\qquad\mbox{and }\qquad
\theta_2E_2=\tfrac{q^2}{q-q^{-1}}(K_2^2-1)+q^2E_2\theta_2
\end{equation*}
Letting $c=\tfrac{q^2}{q-q^{-1}}$ and $i,j\in\{1,3\}$, these entail commutators  
\begin{equation}\label{eq:theatEcomm}   
\begin{aligned}
    [\theta_2,E_iE_2E_j]&=c(qK_2^2-q^{-1})E_iE_j\,, & [\theta_2,E_2E_iE_j]&=c(K_2^2-1)E_iE_j\,, \\
[\theta_2,E_iE_jE_2]&=c(q^2K_2^2-q^{-2})E_iE_j\,. && 
\end{aligned}
\end{equation}
Using the braid relations for the $\Tinv_i$ as well as $\Tinv_2\Tinv_3(E_2)=E_3$ we, further, compute
\begin{equation}\label{eq:T1T3E2comp} 
\begin{aligned}
    \Tinv_2^{-1}\Tinv_1^{-1}\Tinv_2(E_3)&=\Tinv_1\Tinv_2^{-1}\Tinv_1^{-1}(E_3)=\Tinv_1\Tinv_2^{-1}(E_3)
    =\Tinv_1\Tinv_3(E_2)\\
    &=q^{-2}E_3E_1E_2-q^{-1}(E_3E_2E_1+E_1E_2E_3)+E_2E_1E_3\;.
\end{aligned}
\end{equation}
Combining \eqref{eq:theatEcomm} and \eqref{eq:T1T3E2comp} we arrive at
$$
[\Tinv_2^{-1}(E_2),\Tinv_2^{-1}\Tinv_1^{-1}\Tinv_2(E_3)]=-(q-q^{-1})E_1E_3\;,
$$
which yields, after application of $\Tinv_1\Tinv_2$ and using $\Tinv_1\Tinv_2(E_1)=E_2=E_{2,3}\,$, the desired form
$$
[E_{1,3},E_{2,4}]= [\Tinv_1(E_2),\Tinv_2(E_3)]= -(q-q^{-1})\Tinv_1\Tinv_2(E_1)\Tinv_1\Tinv_2(E_3)=
-(q-q^{-1})E_{2,3}E_{1,4}\,. 
$$
\end{proof}
 
To express the coproducts of these elements we introduce notations  $K_{i,j}=K_i\ldots K_{j-1}$ and $E_{i,i}=E_{j,j}=1$.  

 \begin{lem}
     For $i<j$, the coproduct of $E_{i,j}$ is 
     \begin{equation}\label{eq:DeltaEij}
         \Delta(E_{i,j})=\sum_{k=i}^j c_kE_{i,k}\otimes K_{i,k}E_{k,j}
     \end{equation} 
     where $c_i=c_j=1$ and $c_k=-(q-q^{-1})$ for $i<k<j$.
 \end{lem}

\begin{proof} The formula is immediate for the coproduct of $E_{j-1}$ when $i=j-1$. The assertion for $i<j-1$ is then obtained by descending induction in $i$, using the recursion \eqref{eq:higher_roots}. For $T^{i,j}_k=E_{i,k}\otimes K_{i,k}E_{k,j}$
a basic calculation yields
$$
q^{-1}\Delta(E_i)T^{i+1,j}_k-T^{i+1,j}_k\Delta(E_i)=
\begin{cases}
    T^{i,j}_k & \mbox{ for } \ k>i+1\\
    -(q-q^{-1})T^{i,j}_{i+1}+T^{i,j}_{i}& \mbox{ for } \ k=i+1\,.
\end{cases}
$$
Assuming \eqref{eq:DeltaEij} for $\Delta(E_{i+1,j})$ we then use this computation to derive 
\eqref{eq:DeltaEij} for the coproduct as $\Delta(E_{i,j})=q^{-1}\Delta(E_i)\Delta(E_{i+1,j})-\Delta(E_{i+1,j})\Delta(E_i)\,$.
\end{proof}

The coproduct for powers of these generators is derived from commutation relations for the elements $T^{i,j}_k$ defined in the proof above.
First, observe that the identities {\em \ref{item:AComms:5}} and {\em \ref{item:AComms:6}} in Proposition~\ref{prop:AComms} imply that 
\begin{equation}\label{eq:Tij-qcomm}
    T^{i,j}_rT^{i,j}_s=q^{-2}T^{i,j}_sT^{i,j}_r \qquad \mbox{ for all } \; i\leq r < s \leq j\,.
\end{equation}

Thus, if $D_m=\sum_{k=i}^mc_kT^{i,j}_k$ we find from \eqref{eq:Tij-qcomm} that also $D_{m-1}T^{i,j}_m=q^{-2}T^{i,j}_mD_{m-1}\,$. We may now write
$D_m=V+U$ where $V=D_{m-1}$ and $U=c_mT^{i,j}_m\,$. Applying 
\eqref{eq:q-binom-form}, this allows us to infer $D_m^N=\sum_{r}q^{r(N-r)}\qbinsmall{N}{r}
D_{m-1}^{N-r}\left(c_mT^{i,j}_m\right)^r\,$. Iterating this formula, we arrive at the following expression,
\begin{equation}\label{eq:An-Copr-Epwr}
    \Delta(E_{i,j}^N)=\longsum[28]_{\substack{r_i,\ldots,r_j\geqzero\\r_i+\ldots+r_j=N}}q^{\sum_{i\leq s < t\leq j}r_sr_t}(q^{-1}-q)^{\sum_{i<s<j}r_s}\qbin{N}{r_i,\ldots,r_j}{}\left(T^{i,j}_i\right)^{r_i}\ldots \left(T^{i,j}_j\right)^{r_j}\,.
\end{equation}

\subsection{Identification of \texorpdfstring{$\Zsubalgchar^\geqzero_\bullet$}{Z.≥0} as a (Semi) Classical Matrix Algebra}
\label{subsec:Zn=CAn}
We consider next the specialization to a primitive $\kay$-th root of unity $\zeta$, with $\ell$  the order of $\zeta^2\,$. In the simply laced setting all $\ell_\alpha=\ell\,$. Thus, 
for the root $\alpha=\alpha_i+\ldots +\alpha_{j-1}=\epsilon_i-\epsilon_j\,$, we have
$\Epw_\alpha=\Epw_{i,j}=E_{i,j}^\ell$
with respect to the ordering given by  \eqref{eq:An-defzmcmk}. We extend this notation to $X_{i,i}=1$. Write also $\Kpw_{i,j}=\Kpw^\alpha=K_{i,j}^{\ell}\,$ as in \eqref{eq:DefLmu} and
define  $\Zsubalgchar_\bullet^{\geqzero}=\Zsubalgchar_{\longweyl}^{\geqzero}$ as in \eqref{eq:defZ-othtyp}
and \eqref{eq:Zmax}.

The coproducts of the $\Epw_\alpha$  have been computed 
in \cite[(2.24)]{BW04} for a 2-parameter version $U_{r,s}(\mathfrak{sl}_n)\,$, which coincides with $\Uz$ if $r=s^{-1}$ is a respective root of unity. 
The results there imply that the collection of all $\Epw_\alpha$ generate a 
Hopf ideals, leading to  a well-defined (2-parameter) restricted quantum group. 
A general form for the coproduct for other cases can also be found in \cite[Sect. 5.5]{dck92}, though the stated formulae are modulo higher ideal terms and no explicit coefficients are offered. A similar version of the  next lemma can be found in \cite{BW04}.  

\begin{lem}\label{lem:An-CoprX}
For $1\leq i<j\leq n+1$ and with notation as above, we have
    \begin{equation}\label{eq:An-CoprX}
        \Delta(\Epw_{i,j})=\sum_{k=i}^jb_k \Epw_{i,k}\otimes \Kpw_{i,k}\Epw_{k,j} \;,
    \end{equation} 
where $\;b_k=(\zeta^{-1}-\zeta)^\ell\zeta^{\binom{\ell}{ 2}}\;$ for $i<k<j$ and $b_i=b_j=1\,$. 

\noindent In particular, $\Zsubalgchar_\bullet^{\geqzero}$ is a Hopf subalgebra of $\Uz^{\geqzero}\,$. 
\end{lem}
\begin{proof}
  Setting $N=\ell$ in \eqref{eq:An-Copr-Epwr} and with Corollary~\ref{cor:zeros-mulitnom} we are left with a summation over terms in which one index $r_k=\ell$ and all other indices are zero. For $i<k<j$ we, further, note that
  $(T^{i,j}_k)^\ell=(E_{i,k})^\ell\otimes(K_{i,k}E_{k,j})^{\ell}=\zeta^{\binom{\ell }{ 2}}E_{i,k}^\ell\otimes K_{i,k}^{\ell}E_{k,j}^{\ell}$ since $K_{i,k}E_{k,j}=\zeta^{-1}E_{k,j}K_{i,k}\,$. The resulting formula
  for the coproduct in
  \eqref{eq:An-CoprX} as well as $\Delta(\Kpw_{i,j})=\Kpw_{i,j}\otimes \Kpw_{i,j}$ contain only generators of $\Zsubalgchar_\bullet^{\geqzero}$, so that $\Zsubalgchar_\bullet^{\geqzero}$ is a sub-bialgebra. 

  As noted in the proof of Theorem~\ref{thm:Zwordindep} in Section~\ref{subsec:Mho-Invar-Z}, 
  $\Zsubalgchar_\bullet^{\geqzero}$ is also invariant under $\Kinvaut$. From Item~{\em \ref{item:Kscaleprops:grad}}
  in Proposition~\ref{prop:Kscaleprops} we also know that 
  $\Kscale{u}{\idsymm}(X_{i,j})=\lambda K^{\ell(\epsilon_i-\epsilon_j)}X_{i,j}=\lambda L_{i,j}X_{i,j}\,$,
  where 
  $\lambda=\pm\zeta^a$ for some $a\in\mathbb Z$. Relation \eqref{eq:SXiMhoRel} now implies that $\Zsubalgchar_\bullet^{\geqzero}$ is also invariant under the antipode $S\,$.
\end{proof}

A normalization of generators for which the coproduct structure is defined over $\mathbb{Z}$ is discussed in Appendix \ref{sec:integralcoalg}.
The fact that the full skew-central subalgebra forms a Hopf subalgebra allows us now to prove that the ideals defined in Section~\ref{subsec:Z_Ideals} are indeed Hopf ideals.

\begin{thm}\label{thm:A-ZHopfIdeal} 
For any $s\in\Weyl$ the augmentation ideal  $\Zaugideal{s}^{\geqzero}$ as in \eqref{eq:defAugZideal}
    is a Hopf algebra ideal in $\Zsubalgchar_\bullet^{\geqzero}$. 
\end{thm}
\begin{proof} Since $s\leqRB \longweyl$ there are reduced words $w,b\in\wordset$ with $z=w\cdot b\in \wordsetmax$\,, $\longweyl=\Weylpres(z)$, and $s=\Weylpres(w)$. As in Section~\ref{subsec:descroots}, this imposes a total convex 
ordering $\letwt{z}$ on $\proots$. As in the remarks following \eqref{eq:defAugZideal}, word independence of Theorem~\ref{thm:Zwordindep}
implies that $\Zaugideal{s}^{\geqzero}$ is the ideal generated by all $\Epw_v$ with $\emptyword\neq v\leqRB w$. Given the fixed $z\in\wordsetmax\,$,
we may relabel generators $\Epw_\alpha=\Epw_{z[\alpha]}$ analogous to  \eqref{eq:Genbyroot}. Thus, with $\gamma=\wordroot(w)$ and this notation, $\Zaugideal{s}^{\geqzero}$ is the ideal generated by $\Epw_\beta$ with $\beta\leqwt{z}\gamma\,$. 

Next,  note that, by Proposition~\ref{prop:ZPBWbases}, the monomials 
$\Kpw^\mu\Epwbase{z}{\psi}=\pm\Kpw^\mu\prod_{\alpha\in\proots}\Epw_\alpha^{\psi(\alpha)}$ are a basis for 
$\Zsubalgchar_\bullet^{\geqzero}$. Since, by Lemma~\ref{lem:An-CoprX}, $\Zmaxidchar_\bullet^{\geqzero}$ is a Hopf algebra
over $\Zgen$\,, we can write $\Delta(X_\beta)$ with $\beta\in\proots$ as a $\Zgen$-linear combination of terms of the form 
$\Kpw^{\nu}\Epw_{\beta_1}\ldots\Epw_{\beta_s}\otimes\Kpw^{\mu}\Epw_{\beta_{s+1}}\ldots\Epw_{\beta_k}$\,, where the sequence of roots $\beta_i\in\proots$ may have repetitions.
 
By $\wgrad$-grading, we must have $\beta=\sum_i\beta_i$\,. Lemma~\ref{lem:sumposroots} now implies that for at least 
one $j$ we have $\beta_j\leqwt{z}\beta$. Thus, if $\beta\leqwt{z}\gamma$ also $\beta_j\leqwt{z}\gamma$ and, hence,
that $\Epw_{\beta_j}\in\Zaugideal{s}^{\geqzero}\,$. Therefore, any of these terms is in 
$\Zaugideal{s}^{\geqzero}\otimes \Zsubalgchar_\bullet^{\geqzero}\,+\,\Zsubalgchar_\bullet^{\geqzero}\otimes\Zaugideal{s}^{\geqzero}$\,,
which, thus, contains $\Delta(\Epw_\beta)$ for every generator $\Epw_\beta$ of $\Zaugideal{s}^{\geqzero}$ (with 
$\beta\leqwt{z}\gamma\,$) and, consequently, also $\Delta(\Zaugideal{s}^{\geqzero})\,$.
\end{proof}

The proof of Theorem~\ref{thm:A-ZHopfIdeal} extends to other classical Lie types for which $\Zsubalgchar_\bullet^{\geqzero}$ can be shown to be a Hopf subalgebra, as well as other Lie types to which Lemma~\ref{lem:sumposroots} can be generalized.

The $\Zsubalgchar_w^{\geqzero}$ subalgebras, however, will generally {\em not} form sub-bialgebras. For example, for type $\LT{A}_2$ and $w=w_1w_2=z_2[1,2]$ we have that $\Zsubalgchar_w^{+}$ is the subalgebra generated by $\Epw_{1,3}$ and $\Epw_{2,3}$\,. Yet,
$\Delta(\Epw_{1,3})$ will also contain a $\Epw_{1,2}$ factor.

Conversely, call a subset $C\subseteq\proots$ a {\em lower set} if $\mu+\nu\in C$ for $\mu,\nu\in\proots$ implies $\mu,\nu\in C$. The coproduct in \eqref{eq:An-CoprX} implies that 
$\langle \Epw_\alpha,\Kpw_\alpha:\,\alpha\in C\rangle$ 
is a Hopf subalgebra for any lower set $C$, where the generators are defined with respect to the maximal word from \eqref{eq:An-defzmcmk}. For a given root $\beta=\epsilon_{i}-\epsilon_j=\alpha_i+\ldots+\alpha_{j-1}\,$, let $C_\beta$ be the set of all positive roots $\leq \beta$ in the usual partial order, which is just  the root subsystem $\{\epsilon_s-\epsilon_r:i\leq s<r\leq j\}\,$. It is clear that a subset $C$ has the above property if and only if it is a union of (possibly intersecting) sets $C_\beta\,$. 

One may, further, ask for which $s\in\Weyl$, both, the algebra $\Zsubalgchar_s^{\geqzero}$ is a Hopf subalgebra so that $\Zaugideal{s}^{\geqzero}$ is a Hopf ideal. In the $\LT{A}_n$ case with $\Weyl=S_{n+1}\,$ it is not difficult to see that the root set $\descroots{s}$ for $\Zsubalgchar_s^{\geqzero}$ is a lower set if it is the {\em disjoint} union of sets $C_\beta\,$. Particularly, the lower set property for 
$\descroots{s}$ implies that if $(i,j)$ with $i<j$ is an inversion pair for $s\in S_{n+1}$ then any $(a,b)$ with $i\leq a<b\leq j$ also needs to be an inversion pair. Suppose that $[i,j]$ is a maximal interval with $\epsilon_i-\epsilon_j\in \descroots{s}$. It is then not hard to see that $s$ then maps the intervals $[1,i-1]$, $[i,j]$, and $[j+1,n]$ respectively to themselves and that it acts on $[i,j]$
via the longest (full inversion) element 
of $S_{j-1+1}\,$. 

Thus, by induction, $s$ must be the product of longest elements on disjoint intervals $a_s=[i_s,j_s]$, with $j_s<i_{s+1}\,$. 
The associated subalgebra $\Zsubalgchar^{\geqzero}_s$
is then identified with $\Zsubalgchar^{\geqzero}_\bullet(a_1)\otimes \ldots\otimes \Zsubalgchar^{\geqzero}_\bullet(a_m)$, where
each $\Zsubalgchar^{\geqzero}_\bullet(a_s)$ is the subalgebra generated $\Epw_\alpha$ and $\Kpw_\alpha$ with $\alpha$ in the respective $\LT{A}_{j_s-i_s+1}$ subsystem. Requiring the augmentation ideals $\Zaugideal{s}^{\geqzero}$ to derive from  Hopf algebras, thus,  leaves us with only these obvious examples. 

\medskip

In the remainder of this section, we discuss identifications of the Hopf algebra $\Zsubalgchar_\bullet^{\geqzero}$ as in 
Lemma~\ref{lem:An-CoprX} with (semi) classical Hopf algebra constructions 
that are (mostly) independent of a choice of roots of unity $\zeta$.
These are more conveniently worked out for  the quantum group associated to $\mathfrak{gl}_{n+1}\,$,
which contains the standard one for $\mathfrak{sl}_{n+1}\,$.

The generators of $\Uq(\mathfrak{gl}_{n+1})$ are given by the usual $E_i$ and $F_i$ for $1\leq i \leq n$ as well as 
$\Glg_i^{\pm 1}$ for $1\leq i\leq n+1\,$. The elements $\Glg_i^{\pm 1}$ are required to be group like and to commute with each other. Relations \eqref{eq:EFcomm} and \eqref{eq:corels_gen} are replaced by ones in which we substitute $K_i$ by   $\Glg_i\Glg_{i+1}^{-1}$. The conjugation relations in \eqref{eq:Kcomm} are replaced by 
\begin{equation}\label{eq:qGLrels}
    \Glg_iE_j\Glg_i^{-1}=q^{\symbrack{\epsilon_i}{\alpha_j}}E_j=q^{\delta_{i,j}-\delta_{i,j+1}}E_j
\qquad \mbox{ and } \qquad
\Glg_iF_j\Glg_i^{-1}=q^{-\symbrack{\epsilon_i}{\alpha_j}}F_j=q^{-\delta_{i,j}+\delta_{i,j+1}}F_j\,,
\end{equation}
and the Serre relations \eqref{eq:SerreE} and \eqref{eq:SerreF} remain unchanged.

A $\mathbb Z$-grading $\pgrad$ on $\Uq(\mathfrak{gl}_{n+1})$ is provided by setting $\pgrad(P_i)=1$ and $\pgrad(E_i)=\pgrad(F_i)=0\,$ for all $i=1,\ldots,n\,$.
The original algebra $\Uq(\mathfrak{sl}_{n+1})$ is then recovered as the 0-graded subalgebra of $\Uq(\mathfrak{gl}_{n+1})$ with respect to $\pgrad\,$. The inclusion map of $\Uq(\mathfrak{sl}_{n+1})$ into $\Uq(\mathfrak{gl}_{n+1})$ is now given by 
$K_i\mapsto\Glg_i\Glg_{i+1}^{-1}$ on the Cartan subalgebra and identity on all other generators.

The specialization of $\Uq(\mathfrak{gl}_{n+1})$ to a root of unity $\zeta$ is commensurable with this inclusion. The analogous 
skew-central subalgebra $\ZsubalgcharT^{\geqzero}_\bullet$ is generated by $\Glgpw_i=\Glg_i^\ell$ as well as the $\Epw_{i,j}\,$ as before. The $\pgrad$-grading extends to  $\ZsubalgcharT^{\geqzero}_\bullet$ with $\pgrad(\Glgpw_i)=\ell\,$.

The ground ring $\Zgen$ may be any extension of $\Zz=\mathbb Z[\zeta]$, such as the ones from \eqref{eq:defzetarings}.
If $(\zeta-\zeta)^{-1}$ is not in $\Zgen$\,, we denote by $\ZsubalgcharT^{\geqzero}_{\bullet\;\vaccent}$ the subalgebra generated
by $\Glgpw_i=\Glg_i^\ell$ and $(\zeta-\zeta^{-1})^{\ell}\Epw_{i,j}\,$. Thus, for example for $\Zgen=\Zz$\,, the quotient 
$\ZsubalgcharT^{\geqzero}_\bullet/\ZsubalgcharT^{\geqzero}_{\bullet\;\vaccent}$ is a torsion module for $\Zgen\,$. 
Conversely, if the ground ring extends $\Zzv$ we have $\ZsubalgcharT^{\geqzero}_\bullet=\ZsubalgcharT^{\geqzero}_{\bullet\;\vaccent}\,$.

As for the generic case, the original subalgebra $\Zsubalgchar^{\geqzero}_\bullet$ is identified with the zero-graded component of 
$\ZsubalgcharT^{\geqzero}_\bullet$ with respect to $\mathbb p\,$. Analogously, we write $\Zsubalgchar^{\geqzero}_{\bullet\;\vaccent}$
for the subalgebra generated by the $\Kpw_j$ and $(\zeta-\zeta^{-1})^{\ell}\Epw_{i,j}\,$, or, equivalently, the zero-graded
part of $\ZsubalgcharT^{\geqzero}_{\bullet\;\vaccent}\,$.

We next consider a smaller variant of the extended quantum matrix algebra introduced by Artin, Schelter, and Tate in \cite{AST91}.
Suppose $\,\skewcoeffchar:\{1,\ldots,n\}^2\rightarrow\Zgen^\times\,$ are coefficients for which $\skewcoeff{i}{j}=\skewcoeff{j}{i}^{-1}$ and $\skewcoeff{i}{i}=1\,$. Define $\QMatAlg{\Zgen}{\skewcoeffchar}{n}$ to be 
the unital algebra over $\Zgen$ with generator set $\{R_{i,j}:1\leq i\leq j\leq n\}\sqcup\{R_{i,i}^{-1}:i=1,\ldots,n\}\,$ and relations
\begin{equation}\label{eq:defQMatAlg}
    R_{i,j}R_{s,t}\,=\,\frac{\skewcoeff i s}{\skewcoeff j t}\,R_{s,t}R_{i,j} \qquad\mbox{and } \qquad R_{i,i}^{-1}R_{i,i}=1=R_{i,i}R_{i,i}^{-1}
\end{equation}
for all valid indices. This algebra differs from the one in \cite{AST91} in that we only use the coefficients on the subgroup
$\UtriM{n}$ of upper triangular matrices in $\mathrm{GL}(n,\mathbb C)\,$, and that an extra parameter there is $\lambda=1\,$. 
The algebra  $\QMatAlg{\Zgen}{\skewcoeffchar}{n}$ may also be viewed as a central extension of the polynomial 
algebra $\Zgen[\{x_{i,j},x_{i,i}^{-1}\}_{i\leq j}]$, in   which the $\skewcoeffchar$ expression features as the respective 2-cocyle. 

The bialgebra portion of the following result is a straightforward adaption of Theorems~1 and 2 in \cite{AST91}. As opposed
to the situation in \cite{AST91}, an antipode is assured by the invertibility of the diagonal elements and does not require a
determinant relation. Its explicit form below is readily verified by direct computation. 

\begin{prop}\label{prop:QMathAlgHopf}
The algebra $\QMatAlg{\Zgen}{\skewcoeffchar}{n}$ is a free $\Zgen$ module with basis given by monomials $\prod_{i\leq j}R_{i,j}^{n_{i,j}}$\,, where $n_{i,i}\in\mathbb Z$ and $n_{i,j}\in\nnN$ for $1\leq i<j\leq n\,$ (for some fixed ordering).

\noindent Further, $\QMatAlg{\skewcoeffchar}{\Zgen}{n}$ admits the structure of a Hopf algebra over $\Zgen$ with coproduct defined by 
\begin{equation}\label{eq:QMatCoprod}
    \Delta(R_{i,j})=\sum_{k=i}^jR_{i,k}\otimes R_{k,j}\;. 
\end{equation} 
The counit is $\varepsilon(R_{i,j})=\delta_{i,j}\,$.
The unique antipode is given by $S(R_{i,i})=R_{i,i}^{-1}$ and, for $j>i$, by 
$$
S(R_{i,j})=\longsum[20]_{\substack{1\leq m\leq j-i\\ i=s_0<s_1<\ldots<s_m=j}}(-1)^m\,R_{s_0,s_0}^{-1}R_{s_0,s_1}R_{s_1,s_1}^{-1}R_{s_1,s_2}\ldots
R_{s_{m-1},s_{m-1}}^{-1}R_{s_{m-1},s_m}R^{-1}_{s_m,s_m}\,.
$$

\end{prop}

Two special cases of fully symmetric and anti symmetric coefficient sets are given by 
\begin{equation}\label{eq:specskewcoeff}
    \skewcoeffgen{\mathsf s}{i}{j}=1 \qquad \mbox{and} \qquad \skewcoeffgen{\mathsf a}{i}{j}=-(-1)^{\delta_{ij}}\,. 
\end{equation}
We will also abbreviate $\QMatAlgS{\Zgen}{n}=\QMatAlg{\mathsf s}{\Zgen}{n}$ for the classical matrix algebra. Note also, 
that the coefficients in \eqref{eq:specskewcoeff} are units in $\mathbb Z$ and the structure coefficients in Proposition~\ref{prop:QMathAlgHopf} are integers as well. Thus, the matrix algebras can be defined over $\mathbb Z$ so that
\begin{equation}\label{eq:QMatCoeff}
    \QMatAlgS{\Zgen}{n}=\QMatAlgS{\mathbb Z}{n} \otimes \Zgen\qquad \mbox{ and  } \qquad 
    \QMatAlg{\mathsf a}{\Zgen}{n}=\QMatAlg{\mathsf a}{\mathbb Z}{n}\otimes \Zgen\,. 
\end{equation}

The algebra $\QMatAlg{\skewcoeffchar}{\Zgen}{n}$ admits many integer gradings. One of them, given on generators as $\mathbb n(R_{i,j})=\epsilon_j-\epsilon_i$\,, extends to a Hopf algebra grading. Another is the total polynomial degree, given by $\mathbb q(R_{i,j})=1$ and $\mathbb q(R_{i,i}^{-1})=-1$, which does not extend to a Hopf algebra grading. Nevertheless, it is not hard to see that the zero 
graded subalgebra, spanned by monomials $M$ with $\mathbb q(M)=0$, forms indeed a Hopf subalgebra, which we denote as
$\SQMatAlg{\skewcoeffchar}{\Zgen}{n}\,$. Similarly, we abbreviate $\SQMatAlgS{\Zgen}{n}=\SQMatAlg{\mathsf s}{\Zgen}{n}\,$.

The following identifications of the skew-central subalgebras assume that $\Zgen$ contains $\Zz$\,, where $\zeta$ is a primitive $\kay$-th root of unity.  

\smallskip 

\begin{thm}\label{thm:Z=O}
Let $\ZsubalgcharT^{\geqzero}_{\bullet\;\vaccent}$ and $\Zsubalgchar^{\geqzero}_{\bullet\;\vaccent}$  be the subalgebra of $\Uz(\mathfrak{gl}_n)$ and $\Uz(\mathfrak{sl}_n)$ as defined above.
There exist isomorphisms of Hopf algebras over $\Zgen$ as follows.
\begin{equation}
    \begin{aligned}
        \ZsubalgcharT^{\geqzero}_{\bullet\;\vaccent}\,&\cong\,\QMatAlgS{\Zgen}{n} &\quad & \mbox{ and } &\quad \Zsubalgchar^{\geqzero}_{\bullet\;\vaccent}\,&\cong\,\SQMatAlgS{\Zgen}{n}&\quad & \mbox{ if } \; \kay\not\equiv 2\mod 4\\
         \ZsubalgcharT^{\geqzero}_{\bullet\;\vaccent}\,&\cong\,\QMatAlg{\mathsf a}{\Zgen}{n} &\quad & \mbox{ and } &\quad \Zsubalgchar^{\geqzero}_{\bullet\;\vaccent}\,&\cong\,\SQMatAlg{\mathsf a}{\Zgen}{n} &\quad & \mbox{ if } \; \kay\equiv 2\mod 4
    \end{aligned}
\end{equation}
\end{thm}

\begin{proof} Set $b=(\zeta^{-1}-\zeta)^\ell\zeta^{\binom{\ell}{ 2}}$ and consider the assignments of generators of $\QMatAlg{\skewcoeffchar}{\Zgen}{n}$ to generators of 
$\ZsubalgcharT^{\geqzero}_{\bullet\;\vaccent}$ given by 
\begin{equation}\label{eq:TZmap}
    R_{i,j}\mapsto b \Glgpw_i^{-1}\Epw_{i,j} \quad\mbox{for}\;i<j \qquad \mbox{ and } \qquad R_{i,i}\mapsto\Glgpw_i^{-1}\,. 
\end{equation}
Using $\Glgpw_i\Glgpw_j^{-1}=\Glg_i^\ell\Glg_j^{-\ell}=K_{i,j}^\ell=L_{i,j}$ and the coproduct
formula in \eqref{eq:An-CoprX} we find
\begin{align}
\Delta(b\Glgpw_i^{-1}\Epw_{i,j})&=b(\Glgpw_i^{-1}\otimes\Glgpw_i^{-1})\Big(1\otimes \Glgpw_i\Glgpw_i^{-1}\Epw_{i,j}+\sum_{i<k<j}b\Epw_{i,k}\otimes \Glgpw_i\Glgpw_k^{-1}\Epw_{k,j}\,+\,\Epw_{i,j}\otimes \Glgpw_i \Glgpw_j^{-1}\Big)\nonumber\\
& =\Glgpw_i^{-1}\otimes b\Glgpw_i^{-1}\Epw_{i,j}+\sum_{i<k<j}b\Glgpw_i^{-1}\Epw_{i,k}\otimes b\Glgpw_k^{-1}\Epw_{k,j}\,+\,b\Glgpw_i^{-1}\Epw_{i,j}\otimes  \Glgpw_j^{-1}\,.
\nonumber
\end{align}
That is, the assignment preserves the corelation of generators from \eqref{eq:QMatCoprod}. 

The commutation relations~\ref{eq:XYcommrels} specialize to $\Epw_\alpha\Epw_\beta=\LLsign^{\symbrack{\alpha}{\beta}}\Epw_\beta\Epw_\alpha\,$ in the simply laced case. Here
$\LLsign=\zeta^{\ell^2}$, as in \eqref{eq:defspeczetas}, is $\LLsign=-1$ if $\,\kay\equiv 2\mod 4$ and  $\LLsign=+1$ otherwise. From \eqref{eq:qGLrels}
 we also have $\Glgpw_i\Epw_\alpha=\LLsign^{\symbrack{\epsilon_i}{\alpha}}\Epw_\alpha\Glgpw_i\,$. Thus, $\Glgpw_i^{-1}\Epw_\alpha\Glgpw_s^{-1}\Epw_\beta=\LLsign^{\symbrack{\alpha}{\epsilon_s}+\symbrack{\epsilon_i}{\beta}+\symbrack{\alpha}{\beta}}\Glgpw_s^{-1}\Epw_\beta\Glgpw_i^{-1}\Epw_\alpha\,$. 
Setting $\alpha=\epsilon_i-\epsilon_j$ and $\beta=\epsilon_s-\epsilon_t$
with $i\leq j$ and $s\leq t$ we find the exponent is 
$\delta_{i,s}+\delta_{j,t}=(1+\delta_{i,s})+(1+\delta_{j,t})\mod 2\,$.

Thus 
$(\Glgpw_i^{-1}\Epw_{i,j})
 (\Glgpw_s^{-1}\Epw_{s,t})
  =\frac{\skewcoeff i s}{\skewcoeff j t}
  (\Glgpw_s^{-1}\Epw_{s,t})
  (\Glgpw_i^{-1}\Epw_{i,j})$, 
  where $\skewcoeff i j=\LLsign^{1+\delta_{i,j}}$. 
  Clearly, for $\LLsign=1$ this matches the commutation relations in $\QMatAlg{\mathsf s}{\Zgen}{n}$ and for $\LLsign=-1$ the ones for $\QMatAlg{\mathsf a}{\Zgen}{n}\,$.  Consequently, the assignment 
   in \eqref{eq:TZmap} extends to an algebra homomorphism, with the previous observation, also to a bialgebra homomorphism, and, using uniqueness of antipodes, to a Hopf algebra homomorphism. 

It is clear that this homomorphism maps PBW bases to each other, yielding an isomorphism of Hopf algebras and, thus, proving the statements for
the $\Uz(\mathfrak{gl}_n)$ subalgebras. The respective statements for the $\Uz(\mathfrak{sl}_n)$ are inferred from the fact
that the assignment in \eqref{eq:TZmap} maps a generator with $\mathbb q$-grading of 1 to an element with $\mathbb p$-grading $-\ell$. Therefore, the isomorphism restricts to an isomorphism of the respective zero-graded Hopf subalgebras. 
\end{proof}

Theorem~\ref{thm:Z=O} together with \eqref{eq:QMatCoeff} implies  that in the skew-central subalgebras are determined by Hopf algebras that are {\em independent} of the choice of a root of unity, except for the $\!\mod 4$-parity of its order.  

\begin{conj}\label{conj:Z=HopfSubBorel} 
For a given root system $\roots$, denote by $B$ the Borel subgroup of the respective complex, simply connected simple algebraic group and by $\breve B$ the Borel subgroup so associated to the coroot system $\breve\roots$. Let further 
$\ZsubalgcharT^{\geqzero}_{\bullet\;\vaccent}$ be the analogous skew-central subalgebra in an extension of $\Uz^\geqzero$ for a primitive $\kay$-th root of unity $\zeta\,$. 
\vspace*{-3mm}

\begin{enumerate} [label=\roman*)] 
    \vspace{2mm}
        \item \label{item:Z=HopfSubBorel:A}
     $\ZsubalgcharT^{\geqzero}_{\bullet\;\vaccent}$ is also a Hopf subalgebra for any $\kay$, isomorphic to a centrally extended matrix algebra, generalizing the definition of \cite{AST91} and associated to either $B$ or $\breve B$ or a finite quotient of these.  
    \vspace{-2mm}
        \item \label{item:Z=HopfSubBorel:B}
     This extended matrix algebra depends up to isomorphism only on parity of $\kay$ modulo $4\maxd$\,.
\end{enumerate} 
\end{conj}
 
\medskip

\subsection{Ideals from Algebraic Subgroups of \texorpdfstring{$\UtriM{n}$}{Tn} and Bruhat Subgroups}
\label{subsec:AnBruhat} In this section we focus entirely on type $\LT{A}_{n-1}$ algebras and the case 
$\kay\not\equiv 2\mod 4$, for which  $\Zsubalgchar^\geqzero_\bullet$
is commutative. In this setting,
Theorem~\ref{thm:Z=O} yields an isomorphism between commutative rings, whose scalars we can be further extended to a field $\mathbb k=\mathbb Q(\zeta)$ or $\mathbb k=\mathbb C\,$, assuming an embedded ground ring $\Zgen\hookrightarrow\mathbb k\,$.   

As before, let $\UtriM{n}<\mathrm{GL}(n,\mathbb k)$ be the algebraic group of upper triangular matrices over $\mathbb k$. Denote also $\mathrm{P}\UtriM{n}<\mathrm{PGL}(n,\mathbb k)$ the respective subgroup of the projective linear group. We will use the following identifications with coordinate rings. 

\begin{cor}\label{cor:AnCoordids}
Assume $\kay\not\equiv 2\mod 4$ and Lie types $\mathfrak{gl}_{n}$ or $\mathfrak{sl}_{n}$ as above. Then the isomorphisms from Theorem~\ref{thm:Z=O} composed with the assignment $R_{i,j}\mapsto\cofun{x_{i,j}}\,$ yield isomorphisms of Hopf algebras
$$
\ZsubalgcharT^\geqzero_\bullet\otimes\mathbb k \,\cong\, \mathbb k[\UtriM{n}]
\qquad\mbox{and}\qquad 
\Zsubalgchar^\geqzero_\bullet\otimes\mathbb k \,\cong \,\mathbb k[\mathrm{P}\UtriM{n}]\,. 
$$
\end{cor}

\begin{proof} Clearly, $\ZsubalgcharT^\geqzero_\bullet\otimes\mathbb k=\ZsubalgcharT^\geqzero_{\bullet\;\vaccent}\otimes\mathbb k$ and $\Zsubalgchar^\geqzero_\bullet\otimes\mathbb k=\Zsubalgchar^\geqzero_{\bullet\;\vaccent}\otimes\mathbb k$.
It is also easy to see that $\QMatAlgS{\mathbb k}{n}\cong \mathbb k[\UtriM{n}]$ as commutative Hopf algebras by comparing coproducts in (\ref{eq:HopfStrTn}) and (\ref{eq:QMatCoprod}). On the diagonal the assignment extends uniquely to $R_{i,i}^{-1}\mapsto\cofun{x_{i,i}}^{-1}$. Recall that ${\mathrm P}\UtriM{n}$ is defined as the quotient $\UtriM{n}/\mathrm C$ by the center $\mathrm C=\mathbb k^\times\cdot 1$ so that 
$\mathbb k[{\mathrm P}\UtriM{n}]$ embeds in $\mathbb k[\UtriM{n}]$. Functions in $\mathbb k[{\mathrm P}\UtriM{n}]$ are those constant on each $\mathrm C$-coset. For given $\lambda\in \mathbb k^\times$, $g\in\UtriM{n}$\,, and $\cofun{x}^f=\prod_{i\leq j}\cofun{x_{i,j}}^{f_{ij}}$ a general monomial, one readily verifies that
$\cofun{x}^f(\lambda g)=\lambda^{\sum_{i\leq j}f_{ij}}\cdot \cofun{x}^f(g)$. Thus 
$\cofun{x}^f\in \mathbb k[\mathrm{P}\UtriM{n}]$ iff $\sum_{i\leq j}f_{ij}=0$. 
Since $\mathbb q(\prod_{i\leq j}R_{i,j}^{f_{ij}})=\sum_{i\leq j}f_{ij}$\,, 
the isomorphism $\mathbb k[\UtriM{n}]\rightarrow\QMatAlgS{\mathbb k}{n}$ maps 
$\mathbb k[\mathrm{P}\UtriM{n}]$ precisely to $\SQMatAlgS{\mathbb k}{n}$, which is thus
identified with $\Zsubalgchar^\geqzero_\bullet\otimes\mathbb k$.
\end{proof}

Note that the coalgebra structure for $\mathbb k[\mathrm{P}\UtriM{2}]\cong\mathbb k[\hat\lambda,\hat\lambda^{-1},\hat x]$ is given by
(\ref{eq:xlambcop}), whereas the coordinate ring of the Borel part in the 2-fold cover $\widetilde{\mathrm{P}\UtriM{2}}<\mathrm{SL}(2,\mathbb C)$ is of the form 
$\mathbb k[\hat\chi,\hat\chi^{-1},\hat x]$
with coproduct $\Delta(\hat x)=\hat x\otimes\hat\chi+\chi^{-1}\otimes\hat x$. 
Clearly, the two groups and their respective coordinate rings are {\em not} isomorphic. 
To avoid these types of subtleties, we will confine our discussion below mostly to the case 
$\UtriM{n}<\mathrm{GL}(n,\mathbb C)$. Corollary~\ref{cor:AnCoordids} permits us now to
defined large families of Hopf ideals via the composition of maps
\begin{equation}\label{eq:HopfComp}
 \mathscr h:\;   \Zsubalgchar^\geqzero_\bullet\,\hookrightarrow\,\ZsubalgcharT^\geqzero_\bullet\,
\rightarrow\,\ZsubalgcharT^\geqzero_\bullet\otimes\mathbb C\,\stackrel{\cong}{\longrightarrow}\,\mathbb C[\UtriM{n}]\,. 
\end{equation}

\begin{lemma}\label{lem:HIdeal=subgr}
Assume Lie type $\LT{A}_{n-1}$\,, $\Zgen\hookrightarrow\mathbb C$ and $\kay\not\equiv 2\mod 4$ as above. Let $H\leq \UtriM{n}$ be any subgroup of the complex subgroup of upper triangular matrices. Then 
\begin{equation}\label{eq:HIdeal=subgr}
 \mathcal A(H)\,=\,\left\{z\in\Zsubalgchar^\geqzero_\bullet\,:\;\mathscr h(z)(g)=0\;\forall\,g\in H\right\} 
\end{equation}
is a Hopf ideal in $\Zsubalgchar^\geqzero_\bullet\,$. 
Moreover, 
$\mathcal A(H)$ depends only on the (Zariski) closure of $H\,$ and for two closed algebraic subgroups $H,K\leq \UtriM{n}$ we have  $\mathcal A(H\cap K)=\mathcal A(H)+\mathcal A(K)\,$.
\end{lemma}

\begin{proof}  The definition in \eqref{eq:HIdeal=subgr} can be rephrased as
$\mathcal A(H)=\mathscr h^{-1}(\VanIdeal(H))$, with notation for the vanishing ideal of $H$ in $\mathbb C[\UtriM{n}]$ as in Section~\ref{subsec:AlgGrps-basics}. Since $\mathscr h$ is a Hopf algebra homomorphism as a composite of such, the preimage of a Hopf ideal is again a Hopf ideal. The other properties follow directly from 
Lemma~\ref{lm:ASG=HI}.  
\end{proof}

Subgroups of $P\leq \UtriM{2}$ were discussed and classified already in  Corollary~\ref{cor:ClassHI-A1}. 
More general examples of algebraic subgroups may be structured using $\UtriM{n}=\DiagM{m}\ltimes \SUtriM{n}$\,, 
where $\DiagM{m}$ and $\SUtriM{n}$ are as in Section~\ref{subsec:AlgGrps-basics}. Algebraic subgroups $A\leq \DiagM{m}$
are given as products of finite groups and copies of $\mathbb C^\times\,$, yielding further
abelian subgroups $A'=gAg^{-1}$ for any $g\in \SUtriM{n}\,$. Any such $A'$ acts on $\SUtriM{n}$ by conjugation,
and any algebraic subgroup $V\leq \SUtriM{n}$ that is $A'$-stable yields an algebraic subgroup $A'V\leq \UtriM{n}$\,. 
Respective Hopf ideals in $\Zsubalgchar^\geqzero_\bullet\,$ for $A\SUtriM{n}=A'\SUtriM{n}$ are generated by 
expressions of the form $\prod_iL_i^{k_i}-1$, while the ideals for general $A'V$ are typically
generated also by linear expressions in the coordinate functions.

 More relevant for our purposes are algebraic subgroups closely related to the Bruhat decomposition $G=\bigsqcup_{s\in\Weyl} C(s)$ with Bruhat cells $C(s)=B\dot s B\,$ defined as double cosets. Here $G$ is a connected complex reductive algebraic group, 
 $B\leq G$ a Borel subgroup, and $\dot s\in N(T)$ a representative of 
$s\in\Weyl\cong N(T)/T$, where $T\leq B$ is a maximal torus of $G\,$ and $N(T)$ its normalizer. We will call the following the {\em Bruhat subgroup} associated to an element $s\in\Weyl$.
\begin{equation}\label{eq:def:BruhatGrps}
    B(s)\,=\,B\,\cap\, \dot s B\dot s^{-1}\,.
\end{equation}

To explain the relation to the Bruhat decomposition, let $U\leq B$ be the usual unipotent group such that $B$ is a semidirect product of $B=D\ltimes U$ with the Cartan subgroup $D\geq T$ 
(in the sense of  $(\mathbb C^\times)^n\geq (S^1)^n$). 
Analogous to \eqref{eq:def:BruhatGrps}, set $U(s)=U\cap \,\dot s U \dot s^{-1}=B(s)\cap U\,$, which is generated by subgroups $\{U_\alpha:\,\alpha\in\descroots{\flweyl s}\}$, where $\flweyl{s}$ is as in \eqref{eq:def:flweyl}. For each $s\in\Weyl$, the map $U(\flweyl s)\rightarrow C(s)/B:\,u\mapsto [u\dot s]$ is a bijection, thus relating the collection $\{U(s)\}$ of subgroups to the Bruhat partition of the coset space $G/B\,$. 
Note that,  
$U(\flweyl s)=U\cap \,\dot s U^- \dot s^{-1}$, where $U^-=\dot\longweyl U\dot\longweyl^{-1}$ is the respective group associated to negative roots.  
For details see, for example, \cite[Ch30]{Bu04}, \cite[\S8.3]{Sp98}, or \cite[\S10.3,N]{Co17}.

  The augmentation ideals $\Zaugideal{s}$, defined in \eqref{eq:defAugZideal}, are next identified with the vanishing ideals on the Bruhat subgroups.  Here, the latter are defined with respect to $G=\mathrm{GL}(n,\mathbb C)$, $B=\UtriM{n}$\,, and $U=\SUtriM{n}$ as above.

\begin{thm}\label{thm:AB=Ks}
For Lie type $\LT{A}$\,, $\kay\neq 2\mod 4$, and any $s\in\Weyl$ we have 
    $$
    \mathcal{A}(B(s))= \Zaugideal{s}^{\geqzero}\;. 
    $$ 
\end{thm}

\begin{proof} Note first that for $s\in\Weyl=S_{n}$ the representative $\dot s$ is given by the usual permutation matrix,
acting on the canonical basis by $\dot s e_i=e_{s(i)}\,$. Suppose $L\in \UtriM{n}$\,, meaning that $L_{i,i}\neq 0$ and $L_{i,j}=0$ for $i>j\,$. The additional condition $L\in \dot s B\dot s^{-1}$ is then equivalent
to $L_{i,j}=0$ if $i<j$ but $s^{-1}(i)>s^{-1}(j)\,$. Thus, with $\alpha=\epsilon_i-\epsilon_j\in\proots$, the latter is equivalent to $s^{-1}(\alpha)=\epsilon_{s^{-1}(i)}-\epsilon_{s^{-1}(j)}\in\nroots$ and, hence, $\alpha\in\descroots{s}\,$. It is thus clear that
a polynomial in $\cofun{x_{i,j}}$ and $\cofun{x_{i,i}}^{-1}$ is zero on $B(s)$ if and only if it is
in the ideal generated by $\{\cofun{x_{i,j}}:\,\epsilon_i-\epsilon_j\in \descroots{s}\}\,$.
The $\mathscr h$-preimage of this set is then the ideal generated by 
$\{\Epw_\alpha:\,\alpha\in \descroots{s}\}\,$, which is precisely $\Zaugideal{s}^{\geqzero}\,$. 
\end{proof}

It is important to keep in mind that the definition of $\mathcal A(H)$ 
in Lemma~\ref{eq:HopfComp}
depends on the isomorphisms in Theorem~\ref{thm:Z=O}. These, in turn, depend on the choice of the ordering 
 in \eqref{eq:AnCanOrd} which determine the $E_{i,j}$ and thus $\Epw_{i,j}$ generators. The equality
 in Theorem~\ref{thm:AB=Ks} implies now shared properties of the identified ideals. 

\begin{cor}The following statements hold.  \vspace{-1mm}\label{cor:AB=Ks}
    \begin{enumerate}[label=\roman*)]
        \item\label{item:AB=Ks:A} 
        $\mathcal A(H)$ does not depend on a choice of ordering or isomorphism if $H$ is a Bruhat subgroup. \vspace{2mm}
        \item \label{item:AB=Ks:B} 
        If $\kay\not\equiv 2\mod 4\,$,
              then $\widehat\Zaugideal{s}^\geqzero$ is a Hopf ideal in $\Uz^\geqzero(\mathfrak{sl}_n)$
              and $\Uz(\mathfrak{sl}_n)$ for any $s\in\Weyl\,$. 
    \end{enumerate}
\end{cor}

The first statement is immediate from word-independence of $\Zaugideal{s}$, which itself a consequence of Theorem~\ref{thm:Zwordindep} and the definition in  \eqref{eq:defAugZideal}.

For the second assertion recall the notation $\widehat J=J\cdot\Uz^\utypechar\,$ from Proposition~\ref{prop:Zideal}, extending ideals in the subalgebras $\Zsubalgchar_\bullet^\geqzero$ to ideals in the respective quantum algebras. $\Commsigngrp$-invariance 
trivially holds in the commutative case. It is clear that, if $\Zsubalgchar_\bullet^\geqzero$ is a Hopf subalgebra and $J$ a Hopf ideal in  $\Zsubalgchar_\bullet^\geqzero$, then $\widehat J$ is a Hopf ideal in the quantum algebra as well. By Theorem~\ref{thm:AB=Ks} and Lemma~\ref{lem:HIdeal=subgr} this is the case for $J=\Zaugideal{s}^\geqzero$ for any $s\in\Weyl\,$.

A natural question related to the first item in Corollary~\ref{cor:AB=Ks} asks, for which 
closed algebraic subgroups $C\leq B$ the resulting ideal $\mathcal A(C)$ can be 
constructed without the choice of a particular ordering. 
The observation in Item {\em \ref{item:AB=Ks:B} } of the corollary will be proved for all Lie types in  Theorem~\ref{thm:MainIdeals} below.

The full $\ZsubalgcharT_\bullet=\ZsubalgcharT_\bullet^\geqzero\cdot \ZsubalgcharT_\bullet^\leqzero$ subalgebra naturally also forms a Hopf subalgebra, which, over $\mathbb C$, can also be identified with the coordinate ring of the algebraic group $\dBsg{n}=\mathrm{GL}(n,\mathbb C)^\PLexp$ as 
defined in (\ref{eq:DefPoissonLie}). One Borel algebra $B^+=\UtriM{b}$ is given by the 
usual upper triangular matrices and the opposite $B^-=\UtriM{n}^-$ by the lower triangular ones. The algebraic maximal torus $\DiagM{n}=\UtriM{n}^-\cap \UtriM{n}$ is, hence, given by  
the diagonal matrices in $\mathrm{GL}(n,\mathbb C)$. Rephrasing definition (\ref{eq:DefPoissonLie})
in this notation, we  have  
$$
\dBsg{n}=\ker\left(\UtriM{n}^-\times \UtriM{n}\,\longrightarrow\,\DiagM{n}\times \DiagM{n}\,\longrightarrow\, \DiagM{n}\right)\,.
$$

The construction also implies $\dBsg{n}\cong\DiagM{n}\ltimes (\SUtriM{n}^-\times \SUtriM{n})$, where $\SUtriM{n}^-$ and $\SUtriM{n}$ are the respective unipotent subgroups. The next result extends 
Corollary~\ref{cor:AnCoordids}. Although the existence of a Hopf algebra isomorphism is already implied in  
\cite{dck92} for odd $\kay$, we construct the map here explicitly via assignments of the polynomial generators of the form $R_{i,j}^{\pm}\mapsto\cofun{x_{i,j}}^{\pm}\,$ and also include the case
$\kay\equiv 0\mod 4$.

\begin{lem}\label{lem:fullZ=Qn}
Suppose $\zeta$ is a $\kay$-th root of unity and $\Uz=\Uz(\mathfrak{gl}_n)$ as defined above. 
 \vspace{-1mm}
    \begin{enumerate}[label=\roman*)]
\item $\ZsubalgcharT_\bullet$ is a Hopf subalgebra of $\Uz\,$.\vspace{2mm}
\item For $\kay\not\equiv 2\mod 4$ there is an isomorphism of Hopf algebras $\ZsubalgcharT_\bullet\otimes\mathbb C\cong\mathbb C[\dBsg{n}]$.\vspace{2mm}
\item Moreover, for such $\kay$, any closed algebraic subgroup of $\dBsg{n}$ defines a Hopf ideal in $\ZsubalgcharT_\bullet$ and, hence, a Hopf ideal of $\Uz\,$. 
\end{enumerate} 
\end{lem}

\begin{proof} Recall that the involution $\Cartanaut:\Uz\rightarrow\Uz^{\cop}$
from \eqref{eq:defCaut}  
maps $\Uz^{\leqzero(\geqzero)}$ to $\Uz^{\geqzero(\leqzero)}$ with the opposite coproduct. In particular,
since $\ZsubalgcharT_\bullet^\geqzero$ is a Hopf subalgebra, then so is $\ZsubalgcharT_\bullet^\leqzero=\Cartanaut(\ZsubalgcharT_\bullet^\geqzero)\,$. The product map
 $\ZsubalgcharT_\bullet^\leqzero\otimes \ZsubalgcharT_\bullet^\geqzero\rightarrow \ZsubalgcharT_\bullet$
 is a surjective Hopf algebra homomorphism with kernel given by the Hopf ideal $J=\sum_i\lpip\Kpw_i^{-1}\otimes \Kpw_i-1\rpip\,$. Composing this with $\Cartanaut\otimes \id$, we obtain a surjection
 $\ZsubalgcharT_\bullet^{\geqzero,\cop}\otimes \ZsubalgcharT_\bullet^\geqzero\rightarrow \ZsubalgcharT_\bullet$
 with kernel $J'=\sum_i\lpip\Kpw_i\otimes \Kpw_i-1\rpip\,$. Over $\mathbb C$ the isomorphism from
 Corollary~\ref{cor:AnCoordids}  allows us to replace $\ZsubalgcharT_\bullet^\geqzero$ with $\mathbb C[\UtriM{n}]$
 and $\ZsubalgcharT_\bullet^{\geqzero,\cop}$  with  $\mathbb C[\UtriM{n}^{\opp}]$, where $\UtriM{n}^{\opp}$ is naturally
 identified with $\UtriM{n}^-$ via transposition. 
 
 The kernel of the resulting epimorphism 
 $\mathbb C[\UtriM{n}^-\times \UtriM{n}]=\mathbb C[\UtriM{n}^-]\otimes\mathbb C[\UtriM{n}]\rightarrow \ZsubalgcharT_\bullet$
 is now the ideal $J''=\sum_i\lpip(\cofun{x_{i,i}},\cofun{x_{i,i}})-1\rpip=\sum_i\lpip\cofun{x_{i,i}}\otimes\cofun{x_{i,i}}-1\rpip$. The subgroup $\VanLocus(J'')$ now consists of $(a,b)\in \UtriM{n}^-\times \UtriM{n}$ such that
 $\cofun{x_{i,i}}(a)\cofun{x_{i,i}}(b)=1$ for all $i=1,\ldots, n\,$. The set of these pairs is now exactly
 the kernel $\dBsg{n}$ defined above. Lemma~\ref{lm:ASG=HI} then implies $J''=\VanIdeal(\dBsg{n})$ and thus
 $\mathbb C[\dBsg{n}]\cong\ZsubalgcharT_\bullet\,$. The last statement is now also immediate from
 Lemma~\ref{lm:ASG=HI}. 
 \end{proof}

Consider the natural center action of $\mathbb C^\times$ on $\dBsg{n}$\,, given by
$\lambda.(b^-,b^+)=(\lambda^{-1}b^-,\lambda b^+)$ with $b^-\in \UtriM{n}^-$ and $b\in \UtriM{n}\,$. The respective 
$\mathbb C^\times$ action on the coordinate ring of $\dBsg{n}$ over $\mathbb C$ yields a direct decomposition
$\mathbb C[\dBsg{n}]=\bigoplus_{k\in\mathbb Z}\mathbb C[\dBsg{n}]_k$ with $f\in\mathbb C[\dBsg{n}]_k$
characterized by $\lambda.f=\lambda^k\!\cdot\! f\,$. It is easy to check that this decomposition coincides
with the one imposed by the $\pgrad$-grading introduced in Section~\ref{subsec:Zn=CAn}. 

Thus, for $\kay\not\equiv 2\mod 4$, the isomorphism in Lemma~\ref{lem:fullZ=Qn} restricts to an isomorphism 
of Hopf algebras between $\Zsubalgchar_\bullet\otimes\mathbb C$ and $\mathbb C[\dBsg{n}]_0$\,, the $\mathbb C^\times$-invariant part of the coordinate ring, which, in turn, may be identified with
$\mathbb C[\mathrm{PGL}(n,\mathbb C)^\PLexp]\,$. 

The Artin group $\mathscr A=B_n$ 
action on $\Uq(\mathfrak{sl}_{n})$ generated by the $\Tinv_i$ naturally extends to
 an action on $\Uq(\mathfrak{gl}_{n})$ by $\Tinv_s(\Glg_i)=\Glg_{s(i)}$\,, where $i\mapsto s(i)$ is the 
 usual permutation for $s\in\Weyl=S_n\,$. Lemma~\ref{lem:fullZ=Qn} now implies an action of $B_n$ by 
 algebra automorphisms of $\mathbb C[\dBsg{n}]$\,. Invoking the standard correspondence between algebra homomorphisms of
 coordinate rings and morphisms of algebraic varieties, this implies an action of $B_n$ on $\dBsg{n}$\,.

\begin{cor}\label{cor:BnActOnWn}
    There exists a non-trivial action of the braid group $B_n$ on $n$ strands by algebraic automorphisms 
    on the algebraic group 
    $\dBsg{n}$\,, commensurable with Lusztig's $B_n$-action on the quantum group of Lie type $\LT{A}_{n-1}\,$. 
\end{cor}

It is, worth contrasting the correspondence in Lemma~\ref{lem:fullZ=Qn} for the central subalgebras with that of the traditional correspondences for the original quantum groups. For the latter $\Uq^{\geqzero}$ and
$\Uq^{\leqzero}$ are associated with the opposing Borel subgroups and $\Uq=\Uq^{\geqzero}\cdot\Uq^{\leqzero}$
with the original simple or reductive Lie group. 

However, for the central subalgebras, the associated Lie group $\dBsg{n}$ is a solvable group and thus
fundamentally different from $\mathrm{GL}(n,\mathbb C)$\,. Correspondingly, the coordinate ring 
$\mathbb C[\dBsg{n}]\cong\ZsubalgcharT_\bullet$ is cosolvable in the sense of Corollary~\ref{cor:kB=cosolv},
whereas $\mathbb C[\mathrm{GL}(n,\mathbb C)]$ is cosemisimple since $\mathrm{GL}(n,\mathbb C)$ is reductive. 
Similarly, the of group algebraic characters $\AChar{\dBsg{n}}\cong\GrLi{\ZsubalgcharT_\bullet}\cong\mathbb Z^n$ differs from the one in 
\eqref{eq:GLnChar}.

Lemma~\ref{lem:fullZ=Qn} allows us to associate, for example, the Hopf ideal 
$\VanIdeal((K\times H)\cap \dBsg{n})$ to any pair  of subgroups 
$H\leq \UtriM{n}$ and $K\leq \UtriM{n}^-\,$, which then define Hopf ideals $\mathcal A(K,H)$ in 
$\Zsubalgchar_\bullet$ as before. These may then be reexpressed as 
$\mathcal A(K,H)=\mathcal A(K,\UtriM{n})+\mathcal A(\UtriM{n}^-,H)=(\mathcal A(K)^\leqzero+\mathcal A(H)^\geqzero)\cdot \Zsubalgchar_\bullet\,$. Specialized to the Bruhat subgroups this yields, together with 
Theorem~\ref{thm:AB=Ks} and a respective definition of a Bruhat subgroup $B(t)^-<\UtriM{n}^-$\,, the identity of ideals 
\begin{equation}
    \mathcal A(B(t)^-,B(s))=(\Zaugideal{t}^{\leqzero}+\Zaugideal{s}^{\geqzero})\cdot\Zsubalgchar_\bullet\qquad
    \mbox{ for any } s,t\in\Weyl.
\end{equation}
   
\medskip 

Concluding our discussion of the $\LT{A}_{n-1}$ algebras, we briefly address the case $\kay\equiv 2\mod 4$, when $\LLsign=-1$, as well as an approach to the next simply laced Lie types. 

We first provide a concrete description of Tanisaki's isomorphism from \cite{tan16} for the subgroup $\UtriM{n}<\mathrm{GL}(n,\mathbb C)$\,. Let ${\mathsf I}_s\in \UtriM{n}$  be the diagonal matrix with
$({\mathsf I}_s)_{s,s}=-1$ and $({\mathsf I}_s)_{i,i}=1$ for $i\neq s\,$. For any $1\leq s\leq t\leq n$ define the
map $\tau_{s,t}:\UtriM{n}\rightarrow\UtriM{n}$ by $\tau_{s,t}(b)={\mathsf I}_s\cdot b \cdot {\mathsf I}_t\,$. Clearly, all $\tau_{s,t}$ are of order 2 and commute with each other, generating an
elementary abelian 2-group $\mathscr G$  of isomorphisms of $\UtriM{n}$ as an algebraic variety. 
Obviously, none of the $\tau_{s,t}$ are a group homomorphism of $\UtriM{n}$ and 
the quotient variety $\UtriM{n}/\mathscr G$ retains no group structure either.

\begin{lem} Assume $\kay\equiv 2\mod 4$ and Lie type $\LT{A}_{n-1}$\,. Let $\Zsubalgchar_{\mathrm{Fr}}^\geqzero$ be the intersection of $\ZsubalgcharT_\bullet^\geqzero\otimes\mathbb C$ with the center of $\Uzn{\mathbb C}$\,. Then there is a weight-preserving algebraic isomorphism of commutative algebras,
$$
\Zsubalgchar_{\mathrm{Fr}}^\geqzero\,\cong \,\mathbb C[\UtriM{n}]^{\mathscr G}\,\cong\,
\mathbb C[\UtriM{n}/{\mathscr G}]\,.
$$ 
\end{lem}

\begin{proof}
    It follows from the commutation relations in Section~\ref{subsec:CommPrimGen} that $\Zsubalgchar_{\mathrm{Fr}}^\geqzero$ is identical with the center of  $\ZsubalgcharT_\bullet^\geqzero\otimes\mathbb C\cong \QMatAlg{\mathsf a}{\mathbb C}{n}$ itself. 
    Denote $R^f=\prod_{i\leq j}R_{i,j}^{f_{i,j}}$ 
    for some fixed ordering and 
    $\cofun{x}^f=\prod_{i\leq j}\cofun{x_{i,j}}^{f_{i,j}}$ as before. Recall that $R_{i,j}R_{s,t}=(-1)^{\delta_{i,s}+\delta_{j,t}}R_{s,t}R_{i,j}$ and 
    observe that $\tau_{s,t}^*(\cofun{x_{i,j}})=(-1)^{\delta_{i,s}+\delta_{j,t}}\cofun{x_{i,j}}\,$.
    It is  clear from these formulae that
    $R^f$ is central if and only if $\cofun{x}^f$ is $\mathscr G$-stable. Write $M\subseteq \mathbb Z^{\binom{n+1 }{ 2}}$ for the set of exponents $f=(f_{i,j})$ for which this is the case, namely,
    those for which $\sum_{s\leq j}f_{s,j}+\sum_{i\leq t}f_{i,t}\equiv 0\mod 2$ for all $s\leq t\,$.  
    Depending on the chosen ordering, it now follows from commutativity that 
    one can then choose a quadratic form $\vartheta$ on 
    $M\otimes\mathbb Z$ such that for $\tilde R^f=(\scalebox{.75}{$\sqrt{-1}\,$})^{\vartheta(f)}R^f$ we have 
    $\tilde R^f\cdot \tilde R^g=\tilde R^{f+g}$. Thus, $\cofun{x}^f\mapsto \tilde R^f$
    defines an isomorphism of algebras over $\mathbb C$\,. 
\end{proof}

As an elementary example, it is easy to see that $\mathbb C[\UtriM{2}]^{\mathscr G}$ corresponds to the subalgebra generated
by $R_{1,1}^2$\,, $R_{1,2}^2$\,, and $R_{2,2}R_{1,1}^{-1}$\,, which is clearly not a Hopf subalgebra. Also, 
we have $\mathscr G\cong\mathbb F_2\times \mathbb F_2$ acting component-wise on $\UtriM{2}\cong (\mathbb C^\times)^2\times\mathbb C$ by $(x_{11},x_{22})\mapsto(-x_{11},-x_{22})$ and 
$x_{12}\mapsto-x_{12}$\,, respectively. So, $\UtriM{2}/\mathscr G$ has no obvious group structure. 

As an alternative approach to finding an interpretation of $\QMatAlg{\mathsf a}{\mathbb C}{n}$
and thus $\ZsubalgcharT_\bullet^\geqzero\otimes\mathbb C$
we may seek to retain the coordinate ring $\mathbb C[\UtriM{n}]$ as a target set, but modify the 
multiplication and comultiplication maps as weighted averages with respect to some 
$\mathscr G$-action. We propose the following form of a possible Hopf algebra structure. 

\begin{conj} There exist modified multiplication ${\scriptstyle\star}$ and comultiplication
$\ddot\Delta$ on $\mathbb C[\UtriM{n}]$ of the form
\begin{equation}\label{eq:SignedPolynAlg}
    \phi{\scriptstyle\star}\psi=\sum_{\sigma,\tau}\mu(\sigma,\tau)\,\tau^*(\phi)\sigma^*(\psi)
\qquad
\mbox{and}
\qquad
\ddot\Delta(\phi)(a,b)=\sum_{\sigma,\tau}\eta(\sigma,\tau)\,\phi(\sigma(a)\tau(b))\,,
\end{equation}
such that $(\mathbb C[B_n],{\scriptstyle\star},\ddot\Delta)$ is isomorphic to $\QMatAlg{\mathsf a}{\mathbb C}{n}\,$.
Here summations are over algebraic involutions $\sigma:\UtriM{n}\rightarrow \UtriM{n}$ with 
$\sigma^*(\cofun{x_{i,j}})=\pm\cofun{x_{i,j}}$ and the weights $\mu$ and $\eta$ take values
in $\mathbb Z[\frac 12]\,$. 
\end{conj}

Again, the involutions are not required to be homomorphisms.
In this form, the definition of Hopf ideals as vanishing ideals on subgroups
then requires that these subgroups are 
preserved under the groups of involutions. This is indeed the case for the Bruhat subgroups. 
We mention that versions of the specific quantum matrix algebras $\QMatAlg{\mathsf a}{\Zgen}{n}$ are recovered also in Proposition~2 of \cite{Cl94} as images of a Frobenius map. 

Overall, however, the skew matrix algebras $\QMatAlg{\mathsf a}{\Zgen}{n}$ are scarcely studied in the literature despite their natural occurrences, lacking any semi-classical interpretations one might expect. 
Following \cite{AST91} the algebra $\QMatAlg{\mathsf a}{\Zgen}{n}$ is constructable from the skew coordinate algebra
generated by $\{x_i:i=1,\ldots, n\}$ subject to relations $x_ix_j=-x_jx_i$ if $i\neq j\,$. Particularly, the generators $x_i$
remain of infinite order and the algebra is fundamentally distinct from finite-dimensional exterior or Clifford algebras.
There, thus, appear to be no obvious identifications of the skew matrix algebras with function algebras on super or spin groups. 

\smallskip

Among the simply laced Lie types, $\LT{D}_n$ has several well understood 
relations to $\LT{A}$ type algebras.
Clearly, it contains an $\LT{A}_{n-2}$ algebra and we, conjecturally, expect there to be a
quotient of the respective algebras from the $\LT{A}_{2n-1}$ onto the $\LT{D}_n$\,. The derivations of explicit coproduct formulae
as those in (\ref{eq:DeltaEij}) and (\ref{eq:An-CoprX}), however, are considerably more challenging.

Kharchenko \cite{Karch} provides explicit coproduct formulae for the generic $\LT{D}_n$ algebra with coefficients over $\Zqq$ reminiscent of the ones in
\eqref{eq:DeltaEij} for type $\LT{A}$\,. Computations of powers to arrive
at the analogues of \eqref{eq:An-CoprX}, however, are substantially more 
complicated since two types of basis generators are involved. They are 
defined via recursions as the one in \eqref{eq:higher_roots} but in opposing directions,
and neither subset of generators spans a subcoalgebra.

\section{The Structure of \texorpdfstring{$\Zsubalgchar$}{Z} and Hopf Ideals for Type \texorpdfstring{$\mathsf{B}_2$}{B2}}
\label{sec:Ideals_B2=SO5}

The accessible computational tools for $\mathfrak{gl}_n$ allowed us to verify the claims in 
Conjectures~\ref{conj:Z=HopfSubBorel} and \ref{conj:CoordRings} computationally for all 
$\LT{A}_{n-1}$ algebras. As additional evidence for these conjectures, we prove them 
in this section also for the simplest non-simply laced  Lie type $\LT{B}_2=\LT{C}_2$\,.
The descriptions of isomorphisms and subspaces will be in very explicit terms, 
while computations are much more involved. The latter are assured to play a central 
role in extensions to other doubly-laced Lie types, as they illustrate
additional methods and principles for the construction of isomorphisms. 

Throughout this section, we assume that the order of $\zeta$ is $\kay>4$ or $\kay=3$, 
but impose no other restrictions on $\kay$. As pointed out in 
Remark~\ref{rem:B2C2commcentr} for $\LT{B}_2$\,, the subalgebra
$\Zsubalgchar^\geqzero_\bullet$ is commutative for {\em all} $\kay$, 
though central  in $\Uz$ only if $\kay$ is not a
multiple of 4. 

In Section~\ref{subsec:Coalg-ZB2} we describe the explicit Hopf algebra structure of $\Zsubalgchar^\geqzero_\bullet$, 
deferring the computation of coproducts for non-simple roots to Section~\ref{subsec:ProofCoprodB2}.
The fact that $\Zsubalgchar$ is a Hopf subalgebra may be indirectly inferred also from \cite{dck92} 
for odd  $\kay$, though no explicit formulae or computations are given there, leaving  open 
also the case of even roots of unity. 
The Hopf algebra structure of the Borel
coordinate ring $\mathbb C[\mathrm{SO}_5^\geqzero]$ is explicitly determined in Section~\ref{subsec:BorelCordSO5}.
A similarly explicit  isomorphism of $\mathbb C[\mathrm{SO}_5^\geqzero]$ with $\Zsubalgchar^\geqzero_\bullet\otimes\mathbb C$ is established in Section~\ref{subsec:IsomHA-B2}.

\subsection{Coalgebra Structure of \texorpdfstring{$\Zsubalgchar^\geqzero_\bullet$}{Z.≥0} for Type \texorpdfstring{$\mathsf{B}_2$\,}{B2}}\label{subsec:Coalg-ZB2}
This section provides the 
full list of 
coproduct formulae for the generators of $\Zsubalgchar^\geqzero_\bullet$ for Lie type  
$\LT{B}_2\,$, preceded by a  review of essential notation and conventions.

As in Section~\ref{subsec:rank2coxsys}, we denote by $\alpha_i$ the short root of $\LT{B}_2$ and by
$\alpha_j$ the long root. Thus, $s_j(\alpha_i)=\alpha_i+\alpha_j$ is the second short root and 
$s_i(\alpha_j)=2\alpha_i+\alpha_j$ the second long root. The default ground ring for the generic 
$\LT{B}_2$ algebra is 
$\Zgen=\Zqqn{3}=\mathbb Z\textstyle{\left[q^{\pm 1},[2]^{-1}\right]}$ as defined in \eqref{eq:defgenrings}. 

Recall the formulae for the standard generators with respect to $\longtwoword{ij}$ as defined in \eqref{eq:LusztigDictionary}, which follow readily from the automorphisms in \eqref{eq:defLserre} and \eqref{eq:defTserre},  as well as their $\wgrad$-gradings in $\mathbb Z^{\sroots}$. 
\begin{align}
     E_{(iji)}&=\ELu_{(ji)}=\Tinv^{-1}_j(E_i)=q^{-2}E_iE_j-E_jE_i 
           & \wgrad(E_{(iji)})&=\alpha_i+\alpha_j \label{eq:B2Egeniji}
           \\
     \rule{0mm}{7mm}
    E_{(ij)}&=\Tinv_i(E_j)=\frac{1}{[2]}(E_iE_{(iji)}-E_{(iji)}E_i)& \wgrad(E_{(ij)})&=2\alpha_i+\alpha_j \label{eq:B2Egenij}
\end{align}

From the coproducts for $E_i$ and $E_j$\,, as in \eqref{eq:corels_gen}, the following expressions for the coproducts of the other two generators are easily derived. Here $\Esing_i$ and 
$\Esing_j$ are singularized generators as defined in \eqref{eq:def:EgenSing}.
\begin{align}
    \Delta(E_{(iji)})&=E_{(iji)}\otimes K_iK_j-E_i\otimes K_i\Esing_j+1\otimes E_{(iji)}
    \label{eq:B2coprod-iji}
    \\
    \rule{0mm}{7mm}
    \Delta(E_{(ij)})&=E_{(ij)}\otimes K_i^2K_j+q\Esing_i^2\otimes K_i^2E_j-\Esing_i\otimes K_iE_{(iji)}+1\otimes E_{(ij)}
    \label{eq:B2coprod-ij}
\end{align}

Choose next a $\kay$-th primitive root of unity $\zeta$ with $\kay$ as above and $\ell$ the order of $\zeta^2$. We also reemploy the notation from \eqref{eq:gcdle} and \eqref{eq:gcdleopp} with
$\,(\gcdle,\gcdleopp)=(1,2)\,$ if $\ell$ is odd and $\,(\gcdle,\gcdleopp)=(2,1)\,$ if $\ell$ is even.

The terms occurring in the coproduct formulae for the primitive power generators $\Epw_w=E_w^{\ell_w}$ depend on the parity of $\ell$ as they are differently  
 constrained by the $\wgrad$-grading. More precisely, for $\alpha=\wordroot(w)$\,,
the observation in \eqref{eq:wgradX_rvscor} implies that $\,\ell^{-1}\wgrad(\Epw_w)$ is given by $\alpha$ if $\ell$ is odd and
the respective coroot $\breve\alpha$ if $\ell$ is even. 
Recall that the coroot system $\breve \roots_{\LT{B}_n}\,$ is isomorphic to the 
root system $\roots_{\LT{C}_n}\,$. For $n=2$ the latter is isomorphic to  
$\roots_{\LT{B}_2}\,$, yielding a combined root system isomorphism
$\mathscr f:\breve \roots_{\LT{B}_2}\rightarrow \roots_{\LT{B}_2}$.

The composite $\mathscr q(\alpha)=\mathscr{f}(\breve\alpha)$ then defines an involution
in $\roots_{\LT{B}_2}^+$ that exchanges long with short roots and which is given by 
$\mathscr q(\alpha_i)=\alpha_j$\,, 
$\mathscr q(\alpha_j)=\alpha_i$\,, $\mathscr q(\alpha_i+\alpha_j)=2\alpha_i+\alpha_j$\,,
and $\mathscr q(2\alpha_i+\alpha_j)=\alpha_i+\alpha_j\,$. Thus, if we define
$\wgradZBtwo(\Epw_w)=\alpha$ for odd $\ell$ and $\wgradZBtwo(\Epw_w)=\mathscr q(\alpha)$
if $\ell$ is even, we obtain a grading on $\Zsubalgchar^\geqzero_\bullet$ that is 
isomorphic to the induced $\wgrad$-grading, but, for either parity, conveniently valued 
in $\roots_{\LT{B}_2}^+\,$. As such, the coproduct needs to preserve the $\wgradZBtwo$-grading. 

For more convenient reference, 
we list in the table below the $\wgradZBtwo$-weights for the primitive power generators
for words less than $\longtwoword{ij}$ together with the special constants. 
\begin{equation}\label{eq:tab:uweightsB2}
    \begin{array}{ccccccc}
         \qquad&\quad \wgradZBtwo(\Epw_{i}) \qquad&\quad \wgradZBtwo(\Epw_{j})
        \qquad&\quad \wgradZBtwo(\Epw_{(ij)})\qquad&\quad \wgradZBtwo(\Epw_{(iji)}) 
        \qquad&\quad \gcdle\qquad&\quad \gcdleopp\qquad  
         \\
   \rule[-8pt]{0pt}{21pt} \ell \mbox{ odd}  & \alpha_i& \alpha_j
     & 2\alpha_i+\alpha_j& \alpha_i+\alpha_j
     & \quad 1 &\quad 2
   \\
     \ell \mbox{ even} & \alpha_j& \alpha_i
     & \alpha_i+\alpha_j& 2\alpha_i+\alpha_j
     & \quad 2 & \quad 1
    \end{array}
\end{equation}

We state next the main result of this section, assuming notation and conventions as above as well as the notation $\delzeta{j}$ from \eqref{eq:def-delzeta}.

\begin{prop}\label{prop:CoprodB2}
The coproducts of $\Epw_w$ for $\emptyword\neq w\leqRB \longtwoword{ij}$ are as follows:
\begin{align}\label{eq:CoprodB2:simple}
    \Delta(\Epw_i)&= \Epw_i\otimes \Kpw_i+1\otimes \Epw_i\,,
    &\Delta(\Epw_j)&= \Epw_j\otimes \Kpw_j+1\otimes \Epw_j\,,
\end{align}
\begin{align}
    \phantom{\Delta(\Epw_{(iji)})}&\begin{aligned}
        \mathllap{\Delta(\Epw_{(iji)})}&=
1\otimes \Epw_{iji}-\delzeta{j}^{\gcdle} \Epw_i\otimes \Kpw_i \Epw_j^{\gcdle}+\Epw_{(iji)}\otimes \Kpw_i \Kpw_j^{\gcdle}
    \\&
    \qquad-\delta_{\gcdle,2}\cdot 2\cdot [2]^\ell  (-\zeta)^\ellj \Epw_{(ij)}\otimes \Kpw_i \Kpw_j\Epw_j\,,
    \label{eq:CoprodB2:iji}
    \end{aligned}
    \\
    \rule[-8pt]{0pt}{33pt}
    &\begin{aligned}
        \mathllap{\Delta(\Epw_{(ij)})}&=
    1\otimes \Epw_{(ij)}+\zeta^{\binom{2\ellj}{ 2}}\delzeta{i}^\gcdleopp \Epw_i^{\gcdleopp}\otimes \Kpw_i^{\gcdleopp}X_j+X_{(ij)}\otimes \Kpw_i^{\gcdleopp}L_j
    \\&\qquad -
    \delta_{\gcdleopp,2}\cdot 2\cdot  \zeta^{\ell}[2]^{-\ell}\delzeta{i}\cdot  \Epw_i\otimes \Kpw_i\Epw_{(iji)}\,.
    \label{eq:CoprodB2:ij} 
    \end{aligned}
\end{align}
\end{prop}

As usual, $\delta_{a,b}$ is 1 if $a=b$ and 0 is otherwise. In cases when  $\delta_{\gcdle,2}$
or $\delta_{\gcdleopp,2}$ are zero, the respective terms are indeed not consistent with
the $\wgradZBtwo$-grading. For example, if $\ell$ is even and $\gcdleopp=1$ the 
$\wgradZBtwo$-grading of $\Epw_i\otimes\Epw_{(iji)}$ is $(\alpha_j)+(2\alpha_i+\alpha_j)=2(\alpha_i+\alpha_j)$\,, while $\wgradZBtwo(\Epw_{(ij)})=\alpha_i+\alpha_j\,$. 

The expressions in \eqref{eq:CoprodB2:simple} are restatements of \eqref{eq:simplecoprod}.
The explicit proofs of \eqref{eq:CoprodB2:ij} and \eqref{eq:CoprodB2:iji}
are postponed to the following section. 
The formulae in Proposition~\ref{prop:CoprodB2} also involve only elements in
$\Zsubalgchar^\geqzero_\bullet$, implying the latter is a sub-bialgebra. 

Explicit expressions for the antipode can be derived from the axioms, applied as 
$0=\epsilon(\Epw_{(iji)})=\mathrm{m}\circ(\id\otimes S)\circ \Delta(\Epw_{(iji)})$ and $0=\epsilon(\Epw_{(ij)})=\mathrm{m}\circ(S\otimes \id)\circ \Delta(\Epw_{(ij)})$\,.
\begin{align}
S(\Epw_{(ij)})L_i^{\gcdleopp} L_j^{}
&=
-\Epw_{(ij)} -(-1)^{\gcdleopp}\zeta^{\binom{2\ellj}{2}}\delzeta{i}^{\gcdleopp} \cdot \Epw_i^{\gcdleopp}\Epw_j^{}
- \delta_{\gcdleopp,2}\cdot 2\cdot  \zeta^{\ell}[2]^{-\ell}\delzeta{i}
\cdot \Epw_i\Epw_{(iji)} \,\label{eq:AnitpB2ij} 
\\
\rule{0mm}{6.5mm}
S(\Epw_{(iji)})L_iL_j^{\gcdle}&=-\Epw_{(iji)}+(-1)^\gcdle \delzeta{j}^\gcdle  \Epw_i \Epw_j^{\gcdle} -\delta_{\gcdle,2} 2\cdot [2]^\ell  (-\zeta)^\ellj  \Epw_{(ij)}\Epw_j \,\label{eq:AnitpB2iji}
\end{align}

Proposition~\ref{prop:CoprodB2} and the above formulae thus provide further evidence for
the first parts of Conjectures~\ref{conj:Z=HopfSubBorel} and \ref{conj:CoordRings}.

\begin{cor}\label{cor:ZB2=Hopfsubalg}
    For Lie type $\LT{B}_2$ and any $\kay$ as above,  $\Zsubalgchar^\geqzero_\bullet$ is a Hopf subalgebra over $\Zzn{3}\,$. 
\end{cor}

We remark that the Hopf algebra structure computed above provides an alternative 
proof for the statement of Propositions~\ref{prop:XJ-form} and \ref{prop:XI-form} for $\edgenum=2$\,.
Specifically, recall from \eqref{eq:SXiMhoRel} that the antipode is related 
to the $\mho$-involution via $\Kscale{(\htsign\cdot\qitgen)^{-1}}{\idsymm}\circ S\,=\,\Kinvaut\,$. As $\Kscale{(\htsign\cdot\qitgen)^{-1}}{\idsymm}(\Epw_{(iji)})=(-1)^{\ell+1} \Epw_{(iji)}L_iL_j^{\gcdle}$ and $\Kscale{(\htsign\cdot\qitgen)^{-1}}{\idsymm}(\Epw_{(ij)})=(-1)^{\ell}\Epw_{(ij)}L_i^\gcdleopp L_j$\,, applying $\Kscale{(\htsign\cdot\qitgen)^{-1}}{\idsymm}$
to \eqref{eq:AnitpB2ij} and \eqref{eq:AnitpB2iji} yields the relations below.
\begin{align}
(-1)^{\ell}\Kinvaut(\Epw_{(ij)})&=-\Epw_{(ij)}-(-1)^\gcdleopp \zeta^{\binom{2\ellj}{ 2}}\delzeta{i}^\gcdleopp \Epw_i^\gcdleopp \Epw_j -
    \delta_{\gcdleopp,2}\cdot 2\cdot  \zeta^{\ell}[2]^{-\ell}\delzeta{i}\cdot \Epw_{i}\Epw_{(iji)} \label{eq:mhofS-ij}
\\
\rule{0mm}{6.5mm}
(-1)^\ell\Kinvaut(\Epw_{(iji)})&=
 \Epw_{(iji)} -(-1)^{\gcdle}\delzeta{j}^{\gcdle} \Epw_i\Epw_j^{\gcdle}+\delta_{\gcdle,2}\cdot 2\cdot [2]^\ell  (-\zeta)^\ellj \Epw_{(ij)}\Epw_j \,\label{eq:mhofS-iji}
\end{align}

Specializing relation \eqref{eq:Xr2=mho} to $s=2$ and $s=3$ and writing the explicit
words as in \eqref{eq:def:abwords} yields $\mho(\Epw_{(ij)})=\Epw_{b_{m-1}}=\Epw_{(jij)}$
and $\mho(\Epw_{(iji)})=\Epw_{b_{2}}=\Epw_{(ji)}\,$. Comparison of
coefficients then yields \eqref{eq:XI-form} and \eqref{eq:XJ-form} for 
$\edgenum=2\,$. 

See also Appendix \ref{sec:integralcoalg} for an integral form of the coalgebra structure in Proposition \ref{prop:CoprodB2}.

\subsection{Proof of Proposition~\ref{prop:CoprodB2} for Non-simple Roots}\label{subsec:ProofCoprodB2}
The proof of the formulae \eqref{eq:CoprodB2:ij} and \eqref{eq:CoprodB2:iji} relies on a generalized quantum multinomial formula. It involves a special type of quantum number, which we introduce first. 
For $k,t\in\bbn_0$\,, define $c_{k,t}\in\Zqq$ recursively in $t$ by  
    \begin{equation}\label{eq:cRecursion}
         c_{k,t}=\sum_{i=0}^k c_{i,t-1} q^{-(2t-2+i)}[2t-1+i]\;,
    \end{equation}
    with initial conditions $c_{k,0}=1$. For $k,t>0$ one readily derives the relations
    \begin{equation}\label{eq:ckt-special}
        c_{k,t}=c_{k-1,t}+c_{k,t-1}q^{-(2t-2+k)}[2t+k-1]
        \qquad\mbox{and}\qquad
        c_{0,t}=q^{-t(t-1)}[2t-1]!\mkern -1.6mu !\;,
    \end{equation}
    denoting as usual the quantum double factorial $[a]!\mkern -1.6mu !=[a][a-2]\cdots [a-2\floor{(a-1)/2}]$. The following identity is a generalization of \cite[Lemma~1.6]{Lu90a}, where it applies only to the simply laced situation.  

\begin{prop}\label{prop:PowerFormula}
Suppose $A,B,C,D$ satisfy
    \begin{align*}
        DA=q^{-2}AD+B, && CA=q^{-2}AC, && BA=q^{-4}AB, && DC=q^{-2}CD, && DB=q^{-4}BD\,  && CB=BC\,.
    \end{align*}
    Then $(A+C+D)^n$ is equal to
    \begin{align*}
        \sum_{s=0}^n\sum_{t=0}^{\min(s,n-s)}\sum_{k=0}^{n-s-t}q^{-(s+t+k)(n-s-t-k)-(s-t)(2t+k)}\qbin{n}{s+t+k}{}\qbin{s+t+k}{s-t}{}c_{k,t}\,A^{s-t}B^tC^kD^{n-s-t-k}\,.
    \end{align*} 
\end{prop}

The technical proof of this proposition is given in Appendix~\ref{app:powerformula}. 
For the following computations at generic $q$ we consider the quantum algebra $\Uq^\geqzero$ over
the ground ring $\Zqqn{3}$ for Lie type $\mathsf{B}_2\,$. Additional commutation relations follow, for example, from Lemma~\ref{lm:scommgens} with $E_{(ijij)}=E_j\,$:
\begin{equation}\label{eq:EB2-qcomms}
    E_{(iji)}E_{(ij)}=q^2E_{(ij)}E_{(iji)}\,,\quad
    E_{(ij)}E_i=q^2E_iE_{(ij)},
    \quad\mbox{and}\quad 
    E_jE_{(iji)}=q^2E_{(iji)}E_j\,.
\end{equation} 
We first compute the coproduct of the generator for the non-simple short root.

\begin{proof}[{Proof of Equation \eqref{eq:CoprodB2:iji}}] We begin with the computation of $\Delta(E_{(iji)})^n$ for generic $q\,$. Proposition~\ref{prop:PowerFormula} is applied by setting $C=0$ as well as 
\begin{equation}
    A=E_{(iji)}\otimes K_iK_j  
     \;\mbox{,}\quad
    D=D_1+D_2
    \quad\mbox{with}\quad 
    D_1=(q^{-2}-q^2)E_i\otimes K_iE_j
     \quad\mbox{and}\quad
    D_2=1\otimes E_{(iji)}\,, 
\end{equation}
so that $\Delta(\Esing_{(iji)})=A+D\,$. One easily checks that $D_2D_1=q^2D_1D_2$\,. We obtain identities 
\begin{equation}\label{eq:Dm-qbin}
    D^m=\sum_{u=0}^m q^{u(m-u)}\qbin{m}{u}{} D_1^uD_2^{m-u}\qquad\mbox{and}\qquad
    B=q^{-2}[2](q^{-2}-q^2)E_{(ij)}\otimes K_i^2K_jE_j\,,
\end{equation}
where the former now follows from \eqref{eq:q-binom-form} and the latter from 
$B=DA-q^{-2}AD$, as stipulated in Proposition~\ref{prop:PowerFormula}. The other required
commutation relations among $A$, $B$, and $D$ are derived from \eqref{eq:B2Egeniji} and 
\eqref{eq:EB2-qcomms}. Applying 
the proposition to the case $C=0$ forces $k=0$ in the summation, which yields, together
with \eqref{eq:Dm-qbin} and \eqref{eq:ckt-special}, the following expression.
\begin{equation}\label{eq:(A+D)PowN}
    \begin{aligned}
    \Delta(E_{(iji)}^n)&=(A+D)^n=\sum_{s=0}^n\sum_{t=0}^{\min(s,n-s)}q^{-(s+t)(n-s-t)-(s-t)(2t)}\qbin{n}{s+t}{}\qbin{s+t}{s-t}{}c_{0,t}\,A^{s-t}B^tD^{n-s-t}  
    \\
    \rule[-8mm]{0mm}{19mm}&=\sum_{s=0}^n\sum_{t=0}^{\min(s,n-s)}\sum_{u=0}^{n-s-t} q^{\mathfrak e(n,s,t,u)}[2t-1]!\mkern -1.6mu !\qbin{n}{s+t}{}\qbin{s+t}{s-t}{} \qbin{n-s-t}{u}{}
    A^{s-t}B^tD_{1}^uD_{2}^{n-s-t-u} \;,
    \\
    \mbox{where}&\quad \mathfrak e(n,s,t,u)=-ns-nt+s^2+2t^2+t+un-us-ut-u^2\,.
\end{aligned} 
\end{equation}\vspace{-3mm}

Next, we specialize $q$ to a primitive $\kay$-th root of unity $\zeta\,$ with $\kay>4$ or $\kay=3\,$. We further set $n=\ell$ with $\ell$ the order of $\zeta^2\,$. The product of the three binomial coefficients in \eqref{eq:(A+D)PowN} is readily identified with a multinomial coefficient as in 
\eqref{eq:def-qmultinom} for $r=4$, $a_1=\ell-s-t-u$, $a_2=s-t$, $a_3=2t$, and $a_4=u$. Corollary~\ref{cor:zeros-mulitnom} now implies that the only contributing non-zero terms are given by the following combinations of indices.
\[
(s,t,u)\in\bigl\{
(0,0,0), (0,0,\ell), (\ell,0,0),(\ell/2,\ell/2,0)\bigr
\}\,.
\]
Of course, the term $(\ell/2,\ell/2,0)$ only contributes when $\ell$ is even, meaning $\gcdle=2\,$. Summing over these four index-tuples gives 
\begin{align*}
    \Delta(X_{(iji)})=&~D_{2}^{\ell}+D_{1}^\ell +A^{\ell}+\delta_{\gcdle,2}\cdot \zeta_j^{-\binom{\ellj }{ 2}}[\ell-1]!\mkern -1.6mu !\cdot B^{\ell/2}
    \\
    \rule{0mm}{7mm}
    =&~
    1\otimes E_{(iji)}^\ell+(\zeta^{-2}-\zeta^2)^\ell \zeta_j^{\binom{\ell}{ 2}} \cdot E_i^\ell\otimes K_i^\ell E_j^{\ell}+E_{(iji)}^\ell\otimes K_i^\ell K_j^{\ell}
    \\&-\delta_{\gcdle,2}\cdot ([2](\zeta^{-2}-\zeta^{2}))^{\ell/2}\zeta_j^{-\binom{\ellj}{ 2}}\qfacrelA \ellj 2 \cdot E_{(ij)}^\ellj\otimes K_i^\ell K_j^{\ell_j}E_j^\ellj
    \\
    \rule{0mm}{7mm}
    =&~
    1\otimes X_{(iji)}-\delzeta{j}^{\gcdle} \cdot X_i\otimes L_i X_j^{\gcdle}+X_{(iji)}\otimes L_i L_j^{\gcdle}
    \\&-\delta_{\gcdle,2}\cdot 2\cdot (-\zeta)^\ellj [2]^\ell\cdot  X_{(ij)}\otimes L_i L_jX_j\;.
\end{align*}
The last step uses Lemma~\ref{lm:factratformZ-J} to simplify coefficients, yielding 
\eqref{eq:CoprodB2:iji} as the final
expression. 
\end{proof}

The proof of the coproduct formula in the case of the non-simple long root is more involved. 
The required commutation relations are more
naturally expressed in terms of $q_j=q^2$ rather than $q_i=q\,$. Denoting by  $\mathsf{sq}:\Zqq\rightarrow\Zqq$ the homomorphism that maps $q$ to $q^2$ (that is, $q_i$ to $q_j$) we  define 
$c'_{k,t}=\mathsf{sq}(c_{k,t\,})$. 
 
\begin{proof}[{Proof of Equation \eqref{eq:CoprodB2:ij}}] As before, we start with the computation of
$\Delta(\Esing_{(ij)}^n)$ for generic $q\,$. The main tool for this proof is the identity in Proposition~\ref{prop:PowerFormula} with $q$ replaced by $q_j$\,, obtained by applying
$\mathsf{sq}$ to all coefficients, using $\mathsf{sq}\left(\qbinsmall{a}{b}\right)=\qbinsmall{a}{b}_j\,$.
We begin by defining summands 
\begin{align*}
    A'&=q(q-q^{-1})^2E_i^2\otimes K_i^2E_j\,, 
    &
    C'&=-(q-q^{-1})E_i\otimes K_iE_{(iji)}\,,
    \\
    \rule{0mm}{5.6mm}
    D'&=1\otimes E_{(ij)}\,,
    &
    G&=E_{(ij)}\otimes K_i^2K_j\,,
    \\
    \rule[-4mm]{0mm}{10mm}&\mbox{so that}&  \Delta(E_{(ij)})&=A'+C'+D'+G  &&\mbox{by \eqref{eq:B2coprod-ij}.}&&\nonumber
\end{align*}

Commutation relations between these follow again from  
\eqref{eq:B2Egenij} and \eqref{eq:EB2-qcomms}. We find
\begin{align} 
    GA'&= q_j^{2}A'G\,,
    \qquad 
    GC'= q_j^{2}C'G\,,
    \quad\mbox{and}\quad 
    GD'= q_j^{2}D'G\,,\nonumber
    \\
    \rule[-4.5mm]{0mm}{9.5mm}\mbox{so that}&
    \qquad
    G(A'+C'+D')=q_j^{2}(A'+C'+D')G\label{eq:G-comm-A'C'D'}
    \\
    \mbox{as well as }&\qquad 
    C'A'=q_j^{-2}A'C'\,
    \quad \mbox{and}\quad 
    D'C'=q_j^{-2}C'D'\,.\label{eq:A'C'D'-comms}
\end{align} 
Similarly, we compute from previous identities  
\begin{equation}\label{eq:def:B'}
B'=D'A' -q_j^{-2}A'D'
 = -q^{-2}\tfrac{(q-q^{-1})^3}{q+q^{-1}}  E_i^2\otimes K_i^2E_{(iji)}^2\,
 = -q^{-2}\tfrac{q-q^{-1}}{q+q^{-1}}(C')^2\,,
\end{equation}
which entails the additional commutation relations
 \begin{equation}\label{eq:B'-comm-ADC}
 B'A'=q_j^{-4}A'B'\,,
 \qquad
 D'B'=q_j^{-4}B'D'\,,
 \quad\mbox{and}\quad  
C'B'=B'C'\,.
\end{equation}
In the computation below we first apply the binomial formula  
\eqref{eq:q-binom-form} using \eqref{eq:G-comm-A'C'D'}. The second step
invokes Proposition~\ref{prop:PowerFormula}, as we verified the
requirements for $A'$, $B'$, $C'$, and $D'$ in 
\eqref{eq:def:B'}, \eqref{eq:A'C'D'-comms}, and \eqref{eq:B'-comm-ADC} 
for $q_j$ in place of $q$. 
\begin{equation}\label{eq:EijCop}
    \begin{aligned}
    \Delta(E_{(ij)}^n)&=((A'+C'+D')+G)^n =\sum_{m=0}^n\qbin{n}{m}{j}q_j^{m(n-m)}(A'+C'+D')^{n-m}G^m
    \\
    \rule{0mm}{8mm}&=
    \longsum[14]_{
          \substack{m,s,t,u\geq 0,\;s\geq t\\
                    n\geq m+s+t+u
             }
          } 
    q_j^{\mathfrak{c}(n,m,s,t,u)}
    \qbin{n}{m}{j}\qbin{n-m}{s+t+u}{j}\qbin{s+t+u}{s-t}{j}c'_{u,t}  
    \\&\hspace*{53mm}
    \cdot (A')^{s-t}(B')^t(C')^u(D')^{n-m-s-t-u}G^m\,,\\
    \rule{0mm}{9mm}
   \mbox{with} &\qquad\mathfrak{c}(n,m,s,t,u)=m(n-m)-(s+t+u)(n-m-s-t-u)-(s-t)(2t+u)\,.
\end{aligned}
\end{equation}

Next, we specialize $q$ to the root of unity $\zeta$ as before and set $n=\ell_j\,$.
The product of the three binomial coefficients in \eqref{eq:EijCop} is again
a multinomial coefficient for $r=4$, now with 
$a_1=m$, 
$a_2=s-t$, 
$a_3=2t+u$, and 
$a_4=\ell_j-m-s-t-u$.
Corollary~\ref{cor:zeros-mulitnom} then confines the summation to the following tuples of indices.
\[
(m,s,t,u)\in\bigl\{(0,0,0,0),(0,\ellj,0,0),(\ellj,0,0,0)\bigr\}\cup \bigl\{(0,t,t,\ellj-2t): 0\leq t\leq \floor{\ellj/2}\bigr\}\,
\]
Note that with $\mathfrak{c}=(a_1+a_4)(a_2+a_3)+a_1a_4-a_2a_3$ the $q$-exponential in \eqref{eq:EijCop} is 1 in all cases. 
Observe further that on the subset indexed by $t$ the operator term is  
\[(A')^{s-t}(B')^t(C')^u(D')^{n-m-s-t-u}G^m=(B')^t(C')^{\ellj-2t}=
\left(-\zeta^{-2}\tfrac{\zeta-\zeta^{-1}}{\zeta+\zeta^{-1}}\right)^t(C')^{\ellj}\,.
\]
The summation over the non-zero terms, thus, gives 
\begin{align*}
\Delta(X_{(ij)})=
   &(D')^{\ellj}+(A')^{\ellj}+G^{\ellj}+\sum_{t=0}^{\floor{\ellj/2}}c'_{\ellj-2t,t} \left(-\tfrac{\zeta-\zeta^{-1}}{\zeta^{2}[2]}\right)^t(C')^{\ellj}  
   \\
    \rule{0mm}{7.5mm}=&
    1\otimes E_{(ij)}^{\ellj}+\zeta^\ellj(q-q^{-1})^{2\ellj} E_i^{2\ellj}\otimes \zeta_j^{2\binom{\ellj}{ 2}}K_i^{2\ellj}E_j^{\ellj}+E_{(ij)}^{\ellj}\otimes K_i^{2\ellj}K_j^{\ellj}
    \\&\qquad +
    \zeta^{-\binom{\ellj}{ 2}}[2]^{-\ellj}\qfacrelB{\ell_j}{2} \cdot (-1)^{\ellj}(\zeta-\zeta^{-1})^{\ellj} E_i^{\ellj}\otimes K_i^{\ellj}E_{(iji)}^{\ellj}\,.
    \\
    \rule{0mm}{7.5mm}=&
    1\otimes X_{(ij)}+\zeta^{\binom{2\ellj}{ 2}}\delzeta{i}^\gcdleopp X_i^{\gcdleopp}\otimes L_i^{\gcdleopp}X_j+X_{(ij)}\otimes L_i^{\gcdleopp}L_j
    \\&\qquad -
    \delta_{\gcdleopp,2}\cdot 2\cdot  \zeta^{\ell}[2]^{-\ell}\delzeta{i}\cdot  X_i\otimes L_iX_{(iji)}\,.
\end{align*}

In the last step we used  Lemma \ref{lem:SumcIdentity}. Recall from \eqref{eq:factratformZ-I} that $\qfacrelB{\ell_j}{2}\,=\,\delta_{\gcdle,1}\cdot 2\cdot  \zeta^{-\binom{\ell+1}{ 2}}$ is nonzero only if $\ell=\ellj$\,, or equivalently if $\gcdleopp=2$. In this case, $E_i^{\ellj}=X_i$ and $E_{(iji)}^{\ellj}=X_{(iji)}$ which now gives the desired expression.
 \end{proof}

We note that similar coproduct formulae appear in \cite{BDR02}, where the authors 
compute liftings of the Nichols algebra of a Yetter-Drinfeld module of
 type $\LT{B}_2$\,. Indeed, the $q$-binomial theorem in \cite[Sec 4]{BDR02}
is the special case of our Proposition~\ref{prop:PowerFormula} for $B$ a multiple of $C^2$. 

\subsection{The Borel Matrix Algebra on \texorpdfstring{$\mathrm{SO}(5,\mathbb C)$\,}{SO(5,C)}}
\label{subsec:BorelCordSO5}

We adopt here the conventions from \cite[\S 2.1.2]{GW09} for the definition of the {\em complex} orthogonal group. The basis for the symmetric form in odd dimensions contains two isotropic subbases besides an additional vector. The explicit definition in this form is  
\begin{equation}
\begin{array}{l}
     \mathrm{SO}(5,\mathbb C)=\left\{A\in\mathrm{SL}(5,\mathbb C): \,A^\transp J A=  J\right\}  \qquad \mbox{for}\vspace*{8mm}\\
     \quad \mbox{ where }\quad I=\begin{bmatrix}
    0 & 1 \\
    1 & 0
\end{bmatrix}
\quad \mbox{ and }\quad \mathbb 0=\begin{bmatrix}
    0 & 0 \\
    0 & 0
\end{bmatrix}\,. 
\end{array} 
   \quad  J=
\left[{
\begin{NiceArray}{cc:w{c}{6mm}:cc}
\Block{2-2}{{\mathbb 0}} && 0 & \Block{2-2}{I}& \\
 && 0 & &  \\
\hdottedline
\;\;\rule[-2mm]{0mm}{6.5mm} 0 & 0 & 1 & 0 & 0 \;\;\\
\hdottedline
\Block{2-2}{I} & & 0 & \Block{2-2}{{\mathbb 0}}   &   \\
 &  & 0 &  & 
\end{NiceArray}
}\right]\,,
\end{equation}
 Following \cite[\S 2.1.2]{GW09}, a Cartan subgroup $H$ is given by diagonal matrices $\mathrm{diag}[\lambda,\mu,1,\mu^{-1},\lambda^{-1}]$ with $\lambda,\mu\in \mathbb C^\times\,$. For an element $h=\mathrm{diag}[h_1,h_2,0,-h_2,-h_1]\in\mathfrak h$ in the respective Cartan Lie subalgebra,  
 a set of simple roots in $\mathfrak h^*$ is 
 then given by $\alpha_1(h)=h_2$ and $\alpha_2(h)=h_1-h_2\,$, where $\alpha_1$ is short and 
 $\alpha_2$ the long root for $\LT{B}_2\,$, implying $A_{12}=-2$ and $A_{21}=-1$ for the Cartan data defined  in Section~\ref{sec:Weyl}. 
 
 These choices realize the Lie-Kolchin Theorem, meaning, the entailed Borel subgroup $\Borelsof=\mathrm{SO}(5,\mathbb C)\cap \UtriM{5}$ is precisely given by the upper triangular matrices in $\mathrm{SO}(5,\mathbb C)$. 
It is then straightforward to work out the following expression for an isomorphism $\,\SOFBchar:\,(\mathbb C^\times)^2\times \mathbb C^4\rightarrow \Borelsof\,$ of algebraic sets. 
\begin{equation}\label{eq:ParamBSO5}
\SOFB \lambda\mu a b c d =
\left[{
\begin{NiceArray}{cc:w{c}{10mm}:w{c}{12mm}w{c}{12mm}}
\Block{2-2}{F} && \Block{2-1}{-F  v} & \Block{2-2}{F(S-\tfrac 12   v  v^\transp)I}& \\
 &&  &  &  \\
\hdottedline
\;\;\rule[-2.5mm]{0mm}{7.5mm} 0 & 0 & 1 & \Block{1-2}{(I  v)^\transp} & \\
\hdottedline
\Block{2-2}{\mathbb 0} & & 0 & \Block{2-2}{(IF^\transp I)^{-1}} &  \\
 &  & 0 &  & 
\end{NiceArray}
}\right]
\qquad
\begin{array}{l}
  \mbox{ where }   \qquad  F=\begin{bmatrix}
    \lambda & b \\
    0 & \mu
\end{bmatrix}\,, \vspace{6mm}\\
\quad S=\begin{bmatrix}
    0 & d \\
    -d & 0
\end{bmatrix}\,,
\mbox{ and }
v=\begin{bmatrix}
    c  \\
    a  
\end{bmatrix}\,.
\end{array}
\end{equation}

In these coordinates, the group or matrix multiplication of two elements in $\Borelsof$ is given by 
$$
\SOFB {\lambda_1}{\mu_1}{a_1}{b_1}{c_1}{d_1} 
\cdot
\SOFB {\lambda_2}{\mu_2}{a_2}{b_2}{c_2}{d_2} 
= \SOFB {\lambda_3}{\mu_3}{a_3}{b_3}{c_3}{d_3} \;,
$$
where 
\begin{equation}\label{eq:SO5prod}
    \begin{aligned}
 \lambda_3&=\lambda_1\lambda_2
 \qquad \qquad 
 \mu_3=\mu_1\mu_2
 \qquad \qquad 
 a_3=\mu_2^{-1}a_1+a_2\\
 b_3&=\mu_2b_1+\lambda_1b_2 \qquad \qquad 
 c_3=\lambda_2^{-1}c_1+c_2-\lambda_2^{-1}\mu_2^{-1}a_1b_2
 \\
 d_3&=\lambda_2^{-1}\mu_2^{-1}d_1+d_2+\tfrac 12\left(\mu_2^{-1}a_1c_2-\lambda_2^{-1}c_1a_2+\lambda_2^{-1}\mu_2^{-1}a_1a_2b_2\right)\,.
\end{aligned}
\end{equation}

The Cartan subgroup $\Cartsof$ is then the $\SOFBchar$-image of the subspace $(\mathbb C^\times)^2\times \{(0,0,0,0)\}$, that is, the matrices with $a=b=c=d=0\,$. Similarly, the respective unipotent subgroup 
$\Unipsof$ consists of matrices with $\lambda=\mu=1$, meaning the 
$\SOFBchar$-image of $\{(1,1)\}\times \mathbb C^4\,$.

Note that the inverse map $\SOFBchar^{-1}:\,\Borelsof\rightarrow(\mathbb C^\times)^2\times \mathbb C^4 \,$
can be obtained as the restriction of a polynomial map on $\mathrm{GL}(5,\mathbb C)\,$. Specifically,
$\lambda$, $\lambda^{-1}$, $\mu$, $\mu^{-1}$, $a$, $b$, and $c$ are simply given as the restriction of 
some $\cofun{x_{i,j}}$\,, and the expression for the $d$ coordinate is easily worked out from these. 
Thus, $(\mathbb C^\times)^2\times \mathbb C^4$ equipped with the multiplication operation from
\eqref{eq:SO5prod} is isomorphic to $\Borelsof$ as an algebraic group. 

Correspondingly, we introduce the generating functions as the respective coordinate projections 
$\Borelsof\rightarrow(\mathbb C^\times)^2\times \mathbb C^4\rightarrow\mathbb C$ for this parametrization, denoting by $\cofun{\gamma}$ the function corresponding to a coordinate labeled by $\gamma$. So, for example, $\,\cofun{\lambda}(G)=(\SOFBchar^{-1}(G))_{1,1}\,$ or $\,\cofun{d}(G)=(\SOFBchar^{-1}(G))_{2,3}\,$ with $G\in\Borelsof\,$.  Thus, as a commutative algebra,  $\Funsofb$ is naturally identified with 
$\mathbb C[\cofun{\lambda}^{\pm 1},\cofun{\mu}^{\pm 1},\cofun{a},\cofun{b},\cofun{c},\cofun{d}]\,$.

Recall from Section~\ref{subsec:AlgGrps-basics} that the coalgebra structure on $\Funsofb$ is defined
by  $\Delta(f)(a,b)=f(ab)$. 
The respective coproduct formulae for the generators follow immediately from the   expressions in \eqref{eq:SO5prod}. For example, if $G_3=G_1G_2$ then
$\cofun{a}(G_3)=a_3=\mu_2^{-1}a_1+a_2=\cofun{a}(G_1)\cofun{\mu}^{-1}(G_2)+\cofun{a}(G_2)=(\cofun{a}\otimes \cofun{\mu}^{-1}+1\otimes \cofun{a})(G_1,G_2\,)$\,. The complete coalgebra structure  is, thus, given as follows.

\begin{equation}\label{eq:SO5coprod}
    \begin{aligned}
    \Delta(\cofun{\lambda})&=\cofun{\lambda}\otimes \cofun{\lambda}
    \qquad \qquad 
    \Delta(\cofun{\mu})=\cofun{\mu}\otimes \cofun{\mu}
    \qquad \qquad
    \Delta(\cofun{a})=\cofun{a}\otimes \cofun{\mu}^{-1}+1\otimes \cofun{a}
     \vspace*{2.5mm}\\
    \Delta(\cofun{b})&=\cofun{b}\otimes \cofun{\mu}+\cofun{\lambda}\otimes \cofun{b}
    \qquad \qquad 
    \Delta(\cofun{c})=\cofun{c}\otimes \cofun{\lambda}^{-1}
      +1\otimes\cofun{c}
      -\cofun{a}\otimes \cofun{\lambda}^{-1}\cofun{\mu}^{-1}\cofun{b}
      \vspace*{2.5mm}\\
    \Delta(\cofun{d})&=
        \cofun{d}\otimes\cofun{\lambda}^{-1}\cofun{\mu}^{-1}
        +1\otimes \cofun{d}
        +\tfrac{1}{2}\cofun{a}\otimes\cofun{\mu}^{-1}\cofun{c}
        -\tfrac{1}{2}\cofun{c}\otimes\cofun{\lambda}^{-1}\cofun{a}
        +\tfrac{1}{2}\cofun{a}\otimes
                \cofun{\lambda}^{-1}\cofun{\mu}^{-1}\cofun{a}\cofun{b} 
    \end{aligned}
\end{equation}

An antipode is defined by $S(f)(L)=f(L^{-1})$ but can, alternatively, also be computed iteratively from the antipode axiom and the corelations in \eqref{eq:SO5coprod}. Additionally, a $\mathbb Z^{\sroots}$ grading $\wgrad$ is defined on  $\Funsofb$ by assigning to each of the generators a positive roots as follows.
\begin{equation}\label{eq:SO6wgrad}
    \wgrad(\cofun{\lambda})=\wgrad(\cofun{\mu})=0,\;\;
\wgrad(\cofun{a})=\alpha_1,\;\;
\wgrad(\cofun{b})=\alpha_2,\;\;
\wgrad(\cofun{c})=\alpha_1+\alpha_2,\; \mbox{and}\;\;
\wgrad(\cofun{d})=2\alpha_1+\alpha_2
\end{equation}
We summarize our computations next.

\begin{prop}
    The coordinate ring $\Funsofb$ is isomorphic, as a $\wgrad$-graded Hopf algebra, 
    to the complex 
    polynomial algebra $\mathbb C\left[\cofun{\lambda}^{\pm 1},\cofun{\mu}^{\pm 1},\cofun{a},\cofun{b},\cofun{c},\cofun{d}\right]\,$ equipped with the  corelations given by \eqref{eq:SO5coprod}.
\end{prop}

The fact that the abstract polynomial algebra is isomorphic to $\Funsofb$ is clear from the fact that 
monomials $\left\{\cofun{\lambda}^{m_1}\cofun{\mu}^{m_2}\cofun{a}^{k_1}\cofun{b}^{k_2}\cofun{c}^{k_3}\cofun{d}^{k_4}\,:\,m_i\in\mathbb Z, k_i\in\nnN\right\}\,$ form a basis  for both the abstract algebra and 
for $\Funsofb$. The former is well-defined as a coalgebra since the invertible generators are group-like. It is also easy to check that the corelations in \eqref{eq:SO5coprod} preserve the grading 
from \eqref{eq:SO6wgrad}.

Recall that the Weyl group $\Weyl$ for $\LT{B}_2$ is the dihedral group of order 8, generated by reflections $s_1$ and $s_2$  along $\alpha_1=\epsilon_2$ and $\alpha_2=\epsilon_1-\epsilon_2\,$, respectively,
so that $s_1(\alpha_2)=2\alpha_1+\alpha_2$ and $s_2(\alpha_1)=\alpha_1+\alpha_2$\,.
On the Cartan algebra for $\mathrm{SO}(5,\mathbb C)$, these reflections correspond to $(h_1,h_2)\mapsto(h_1,-h_2)$
and $(h_1,h_2)\mapsto(h_2,h_1)$ for $s_1$ and $s_2\,$, respectively. 

The lifted maps on the Cartan subgroup $\Cartsof$ in the form given in \eqref{eq:ParamBSO5} are implemented by conjugation with the following representatives of $s_1$ and $s_2$ as elements of the geometric Weyl group in the normalizer of $\Cartsof$
in $\mathrm{SO}(5,\mathbb C)$.

\begin{equation}\label{eq:RepGeomWeylSO5}
\dot s_1=\left[{
\begin{NiceArray}{cc:c:cc}[margin]
 1 & 0 & 0 & 0 & 0 \; \\
 0 & 0 & 0 & 1 & 0 \; \\
\hdottedline
 \rule[-2mm]{0mm}{6.5mm} 
0 & 0 & 1 & 0 & 0 \;\\
\hdottedline
\rule[-2mm]{0mm}{6.5mm} 
  0 & -1 & 0 & 0 & 0 \;\\
  0 &  0 & 0 & 0 & 1 \;
\end{NiceArray}
}\right]
\qquad \mbox{ and } \qquad 
\dot s_2=\left[{
\begin{NiceArray}{cc:c:cc}[margin]
 0 & 1 & 0 & 0 & 0 \; \\
 1 & 0 & 0 & 0 & 0 \; \\
\hdottedline
 \rule[-2mm]{0mm}{6.5mm} 
0 & 0 & 1 & 0 & 0 \;\\
\hdottedline
\rule[-2mm]{0mm}{6.5mm} 
  0 & 0 & 0 & 0 & 1 \;\\
  0 &  0 & 0 & 1 & 0 \;
\end{NiceArray}
}\right] \,.
\end{equation}

Representatives $\dot s\in N(\Cartsof)$ for other elements $s\in\Weyl$ are easily worked out from
those in \eqref{eq:RepGeomWeylSO5}.
This allows us to define the parabolic and unipotent 
 Bruhat subgroups for $\mathrm{SO}(5,\mathbb C)$.
\begin{equation}
    \Borelsof(s)\,=\,\Borelsof\,\cap\, \dot s\Borelsof\dot s^{-1}
\qquad \mbox{ and } \qquad
\Unipsof(s)\,=\,\Unipsof\,\cap\, \dot s\Unipsof\dot s^{-1}\,
\end{equation}

Clearly, $\Borelsof(s)\cong\Cartsof\ltimes \Unipsof(s)$\,. 
As before,  we define for any $s\in\Weyl$ a vanishing ideal for the Bruhat subgroup as
\begin{equation}
\SOFAugIdchar(s)\,=\,\left\{f\in\Funsofb\,:\,f(h)=0\;\forall h\in \Borelsof(s)\right\}\,,
\end{equation}
which is  a Hopf ideal by Lemma~\ref{lm:ASG=HI}. For $s\in\Weyl$ we also introduce the ideal in $\Funsofb$ generated by subsets of coordinate functions  of 
$\Funsofb$ that have grading in $\descroots{s}$. That is,
\begin{equation}
    \mathcal{J}(s)=\left(G(s)\right)\qquad\mbox{where}\quad
G(s)=\left\{\cofun{\gamma}\in\{\cofun{a},\cofun{b},\cofun{c},\cofun{d}\}\,:\wgrad(\cofun{\gamma})\in\descroots{s}\right\}\,.
\end{equation}
For reference and convenience, we list all the non-trivial cases here explicitly.
\begin{equation}\label{eq:JsList}
    \begin{aligned}
    \mathcal{J}(s_1)&=\left(\cofun{a}\right)\qquad &\qquad\mathcal{J}(s_1s_2)&=\left(\cofun{a},\cofun{d}\right)\qquad&\qquad
            \mathcal{J}(s_1s_2s_1)&=\left(\cofun{a},\cofun{c},\cofun{d}\right)\\
    \mathcal{J}(s_2)&=\left(\cofun{b}\right)&\mathcal{J}(s_2s_1)&=\left(\cofun{b},\cofun{c}\right)&
            \mathcal{J}(s_2s_1s_2)&=\left(\cofun{b},\cofun{c},\cofun{d}\right)
\end{aligned}
\end{equation}

The following can be verified immediately from the coproduct formulae in \eqref{eq:SO5coprod} in each of the cases listed in \eqref{eq:JsList}. 
\begin{lem}\label{lem:Js=HopfIdeal}
    For any $s\in\Weyl$ the ideal $\mathcal{J}(s)$ is a Hopf ideal. 
\end{lem}

The two types of Hopf ideals are identified next.

\begin{prop}\label{prop:SO5eqalIdeal}
    For any $s\in\Weyl$ we have $\mathcal J(s)=\SOFAugIdchar(s)\,$.
\end{prop}

\begin{proof} First, observe that the elements in \eqref{eq:RepGeomWeylSO5} also
represent elements of the Weyl group of the ambient $\mathrm{GL}(5,\mathbb C)$,
which we identify with $S_5\,$. More precisely, consider the Cayley embedding 
$\iota:\Weyl\hookrightarrow S_5$ defined by $\iota(s_1)=(2\,4)$ and 
$\iota(s_2)=(1\,2)(4\,5)$, mapping the longest element $\longweyl=(s_1s_2)^2$ in $\Weyl$ to the longest
element $\iota(\longweyl)=(1\,5)(2\,4)$ in $S_5\,$. Thus, if $\pi=\iota(s)$, we have that $\dot s$ is given
by the ordinary permutation matrix $\bar\pi$ for $\pi$ in $\mathrm{GL}(5,\mathbb C)$ up to sign changes of the 1-entries. 

Next, observe that, because of $\Borelsof=\UtriM{5}\cap\mathrm{SO}(5,\mathbb C)$, we have
$\Borelsof(s)=\Borelsof\cap\dot s\Borelsof\dot s^{-1}=\UtriM{5}\cap \dot s\UtriM{5}\dot s^{-1}\cap \mathrm{SO}(5,\mathbb C)
=\UtriM{5}\cap \bar\pi \UtriM{5} {\bar\pi}^{-1}\cap \mathrm{SO}(5,\mathbb C)$. Hence,
$$
\Borelsof(s)=B(\iota(s))\cap\mathrm{SO}(5,\mathbb C)
\qquad\mbox{and}\qquad
\Unipsof(s)=U(\iota(s))\cap\mathrm{SO}(5,\mathbb C)\,,
$$
with analogous relations for the unipotent parts. For $\pi\in S_5$ let 
$$
\mathrm{Inv}(\pi)=\left\{(i,j)\in\mathbb N^2:\,1\leq i<j\leq 5 \,\mbox{ and }\,\pi^{-1}(i)>\pi^{-1}(j)\right\}\,. 
$$
Recall from the proof of Theorem~\ref{thm:AB=Ks} that $L\in B_5(\pi)$ iff $L_{i,j}=0$ whenever $(i,j)\in \mathrm{Inv}(\pi)\,$. 

Thus, for a given $s\in\Weyl$ and setting $N(s)=\mathrm{Inv}(\iota(s))$, any $(i,j)\in N(s)$  imposes a constraint on the parameters 
in $(\mathbb C^\times)^2\times \mathbb C^4$ in the parametrization of $\Borelsof\,$, given in \eqref{eq:ParamBSO5} by
setting the $(i,j)$ coefficient to zero. The corresponding constraints $\SOFBchar(y)_{i,j}=0$
on elements $y\in(\mathbb C^\times)^2\times \mathbb C^4$ then define a subvariety in 
$(\mathbb C^\times)^2\times \mathbb C^4$, for which the vanishing ideal can be 
expressed in terms of the generating coordinate functions.  

We apply this method next to verify all six non-trivial cases above. One can readily work out $N(s_1)=\left\{(2,3),(2,4),(3,4)\right\}$ and  $N(s_2)=\left\{(1,2),(4,5)\right\}$ for the generators. Setting the respective matrix coefficients to
zero yields the constraint $a=0$ for the former and $b=0$ for the latter. The vanishing ideals are thus
$\left(\cofun{a}\right)$ and $\left(\cofun{b}\right)$, respectively. 

For the first length 2 case, we have $N(s_1s_2)=\left\{(1,4),(2,3),(2,4),(2,5),(3,4)\right\}$, which imposes  constraints that are equivalent to the condition $a=d=0$. Similarly, we
find $N(s_2s_1)=\left\{(1,2), (1,3), (1,5), (3,5), (4,5)\right\}$ for the second case, yielding $b=c=0$. Finally, the length 3 constraints are given by the following data.
\begin{align*}
    N(s_2s_1s_2)&=\left\{(1,2),(1,3),(1,4),(1,5),(2,5),(3,5),(4,5)\right\}  &&\Rightarrow\quad b=c=d=0\\
    N(s_1s_2s_1)&=\left\{(1,3),(1,4),(1,5),(2,3),(2,4),(2,5),(3,4),(3,5)\right\}  &&\Rightarrow\quad a=c=d=0\;,
\end{align*}
which implies the remaining identifications of ideals.
\end{proof}

In order to better match the Hopf algebra structure of $\Funsofb$ with its quantum group counterpart, it is convenient to introduce a modified set of generators. For a given parameter $t$, we define them by the following polynomial expressions. 
\begin{equation}\label{eq:MsofNewGent}
    \begin{aligned}
        \cofuntr{g}&=\cofun{\lambda}^{-1}\cofun{\mu} \qquad&\qquad \cofuntr{a}&=\cofun{a}    
                     \qquad&\qquad \cofuntrpar{c}{t}&=\cofun{c}+t\cofun{\lambda}^{-1}\cofun{a}\cofun{b}\\
        \cofuntr{h}&=\cofun{\mu}^{-1} & \cofuntr{b}&=\cofun{\lambda}^{-1}\cofun{b} 
            &        \cofuntrpar{d}{t}&=\cofun{d}+(\tfrac{1}{2}-t)\cofun{a}\cofun{c}
                  -\tfrac{1}{2}t^2\cofun{\lambda}^{-1}\cofun{a}^2\cofun{b} 
    \end{aligned}
\end{equation}

It is easy to check that the generators in $\left\{\cofun{\lambda}^{\pm 1},\cofun{\mu}^{\pm 1},\cofun{a},\cofun{b},\cofun{c},\cofun{d}\right\}\,$ are, conversely, given as 
polynomial expressions in the $\left\{\cofuntr{g}^{\pm 1},\cofuntr{h}^{\pm 1},\cofuntr{a},\cofuntr{b},\cofuntrpar{c}{t},\cofuntrpar{d}{t}\right\}\,$, which implies an isomorphism between the corresponding polynomial algebras for each parameter $t\,$. 
The assignment is, further, grading preserving in the sense that
$\wgrad(\cofuntr{a})=\alpha_1\,$, 
$\wgrad(\cofuntr{b})=\alpha_2\,$, 
$\wgrad(\cofuntrpar{c}{t})=\alpha_1+\alpha_2\,$,
and
$\wgrad(\cofuntrpar{d}{t})=2\alpha_1+\alpha_2\,$.

Next, observe that the ideals in \eqref{eq:JsList} are unchanged if the new generators are used instead. That is, we have, for example, $\mathcal{J}(s_1s_2)=\left(\cofuntr{a},\cofuntrpar{d}{t}\right)$
and $\mathcal{J}(s_2s_1s_2)=\left(\cofuntr{a},\cofuntrpar{c}{t},\cofuntrpar{d}{t}\right)$.
The coproducts for the generators in \eqref{eq:MsofNewGent} are determined from those
in \eqref{eq:SO5coprod} via an elementary but somewhat lengthy computation. 
\begin{equation}\label{eq:SO5coprodAlt}
    \begin{aligned}
        \Delta(\cofuntr{g})&=\cofuntr{g}\otimes\cofuntr{g}\qquad 
        \Delta(\cofuntr{h})=\cofuntr{h}\otimes\cofuntr{h}\qquad 
        \Delta(\cofuntr{a})=\cofuntr{a}\otimes\cofuntr{h}+1\otimes\cofuntr{a}
     \\ \rule[-8pt]{0pt}{22pt}
        \Delta(\cofuntr{b})&=\cofuntr{b}\otimes\cofuntr{g}+1\otimes\cofuntr{b}\qquad 
        \Delta(\cofuntrpar{c}{t})=\cofuntrpar{c}{t}\otimes \cofuntr{h}\cofuntr{g}
                                +1\otimes \cofuntrpar{c}{t}
                                +t\cofuntr{b}\otimes\cofuntr{g}\cofuntr{a}
                               -(1-t)\cofuntr{a}\otimes\cofuntr{h}\cofuntr{b}
     \\
        \Delta(\cofuntrpar{d}{t})&=
                \cofuntrpar{d}{t}\otimes \cofuntr{h}^2\cofuntr{g}
                +1\otimes \cofuntrpar{d}{t}
                + (1-t) \cofuntr{a}\otimes\cofuntr{h}\cofuntrpar{c}{t}
                - \tfrac 12 (1-t)^2\cofuntr{a}^2\otimes \cofuntr{h}^2\cofuntr{b}
                -t\cofuntrpar{c}{t}\otimes \cofuntr{h}\cofuntr{g}\cofuntr{a}
                -\tfrac 12 t^2\cofuntr{b}\otimes \cofuntr{g}\cofuntr{a}^2
    \end{aligned}
\end{equation}

An important feature of this form is that in each of the tensor factors only one type of non-Cartan generator occurs, avoiding mixed terms such as $\cofuntr{a}\otimes\cofuntr{a}\cofuntr{b}$ or 
$\cofuntr{a}\cofuntrpar{c}{t}\otimes 1\,$. The two relevant parameter choices are $t=1$ and $t=0$,
for which terms in the coproducts of $\cofuntrpar{c}{t}$ and $\cofuntrpar{d}{t}$ vanish.

The proposition below summarizes our findings, relevant to the identifications with quantum algebras.
Here, it is useful to introduce for a parameter $t$ 
the free abstract polynomial algebra $\AbstMatSof{t}$ with generators $\left\{\cofuntr{g}^{\pm 1},\cofuntr{h}^{\pm 1},\cofuntr{a},\cofuntr{b},\cofuntrpar{c}{t},\cofuntrpar{d}{t}\right\}\,$
with the usual commutation and inverse relation. The algebra $\AbstMatSof{t}$ is equipped with the 
coalgebra structure defined in \eqref{eq:SO5coprodAlt},  which uniquely extends to a Hopf algebra structure. 

\begin{prop}\label{prop:SO5summ}
The following statements hold for any choice of $t$. \vspace{-1mm}
    \begin{enumerate}[label=\roman*)]
        \item \label{item:SO5summ:iso}
         $\Funsofb$ is isomorphic to $\AbstMatSof{t}$ as a $\mathbb Z^{\sroots}$-graded Hopf algebra.\vspace{2mm}
        \item \label{item:SO5summ:ideal}
         Any subgroup of $\Borelsof$ defines a Hopf ideal in $\AbstMatSof{t}$\,.\vspace{2mm}
        \item \label{item:SO5summ:Bruhat}
         The Hopf ideal corresponding to a Bruhat subgroup $\Borelsof(s)$ coincides with the ideal  
         in $\AbstMatSof{t}$
           generated by the set $\left\{\cofuntr{\gamma}\in \{\cofuntr{a},\cofuntr{b},\cofuntrpar{c}{t},\cofuntrpar{d}{t}\}:\,\wgrad(\cofuntr{\gamma})\in\descroots{s}\right\}\,$. 
    \end{enumerate}
\end{prop}

The isomorphism in {\em \ref{item:SO5summ:iso}}, of course, implies that the $\AbstMatSof{t}$ are
isomorphic to each other for different $t\,$. The proof of the statement in 
{\em \ref{item:SO5summ:ideal}} is the same as the one for Lemma~\ref{lem:HIdeal=subgr}. As noted before, Item {\em \ref{item:SO5summ:iso}} follows readily from Proposition~\ref{prop:SO5eqalIdeal}
and inspection of the change of generators in \eqref{eq:MsofNewGent}. 

\medskip
 
\subsection{Isomorphisms between Hopf Algebras related to \texorpdfstring{$\mathsf{B}_2$\,}{B2}}\ 
\label{subsec:IsomHA-B2}
The weight grading of either $\Zsubalgchar^\geqzero_\bullet$ or $\mathbb C[\Borelsof]$ implies
a sequence of (commutative)
Hopf subalgebras as in the definition of cosolvability at the end of 
Section~\ref{subsec:AlgGrps-basics}. This facilitates classifications of weight preserving
coalgebra structures, as they can be viewed as successive extensions of each other. 
We develop here a  restricted classification of coalgebra structures
associated to the $\LT{B}_2$ root system, for which the employed methods are 
likely to generalize to much larger 
classes of  graded, cosolvable Hopf algebras.

The underlying associative algebra is, as before,
of the form $\mathbb C[\cogen{g}^{\pm1},\cogen{h}^{\pm1},\cogen{a},\cogen{b},\cogen{c},\cogen{d}]$
with weights $\wgrad(\cogen{g})=\wgrad(\cogen{h})=0$\,, 
$\wgrad(\cogen{a})=\alpha_1$\,, $\wgrad(\cogen{b})=\alpha_2$\,, $\wgrad(\cogen{c})=\alpha_1+\alpha_2$\,,
and $\wgrad(\cogen{d})=2\alpha_1+\alpha_2\,$. Aside from preserving the $\wgrad$-grading, we require
 $\cogen{a}$ and $\cogen{b}$ to be skew-primitive with respect to group-likes 
$(1,\cogen{h})$ and $(1,\cogen{g})$.

We further constrain the allowable coproducts by requiring lead coefficients to be 1 and that the group-likes in the second tensor-factor are determined by the weight of the first tensor-factor as, for example, in \eqref{eq:xgradcoprod}. Moreover, 
as discussed in \eqref{eq:SO5coprodAlt}, we omit terms with mixed monomials in tensor-factors as 
these  do not occur or can be removed in the previous examples. In this restricted setting
we find six allowable terms besides the lead terms, with coefficients given by 
six parameters $(u,v,x,y,r,s)\,$ as follows. 
\begin{equation}\label{eq:GenB2Coprod}
    \begin{aligned}
        \Delta(\cogen{g})&=\cogen{g}\otimes \cogen{g}\qquad 
        \Delta(\cogen{h})=\cogen{h}\otimes \cogen{h}\qquad 
        \Delta(\cogen{a})=\cogen{a}\otimes\cogen{h}+1\otimes\cogen{a}\qquad
        \Delta(\cogen{b})=\cogen{b}\otimes\cogen{g}+1\otimes\cogen{b}\qquad
     \\ \rule[-8pt]{0pt}{22pt} 
        \Delta(\cogen{c})&=\cogen{c}\otimes \cogen{h}\cogen{g}
                                \,+\,1\otimes \cogen{c}
                                \,+\,u\cdot\cogen{a}\otimes\cogen{h}\cogen{b}
                               \,+\,v\cdot\cogen{b}\otimes\cogen{g}\cogen{a}
     \\
        \Delta(\cogen{d})&=
                \cogen{d}\otimes \cogen{h}^2\cogen{g}
                \,+\,1\otimes \cogen{d}
                \,+\, x \cdot\cogen{a}\otimes\cogen{h}\cogen{c}
                \,+\,y \cdot\cogen{c}\otimes \cogen{h}\cogen{g}\cogen{a}
                \,+\, r \cdot\cogen{a}^2\otimes \cogen{h}^2\cogen{b}
                \,+\,s \cdot\cogen{b}\otimes \cogen{g}\cogen{a}^2
    \end{aligned}
\end{equation}

The axiom of coassociativity now imposes additional conditions on these parameters. 
It is easy to see that for
the subalgebra in which $\cogen{d}$ is omitted as a generator, the corelations 
in \eqref{eq:GenB2Coprod} define a coalgebra structure for any choices of 
$u$ and $v$. Testing coassociativity for $\cogen{d}$ yields non-trivial
constraints only for the comparisons of the $\cogen{a}\otimes\cogen{h}\cogen{a}\otimes\cogen{h}^2\cogen{b}\,$,
$\cogen{a}\otimes\cogen{h}\cogen{b}\otimes\cogen{h}\cogen{g}\cogen{a}\,$, and 
$\cogen{b}\otimes\cogen{g}\cogen{a}\otimes\cogen{h}\cogen{g}\cogen{a}\,$ terms. 
The existence of an antipode for any bialgebra structure is immediate.  

\begin{lem}\label{lem:ParHopfB2}
The formulae in \eqref{eq:GenB2Coprod} define a Hopf algebra structure on 
$\mathbb C[\cogen{g}^{\pm1},\cogen{h}^{\pm1},\cogen{a},\cogen{b},\cogen{c},\cogen{d}]$ 
 iff
\begin{equation}\label{eq:ParHopfB2}
    vx=uy\,, \qquad s=\tfrac{1}{2}vy\,, \quad\mbox{and}\quad r=\tfrac{1}{2}ux\,.  
\end{equation}
\end{lem}

In particular, the parameters $r$ and $s$  are determined by the four others. We, thus,
denote by $\mathcal M(u,v,x,y)$ the Hopf algebra defined by the coproducts in \eqref{eq:GenB2Coprod}
subject to the parameter constraints in \eqref{eq:ParHopfB2}. The obvious generator to generator
map yields the following isomorphim with the Hopf algebra $\Borelsof$ as presented in \eqref{eq:SO5coprodAlt}. 
\begin{equation}\label{eq:BSO5=M4}
    \mathcal M_t\,\cong\, \mathcal M(t-1,t,1-t,-t)
\end{equation}

We next determine isomorphism classes within this family of Hopf algebras. We require an isomorphism to be polynomial, weight-preserving, as well as preserving the general structure in 
\eqref{eq:GenB2Coprod}. It is an easy exercise to show that such a transformation must be the
identity on group-likes and maps the subring  $\mathbb C[\cogen{a},\cogen{b},\cogen{c},\cogen{d}]$ of non-group-like generators to itself. Hence, an isomorphism 
$\varphi:\mathcal{M}(u,v,x,y)\rightarrow\mathcal{M}(u',v',x',y')$ 
necessarily acts on the simple root generators
as $\varphi(\cogen{a})=\nu_1\cogen{a}$ and $\varphi(\cogen{b})=\nu_2\cogen{b}$ for $\nu_1,\nu_2\in\mathbb C^\times$. Moreover,  for any such pair $(\nu_1,\nu_2)$, there is an 
automorphism of the same  $\mathcal{M}(u,v,x,y)$ to itself. So, without loss of generality, we may 
assume $\varphi$ is identity on the Hopf subalgebra $\mathcal M^0=\mathbb C[\cogen{h}^{\pm 1},\cogen{g}^{\pm 1},\cogen{a},\cogen{b}]$.

It follows that $\varphi(\cogen{c})$ is a $\mathbb C$-linear combination of $\cogen{c}$ and 
$\cogen{a}\cogen{b}\,$, and $\varphi(\cogen{c})$ is a $\mathbb C$-linear combination of 
$\cogen{d}$, $\cogen{a}\cogen{c}$, and $\cogen{a}^2\cogen{b}\,$. The two diagonal coefficients 
need to be non-zero in order for $\varphi$ to be invertible. Two additional conditions on the
five coefficients of $\varphi$ stem from the requirement to preserve the form in 
\eqref{eq:GenB2Coprod} and the relations in \eqref{eq:ParHopfB2}, each depending also
on the parameters $(u,v,x,y)$ of the source algebra. 

One can then work out that an isomorphism
$\varphi:\mathcal{M}(u,v,x,y)\rightarrow\mathcal{M}(u',v',x',y')$ exists iff 
$(u',v')=\sigma(u,v)+\eta(1,1)$ and $(x',y')=\mu(x,y)+\tau(1,1)$ for parameters 
$\sigma,\mu\in\mathbb C^\times$ and $\eta,\tau\in\mathbb C$ with
$\mu\eta(x-y)=\sigma\tau(u-v)\,$. This now classifies
Hopf algebras of the above form. 

\begin{prop}\label{prop:ClassifB2Calg} There exist weight-preserving isomorphisms of Hopf algebras as follows.
    $$
\mathcal{M}(u,v,x,y)\quad\cong\quad
\begin{cases}
    \,\mathcal{M}(1,0,1,0) & \mbox{ if } u\neq v \mbox{ and } x\neq y\\
    \,\mathcal{M}(1,0,0,0) & \mbox{ if } u\neq v \mbox{ and } x = y\\
    \,\mathcal{M}(0,0,1,0) & \mbox{ if } u= v \mbox{ and } x \neq y\\
    \,\mathcal{M}(0,0,0,0) & \mbox{ if } u= v \mbox{ and } x = y\\ 
\end{cases}
$$
Furthermore, there are no weight-preserving isomorphisms of the required form  between any of the 
four listed special cases.   
\end{prop}

Observe that $(u,v)=(0,0)$ means that $\cogen{c}$ is skew-primitive and $(x,y)=(0,0)$ that 
$\cogen{d}$ is skew-primitive. Given $\mathcal{M}^0$ as before, denote also by 
$\mathcal{M}^1$ the Hopf subalgebra with additional generator $\cogen{c}$ and 
$\mathcal{M}^2=\mathcal{M}(u,v,x,y)\,$ the full algebra. The four cases above may
then be characterized by whether the extension from $\mathcal{M}^0$ to $\mathcal{M}^1$ 
as well as the one $\mathcal{M}^1$ to $\mathcal{M}^2$ is trivial (in the sense that
only an independent skew-primitive is added) or non-trivial.

Note, finally, that $u= v\neq 0$ implies by \eqref{eq:ParHopfB2} that $x = y$. So, the third case 
in Proposition~\ref{prop:ClassifB2Calg}
can occur only if $u=v=0$. Analogously, the second case implies $x=y=0\,$. We refer to the first version as  the {\em generic} case. Clearly, $\mathbb C[\Borelsof]$ is an example by \eqref{eq:BSO5=M4} and Proposition~\ref{prop:SO5summ}.

The classification immediately entails the identifications of the Hopf algebras and Hopf ideals summarized 
below. For an element $s\in\Weyl$, we denote here by $\breve s\in\Weyl$ the action of $s$ on the 
coroot system $\coproots$ conjugated by the identification of $\coproots$ with $\proots$
discussed in the beginning of Section~\ref{subsec:Coalg-ZB2}. So, $s\mapsto\breve s$ is the 
group homomorphism given by $\breve s_1=s_2$ and $\breve s_2=s_1\,$. 

\smallskip

\begin{thm}\label{thm:Z=CSO5}
Suppose $\Uz^\geqzero$ is of Lie type $\LT{B}_2=\LT{C}_2$ with $\zeta$ a root of unity of
any order $\kay>4\,$ or $\kay=3$.\vspace{-1mm}
\begin{enumerate}[label=\roman*)]
        \item \label{item:Z=CSO5:iso}
    $\Zsubalgchar^\geqzero_\bullet\otimes \mathbb C$ is isomorphic as a Hopf algebra to $\mathbb C[\Borelsof]$.
    \vspace{2mm}
    \item \label{item:Z=CSO5:ideal}
    Any subgroup $H<\Borelsof$ yields Hopf ideals $\mathcal{A}(H)$ and $\widehat\mathcal{A}(H)$
    of $\Zsubalgchar^\geqzero_\bullet$ and $\Uz^\geqzero\,$, respectively. 
    \vspace{2mm}
    \item \label{item:Z=CSO5:weyl}
    For $\ell$ odd, $\mathcal A(\Borelsof(s))=\Zaugideal{s}^\geqzero$ and, for $\ell$ even, 
    $\mathcal A(\Borelsof(\breve s))=\Zaugideal{s}^\geqzero$ for any $s\in\Weyl\,$. 
    \end{enumerate}
\end{thm}

\begin{proof} Remark~\ref{rem:B2C2commcentr} implies that $\Zsubalgchar^\geqzero_\bullet$ is commutative for any root of unity and, thus, isomorphic to the polynomial ring  
$\mathbb C[L_i^{\pm 1}, L_j^{\pm 1}, \{\Epw_\alpha:\alpha\in\proots\}]\,$. 
We next note that the coalgebra structure from Proposition~\ref{prop:CoprodB2} is indeed of the
form given in \eqref{eq:GenB2Coprod}.

Specifically, if $\ell$ is odd ($\gcdleopp=2$ and $\gcdle=1$) assign 
$\cogen{h}\mapsto \Kpw_i$, $\cogen{g}\mapsto \Kpw_j$, 
$\cogen{a}\mapsto \Epw_i$, $\cogen{b}\mapsto \Epw_j$, 
$\cogen{c}\mapsto \Epw_{(iji)}$, and  $\cogen{d}\mapsto \Epw_{(ij)}\,$. Comparison shows
that 
the coalgebra structure coincides with the one of $\mathcal{M}(u,0,x,0)$, where $u\neq 0\neq x$ 
are given
by the coefficents in \eqref{eq:CoprodB2:ij} and \eqref{eq:CoprodB2:iji}. Thus, 
$\Zsubalgchar^\geqzero_\bullet$ coincides with the generic case in 
Proposition~\ref{prop:ClassifB2Calg} and is, therefore, isomorphic to $\mathbb C[\Borelsof]$. 

Conforming with the $\wgradZBtwo$-grading from \eqref{eq:tab:uweightsB2}, 
the assignments for even $\ell$ ($\gcdleopp=1$ and $\gcdle=2$) are $\cogen{h}\mapsto \Kpw_j$, $\cogen{g}\mapsto \Kpw_i$, 
$\cogen{a}\mapsto \Epw_j$, $\cogen{b}\mapsto \Epw_i$, 
$\cogen{c}\mapsto \Epw_{(ij)}$, and  $\cogen{d}\mapsto \Epw_{(iji)}\,$. This yields
a Hopf algebra isomorphism from $\mathbb C[\Borelsof]$ with $\Zsubalgchar^\geqzero_\bullet$ via 
the generic form $\mathcal{M}(0,v,0,y)$ 
with $v\neq 0\neq y\,$. 

The construction of Hopf ideals is the same as in Lemma~\eqref{eq:HIdeal=subgr}. 
The identifications of the ideals associated to the Bruhat groups 
follow from the construction of $\wgradZBtwo$-grading preserving isomorphisms and may also
be checked from the explicit forms in \eqref{eq:JsList}. 
\end{proof}

Comparing \eqref{eq:BSO5=M4} and the isomorphisms described in the above proof, we observe that  
for odd $\ell$ the non-zero terms occurring in the coproducts for $\Zsubalgchar^\geqzero_\bullet$ are the
same as for $\mathcal M_0$, and for  even $\ell$ they coincide with those of $\mathcal M_1$\,. 
Specifically, 
\begin{align}
    &\mbox{odd $\ell$:}&  
    u &= -\delzeta{j}&  v &= 0&  x&=-2\zeta^\ell[2]^{-\ell}\delzeta{i}& 
    y&=0
    \\
    &\mbox{even $\ell$:}&  
    u&=0 & v &= \zeta^{\binom{\ell}{ 2}}\delzeta{i}
& x&=0 & y&= -2[2]^\ell (-\zeta)^\ellj
\end{align}
Thus, in conclusion, we can provide explicit grading preserving maps given only by scalings of
generators.

Theorem~\ref{thm:Z=CSO5} confirms Conjecture~\ref{conj:Z=HopfSubBorel} in the $\LT{B}_2$ case and illustrates 
the need to switch to the coroot system for certain $\kay$ congruences. 
Furthermore, it is not difficult to construct and prove an analog statement of Lemma~\ref{lem:fullZ=Qn} for a respective
subgroup $W(\LT{B}_2)$ of $\Borelsofopp\times\Borelsof\,$. As for Corollary~\ref{cor:BnActOnWn}, 
Theorem~\ref{thm:ZTinvInvar} now implies an action of the finite-type dihedral Artin group $\mathscr A\cong\left\langle a, b\,|\,[a^2,b]\right\rangle$ on $W(\LT{B}_2)\,$.

We conclude our discussions of special cases in Section~\ref{sec:Ideals_An=GLn} and \ref{sec:Ideals_B2=SO5}
with a list of conjectures, generalizing our findings for $\LT{A}_n$ and $\LT{B}_2$\,. As already mentioned above, the existence of an isomorphism for odd $\kay$ in items {\em\ref{item:CoordRings:Borel}}
and {\em\ref{item:CoordRings:Full}} below has been proved in \cite{dck92} via additional Poisson structures and not in the suggested explicit form. For a given simple Lie type $\roots$, denote by $G$ the associated simply connected algebraic group and $B$ a Borel algebra over $\mathbb C$. Denote also $\breve G$ and $\breve B$ the dual groups associated with the 
coroot lattice, as well as $G^\PLexp$ and $\breve G^\PLexp$ as in (\ref{eq:DefPoissonLie}).

\begin{conj}\label{conj:CoordRings}\ 

\begin{enumerate} [label=\roman*)] 
    \item \label{item:CoordRings:Borel}
     In all the $\kay$-congruence cases for which subalgebras are strictly commutative (Corollary~\ref{cor:Zcommutmax}), the algebra $\Zsubalgchar^{\geqzero}_{\bullet}\otimes\mathbb C$ is
    isomorphic to $\mathbb C[\tilde B]$, where $\tilde B$ is $B$ or $\breve B$ or a finite covering quotient (isogeny) of these. 
    The isomorphism map can be chosen as an explicit assignment of generators $\Epw_w$ (with $w\leqRB z$ for some $z\in\wordsetmax$) to 
    weighted polynomials of matrix elements.
    \vspace{2mm}
    \item \label{item:CoordRings:Bruhat}
    For any $s\in\Weyl$, the vanishing ideal of the Bruhat group $\tilde B(s)$, as defined in
    (\ref{eq:def:BruhatGrps}), is mapped exactly to the $\Zaugideal{s}^{\geqzero}\otimes\mathbb C$ ideal.
    \vspace{2mm}
    \item \label{item:CoordRings:Full}
    The isomorphism from {\em\ref{item:CoordRings:Borel}} extends to an isomorphism from 
    $\Zsubalgchar_{\bullet}\otimes\mathbb C$ to a respective coordinate ring of $G^\PLexp$ or $\breve G^\PLexp$ or a respective finite quotient.
    \vspace{2mm}
    \item \label{item:CoordRings:Artin}
    Lusztig's action of the Artin group $\mathscr A$ on $\Zsubalgchar_{\bullet}\otimes\mathbb C$
    from Theorem~\ref{thm:ZTinvInvar} coincides with the induced action of $\mathscr A$ on
    $G^\PLexp$ in \cite[Sec 7.5]{dck92} also for the cases when $\kay$ is even and 
    $\Zsubalgchar^{\geqzero}_{\bullet}$ commutative. 
    \vspace{2mm}
\end{enumerate} 
\end{conj}

\section{Hopf Ideals via Quasi-\texorpdfstring{$R$}{R}-Matrices}\label{sec:QuasiTriang}

 $R$-matrices are the primary motivation for studying quantum groups, as they give rise to solutions to Yang Baxter equations and, more generally, braided tensor categories. Their construction for quantum groups at roots of unity is one of the goals of this section. 

As another application, we use here the 
$R$-matrix formalism in an iterative argument, showing that the $\Zaugidealhat{s}$ ideals 
are Hopf ideals. Their quotients, therefore, produce a family of well-defined Hopf algebras $\Uzn{\ell}^{[s]}$. 
The methods introduced and developed  in this section also include truncated quasi-$R$-matrices,
Tanisaki's formalism of
pre-triangular Hopf algebras, and the Tanisaki-Lusztig pairings.

\subsection{Enhanced Tensor Structures} \label{subsec:GenTensStr}

In this section, we provide some general constructions 
relevant to defining $R$-matrices in the generic setting. The first helps to make sense of the
infinite sums used in the formulae for $R$-matrices using certain inverse limits. 

Generally, if $A$ is an algebra over a ring $\Zgen$
and $I^n$ a descending sequence of {\em two-sided} ideals, 
$\varprojlim A/I^n$ admits a canonical algebra structure as an inverse system in the
category of algebras. For left (or right) ideals $A^n$, however,  
the spaces $A/A^n$ are not algebras so that there is no obvious algebra structure 
on their limit. This can be remedied by an additional assumption, roughly stating
that the $A^n$ are right ideals up to a shift that is locally bounded.  

 Specifically, let $A=\bigoplus_{n\in\nnN}A_n$ and assume there is a function  $\lclbd:\bigcup_nA_n\rightarrow \nnN$ 
 such that
 \begin{equation}\label{eq:Asubgrad}
     \mbox{for all}\quad n,m\in\nnN\,,\, x\in A_n\,,\, y\in A_m \quad \mbox{we have}\quad xy\in A^{m+\max(0,n-\lclbd(y))}\,, 
 \end{equation}
 where \quad 
 $A^n=\bigoplus_{k\geq n} A_k$\,. Denote \quad
 $\displaystyle\hat A=\varprojlim_n A/A^n\,$. We shall call $\lclbd$ the {\em local bound} of $A\,$. 

Clearly, by \eqref{eq:Asubgrad}, the $A^n$ are left (but not necessarily right) ideals, their intersection is zero, and we have the standard embedding $A\hookrightarrow\hat A$ as a $\Zgen$-module. For a second algebra $B$ of this form with local bound 
$\gamma:\bigcup_nB_n\rightarrow \nnN$ one may consider 
degree preserving homomorphisms
$f:A\rightarrow B$ in the sense that $f(A_n)\subseteq B_n$   as well as a grading of the tensor algebra $A\otimes B$ with
components $(A\otimes B)_n=\bigoplus_{s}A_s\otimes B_{n-s}$\,. 

The assumption of a decomposition into components $A_n$ can be circumvented with further work, but fits our applications to quantum groups. 

 \begin{lem}\label{lem:InvLimAlg}
 Let $A, B$ be algebras and $f:A\rightarrow B$ a homomorphism, all as above.  \vspace{-1mm}
    \begin{enumerate}[label=\roman*)]
        \item \label{item:InvLimAlg:algst}
     $\hat A$ admits a unique algebra structure for which
     $A$ is a subalgebra. \vspace{2mm}
     \item \label{item:InvLimAlg:homom}
     There is a unique algebra homomorphism $\hat f:\hat A\rightarrow\hat B$ that restricts to $f$ on $A$. \vspace{2mm}
     \item \label{item:InvLimAlg:tensor}
     $A\otimes B$ has the property \eqref{eq:Asubgrad} with local bound 
      determined by $\eta(y\otimes y')=\lclbd(y)+\gamma(y')\,$. The respective completion canonically contains $\hat A\otimes\hat B$ as a subalgebra and gives rise to an associative completion $\hat\otimes$ of the tensor product as $\widehat{A\otimes B}=\hat A\hat\otimes\hat B$. \vspace{2mm}
     \item \label{item:InvLimAlg:bialg}
     If $A$ admits a bialgebra structure with coproduct $\Delta(A_n)\subseteq (A\otimes A)_n\,$,  then
     $\hat A$ admits a canonical bialgebra structure with respect to the completed tensor, in the sense that $\hat\Delta:\hat A\rightarrow\hat A\hat \otimes\hat A\,$.\vspace{2mm} 
    \end{enumerate}
 \end{lem}

 \begin{proof}
     By definition, an element of the inverse limit $\hat A$ is uniquely presented by a sequence $(x_0,x_1,\ldots)\in\prod_{n\geq 0}A_n$\,, which is more suggestively written as a formal sum $\sum_{k\geq 0}x_k\,$. The submodule $A$ is given by sequences with only finitely many non-zero terms. For two elements $\sum_nx_n\,, \sum_my_m\in\hat A\,$, \eqref{eq:Asubgrad} implies that we
     can write 
     $$
     x_ny_m=\sum_{k\geq m, \; k+\lclbd(b_m)\geq n+m}z^{n,m}_k
     \qquad\mbox{for some}\quad z^{n,m}_k\in A_k\,. 
     $$ 
     For fixed $k$, the bounds imply $m\leq k$ and $n\leq k+r_k$ with $r_k=\max_{0\leq m\leq k}\lclbd(b_m)-m$. Thus, $x_ny_m$ has a component in $A_k$ only for a finite number of pairs $(n,m)$ so that
     $(xy)_k=\sum_{n,m}z^{n,m}_k$ is well-defined. It is clear that the so defined product yields an
     algebra structure as proposed in Item~{\em \ref{item:InvLimAlg:algst}}. Item~{\em \ref{item:InvLimAlg:homom}} is then also immediate.

     Suppose $x\in A_i$\,, $x'\in B_j$\,, $y\in A_s$\,, and $y'\in B_t$\,. Then the component of $(x\otimes x')(y\otimes y')=xy\otimes x'y'$ in $(A\otimes B)_k$ is, by \eqref{eq:Asubgrad}, applied to each tensor factor, zero if  
     $k\leq (s+\max(0,i-\lclbd(y))+(t+\max(0,j-\gamma(y'))\leq s+t+\max(0,(i+j-(\lclbd(y)+\gamma(y')))$, which implies the claim for pure component tensors. For a general tensor in $A_i\otimes B_j$\,, $\eta$ may be defined as the smallest number $d$ such that the tensor can be written as a sum of
     pure tensors, each of which with $\eta$-value less than $d$. For a general element in 
     $(A\otimes B)_n\,$, take the maximum over the $A_s\otimes B_{n-s}$ graded components. 

     The embedding of $\hat A\otimes\hat B$ into $\hat A\hat \otimes\hat B$ is obvious, and standard properties of $\hat \otimes$ are readily checked. Item~{\em \ref{item:InvLimAlg:bialg}} is now
     a consequence of {\em \ref{item:InvLimAlg:homom}} and {\em \ref{item:InvLimAlg:tensor}}.
 \end{proof}

For an $A$-module $V$ set $V_n=\{v\in V:\,A^n.v=0\}$, yielding an ascending chain of subspaces
$0=V_0\subseteq V_1\subseteq V_2\subseteq\ldots\,$. We say $V$ is {\em locally finite} 
 if $V=\bigcup_nV_n\,$. We collect several basic observations next.

\begin{lem}
\label{lem:LocFinMod}
 Suppose $V$ and $W$ are locally finite $A$-modules as above.  \vspace{-1mm}
    \begin{enumerate}[label=\roman*)]
    \item \label{item:LocFinMod:shift}
           If $y\in A_k$ and $v\in V_n$ then $y.v\in V_{n-k+\lclbd(y)}\,$. \vspace{2mm}
     \item \label{item:LocFinMod:extend}
           The action of $A$ extends to an action of $\hat A$ on $V\,$. \vspace{2mm}
     \item \label{item:LocFinMod:tensor}
           The action of $A\otimes A$ on $V\otimes W$ is locally finite and, thus, extends to an action of $\hat A\hat\otimes\hat A$. \vspace{2mm}
     \item \label{item:LocFinMod:category}
      If $A$ has a bialgebra structure as above, the locally finite modules form a tensor category. 
\end{enumerate} 
\end{lem}

 \begin{proof}
     Item~{\em\ref{item:LocFinMod:shift}} follows readily from the definitions, Item~{\em\ref{item:LocFinMod:extend}} holds since only finitely many terms are
     non-zero in any $\sum_ka_kv$\,, Item~{\em\ref{item:LocFinMod:tensor}} is a consequence
     of $(A\otimes A)^{n+m}(V_n\otimes W_m)=0\,$, and  Item~{\em\ref{item:LocFinMod:category}}
     is immediate from Item~{\em\ref{item:LocFinMod:tensor}}.
 \end{proof}
 
Observe that the category noted in the last statement does {\em not} need to be rigid. In fact,
the locally finite condition above emulates the respective property for Verma modules (or, more generally, modules in category $\mathcal O$) for which standard duals do not exist. Thus, at the level of algebras, 
we generally do not expect the antipode on a Hopf algebra $A$ to extend to the limit $\hat A$. 

The second ingredient required to define a braided structure from formal series are Tanisaki's axioms \cite{tan} that aim to reexpress the ill-defined diagonal parts of an $R$-matrix in terms of certain automorphisms. 

\begin{defn}\label{def:tani2}
    Let $A$ be a bialgebra with structure as in \eqref{eq:Asubgrad}, 
    $\Rtani\in \hat A\hat \otimes \hat A$ invertible, and $f$ a component-preserving algebra automorphism of $A\otimes A$. We say $(A,\Rtani, f)$ is a {\em pre-triangular bialgebra} if the following hold:
    \begin{equation}\label{eq:tani2}
\begin{aligned}
     f_{12}(\Rtani_{13})\Rtani_{12}&=(\id\otimes\Delta)(\Rtani), \qquad f_{23}(\Rtani_{13})\Rtani_{23}=(\Delta\otimes \id)(\Rtani)\,,\vspace{4mm}
  \\
    \mbox{and}&\qquad \Rtani\Delta(x)=(f\circ \Deltaopp(x))\Rtani,
    \quad \text{for all } \; x\in A\,. 
\end{aligned} 
    \end{equation} 
\end{defn}
 
The notation in \eqref{eq:tani2} omits the $\,\hat{}\,$ accent on automorphisms and coproducts.
Since $f$ is assumed to preserve components, such an extension to $\hat A\hat \otimes \hat A$
exists. Similarly, one checks that, for example, 
$\hat\Delta\hat\otimes \id:\hat A^{\hat \otimes 2}\rightarrow\hat A^{\hat \otimes 3}$ is well-defined. 
The subscripts for $\Rtani$ indicate the usual canonical inclusions extended to limits
$\hat A^{\hat \otimes 2}\hookrightarrow\hat A^{\hat \otimes 3}$ and the subscripts on
the automorphism mean, for example, $f_{23}=\id\hat\otimes\hat f$. 

Assume a pre-triangular bialgebra $A$ with properties as in \eqref{eq:Asubgrad} as well as a 
tensor (sub) category $\mathcal M$ of locally finite modules. We say $(A,\Rtani,f)$ is {\em implementable} in
$\mathcal M$ if for any two modules $V,W\in\mathcal M$ there is an isomorphism
$\RtaniImp{V}{W} :V\otimes W\rightarrow V\otimes W$ such that for all $V,W,Z\in\mathcal M$ the
following hold. 
\begin{equation}\label{eq:DVWconds}
\begin{aligned}
    f(c)&=\RtaniImp{V}{W}^{-1} c \RtaniImp{V}{W} \quad\mbox{for all}\; c\in A\otimes A\;\mbox{acting on }\;V\otimes W\, \vspace*{3mm}\\
    \RtaniImp{V\otimes W}{Z}&=(T\otimes \id)(\id\otimes\RtaniImp{V}{Z})(T\otimes \id)(\id\otimes\RtaniImp{W}{Z}) \vspace*{3mm}\\
    \RtaniImp{V}{W\otimes Z}&=(T\otimes \id)(\id\otimes\RtaniImp{V}{Z})(T\otimes \id)(\RtaniImp{V}{W}\otimes \id) 
\end{aligned}
\end{equation}
 
Here $T$ denotes the ordinary transposition of tensor factors. We suppress notation 
for the representation map $A\rightarrow \mathrm{End}(V)$ or 
$A\otimes A\rightarrow \mathrm{End}(V\otimes W)$ whenever the respective actions are clear. 
A routine verification of categorical triangle (hexagon) relations yields the next observation,
which is central to many applications of quantum groups.

\begin{cor}\label{cor:Mbraiding}
    If $(A,\Rtani,f)$ is implementable in $\mathcal{M}$, then the natural isomorphism 
    $$
    R(V,W)=T\RtaniImp{V}{W}\Rtani\,:\;V\otimes W\rightarrow W\otimes V
    $$
defines a braiding on $\mathcal{M}\,$.
\end{cor}

An immediate consequence is the existence of representations of the braid group $B_n$ on
$V^{\otimes n}$. The axioms in \cite{tan} assume also relations 
$f_{23}\circ f_{13}(\Rtani_{12})=\Rtani_{12}$ and 
$f_{13}\circ f_{23}(\Rtani_{23})=\Rtani_{23}$ in addition 
to those in \eqref{eq:tani2} but omit the relations for $\RtaniImp{V\otimes W}{Z}$
and $\RtaniImp{V}{ W \otimes Z}$ in \eqref{eq:DVWconds}. This leads to the same braid
representations without the categorical context. 

\subsection{Completions and Quasi-\texorpdfstring{$R$}{R}-Matrices for \texorpdfstring{$\Uq$\,}{Uq}}\label{subsec:QuasiR_genq}
We will construct the quasi-$R$-matrices for $\UqQ=\Uq\otimes\mathbb Q(q)$ from
simple quasi-$R$-matrices, \textit{a priori} given by the formal expressions
\begin{equation}\label{eq:ElemRmatGen}
    \Ri=\sum_{s=0}^{\infty}\dfrac{q^{s(s-1)/2}_{i}}{[s]!_{i}}(q_{i}-q^{-1}_{i})^s E_{i}^s\otimes F_{i}^s =\sum_{s=0}^{\infty}q^{s(s-1)/2}_{i}(-1)^s \Esing_{i}^s\otimes F_{i}^{(s)} \qquad
    \in {\UqQC}\hat\otimes {\UqQC}
\end{equation}
for any $i=1,\ldots, n\,$. The renormalized generators in the second expression are as in \eqref{eq:def:EgenDivPow} and \eqref{eq:def:EgenSing}. Several limits $\UqQC$ and completions $\hat\otimes$ may be considered to 
provide a rigorous meaning to \eqref{eq:ElemRmatGen}. The ones relevant here start from the
$\nnN$-gradings $\wgradp$ and $\wgradn$ on $\Uq^+$ and $\Uq^-$ introduced in 
Section~\ref{subsec:gradings}. The height function from Section~\ref{subsec:LettSpecElem} extends to  a homomorphism $\height:\nnN^{\sroots}\rightarrow\nnN$\,, which, thus, entails respective integer-gradings
$\height\circ \wgrad^{\pm}$\,. We denote the corresponding components as
\begin{equation}
    \UqQ^{k,+}=\bigoplus_{\height(\mu)=k}(\UqQ^+)_\mu^{\wgradp}\qquad\mbox{and}\qquad \UqQ^{k,-}=\bigoplus_{\height(\mu)=k}(\UqQ^-)_\mu^{\wgradn}\,. 
\end{equation}
Also set \quad $\UqQ^{\geq k,\pm}=\bigoplus_{s\geq k}\UqQ^{s,\pm}\,$. 
The completion suggested, for example, in \cite{CP95} fits  the construction of Section~\ref{subsec:GenTensStr} with
\begin{equation}\label{eq:CP95limAn}
    A_n=\UqQ\cdot\UqQ^{n,+}\qquad\mbox{and}\qquad A^n=\UqQ\cdot\UqQ^{\geq n,+}\,.
\end{equation}
Note that, by the PBW theorems, $A_n=\bigoplus_{t\in\nnN}A_{t,n}\,$, where
$A_{t,n}=\UqQ^{t,-}\UqQ^0\UqQ^{n,+}$\,. A function
$\lclbd:A_n\rightarrow \nnN$ is now  defined by setting  $\lclbd(y)=m$ to be the 
smallest $m$ such that $y\in \bigoplus_{t\leq m}A_{t,n}\,$. 

\begin{lem}\label{lem:CP95limCheck}
    The $A_n$ and $\lclbd$, as defined above, fulfill \eqref{eq:Asubgrad}. Thus, $\UqQC=\varprojlim\UqQ/A^n$ is a bialgebra with $\hat\otimes$ as in Lemma~\ref{lem:InvLimAlg}.
\end{lem}

\begin{proof}
    Set $B_{r,t}=\UqQ^{r,+}\cdot\UqQ^{t,-}\,$. We first observe that  
    $B_{r,t}\subseteq\bigoplus_{j=0}^{\min(r,t)}A_{t-j,r-j}\,$. This follows by 
    induction on expressions $E_{u_1}\ldots E_{u_r}F_{v_1}\ldots F_{v_t}$
    that span $B_{r,t}$\,, using the basic commutation relations in 
    \eqref{eq:EFcomm} and \eqref{eq:Kcomm}. Multiplying this
    inclusion from the left with $\UqQ$ and from the right with $\UqQ^0\cdot \UqQ^{m,+}$ we find 
    $A_r\cdot A_{t,m}\subseteq \bigoplus_{j=0}^{\min(t,r)}\UqQ\cdot A_{t-j,r-j}\cdot \UqQ^{m,+}=\bigoplus_{j=0}^{\min(t,r)}A_{r+m-j}\subseteq A^{r+m-\min(t,r)}=A^{m+\min(r-t,0)}$, which implies \eqref{eq:Asubgrad}.
    Since the coproduct on $\UqQ^\geqzero$ preserves the $\wgradp$-grading, the additional assumption
    in Lemma~\ref{lem:InvLimAlg} holds as well, leading to a well-defined bialgebra structure on the limit with respect to the completion $\hat\otimes\,$. 
\end{proof}

The locally finite representations for this completion encompass all BGG-Category $\mathcal{O}$ modules as in \cite{AM15}. This includes infinite-dimensional Verma modules, for which the entailed braid representations are of interest, as they relate to geometrically constructed ones \cite{JK11}.

To see that the antipode does not extend to the limit, consider the formal element $x=\sum_{m\geq 0} F^mE^m\,$ in $\hat U_q(\mathfrak{sl}_2)$\,. From \eqref{eq:antipode_gen} we find $S(F^mE^m)=E^mF^m=\sum_{j=0}^mc^m_jF^jE^j$ with $c^m_j\in\Uq^0\,$. In particular,
$c^m_0=[m]!\prod_{s=0}^{m-1}\frac{Kq^{-s}-K^{-1}q^s}{q-q^{-1}}\,$ leading to ill-defined polynomials of infinite order in the 0-component of $S(x)\,$.

We do, however, have a well-defined extension of the antipode to the respective limits of Borel subalgebras $\UqQC^\geqzero$ and $\UqQC^\leqzero\,$. 

As an alternative to the above limit, we may consider, for example, the two-sided ideals $I^{n,+}=\langle\UqQ^{n,+}\rangle$
generated by elements in $\UqQ^{n,+}$. It is not hard to see that
$\Delta(I^{n,+})\subseteq\bigoplus_{k}I^{n-k,+}\otimes I^{k,+}$. Thus, the inverse limit
$\tilde U_{q,\mathbb Q}=\varprojlim\Uq/I^{n,+}$ yields again a bialgebra structure and tensor-completion 
for which \eqref{eq:ElemRmatGen} is well-defined.

As opposed to the previous situation, the algebra structure of $\tilde U_{q,\mathbb Q}$ here follows trivially from the quotients, which are algebras. Moreover, the antipode is well-defined on the limit. 
However, the set of $\UqQ$-modules to which $\tilde U_{q,\mathbb Q}$ extends is much smaller. 
Specifically, for any cyclic submodule $W$ there needs to exist some $n$ such that 
$\UqQ^{n,+}$ vanishes on all of $W$. Thus, the only allowed highest weight modules are finite-dimensional, which excludes, for example, all Verma modules.

Using \eqref{eq:q-combin-form} with $z=-q^2$ one can check by direct computation that the following expression is an inverse in $\UqQC\hat{\otimes}\UqQC$\,.
\begin{align}\label{eq:eq:ElemRmatInv}
    \Ri^{-1}=\sum_{s=0}^{\infty}\dfrac{q^{-s(s-1)/2}_{i}}{[s]!_{i}}(-1)^s(q_{i}-q^{-1}_{i})^s E_{i}^s\otimes F_{i}^s
\end{align}

Note that the braid automorphisms from Sections~\ref{subsec:LATM} also do {\em not} extend to the full 
completion $\UqQC\,$. Specifically, in rank 1, $\Tinv_1(x)$ is ill-defined on the same element $x=\sum_{m\geq 0}F^mE^m$ for the same reason as for the anitpode above. The automorphisms are, however, well-defined as linear operators 
on certain linear subspaces in $\UqQC$\,, whose construction we outline next.  

In the notation of Section~\ref{subsec:Emonomials}, observe that for fixed $z\in\wordsetmax$ and $k\in\nnN$\,, there are only finitely many $\psi\in \expsetupmax$ with $n(\psi)=\wgrad(\Ebaseopp{z}{\psi})=k$. 
Thus, an element in $\UqQC$ can be written uniquely as a formal sum 
$\sum_{\psi\in\expsetupmax}R_\psi\Ebaseopp{z}{\psi}$ with $R_\psi\in\UqQ^{\leqzero}\,$. Similarly, an element in $\UqQC\hat{\otimes}\UqQC$ is formally given by $\sum_{\phi,\psi\in\expsetupmax}R_{\phi,\psi}
(\Ebaseopp{z}{\phi}\otimes\Ebaseopp{z}{\psi})$ with $R_{\phi,\psi}\in(\UqQ^{\leqzero})^{\otimes 2}\,$. 

As in \eqref{eq:DefSspans}, denote by $\bspanopp{w}{-}$ the span of $\Fbaseopp{w}{\psi}$
for a (not necessarily maximal) reduced word $w\in\wordset$, $s=\Weylpres(w)\,$, and  $\psi\in\expsetup{s}\,$. Next, for  $s=\Weylpres(w)$, define the linear subspace of formal sums
\begin{align}\label{def:Lsubspace}
    \TCompModGenopp{w}=\biggl\{\sum_{\psi\in\expsetup{s}}R_\psi\Ebaseopp{w}{\psi}:R_\psi\in\bspanopp{w}{-}\!\cdot\UqQ^0\biggr\} \quad \subseteq \UqQC\,.
\end{align}
These subspaces now generate the $\UqQ$-sub-bimodules 
$\TCompModopp{w}=\UqQ\cdot\TCompModGenopp{w}\cdot\UqQ$\,. 
Analogously, the subspaces $\TCompModGenopptwo{w}=\TCompModGenopp{w}\hat\otimes\TCompModGenopp{w}$ are defined
from formal sums $\sum_{\phi,\psi\in\expsetup{s}}R_{\phi,\psi}(\Ebaseopp{w}{\chi}\otimes\Ebaseopp{w}{\psi})$ with $R_{\phi,\psi}\in (\bspanopp{w}{-}\!\cdot\UqQ^0)^{\otimes 2}$, generating the 
$\UqQ^{\otimes 2}$-sub-bimodules $\TCompModopptwo{w}=\TCompModopp{w}\hat\otimes\TCompModopp{w}$\,.

The next lemma follows by applying  \eqref{eq:BasisRecurs} or \eqref{eq:PBWgenRecur} to the limiting expressions.

\begin{lem}\label{lem:WordLims}
Suppose $u_1,u_2\in\wordset$ are reduced words such that
$w=u_1\cdot u_2$ is reduced. Let $t_i=\Weylpres(u_i)$ and $s=t_1\cdot t_2=\Weylpres(w)$
so that $\len s =\len {t_1}+\len {t_2}\,$. 
Then $\Tinv_{t_1}$ extends to a linear map on $\TCompModGenopp{u_2}$ as well as  $\TCompModopp{u_2}$ with
images $\Tinv_{t_1}(\TCompModopp{u_2})\subseteq\TCompModopp{w}$ and
$\Tinv_{t_1}(\TCompModGenopp{u_2})\subseteq\TCompModGenopp{w}$ in $\UqQC\,$. 

\noindent Analogously, $\Tinv_{t_1}^{\otimes 2}$ extends to a linear map on $\TCompModGenopptwo{u_2}$ and  $\TCompModopptwo{u_2}$ with
images $\Tinv_{t_1}^{\otimes 2}(\TCompModopptwo{u_2})\subseteq\TCompModopptwo{w}$ and
$\Tinv_{t_1}^{\otimes 2}(\TCompModGenopptwo{u_2})\subseteq\TCompModGenopptwo{w}$ in $\UqQC^{\hat\otimes 2}\,$.
\end{lem}

Clearly, $\Ri,\Ri^{-1}\in\TCompModGenopptwo{w_i}$  with $w_i$ a word of length one. Suppose $w\in \wordset$
with $\tau(w)=i$, so that $w=w^\flat \cdot w_i$\,. Then Lemma~\ref{lem:WordLims} allows us to define an \emph{elementary} quasi-$R$-matrix associated to the reduced word $w$ as
\begin{align}\label{eq:defPw}
    \Rw=
    \Tinv_{w^\flat}\otimes  \Tinv_{w^\flat}(\qRm_{i})=\sum_{j=0}^{\infty}\dfrac{q^{j(j-1)/2}_{i}}{[j]!_{i}}(q_{i}-q^{-1}_{i})^j E_{w}^j\otimes F_{w}^j\,.
\end{align}

Following the conventions of \eqref{eq:root2word} and Section~\ref{subsec:descroots}, we write $\beta^j=\wordroot(w[1,j])$ and $w^j=w[1,j]=w[\beta_j]$\,. Assume that $L=\len{w}$\,. In the notation of Section~\ref{subsec:Emonomials},
we define the {\em partial} quasi-$R$-matrix associated to $w$ by the following products in $\UqQC\hat\otimes\UqQC\,$.
\vspace{2mm}
\begin{align}\label{eq:Pw-genform}
    \Rprod{w}{}=\Ra{w^L}\ldots\Ra{w^1}
    =\Ra{w[\beta_L]}\ldots\Ra{w[\beta_1]}=\sum_{\psi\in\expsetup{s}}\Rcoeff{\psi}{q}^{-1}\cdot
    \Ebaseopp{w}{\psi}\otimes \Fbaseopp{w}{\psi}\quad\in \TCompModGenopptwo{w}\,
\end{align}
Here $s=\Weylpres(w)$ and the coefficients in $\mathbb Q(q)$ are given by 
\begin{equation}\label{eq:RcoeffGen}
\begin{aligned}
    \Rcoeff{\psi}q&=\frac {[\psi]!} {\Rcoeffdenom{\psi}q}\quad\in\Zqqvn{\dpone}
\quad\mbox{where}\quad
[\psi]!=\prod_{\alpha\in\proots}[\psi(\alpha)]!_{\alpha}\\
\rule{0mm}{6mm}
\quad\mbox{and}\quad   \Rcoeffdenom{\psi}q&= \prod_{\alpha\in\proots}\!q_\alpha^{\psi(\alpha)(\psi(\alpha)-1)/2}(q_\alpha-q_\alpha^{-1})^{\psi(\alpha)}=(-1)^{n(\psi)}\cdot\SingCoeffvar{\psi}{q}\cdot q^{\sum_\alpha d_\alpha\binom{\psi(\alpha)}{ 2}}\;.
\end{aligned}
\end{equation}
The elements $[\psi]!$ and $\SingCoeffvar{\psi}{q}$ are as defined in \eqref{eq:PBWEgenDivDef} and \eqref{eq:PBWEgenSingDef},  respectively.
Thus, in the product expressions, we consider all positive roots rather than just those in $\descroots{s}$
and assume $\psi\in\expsetup{s}$ is extended to all of $\proots$ by setting $\psi$ to zero outside of 
$\descroots{s}$.

The inclusion in $\Zqqvn{\dpone}$ is obvious from the definitions, and the 
formal sum expression in \eqref{eq:Pw-genform} is immediate from the definitions 
in \eqref{eq:PBWEgenDef}. It also implies the indicated inclusion into $\TCompModGenopptwo{w}\,$. 

Let $w=u_1\cdot u_2\in\wordset$ and $t_i=\Weylpres(u_i)$ as above so that $\len s=\len{t_1}+\len{t_2}$
for $s=t_1t_2=\Weylpres(w)\,$. Lemma~\ref{lem:WordLims} now implies that $\Tinv^{\otimes 2}_{t_1}(\Rprod{u_2}{})$ is well-defined and given by applying $\Tinv^{\otimes 2}_{t_1}$ to individual summands. 
Since the $q_\alpha$ depend only the $\Weyl$-orbit of $\alpha$, one readily checks that  
$\Rcoeff{s^*(\psi)}q=\Rcoeff{\psi}q$\,. In the situation where $\psi$ and $(\psi_1,\psi_2)$ are
related as in \eqref{eq:ExpsetCorr}, we infer from this that 
$\Rcoeff{\psi}q=\Rcoeff{\psi_1}q\Rcoeff{(t_1^{-1})^*(\psi_2)}q=\Rcoeff{\psi_1}q\Rcoeff{\psi_2}q$.
Together with \eqref{eq:PBWgenRecur} we then find the following well-defined recursion.
\begin{equation}\label{eq:PTrecurs}
    \Rprod{w}{}\,=\,\Tinv^{\otimes 2}_{t_1}\left(\Rprod{u_2}{}\right)\cdot\Rprod{u_1}{}\quad 
    \in\TCompModGenopptwo{w}\,
\end{equation}

A central property of the $\Rprod{w}{}$ elements is that they intertwine
the standard coproduct $\Delta$ with $\Delta^{\Tinv_w}$, the coproduct conjugated by the braid automorphism for the same word. 
Variants of the formula in the following proposition also appear in \cite[pg.429]{KR90}, \cite[Prop.2.4.1]{LS90}, and \cite[Prop.3.2.1]{KS98}, where they are derived using a calculus of quantum Weyl elements. We discuss in Appendix \ref{sec:InExWeyl} the nonexistence of such quantum Weyl elements in our more general setting and will avoid their use in the proof below. 

\begin{prop}\label{prop:AdP=DelT} Suppose $w\in\wordset$ and $x\in \UqQ$\,. Then the following
relation holds in $\TCompModopptwo{w}\,$,
\begin{equation}\label{eq:AdP=DelT} 
    \Rprod{w}{}\cdot\Delta(\Tinv_w(x))=\Tinv_w^{\otimes 2}(\Delta(x))\cdot\Rprod{w}{} \,.
\end{equation} 
\end{prop}

\begin{proof}  The proof proceeds by induction on the length of $w\,$. For  a length one word $w=w_i$ we have $\Rprod{w_i}{}=\Ri$\,, reducing \eqref{eq:AdP=DelT} to 
$$
\Ri\cdot\Delta(\Tinv_i(x))=\Tinv_i^{\otimes 2}(\Delta(x))\cdot\Ri\,.
$$
This case is verified, for example, by taking a formal limit $\elli\rightarrow \infty$ for the computations in Appendix~\ref{sec:proofcoprodbase}. This relation for simple roots is also derived by direct computation in 
\cite[App.C.4]{AJS94}.

Assume next that \eqref{eq:AdP=DelT} holds for $w\in\Weyl$ and that $w\cdot w_i\in\Weyl$ is reduced. From \eqref{eq:PTrecurs} we  obtain $\Rprod{w\cdot w_i}{}\,=\,\Tinv^{\otimes 2}_{w}(\Ri)\cdot\Rprod{w}{}$ as a well-defined relation in $\TCompModGenopp{w\cdot w_i}\hspace{-1ex}$. 
The expressions and calculations below are now well-defined on $\TCompModopp{w\cdot w_i}\hspace{-1ex}$
by Lemma~\ref{lem:WordLims}. They 
use both the induction hypothesis and the base case relations.
\begin{align*}
    \Tinv^{\otimes 2}_{w\cdot w_i}(\Delta(x))\Rprod{w\cdot w_i}{}
       &=
          \Tinv^{\otimes 2}_{w}\circ\Tinv^{\otimes 2}_{i}(\Delta(x))\cdot 
          \Tinv^{\otimes 2}_{w}(\Ri)\cdot \Rprod{w}{}
       &&=
          \Tinv^{\otimes 2}_{w}\left(\Tinv^{\otimes 2}_{i}\left(\Delta(x)\right)\cdot\Ri\right)\cdot\Rprod{w}{}\\
       \rule{0mm}{5mm}&=
          \Tinv^{\otimes 2}_{w}\left(\Ri \cdot \Delta(\Tinv_{i}(x))\right)\cdot\Rprod{w}{}
       &&=
          \Tinv^{\otimes 2}_{w}\left(\Ri\right)\cdot
          \Tinv^{\otimes 2}_{w}\left(\Delta(\Tinv_{i}(x))\right)\cdot\Rprod{w}{}\\
       \rule{0mm}{5mm}&=
          \Tinv^{\otimes 2}_{w}(\Ri)\cdot\Rprod{w}{}\cdot\Delta(\Tinv_{w}(\Tinv_{i}(x)))
       &&=
          \Rprod{w\cdot w_i}{}\cdot\Delta(\Tinv_{w\cdot w_i}(x))   
\end{align*}
This shows the assertion for $w\cdot w_i$\,, completing the proof.
\end{proof}

An identity similar to \eqref{eq:AdP=DelT} can be found in \cite[Prop 8.3.7]{CP95} in the special case of $x=E_i$ with $w\cdot w_i$ reduced. The derivation there also involves certain quantum Weyl elements, which generally exist only in representations that admit a compatible $\Weyl$-action and, hence, limiting the statement to an identity between linear maps on such representations.
We, therefore, abstain from the use of techniques using Weyl elements here. 
See, Appendix~\ref{sec:InExWeyl} for further details. 

A {\em full} quasi-$R$-matrix is associated to any reduced word of maximal length. The first step in verifying independence of the choice of word is an intertwining relation with a twisted coproduct defined on
generators by
\begin{equation}\label{eq:AltCoprodDef}
    \barDelta(E_i)=E_i\otimes K_i^{-1}+1\otimes E_i\,, 
    \qquad \barDelta(F_i)=F_i\otimes 1 + K_i\otimes F_i\,,
    \quad\mbox{and}\quad 
    \barDelta(K_i)=K_i\otimes K_i\;. 
\end{equation}
It is easy to check that $\barDelta\circ\Kconaut=\Kconaut^{\otimes 2}\circ\Delta$, where 
$\Kconaut$ is as in \eqref{eq:defCaut}. The latter involution coincides with the bar involution
in \cite[3.1.12]{lu} so that $\barDelta$ coincides with the alternate coproduct defined
in \cite[4.1.1]{lu}. The next result uses the fact that, for maximal length words, the 
coproduct is conjugated by the Garside automorphism, for which we derived an explicit form 
in Section~\ref{subsec:GarsAut+DictGens}. 

\begin{lem}\label{lm:PzDel=DelbPz}
Suppose $z\in\wordsetmax$ is a reduced word of maximal length and $x\in\UqQ$\,.
    Then
    \begin{equation}\label{eq:PzDel=DelbPz}
        \Rprod{z}{}\cdot \Delta(x)=\barDelta(x)\cdot\Rprod{z}{}\,.
    \end{equation} 
\end{lem}

\begin{proof}
    Recall from Corollary~\ref{cor:GarsideAutom} that $\Tinv_{\longweyl}
    =\DynkInv\circ\Cartanaut\circ\Kinvaut\circ S$ for the longest 
    element $\longweyl\in\Weyl\,$. It is easy to check that
    $\Cartanaut\circ\Kinvaut=\Kconaut\circ \Cartaninv$ for the commuting involutions defined in
    Section~\ref{subsec:Kscale}. Thus, $\Tinv_{\longweyl}=\Kconaut\circ Q$\,, where
    $Q=\DynkInv\circ\Cartaninv\circ S$. Recall that both $\Cartaninv, S:U_q\rightarrow U_q^{\opp,\cop}$ are anti-coalgebra anti-automorphisms and $\DynkInv$, defined in \eqref{eq:defDynkInv}, is a Hopf algebra automorphism. Thus, $Q$ intertwines the coalgebra structure $Q^{\otimes 2}(\Delta(y))=\Delta(Q(y))$. Given that $\longweyl=\Weylpres(z)\,$, Proposition~\ref{prop:AdP=DelT} now implies
    \begin{align*}
        \Rprod{z}{}\cdot \Delta(x)
    &=
    \Tinv_{\longweyl}^{\otimes 2}\left(\Delta\bigl(\Tinv_{\longweyl}^{-1}(x)\bigr)\right)\cdot\Rprod{z}{}
     &&=
    \Kconaut^{\otimes 2}\circ Q^{\otimes 2}\left(\Delta\bigl(Q^{-1}\circ\Kconaut^{-1}(x)\bigr)\right)\cdot\Rprod{z}{}\\
   \rule{0mm}{6mm} &=
    \Kconaut^{\otimes 2}\left(\Delta(\Kconaut^{-1}(x))\right)\cdot\Rprod{z}{}
    &&=\barDelta(x)\cdot\Rprod{z}{}\;.
    \end{align*} \vspace*{-15mm}
    
\end{proof}

In \cite[Thm.4.1.2(a)]{lu} Lusztig shows that an element $\Theta$ in some completion of 
$\oplus_\nu(\UqQ^-)_\nu\otimes(\UqQ^+)_\nu$ is uniquely determined by the relation
$\Theta\cdot\barDelta^{\opp}(y)=\Deltaopp(y)\cdot\Theta$ (since \cite{lu} uses the opposite coproduct of ours).
The required assumptions can be readily checked for $\Theta=(\Rprod{z}{-1})^{\opp}$ using the 
general form as in \eqref{eq:Pw-genform}, the commutation \eqref{eq:PzDel=DelbPz},  and 
noting that the completions align. From the uniqueness of solutions to (\ref{eq:PzDel=DelbPz}) we infer word independence for quasi-$R$-matices associated to reduced words of maximal length. 

\begin{cor}\label{cor:RprodWI}
There exists $\,\Rprodmax\in\UqQC^{\hat\otimes 2}\,$ such that $\Rprodmax=\Rprod{z}{}\;$  for all  $z\in\wordsetmax$\,.  
Moreover, $\Kconaut^{\otimes 2}(\Rprodmax)=\Rprodmax^{-1}$.
\end{cor}  

\medskip

\subsection{Pre-Triangular Structures for \texorpdfstring{$\Uq$\,}{Uq}}\label{subsec:PreTriang}

Since the coefficients in \eqref{eq:defPw} may be evaluated at $q=1$\,, the element $\,\Rprodmax\,$ 
has a well-defined image $\,\Rprodmax^\hbar\,$ in the power series completion $\Uhg\hat\otimes\Uhg$
over $\Qh$ as in Section~\ref{subsec:qgroups-genrel}. 
Originally, universal $R$-matrices were constructed in this setting by Drinfeld \cite{Dri87} as products 
$\mathcal R^\hbar=\mathcal \Rexp\cdot\Rprodmax^\hbar$ in $\Uhg^{\hat\otimes 2}\,$. The ``diagonal'' term is given by 
\begin{equation}\label{eq:defRexp}
    \Rexp=\exp\left(\hbar \sum_{i,j=1}^n (\symCartinv^{-1})_{ij}H_i\otimes H_j \right)\,,
\end{equation}
where 
$\symCartinv_{ij}
   =A_{ij}d_j^{-1}
    =\symbrack{\breve\alpha_i}{\breve\alpha_j}
    =\symCart_{ij}d_i^{-1}d_j^{-1}\in {\maxd}^{-1}\mathbb Z
$ 
with notation as in Section~\ref{subsec:rootweyl}. 
Explicit formulae are obtained in \cite{Ro89} for the $\LT{A}_n$ case and \cite{KR90} for general Lie types, see also \cite{CP95}. Tanisaki's axioms in \cite{tan} or \eqref{eq:tani2} above are motivated by the fact
that $\Rexp$ cannot be defined in $\UqQ^{\otimes 2}$\,, but conjugation by $\Rexp$ can be expressed as an automorphism on $\UqQC^{\hat\otimes 2}\,$.

Specifically, if $x$ and $y$ are homogeneously graded elements with 
$\wgrad(x)=\mu,\wgrad(y)=\nu\in\mathbb Z^{\sroots}$\,, we have $[H_i\,,x]=\symbrack{\mu}{\breve \alpha_i}x$
and $[H_j\,,y]=\symbrack{\nu}{\breve \alpha_j}y$. From this we obtain
$\sum_{i,j}(B^{-1})_{ij}[H_i\otimes H_j\,,x\otimes y]=x\otimes (H_\mu \cdot y)+(x\cdot H_\nu)\otimes y$, where
$H_\mu=\sum_i\mu_id_iH_i$ if $\mu=\sum_i\mu_i\alpha_i\,$. Exponentiation $t\mapsto e^{-\hbar t}$ of the implied adjoint and
regular actions on $x\otimes y$ together with the notation from \eqref{eq:Kexpdef} 
then yields the identities  below. 
\begin{equation} \label{eq:tanipsi}
\taniauto(x\otimes y)=\Rexp^{-1}(x\otimes y)\Rexp
=
 xK^{-\nu}\otimes K^{-\mu}y
 =
 q^{-\symbrack{\mu}{\nu}}xK^{-\nu} \otimes yK^{-\mu}\,
\end{equation}

Thus, although the Tanisaki automorphism is initially defined as $\taniauto=\mathrm{Ad}(\Rexp^{-1})$
on $\Uhg^{\hat\otimes 2}\,$, it clearly maps the subalgebras $\Uq^{\otimes 2}$ and $\UqQC^{\hat\otimes 2}$
to themselves. The form in (\ref{eq:defRexp}) can often still be defined, at least  up to scalars, on
tensor products of certain modules of $\Uq^{\otimes 2}$, implying 
solutions to the Yang Baxter equations or solutions to (\ref{eq:DVWconds}). 

For example, for generic highest weight representations $V$ and $W$, the maps $\RtaniImp V W$ can be designed starting from generalized Verma modules. For a tuple $\kappa=(\kappa_1,\ldots,\kappa_n)\in\mathbf u^{\sroots}$ with $\mathbf u$ a group of units in the ground ring, the {\em generic Verma module} $V(\kappa)$ is defined as the induced representation from the one-dimensional $\Uq^\geqzero$ representation, given by $K_iv_\kappa=\kappa_iv_\kappa$ and $E_iv_\kappa=0$ on a basis vector $v_\kappa$\,. As usual, there is a
linear isomorphism $\Uq^-\rightarrow V(\kappa): x\mapsto x.v\,$.

Assume now for homogeneous $x,y\in\Uq^-$ that $-\wgrad(x)=\wgradn(x)=\mu$ and $-\wgrad(y)=\wgradn(y)=\nu$ and write $\kappa^\mu=\kappa_1^{\mu_1}\ldots\mu_n^{\mu_n}$ for $\mu\in\mathbb Z^{{\sroots}}$. Then linear maps on tensor 
products of generic Verma modules are well-defined by the formula below. They are unique solutions to
the equation 
 $D^{-1}cD=\taniauto(c)$ of module maps (with $\taniauto$ as in \eqref{eq:tanipsi})  up to an overall scalar factor.
\begin{equation}\label{eq:DmapVerma}
    D_{V(\kappa),V(\lambda)}((x.v_\kappa)\otimes(y,v_\lambda))=q^{\symbrack{\mu}{\nu}}\kappa^{-\nu}\lambda^{-\mu}((x.v_\kappa)\otimes(y.v_\lambda))\,
\end{equation}

It is not hard to check that the respective isomorphisms 
$V(\kappa)\otimes V(\lambda)\rightarrow V(\lambda)\otimes V(\kappa)$, as defined in Corollary~\ref{cor:Mbraiding}, obey the Yang Baxter Equation and, thus, produce a 
representation of the colored braid groupoids over $\mathbb Q(q)[\{\kappa_i^{\pm 1},\lambda_i^{\pm 1},\ldots\}]\,$.

The categorical triangle relations may be realized in a more general setting, in which we consider a
commutative ring $\ExpGrp$ with $\mathrm{char}(\ExpGrp)=0$ and an extension of 
$\symbrack{\,\cdot\,}{\,\cdot\,}\,:\ExpGrp^{\sroots}\times \ExpGrp^{\sroots}\rightarrow \ExpGrp$ for the form from \eqref{eq:def:symbrack}. Elements of the group ring $\mathbb Z[\ExpGrp]$ may be written
formally as $q^x$ with $x\in\ExpGrp$\,. Denoting $\mathbb Q(\ExpGrp)$ as the field of fractions,
we have $\mathbb Q(\mathbb Z)=\mathbb Q(q)\,$. 
As a tensor category, one may then consider representations of $\Uq$ on $\ExpGrp^{\sroots}$-graded
vector spaces $V=\bigoplus_\gamma V_\gamma$ subject to the condition that $K_iv=q^{\symbrack{\breve\alpha_i}{\gamma}}v$ for
$v\in V_\gamma$ and $\gamma\in \ExpGrp^{\sroots}\,$. 
For two graded representations $V$ and $W$,
define then $D_{V,W}$ to be $q^{\symbrack{\gamma}{\eta}}$ on the $V_\gamma\otimes V_\eta$
component. 

The relations (\ref{eq:DVWconds}) and (\ref{eq:tanipsi}) are now easily verified
on such modules, yielding a braided tensor category $\mathcal C$ over $\mathbb Q(\ExpGrp)$\,. As an abelian
category the latter naturally splits in a sum $\bigoplus_\rho \mathcal C_\rho$ with $\rho\in\ExpGrp^1=\ExpGrp/\mathbb Z\,$ with the obvious $\ExpGrp^1$-graded behavior for tensors. 

The main result for the generic case is summarized in the proposition below. It has been stated at various levels of rigor and formality in numerous other places in the literature. We adapt here results from \cite{lu} to our conventions, which itself is following similar calculations as those in \cite[Theorem 4.3.3]{tan}. 

\begin{prop}
    The triple $(\UqQC\,,\Rprodmax\,,\taniauto)$ is a pre-triangular bialgebra.
\end{prop}

\begin{proof} A straightforward verification on generators shows that $\barDelta=\taniauto\circ\Deltaopp$. 
Lemma~\ref{lm:PzDel=DelbPz} then implies the last relation in \eqref{eq:tani2}.

The argument preceding Corollary~\ref{cor:RprodWI} identifies Lusztig's universal element with our quasi-$R$-matrix via $\Theta=(\Rprodmax^{-1})^\opp$ or, equivalently,  $\overline\Theta^\opp=\Rprodmax\,$ so that, restricted to
weight components, we have for any $\nu\in\mathbb Z^{\sroots}$\,
$$
\overline\Theta^\opp_\nu=(\Rprodmax)_\nu=\sum_{\psi:\,\vec \psi=\nu} \Rcoeff{\psi}{q}^{-1}\cdot
    \Ebaseopp{z}{\psi}\otimes \Fbaseopp{z}{\psi}\;,
$$
with $\vec\psi$ as defined in Section~\ref{subsec:Emonomials} and $z\in\wordsetmax\,$. 
Exchanging Lusztig's conventions with ours, namely $\Delta$ with $\Deltaopp$ and $\tilde K_{-\nu}$ by $K^{-\nu}$, the formula in \cite[pg 38]{lu} reads 
$$(\Delta^\opp\otimes 1)\overline\Theta_{\mu}=\sum_{\nu'+\nu''=\mu}\overline\Theta_{\nu'}^{13}(1\otimes  K^{-\nu'}\otimes 1)\overline\Theta_{\nu''}^{23}\,.
$$
Applying the $(13)$ permutation to the tensor factors and substituting the $\Rprodmax$ components above we arrive at
\begin{equation}\label{eq:RpmTaniRel1}
(1\otimes \Delta){\Rprodmax}_{\mu}
=\sum_{\nu'+\nu''=\mu}{\Rprodmax}_{\nu'}^{13}(1\otimes K^{-\nu'}\otimes 1){\Rprodmax}_{\nu''}^{12} 
=\sum_{\nu'+\nu''=\mu}\taniauto_{12}({\Rprodmax}_{\nu'}^{13}){\Rprodmax}_{\nu''}^{12}\,,
\end{equation}
using formula \eqref{eq:tanipsi} for $\taniauto\,$. Summation, in the sense of the inverse limits, 
then yields the first identity in
\eqref{eq:tani2}. The second follows analogously. 
\end{proof}

A left inverse of $\Rprodmax$ may be constructed by applying the map $a\otimes b\otimes c\mapsto a\otimes S^{-1}(c)b$
to the first relation in (\ref{eq:RpmTaniRel1}). Using the antipode axiom for $\Delta^\opp$, 
the expression in (\ref{eq:Pw-genform}) for a maximal word $z\in\wordsetmax$\,,
the formulae for the antipode on basis elements in (\ref{eq:AntipBasis}), as well as notations from (\ref{eq:ExpsetDefMax}) and (\ref{eq:def:dyninv}) we obtain that 
\begin{equation}\label{eq:PmaxInv}
    \RprodInv{z}{}=
    \sum_{\psi\in\expsetupmax} (-1)^{\height(\vec\psi)}q^{-\frac 12(\symbrack{\vec\psi}{\vec\psi}-2\symbrack{\weightvecchar}{\vec\psi})}\Rcoeff{\psi}{q}^{-1}\cdot
    \Ebaseopp{z}{\psi}\otimes \Fbaseopp{z^\winvchar}{\psi}\qquad\in \TCompModGenopptwo{z}
\end{equation}
is a left inverse of $\Rprod{z}{}$
by summation over all graded components in the limit. A similar calculation, applying  
$a\otimes b\otimes c\mapsto a\otimes cS^{-1}(b)$ and the automorphism $x\otimes y\mapsto x\otimes yK^\nu$ for $\wgrad(x)=\nu$, shows that $\RprodInv{z}{}$ is also a right inverse of $\Rprod{z}{}\,$. 

Applying $a\otimes b\otimes c\mapsto aS(b)\otimes c$ and $a\otimes b\otimes c\mapsto S(a)b\otimes c$ to the second relation in (\ref{eq:RpmTaniRel1}) shows that  $\RprodInv{z^\winvchar}{}$ is also a two-sided inverse
of $\Rprod{z}{}$ so  that
\begin{equation}\label{eq:PqInvMaw}
\Rprod{z}{-1}=\RprodInv{z}{}=\RprodInv{z^\winvchar}{}=\Rprod{z^\winvchar}{-1}\;. 
\end{equation} 
From this identity  together with
(\ref{eq:AntipBasis}), (\ref{eq:tanipsi}), and word independence one readily derives  
\begin{equation}\label{eq:S2Pmax=TaniPmax}
    S\hat\otimes S(\Rprodmax)\,=\,\taniauto^{-1}(\Rprodmax)\,.
\end{equation} 

It is useful to express \eqref{eq:RpmTaniRel1} also as a relation of (co)multiplication coefficients.
Suppose $z\in\wordsetmax$ is a  reduced word of maximal length. 
The PBW theorems from Section~\ref{sec:PBW}, the fact that $\Uq^+$ is an algebra over $\Zqqn{\dpone}$\,, and $\Uq^\leqzero$ a coalgebra over $\Zqq$ then imply that there are elements 
$\Multcoeff{\chi}{\psi}{\phi}{q}\in\Zqqn{\dpone}$ and $\CoMultcoeff{\chi}{\phi}{\psi}{q}\in\Zqq$ with
$\psi,\phi,\chi\in\nnN^{\proots}$ such that 
\begin{equation}\label{eq:CoProdCoeffs}
    \Ebaseopp{z}{\psi}\Ebaseopp{z}{\phi}=\sum_{\chi}\Multcoeff{\chi}{\psi}{\phi}{q}\cdot\Ebaseopp{z}{\chi}
    \qquad\mbox{and}\qquad
    \Delta(\Fbaseopp{z}{\chi})=\sum_{\phi,\psi}\CoMultcoeff{\chi}{\phi}{\psi}{q}\cdot K^{-\vec\psi}\Fbaseopp{z}{\phi}\otimes\Fbaseopp{z}{\psi}\,.
\end{equation}
The summations are over only finitely many elements in $\nnN^{\proots}$ since they are subject to the grading constraint $\vec\chi=\vec\phi+\vec\psi$. It is clear that the only allowed terms for 
$\Delta(\Fbaseopp{z}{\chi})$ are those appearing in the sum \eqref{eq:CoProdCoeffs}.
Linear independence of the elements $\Ebaseopp{z}{\phi}\otimes K^{-\vec\psi}\Fbaseopp{z}{\phi}\otimes\Fbaseopp{z}{\psi}$ now implies that equation \eqref{eq:RpmTaniRel1} 
is, indeed, {\em equivalent} to the condition that 
\begin{equation}\label{eq:CoeffId}
\Rcoeff{\chi}{q}
    \Multcoeff{\chi}{\psi}{\phi}{q}
=
    \Rcoeff{\psi}{q}\Rcoeff{\phi}{q}
\CoMultcoeff{\chi}{\phi}{\psi}{q}\qquad \mbox{for all}\; \chi, \phi, \psi\in\nnN^{\proots}\,.
\end{equation}
An analogous equation of (co) multiplication coefficients, equivalent to 
the second Tanisaki relation in \eqref{eq:RpmTaniRel1}, can be derived in the same manner. 

\medskip

\subsection{Maximal Hopf Ideals from Tanisaki-Lusztig Pairings}\label{subsec:LuszTaniPair}

Yet another viewpoint of the same axioms in \eqref{eq:RpmTaniRel1}
is derived from the pairings used in \cite{lu,tan}. Fix some $z\in\wordsetmax\,$ and observe that, by \eqref{eq:RcoeffGen}, 
we may define a pairing of $\Uq^+$ with $\Uq^-$ over $\Zqqn{\dpone}$ as follows using PBW bases. 
\begin{equation}\label{eq:LTpairs}
    \langle\,\cdot\,,\,\cdot\,\rangle_z:\Uq^-\times \Uq^+\rightarrow \Zqqvn{\dpone}\qquad\mbox{with}\quad 
\langle \Fbaseopp{z}{\psi},\Ebaseopp{z}{\phi}\rangle_z\,=\,\delta_{\psi,\phi}\cdot\Rcoeff{\psi}{q}\,.
\end{equation}

Note that (\ref{eq:LTpairs}) is dual to (\ref{eq:Pw-genform}) in the sense that
$\left(\langle \Fbaseopp{z}{\psi},\,\cdot\,\rangle_z\,\hat\otimes\, \id\right)(\Rprod{z}{})=\Fbaseopp{z}{\psi}\,$. Word independence and
linearity of this relation imply that $\left(\langle f,\,\cdot\,\rangle_z\otimes \id\right)(\Rprodmax)=f$ for all $f\in\Uq^-\,$. 
This, in turn, shows that the pairing in (\ref{eq:LTpairs}) is independent of $z$ as well.  

Define a linear isomorphims $\taniauto_R:A^{\otimes 2}\rightarrow A^{\otimes 2}$ by $\taniauto_R(x\otimes y)=x\otimes K^{-\mu}y$, where $x$ is homogeneous with $\wgrad(x)=\mu$ and $A$ is any graded subalgebra that contains $\Uq^0\,$. The form of the coproduct
 in \eqref{eq:CoProdCoeffs} now implies that we have a well-defined map
 \begin{equation}\label{eq:hatDeltaDef}
     \hat\Delta=\taniauto_R\circ\Deltaopp:\Uq^-\rightarrow\Uq^-\otimes\Uq^-\;, 
 \end{equation}
 which is a homomorphism with respect to the braided product structure. The latter is given on homogeneous elements in $(\Uq^-)^{\otimes 2}$ by
 $(x\otimes y)(x'\otimes y')=q^{-\symbrack{\mu}{\nu}} xx'\otimes yy'\,$, where $\mu=\wgrad(y)$ and 
 $\nu=\wgrad(x')$. One easily checks that $\hat\Delta$ is coassociative. 

 In this language, 
 the first relation in \eqref{eq:tani2} as well as the coefficient relation in \eqref{eq:CoeffId} 
 are equivalent to the requirement that $\hat\Delta$ is adjoint to the multiplication in $\Uq^+$ with respect to the pairing in 
 \eqref{eq:LTpairs}. That is, 
 \begin{equation}\label{eq:UminUplusPair}
     \langle \hat\Delta(a),b\otimes c\rangle=\langle a,bc\rangle \quad\mbox{for all}\quad  a\in\Uq^-  
     \quad\mbox{and}\quad b,c\in\Uq^+\,.
 \end{equation}
 As before, analogous statements
 hold for the second Tanisaki relation. 

Consider next the specialization to a primitive $\kay$-th root of unity $\zeta\,$ such that $\ell>\maxd\,$. As in (\ref{eq:defzetaQG}), denote for $\maxd+1=\dpone\leq m\leq \ell$  the algebras 
\begin{equation}\label{eq:UzmDef}
    \Uzn{m}=\Uq\otimes\Zzvn{m}\qquad\mbox{and}\qquad \Uzn{m}^\utypechar=\Uq^\utypechar\otimes\Zzn{m} \quad \mbox{if} \quad
\utypechar\in\{\geqzero,\leqzero,+,-\}\,. 
\end{equation}

The form (\ref{eq:LTpairs}) clearly specializes to a form 
$\langle\,\cdot\,,\,\cdot\,\rangle_{\zeta,m}:\Uzn{m}^-\times\Uzn{m}^+\rightarrow\Zzvn{m}\,$ but will be degenerate. 
Denote by $N_{\zeta,m}^{\pm}\subset \Uzn{m}^{\pm}$ the respective null spaces.

\begin{lem}\label{lem:Null=Kmax}
With notation as in (\ref{eq:ZidealMaxNotHat}) and (\ref{eq:Kw=SpangeqL}) we have $N_{\zeta,m}^{\pm}=\ZmaxIdHat{\bullet}^{\pm}\otimes\Zzn{m}\,$. In particular,
$\ZmaxIdHat{\bullet}^{\pm}$ is uniquely defined for all Lie types.
\end{lem}
\begin{proof} Observe  that, with the exponent set from (\ref{eq:LminExpSet}), 
$$
\Rcoeff{\psi}{\zeta}=0 \qquad\mbox{if and only if}\quad \psi\in\Lminexpset \bullet\,.
$$
The claim is now immediate from the relation in  (\ref{eq:Kw=SpangeqL}). 
\end{proof}

Recall from Section~\ref{subsec:Z_Ideals} the remaining ambiguity for the $\LT{G}_2$ type, which
is resolved here via the word independence of the pairing. As before, we will suppress the 
ground ring from our notation $N_{\zeta,m}^{\pm}=\ZmaxIdHat{\bullet}^{\pm}$ whenever it is clear 
from the context.

\begin{prop}\label{prop:HopfIdealEasy}
    The ideals $\ZmaxIdHat{\bullet}^\leqzero\subset\Uzn{m}^\leqzero$ and $\ZmaxIdHat{\bullet}^\geqzero\subset\Uzn{m}^\geqzero$ are Hopf ideals, independent of choices of reduced  words of maximal length. 
\end{prop}
\begin{proof} The identity (\ref{eq:UminUplusPair}) implies that the left null space has to
be a co-ideal with respect to $\hat\Delta$. Together with Lemma~\ref{lem:Null=Kmax} we thus
obtain  
$$
    \hat\Delta(\ZmaxIdHat{\bullet}^-)\subseteq \ZmaxIdHat{\bullet}^-\otimes\Uzn{m}^-+\Uzn{m}^- \otimes \ZmaxIdHat{\bullet}^-\;. 
$$
With $\Uzn{m}^\leqzero=\Uzn{m}^0\cdot \Uzn{m}^-$ and $\ZmaxIdHat{\bullet}^\leqzero=\Uzn{m}^0\cdot \ZmaxIdHat{\bullet}^-$ it is easy to check that $\taniauto_R$ maps  $\ZmaxIdHat{\bullet}^\leqzero\otimes\Uzn{m}^\leqzero$ as well as 
$\Uzn{m}^\leqzero\otimes\ZmaxIdHat{\bullet}^\leqzero$ to itself. 
The co-ideal equation for $\hat\Delta$ and (\ref{eq:hatDeltaDef}) now imply the respective co-ideal 
condition for the regular coproduct with respect to $\ZmaxIdHat{\bullet}^\leqzero\,$.  
Word independence of $\Zaugidealhat{z}^\leqzero\,$ from Lemma~\ref{lem:Null=Kmax} and 
relation (\ref{eq:KhatInvolIds}) now entails that  $\ZmaxIdHat{\bullet}^\leqzero$ is invariant
under the antipode and, hence, a Hopf ideal. 
Finally, recall that $\Cartaninv(\Zaugidealhat{z}^\leqzero)\subseteq \Zaugidealhat{z}^\geqzero$\,,
$\Cartaninv(\Uzn{m}^\leqzero)\subseteq\Uzn{m}^\geqzero$, $\Cartaninv^{\otimes 2}\circ\Delta=\Delta^\opp\circ\Cartaninv$\,, and $\Cartaninv\circ S=S\circ \Cartaninv$\,.
From these relations 
one readily derives that $\ZmaxIdHat{\bullet}^\geqzero$ is also a Hopf ideal. 
\end{proof}

We may consider now the full maximal ideal $\ZmaxIdHat{\siAA}=(\ZmaxIdHat{\bullet}^\geqzero+\ZmaxIdHat{\bullet}^\leqzero)\cdot \Uzn{m}$
for the full algebra $\Uzn{m}$ over $\Zzvn{m}\,$, which has all desired properties. 

\begin{cor}\label{cor:MaxIdealEasy} \ Suppose $\dpone\leq m\leq \ell$.\vspace*{-1mm}

\begin{enumerate}[label=\roman*), leftmargin=12mm,] 
        \item 
        For any Lie type and root of unity with $\kay$ as in (\ref{eq:kaycond}), 
    $\ZmaxIdHat{\siAA}$ is a Hopf ideal in $\Uzn{m}$\,.
        \vspace*{1.8mm}
        \item 
        $\ZmaxIdHat{\siAA}$ is equal to the ideal generated by
    the set $\bigl\{\Epw_w,\,\Fpw_{w'}:\;\emptyword\neq w\leqRB z,\, \emptyword\neq w'\leqRB z'\bigr\}$
    for any choice of $z,z'\in\wordsetmax\,$.
         \vspace*{1.8mm}
         \item
         $\ZmaxIdHat{\siAA}$ is stable under all automorphisms
    discussed in Section~\ref{sec:gradingauts}. 
    \end{enumerate}
\end{cor}

\begin{proof}
   The first part is immediate from
Proposition~\ref{prop:HopfIdealEasy}. Invariance under the 
automorphisms $\Cartaninv$, $\Kinvaut$, and $S$ follows from (\ref{eq:KhatInvolIds})
and word independence. Gradings and identities, such as (\ref{eq:ConjWordGen}) and $\Cartanaut=\Qinvaut\circ\Cartaninv$, imply invariance for all others in Section~\ref{subsec:Kscale}. 

For the Artin generators $\Tinv_i$ we may repeat the proof of Theorem~\ref{thm:ZTinvInvar} nearly verbatim. Specifically, for the choice of $w$ therein, it is shown that for all $v\leqRB w$ the element $\Tinv_i(\Epw_v)$ is again some power generator and, hence, in $\ZmaxIdHat{\siAA}$
by word independence.
\end{proof}

Thus, if we define the {\em \mytypeqg quantum groups} by the quotient 
\begin{equation}\label{eq:DefReduced}
  \Uzn{m}^{\redsup}=\Uzn{m}/\ZmaxIdHat{\siAA}\,,
\end{equation}
then Corollary~\ref{cor:MaxIdealEasy} implies that $\Uzn{m}^{\redsup}$ is indeed a Hopf algebra, which is independent of any choices made in its definition. The construction of a full 
poset of ideals under $\ZmaxIdHat{\siAA}$ and respective quasi-$R$-matrices will be addressed in
the next sections with different techniques.

\subsection{Truncated Quasi-\texorpdfstring{$R$}{R}-Matrices and Coproduct Formulas} 
\label{subsec:TruncQmats}
In this section we develop a calculus of quasi-$R$-matrices at roots of unity analogous to the 
one in Section~\ref{subsec:QuasiR_genq}. To avoid singularities, the summations defining these 
$R$-matrices need to be truncated so that the corresponding identities hold only up to certain ideals. 
 
It is, therefore, useful to 
extend  the ideals in $\Uz^{\pm}$ defined (\ref{eq:Kw=SpangeqL}) to the full algebra $\Uzn{m}$ over $\Zzvn{m}$ for a 
given  primitive $\kay$-th root of unity $\zeta$ and $\dpone\leq m\leq\ell\,$.  For any pair of reduced words $u,w\in\wordset\,$, we write  
\begin{equation}\label{eq:KhatTwoDefs}
    \ZaugidealhatIn{u,w}{m}
=(\Zaugidealhat{u}^{+}+\Zaugidealhat{w}^-)\cdot\Uzn{m}
=\Zaugideal{u,w}\cdot\Uzn{m}
=
\ZaugidealhatIn{u,\emptyword}{m}+\ZaugidealhatIn{\emptyword,w}{m}\,,
\end{equation}
which may, alternatively, be viewed as the two-sided ideal in $\Uzn{m}$ generated by the set $\{\Epw_a,\Fpw_b\,:\,\emptyword\neq a,b\leqRB u\}\,$. Correspondingly, for reduced words $a,b\in\wordset\,$ for which $a\cdot b$ is also reduced, we infer that, for any Lie type,  
\begin{equation}\label{eq:KhatTwoIds}
    \ZaugidealhatIn{a\cdot b,\emptyword}{m}=
         \ZaugidealhatIn{a,\emptyword}{m}+\Tinv_a\bigl(\ZaugidealhatIn{b,\emptyword}{m}\bigr)
    \qquad\mbox{and}\qquad
    \Cartaninv\bigl(\ZaugidealhatIn{a,\emptyword}{m}\bigr)=\ZaugidealhatIn{\emptyword,a}{m}\,
\end{equation} 
hold as well as the respective identities with flipped arguments and summations over ideals. A useful symmetry for 
various tensor identities is   
\begin{equation}\label{eq:DefCatinvHash}
    \CartaninvTwo=(1\,2)\circ\Cartaninv^{\otimes 2}:\,\Uzn{\ell}^{\otimes 2}\rightarrow\Uzn{\ell}^{\otimes 2}\,:\quad
    x\otimes y\mapsto  \CartaninvTwo(x\otimes y)=\Cartaninv(y)\otimes \Cartaninv(x)\,,
\end{equation}
which is  an anti-linear anti-involution by the properties in (\ref{eq:defCinv}). 

To ensure minimal truncations, we consider $R$-matrices defined over the maximal ground rings $\Zzn{\ell}$ or
$\Zzvn{\ell}$\,, observing that $[s]!_\alpha\neq 0$ whenever $s< \ell_\alpha\,$. We begin with the definition
of the {\em simple truncated} quasi-$R$-matrix associated to a given $\alpha_i\in\sroots$\,,
\begin{align}\label{eq:DefElemTrRmat}
    \Rti\,=\,\sum_{s=0}^{\elli-1}\dfrac{\zeta_i^{s(s-1)/2}}{[s]!_i}(\zeta_i-\zeta_i^{-1})^s \cdot E_i^s\otimes F_i^s
    \qquad\in \;\Uzn{\ell}^+\otimes \Uzn{\ell}^-\,,
\end{align}
which is well-defined over $\Zzn{\ell}\,$. Analogous to the generic case we, further, denote the conjugate $R$-matrix as
\begin{align}\label{eq:DefElemTrRmatinv}
    \Rtii\,=\,\CartaninvTwo(\Rti)\,=\,\sum_{s=0}^{\elli-1}\dfrac{\zeta_i^{-s(s-1)/2}}{[s]!_i}(-1)^s(\zeta_i-\zeta_i^{-1})^s \cdot E_i^s\otimes F_i^s
    \qquad\in \;\Uzn{\ell}^+\otimes \Uzn{\ell}^-\,,
\end{align}
where $\CartaninvTwo$ is as in (\ref{eq:DefCatinvHash}). To express these as inverses, we introduce  for any 
$w\in\wordset$ the ideals
\begin{equation}\label{eq:JwDef}
    \pTmIdeal{w}=(\Zaugidealhat{w}^{+}\otimes \Zaugidealhat{w}^-)\otimes \Zzn{\ell}= (\ZaugidealhatIn{w,\emptyword}{\ell}\otimes \ZaugidealhatIn{\emptyword,w}{\ell})\cap (\Uzn{\ell}^+\otimes \Uzn{\ell}^-)\,,
\end{equation}
generated by all $X_a\otimes  Y_b$ with $\emptyword\neq a,b\leqRB w\,$. 
In calculations throughout this
section, we also employ the standard notation $x\equiv y\mod V$ to mean
$x-y\in V\,$ for a given submodule $V$.

\begin{lem}[\textit{cf.} {\cite[Lemma A.1]{oh}}]\label{lem:Rtinverse} 
For all $\alpha_i\in\sroots$ we have 
    \begin{align*}
    \Rtii\cdot\Rti\,\equiv\,\Rti\cdot \Rtii \,\equiv\,1\otimes 1\quad \mod \,\pTmIdeal{w_i} \,. 
\end{align*}
\end{lem}
\begin{proof}
    We start with a direct computation and resummation. 
\begin{align*}
    \Rti\cdot \Rtii
    &=
    \sum_{r,n=0}^{\elli-1}\dfrac{\zeta_i^{r(r-1)/2-n(n-1)/2}}{[r]!_i[n]!_i}(-1)^n(\zeta_i-\zeta_i^{-1})^{r+n} E_i^{r+n}\otimes F_i^{r+n}
    \\&=
    \sum_{k=0}^{2\elli-2}\rho_k\cdot\zeta_i^{-k(k-1)/2}(-1)^k(\zeta_i-\zeta_i^{-1})^{k} E_i^{k}\otimes F_i^{k} \;,
    \\
    \mbox{where}\quad & \qquad 
    \rho_k=\sum_{r=\max(0,k-(\elli-1))}^{\min(k,\elli-1)} \dfrac{1}{[r]!_i[k-r]!_i}(-1)^r\zeta_i^{r(k-1)}\;. 
 \end{align*}
 For $0<k<\ell_i$\,, we compute
\begin{align*}
\rho_k&=\sum_{r=0}^{k} \dfrac{1}{[r]!_i[k-r]!_i}(-1)^r\zeta_i^{r(k-1)}
\quad =\dfrac{1}{[k]!}\sum_{r=0}^{k} \qbin{k}{r}{i}(-1)^r\zeta_i^{r(k-1)}\\
&=\dfrac{1}{[k]!}\prod_{t=1}^k(-\zeta_i^{2(t-1)}+1)\quad =0\,,
\end{align*}
where we used (\ref{eq:q-combin-form}) with substitutions $a\mapsto k$, $k\mapsto r$, $q\mapsto\zeta_i^{-1}$, and $z\mapsto -\zeta_i^{-2}\,$. Observing that $\rho_0=1$, $\Epw_i=E_i^{\ell_i}$, and $\Fpw_i=F_i^{\ell_i}$ we now find
\begin{align}\label{eq:PPbar=1ModXY}
    \Rti\cdot \Rtii=1\otimes 1\,+\,(\Epw_i\otimes \Fpw_i)\sum_{k=\ell_i}^{2\elli-2}\rho_k\cdot\zeta_i^{-k(k-1)/2}(-1)^k(\zeta_i-\zeta_i^{-1})^{k} E_i^{k-\ell_i}\otimes F_i^{k-\ell_i}\, ,
\end{align}
for which the latter term is clearly in $\pTmIdeal{w_i}\,$. The verification for 
 $\Rtii\cdot \Rti$ is nearly identical.
\end{proof} 

Analogous to the generic case and using definitions from (\ref{eq:defEgenLu}), we now define the \emph{elementary} truncated quasi-$R$-matrices associated to a reduced word $w\in\wordset$ as 
\begin{equation}\label{eq:ElemTrRmatDef}
    \begin{aligned}
    &&    \Rtsub{w}&=
    \Tinv^{\otimes 2}_{w^\flat}\left(\Rtsub{\tau(w)}\right)=\sum_{s=0}^{\ell_{w}-1}
    \dfrac{\zeta^{s(s-1)/2}_{w}}{[s]!_{w}}(\zeta_{w}-\zeta^{-1}_{w})^s E_{w}^s\otimes F_{w}^s\,\\
    \rule{0mm}{8.4mm}
    \mbox{and} \qquad &&
    \Rtinvsub{w}&=
    \Tinv^{\otimes 2}_{w^\flat}\left(\Rtinvsub{\tau(w)}\right)=\sum_{s=0}^{\ell_{w}-1}
    \dfrac{\zeta^{-s(s-1)/2}_{w}}{[s]!_{w}}(-1)^s(\zeta_{w}-\zeta^{-1}_{w})^s E_{w}^s\otimes F_{w}^s\,, 
    \end{aligned}
\end{equation} 
with notations as in (\ref{def:flat+tau_word}) and $\zeta_w=\zeta_{\tau(w)}=\zeta_{\wordroot(w)}\,$.  
The following assertion is readily derived by applying $\Tinv^{\otimes 2}_{w^\flat}$ to  (\ref{eq:PPbar=1ModXY}) with $i=\tau(w)\,$, resulting in the same equation with $i$ replaced by $w\,$. 

\begin{cor}\label{cor:Rtinv}
    For all $w\in\wordset$ we have 
    \begin{align*}
    \Rtinvsub{w}\cdot\Rtsub{w}\,\equiv\,\Rtsub{w}\cdot\Rtinvsub{w} \,\equiv\,1\otimes 1\quad \mod \,\pTmIdeal{w}\,. 
\end{align*}
\end{cor}

We define, similarly, the {\em partial } truncated quasi-$R$-matrices associated to some $w\in\wordset$ as
\begin{equation}\label{eq:TrQuasRmatWord}
\begin{aligned}
    &&
    \Rtprod{w}=\Rtsub{w^L}\ldots\Rtsub{w^1}
    &=\sum_{\psi\in \Lmaxexpset {s}}\Rcoeff{\psi}{\zeta}^{-1}\cdot
    \Ebaseopp{w}{\psi}\otimes \Fbaseopp{w}{\psi}\,\\
    \rule{0mm}{7mm}
    \mbox{and} \qquad &&   
    \Rtprodi{w}=\Rtinvsub{w^1}\ldots\Rtinvsub{w^L}
    &=\sum_{\psi\in \Lmaxexpset {s}}\Rcoeff{\psi}{\zeta^{-1}}^{-1}\cdot
    \Ebase{w}{\psi}\otimes \Fbase{w}{\psi}\;,\\
\end{aligned}
\end{equation}
with notation as in (\ref{eq:Pw-genform}), $L=\len{w}\,$, and $\Lmaxexpset {s}=\Lmaxexpset {\bullet}\cap \expsetup{s}$ as in (\ref{eq:Lmaxexpset}). 
The identities below, with $a,b\in\wordset$ such that $a\cdot b$ is reduced,  are straightforward from Corollary~\ref{cor:Rtinv}, the mentioned properties of $\CartaninvTwo$, and (\ref{eq:PBWgenRecur}).
\begin{equation}\label{eq:PwordProps}
\begin{aligned}  
    \Rtprodi{w}&=\CartaninvTwo(\Rtprod{w})
    &
    \qquad
    \Rtprodi{w}\cdot\Rtprod{w}&\equiv 1^{\otimes 2}\equiv\Rtprodi{w}\cdot\Rtprod{w}\mod \pTmIdeal{w}
    \\
  \rule{0mm}{6mm}  \Rtprod{a\cdot b}&=\Tinv_a^{\otimes 2}(\Rtprod{b})\cdot\Rtprod{a}
    &
    \Rtprodi{a\cdot b}&=\Rtprodi{a} \cdot\Tinv_a^{\otimes 2}(\Rtprodi{b})\,
   \end{aligned} 
\end{equation}

The relevant ideals for the intertwining relations of the truncated $R$-matrices are
given for each $w\in\wordset$ as 
\begin{equation}\label{eq:DefNideals}
    \FullpTmIdeal w = \ZaugidealhatIn{w,\emptyword}{\ell}\otimes\Uzn{\ell}+\Uzn{\ell}\otimes\ZaugidealhatIn{\emptyword,w}{\ell}\,, 
\end{equation}
which is the ideal in $\Uzn{\ell}^{\otimes 2}$ generated by all $X_a\otimes 1$ and  $1\otimes Y_b$ with $\emptyword\neq a,b\leqRB w\,$. Assuming $w\cdot u$ is reduced for $w,u\in\wordset$, one easily verifies the identities
\begin{equation}\label{eq:NIdealProps}
    \pTmIdeal{w}\subset \FullpTmIdeal {w}\,,\quad
\FullpTmIdeal {w\cdot u}=\Tinv_w^{\otimes 2}\bigl(\FullpTmIdeal u\bigr) +\FullpTmIdeal w\,,
\quad\mbox{and}\quad
\Cartaninv^{\#}(\FullpTmIdeal w)= \FullpTmIdeal w\,. 
\end{equation}
The truncated analog of Proposition~\ref{prop:AdP=DelT} for a simple reflection $w_i$ and simple generators is now stated as follows.  

\begin{lem}\label{lem:coprodbase}
     Let $\roots$ be of any Lie type and $1\leq i,j\leq n$. Then \begin{equation}\label{eq:coprodbase}
          \Delta\circ\Tinv_i(E_j)\,\equiv\, \Rtii\cdot\Tinv^{\otimes 2}_i\circ\Delta(E_j)\cdot\Rti \qquad\mod \FullpTmIdeal {w_i}\,.
     \end{equation}
 \end{lem}

The proof of this lemma is provided in Appendix~\ref{sec:proofcoprodbase} by explicit computations of both sides of \eqref{eq:coprodbase}, depending on the Cartan matrix entries 
$A_{ij}$ and $A_{ji}$\,. The properties of automorphisms and ideals developed thus far allow us to extend this identity 
to all words and generators. 

\begin{prop} \label{prop:CoprodIntWord}
Let $\roots$ be of any Lie type, $w\in\wordset$, and $x\in\Uzn{\ell}$\,. Then
$$
    \Delta\circ\Tinv_w(x)\,\equiv\,\Rtprodi{w}\cdot\Tinv^{\otimes 2}_w\circ\Delta(x)\cdot\Rtprod{w}\qquad \mod \FullpTmIdeal {w}\,.
    $$
\end{prop}

\begin{proof}
    We start with the case $w=w_i$ for which $\Tinv_w=\Tinv_i$ and $\Rtprod{w_i}=\Rti\,$. Denote the quotient algebra  
    $\mathbf{U}(w)=\Uzn{\ell}^{\otimes 2}/\FullpTmIdeal{w}$ and note that by Corollary~\ref{cor:Rtinv} and 
    (\ref{eq:NIdealProps}) the images of $\Rtii$ and $\Rti$ in $\mathbf{U}(w_i)$ are inverse to each other. The maps from $\Uzn{\ell}$ to $\mathbf{U}(w_i)$ given by $x\mapsto [\Delta\circ\Tinv_i(x)]$ and 
    $x\mapsto [\Rtii\cdot\Tinv^{\otimes 2}_i\circ\Delta(x)\cdot\Rti ]$ are thus algebra homomorphisms. 
    Since, by Proposition~\ref{lem:coprodbase}, they coincide on generators, they have to be  equal  
    on all of $\Uzn{\ell}^+\,$. 

    Note next that $\Delta\circ\Tinv_i$ and $\Tinv_i^{\otimes 2}\circ\Delta$ coincide on any $K^\nu$ and, further, that $K^\nu\otimes K^\nu$ commutes with $\Rti$. This implies that the 
    two homomorphisms also coincide on $\Uzn{\ell}^0$ and, hence, that the identity above
    holds for $x\in\Uzn{\ell}^\geqzero$ and $w=w_i\,$.  
    
    Recall that, with $\CartaninvTwo$ as in (\ref{eq:DefCatinvHash}), $\CartaninvTwo(\Delta(x))=\Delta(\Cartaninv(x))$ and $\Cartaninv^\#$ commutes with $\Tinv_i^{\otimes 2}$. It readily follows from the fact that $\Cartaninv^\#$ is an anti-linear anti-involution,
    \eqref{eq:defCinv}, and \eqref{eq:DefElemTrRmatinv}
    that $\Cartaninv^\#(\Rti)=\Rtii\,$. Applying $\Cartaninv^\#$ to the identity in Proposition~\ref{lem:coprodbase}
    for some $x\in\Uzn{\ell}^\geqzero\,$, these observations now imply the same relation for $\Cartaninv(x)\in\Uzn{\ell}^\geqzero\,$, using also that $\FullpTmIdeal {w}$ is $\CartaninvTwo$-invariant by (\ref{eq:NIdealProps}). Hence, the relation holds for any $x\in\Uzn{\ell}\,$, completing the
    proof for $w=w_i\,$. 

    The extension to general words is similar to the one in  Corollary~\ref{cor:Rtinv}. We proceed by induction on $L=\len{w}$, starting from the base case  already proven above. For a reduced word $u\in\wordset$ let $w=u^\flat$
    and $i=\tau(u)$ so that $\len{u}>\len{w}$. Assuming the assertion for $w$, we first observe   
\begin{align*}
    \Delta\circ\Tinv_{u} \left(x\right)\,
    =\,
    \Delta\circ\Tinv_{w} \left(\Tinv_{i}(x)\right)\,\equiv\,
    \Rtprodi{w}
    \cdot \Tinv_w^{\otimes 2}\circ\Delta \left(
    \Tinv_{i}(x)\right)\Rtprod{w} \qquad\mod \FullpTmIdeal {w}\,.
\end{align*}
Applying the relation for $w_i$ to the $\Delta(\Tinv_{i}(x))$ term, we find
\begin{align*}
    \Delta\circ\Tinv_{u} \left(x\right)&\,\in\,
    \Rtprodi{w}
    \cdot \Tinv_w^{\otimes 2}\left(
   \Rtii\cdot\Tinv_{i}^{\otimes 2}(\Delta(x))\cdot \Rti
    +\FullpTmIdeal {w_i}\right)\Rtprod{w}
    \,+\,
    \FullpTmIdeal {w}
    \\
    &
    \rule{0mm}{6mm}
    \subseteq
    \Rtprodi{w}\cdot\Tinv_w^{\otimes 2}(\Rtinvsub{w_i})\cdot\Tinv^{\otimes2}_{w\cdot w_i}(\Delta(x))\cdot \Tinv_w^{\otimes 2}(\Rtsub{w_i})\cdot\Rtprod{w}
    \,+\,
    \Rtprodi{w}\cdot\Tinv_w^{\otimes 2}\left(\FullpTmIdeal {w_i}\right) \cdot\Rtprod{w} 
    \,+\,
    \FullpTmIdeal {w}
    \\
   &
    \rule{0mm}{6mm}
   \subseteq
    \Rtprodi{w\cdot w_i}\cdot\Tinv^{\otimes2}_{w\cdot w_i}(\Delta(x))\cdot \Rtprod{w\cdot w_i}
    \,+\,
    \Tinv_w^{\otimes 2}\left(\FullpTmIdeal {w_i}\right) \,+\,
    \FullpTmIdeal {w}
    \\
   &
    \rule{0mm}{6mm}
   =
    \Rtprodi{u}\cdot\Tinv^{\otimes2}_{u}(\Delta(x))\cdot \Rtprod{u}
    \,+\,
    \FullpTmIdeal {u}
    \,.
\end{align*}
The step from the second to the third line uses (\ref{eq:PwordProps}) for the first term and the fact 
that $\Tinv_w^{\otimes 2}\left(\FullpTmIdeal {w_i}\right)$ is itself a two-sided ideal for the second term.
The last step uses the iteration of ideals from (\ref{eq:NIdealProps}) and $u=w\cdot w_i\,$. Taking the
inclusion modulo $\FullpTmIdeal {u}$ yields the claim.
\end{proof}

Recall from the proof of Lemma~\ref{lm:PzDel=DelbPz} that 
$\barDelta(x)=\Tinv_{\longweyl}^{\otimes 2}\bigl(\Delta\bigl(\Tinv_{\longweyl}^{-1}(x)\bigr)\bigr)$ for the 
twisted coproduct defined in (\ref{eq:AltCoprodDef}). The relation clearly descends to $\Uzn{\ell}$\,. 
Applying, Proposition~\ref{prop:CoprodIntWord} to a reduced word of maximal length now yields the following analog
of Lemma~\ref{lm:PzDel=DelbPz}. 

\begin{cor}\label{cor:RtrmaxDeltaIntw}
Suppose $z\in\wordsetmax$  and $x\in\Uzn{\ell}\,$. Then 
    $$
    \barDelta(x)\cdot\Rtprod{z}\,\equiv\, \Rtprod{z}\cdot\Delta(x) \qquad\mod \FullpTmIdeal {z}\,.
    $$
\end{cor}

The iteration of the coproduct with truncated quasi-$R$-matrices provided in the above proposition allows us
to establish further Hopf ideals. It is immediate from the form in (\ref{eq:DefNideals}) that
\begin{equation}\label{eq:NidealInDiag}
    \FullpTmIdeal u\,\subseteq\, \FullpTmIdealSym u =\ZaugidealhatIn{u,u}{\ell}\otimes \Uzn{\ell}\,+\,\Uzn{\ell}\otimes \ZaugidealhatIn{u,u}{\ell}\,,
\end{equation}
which suggests that we first consider these diagonal ideals. 

\begin{prop}\label{prop:KhatIdealDiagWord}
  For every $u\in\wordset$ the ideal $\ZaugidealhatIn{u,u}{m}$ is a Hopf ideal in $\Uzn{m}\,$. 
\end{prop}

\begin{proof} We start by showing that the $\ZaugidealhatIn{u,u}{\ell}$ are bi-ideals.
Given their definition, this means verifying that $\Delta(\Epw_v), \Delta(\Fpw_v)\in \FullpTmIdealSym u$
for all $\emptyword\neq v\leqRB w\,$. Since $\FullpTmIdealSym v\subseteq \FullpTmIdealSym u$ and
using $\Cartaninv$-symmetry as before, it suffices to check 
$\Delta(\Epw_v)\in \FullpTmIdealSym v$ for any reduced word $v$. Setting $w=v^\flat$ and $i=\tau(v)$
and using (\ref{eq:simplecoprod}) we first observe that  
\begin{align*}
\Tinv^{\otimes2}_w\circ \Delta(X_i)&=\Tinv^{\otimes2}_w(\Epw_i\otimes\Kpw_i+1\otimes\Epw_i )
  =\Epw_{w\cdot w_i}\otimes\Kpw_{w\cdot w_i}+1\otimes\Epw_{w\cdot w_i}\\
&=\Epw_{v}\otimes\Kpw_{v}+1\otimes\Epw_{v}\;\in\; \FullpTmIdealSym v\;.
\end{align*}
Using this, Proposition~\ref{prop:CoprodIntWord}, the coproduct from (\ref{eq:simplecoprod}), the fact that 
$\FullpTmIdealSym v$ is a two-sided ideal in $\Uzn{\ell}^{\otimes 2}\,$, and $\FullpTmIdeal w\subseteq \FullpTmIdealSym v$
we compute 
    \begin{align*}
        \Delta(X_{v})&\;=\;\Delta(X_{w\cdot w_i})\;=\;\Delta\circ\Tinv_w(X_i) 
        \;\in \;
        \Rtprodi{w}\cdot \Tinv^{\otimes2}_w\circ \Delta(X_i)\cdot \Rtprod{w} \,+\,\FullpTmIdeal w
        \\
        \rule{0mm}{6mm}
        &\subseteq\;
        \Rtprodi{w}\left(\FullpTmIdealSym v\right)\Rtprod{w} \,+\,\FullpTmIdeal w
       \;\subseteq\;
        \FullpTmIdealSym v\,,
    \end{align*}
so that $\ZaugidealhatIn{u,u}{\ell}$ is indeed a bi-ideal.

We now invoke a result of Nichols, which states that any bi-ideal in a pointed Hopf algebra is necessarily a Hopf ideal. See \cite[Thm 1(v)]{Ni78}. By Proposition~\ref{prop:pointed}, $\Uqre{\mathbb k}\cong \Uzn{\mathbb Q}=\Uzn{\ell}\otimes\mathbb Q(\zeta)$ is pointed for $\bbk=\Qz$ via the obvious map $\vartheta:\Zqqvn{\dpone}\rightarrow\Qz\,$. 
Since $\ZaugidealhatIn{u,u}{\mathbb Q}=\ZaugidealhatIn{u,u}{\ell}\otimes \Qz$ is a bi-ideal, Nicols's result 
implies that 
$\ZaugidealhatIn{u,u}{\mathbb Q}$ is indeed a Hopf ideal, meaning it is preserved under the antipode $S$.

Finally, by (\ref{eq:Kw=SpangeqL}) and (\ref{eq:KhatTwoDefs}), the ideals $\ZaugidealhatIn{u,v}{m}$ are the
free $\Zzvn{m}$-submodules spanned by a subset of elements of a PBW basis of $\Uzn{m}\,$, implying
$\ZaugidealhatIn{u,v}{m}=\ZaugidealhatIn{u,v}{\mathbb Q}\cap \Uzn{m}\,$. Since the intersection
of a Hopf subalgebra and Hopf ideal is again a Hopf ideal, we arrive at the claim for $\ZaugidealhatIn{u,u}{m}\,$.
\end{proof}

We next seek to extend the observation above to the related families of ideals defined in previous sections. 
The two-sided ideals $\Zaugidealhat{w}^+$ defined in (\ref{eq:Kw=SpangeqL}) are  not 
Hopf ideals since $\Uzn{\dpone}^+$ is not a bialgebra, nor are they even Hopf subalgebras. We, thus, consider the extension
to the Hopf algebra $\Uzn{\dpone}^\geqzero$ by including the Cartan generators, and defining
\begin{equation}\label{eq:KgeqzeroDef}
    \Zaugidealhat{w}^{\geqzero}=\Zaugidealhat{w}^+\cdot\Uzn{\dpone}^\geqzero
\qquad\mbox{with \ $\Zzn{\dpone}$-basis}\qquad 
\bigl\{b\cdot K^\nu\,|\,\nu\in\mathbb Z^{\sroots},\,b\in\basis{z}{\iota^*(\Lminexpset s),+}\bigr\}\,.
\end{equation}
Here, $z\in\wordsetmax$ with $w\leqRB z$, $s=\Weylpres(w)$, and $\iota:\descroots{s}\hookrightarrow\proots\,$. The definition of $\Zaugidealhat{w}^{\leqzero}$ is  analogous and the 
$\ZaugidealhatIn{u,v}{m}$ are obtained from the $\Zaugidealhat{w}^{\geqzero,\leqzero}$ in the same way as in
(\ref{eq:KhatTwoDefs}).

The statement in Item {\em \ref{item:thm:MainIdeals:Kpm}} in Theorem~\ref{thm:MainIdeals} already follows,
independently, from
the results in Sections~\ref{sec:Ideals_An=GLn} and \ref{sec:Ideals_B2=SO5} for the respective 
Lie types $\LT{A}_n$ and $\LT{B}_2$\,, where the ideals are explicitly identified with vanishing ideals
of Bruhat subgroups. 

Word independence, asserted in Item {\em \ref{item:thm:MainIdeals:Wind}},
is already implied by Theorem~\ref{thm:Zwordindep} for Lie types different from $\LT{G}_2$\,. 
The use of Nicols's result for Hopf ideals in Proposition~\ref{prop:KhatIdealDiagWord} allows
us to extend word independence also to the $\LT{G}_2$ case for the ideals in the full algebras.

\begin{thm}\label{thm:MainIdeals}
   Suppose $\roots$ is of {\em any} Lie type and $\dpone\leq m\leq \ell$. Assume further 
   $u,v,u',v'\in\wordset$ are reduced words with $s=\Weylpres(u)=\Weylpres(u')$ and $t=\Weylpres(v)=\Weylpres(v')$.\vspace{-1mm}
    \begin{enumerate}[label=\roman*)]
        \item \label{item:thm:MainIdeals:Kpm}
           The ideals $\Zaugidealhat{u}^{\geqzero}$ and $\Zaugidealhat{u}^{\leqzero}$ are Hopf ideals 
            in $\Uzn{\dpone}^\geqzero$ and  $\Uzn{\dpone}^\leqzero$\,, respectively. \vspace{2mm}
         \item \label{item:thm:MainIdeals:Kuv}
           The ideal $\ZaugidealhatIn{u,v}{m}$ is a Hopf ideal in $\Uzn{m}$\,.\vspace{2mm}
        \item \label{item:thm:MainIdeals:Wind} 
        The ideals depend only on the elements represented in $\Weyl$. \newline That is, $\Zaugidealhat{u}^{\geqzero}=\Zaugidealhat{u'}^{\geqzero}$\,, $\Zaugidealhat{u}^{\leqzero}=\Zaugidealhat{u'}^{\leqzero}$\,,  and $\ZaugidealhatIn{u,v}{m}=\ZaugidealhatIn{u',v'}{m}$\,.  
    \end{enumerate}
\end{thm}

\begin{proof} The same reasoning as in the proof of Proposition~\ref{prop:KhatIdealDiagWord}
applies to show that $\Zaugidealhat{u}^{^\geqzero}=\ZaugidealhatIn{u,u}{m}\cap \Uzn{\dpone}^\geqzero\,$, by restricting
both the basis elements defined with respect to $u$ to those in 
(\ref{eq:KgeqzeroDef})
and reducing the ground ring to $\Zzn{\dpone}\,$. 
As before, since $\Uzn{\dpone}^\geqzero$ is a Hopf subalgebra it follows from 
Proposition~\ref{prop:KhatIdealDiagWord} that $\Zaugidealhat{u}^{\geqzero}$ is a Hopf ideal. 
The argument for $\Zaugidealhat{u}^{\leqzero}$ is identical, and the assertion in Item   
{\em \ref{item:thm:MainIdeals:Kuv}} is now immediate from the definition in (\ref{eq:KhatTwoDefs}).

The statement in Item  {\em \ref{item:thm:MainIdeals:Kpm}} implies, in particular, that for any $w\in\wordset$,
 $\Zaugidealhat{w}^{\geqzero}$ is invariant under the antipode. The formulae in (\ref{eq:SXiMhoRel})
 and the fact that the $\Kscale{u}{\mathsfit h}$ preserves any subspace of $\Uzn{\ell}^\geqzero$ with a 
 $\wgrad$-graded decomposition implies that $\Kinvaut$ maps $\Zaugidealhat{w}^{\geqzero}$ to itself. The
 decomposition into a positive and Cartan part now implies $\Kinvaut(\Zaugidealhat{w}^{+})=\Zaugidealhat{w}^{+}\,$. 
  
Relation (\ref{eq:Kw=SpangeqL}) shows that word independence of $\Zaugidealhat{w}^{+}$ is 
implied by word independence of $\bspan{w}{\Lminexpset s,+}\!$. To prove the latter we 
first oberve that $\NexpExpSet{s}=\Lmaxexpset s$ is of the split form (\ref{eq:VexpSetEx}) with
$V_d=\{0,1,\ldots,\ell_d-1\}\,$. The exponent set $\NexpExpCompSet{s}=\Lminexpset s$ is
defined as the complement and thus of the required in the case {\em \ref{item:BaseOrdSplit:VsComp}}
of Corollary~\ref{cor:BaseOrdSplit}.

As in the remark following (\ref{eq:Kw=SpangeqL}), we have 
$\Zaugidealhat{\longtwoword{ij}}^{+}\,=\,\bspan{\longtwoword{ij}}{\Lminexpset{\bullet},+}
\!=\,\bspan{\longtwoword{ij}}{\NexpExpCompSetij{ij},+}$ in $\Uzrest{+}{ij}\,$.  By the previous argument we have that $\Zaugidealhat{\longtwoword{ij}}^{+}$ is stable under
$\Kinvaut$ so that the second assumption of Corollary~\ref{cor:BaseOrdSplit} is also fulfilled.
It follows that that $\Zaugidealhat{w}^{+}$ is independent of the choice of $w$. The 
respective statements for the other ideals are now immediate from the identities in 
(\ref{eq:KhatTwoDefs}) and (\ref{eq:KgeqzeroDef}).
\end{proof}

Theorem~\ref{thm:MainIdeals} implies, in particular, that the ideals $\ZmaxIdHat{\siAN}$\,, $\ZmaxIdHat{\siNA}$\,, and $\ZmaxIdHat{\siAA}$ defined in (\ref{eq:ZidealMaxNotHat}) are indeed independent of choices of maximal reduced words as well, including the $\LT{G}_2$ case. 

The intertwining relations derived above may also be understood in terms of partial Hopf algebra quotients. For any $s\in\Weyl$ and $w\in\wordset$ with $s=\Weylpres(w)$
define 
\begin{equation}\label{eq:ArtDeltPartQuot}
    \Delta^s=\Tinv_s^{\otimes 2}\circ \Delta\circ \Tinv_s^{-1}
    \qquad \mbox{and}\qquad 
    \Uzn{\ell}^{[s]}=\Uzn{\ell}/\ZaugidealhatIn{w,w}{\ell}\,.
\end{equation}
As with any automorphism, the conjugate $\Delta^s$ defines an alternative Hopf algebra structure 
with antipode $S^{\Tinv_s}\,$. Theorem~\ref{thm:MainIdeals} implies that $\Uzn{\ell}^{[s]}$ is a Hopf algebra defined independent of the choice of the presentation $w$ of $s$.

\begin{prop}\label{prop:PartialIntertw}
Suppose $t=\Weylpres(u)$ as above and $t\leqRB s$ for $s\in\Weyl$. Let $\Rtprodfac u=[\Rtprod u]$ be the image of the partial quasi-$R$-matrix in $(\Uzn{\ell}^{[s]})^{\otimes 2}$\,.
Then $\Delta^t$ descends to a coproduct $\Delta^{[t]}$ on $\Uzn{\ell}^{[s]}$\,, $\Rtprodfac u$ is invertible, and 
$$
\Delta^{[t]}(x)\cdot \Rtprodfac u\,=\,\Rtprodfac u\cdot \Delta(x) \qquad\mbox{for all } x\in \Uzn{\ell}^{[s]}\,. 
$$
\end{prop}

\begin{proof} Recall first from (\ref{eq:NIdealProps}) and (\ref{eq:NidealInDiag}) that 
$\pTmIdeal{u}\subset \FullpTmIdeal {u}\subset \FullpTmIdealSym u= \FullpTmIdealSym t\subseteq \FullpTmIdealSym s$. Thus, 
equations in Corollary~\ref{cor:Rtinv} and Proposition~\ref{prop:CoprodIntWord} also hold modulo 
$\FullpTmIdealSym t$. The former shows that $\Rtprodfac w$ is invertible. The latter can be rewritten
as 
$$
    \Delta^t(x)\,\equiv\,\Rtprod{u}\cdot\Delta(x)\cdot\Rtprodi{u} \qquad \mod \FullpTmIdealSym {s}.
    $$
If $x\in\ZaugidealhatIn{u,u}{\ell}$\,, Proposition~\ref{prop:KhatIdealDiagWord} implies that
$\Delta(x)\in \FullpTmIdealSym {s}$. Thus, the right-hand side and, hence, also $\Delta^t(x)$ is in $\FullpTmIdealSym {s}$. It follows that  $\Delta^t$ descends to a coproduct $\Delta^{[t]}$ on $\Uzn{\ell}^{[s]}$ for which the claimed identity now clearly holds. 
\end{proof}

It is important to keep in mind, though, that the automorphisms $\Tinv_s$ themselves generally do not factor through these Hopf algebra quotients.

\medskip

\subsection{Quasi-\texorpdfstring{$R$}{R}-Matrices for Restricted Quantum Algebras}\label{subsec:RestQRMats}
In this section, we return to the maximal ideal case and collect previous observations for the construction of unique quasi-$R$-matrices for the restricted quantum algebras. The defining relations in tensor products of $\Uzn{\ell}$ will
generally hold only up to ideals, which we introduce first. For $a\in\mathbb N$\,,
let $\FullpTmIdealMaxT a$ be the ideal in $\Uzn{\ell}^{\otimes a}$ generated
by all elements $1^{\otimes (j-1)}\otimes \Epw_v\otimes 1^{\otimes (a-j)}$ and 
$1^{\otimes(j-1)}\otimes \Fpw_v\otimes 1^{\otimes (a-j)}$ for all $\emptyword\neq v\leqRB z$ and $j=1,\ldots,a$ for some $z\in\wordsetmax$\,. Theorem~\ref{thm:MainIdeals} implies 
 the word independent presentation 
\begin{equation}
    \FullpTmIdealMaxT a=
   \ZmaxIdHat{\siAA}\otimes \Uzn{\ell}^{\otimes (a-1)}
     +\Uzn{\ell}\otimes\ZmaxIdHat{\siAA}\otimes \Uzn{\ell}^{\otimes (a-2)}
     +\ldots 
+\Uzn{\ell}^{\otimes (a-1)}\otimes\ZmaxIdHat{\siAA}\,,
\end{equation}
where $\ZmaxIdHat{\siAA}$ is as in (\ref{eq:ZidealMaxNotHat}). Since, for example,
$\FullpTmIdeal {z}\subseteq \FullpTmIdealMaxT 2$\,, Corollary~\ref{cor:RtrmaxDeltaIntw} implies
\begin{equation}\label{eq:PDeltaRestr}
    \barDelta(x)\cdot\Rtprod{z}\,\equiv\, \Rtprod{z}\cdot\Delta(x) \qquad\mod \FullpTmIdealMaxT 2\,,
\end{equation}
for any $x\in\Uzn{\ell}$ and $\Rtprod{z}$ as in (\ref{eq:ElemTrRmatDef}). The remainders in these ideals have been explicitly computed
in Appendix~\ref{sec:proofcoprodbase}. Likewise, the Drinfeld-Tanisaki relations from (\ref{eq:tani2}) hold for the truncated $R$-matrices also only up to remainder terms in these ideals. For example, in the rank one case with
$\Rt=\sum_{s=0}^{\ell-1}{\Rcoeff{s}{\zeta}}^{-1}E^s\otimes F^s$
we have 
$$
\taniauto_{12}(\Rt_{13})\cdot\Rt_{12}\,-\,(\id\otimes\Delta)(\Rt)
=(\Epw\otimes 1\otimes 1)\cdot 
\longsum[18]_{\substack{0\leq s,t <\ell\\ s+t\geq\ell}}
(\Rcoeff{s}{\zeta}\Rcoeff{t}{\zeta})^{-1}
E^{t+s-\ell}\otimes K^{-s}F^t \otimes F^s\,, 
$$\vspace{-4mm}

\noindent which is clearly in $\FullpTmIdealMaxT 3\,$, but not zero. The general case follows similarly. 

\begin{prop}\label{prop:Tani1Restr}
For a given maximal reduced word $\,z\in\wordsetmax\,$ and any Lie type denote $\,\Rt=\Rtprod{z}\,$.  
  \begin{align*}
  \text{Then}    &&(\id\otimes\Delta)(\Rt)&\equiv
    \taniauto_{12}(\Rt_{13})\cdot \Rt_{12}\quad\mod \FullpTmIdealMaxT 3
    &
    \\
    \rule{0mm}{6mm}\text{and}&& 
    (\Delta\otimes \id)(\Rt)&\equiv
    \taniauto_{23}(\Rt_{13})\cdot\Rt_{23}\quad\mod \FullpTmIdealMaxT 3\,.
    &
  \end{align*} 
\end{prop}

\begin{proof} The first identity is straightforwardly verified  by juxtaposing the explicit expressions for both sides using the expressions in (\ref{eq:TrQuasRmatWord})
and (\ref{eq:CoProdCoeffs}), all of which specialize to $q\mapsto \zeta$. 
\begin{align*}
    (\id\otimes\Delta)(\Rt)&
       =\longsum[18]_{\substack{\chi\in\Lmaxexpset{\bullet}\\ \phi,\psi\in\expsetupmax}}
          \Rcoeff{\chi}{\zeta}^{-1}\CoMultcoeff{\chi}{\phi}{\psi}{\zeta}\cdot 
          \Ebaseopp{z}{\chi}\otimes K^{-\vec \psi}\Fbaseopp{z}{\phi}\otimes \Fbaseopp{z}{\psi}\\
          \rule{0mm}{8mm}
    \taniauto_{12}(\Rt_{13})\cdot \Rt_{12}&
          =\longsum[18]_{\substack{\chi\in\expsetupmax\\ \phi,\psi\in\Lmaxexpset{\bullet}}}
          \Rcoeff{\phi}{\zeta}^{-1}\Rcoeff{\psi}{\zeta}^{-1}\Multcoeff{\chi}{\phi}{\psi}{\zeta}\cdot 
        \Ebaseopp{z}{\chi}\otimes K^{-\vec \psi}\Fbaseopp{z}{\phi}\otimes \Fbaseopp{z}{\psi}
\end{align*}

Recall here that $\Rcoeff{\phi}{\zeta}\neq 0$ iff $\phi\in\Lmaxexpset{\bullet}\,$ and observe that
(\ref{eq:CoeffId}) has a well-defined specialization of $q$ to $\zeta$. We infer from this that
for terms for which all three exponents $\chi$, $\phi$, and $\psi$ are in $\Lmaxexpset{\bullet}$
the respective coefficients indeed coincide. The difference between the
equations above is, thus, a sum of terms, each of which has at least one exponent in the complement 
$\Lminexpset{\bullet}$\,. The monomial element associated to this exponent is then, by Corollary~\ref{cor:MaxIdealEasy}, in $\ZmaxIdHat{\siAA}$ so that the full tensor term is in
$\FullpTmIdealMaxT 3$\,. 

This implies the first congruence above. The second identity follows analogously. An alternative argument can be designed by immediately passing to the quotient $\Uzn{\ell}^{\redsup}$ from (\ref{eq:DefReduced}) and observing that  the pairing from (\ref{eq:LTpairs}) factors into a non-degenerate pairing between
the positive and negative parts of $\Uzn{\ell}^{\redsup}$ upon the specialization $q\mapsto \zeta$. 
The image of
$\Rt$ in $(\Uzn{\ell}^{\redsup})^{\otimes 2}$ is then the  element dual to this form. The duality
relation (\ref{eq:UminUplusPair}) also factors into the form on $\Uzn{\ell}^{\redsup}$ and is equivalent
to the Tanisaki relation on $\Uzn{\ell}^{\redsup}$\,.
\end{proof}

Indeed, combining Proposition~\ref{prop:Tani1Restr} and (\ref{eq:PDeltaRestr}) implies that, for any maximal reduced word
$z\in\wordsetmax$, the triple $\,(\Uzn{\ell}^{\redsup},[\Rtprod{z}],\taniauto)\,$ is a pre-triangular Hopf algebra 
in the sense of (\ref{def:tani2}), where $[\Rtprod{z}]$ denotes the image of $\Rtprod{z}$ in 
$(\Uzn{\ell}^{\redsup})^{\otimes 2}\,$. The remaining question is the word independence of 
$[\Rtprod{z}]$ or $\Rtprod{z}$ modulo $\FullpTmIdealMaxT 2\,$.

At this point, we have the tools to prove this claim in several different ways. One is to reconsider the
pairing from (\ref{eq:LTpairs}), which is independent of a choice of $z\in\wordsetmax$ over
$\Zqqvn{\dpone}$ and, thus, also over $\Zzvn{\dpone}$ as a non-degenerate pairing of the positive and 
negative parts of $\Uzn{\ell}^{\redsup}$\,. Since the element $[\Rtprod{z}]$ is dual to this pairing it is also word independent. A second approach is to adapt the uniqueness proof of \cite[Thm.4.1.2(a)]{lu}
to $\Uzn{\ell}^{\redsup}$ and use it directly to show the analog for Corollary~\ref{cor:RprodWI}
for $\Uzn{\ell}^{\redsup}$\,. We offer here a third option that exhibits further symmetries of
the quasi-$R$-matrices. 

\begin{lem}\label{lem:PzMhoSym}
    Suppose $z\in\wordsetmax$\,. Then \quad $\Rtprod{z^\winvchar} =\Kinvaut\otimes \Kinvaut(\Rtprod{z})\equiv\Rtprod{z} \mod \FullpTmIdealMaxT 2$\,.
\end{lem}
\begin{proof} Clearly, the automorphisms $\taniauto$, $S$, and $\Kinvaut$ preserve the maximal ideals and are, hence, well-defined on $\Uzn{\ell}^{\redsup}\,$. The Tanisaki relations from Proposition~\ref{prop:Tani1Restr}, therefore, hold as equalities in the restricted quantum group for the image $\Rtred_z=[\Rtprod{z}]\in (\Uzn{\ell}^{\redsup})^{\otimes 2}$. There, we may now repeat the arguments in
(\ref{eq:PmaxInv}) to show that $\RtredInv_z=[\RtprodInv{z}]$ and $\RtredInv_{z^\winvchar}$ are both 
two-sided inverses of $\Rtred_z$\,, where $\RtprodInv{z}$ denotes the truncated version of the element from
(\ref{eq:PmaxInv}).

Setting ${\mathbf f}(\mu)=(-1)^{\height(\mu)}\zeta^{-\frac 12(\symbrack{\mu}{\mu}-2\symbrack{\weightvecchar}{\mu}}$ for $\mu\in\mathbb Z^{\sroots}$, we may re-express the implied equality $\RtredInv_{z}=\RtredInv_{z^\winvchar}$ 
as the following congruence,
$$
\sum_{\psi\in\Lmaxexpset{\bullet}} {\mathbf f}(\vec\psi)\Rcoeff{\psi}{\zeta}^{-1}\cdot
    \Ebaseopp{z}{\psi}\otimes \Fbaseopp{z^\winvchar}{\psi}
    =
\sum_{\psi\in\Lmaxexpset{\bullet}} {\mathbf f}(\vec\psi)\Rcoeff{\psi}{\zeta}^{-1}\cdot
    \Ebaseopp{z^\winvchar}{\psi}\otimes \Fbaseopp{z}{\psi}
    \quad \mod \FullpTmIdealMaxT 2\,. 
$$
Applying to this the map $x\otimes y\mapsto {\mathbf f}(\mu)^{-1}x\otimes y$ for $\wgrad(x)=\mu$ as well 
as $\Kinvaut\otimes \id$, yields the claimed identity.
\end{proof}

\begin{cor}
    The element $\Rtred=[\Rtprod{z}]\in(\Uzn{\ell}^{\redsup})^{\otimes 2}$ is independent of the choice of $z\in\wordsetmax$\,. 
\end{cor}

\begin{proof}
    By Matsumoto's Theorem, it suffices to show $\Rtprod{z}\equiv\Rtprod{w}\mod\FullpTmIdealMaxT 2$ for maximal reduced words of the form $z=a\cdot \longtwoword{ij}\cdot b$ and $w=a\cdot \longtwoword{ji}\cdot b$.  Following (\ref{eq:PwordProps}) we have 
    $$
    \Rtprod{z}=\Tinv_{a\cdot \longtwoword{ij}}^{\otimes 2}(\Rtprod{b})
      \cdot \Tinv_{a}^{\otimes 2}(\Rtprod{\longtwoword{ij}})
      \cdot\Rtprod{a}\,,
    $$
    and the analogous expression for $\Rtprod{w}$\,. By Lemma~\ref{lem:PzMhoSym} the difference between
    $\Rtprod{\longtwoword{ij}}$ and  
    $\Rtprod{\longtwoword{ji}}=\Kinvaut\otimes \Kinvaut(\Rtprod{\longtwoword{ij}})$ is sum of terms 
    of the form $(\Epw_c\otimes 1)A$ and $(1\otimes \Fpw_c)B$ with $c\leqRB\longtwoword{ij}$\,. Since 
    for these $a\cdot c$ is also reduced, we have $\Tinv_a(\Epw_c)=\Epw_{a\cdot c}$ and 
    $\Tinv_a(\Fpw_c)=\Fpw_{a\cdot c}$\,. It follows from their above decompositions that the difference
    of $\Rtprod{z}$ and $\Rtprod{w}$ is a sum of terms of the form  
    $(\Epw_{a\cdot c}\otimes 1)A'$ and $(1\otimes \Fpw_{a\cdot c})B'\,$, which is, therefore, in  
    $\FullpTmIdealMaxT 2$ by word independence of the ideal.
\end{proof}

The findings of this section together with Corollary~\ref{cor:MaxIdealEasy}
can now be summarized as follows. 

\begin{thm}\label{thm:final}
   For any Lie type and root of unity $\zeta$ of order $\kay$ as in (\ref{eq:kaycond}),
   the triple $\bigl(\Uzn{\ell}^{\redsup},\Rtred,\taniauto\bigr)$ is a pre-triangualar Hopf algebra over $\Zzvn{\ell}$\,, independent of choices of maximal reduced words (or convex orderings)  used in its construction. 
   
   \noindent Moreover, all (anti)automorphisms defined in Section~\ref{subsec:gradings}, including the Artin group action, factor into well-defined (anti)automorphisms on $\Uzn{\ell}^{\redsup}$. 
   The construction is, thus, also independent of chosen directions of multiplications and
   independent of whether the $\Tinv_i$ or $T_i$ are used in the definition of generators. 
\end{thm}

The construction of braided tensor categories follows  along the same steps as described in Section~\ref{subsec:PreTriang}, where $\ExpGrp$ is now allowed to have positive characteristic.

\newpage

\appendix

\section{Proof of Lemma \ref{lem:coprodbase}}\label{sec:proofcoprodbase}
In this section, we prove Lemma \ref{lem:coprodbase}. It suffices to show that equation \eqref{eq:coprodbase}
holds for each choice of $A_{ij}\in \{2,0,-1,-2,-3\}$. The proof is carried out  in a series of computations, with a subsection dedicated to each choice of $A_{ij}$\,. In most cases, we follow the same sequence of steps.

\begin{proof} 
The $A_{ij}=0$ case is obvious because $\Rti$ commutes with $\Delta(E_j)$ and since the action  $\Tinv_i(E_j)=E_j\,$ is trivial. 

In the remaining cases, we make use of the fact that the coefficients $a_s=\dfrac{\zeta_i^{s(s-1)/2}}{[s]!_i}(\zeta_i-\zeta_i^{-1})^s$ of $\Rti=\,\sum_{s=0}^{\elli-1}a_s \cdot E_i^s\otimes F_i^s$ satisfy the 
recursion relation
\begin{equation} 
    a_{s+1}=\dfrac{\zeta_i^{s}(\zeta_i-\zeta_i^{-1})}{[s+1]_i}a_{s}\,.
\end{equation}

\subsection{Case \texorpdfstring{$A_{ij}=2$\,}{Aij=2}} In this case, we write $i=j$.
Using \eqref{eq:defTeasy}, equation \eqref{eq:coprodbase} then reduces to   
\[
-K_i^{-1}F_i\otimes K_i^{-1}-K_i^{-2}\otimes K_i^{-1}F_i\equiv\, \Rtii\cdot(-K_i^{-1}F_i\otimes K-1\otimes K_i^{-1}F_i)\cdot\Rti \qquad\mod \FullpTmIdeal {w_i}\,. 
\]
We will use the commutation relation
 \begin{align*}
    [F_i,E_i^s]=-[s]_i\dfrac{\zeta_i^{s-1}K_i-\zeta_i^{-(s-1)}K_i^{-1}}{\zeta_i-\zeta_i^{-1}}E_i^{s-1}\;,
\end{align*}
which follows from a  straightforward induction. Applying these commutators to each summand 
of the $R$-matrix we find the relations 
\begin{align*}
    (K_i^{-1}F_i\otimes K_i^{-1})\Rti
    &=
    \Rti(K_i^{-1}F_i\otimes K_i^{-1})
    -\sum_{s=0}^{\elli-2}a_s(E_i^s\otimes F_i^s)\left(\left(\zeta_i^{2s}-K_i^{-2}\right)\otimes K_i^{-1}F_i\right)
    \\
    &
    =\Rti(K_i^{-1}F_i\otimes K_i^{-1}+K_i^{-2}\otimes K_i^{-1}F_i)
    -\sum_{s=0}^{\elli-2}a_s\zeta_i^{2s}(E_i^s\otimes F_i^s)\left(1\otimes K_i^{-1}F_i\right)\,,
    \\
    (1\otimes K_i^{-1}F)\Rt_i
    &=\sum_{s=0}^{\elli-1}a_s\zeta_i^{2s}(E_i^s\otimes F_i^s)(1\otimes K_i^{-1}F_i)
    \,.
\end{align*}
Combining terms now yields
\begin{align*}
    (-K_i^{-1}F_i\otimes K_i^{-1}-1\otimes K_i^{-1}F_i)\Rt_i
    &=
    \Rt_i(-K_i^{-1}F_i\otimes K_i^{-1}-K_i^{-2}\otimes K_i^{-1}F_i)
    +
    a_{\elli-1}q^{2\elli}E_i^{\ell-1}K_i^{-2}\otimes F_i^\ell K_i^{-1}\;, 
\end{align*}
which implies the desired equality modulo the ideal $\FullpTmIdeal {w_i}$.

\subsection{Case \texorpdfstring{$A_{ij}=-1$\,}{Aij=-1}} In this case $E_{(ij)}=-E_jE_i+\zeta^{-1}E_iE_j$\,. We use the following relations from \cite[Section 5.3]{lu90b}. 
\begin{align*}
    E_{(ij)}E_i=\zeta E_iE_{(ij)}\,,&&
    [E_{(ij)},F_i^s]=-\zeta^{1-s}[s]_iF_i^{s-1}K_iE_j\,,
\end{align*}
with the latter relation obtained by applying $\Tinv_i\circ \Kinvaut$ to Lusztig's equation (c) with $k'=1$\,. These give the following relations between $\Rti$ and terms in $\Tinv_i^{\otimes 2}\circ \Delta(E_j)$\,. 

\begin{align*}
    (E_{(ij)}\otimes K_iK_j) \Rti
    =&\, \Rti(E_{(ij)}\otimes K_iK_j)\,,
    \\
    (1\otimes E_{(ij)}) \Rti
    =&\, \Rti(1\otimes E_{(ij)})-(\zeta-\zeta^{-1})\sum_{s=0}^{\elli-2}a_s(E_i^s\otimes F_i^s)(E_i\otimes K_iE_j)\,,
\\
    \Rtii \cdot \Tinv_i^{\otimes 2}\circ \Delta(E_j)\Rti
    =&\,
     \Rtii(E_{(ij)}\otimes K_iK_j+1\otimes E_{(ij)})\Rti
    \\=&\,
    E_{(ij)}\otimes K_iK_j+1\otimes E_{(ij)}
    -(\zeta-\zeta^{-1})E_i\otimes K_iE_j
    \\&
    +\Rtii(\zeta-\zeta^{-1})a_{\elli-1}
    (E_i^\elli\otimes F_i^{\elli-1}K_iE_j)
    \\
    \equiv&\,
    E_{(ij)}\otimes K_iK_j+1\otimes E_{(ij)}
    -(\zeta-\zeta^{-1})E_i\otimes K_iE_j
    \qquad\mod \FullpTmIdeal {w_i}\,.
\end{align*}

This expression agrees with the coproduct of $E_{(ij)}$ computed below and, therefore, verifies the claim in this case.
\begin{align*}
\Delta\circ\Tinv_i(E_j)&=\left\lbrace\begin{matrix*}[l]
    -(E_jE_i\otimes K_iK_j+E_j\otimes K_jE_i+E_i\otimes E_jK_i+1\otimes E_jE_i)
    \\
    +\zeta^{-1}(E_iE_j\otimes K_iK_j+E_j\otimes E_iK_j+E_i\otimes K_iE_j+1\otimes E_iE_j)
    \end{matrix*}\right.
    \\&=
    E_{(ij)}\otimes K_iK_j+1\otimes E_{(ij)}-(\zeta-\zeta^{-1})E_i\otimes K_iE_j\,
\end{align*}

\subsection{Case \texorpdfstring{$A_{ij}=-2$\,}{Aij=-2}} The following equations
are obtained from 
$$
    E_{(ij)}
=\zeta^{-2}E_i^{(2)}E_j-\zeta^{-1}E_iE_jE_i+E_jE_i^{(2)}
$$ 
as well as the ones listed in \cite[Section 5.3]{lu90b}.
\begin{align*}
    E_{(ij)}E_i&=\zeta^2E_iE_{(ij)}\,
    \\
    [E_{(ij)},F_i^s]
    &=
    -[s]\zeta^{1-s}F_i^{s-1}K_iE_{(iji)}
    +[s][s-1]\zeta^{-2s+4}F_i^{s-2}K_i^2E_j 
\end{align*}
The latter equation is obtained by applying $\Tinv_i\circ\Kinvaut$ to Lusztig's equation (i). 
We apply these next to compute the commutator of $\Rti$ with $\Tinv^{\otimes 2}\circ \Delta(E_j)$\,.
\begin{align*}
    (E_{(ij)}\otimes K_i^2K_j)\Rti
    =&\Rti(E_{(ij)}\otimes K_i^2K_j)\,,
    \\
    (1\otimes E_{(ij)})\Rti
    =&
    \Rti(1\otimes E_{(ij)})
    -(\zeta-\zeta^{-1})\sum_{s=0}^{\ell-2}a_s(E_i^s\otimes F_i^s)(E_i\otimes K_iE_{(iji)})
    \\&
    +\zeta(\zeta-\zeta^{-1})^2 \sum_{s=0}^{\ell-3}a_s(E_i^s\otimes F_i^s)(E_i^2\otimes K_i^2E_j)\,.
\end{align*}

\noindent
We infer from these relations that $\Rtii\cdot\Tinv_i^{\otimes 2}\circ\Delta(E_j)\Rti$ equals 
\begin{samepage}
\begin{align*}
    &E_{(ij)}\otimes K_i^2K_j+1\otimes E_{(ij)}
    -(\zeta-\zeta^{-1})E_i\otimes K_iE_{(iji)}
    +\zeta(\zeta-\zeta^{-1})^2E_i^2\otimes K_i^2E_j
    \\
    &+\Rtii\left(a_{\ell-1}(\zeta-\zeta^{-1})E_i^{\ell}\otimes F_i^{\ell-1}K_iE_{(iji)}
    -\zeta(\zeta-\zeta^{-1})^2\sum_{s=\ell-2}^{\ell-1}a_sE_i^{s+2}\otimes F_i^sK_i^2E_j\right)
    \\
    \equiv&\,
    E_{(ij)}\otimes K_i^2K_j+1\otimes E_{(ij)}
    -(\zeta-\zeta^{-1})E_i\otimes K_iE_{(iji)}
    +\zeta(\zeta-\zeta^{-1})^2E_i^2\otimes K_i^2E_j
    \qquad \mod\FullpTmIdeal {w_i}\,.
\end{align*}
\end{samepage}

\noindent
Since
\begin{align*}
    \Delta\circ\Tinv_i(E_j)&=
    \left\lbrace\begin{matrix*}[l]
    \zeta^{-2}\begin{pmatrix}
    E_i^{(2)}E_j\otimes K_i^2K_j
    +\zeta E_iE_j\otimes E_iK_iK_j
    +E_j\otimes E_i^{(2)}K_j
    \\+E_i^{(2)}\otimes K_i^2E_j
    +\zeta_iE_i\otimes E_iK_iE_j
    +1\otimes E_i^{(2)}E_j
    \end{pmatrix}
    \\
    \rule[-8mm]{0mm}{18mm}
    -\zeta^{-1}
    \begin{pmatrix}
    E_iE_jE_i\otimes K_i^2K_j
    +E_iE_j\otimes K_iK_jE_i
    +E_i^{2}\otimes K_iE_jK_i
    +E_jE_i\otimes E_iK_iK_j
    \\+E_i\otimes K_iE_jE_i
    +E_j\otimes E_iK_jE_i
    +E_i\otimes E_iE_jK_i
    +1\otimes E_iE_jE_i
    \end{pmatrix}\\
    +\begin{pmatrix}
    E_jE_i^{(2)}\otimes K_jK_i^2
    +\zeta E_jE_i\otimes K_jE_iK_i
    +E_j\otimes K_j E_i^{(2)}
    \\+E_i^{(2)}\otimes E_jK_i^2
    +\zeta E_i\otimes E_jE_iK_i+1\otimes E_jE_i^{(2)}
    \end{pmatrix}
    \end{matrix*}\right.
    \\
    &=E_{(ij)}\otimes K_jK_i^2+1\otimes E_{(ij)}
    -(\zeta-\zeta^{-1})E_i\otimes (K_i(\zeta^{-2}E_iE_j-E_jE_i))
    +\zeta(\zeta-\zeta^{-1})^2E_i^2\otimes K_i^2E_j
    \\&=
    E_{(ij)}\otimes K_i^2K_j+1\otimes E_{(ij)}
    -(\zeta-\zeta^{-1})E_i\otimes K_i E_{(iji)}
    +\zeta(\zeta-\zeta^{-1})^2E_i^2\otimes K_i^2E_j\;,
\end{align*}
this proves the case for $A_{ij}=-2$\,.

\subsection{Case \texorpdfstring{$A_{ij}=-3$\,}{Aij=-3}} In this case,
\begin{align*}
    E_{(ij)}
=\Tinv_i(E_j)
=\zeta^{-3}E_i^{(3)}E_j-\zeta^{-2}E_i^{(2)}E_jE_i+\zeta^{-1}E_iE_jE_i^{(2)}-E_jE_i^{(3)}
=(T_{w_jw_i})^2(E_j)\,.
\end{align*}
Below, we refer to Lusztig's equations (a1) and (a3) in \cite[Section 5.4]{lu90b}.
\begin{align*}
    E_{(ij)}E_i&=\zeta^3E_iE_{(ij)}\,
    \\
    [E_{(ij)},F_i^s]&=-[s]\zeta^{1-s}F_i^{s-1}K_iE_{(iji)}
    +[s]!\zeta^{4-2n}F_i^{(s-2)}K_i^2E_{(ijiji)}-\zeta^{9-3n}[s]!F_i^{(s-3)}K_i^3E_j
\end{align*}

\noindent
These relations yield expressions for conjugation of 
$\Tinv_i^{\otimes 2}\circ\Delta(E_j)
    =
E_{(ij)}\otimes K_i^3K_j+1\otimes E_{(ij)}$ by $\Rti$\, via the following computation.
\begin{align*}
    (E_{(ij)}\otimes K_i^3K_j)\Rti
    =&\,\Rti(E_{(ij)}\otimes K_i^3K_j)\,,
    \\
    (1\otimes E_{(ij)})\Rti
    =&\,
    \Rti(1\otimes E_{(ij)})
    -(\zeta-\zeta^{-1})\sum_{s=0}^{\ell-2}a_s(E_i^s\otimes F_i^s)(E_i\otimes K_iE_{(iji)})
    \\&
    +\zeta(\zeta-\zeta^{-1})^2 \sum_{s=0}^{\ell-3}a_s(E_i^s\otimes F_i^s)(E_i^2\otimes K_i^2E_{(ijiji)})
    \\&
    -\zeta^3(\zeta-\zeta^{-1})^3 \sum_{s=0}^{\ell-4}a_s(E_i^s\otimes F_i^s)(E_i^3\otimes K_i^3E_j)\,,
    \\
    \rule{0mm}{7mm}
    \Rtii\cdot\Tinv_i^{\otimes 2}\circ\Delta(E_j)\Rti
    =&\,
    E_{(ij)}\otimes K_i^3K_j+1\otimes E_{(ij)}
    -(\zeta-\zeta^{-1})E_i\otimes K_iE_{(iji)}
    \\
    &+\zeta(\zeta-\zeta^{-1})^2E_i^2\otimes K_i^2E_{(ijiji)}
    -\zeta^3(\zeta-\zeta^{-1})^3E_i^3\otimes K_i^3E_j
    \\
    &+\Rtii
    \left(
    \begin{matrix}
        (\zeta-\zeta^{-1})a_{\ell-1}E_i^\ell\otimes F_i^{\ell-1} K_iE_{(iji)}
        \\
   \rule[-6mm]{0mm}{14mm}     
   -\zeta(\zeta-\zeta^{-1})^2 \sum_{s=\ell-2}^{\ell-1}a_s(E_i^{s+2}\otimes F_i^s K_i^2E_{(ijiji)})
    \\
    +\zeta^3(\zeta-\zeta^{-1})^3 \sum_{s=\ell-3}^{\ell-1}a_s(E_i^{s+3}\otimes F_i^s K_i^3E_j)
    \end{matrix}
    \right)
    \\
    \equiv&\,
    E_{(ij)}\otimes K_i^3K_j+1\otimes E_{(ij)}
    -(\zeta-\zeta^{-1})E_i\otimes K_iE_{(iji)}
    \\
    &+\zeta(\zeta-\zeta^{-1})^2E_i^2\otimes K_i^2E_{(ijiji)}
    -\zeta^3(\zeta-\zeta^{-1})^3E_i^3\otimes K_i^3E_j
    \qquad \mod\FullpTmIdeal{w_i}\,.
\end{align*}

\noindent
It remains to verify the coproduct identity 
\begin{align*}
    \Delta\circ \Tinv_i(E_j)=\Delta(E_{(ij)})=&E_{(ij)}\otimes K_jK_i^3+1\otimes E_{(ij)}-(\zeta-\zeta^{-1})E_i\otimes K_iE_{(iji)}
    \\
    &+\zeta(\zeta-\zeta^{-1})^2E_i^{2}\otimes K_i^2E_{(ijiji)}-\zeta^3(\zeta-\zeta^{-1})^3E_i^{3}\otimes K_i^3E_j\,.
\end{align*}

\noindent
We begin by expanding the coproduct for the divided powers expressions in $\Delta(E_{(ij)})$\,:
\allowdisplaybreaks
\begingroup
\renewcommand*{\arraystretch}{1.3}
\begin{align*}
    &\,\begin{matrix*}[l]
    \zeta^{-3}\begin{pmatrix}
    E_i^{(3)}E_j\otimes K_i^3K_j
    +\zeta^2E_i^{(2)}E_j\otimes E_iK_i^2K_j
    +\zeta^2E_iE_j\otimes E_i^{(2)}K_iK_j
    +E_j\otimes E_i^{(3)}K_j
    \\+E_i^{(3)}\otimes K_i^3E_j
    +\zeta^2E_i^{(2)}\otimes E_iK_i^2E_j
    +\zeta^2E_i\otimes E_i^{(2)}K_iE_j
    +1\otimes E_i^{(3)}E_j
    \end{pmatrix}
    \\
    -\begin{pmatrix}
    \zeta^{-2}\begin{pmatrix}
    E_i^{(2)}E_j\otimes K_i^2K_j
    +\zeta E_iE_j\otimes E_iK_iK_j
    +E_j\otimes E_i^{(2)}K_j
    \\+E_i^{(2)}\otimes K_i^2E_j
    +\zeta E_i\otimes E_iK_iE_j
    +1\otimes E_i^{(2)}E_j
    \end{pmatrix}\\
    -\dfrac{\zeta^{-1}}{[2]}\begin{pmatrix}
    E_iE_jE_i\otimes K_i^2K_j
    +E_iE_j\otimes K_iK_jE_i
    +E_i^{2}\otimes K_iE_jK_i
    +E_jE_i\otimes E_iK_iK_j
    \\+E_i\otimes K_iE_jE_i
    +E_j\otimes E_iK_jE_i
    +E_i\otimes E_iE_jK_i
    +1\otimes E_iE_jE_i
    \end{pmatrix}
    \\
    +\dfrac{1}{[3]}\begin{pmatrix}
    E_jE_i^{(2)}\otimes K_jK_i^2
    +\zeta E_jE_i\otimes K_jE_iK_i
    +E_j\otimes K_j E_i^{(2)}
    \\+E_i^{(2)}\otimes E_jK_i^2
    +\zeta E_i\otimes E_jE_iK_i+1\otimes E_jE_i^{(2)}
    \end{pmatrix}
    \end{pmatrix}\Delta(E_i)
    \end{matrix*}
    \\[2em]
=&\,E_{(ij)}\otimes K_jK_i^3+1\otimes E_{(ij)}
    \\
    &\begin{matrix*}[l]
    +\zeta^{-3}\begin{pmatrix}
    \zeta^2E_i^{(2)}E_j\otimes E_iK_i^2K_j
    +\zeta^2E_iE_j\otimes E_i^{(2)}K_iK_j
    +E_j\otimes E_i^{(3)}K_j
    \\+E_i^{(3)}\otimes K_i^3E_j
    +\zeta^2E_i^{(2)}\otimes E_iK_i^2E_j
    +\zeta^2E_i\otimes E_i^{(2)}K_iE_j
    \end{pmatrix}\\
    +
    \begin{pmatrix}
    -\zeta^{-2}E_i^{(2)}E_j\otimes K_i^2K_jE_i
    -\zeta^{-2}E_i\otimes E_i^{(2)}E_jK_i
    +\dfrac{\zeta^{-1}}{[2]}E_iE_jE_i\otimes K_i^2K_jE_i
    \\+\dfrac{\zeta^{-1}}{[2]}E_i\otimes E_iE_jE_iK_i
    -\dfrac{1}{[3]}E_jE_i^{(2)}\otimes K_jK_i^2E_i
    -\dfrac{1}{[3]}E_i\otimes E_jE_i^{(2)}K_i
    \end{pmatrix}
    \\
    -\begin{pmatrix}
    \zeta^{-2}\begin{pmatrix}
    \zeta E_iE_j\otimes E_iK_iK_j
    +E_j\otimes E_i^{(2)}K_j
    \\+E_i^{(2)}\otimes K_i^2E_j
    +\zeta E_i\otimes E_iK_iE_j
    \end{pmatrix}\\
    -\dfrac{\zeta^{-1}}{[2]}\begin{pmatrix}
    E_iE_j\otimes K_iK_jE_i
    +E_i^{2}\otimes K_iE_jK_i
    +E_jE_i\otimes E_iK_iK_j
    \\+E_i\otimes K_iE_jE_i
    +E_j\otimes E_iK_jE_i
    +E_i\otimes E_iE_jK_i
    \end{pmatrix}\\
    +\dfrac{1}{[3]}\begin{pmatrix}
    \zeta E_jE_i\otimes K_jE_iK_i
    +E_j\otimes K_j E_i^{(2)}
    \\+E_i^{(2)}\otimes E_jK_i^2
    +\zeta E_i\otimes E_jE_iK_i
    \end{pmatrix}
    \end{pmatrix}\Delta(E_i)
    \end{matrix*}
    \\[2em]
=&\,E_{(ij)}\otimes K_jK_i^3+1\otimes E_{(ij)}
    \\
    &\begin{matrix*}[l]
    +\zeta^{-3}\begin{pmatrix}
    \zeta E_i^{(2)}E_j\otimes K_i^2K_jE_i
    +\zeta^4E_iE_j\otimes K_iK_jE_i^{(2)}
    +\zeta^9E_j\otimes K_jE_i^{(3)}
    \\+E_i^{(3)}\otimes K_i^3E_j
    +\zeta^{-2}E_i^{(2)}\otimes K_i^2E_iE_j
    +\zeta^{-2}E_i\otimes K_iE_i^{(2)}E_j
    \end{pmatrix}\\
    +
    \begin{pmatrix}
    -\zeta^{-2}E_i^{(2)}E_j\otimes K_i^2K_jE_i
    -\zeta^{-3}E_i\otimes K_iE_i^{(2)}E_j
    +\dfrac{\zeta^{-1}}{[2]}E_iE_jE_i\otimes K_i^2K_jE_i
    \\+\dfrac{\zeta^{-2}}{[2]}E_i\otimes K_iE_iE_jE_i
    -\dfrac{1}{[3]}E_jE_i^{(2)}\otimes K_jK_i^2E_i
    -\dfrac{\zeta^{-1}}{[3]}E_i\otimes K_iE_jE_i^{(2)}
    \end{pmatrix}
    \\
    -\begin{pmatrix}
    \dfrac{\zeta}{[2]}E_iE_j\otimes K_iK_jE_i
    +\dfrac{\zeta^6}{[3]}E_j\otimes K_iE_i^{(2)}
    +\dfrac{\zeta^6-\zeta^4-\zeta^2+\zeta^{-2}+\zeta^{-4}}{[3]}E_i^{(2)}\otimes K_i^2E_j
    \\+
    \dfrac{\zeta^{-4}+\zeta^{-2}-1}{[2]}E_i\otimes K_iE_iE_j-\dfrac{\zeta^2}{[2][3]}E_jE_i\otimes K_iK_jE_i+\dfrac{\zeta^3-\zeta^{-1}-\zeta^{-3}}{[2][3]}E_i\otimes K_iE_jE_i
    \end{pmatrix}\Delta(E_i)
    \end{matrix*}
    \\[2em]
=&\,E_{(ij)}\otimes K_jK_i^3+1\otimes E_{(ij)}-\zeta^{-4}(\zeta-\zeta^{-1})E_i\otimes K_iE_i^{(2)}E_j
    \\
&\begin{matrix*}[l]
    +\zeta^{-3}\begin{pmatrix}
    \zeta^4E_iE_j\otimes K_iK_jE_i^{(2)}
    +\zeta^9E_j\otimes K_jE_i^{(3)}
    \\+E_i^{(3)}\otimes K_i^3E_j
    +\zeta^{-2}E_i^{(2)}\otimes K_i^2E_iE_j
    \end{pmatrix}
    +
    \begin{pmatrix}
    \dfrac{\zeta^{-1}}{[2]}E_iE_jE_i\otimes K_i^2K_jE_i
    +\dfrac{\zeta^{-2}}{[2]}E_i\otimes K_iE_iE_jE_i\\
    -\dfrac{1}{[3]}E_jE_i^{(2)}\otimes K_jK_i^2E_i
    -\dfrac{\zeta^{-1}}{[3]}E_i\otimes K_iE_jE_i^{(2)}
    \end{pmatrix}
    \\
    -\begin{pmatrix}
    \dfrac{\zeta^{-1}}{[2]}E_iE_jE_i\otimes K_i^2K_jE_i
    +\dfrac{\zeta}{[2]}E_iE_j\otimes K_iK_jE_i^2
    +\dfrac{\zeta^6}{[2][3]}E_j\otimes K_iE_i^{3}\\
    +\dfrac{\zeta^9-\zeta^7-\zeta^5+\zeta+\zeta^{-1}}{[2][3]}E_i^{3}\otimes K_i^3E_j
    +\zeta(\zeta-\zeta^{-1})E_i^{2}\otimes K_i^2E_jE_i
    +
    \dfrac{\zeta^{-3}+\zeta^{-1}-\zeta}{[2]}E_i^2\otimes K_i^2E_iE_j\\
    +\dfrac{\zeta^{-4}+\zeta^{-2}-1}{[2]}E_i\otimes K_iE_iE_jE_i
    -\dfrac{1}{[2][3]}E_jE_i^2\otimes K_i^2K_jE_i
    +\dfrac{\zeta^3-\zeta^{-1}-\zeta^{-3}}{[2][3]}E_i\otimes K_iE_jE_i^2
    \end{pmatrix}
    \end{matrix*}
\end{align*}
\endgroup

\noindent
Further collapsing these expressions, we obtain
\begin{align*}  
    \Delta(E_{(ij)})=&\,E_{(ij)}\otimes K_jK_i^3+1\otimes E_{(ij)}-\zeta^{-4}(\zeta-\zeta^{-1})E_i\otimes K_iE_i^{(2)}E_j
    -(\zeta-\zeta^{-1})E_i\otimes K_iE_jE_i^{(2)}
    \\
    &+\zeta^{-2}(\zeta-\zeta^{-1})E_i\otimes K_iE_iE_jE_i
    -\zeta(\zeta-\zeta^{-1})^2E_i^{2}\otimes K_i^2E_jE_i
    \\&
    -\zeta^3(\zeta-\zeta^{-1})^3E_i^{3}\otimes K_i^3E_j
    -\zeta^{-2}(\zeta-\zeta^{-1})^2E_i^2\otimes K_i^2E_iE_j
    \\
    \rule{0mm}{7mm}
    =&\,E_{(ij)}\otimes K_jK_i^3+1\otimes E_{(ij)}-(\zeta-\zeta^{-1})E_i\otimes K_i(\zeta^{-4}E_i^{(2)}E_j-\zeta^{-2}E_iE_jE_i+E_jE_i^{(2)})
    \\
    &+\zeta(\zeta-\zeta^{-1})^2E_i^{2}\otimes K_i^2(\zeta^{-3}E_iE_j-E_jE_i)-\zeta^3(\zeta-\zeta^{-1})^3E_i^{3}\otimes K_i^3E_j
    \\
    \rule{0mm}{7mm}
    =&\,E_{(ij)}\otimes K_jK_i^3+1\otimes E_{(ij)}-(\zeta-\zeta^{-1})E_i\otimes K_iE_{(iji)}
    \\
    &+\zeta(\zeta-\zeta^{-1})^2E_i^{2}\otimes K_i^2E_{(ijiji)}-\zeta^3(\zeta-\zeta^{-1})^3E_i^{3}\otimes K_i^3E_j\,.
\end{align*}

\noindent
This now proves $\Delta\circ\Tinv_i(E_j)\,\equiv\, \Rtii\cdot\Tinv^{\otimes 2}_i\circ\Delta(E_j)\cdot\Rti \quad\mod \FullpTmIdeal {w_i}$\,.
\end{proof}

\medskip

\section{Proofs and Computations for Section~\ref{subsec:ProofCoprodB2}}\label{app:powerformula} 
We supply here the computational proofs of 
Proposition~\ref{prop:PowerFormula} and Lemma~\ref{lem:SumcIdentity}
used in Sections~\ref{subsec:Coalg-ZB2} and \ref{subsec:ProofCoprodB2}. 

\subsection{Proof of Proposition~\ref{prop:PowerFormula}}
We begin with a variation of the recursive identity used to define the coefficients $c_{k,t}$ in \eqref{eq:cRecursion} that applies more directly to the desired proof. To this end we introduce
notation for the coefficients actually appearing in the formula of Proposition~\ref{prop:PowerFormula}, namely,
\begin{equation}\label{eq:def-a-stk}
    a^n_{s,t,k}=q^{-(s+t+k)(n-s-t-k)-(s-t)(2t+k)}\qbin{n}{s+t+k}{}\qbin{s+t+k}{s-t}{}c_{k,t}\,,
\end{equation}
for $s\geq t\geq 0$\,, $k\geq 0$\,, and $n\geq s+t+k\,$. We further have $a^0_{0,0,0}=1$ and use the convention that $a^n_{s,t,k}=0$ whenever any one of the index bounds is exceeded. That is, whenever
 \begin{equation}
     s,t,k<0\,,\qquad s>n\,,\qquad t>\min{(s,n-s)}\,,\quad\mbox{or} \quad k>n-s-t\,.
 \end{equation}

 \begin{lem}\label{lm:rec-a-stk}
 With definitions and conventions as above, we have the  recursion 
    \[a^{n+1}_{s,t,k}=
    q^{-2(n+1-s+t)}a^n_{s-1,t,k}
    +
    q^{-3(n+1-s-t-k)}[n+2-s-t-k]a^n_{s-1,t-1,k}
    +
    q^{-2(n+1-s-t-k)}a^n_{s,t,k-1}
    +a^{n}_{s,t,k}\,.
    \]
\end{lem}

\begin{proof} The following straightforward verification makes repeated use of the 
$q$-binomial formula from \eqref{eq:qbin-recurs}. To shorten some formulae, we use $u=s+t+k$\,.
\begin{align*}
        a^{n+1}_{s,t,k}=&~
        q^{-u(n+1-u)-(s-t)(2t+k)} \qbin{n+1}{u}{}\qbin{u}{s-t}{}c_{k,t}
        \\
        =&~
        q^{-u(n+1-u)-(s-t)(2t+k)}\left(
        q^u\qbin{n}{u}{}
        +
        q^{-(n+1-u)}\qbin{n}{u-1}{}
        \right)\qbin{u}{s-t}{}c_{k,t}
        \\
        =&~
        a^n_{s,t,k}+
        q^{-(u+1)(n+1-u)-(s-t)(2t+k)}\qbin{n}{u-1}{}
        \left(q^{s-t}\qbin{u-1}{s-t}{}+q^{-(2t+k)}\qbin{u-1}{s-1-t}{}\right)c_{k,t}
        \\
        =&~
        a^n_{s,t,k}+q^{-2(n+1-s+t)}a^n_{s-1,t,k}
        \\
        &+
        q^{-(u+1)(n+1-u)-(s-t)(2t+k-1)}\qbin{n}{u-1}{}\qbin{u-1}{s-t}{}(c_{k-1,t}+q^{-(2t+k-2)}[2t+k-1]_{}c_{k,t-1})
        \\
        =&~
        a^n_{s,t,k}+q^{-2(n+1-s+t)}a^n_{s-1,t,k}
        +q^{-2(n+1-u)}a^n_{s,t,k-1}
        \\
        &+
        q^{-(u+1)(n+1-u)-(s-t+1)(2t+k-1)+1}[n+2-u]_{}\qbin{n}{u-2}{}\qbin{u-2}{s-t}{}c_{k,t-1}
        \\
        =&
        a^n_{s,t,k}+q^{-2(n+1-s+t)}a^n_{s-1,t,k}
        +q^{-2(n+1-u)}a^n_{s,t,k-1}
        +q^{-3(n+1-u)}[n+2-u]_{}a^n_{s-1,t-1,k}\,.
    \end{align*}
\end{proof}

Observe that we may also consider the $a^n_{s,t,k}$ to be defined via the
recursion in Lemma~\ref{lm:rec-a-stk} and the mentioned boundary conditions. 
The following identity is proved in  \cite[Lemma 1.6]{Lu90a} for the special case $v=1$. 

\begin{lem}\label{lem:uvComm}
    Suppose $A,B,D$ satisfy
    \begin{align*}
        DA=v^2AD+B, && BA=q^{-2}v^2AB, \qquad \mbox{and}&& DB=q^{-2}v^2BD\,.
    \end{align*}
    Then \ $D^mA=v^{2m}AD^m+(q^{-1}v^2)^{m-1}[m]BD^{m-1}$ \ and
    \[
        D^mA^n=\sum_{i=0}^{\min(m,n)}v^{2mn-i(i+1)}q^{-i(n+m-2i)-i(i-1)/2}\qbin{n}{i}{}\qbin{m}{i}{}[i]! A^{n-i}B^iD^{m-i}\,,
    \]
    where quantum numbers and coefficients are with respect to $q\,$. 
\end{lem}
\begin{proof}
    Each of the two claims of the lemma is proved by induction. The base case of the former is the assumption $DA=v^2AD+B$. For the induction step, we compute
    \begin{align*}
        D^{m+1}A&=D(v^{2m}AD^m+(q^{-1}v^2)^{m-1}[m]BD^{m-1})
        \\
    &=v^{2(m+1)}AD^{m+1}+v^{2m}BD^m+(q^{-2}v^2)(q^{-1}v^2)^{m-1}[m]BD^m\\
    &=v^{2(m+1)}AD^{m+1}+(q^{-1}v^2)^m[m+1]BD^m\,,
    \end{align*}
    where the last equality uses $1+q^{-(m+1)}[m]=q^{-m}[m+1]\,$. The second statement 
    is now derived by induction in $n$ using the first identity as the base case. 
    It amounts to repeated use of \eqref{eq:qbin-recurs} and resummations. 
\end{proof}

Note that the $q$-binomials disappear in the formula of Lemma~\ref{lem:uvComm} if it is rewritten
in terms of divided power generators $A^{(n)}=\frac 1{[n]!}A^n$. That is, we find instead 
    \[
    D^{(m)}A^{(n)} = \sum_{i=0}^{\min(m,n)}v^{2mn-i(i+1)}q^{i(n+m-2i)+i(i-1)/2}A^{(n-i)}B^{(i)}D^{(m-i)}\,.
    \]
    
With these preparations we proceed to the main goal of this section. 

\begin{proof}[Proof of Proposition~\ref{prop:PowerFormula}] Assume $A$, $B$, $C$, and $D$ fulfill the commutation relations stated in the proposition. The assertion is trivial for $n=0$ and easily worked
out for $n=1\,$. Assume, thus, that the formula hold for $(A+C+D)^n$, with coefficients $a^n_{s,t,k}$ as in \eqref{eq:def-a-stk}. We compute 
    \begin{align*}
        (A+C+D)^{n+1}
        =&~
        \sum_{s=0}^n\sum_{t=0}^{\min(s,n-s)}\sum_{k=0}^{n-s-t}a^n_{s,t,k}A^{s-t}B^tC^kD^{n-s-t-k}(A+C+D)
        \\ \rule{0mm}{11mm}=&~
        \sum_{s=0}^n\sum_{t=0}^{\min(s,n-s)}\sum_{k=0}^{n-s-t} 
        a^n_{s,t,k}q^{-2(n-s+t)} A^{s-t+1}B^tC^kD^{n-s-t-k}
        \\&\qquad +
        a^n_{s,t,k}q^{-3(n-1-s-t-k)}[n-s-t-k] A^{s-t} B^{t+1}C^kD^{n-s-t-k-1}
        \\&\qquad +
        a^n_{s,t,k}q^{-2(n-s-t-k)}A^{s-t}B^tC^{k+1}D^{n-s-t-k}
        \\&\qquad +
        a^n_{s,t,k}
        A^{s-t}B^tC^kD^{n+1-s-t-k}
        \\ \rule{0mm}{11mm}=&
        \sum_{s=1}^{n+1}\sum_{t=0}^{\min(s-1,n+1-s)}\sum_{k=0}^{n+1-s-t} 
        a^n_{s-1,t,k}q^{-2(n+1-s+t)} A^{s-t}B^tC^kD^{n+1-s-t-k}
        \\&+
        \sum_{s=1}^{n+1}\sum_{t=1}^{\min(s,n+1-s)}\sum_{k=0}^{n+1-s-t} 
        a^n_{s-1,t-1,k}q^{-3(n+1-s-t-k)}[n+2-s-t-k] A^{s-t} B^{t}C^kD^{n+1-s-t-k}
        \\&
        +
        \sum_{s=0}^n\sum_{t=0}^{\min(s,n-s)}\sum_{k=1}^{n+1-s-t}
        a^n_{s,t,k-1}q^{-2(n+1-s-t-k)}A^{s-t}B^tC^{k}D^{n+1-s-t-k}
        \\&
        +
        \sum_{s=0}^n\sum_{t=0}^{\min(s,n-s)}\sum_{k=0}^{n-s-t}a^n_{s,t,k}
        A^{s-t}B^tC^kD^{n+1-s-t-k}
        \\ \rule{0mm}{11mm}=&
        \sum_{s=0}^{n+1}\sum_{t=0}^{\min(s,n+1-s)}\sum_{k=0}^{n+1-s-t}a^{n+1}_{s,t,k}A^{s-t}B^tC^kD^{n+1-s-t-k}\,.
    \end{align*}
In the first step, we employ the original commutation relations as well as Lemma~\ref{lem:uvComm}.
In the second step, we shift summation indices and make use of the boundary conditions on 
the $a^n_{s,t,k}\,$. The last equation then results from Lemma~\ref{lm:rec-a-stk}, completing the induction step.
\end{proof}

\subsection{A Summation Identity for the \texorpdfstring{$c'_{k,p}$\,}{c'kp}}
As before, let $c_{k,p}$ be the coefficients defined recursively in \eqref{eq:cRecursion} 
and $c'_{k,p}=\mathsf{sq}(c_{k,p})$ the element in $\Zqq$ with $q$ replaced by $q^2$ as
explained in Section~\ref{subsec:ProofCoprodB2}.

\begin{lem}\label{lem:SumcIdentity}
    For all $n\in\bbn_0$\,, the following equality holds in $\Zqqn{3}\,$,
    \[
    \sum_{p=0}^{\floor{n/2}}\left(-\frac{q-q^{-1}}{q^2[2]}\right)^p c'_{n-2p,p}
    =
    q^{-\binom{n}{2}}[2]^{-n}\qfacrelB{n}{2}\,.
    \] 
\end{lem}
\begin{proof}
    Recall that $\qfacrelB{n}{2}=\prod_{i=1}^n(q^i+q^{-i})$ under the specialization of $\edgenum=2$ in \eqref{eq:factorialratios}. We prove that both expressions in the proposed equality are sequences in $n$ with the same initial values and satisfy the same recursion relation. Denote the sides of the desired equality by
    \begin{align*}
        A_n &=  \sum_{p=0}^{\floor{n/2}} \phi_p c'_{n-2p,p}\,, &&\mbox{where $\phi_p=\left(-\frac{q-q^{-1}}{q^2[2]}\right)^p$, and }\\
        B_n &=  q^{-\binom{n}{ 2}} \psi_n\,, &&\mbox{where  $\psi_n=[2]^{-n}\qfacrelB{n}{2}$}\,.
    \end{align*}
    Observe that $A_0=B_0=A_1=B_1=1$. The recursion for $A_n$ is given below. Note that if $n$ is even, $c'_{n-2(\floor{n/2}-1),p}=0$ and $n/2-1=\floor{(n-1)/2}$. Whereas if $n$ is odd, $\floor{n/2}=\floor{(n-1)/2}$ and $c'_{n-2(\floor{n/2}-1),p}\neq 0$.
    \begin{align*}
        A_{n}&=\sum_{p=0}^{\floor{n/2}} \phi_p c'_{n-2p,p}=
        \sum_{p=0}^{\floor{n/2}} \phi_p ( c'_{n-2p-1,p}+ c'_{n-2p,p-1}q^{-2(n-2)}[n-1]_{q^2})
        \\
        &=
        A_{n-1}+q^{-2(n-2)}[n-1]_{q^2} \sum_{p=0}^{\floor{n/2}-1} \phi_{p+1}  c'_{n-2p-2,p}
        \\
        &=A_{n-1}-q^{-2(n-2)}[n-1]_{q^2}\left(\frac{q-q^{-1}}{q^2[2]}\right)A_{n-2}
        \\
        &=
        A_{n-1}-q^{-2(n-1)}[n-1]_{q^2}\frac{q-q^{-1}}{[2]}A_{n-2}
    \end{align*}
    Hence $A_{n}-A_{n-1}=-q^{-2(n-1)}[n-1]_{q^2}\frac{q-q^{-1}}{[2]}A_{n-2}$\,. An identical recursion for $B_n$ is found by computing
    \begin{align*}
        B_{n}-B_{n-1}&=
        \psi_{n-2}[2]^{-1}(q^{n-1}+q^{-(n-1)})\left(q^{-\binom{n}{ 2}} (q^n+q^{-n})[2]^{-1}-q^{-\binom{n-1}{ 2}}\right)
        \\
        &=
        \psi_{n-2}[2]^{-1}(q^{n-1}+q^{-(n-1)}) q^{-\binom{n-2}{ 2}}\left(q^{-2n+3}(q^n+q^{-n})[2]^{-1}-q^{-n+2}\right) 
        \\
        &=  
        B_{n-2}\frac{[n-1]_{q^2}}{[2][n-1]}\left((q^{-n+3}+q^{-3n+3})-(q^{-n+3}+q^{-n+1})\right) 
        \\
        &=
        -q^{-2(n-1)}\frac{[n-1]_{q^2}}{[2]}(q-q^{-1})B_{n-2} 
        \,.
    \end{align*}
    It follows that $A_n=B_n$ by strong induction, which is the identity claimed in the lemma. 
\end{proof}

\medskip

\section{Inexistence of Quantum Weyl Elements and Related Notes}\label{sec:InExWeyl}

In the finite-dimensional representation theory of classical, simple Lie algebras, Weyl elements 
are formal elements that are instrumental in the derivation of various dimension and character formulae. For a  simply connected simple Lie group $G$ with maximal torus $T\subset G$ the geometric Weyl group $\Weyl$, defined as 
$N(T)/T$, is acting on $\mathfrak h$, the Lie algebra of $T$, via conjugation by representatives in the normalizer $N(T)\,$. 
The action of a standard generator $s_i\in\Weyl$ is given by $S_i=\mathrm{Ad}(\mathscr w_i)$, 
where $\mathscr w_i= \exp(E_i)\exp(-F_i)\exp(E_i)\in N(T)$. (see, for example, \cite{Hu78}). The elements $\mathscr w_i$ also appear in the theory of Chevalley groups and Tits systems. 

The $\Weyl$-action extends naturally to $\mathfrak g$ and hence to
$U(\mathfrak g)$. There, the generators $S_i$ coincide with the $T_i$ or $\Tinv_i$ actions
on $\Uhg$ from Section~\ref{subsec:LATM} in the specialization $\hbar\rightarrow 0\,$. On finite-dimensional highest weight representations $V$ of the corresponding Lie algebra $\mathfrak g$ the $\mathscr w_i$\,, as elements of $\mathrm{End}(V)$, are well-defined since $E_i$ and $F_i$ act nilpotently on $V$.

Quantum Weyl elements attempt to define analogs of $\mathscr w_i$ for certain versions of quantum groups, which implement the $T_i$ or $\Tinv_i$ actions in the same manner by conjugation. Constructions of such elements are described in 
\cite{So90,LS90,KR90} and \cite[8.2]{CP95}. In essence, these approaches map a quantum group $U$ into a completion of $\bigoplus_\lambda\mathrm{End}(V_\lambda)$, where $\lambda$ ranges over dominant highest weights so that
each $V_\lambda$ naturally carries a $\Weyl$-action. The $\mathscr w_i$ 
elements are then used to derive various formulae for $R$-matrices and 
$\mathscr A$-actions.

For relevant versions of quantum groups at roots of unity, however, quantum Weyl elements implementing the $T_i$-actions {\em almost never} exist for $\Uz$ and related quantum algebras, even if formulated in a very weak sense. The criteria discussed in this section illustrate that such implementations exist only in very exceptional types of  representations. 

Consider a representation $\rho:\UzQ\rightarrow\mathrm{End}(V)$ on a vector space $V$ over $\mathbb C$
give an embedding $\mathbb Q(\zeta)\subset\mathbb C$. 
Let 
$\mathfrak z$ be the center of $\UzQ$ and $\Zsubalgchar=\Zsubalgchar_{\bullet}\otimes\mathbb C$ the full subalgebra described in Section~\ref{subsec:Z=SignPolyn}. With primes indicating commutants, we have inclusions
$\,\mathrm{im}(\rho)\subseteq \mathrm{im}(\rho)''\subseteq (\rho(\mathfrak z))'\subseteq (\rho(\mathfrak z\cap \Zsubalgchar))'=:C_\rho\,$ of subalgebras of $\mathrm{End}(V)\,$. We say that a representation $(\rho,V)$ is {\em weakly $\mathscr A$-equivariant} if there exist 
invertible operators $\Wops_i\in C_\rho$ such that $\,\rho(\Tinv_i(x))=\Wops_i\rho(x)\Wops_i^{-1}\,$ for all $i=1,\ldots, n\,$ and 
$x\in\UzQ\,$. The existence of quantum Weyl elements $\mathscr w_i\in\UzQ$ would immediately imply this property 
for {\em all} representations
by setting $\Wops_i=\rho(\mathscr w_i)\in \mathrm{im}(\rho)$ and is, thus, a far stronger requirement.

For the assertions and counterexamples that follow, assume $\kay$ is odd and exclude Lie types $\LT{B}$ and $\LT{G}$\,. Fix a maximal word $z\in\wordsetmax$ and denote respective generators by $E_\alpha$ in the sense of \eqref{eq:Genbyroot}. Similar results hold for other Lie types and even $\kay$ but are technically more involved.  

For a given representation,
$(\rho,V)$ use the shorthand $\bar A=\rho(A)$ for $A\in\UzQ$ as well as  $F_\alpha=E_{-\alpha}$
and $\Fpw_\alpha=\Epw_{-\alpha}$ for $\alpha\in\proots\,$. Let also  $\grf{\Lambda}_\kay=\mathbb Z^{{\sroots}}/\kay \mathbb Z^{{\sroots}}\cong (\mathbb F_\kay)^{\sroots}$ be the 
weight space mod $\kay$, generated by the fundamental weights $\varpi_i\,$ defined via 
$\symbrack{\varpi_i}{\breve\alpha_j}=\delta_{ij}\,$. Assuming $(\rho,V)$ is weakly $\mathscr A$-equivariant,  denote by $\Wopsalg$ the subgroup of $\mathrm{GL}(V)\cap C_\rho$ generated by 
$\Wops_i\,$.  Moreover, for a reduced word $w=w_{i_1}\ldots w_{i_k}\,$, write $\Wops_w=\Wops_{i_1}\cdot\ldots\cdot \Wops_{i_k}\,$. Recall also the augmentation ideals
$\ZmaxId{\bullet}^\utypechar$ and 
$\ZmaxIdHat{\siAA}$ as defined in \eqref{eq:MaxKidealNots} and \eqref{eq:ZidealMaxNotHat} of Section~\ref{subsec:Z_Ideals}.

\begin{prop}\label{prop:AequivConds}
Suppose $(\rho,V)$ is weakly $\mathscr A$-equivariant with $\{\Wops_i\}$ and $\Wopsalg$ as above, $\kay$ odd, and the Lie type is neither $\LT{B}$ nor $\LT{G}\,$. Then

\vspace{-1mm}

\begin{enumerate}[label=\roman*), leftmargin=13mm,] 
    \item We have decompositions $V=\bigoplus_{\lambda\in \grf{\Lambda}_\kay} V_\lambda\,$, where $\,K_iv=\zeta^{\symbrack{\lambda}{\alpha_i}}v\,$ for $\,v\in V_\lambda\,$. Moreover, $\,\bar E_\alpha(V_\lambda)\subseteq V_{\lambda+\alpha}\,$ and $\,\bar F_\alpha(V_\lambda)\subseteq V_{\lambda-\alpha}\,$.\vspace{1.7mm}\label{item:AequivConds:weights} 
    \item For any $w\in\wordset$ we have $\Wops_w(V_\lambda)=V_{s(\lambda)}$  where $s=\wordroot(w)$.\vspace{1.7mm}\label{item:AequivConds:weylmap} 
    \item For any pair $\alpha,\beta\in \proots$ with $d_\alpha=d_\beta\,$  there exists an $\,\Wops\in\Wopsalg\,$ such 
    that $\,\Wops\bar E_\alpha \Wops^{-1}=\bar E_\beta\,$. Similarly, 
    $\bar F_\alpha$ and  $\bar F_\beta$ are conjugate by some element in $\,\Wopsalg\,$, as are $\bar F_\alpha$ and
    $-\bar E_\beta\bar K^\beta \,$. 
    \vspace{1.4mm}\label{item:AequivConds:genconj} 
    \item There are operators $\,\bar Z_1, \bar Z_{\maxd}\in \Wopsalg'\cap (\mathrm{im}(\rho))'\subset\mathrm{End}(V)\,$ such that $\bar \Epw_\alpha=- \bar \Fpw_\alpha=\bar Z_1$ for all short roots $\alpha\in\proots$ (that is, $d_\alpha=1$) and $\bar \Epw_\alpha=- \bar \Fpw_\alpha=\bar Z_\maxd$ if $\alpha\in\proots$ is a long root ($d_\alpha=\maxd\,$).\vspace{1.4mm}\label{item:AequivConds:Zconst} 
    \item Suppose $V$ is cyclic with a cyclic vector $\mathbfit v\in V$ such that $\UzQ^+\mathbfit v=0$\,. Then
    $\ZmaxIdHat{\siAA}\subseteq \mathrm{ker}(\rho)$ so that $(\rho,V)$ factors into a representation of a respective {\em small} quantum group.\vspace{1.4mm}\label{item:AequivConds:hw=small}  
\end{enumerate}
\end{prop}

\begin{proof} Note first that the definition of weak $\mathscr A$-equivariance implies that for any $a\in \Zsubalgchar\cap \mathfrak z$ we have $\rho(\Tinv_s(a))=\rho(a)\,$. Now, since $\kay$ is odd we have $\,L_i=K_i^{\kay}\in\Zsubalgchar\cap \mathfrak z\,$, which implies with $\Tinv_j(L_i)=L_iL_j^{-A_{ji}}$ that $\,\bar L_j^{-A_{ji}}=\mathbb 1\in\mathrm{End}(V)\,$ for all $\,i,j=\{1,\ldots,n\}\,$. For Lie types different from $\LT{B}$ we have $\gcd(A_{j1},\ldots,A_{jn})=1$ for any $j$ so that, indeed, $\bar L_j=\bar K_j^{\kay}=\mathbb 1$ for all $j\,$. Over a field containing $\zeta$, this implies a decomposition into common eigenspaces of the $\bar K_j$ with eigenvalues of $\bar K_j$ given by some powers $\zeta^{\lambda_i}\,$ with $\lambda_i\in\mathbb F_\kay\,$. Now, for $\lambda=\sum_i\lambda_i\varpi_i$ we have $\lambda_i=\symbrack{\lambda}{\breve \alpha_i}\equiv \symbrack{\lambda}{\alpha_i}\mod\kay$ since $\kay$ is coprime to all $d_i$ if the Lie type $\LT{G}$ is also excluded, proving the decomposition in Item {\em \ref{item:AequivConds:weights}}. 
The other claims in Item {\em \ref{item:AequivConds:weights}} as well as Item {\em \ref{item:AequivConds:weylmap}} are immediate from commutation relations with the $K_j$ and the actions of the $\Tinv_i$ on the $K_j$\,. 

For Item {\em \ref{item:AequivConds:genconj}} write $E_\alpha=\Tinv_s(E_i)$ and $E_\beta=\Tinv_t(E_j)$ for some
$s,t\in\Weyl$ and $i,j\in\{1,\ldots,n\}\,$ with $d_i=d_j$\,. So,  $\bar E_\alpha=\Wops_w\bar E_i\Wops_w^{-1}$ and 
$\bar E_\beta=\Wops_u\bar E_j\Wops_u^{-1}\,$ with $s=\Weylpres(w)$ and $t=\Weylpres(u)\,$. Since $\alpha_i$
and $\alpha_j$ have the same length, there is some $r\in\Weyl$ such that $r(\alpha_i)=\alpha_j\,$. From
Proposition~\ref{prop:wordgencont} we then find that $\Tinv_r(E_i)=E_j$ and, hence, 
$\Wops_z \bar E_i\Wops_z^{-1}=\bar E_j$ where $r=\Weylpres(z)\,$. The claim now follows for $\Wops=\Wops_u\Wops_z\Wops_w^{-1}\in\Wopsalg\,$. The remaining assertions follow from analogous arguments.

Since all $\Epw_\alpha,Y_\alpha\in\Zsubalgchar\cap\mathfrak z$ by Proposition~\ref{prop:Zcentral}, we infer that $\bar\Epw_\alpha,\bar Y_\alpha \in \Wopsalg'\cap (\mathrm{im}(\rho))'\,$ for all $\alpha\in\proots\,$. Item~{\em\ref{item:AequivConds:genconj}} now implies that $\bar\Epw_\alpha=\bar\Epw_\beta$ and  
$\bar\Fpw_\alpha=\bar\Fpw_\beta$ whenever $d_\alpha=d_\beta\,$. Next, $\bar\Kpw_i=\mathbb 1$ and \eqref{eq:TinvXYsame} also imply
$\bar \Epw_i=-\bar \Fpw_i$ for all $i\,$ so that also $\bar\Epw_\alpha=-\bar\Fpw_\beta$
whenever $d_\alpha=d_\beta\,$, proving the claim in Item {\em \ref{item:AequivConds:Zconst}}.

For a cyclic vector $ \mathbfit v\in V$ as in Item {\em \ref{item:AequivConds:hw=small}}, note that, 
by Item {\em \ref{item:AequivConds:Zconst}}, any $\bar Z\in\{\bar Z_1,\bar Z_\maxd\}$ is given as
$\bar Z=\bar \Epw_i$ for some $i$ so that $\bar Z\mathbfit v=\bar \Epw_i\mathbfit v=0\,$. Given that any 
$\bar A\in \mathrm{im}(\rho)\,$ commutes with $\bar Z$\,, we also find  $\bar Z\bar A\mathbfit v=\bar A\bar Z\mathbfit v=0$
and hence $\bar Z = \bar0$ by the cyclicity assumption. Item~{\em \ref{item:AequivConds:Zconst}} 
now implies that all $\bar \Epw_\alpha=\bar \Fpw_\alpha=0\,$. Thus, $\Epw_\alpha=E_\alpha^{\ell_\alpha}$, 
$\Fpw_\alpha=F_\alpha^{\ell_\alpha}$, and $K_i^{\ell_i}-1=\Kpw_i-1$ are all in the kernel of $\rho$, which then factors into the small quantum group. 
\end{proof}

In the remainder of this section, we provide several examples of representations that are not weakly $\mathscr A$-equivariant using the criteria from the above proposition. 

Denote by $V_\kappa$ the weight space on which each $K_i$ acts as $\kappa_i\mathbb 1\,$ for an $n$-tuple $\kappa=(\kappa_1,\ldots,\kappa_n)\in \mathbb t=(\mathbb C^{\times})^n\,$\,. For any $\kappa\in\mathbb t$, 
a representation with $V_\kappa\neq 0$ can be easily constructed, for example, as a highest weight representation with  highest weight $\kappa\,$. The latter can also chosen to be finite-dimensional and irreducible by dividing out maximal submodules as usual. The condition in Item~{\em \ref{item:AequivConds:weights}} thus restricts the set of possible weights from the infinite group $\mathbb t$ to the finite subgroup $\mathbb s=\{\kappa\in\mathbb t:\kappa_i^\kay=1\,\forall i\}$ of tuples of roots of unity. In this sense, even the weak $\mathscr A$-equivariance condition forces a restriction to a null set.

Requiring $\bar\Kpw_i=\mathbb 1$ still leaves large families of finite-dimensional 
representations that are not waekly $\mathscr A$-equivariant. Assuming that $\kay$ is odd,
these are readily constructed by choosing 
{\em any} ideal $J\subsetneq \ZmaxId{\bullet}^-\subsetneq\Zsubalgchar^-\cong\mathbb C[\{\Fpw_\alpha:\alpha\in\proots\}]\,$ such that $\Zsubalgchar^-/J$ is of finite rank $m>1\,$.
Since $\Zsubalgchar$ is central, any weight $\kappa\in\mathbb s$ defines a 
one-dimensional representation of $B_J=J\cdot\Uzk{\mathbb k}^{\geqzero}$ with $K_i\mathbfit v=\kappa_i\mathbfit v$ and 
$E_i\mathbfit v=b\mathbfit v=0$ for all $i\in\{1,\ldots,n\}$ and all $b\in J\,$. The induced
module $\Uzk{\mathbb k}\otimes_{B_J}\mathbfit v$ is now finite-dimensional with $\bar\Kpw_i=\mathbb 1$,
but violates Item~{\em  \ref{item:AequivConds:hw=small}} in the proposition above and is, thus, not 
$\mathscr A$-equivariant. There are many choices for $J$ for which the $\bar\Fpw_\alpha$ are not zero and may in fact even be invertible.  

\subsection{An Example in Rank 1} Even among representations that factor through the small quantum groups, most are still not weakly $\mathscr A$-equivariant. This can already be seen in the rank 1 case $\mathfrak g=\mathfrak{sl}_2$ and $\kay=3\,$ for which $E^3=F^3=0$ and $K^3=1\,$. 

Denote by $V_0$\,, $V_-$\,, and $V_+$ the eigenspaces of $K$ for eigenvalues 
1, $\zeta^2=\zeta^{-1}$, and $\zeta^{-2}=\zeta^{1}$, respectively.
We consider a special family of representations in which $V_+$ and $V_-$ are isomorphic as vector spaces
to a common space $V_*$ so that, after conjugation by respective isomorphisms, 
the restrictions $E:V_-\rightarrow V_+$ and $F:V_+\rightarrow V_-$ may be viewed as linear endomorphisms of $V_*\,$.
Further specializing our example, we assume that these two endomorphisms are the identity on $V_*\,$.
The restrictions of $E$ and $F$ to other eigenspaces are then denoted as 
 $E_-,F_+:V_0\rightarrow V_*$ and $E_+,F_-:V_*\rightarrow V_0\,$. They form a representation of the small quantum group iff
 \begin{equation}\label{eq:QuiverRep}
     \begin{aligned}
       &  E_+E_-=F_-F_+=0 \quad\mbox{and}\quad E_+F_+=F_-E_-\qquad &\mbox{ on }V_0\\
      \mbox{and} \hspace*{20mm} &  E_-E_+=F_+F_-=E_-F_-=F_+E_+=0 \qquad &\mbox{ on }V_*\,.
     \end{aligned}
 \end{equation}

 This type of reduction of representations of $\Uz(\mathfrak{sl}_2)$ to a quiver with relations as in \eqref{eq:QuiverRep} extends to all roots of 
 unity as shown, for example, in \cite{Ke89}. The representation given by
 the $E_{\pm}$ and $F_{\pm}$ as above is $\mathscr A$-equivariant if and only if there are isomorphisms $W_0\in\mathrm{GL}(V_0)$
and $W_*\in\mathrm{GL}(V_*)$ such that
\begin{equation}\label{eq:QuiverAequiv}
    W_0E_+W_*^{-1}=F_-\,,
\quad
W_*E_-W_0^{-1}=-F_+\,,
\quad 
W_*F_+W_0^{-1}=E_-\,,
\quad
\mbox{and}
\quad 
W_0F_-W_*^{-1}=-E_+
\end{equation}
implying, in particular, $\mathrm{rank}(E_+)=\mathrm{rank}(F_-)$ and $\mathrm{rank}(E_-)=\mathrm{rank}(F_+)\,$. 

It is not hard to construct maps that fulfill \eqref{eq:QuiverRep} but violate \eqref{eq:QuiverAequiv}. For example,
setting $E_+=F_-=0$ automatically solves all equations in \eqref{eq:QuiverRep} so that any choice of maps 
$E_-,F_+:V_0\rightarrow V_*$ produces a representation. The residual representation theory is, thus, equivalent to that
of the  (2) Kronecker quiver, which is tame and for which the indecomposable modules are classified \cite[Theorem 4.3.2]{Be95}. 

Kronecker's classification includes a generic family in which $\dim(V_0)=\dim(V_*)$ and one of the maps can be chosen as the identity matrix and the other an indecomposable rational form. For example, if $\mathbb k=\mathbb C$ the latter implies an indecomposable Jordan block $\,J=\kappa\mathbb 1+N$, where $N$ is maximally nilpotent. It is not hard to check that this representation is weakly $\mathscr A$-equivariant if and only if $\kappa^2=-1\,$, again singling out a finite set from an infinite parameter space. 

For the second family in the classification, the dimensions of $V_0$ and $V_*$ differ by one and there is exactly one representation for each such choice of dimensions. 
Bases $\{v^0_i\}$ of $V_0$ and  $\{v^*_i\}$ of $V_*$ 
may be chosen such that $E_-v^0_i=v^*_i$ and $F_+v^0_i=v^*_{i+1}$ or $F_+v^0_i=v^*_{i-1}\,$.  
These special cases are indeed weakly $\mathscr A$-equivariant. The $W_0$ and $W_*$ may be chosen as anti-diagonal matrices with alternating $+1$ and $-1$ entries. 

Note that if we require representations to lift to a larger (unrolled) quantum group that contains the element $H$ as in \eqref{eq:Lierels2}, the generic family is excluded unless $\kappa=0\,$ (which does not yield a weakly  $\mathscr A$-equivariant module due to the rank condition). In this situation, however, the notion of weak $\mathscr A$-equivariance on the larger algebra implies
also $-\bar H=\Wops\bar H\Wops^{-1}$ since $\Tinv(H)=-H\,$, implying symmetry of $H$-weights around $0$. The latter is easily broken by shifts $\bar H\mapsto \bar H+m\kay\mathbb 1$\,, producing infinite families of lifts of a representation that are not weakly $\mathscr A$-equivariant.

For irreducible representations arising from sufficiently small 
dominant highest weights (in a respective principal affine Weyl alcove) 
one may expect weak $\mathscr A$-equivariance to apply based on the original
constructions of quantum Weyl elements. In the $\Uz(\mathfrak{sl}_2)$ 
situation above, this corresponds to the trivial solution of 
\eqref{eq:QuiverRep} and \eqref{eq:QuiverAequiv} in which all maps are zero.

Finally, we note that the standard projective representation in the special $\Uz(\mathfrak{sl}_2)$ case is in fact weakly $\mathscr A$-equivariant. In terms of the maps
above it is given in a suitable basis by matrices
$$
E_+=\begin{bmatrix}
    0 & 0 \\ 1 & 0
\end{bmatrix}
\,,\;\;
E_-=\begin{bmatrix}
    0 & 0 \\ 1 & 0
\end{bmatrix}
\,,\;\;
F_+=\begin{bmatrix}
    1 & 0 \\ 0 & 0
\end{bmatrix}
\,,
\;\;\mbox{ and }\;\;
F_-=\begin{bmatrix}
    0 & 0 \\ 0 & 1
\end{bmatrix}
\,.
$$
The conditions \eqref{eq:QuiverAequiv} are then solved, for example, by $W_0=\mathbb 1$
and 
$W_*=\begin{bmatrix}
    0 & -1 \\ 1 & 0
\end{bmatrix}\,$.

\medskip

\section{Integral Form of Coalgebras}
\label{sec:integralcoalg}
We introduce an additional normalization $\hat X_\alpha$ of the De Concini-Procesi singularized power generators $\breve\Epw_\alpha\,$, which entails integral coproduct expressions in our formulae for types $\mathsf{A}_n$ and $\mathsf{B}_2$\,. Notice from \eqref{eq:CoprodB2:iji} that the singularized generator $\breve\Epw_{(iji)}$ has a coproduct over the Gauss integers $\mathbb{Z}[\scalebox{.75}{$\sqrt{-1}\,$}]$ when $\ell$ is even. Thus, the integral coordinate rings arise as Hopf subalgebras of $\Uzg$ in these types. These integral forms also support simplified computations in future work.

In type $\mathsf{A}_n$\,, we gave in Lemma \ref{lem:An-CoprX} a formula for the coproduct of generators $X_{i,j}$
\begin{equation}
        \Delta(\Epw_{i,j})=\sum_{k=i}^jb_k \Epw_{i,k}\otimes \Kpw_{i,k}\Epw_{k,j} \;,
\end{equation} 
where $\;b_k=(\zeta^{-1}-\zeta)^\ell\zeta^{\binom{\ell}{ 2}}\;$ for $i<k<j$ and $b_i=b_j=1\,$. Replacing all instances of $\Epw_\alpha$ with the singular generator $\breve\Epw_\alpha$ removes the coefficient $(\zeta^{-1}-\zeta)^\ell$ from $b_k$\,. If we set $\hat\Epw_\alpha = \zeta^{\binom{\ell}{ 2}}\breve\Epw_\alpha = (\zeta^{-1}-\zeta)^\ell\zeta^{\binom{\ell}{ 2}}\Epw_\alpha$\,, then we obtain a coproduct formula in which all nonzero coefficients are $1$\,.
\begin{equation}
    \Delta(\hat\Epw_{i,j})=\sum_{k=i}^j \hat\Epw_{i,k}\otimes \Kpw_{i,k}\hat\Epw_{k,j}  
\end{equation}

More generally, set $\hat E_i^{k} = \zeta_i^{\binom{k}{ 2}} \breve E_i^k$, then write $\hat \Epw_i = \hat E_i^{\elli}$. The coproduct formulae of Proposition \ref{prop:CoprodB2} have simpler presentations using the rescaled singularized generators
\begin{align}
    \phantom{\Delta(\Epw_{(iji)})}&\begin{aligned}
        \mathllap{\Delta(\hat\Epw_{(iji)})}&=
1\otimes \hat\Epw_{(iji)}+\hat\Epw_i\otimes \Kpw_i \hat\Epw_j^{\gcdle}+\hat\Epw_{(iji)}\otimes \Kpw_i \Kpw_j^{\gcdle}\,,
    \\&
    \qquad+\delta_{\gcdle,2}\cdot 2\cdot  (-1)^\ellj \hat\Epw_{(ij)}\otimes \Kpw_i \Kpw_j\hat\Epw_j
    \end{aligned}
    \\
    \rule{0pt}{31pt}
    &\begin{aligned}
        \mathllap{\Delta(\hat\Epw_{(ij)})}&=
    1\otimes \hat\Epw_{(ij)}+\hat\Epw_i^{\gcdleopp}\otimes \Kpw_i^{\gcdleopp}\hat\Epw_j+\hat \Epw_{(ij)}\otimes \Kpw_i^{\gcdleopp}L_j
    \\&\qquad +
    \delta_{\gcdleopp,2}\cdot 2\cdot  \zeta^{\binom{\ell+1}{ 2}}\hat  \Epw_i\otimes \Kpw_i\hat\Epw_{(iji)}\,.
    \end{aligned}
\end{align}
These follow from the observations $\hat\Epw_j^\gcdle = \zeta_j^{\binom{\ell}{ 2}} \breve\Epw_j^\gcdle=(-1)^{\ell+1} \breve\Epw_j^\gcdle$ and $\hat \Epw_i^\gcdleopp = \zeta^{\binom{2\ellj }{ 2}}\breve \Epw_i^\gcdleopp$. 
Note also that $\delta_{\gcdle,2} \zeta^{\binom{\ell}{ 2}}=\delta_{\gcdle,2} (-\zeta)^{-\ellj}$, 
$\delta_{\gcdleopp,2} \zeta^{2\binom{\ell}{ 2}}=1$, 
$\delta_{\gcdleopp,2} \zeta^{\binom{\ell}{ 2}+\ellj}=\delta_{\gcdleopp,2} \zeta^{\binom{\ell}{ 2}+\ell}=\delta_{\gcdleopp,2} \zeta^{\binom{\ell+1}{ 2}}$ are all valued in $\{-1,0,1\}$.

The automorphism $\Kinvaut$ is given for certain power generators in \eqref{eq:XJ-form} and \eqref{eq:XI-form}. Here, we regard it as a factor of the antipode, which takes values over the same ring as the coproduct. We deduce integrality for the $\hat\Epw_\alpha$ from the respective equations.
\begin{equation}
\begin{aligned}
   (-1)^\ell\hat\Epw_{(ji)}\,&= \,\zeta_j^{\ell}\hat\Epw_{a_{m-1}}\,+\,
(-1)^\ell\hat\Epw_j^\gcdle\hat\Epw_i\,+\,
(1-\delta_{\gcdle,1})\cdot\edgenum\cdot(-1)^{\ell} \cdot T_1\vspace*{8mm}\\
\mbox{where} & \hspace*{24mm} 
     T_1= 
     \begin{cases}
    (-1)^{\ell_j+1}\hat\Epw_j\hat\Epw_{(ij)} & \edgenum=2\vspace*{2.5mm}\\
    \zeta_j^{2\binom{\ellj}{ 2}}\hat\Epw_j^2\hat\Epw_{(ij)}
    \,+\,
    \hat\Epw_j\hat\Epw_{(ijij)}& \edgenum=3 \,,
    \end{cases}
\end{aligned}
\end{equation}

\begin{equation}
\begin{aligned}
   (-1)^{\edgenum\ell_j}\hat\Epw_{b_{m-1}}\,&= \,\zeta_j^{\ell_j}\hat\Epw_{(ij)}\,+\,
(-1)^{\gcdleopp\cdot\ell}\hat\Epw_j\hat\Epw_i^\gcdleopp\,+\,
\delta_{\gcdle,1}\cdot(1-\delta_{\edgenum,1})\cdot\edgenum\cdot (-1)^{\ell}\cdot T_2\vspace*{8mm}\\
\mbox{where} & \hspace*{24mm} 
     T_2= 
     \begin{cases}
    \zeta^{\ell} \hat\Epw_{(iji)}\hat\Epw_{i} & \edgenum=2\vspace*{2.5mm}\\
    \hat\Epw_{(iji)}\hat\Epw_i\,
    +\,(-1)^{\ell+1}\hat\Epw_{(ijiji)}\hat\Epw_{i}^2
    & \edgenum=3 \,.
    \end{cases}
\end{aligned}
\end{equation}

The integrality of these expressions suggests that the normalized generators above produce integral coproducts for all Lie types and roots of unity.
\begin{conj}
    Let $z\in\wordsetmax$\,. The set of $\hat\Epw_w$ for $w\leqRB z$ together with all $L_i$ generate a $\bbz$-Hopf subalgebra of $\Uzg$\,.
\end{conj}

\bibliographystyle{alpha}
\bibliography{biblio}
\end{document}